\documentclass[11pt,reqno]{amsart}
\usepackage{amsmath,amssymb,amsthm,amsfonts,latexsym,times,color,bm}
\usepackage{amsmath,amsfonts}
\usepackage{amsthm}

\usepackage{inputenc}

\usepackage[linktocpage=true]{hyperref}

\usepackage{graphicx}
\usepackage{color}
\usepackage{multicol}

\usepackage{xcolor}

\usepackage[width=14cm, left=2cm, right=2cm,marginpar=2cm]{geometry}

\newcounter{mnotecount}[section]

\newtheorem{thm}{Theorem}[section]
\newtheorem{definition}{Definition}[section]

\newtheorem{Lemma}[thm]{Lemma}
\newtheorem{remark}{Remark}[section]
\newtheorem{theorem}[thm]{Theorem}
\newtheorem{proposition}[thm]{Proposition}

\newtheorem{corollary}[thm]{Corollary}

\numberwithin{equation}{section}

\newcommand{\beq}{\begin{equation}}
\newcommand{\eeq}{\end{equation}}
\newcommand{\ben}{\begin{eqnarray}}
\newcommand{\een}{\end{eqnarray}}
\newcommand{\beno}{\begin{eqnarray*}}
\newcommand{\eeno}{\end{eqnarray*}}

\newcommand{\bu}{\mathbf{u}}

\newcommand{\bv}{\mathbf{v}}

\newcommand{\bF}{\mathbf{F}}
\newcommand{\bG}{\mathbf{G}}

\newcommand{\bW}{\mathbf{W}}
\newcommand{\bU}{\mathbf{U}}

\newcommand{\bB}{\mathbf{B}}

\newcommand{\bQ}{\mathbf{Q}}
\newcommand{\bA}{\mathbf{A}}
\newcommand{\Bm}{\mathbf{m}}
\newcommand{\bw}{\mathbf{w}}
\newcommand{\bE}{\mathbf{E}}

\newcommand{\bx}{\mathbf{x}}
\newcommand{\by}{\mathbf{y}}

\newcommand{\sstar}{s}

\numberwithin{equation}{section}

\begin{document}
\title[Relativistic Euler equations]{Strichartz estimates and low regularity solutions of 3D relativistic Euler equations}


\author{Huali Zhang}
\address{School of Mathematics, Hunan University, Changsha 410082, CHINA.}
\email{hualizhang@hnu.edu.cn}

\date{\today}


\begin{abstract}
We study the low regularity well-posedness for Cauchy problem of 3D relativistic Euler equations. Firstly, we introduce a new decomposition for relativistic velocity and derive new transport equations for vorticity, which both play a crucial role in energy and Strichartz estimates. According to Smith-Tataru's approach, we then establish a Strichartz estimate of linear wave equations endowed with the acoustic metric. This leads us to prove a complete local well-posedness result if the initial logarithmic enthalpy, velocity, and modified vorticity $(h_0, \bu_0, \bw_0) \in H^s \times H^s \times H^{s_0} (2<s_0<s)$. Therefore, we give an affirmative answer to "Open Problem D" proposed by Disconzi.

Moreover, for $(h_0,{\bu}_0,\bw_0) \in H^{2+} \times H^{2+} \times H^2$, by frequency truncation, there is a stronger Strichartz estimate for solutions on a short-time-interval. By semi-classical analysis and induction method, these solutions can be extended from short time intervals to a regular time interval, and a uniform Strichartz estimate with loss of derivatives can be obtained. This allows us to prove the local well-posedness of 3D relativistic equations if $(h_0,{\bu}_0,\bw_0) \in H^{2+} \times H^{2+} \times H^2$.

\end{abstract}

\maketitle
\newpage

\tableofcontents

\newpage

\section{Introduction}
In the paper, we study the low regularity well-posedness theory for the Cauchy problem of relativistic Euler equations in the Minkowski background $\mathbb{R}^{1+3}$. The relativistic Euler equation describes the motion of a relativistic fluid, which is an essential model in mathematical general relativity. To start our problem, let us introduce the formulation of relativistic Euler equations in mathematics.

The fluid state is represented by the energy density $\varrho$, and the relativistic velocity $\bu=(u^0,u^1,u^2,u^3)^{\textrm{T}}$. The velocity is assumed to be a forward time-like vector field, normalized by\footnote{We also use the Greek indices $\alpha,\beta,\gamma,\cdots$ take on the values $0,1,2,3$, while the Latin spatial indices $a,b,c,\cdots,i,j,k,\cdots$ take on the values $1,2,3$. Repeated indices are summed over (from $0$ to $3$ if they are Greek, and from $1$ to $3$ if they are Latin). Greek and Latin indices are lowered and raised with the Minkowski metric $\Bm$ and its inverse $\Bm^{-1}$. We will use this convention throughout the paper unless otherwise indicated.} 
\begin{equation}\label{muu}
u^\alpha u_\alpha =-1,
\end{equation}
where $u_{\alpha }= m_{\alpha \beta} u^\beta$, and $\Bm=(m_{\alpha \beta})_{4\times 4}$ is the Minkowski metric. Due to the constraint \eqref{muu}, so there only three of the quantities $u^1,u^2,u^3$
are independent. We denote
\begin{equation}\label{rev}
	\mathring{\bu}=(u^1,u^2,u^3).
\end{equation}
Therefore, we can also write the velocity as
\begin{equation}\label{reva}
	\bu=(u^0,\mathring{\bu})^{\mathrm{T}}.
\end{equation}
An appropriate equation of state (please c.f. Rendall \cite{Ren}) is of the form
\begin{equation}\label{mo3}
	p=p(\varrho)=\varrho^{\vartheta}, \quad \mathrm{where} \ \vartheta>1 \ \mathrm{is} \ \mathrm{a}\ \mathrm{constant}.
\end{equation}
Let $\mathcal{Q}$ be the stress energy tensor, which is defined by
\begin{equation*}
	\mathcal{Q}^{\alpha \beta}=(p+\varrho)u^\alpha u^\beta+ p m^{\alpha \beta}.
\end{equation*}
Then the motion of the relativistic Euler equations can be described by 
\begin{equation}\label{mo1}
	\partial_{\alpha} \mathcal{Q}^{\alpha \beta}=0.
\end{equation}
Projecting \eqref{mo1} into the subspace
parallel and orthogonal to $u^\alpha$ yields (please c.f. Disconzi-Ifrim-Tataru \cite{DIT} or Oliynyk \cite{Todd1,Todd2})
\begin{equation}\label{OREE}
	\begin{cases}
		u^\kappa\partial_\kappa \varrho+(p+\varrho) \partial_\kappa u^\kappa=0,
		\\
		(p+\varrho)u^\kappa\partial_\kappa u^\alpha + (m^{\alpha \kappa}+u^{\alpha}u^{\kappa}) \partial_\kappa p=0.
	\end{cases}
\end{equation}
To study the Strichartz estimate and low regularity solutions, a good start is to see the specific structure of \eqref{OREE}, ie. a wave-transport system \cite{DS}. To get the wave-transport structure, we first introduce some good variables as follows.
\subsection{Definitions and good variables}
Let us introduce some variables--acoustic speed, number density, enthalpy, and the modified vorticity.
\begin{definition}\cite{DS}
We denote the acoustic speed $c_s$
\begin{equation}\label{mo4}
	c_s=\sqrt{\frac{dp}{d\varrho}},
\end{equation}
the number density $q$
\begin{equation}\label{mo5}
	p+\varrho=q\frac{d\varrho}{dq},
\end{equation}
and the enthalpy per particle $H$
\begin{equation}\label{mo6}
	H=\frac{p+\varrho}{q}.
\end{equation}
The vanishing of $c_s$ causes a degeneracy in the system, so we restrict our attention to solution in which
\begin{equation}\label{mo7}
	0<c_s \leq 1.
\end{equation}
\end{definition}
\begin{definition}\cite{DS}
	Let $\bar{H}>0$ be a fixed constant value of the enthalpy. We denote the (dimensionless) logarithmic enthalpy $h$
	\begin{equation*}
		h=\log\frac{H}{\bar{H}}.
	\end{equation*}
To be simple, we assume $\bar{H}=1$. Then we have
\begin{equation}\label{mo9}
	h=\log{H}.
\end{equation}
\end{definition}
Under the settings \eqref{mo4}, \eqref{mo5}, \eqref{mo6}, and \eqref{mo9}, then $\varrho, p$, and $c_s$ are functions of $h$. To be simple, we record
\begin{equation}\label{und}
\varrho=\varrho(h), \quad p=p(h), \quad c_s=c_s(h).
\end{equation}
In terms of these variables, after calculations (please refer Lemma \ref{equ0} below), the relativistic Euler equations \eqref{OREE} becomes
\begin{equation}\label{REE}
	\begin{cases}
		u^\kappa \partial_\kappa h=-c^2_s \partial_\kappa u^\kappa,
		\\
		u^\kappa \partial_\kappa u^\alpha =- (m^{\alpha \kappa}+u^{\alpha}u^{\kappa}) \partial_\kappa h.
	\end{cases}
\end{equation}
\begin{definition}
Let $H$ and $h$ be stated in \eqref{mo6}- \eqref{mo9}. For any given one-form $\bA=(A^0,A^1,A^2,A^3)^\mathrm{T}$, define the $\bu$-vorticity\footnote{In relativistic fluids, the vorticity $w_{\alpha \beta}$ is defined by $w_{\alpha \beta}=\partial_\alpha u_\beta- \partial_\beta u_\alpha$(cf. Choquet-Bruhat \cite{Bru} ). Our purpose here is to find a good structure of the $\bu$-vorticity and its higher derivatives, please see Lemma \ref{tr0} and Lemma \ref{VC}.}
\begin{equation}\label{VAd}
	\mathrm{vort}^\alpha(\bA)= -\epsilon^{\alpha \beta \gamma \delta}u_{\beta}\partial_{\gamma} A_\delta,
\end{equation}
where $\epsilon_{\alpha \beta \gamma \delta}$ is the fully antisymmetric symbol normalized by $\epsilon_{0123}=1$. We also denote the modified vorticity $\bw=(w^0,w^1,w^2,w^3)^{\mathrm{T}}$ by
\begin{equation}\label{VVd}
  {w}^\alpha=\mathrm{vort}^\alpha(H \bu)=-\epsilon^{\alpha \beta \gamma \delta}u_{\beta}\partial_{\gamma}(Hu_\delta)=-\epsilon^{\alpha \beta \gamma \delta}\mathrm{e}^{h}u_{\beta}\partial_{\gamma}u_\delta.
\end{equation}
\end{definition}
\begin{definition}
We define the acoustical metric $g=(g_{\alpha\beta})_{4\times 4}$ and its inverse $g^{-1}=(g^{\alpha\beta})_{4\times 4}$ as \footnote{In this setting, $g$ is a Lorentz metric and $g^{00}=-1$.}
\begin{equation}\label{AMd}
\begin{split}
  g_{\alpha\beta}&=\Omega^{-1} \left(  c_s^{-2}m_{\alpha\beta}+(c_s^{-2}-1)u_\alpha u_\beta \right)    ,
  \\
  g^{\alpha\beta}&=\Omega \left(  c_s^{2}m^{\alpha\beta}+(c_s^2-1)u^\alpha u^\beta  \right)  ,
\end{split}
\end{equation}
where $\Omega$ is defined by
\begin{equation}\label{AMd2}
	\begin{split}
		\Omega=\left(   c_s^{2}+(1-c_s^{2})(u^0)^2  \right)^{-1}.
	\end{split}
\end{equation}
\end{definition}
\begin{definition}
Let $\bu$ be the relativistic velocity and $h$ the logarithmic enthalpy. Let $\bw$ be set in \eqref{VVd}. We denote the fluid variables $\bW=(W^0,W^1,W^2,W^3)^{\mathrm{T}}$ and $\bG=(G^0,G^1,G^2,G^3)^{\mathrm{T}}$ by
\begin{equation}\label{MFd}
	\begin{split}
		W^\alpha = & \mathrm{vort}^{\alpha}(\mathrm{e}^{h}\bw)= -\epsilon^{\alpha \beta \gamma \delta}u_{\beta}\partial_{\gamma}w_\delta+c_s^{-2}\epsilon^{\alpha\beta\gamma\delta}u_{\beta}w_{\delta}\partial_{\gamma}h,
		\\
		G^\alpha = & \mathrm{vort}^{\alpha}\bW= -\epsilon^{\alpha \beta \gamma \delta}u_{\beta}\partial_{\gamma}W_\delta.
	\end{split}
\end{equation}
\end{definition}
\begin{definition}
Let $\bu$ be the relativistic velocity and $h$ the logarithmic enthalpy. Let $\bw, \bW$ and $\bG$ be defined in \eqref{VVd} and \eqref{MFd}. We denote
	\begin{equation}\label{MFda}
		\mathring{\bw}=(w^1,w^2,w^3)^{\mathrm{T}}, \quad \mathring{\bW}=(W^1,W^2,W^3)^{\mathrm{T}}, \quad \mathring{\bG}=(G^1,G^2,G^3)^{\mathrm{T}}.
	\end{equation}
\end{definition}
We are now ready to introduce a wave-transport system of \eqref{REE}.
\subsection{Wave-transport system}
\begin{Lemma}\label{WT}
	Let $(h,\bu)$ be a solution of \eqref{REE}. Let $\bw$ and $\bW$ be defined in \eqref{VVd} and \eqref{MFd}. Then System \eqref{REE} can be written as
	\begin{equation}\label{WTe}
		\begin{cases}
			\square_g h=D,
			\\
			\square_g u^\alpha=- c_s^2\Omega \mathrm{e}^{-h}W^\alpha+Q^\alpha,
			\\
			u^\kappa \partial_\kappa w^\alpha = -u^\alpha w^\kappa \partial_\kappa h+ w^\kappa \partial_\kappa u^\alpha- w^\alpha \partial_\kappa u^\kappa,
			\\
			\partial_{ \alpha } w^\alpha= -w^\kappa \partial_\kappa h,
		\end{cases}
	\end{equation}
	where the operator $\square_g=g^{\beta \gamma}\partial^2_{\beta \gamma}$, quantities $D$, and $Q^{\alpha}$ are respectively given by
	\begin{equation}\label{err1}
			D=\Omega \left\{ (1-c^2_s)u^\beta \partial_{\beta}u^\kappa  \partial_{\kappa} h +  u^\beta \partial_{\kappa} u^\kappa \partial_{\beta}(c^2_s)
			- c^2_s  u^\beta \partial_{\kappa} u^\kappa \partial_\beta h
			-c^2_s \partial_\kappa  u^\beta \partial_{\beta}  u^\kappa \right\},
	\end{equation}
	and
		\begin{equation}\label{err}
		\begin{split}
				Q^\alpha=&\Omega \big\{  u^\beta \partial_{\beta} u^\kappa \partial_{\kappa} u^\alpha
			+ u^\beta \partial_{\beta} (m^{\alpha \kappa}+u^{\alpha}u^{\kappa}) \partial_\kappa h-(m^{\alpha \kappa}+ u^{\alpha}u^{\kappa}) \partial_{ \beta} u^\beta \partial_{\kappa} (c^2_s)
			\\
			& -(m^{\alpha \kappa}+ u^{\alpha}u^{\kappa}) \partial_{\kappa} u^\beta  \partial_{\beta} h+c^2_s \epsilon^{\alpha \beta \gamma \delta}\mathrm{e}^{-h} w_\delta \partial_{ \beta} u_\gamma
			+
			c^2_s \partial_\beta u^\beta \partial^\alpha h
			-c^2_s \partial^\beta u^\alpha \partial_\beta h
			\\
			& +c^2_s u^{\alpha} u^\kappa \partial_{\beta} u^\beta  \partial_{\kappa} h
			+c^2_s u^{\alpha} \partial_{\beta}  u^{\kappa} \partial_{\kappa} u^\beta
			-c^2_s u^\beta \partial_{\beta} u^\kappa \partial_{\kappa} u^\alpha
			- c^2_s  u^\beta u^\kappa \partial_{\beta} u^\alpha \partial_{\kappa} h \big\}.
		\end{split}
	\end{equation}
\end{Lemma}
\begin{remark}
The system \eqref{WTe} is first proposed by Disconzi-Speck \cite{DS} (please see Theorem 3.1 and set $S=0$). We also give a derivation for it, i.e. Lemma \ref{WT} in Appendix \ref{appa}.
\end{remark}
\begin{remark}
	For brevity, we set $\bQ=(Q^0,Q^1,Q^2,Q^3)^{\mathrm{T}}$. By \eqref{err1} and \eqref{err}, we can also write $Q^\alpha=Q_{1 \kappa }^{\alpha \beta \gamma } \partial_\beta h \partial_\gamma u^\kappa+Q_{2\kappa  \delta}^{\alpha \beta \gamma } \partial_\beta u^\kappa \partial_\gamma u^\delta +Q_{3}^{\alpha \beta \gamma } \partial_\beta h \partial_\gamma h $, $D=D_{1 \kappa }^{\beta \gamma } \partial_\beta h \partial_\gamma u^\kappa+D_{2\kappa  \delta}^{ \beta \gamma } \partial_\beta u^\kappa \partial_\gamma u^\delta +D_{3}^{ \beta \gamma } \partial_\beta h \partial_\gamma h $. Here $Q_{1 \kappa }^{\alpha \beta \gamma }$, $Q_{2\kappa  \delta}^{\alpha \beta \gamma }$, $Q_{3}^{\alpha \beta \gamma }$, $D_{1 \kappa }^{\beta \gamma }$, $D_{2\kappa  \delta}^{ \beta \gamma }$, and $D_{3}^{ \beta \gamma }$ are smooth functions of $\bu$ and $h$, and they are uniquely determined by \eqref{err}.
\end{remark}
\begin{remark}
	Under the normalized condition \eqref{mo6} and \eqref{mo9}, the system \eqref{WTe} is derived from \eqref{REE}, with \eqref{VVd}-\eqref{MFd}, and using this convention, the  systems \eqref{OREE}, \eqref{REE}, and \eqref{WTe}, as well as \eqref{QHl} below, are equivalent.
\end{remark}
Seeing \eqref{WTe}, we also need a decomposition for the velocity as follows.
\begin{definition}
	Let $\bu$ be the relativistic velocity and $h$ the logarithmic enthalpy. Let $\bw$ and $\bW$ be defined in \eqref{VVd} and \eqref{MFd}. Set $\mathrm{I}$ be the identity operator. Introduce a composition of $\bu$ by
	\begin{equation}\label{De}
		u^\alpha=u^\alpha_{+}+u^\alpha_{-},
	\end{equation}
	where $\bu_{-}=(u_{-}^0,u_{-}^1,u_{-}^2,u_{-}^3)^{\mathrm{T}}$ is denoted by
	\begin{equation}\label{De0}
		\begin{split}
			\mathbf{P} u_{-}^\alpha& = \mathrm{e}^{-h}W^\alpha, \quad \mathbf{P}=\mathrm{I}-(m^{\beta \gamma}+2u^{\beta}u^{\gamma}) \partial^2_{\beta \gamma}.
		\end{split}
	\end{equation}
\end{definition}
\begin{remark}
	\begin{enumerate}
		\item
		By \eqref{De0} and \eqref{muu}, then $\mathbf{P}$ is a space-time elliptic operator on $[0,T]\times \mathbb{R}^{3}$. For any function $f$ defined on $[0,T]\times \mathbb{R}^{3}$, by a simple extension technique, we can derive that
		\begin{equation}\label{ellip}
			\| f \|_{L_{[0,T]}^2 H_x^a} \lesssim \|\mathbf{P} f \|_{L_{[0,T]}^2 H_x^{a-2}}, \quad a \in \mathbb{R}.
		\end{equation}
		This will be proved in Lemma \ref{app1} below. Therefore, we can record $\mathbf{P}^{-1}$ as the inverse operator for $\mathbf{P}$, and also
		\begin{equation}\label{uf}
			\begin{split}
				u^\alpha_{-}=& \mathbf{P}^{-1}(\mathrm{e}^{-h}W^\alpha).
			\end{split}
		\end{equation}
		\item This construction for $\mathbf{P}$ is new in the relativistic Euler equations. For non-relativistic compressible Euler equations, referring the work \cite{WQEuler, AZ}, $\mathbf{P}$ is the space-elliptic operator $\mathrm{I}-\Delta$ ($\Delta=\partial_1^2+\partial_2^2+\partial_3^2$). While, in relativistic Euler equations, if we take $\mathbf{P}$ as a space-elliptic operator, it doesn't work. We make a short explanation as follows.

		If $\mathbf{P}$ is a space-elliptic elliptic operator, then
		\begin{equation*}
			\begin{split}
				\square_g \bu_{+}=& \square_{{g}} \bu - \square_g \bu_{-}
				\\
				=& \underbrace{- c_s^2\Omega \mathrm{e}^{-h}W^\alpha-\Omega c_s^2 m^{\alpha \beta} \partial^2_{ \alpha \beta} \bu_{-}}_{= \psi}
				-\Omega ( c_s^2 -1 ) u^\alpha u^\beta \partial^2_{ \alpha \beta} \bu_{-}+ \mathrm{l.o.t.},
			\end{split}
		\end{equation*} 
		For $\bu \in H^{s_0+1}$, then $d^2\bu \in H^{s_0-1}$ and $\bu_{+}\in H^{s_0}$. This implies $\bu \in H^{s_0}$, which has $(s-s_0)$-order loss of derivatives. 
		
		If $\mathbf{P}$ is defined as in \eqref{De0}, then $\psi\approx u^\alpha u^\beta \partial^2_{ \alpha \beta} \bu_{-}$. Moreover, we can calculate (see Lemma \ref{um} below) 
		\begin{equation*}
			u^\alpha u^\beta \partial^2_{ \alpha \beta} \bu_{-}\approx g^{0\alpha} \partial_{ \alpha } \mathbf{T} \bu_{-} + \nabla \mathbf{T} \bu_{-}+ \mathrm{l.o.t.}
		\end{equation*}
		By Duhamel's principles, we only need to check if $\mathbf{T} \bu_{-}\in L^2_t H^s_x$ holds or not. Very fortunately, combining \eqref{De0}-\eqref{ellip}, together with the transport equation for $\bW$, we can prove $\mathbf{T} \bu_{-}\in L^2_t H^s_x$. Therefore, we have $\bu_{+} \in H^s$, and then $\bu \in H^s$. In a word, there is no loss of derivatives if we define $\mathbf{P}$ as in \eqref{De0}.
		
	\end{enumerate}
\end{remark}
Our goal in this paper is to give a complete well-posedness result of rough solutions for the Cauchy problem of \eqref{REE}, and answer the open question "Problem D" proposed by Disconzi \cite{Dis}. Set the initial data
\begin{equation}\label{REEi}
	(h,\bu)|_{t=0}=(h_0,\bu_0), \qquad  \bu_0=(u_{0}^0,u_{0}^1,u_{0}^2,u_{0}^3)^\mathrm{T}.
\end{equation}
By \eqref{muu}, $\bu_0$ should satisfy
\begin{equation}\label{REEia}
	u^\alpha_0 u_{0\alpha}=-1,
\end{equation}
where
\begin{equation}\label{REEib}
	u_{0\alpha}=m_{\alpha \beta}u^\beta_0.
\end{equation}
\subsection{History and motivation}
In the non-relativistic case, namely, compressible Euler equations, the local well-posedness result began with Majda's work \cite{M}. In \cite{M}, the initial velocity, density and entropy belong to $H^{s}(\mathbb{R}^n), s>1+\frac{n}{2}$ and the density bounded away from vacuum. If the flow is isentropic and irrotational, the compressible Euler equations can be written as a quasilinear wave system. A general quasilinear wave equation takes the form
\begin{equation}\label{qwe}
	\begin{cases}
		&\square_{h(\phi)} \phi=q(d \phi, d \phi), \quad (t,x) \in \mathbb{R}^+ \times \mathbb{R}^{n},
		\\
		& (\phi, \partial_t \phi)|_{t=0}=(\phi_0, \phi_1)\in H^s(\mathbb{R}^n) \times H^{s-1}(\mathbb{R}^n),
	\end{cases}
\end{equation}
where $\phi$ is a scalar function, $h(\phi)$ is a Lorentzian metric depending on $\phi$, $d=(\partial_t, \partial_1, \partial_2, \cdots, \partial_n)$, and $q$ is quadratic in $d \phi$. 
According to classical energy methods and Sobolev imbeddings, Hughes-Kato-Marsden \cite{HKM} proved the local well-posedness of the problem \eqref{qwe} for $ s>\frac{n}{2}+1$. On the other side, Lindblad \cite{L} constructed some counterexamples for \eqref{qwe} when $s=2,n=3$. For $s={\frac{7}{4}}, n=2$, a 2D counterexample was constructed by Ohlman \cite{Oman}. There is a gap between the result \cite{HKM} and \cite{L,Oman}. To lower the regularity of the initial data, 
a classical idea is to prove Strichartz estimates of solutions. There are several steps to obtain the sharp Strichartz estimates for \eqref{qwe}. The first step is to consider the the wave equation with variable coefficients
\begin{equation}\label{qw0}
	\square_{h(t,x)}\phi=0,
\end{equation}
and then exploit it to obtain the low regularity solutions of \eqref{qwe}. For smooth coefficients $h$, Kapitanskij \cite{Kap} and Mockenhaupt-Seeger-Sogge \cite{MSS} both proved the Strichartz estimates for \eqref{qw0}. For rough coefficients $h$, Smith \cite{Sm} firstly proved the Strichartz eatimate for $h \in C^2$ in two or three dimensions. Inspired by \cite{Sm}, Tataru \cite{T2} established the Strichartz estimates in all dimensions. However, for $h \in C^\alpha$ ($0<\alpha < 2$), the Strichartz estimates fails, one can see the counterexample constructed by Smith-Sogge \cite{SS}.

The second step was independently achieved by Bahouri-Chemin \cite{BC2} and Tataru \cite{T1}, who established the local well-posedness of \eqref{qwe} with $s > \frac{n}{2} + \frac{7}{8}, n=2$ or $s > \frac{n}{2} + \frac{3}{4}, n\geq3$. Shortly afterward, Tataru \cite{T3} relaxed the Sobolev indices $s>\frac{n+1}{2}+\frac{1}{6}, n \geq 3$. Additionally, Smith-Tataru \cite{ST0} showed that the $\frac{1}{6}$ loss is sharp for general variable coefficients $h$. Thus, to improve the above results, one needs to exploit a new way or structure of Equation \eqref{qwe}.

By finding a good decomposition of curvature and vector-field approach, Klainerman-Rodnianski \cite{KR2} realized the next key progress in 3D rough solutions of \eqref{qwe} with $s>2+\frac{2-\sqrt{3}}{2}$.
Based on \cite{KR2}, Geba \cite{Geba} also studied the local well-posedness of \eqref{qwe} in two dimensions for $s > \frac{7}{4} + \frac{5-\sqrt{22}}{4}$. By using wave packets of localization to represent solutions of a linear equation, a sharp result was proved by Smith-Tataru \cite{ST}, who established the local well-posedness of \eqref{qwe} if $s>\frac{7}{4}, n=2$ or $s>2, n=3$ or $s>\frac{n+1}{2}, 4 \leq n \leq 5$. An alternative proof of the 3D case in \cite{ST} was also obtained through a vector-field approach by Wang \cite{WQSharp}.


There are also other two important quasilinear wave models, namely, Einstein vacuum and time-like minimal surface equations. In the aspect of Einstein vacuum equations, the results depend on the choice of gauge. In Yang-Mills gauge, Klainerman-Rodnianski-Szeftel \cite{KR1} solved the well-known $L^2$ conjecture. In wave gauge, the equation is well-posed in $H^{2+}$ and ill-posed in $H^{2}$, cf. Klainerman-Rodnianski \cite{KR} and Ettinger-Lindblad \cite{EL}. In CMC guage, the system is well-posed in $H^{2+}$, please refer Andersson-Moncrief \cite{AM} and Q. Wang \cite{WQRough}. For the time-like minimal surface equations, which satisfy null conditions, Ai-Ifrim-Tataru \cite{AIT} obtained a substantial improvement, namely by $\frac18$ derivatives in two space dimensions and by $\frac14$ derivatives in higher dimensions.

If the flow is rotational or isentropic, the compressible Euler equations can be written as a coupled wave-transport system, which is proposed by Luk-Speck \cite{LS2}. Based on Luk-Speck's work, the first study of rough solutions was independently obtained by  Disconzi-Luo-Mazzone-Speck \cite{DLS} and Wang \cite{WQEuler}. Disconzi-Luo-Mazzone-Speck \cite{DLS} proved the existence and uniqueness of solutions of compressible Euler equations if\footnote{Here $S_0$ is the initial entropy.} $(\bv_0,\rho_0,\bw_0,S_0)\in H^{2+}(\mathbb{R}^3) \times H^{2+}(\mathbb{R}^3) \times H^{2+}(\mathbb{R}^3) \times H^{3+}(\mathbb{R}^3)$ with an additional H\"older condition $\mathrm{curl} \bw_0, \Delta S_0\in C^{\sigma}, 0<\sigma<1$. Wang \cite{WQEuler} established the existence and uniqueness of solutions of isentropic compressible Euler equations if $(\bv_0,\rho_0,\bw_0) \in H^{s}(\mathbb{R}^3) \times H^{s}(\mathbb{R}^3) \times H^{s_0}(\mathbb{R}^3)$, where $2<s_0<s$. The work \cite{DLS} and \cite{WQEuler} both are based on classical vector-fields approach. In a different view, using Smith-Tataru's method \cite{ST}, together with semi-classical analysis and Tao's frequency envelope, Andersson-Zhang \cite{AZ} proved a complete local well-posedness of 3D compressible Euler equations with relaxed assumptions $(\bv_0,\rho_0,\bw_0)\in H^{2+}(\mathbb{R}^3) \times H^{2+}(\mathbb{R}^3) \times H^{2}(\mathbb{R}^3)$ or $(\bv_0,\rho_0,\bw_0,S_0)\in H^{\frac52}(\mathbb{R}^3) \times H^{\frac52}(\mathbb{R}^3) \times H^{\frac32+}(\mathbb{R}^3)\times H^{\frac52+}(\mathbb{R}^3)$. We should also mention some results in two dimensions. Zhang \cite{Z1,Z2} proved the well-posedness of solutions of 2D isentropic compressible Euler equations when $(\bv_0, \rho_0, \bw_0) \in H^{\frac74+}(\mathbb{R}^2), \nabla \bw_0 \in L^\infty(\mathbb{R}^2) $ or $(\bv_0, \rho_0, \bw_0) \in H^{\frac74+}(\mathbb{R}^2) \times H^{\frac74+}(\mathbb{R}^2) \times H^2(\mathbb{R}^2)$. In the reverse direction, there are two ill-posedness results if $(\bv_0, \rho_0) \in \dot{H}^{2}(\mathbb{R}^3)$ or $(\bv_0, \rho_0) \in \dot{H}^{\frac74}(\mathbb{R}^2)$ with smooth vorticity, please refer the joint works by An-Chen-Yin \cite{ACY, ACY1}.

In the study of relativistic Euler equations, Makino-Ukai \cite{MU1, MU2} proved the Cauchy problem of local well-posedness of the 3D relativistic Euler equations when $s\geq 3$ by studying the symmetric hyperbolic structure. For an overview of the standard theory, we refer to the relevant chapter in Choquet-Bruhat's book \cite{Bru}. With vacuum, the pioneering result was obtained by Rendall \cite{Ren} and Guo-Tahvildar-Zadeh \cite{GT}. By proposing a new symmetrization, LeFloch-Ukai \cite{LU} established a local-in-time existence result for solutions in $n$-D containing vacuum states when $s>1+\frac{n}{2}$. Recently, Disconzi-Speck \cite{DS} proposed a new formulation of covariant wave equations coupled to transport equations and two transport-div-curl systems for the 3D relativistic Euler equations. On the basis of \cite{DLS} and \cite{DS}, Yu \cite{Yu} proved the existence and uniqueness of rough solutions of 3D relativistic Euler equations if $(\bu_0,h_0,\bw_0,S_0)\in H^{2+} \times H^{2+} \times H^{2+}\times H^{3+}$, and $\mathrm{vort} \bw_0, \Delta S_0 \in C^\sigma (0<\sigma<1)$. Very recently, Disconzi \cite{Dis} proposed some open problems, one is about establishing the well-posedness of low regularity solutions for the 3D relativistic Euler equations with $(h_0,\bu_0,\bw_0) \in H^s \times H^s \times H^{s_0} (2<s_0<s)$. We note that there are also some well-posed results on free boundary problems due to Jang-LeFloch- Masmoudi \cite{JLM}, Had$\check{\mathrm{z}}$i$\acute{\mathrm{c}}$-Shkoller-Speck \cite{HSS}, Disconzi-Ifrim-Tataru \cite{DIT}, and Miao-Shahshahani \cite{MS} and so on.

Strongly motivated by these insightful works, we further study the Hadamard well-posedness of rough solutions for 3D relativistic Euler equations, and give a definite answer to the open question "Problem D" in \cite{Dis}. Firstly, we prove the local well-posedness of solutions for $(h_0,{\bu}_0,\bw_0)\in H^{2+} \times H^{2+} \times H^{s_0}$ ($2<s_0<s$). In the proof, we also obtain: (i) a Strichartz estimate of linear wave equations endowed with the acoustic metric, (ii) a type of Strichartz estimate for solutions with a regularity of the velocity $2+$, enthalpy $2+$, and modified vorticity $\frac{3}{2}+$. Secondly, taking $(h_0, {\bu}_0, \bw_0) \in H^{2+} \times H^{2+} \times H^{2}$, then a stronger Strichartz estimates on a short time-interval hold. By summing up these short time intervals to a regular time-interval, a Strichartz estimate with loss of derivatives can be obtained\footnote{This argument is introduced by Ai-Ifrim-Tataru \cite{AIT} in the study of time-like minimal surfaces.}. So we prove the local well-posedness of solutions of 3D relativistic Euler equations if $(h_0, {\bu}_0, \bw_0) \in H^{2+} \times H^{2+} \times H^{2}$. 
\subsection{Statement of result}
Before state our results precisely, let us introduce some notations and functional spaces as follows. Let $\nabla=(\partial_{1}, \partial_{2}, \partial_{3})^\mathrm{T}$, $d=(\partial_t, \partial_{1}, \partial_{2}, \partial_{3})^\mathrm{T}$, $\left< \xi \right>=(1+|\xi|^2)^{\frac{1}{2}}, \ \xi \in \mathbb{R}^3$. Denote by $\left< \nabla \right>$ the corresponding Bessel potential multiplier. Define the operator
\begin{equation}\label{opt}
	\mathbf{T}=\frac{1}{u^0}u^\kappa \partial_\kappa=\partial_t+ \frac{u^i}{u^0}\partial_i.
\end{equation}
and
\begin{equation*}
	\Delta=\partial^2_{1}+\partial^2_{2}+\partial^2_{3}, \quad \Lambda_x=(-\Delta)^{\frac{1}{2}}.
\end{equation*}

Let $\zeta$ be a smooth function with support in the shell $\{ \xi\in \mathbb{R}^3: \frac{1}{2} \leq |\xi| \leq 2 \}$. Here, $\xi$ denotes the variable of the spatial Fourier transform. Let $\zeta$ also satisfy the
condition $\sum_{k \in \mathbb{Z}} \zeta(2^{k}\xi)=1$. Let ${P}_j$ be the Littlewood-Paley operator with frequency $2^j, j \in \mathbb{Z}$ (cf. Bahouri-Chemin-Danchin \cite{BCD}, page 78),
\begin{equation}\label{Dej}
	{P}_j f = \int_{\mathbb{R}^3} \mathrm{e}^{-\mathrm{i}x\cdot \xi} \zeta(2^{-j}\xi) \hat{f}(\xi)d\xi.
\end{equation}
For $f \in H^s(\mathbb{R}^3)$, let
\begin{align*}
	\|f\|_{H^s}= \|f\|_{L^2(\mathbb{R}^3)}+\|f\|_{\dot{H}^s(\mathbb{R}^3)},
\end{align*}
with the homogeneous norm $\|f\|^2_{\dot{H}^s} = {\sum_{j \geq -1}} 2^{2js}\|{P}_j f\|^2_{L^2(\mathbb{R}^3)} $. We shall also make use of the homogenous Besov norm
\begin{align*}
	\|f\|^r_{\dot{B}^s_{p,r}} = {\sum_{j \geq -1}}2^{jsr}\|{P}_j f\|^r_{L^p(\mathbb{R}^3)}.
\end{align*}
To avoid confusion, when the function $f$ is related to both time and space variables, we use the notation $\|f\|_{H_x^s}=\|f(t,\cdot)\|_{H^s(\mathbb{R}^3)}$.

Assume there holds\footnote{In the paper, $c_s|_{t=0}>c_0$ represents that it's far away from vacuum. $c_s<1$ is from \eqref{mo7}.}
\begin{equation}\label{HEw}
	|h_0,u_0^0-1,\mathring{\bu}_0 | \leq C_0, \qquad u_0^0 \geq 1, \qquad c_0<c_s|_{t=0}<1 .
\end{equation}
where $C_0, c_0 > 0 $ are positive constants. In the following, constants $C$ depending only on $C_0, c_0$ shall be called universal. Unless otherwise stated, all constants that appear are universal in this sense.
The notation $X \lesssim Y$ means $X \leq CY$, where $C$ is a universal constant, possibly depending on $C_0, c_0$. Similarly, we write  $X \simeq Y$ when $C_1 Y \leq X \leq C_2Y$, with $C_1$ and $C_2$ universal constants, and $X \ll Y$ when $X \leq CY$ for a sufficiently large constant $C$. The universal constant may change from line to line.

Let $s_0$ and $s$ satisfy $2< s_0 < s\leq \frac{5}{2}$.
We also set
\begin{equation}\label{a1}
 \delta\in (0, s-2),\quad \delta_0 \in (0, s_0-2), \quad \delta_1=\frac{\sstar-2}{40}.
\end{equation}
We use three small parameters
\begin{equation}\label{a0}
	0 < 	\epsilon_3 \ll \epsilon_2 \ll \epsilon_1 \ll \epsilon_0 \ll 1.
\end{equation}

Our first result is stated as follows.
\begin{theorem}\label{dingli}
	Let $2< s_0 < s\leq \frac{5}{2}$.  Consider the Cauchy problem \eqref{REE}-\eqref{REEi}. Let $\bw$ be defined in \eqref{VVd}.
	For any given initial data $(h_0, \bu_0, \bw_0)$ satisfies \eqref{REEia}, \eqref{HEw}, and for any $M_0>0$ such that\footnote{Throughout our paper, $\bu \in H^s_x$ means $\bu-(1,0,0,0)^{\mathrm{T}} \in H^s_x$.}
	\begin{equation}\label{chuzhi1}
		\| (h_0,{\bu}_0)\|_{H^{s}}  + \| \bw_0\|_{H^{s_0}}
		\leq M_0,
	\end{equation}
	where
		\begin{equation}\label{A01}
	\bw_0=\mathrm{vort}(\mathrm{e}^{h_0} \bu_0),
	\end{equation}
	there exist positive constants $T>0$
	and $M_1>0$ ($T$ and ${M}_1$ depends on $C_0, c_0, s, s_0, {M}_0$) such that \eqref{REE}-\eqref{REEi} has a unique solution $(h, \bu,\bw)$ on $[0,T]\times \mathbb{R}^3$. Moreover,
	the following statements hold.
	\begin{enumerate}
		\item \label{point:1}
		The solutions  $(h, \bu, \bw)$ satisfy the energy estimate
		\begin{equation}\label{A02}
			\begin{split}
				&\|( h, \bu )\|_{L^\infty_{[0,T]} H_x^s}+ \|\bw\|_{L^\infty_{[0,T]} H_x^{s_0}} \leq M_1,
				\\
				& \|(\partial_t {\bu}, \partial_t h)\|_{L^\infty_{[0,T]}H_x^{s-1}}+\|\partial_t\bw\|_{L^\infty_{[0,T]}H_x^{s_0-1}} \leq M_1,
			\end{split}
		\end{equation}
		and
		\begin{equation*}
		\|h,u^0-1,\mathring{\bu}\|_{L^\infty_{ [0,{T}] \times \mathbb{R}^3}} \leq 1+C_0,  \qquad u^0 \geq 1.
		\end{equation*}
		\item \label{point:2}
		The solution also satisfies the Strichartz estimate
		\begin{equation}\label{SSr}
			\|(d\bu, dh, d\bu_{+})\|_{L^2_{[0,T]}L_x^\infty}+ \|d\bu, dh\|_{L^2_{[0,T]} \dot{B}^{s_0-2}_{\infty,2}}+ \|\nabla \bu_{+}\|_{L^2_{[0,T]} \dot{B}^{s_0-2}_{\infty,2}} \leq M_1,
		\end{equation}
		where $\bu_+$ is denoted in \eqref{De}-\eqref{De0}.
		\item \label{point:3}
		For any $1 \leq r \leq s+1$, and for each $t_0 \in [0,T)$, the linear equation
		\begin{equation}\label{linear}
			\begin{cases}
				& \square_g f=0, \qquad (t,x) \in (t_0,T]\times \mathbb{R}^3,
				\\
				&(f,\partial_t f)|_{t=t_0}=(f_0,f_1) \in H^r(\mathbb{R}^3) \times H^{r-1}(\mathbb{R}^3),
			\end{cases}
		\end{equation}
		admits a solution $f \in C([0,T],H_x^r) \times C^1([0,T],H_x^{r-1})$ and the following estimates hold:
		\begin{equation*}
			\| f\|_{L_{[0,T]}^\infty H_x^r}+ \|\partial_t f\|_{L_{[0,T]}^\infty H_x^{r-1}} \leq  C_{M_0}( \|f_0\|_{H_x^r}+ \|f_1\|_{H_x^{r-1}} ).
		\end{equation*}
		Additionally, the following estimates hold, provided $k<r-1$,
		\begin{equation}\label{SE1}
			\| \left<\nabla \right>^k f\|_{L^2_{[0,T]}L^\infty_x} \leq  C_{M_0}( \|f_0\|_{H_x^r}+ \|f_1\|_{H_x^{r-1}} ),
		\end{equation}
		and the same estimates hold with $\left< \nabla \right>^k$ replaced by $\left< \nabla \right>^{k-1}d$. Here, $C_{M_0}$ is a constant depending on $C_0, c_0, s, s_0$ and $M_0$.
		\item \label{point:4}
		the map from $(h_0,{\bu}_0,\bw_0) \in H^s \times H^s \times H^{s_0}$  to $(h,{\bu},\bw)(t,\cdot) \in C([0,T];H_x^s \times H_x^s \times H_x^{s_0})$ is continuous.
	\end{enumerate}
\end{theorem}
\begin{remark}
	Compared with Yu's result \cite{Yu}, there is no H\"older condition on $\mathrm{vort}\bw$ in Theorem \ref{dingli}. Moreover, We provide the proof of Hadamard well-posedness for the relativistic Euler equations with rough data.
\end{remark}
Our second result is as follows.
\begin{theorem}\label{dingli2}
	Let $2<\sstar<\frac52$. Let $\bw$ be defined in \eqref{VVd}. Consider the Cauchy problem \eqref{REE}-\eqref{REEi}. Assume that \eqref{HEw} holds. For any $M_*>0$, if
	\begin{equation}\label{chuzhi3}
		\|h_0\|_{H_x^{\sstar}} + \|\bu_0\|_{H_x^{\sstar}}+ \|\bw_0\|_{H_x^{2}}
		\leq M_*,
	\end{equation}
	where
	\begin{equation}\label{A04}
	\bw_0=\mathrm{vort}(\mathrm{e}^{h_0} \bu_0),
	\end{equation}
	then there exist positive constants $T^*>0$ and ${M}_{2}>0$ ($T^{*}$ and ${M}_{2}$ depends on $C_0, c_0, s, M_{*}$) such that \eqref{REE}-\eqref{REEi} has a unique solution $(h,\bu,\bw)$ on $[0,T]\times \mathbb{R}^3$. To be precise,
	\begin{enumerate}
		\item \label{point:d3:1}		
		the solution $(h, \bu, \bw)$ satisfies the energy estimate
		\begin{equation}\label{A03}
			\begin{split}
				&\|(h,{\bu})\|_{L^\infty_{[0,T^*]}H_x^s}+ \|\bw\|_{L^\infty_{[0,T^*]}H_x^{2}}  \leq {M}_{2},
				\\
				& \|(\partial_t {\bu}, \partial_t h)\|_{L^\infty_{[0,T^*]}H_x^{s-1}}+\|\partial_t\bw\|_{L^\infty_{[0,T^*]}H_x^{1}} \leq {M}_{2},
			\end{split}
		\end{equation}
		and
		\begin{equation}\label{dA04}
		\|h,u^0-1,\mathring{\bu}, h\|_{L^\infty_{ [0,{T}^*] \times \mathbb{R}^3}} \leq 2+C_0, \qquad u^0 \geq 1.
		\end{equation}
		\item \label{point:d3:2}	
		the solution $(h,\bu)$ satisfies the Strichartz estimate
		\begin{equation}\label{A05}
			\|(d\bu, dh)\|_{L^2_{ [0,{T}^*]} L_x^\infty} \leq {M}_{2}.
		\end{equation}	
		\item \label{point:d3:3}
		for any $\frac{s}{2} \leq r \leq 3$, and for each $t_0 \in [0,T^*]$, the linear wave equation \eqref{linear} has a unique solution on $[0,T^*]$. Moreover, for $a\leq r-\frac{s}{2}$, the solution satisfies
		\begin{equation}\label{tu02}
			\begin{split}
				&\|\left< \nabla \right>^{a-1} d{f}\|_{L^2_{[0,T^*]} L^\infty_x}
				\leq  {M}_3 (\|{f}_0\|_{{H}_x^r}+ \|{f}_1 \|_{{H}_x^{r-1}}),
				\\
				&\|{f}\|_{L^\infty_{[0,T^*]} H^{r}_x}+ \|\partial_t {f}\|_{L^\infty_{[0,T^*]} H^{r-1}_x} \leq {M}_3 (\| {f}_0\|_{H_x^r}+ \| {f}_1\|_{H_x^{r-1}}).
			\end{split}
		\end{equation}
		Here ${M}_3>0$ depends on $C_0, c_0, s, {M}_*$.
		\item \label{point:d3:4}
		the map from $(h_0, \bu_0, \bw_0) \in H^s \times H^s \times H^{2}$  to $(h, \bu,\bw)(t,\cdot) \in C([0,T^*];H_x^s \times H_x^s \times H_x^{2})$ is continuous.
	\end{enumerate}
\end{theorem}
\begin{remark}
	\begin{enumerate}
			\item
	Controlling the regularity of velocity and density, some good structure of \eqref{REE} may allow us lowering the regularity of the vorticity. In Theorem \ref{dingli2}, the regularity of vorticity is $2$. So it provides us $(s_0-2)$-order ($s_0>2$) regularity improvement of the vorticity compared with Theorem \ref{dingli}. 
		\item  Whether the results in Theorem \ref{dingli}, and \ref{dingli2} are sharp remains a challenging problem. We may expect that the local well-posedness of solutions for \eqref{REE} also hold when $(h_0, \bu_0, \bw_0) \in H^{2+} \times H^{2+} \times H^{\frac32+}$.
	\end{enumerate}
\end{remark}

\begin{remark}
	\begin{enumerate}
			\item
		By comparison, \eqref{SE1} is better than \eqref{tu02}. There is a Strichartz estimate with a loss of derivatives  in \eqref{SE1}, which is proved by extending the short (semi-classical) time-interval to a regular time-interval by summing up estimates on these short time-intervals. The idea is introduced by Ai-Ifrim-Tataru \cite{AIT} in time-like minimal surfaces, and developed by Andersson-Zhang \cite{AZ} in 3D compressible Euler equations.
		\item The estimate \eqref{SE1} is a key observation to prove the continuous dependence of solutions, as well as Theorem \ref{dingli2}.
	\end{enumerate}
\end{remark}

\subsection{Strategy of the proof} We provide a strategy of proofs
for Theorem \ref{dingli} and Theorem \ref{dingli2}.
\subsubsection{Outline proof of Theorem \ref{dingli}} 
The main idea is as follows.

\textit{$\bullet$ Step 1: Energy estimates}. For $\bu$ and $h$, we use the hyperbolic system to get its energy(see Theorem \ref{VHE}) 
\begin{equation*}
	\| (h,\bu) \|_{H^s_x} \lesssim \| (h_0, \bu_0) \|_{H^s} \exp\left( \int^t_0 \| dh,d\bu \|_{L^\infty_x} \right).
\end{equation*} 
To bound $\|\bw \|_{H^{s_0}_x}$, we only need to consider $\|\bw\|_{L^2_x}$ and $\|\bw\|_{\dot{H}^{s_0}_x}$. Recalling the second equation in System \eqref{WTe}, by using classical method, we therefore get
\begin{equation*}
	\| \bw \|_{L^2_x} \lesssim \| \bw_0 \|_{L^2} \exp\left( \int^t_0 \| dh,d\bu \|_{L^\infty_x} \right).
\end{equation*} 
However, the bound of $\|\bw\|_{\dot{H}^{s_0}_x}$ is much difficult, for there are two terms $\nabla^3 \bu$ and $\nabla^3 h$ if we operate spatial second-order derivatives on the third equation in \eqref{WTe}. Fortunately, we find a good decomposition for $\partial^\gamma \partial_\gamma \bu$ (see Lemma \ref{HD} for details):
\begin{equation*}
	\partial^\gamma \partial_\gamma \bu^\alpha \approx \mathrm{vort}^\alpha \bw+ \partial^\alpha (\partial_\gamma u^\gamma)  -u^\kappa \partial_\kappa ( u^\gamma \partial_\gamma u^\alpha- u^\alpha \partial_\gamma u^\gamma ) + (d\bu, dh) \cdot (du, dh),
\end{equation*}
whose role is very similar to Hodge decomposition in non-relativistic fluid dynamics. This allows us to derive a good transport equation for $\bG$ (see Definition \eqref{MFd}):
\begin{equation}\label{i00}
	u^\kappa \partial_\kappa (G^\alpha-F^\alpha)=\partial^\alpha \Gamma+ E^\alpha.
\end{equation}
Above,
\begin{equation*}
	\begin{split}
		& \Gamma \approx d\bu \cdot d\bw,
		\\
		& \bF\approx (d\bu,dh)\cdot d\bw + (d\bu,dh)\cdot (d\bu,dh)\cdot \bw, 
		\\
		& \bE\approx (d\bu,dh)\cdot d^2\bw + (d\bu,dh)\cdot (d\bu,dh)\cdot (d\bu,dh)\cdot \bw
		\\
		& \quad \quad  + (d\bu,dh)\cdot \bw \cdot (d^2\bu,d^2h)+ (d\bu,dh)\cdot (d\bu,dh)\cdot d\bw.
	\end{split}
\end{equation*}
See Lemma \ref{VC} for details. So we transfer our goal to bound $\|\bG\|_{\dot{H}^{s_0-2}_x}$. Since \eqref{i00}, we can see
\begin{equation}\label{i01}
	\frac{d}{dt} \| \Lambda^{s_0-2}_x (\bG-\bF)\|^2_{L^2_x}= \textrm{H}_1+ \textrm{good \ terms},
\end{equation}
where $\textrm{H}_1$ is the most difficult term 
\begin{equation}\label{i02}
	\textrm{H}_1=\int^t_0 \int_{\mathbb{R}^3}   \partial^\alpha \left\{  \Lambda_x^{s_0-2}  \big(  (u^0)^{-1}\Gamma \big) \right\}  \cdot \Lambda_x^{s_0-2} \left(  G_\alpha - F_\alpha  \right)   dxd\tau,
\end{equation}
and "good terms" can be estimated by product and commutator estimates. To calculate $\textrm{H}_1$, we divide \eqref{i02} into two parts
\begin{equation*}
	\textrm{H}_1=\underbrace{ \int_{\mathbb{R}^3} \partial_\alpha (\Lambda_x^{s_0-2} \big( (u^0)^{-1} \Gamma\big))\cdot \Lambda_x^{s_0-2}G^\alpha dx }_{\equiv \mathrm{H}_{11}} \underbrace{-\int_{\mathbb{R}^3} \partial_\alpha (\Lambda_x^{s_0-2} \big( (u^0)^{-1} \Gamma\big))\cdot \Lambda_x^{s_0-2}F^\alpha dx }_{\equiv \mathrm{H}_{12}}.
\end{equation*}
For $\mathrm{H}_{11}$, integrating it parts, which yields
\begin{align}
		\nonumber \mathrm{H}_{11}
		=&\int_{\mathbb{R}^3} \partial_\alpha (\Lambda_x^{s_0-2} \big( (u^0)^{-1} \Gamma\big) \cdot \Lambda_x^{s_0-2}G^\alpha ) dx-  \int_{\mathbb{R}^3}\Lambda_x^{s_0-2} \big( (u^0)^{-1} \Gamma\big) \cdot \Lambda_x^{s_0-2}(\partial_\alpha G^\alpha) dx
		\\
		\label{i06}
		=&\frac{d}{dt}\int_{\mathbb{R}^3} \Lambda_x^{s_0-2} \big( (u^0)^{-1} \Gamma\big)\cdot \Lambda_x^{s_0-2}G^0 dx
		-    \int_{\mathbb{R}^3}\Lambda_x^{s_0-2} \big( (u^0)^{-1} \Gamma\big) \cdot \Lambda_x^{s_0-2}(\partial_\alpha G^\alpha) dx .
\end{align}
For the first term in \eqref{i06}, it's a lower order term compared with the left-hand side of \eqref{i01}. So we can put it into the left-hand side of \eqref{i01}. For the second term in \eqref{i06}, we can bound it by product and commutator estimates.

Similarly, we can calculate $\mathrm{H}_{12}$ by
\begin{equation}\label{i08}
	\begin{split}
		\mathrm{H}_{12}=&	- \frac{d}{dt} \int_{\mathbb{R}^3} \Lambda_x^{s_0-2} \big( (u^0)^{-1} \Gamma\big) \cdot \Lambda_x^{s_0-2}F^0 dx +  \int_{\mathbb{R}^3}  \Lambda_x^{s_0-2} \big( (u^0)^{-1} \Gamma\big) \cdot \Lambda_x^{s_0-2}(\partial_{\alpha} F^{\alpha}) dx.
	\end{split}
\end{equation}
For the first term in \eqref{i08}, it's also a lower order term compared with the left-hand side of \eqref{i01}. So we can also put it into the left-hand side of \eqref{i01}. For the second term in \eqref{i08}, a delicate analysis can lead us to get the desired bounds. Summing up the above outcome, and using Young inequality, Gronwall inequality, we can obtain the following desired energy estimate
\begin{equation}\label{i09}
	E_s(t) \leq E_0  \exp \left(    5 M(t) \cdot \mathrm{e}^{5 M(t)} \right),
\end{equation}
where
\begin{equation*}
	\begin{split}
		& E_0= C \left( \|(h_0,\mathring{\bu}_0)\|^2_{H^s} +\|\bw_0\|^2_{H^{s_0}}+\|(h_0,\mathring{\bu}_0)\|^{10}_{H^s}+\|\bw_0\|^{10}_{H^{s_0}} \right),
		\\
		& M(t)=   {\int^t_0} (\|d\bu, dh\|_{L^\infty_x}+\|d\bu, dh\|_{\dot{B}^{s_0-2}_{\infty,2}})  d\tau.
	\end{split}
\end{equation*}
See Section \ref{sec3.1} for details.

Therefore, the bound of total energy $E_s(t)$ depends on the Strichartz estimate for
\begin{equation}\label{i10}
	{\int^t_0} (\|d\bu, dh\|_{L^\infty_x}+\|d\bu, dh\|_{\dot{B}^{s_0-2}_{\infty,2}})  d\tau.
\end{equation}

\textit{$\bullet$ Step 2: Reduction to a existence result with small, smooth, and supported initial data}. For general large data, it's hard to prove \eqref{i10}. For small data problem, we always have some idea to start. Since the system \eqref{WTe} has finite propagation speed, by a standard compactness method, scaling and localization technique, we can reduce Theorem \ref{dingli} to the following statement. For any smooth, supported data $(h_0, \bu_0, \bw_0)$ satisfying
\begin{equation*}
	\begin{split}
		&\|h_0\|_{H^s} + \|\bu_0 \|_{H^s} + \|\bw_0\|_{H^{s_0}} \leq \epsilon_3,
	\end{split}
\end{equation*}
then the Cauchy problem \eqref{WTe} admits a smooth solution $(h,\bu,\bw)$ on $[-1,1] \times \mathbb{R}^3$, which has the following properties:

$\mathrm{(i)}$ energy estimates
\begin{equation}\label{pre3}
	\begin{split}
		&\|h\|_{L^\infty_{[-1,1]} H_x^{s}}+\| \bu \|_{L^\infty_{[-1,1]} H_x^{s}} + \| \bw\|_{L^\infty_{[-1,1]} H_x^{s_0}}  \leq \epsilon_2.
	\end{split}
\end{equation}

$\mathrm{(ii)}$ dispersive estimate for $h$ and $\bu$
\begin{equation}\label{pre4}
	\|d h, d \bu\|_{L^2_{[-1,1]} C^\delta_x}+\| d h, \nabla \bu_{+}, d \bu\|_{L^2_{[-1,1]} \dot{B}^{s_0-2}_{\infty,2}} \leq \epsilon_2,
\end{equation}

$\mathrm{(iii)}$ dispersive estimate for the linear equation

Let $f$ satisfy
the equation
\begin{equation}\label{pre50}
	\begin{cases}
		& \square_{g} f=0, \qquad (t,x) \in [-1,1]\times \mathbb{R}^3,
		\\
		&f(t_0,\cdot)=f_0 \in H_x^r(\mathbb{R}^3), \quad \partial_t f(t_0,\cdot)=f_1 \in H_x^{r-1}(\mathbb{R}^3).
	\end{cases}
\end{equation}
For each $1 \leq r \leq s+1$, the Cauchy problem \eqref{pre50} is well-posed in $H_x^r \times H_x^{r-1}$, and the following estimate holds:
\begin{equation}\label{pre5}
	\|\left< \nabla \right>^k f\|_{L^2_{[-1,1]} L^\infty_x} \lesssim  \| f_0\|_{H^r}+ \| f_1\|_{H^{r-1}},\quad k<r-1.
\end{equation}
See Section \ref{sec4.1} and Section \ref{sec4.2} for the complete proof.

\textit{$\bullet$ Step 3: Reduction to a bootstrap argument}. When the problem is reduced to small smooth supported initial data $(h_0, \bu_0, \bw_0)$, then we can consider it as a perturbation problem around Minkowski space. Following \cite{ST}, we further reduce the proof of \eqref{pre3}-\eqref{pre5} to: there is a continuous functional $\Re: \mathcal{H} \rightarrow \mathbb{R}^{+}$, satisfying $\Re(0,\mathbf{e},\mathbf{0})=0$ (for $\mathbf{0}$ and $\mathbf{e}$, see \eqref{0e} below), so that for each $(h, \bu, \bw) \in \mathcal{H}$ satisfying $\Re(h, \bu, \bw) \leq 2 \epsilon_1$ the following hold:

\quad $\mathrm{(i)}$ the function $h, \bu$, and $\bw$ satisfies $\Re(h,\bu,\bw) \leq \epsilon_1$;

\quad $\mathrm{(ii)}$ the estimates \eqref{pre3}-\eqref{pre4} and \eqref{pre5} both hold.

By using energy estimates in Theorem \ref{VE}, the estimate \eqref{pre3} follows from \eqref{pre4}. Hence, the proof of $\Re(h, \bu, \bw) \leq \epsilon_1$ and \eqref{pre4}-\eqref{pre5} are the core. See Section \ref{sec4.3} for details.

{\textit{$\bullet$ Step 4: Definition of $\Re$ and regularity of characteristic null surface}.}

Following \cite{ST}, for $\theta \in \mathbb{S}^2$, $\bx \in \mathbb{R}^3$, and let $\Sigma_{\theta, r}$ be the flowout of the set $\theta \cdot \bx = r-2$ along the null geodesic flow with respect of $\mathbf{g}$ in the direction $\theta$ at $t=-2$. Here $\mathbf{g}$ is a localization of $g$ (see in \eqref{AMd3}). Let $\bx'_{\theta}$ be given orthonormal coordinates on the hyperplane in $\mathbb{R}^3$ perpendicular to $\theta$. Then, $\Sigma_{\theta,r}$ is of the form
\begin{equation*}
\Sigma_{\theta,r}=\left\{ (t,\bx): \theta\cdot \bx-\phi_{\theta, r}=0  \right\}
\end{equation*}
for a smooth function $\phi_{\theta, r}(t,\bx'_{\theta})$.
For a given $\theta$, the family $\{\Sigma_{\theta,r}, \ r \in \mathbb{R}\}$ defines a foliation of $[-2,2]\times \mathbb{R}^3$, by characteristic hypersurfaces with respect to $\mathbf{g}$.

We define (see \eqref{500} for details)
\begin{equation}\label{dG}
\Re= \sup_{\theta, r} \vert\kern-0.25ex\vert\kern-0.25ex\vert d \phi_{\theta,r}-dt\vert\kern-0.25ex\vert\kern-0.25ex\vert_{s_0,2,{\Sigma_{\theta,r}}}.
\end{equation}
For simplicity, let us suppose $\theta=(0,0,1)$ and $r=0$. In this case, we set $\Sigma=\Sigma_{\theta,r}$ and $\phi=\phi_{\theta,r}$. To prove $\Re \leq \epsilon_1$ and catch the geometry properties of $\Sigma$, let us introduce a null frame along $\Sigma$ as follows. Let
\begin{equation*}
V=(dr)^*,
\end{equation*}
where $r$ is the foliation of $\Sigma$, and where $*$ denotes the identification of covectors and vectors induced by $\mathbf{g}$. Define a null frame with respect to $\mathbf{g}$ by setting
\begin{equation*}
l=\left< dt,dx_3-d\phi\right>^{-1}_{\mathbf{g}} \left( dx_3-d \phi \right)^*, \quad \underline{l}=l+2\partial_t,
\end{equation*}
and choosing orthonormalized vector fields tangent to the fixed-time slice $\Sigma^t$ of $\Sigma$.

Based on this null frame, the bound of $\Re$ relies on the connection coefficients on $\Sigma$. Based on the well-known Raychaudhuri equation\footnote{It is introduced by Christodoulou-Klainerman \cite{CK} and Klainerman \cite{K}.}, together with the bound of $\vert\kern-0.25ex\vert\kern-0.25ex\vert (h, \bu )\vert\kern-0.25ex\vert\kern-0.25ex\vert_{s_0,2,{\Sigma}}$ and $\vert\kern-0.25ex\vert\kern-0.25ex\vert \bW  \vert\kern-0.25ex\vert\kern-0.25ex\vert_{s_0-1,2,{\Sigma}}$ (see Corollary \ref{vte} and Lemma \ref{te20}), then $\Re\lesssim \epsilon_2$ is established in Proposition \ref{r2}. See Section \ref{rch} for details.

\textit{$\bullet$ Step 5: Strichartz estimates}. In this step, our goal is to derive the Strichartz estimate \eqref{pre4}. A classical idea is to use the Strichartz estimate of linear wave and Duhamel's principle. Only using System \eqref{WTe}, the classical Duhamel's principle, and \eqref{pre50}-\eqref{pre5}, are not sufficient for us to prove \eqref{pre4}. This pushes us to introduce a modified Duhamel's principle (see Lemma \ref{LD} below). By Lemma \ref{LD}, we are able to establish the following Strichartz estimate for a non-homogeneous linear wave equation. For any $1 \leq r \leq s+1$, for each $t_0 \in [-2,2)$, then the Cauchy problem
\begin{equation}\label{pre7}
	\begin{cases}
		& \square_{\mathbf{g}} f=\mathbf{g}^{0\alpha} \partial_\alpha F, \qquad (t,x) \in (t_0,2]\times \mathbb{R}^3,
		\\
		&(f,\partial_t f)|_{t=t_0}=( f_0,F(t_0,\cdot) ),
	\end{cases}
\end{equation}
admits a solution $f \in C([-2,2],H_x^r) \times C^1([-2,2],H_x^{r-1})$ and the following estimates holds, provided $k<r-1$,
\begin{equation}\label{pre6}
	\| \left<\nabla \right>^k f, \left<\nabla \right>^{k-1} df\|_{L^2_{[-2,2]}L^\infty_x} \lesssim \|f_0\|_{H_x^r}+\| F \|_{L^1_{[-2,2]}H_x^r}.
\end{equation}
Seeing \eqref{WTe} again, $h$ satisfies a good wave equation, while $\bu$ does not. So we find a good component\footnote{This is inspired by Wang's work \cite{WQEuler} on 3D non-relativistic compressible Euler equations.} $\bu_{+}$ instead of $\bu$. Precisely, the good wave equation for $h$ and $\bu_{+}$ take the form (see Lemma \ref{um} below):
\begin{equation}\label{pre8}
	\begin{cases}
		&{\square}_g h \approx (d\bu, dh) \cdot (d\bu,dh),
		\\
		&{\square}_g \bu_{+}\approx g^{0\alpha} \partial_{ \alpha } ( \mathbf{T} \bu_{-})+\nabla ( \mathbf{T} \bu_{-}) + (d\bu, dh) \cdot (d\bu,dh) + \textrm{lower\ order\ terms}.
	\end{cases}
\end{equation}
Since \eqref{pre8}, $\mathbf{T} \bu_{-}$ is the possible trouble term. Fortunately, using the elliptic setting for $\bu_{-}$ (see \eqref{uf}) and the transport structure for $\bW$ (see Lemma \ref{tr0}), then a new space-time elliptic estimate for $\mathbf{T} \bu_{-}$ (see Lemma \ref{tu} below) holds. This combing with \eqref{pre7}, \eqref{pre6}, and \eqref{pre8}, we therefore obtain the desired Strichartz estimate
\begin{equation*}
	\|d h, d \bu\|_{L^2_{[-2,2]} C^\delta_x}+\| d h, d \bu_{+}, d \bu\|_{L^2_{[-1,1]} \dot{B}^{s_0-2}_{\infty,2}} \leq \epsilon_2.
\end{equation*}
 For the entire proof, we refer the readers to Section \ref{SEs}.

\textit{$\bullet$ Step 6: The proof of \eqref{pre7}-\eqref{pre6}}. Through the above steps, we have obtained sufficient regularity of the characteristic hypersurfaces to control the characteristic hypersurfaces. This allows us to use the idea of constructing approximate solutions by wave-packets \cite{ST,Sm}. The whole proof is presented in Section \ref{spp}.

\textit{$\bullet$ Step 7: Continuous dependence.} 
In the proof of continuous dependence, for quasilinear system  \eqref{WTe} with rough data, there is loss of derivatives if we only use the Strichartz estimates of solutions. To overcome this difficulty, using frequency envelope approach, together with the key Strichartz estimate \eqref{linear}-\eqref{SE1}, we can prove the continuous dependence. We refer Section \ref{cd} for details.


\subsubsection{Outline proof of Theorem \ref{dingli2}} The main idea is as follows.

$\bullet \ \textit{Step 1: Energy estimates}$. By delicate calculations, we find that the bound of $\|  h\|_{L^\infty_{[0,t]}H_x^{\sstar}}+\| \bu\|_{L^\infty_{[0,t]}H_x^{\sstar}}+\|\bw\|_{L^\infty_{[0,t]} H_x^{2}}$ is controlled by the quantity
 	\begin{equation*}
 	\tilde{E}_0  \exp \left\{    5{\int^t_0} \|(dh,d\bu)\|_{L^\infty_x}   d\tau \cdot \exp \big(  5{\int^t_0} \|(dh,d\bu)\|_{L^\infty_x}   d\tau \big)\right\},
 \end{equation*}
 where $\tilde{E}_0$ is determined by the initial norm $\|h_0\|_{H^s}$, $\|{\bu}_0\|_{H^s}$, and $\|\bw_0\|_{H^{2}}$.  We refer the reader to Section \ref{sec3.2}.

 So the main goal is to control the following Strichartz estimate 
 \begin{equation}\label{pqqq}
 	\int^t_0 \|dh, d\bu\|_{ L_x^\infty} d\tau.
 \end{equation}

 $\bullet \ \textit{Step 2: Constructing a strong solution as a limit of smooth solutions}$.
 Since the initial data is not so smooth, we apply a mollifier to $(h_{0},\bu_0,\bw_0)$. Let $P_j=\sum_{j' \leq j}P_{j'}$ and $P_{j'}$ is a LP operator with frequency supports $\{ 2^{j'-3} \leq |\xi| \leq 2^{j'+3} \}$. We introduce the smooth sequence $(h_{0j},\bu_{0j})$ by
 \begin{equation*}
 	\begin{split}
 		& \mathring{\bu}_{0j}=P_{\leq j}\mathring{\bu}_0, \quad u^0_{0j}=\sqrt{1+|\mathring{\bu}_{0j}|^2 },
 		\\
 		& h_{0j}=P_{\leq j}h_0, \quad \bu_{0j}=(u^0_{0j},\mathring{\bu}_{0j}),
 	\end{split}
 \end{equation*}
Considering \eqref{VVd}, so we define
 \begin{equation*}
 	\bw_{0j}=\mathrm{vort}( \mathrm{e}^{h_{0j}} \bu_{0j}).
 \end{equation*}
Therefore, $\left\{ (h_{0j},\bu_{0j}, \bw_{0j}) \right\}_{j \in \mathbb{Z}}$ is a smooth initial sequence.

 Considering above initial datum $(h_{0j},\bu_{0j}, \bw_{0j})$, there exists positive numbers $T^*$, ${M}_2$, and $M_3$ ($T^*$, ${M}_2$ and ${M}_3$ only depends on $C_0, c_0, s$ and $M_*$) such that a sequence solutions of \eqref{WTe} satisfy $(h_{j},\bu_{j})\in C([0,T^*];H_x^s)\cap C^1([0,T^*];H_x^{s-1})$, $\bw_{j}\in C([0,T^*];H_x^2) \cap C^1([0,T^*];H_x^1)$. Moreover, for all $j\geq 1$, $h_j$ and $\bu_j$ satisfy a uniform Strichartz estimate
 \begin{equation}\label{prec1}
 	\|dh_j, d\bu_j\|_{L^2_{[0,T^*]}L_x^\infty} \leq {M}_{2}.
 \end{equation}
 Additionally, for $\frac{s}{2} \leq r \leq 3$, considering
 \begin{equation}\label{prec2}
 	\begin{cases}
 		\square_{{g}_j} f_j=0, \qquad [0,T^*]\times \mathbb{R}^3,
 		\\
 		(f_j,\partial_t f_j)|_{t=0}=(f_{0j},f_{1j}),
 	\end{cases}
 \end{equation}
 where $(f_{0j},f_{1j})=(P_{\leq j}f_0,P_{\leq j}f_1)$ and $(f_0,f_1)\in H_x^r \times H^{r-1}_x$, there is a unique solution $f_j$ on $[0,T^*]\times \mathbb{R}^3$ such that for $a\leq r-\frac{s}{2}$,
 \begin{equation}\label{prec4}
 	\begin{split}
 		&\|\left< \nabla \right>^{a-1} d{f}_j\|_{L^2_{[0,T^*]} L^\infty_x}
 		\leq  M_3(\|{f}_0\|_{{H}_x^r}+ \|{f}_1 \|_{{H}_x^{r-1}}),
 		\\
 		&\|{f}_j\|_{L^\infty_{[0,T^*]} H^{r}_x}+ \|\partial_t {f}_j\|_{L^\infty_{[0,T^*]} H^{r-1}_x} \leq  M_3 (\| {f}_0\|_{H_x^r}+ \| {f}_1\|_{H_x^{r-1}}).
 	\end{split}
 \end{equation}
Based on \eqref{prec1}-\eqref{prec4}, we are able to obtain a strong solution of Theorem \ref{dingli2} as a limit of the sequence $(h_{j},\bu_{j}, \bw_j)$. Please refer Section \ref{keypra} for details.

It remains for us to prove \eqref{prec1}-\eqref{prec4}.

$\bullet \ \textit{Step 3: Obtaining a sequence of solutions on small time-intervals}$. Note $\|{\bw}_{0j}\|_{H^{2+r}}\lesssim 2^{jr}$($r \geq 0$). By using compactness, scaling, and physical localization, we shall get a sequence of small initial datum. By using Proposition \ref{DDL3}, we can prove the existence of small solutions with these small datum. On one hand, Proposition \ref{DDL3} can be derived from \eqref{pre3}-\eqref{pre4}. See Section \ref{pps} for details.

On the other hand, returning back from small solutions to large solutions $(h_{j},\bu_{j},\bw_{j})$, then the time of existence of solutions $(h_{j},\bu_{j},\bw_{j})$ depends on $j\in \mathbb{Z}^+$. Because the norm $\|{\bw}_{0j}\|_{H^{2+}}$ depending on $j$. Moreover, we precisly calculate $T^*_j=2^{-\delta_1j} [E(0)]^{-1}$($\delta_1=\frac{\sstar-2}{40}$ and $E(0)=C(M_*+M_*^5) $, c.f \eqref{pu0} and \eqref{DTJ}). On $[0, T^*_j]$, the solutions $h_j$ and $\bu_j$ yields a higher-order Strichartz estimates,
\begin{equation}\label{prec}
	\|\left< \nabla \right>^{a-1} dh_j, \left< \nabla \right>^{a-1} d\bu_{j}\|_{L^2_{[0,T^*_j]} L^\infty_x} \leq C(1+E(0)), \quad a<\sstar-1.
\end{equation}
Therefore, by choosing a suitable number $a=9\delta_1$, for high frequency $k\geq j$, then we get
\begin{equation}\label{prec0}
	\|P_k d h_j, P_k d\bu_{j}\|_{L^2_{[0,T^*_j]} L^\infty_x} \leq C(1+E^3(0))2^{-\delta_{1} k} 2^{-7\delta_{1} j}.
\end{equation}
Decompose
\begin{equation*}
	\begin{split}
		d \bu_j= P_{\geq j}d \bu_j+ \textstyle{\sum}^{j-1}_{k=1} \textstyle{\sum}_{m=k}^{j-1} P_k  (d\bu_{m+1}-d\bu_m)+ \textstyle{\sum}^{j-1}_{k=1}P_k d \bu_k .
	\end{split}
\end{equation*}
Then we need to see if there is a similar estimates for $P_k  (d\bu_{m+1}-d\bu_m)$. By using a Strichartz estimates for linear wave, for $k<j$, we shall also obtain
\begin{equation}\label{prec00}
	\|P_k (dh_{m+1}-dh_m), P_k (d\bu_{m+1}-d\bu_m)\|_{L^2_{[0,T^*_{m+1}]} L^\infty_x} \leq C(1+E^3(0))2^{-\delta_{1} k} 2^{-6\delta_{1} m}.
\end{equation}
We prove \eqref{prec0}-\eqref{prec00} in Section \ref{esest}.

$\bullet \ \textit{Step 4: Extending a sequence of solutions to a regular time-interval}$. Compared \eqref{prec} with \eqref{prec0}-\eqref{prec00}, there is a loss of derivatives in \eqref{prec0}-\eqref{prec00}. The benifit relies on some decay estimates \eqref{prec0}-\eqref{prec00} on frequency number $j$ and $k$. Inspired by Bahouri-Chemin \cite{BC2}, Tataru \cite{T1} and Ai-Ifrim-Tataru \cite{AIT}, if we can always extend it with a length $\approx T_j^*$ of time intervals, and with a number $\approx (T_j^*)^{-1}$, then we shall sum these estimates up and obtain a uniform Strichartz estimates $\|d h_j, d\bu_{j}\|_{L^1_{[0,T^*]} L^\infty_x}$ and $\|d h_j, d\bu_{j}\|_{L^2_{[0,T^*]} L^\infty_x}$ on a regular time-interval $[0,T^*]$. In fact, we can extend it in this way from $[0,T^*_j]$ to $[0,T^*_{N_0}]$. We first extend it from $[0,T^*_j]$ to $[0,T^*_{j-1}]$. In this process, the growth in Strichartz and energy estimates can be calculated precisely. Then we conclude it by induction method. Finally, we get
\begin{equation*}
	T^*=2^{-\delta_{1} N_0} [E(0)]^{-1},
\end{equation*}
where $N_0$ is an fixed integer depending on $c_0, s, C_0$ and $M_*$(see \eqref{pp8} below). Furthermore,
\begin{equation*}
	\| h_j \|_{L^\infty_{[0,T^*]}H_x^s}+ \| \bu_j \|_{L^\infty_{[0,T^*]}H_x^s} + \| \bw_j \|_{L^\infty_{[0,T^*]}H_x^2},
\end{equation*}
and
\begin{equation*}
	\|d h_j\|_{L^2_{[0,T^*]} L^\infty_x} + \|d\bu_{j}\|_{L^2_{[0,T^*]} L^\infty_x},
\end{equation*}
are uniformly bounded (See \eqref{kz65}-\eqref{kz66}). Please refer Subsection \ref{finalk} for details.

$\bullet \ \textit{Step 5: Strichartz estimates for linear wave on the regular time interval}$. 
When the solutions in Step 4 are extended, then the corresponding solutions of linear wave are also extended on the same time-interval. By using a similar idea, summing up these bounds, we can prove \eqref{prec2}-\eqref{prec4}. We refer the readers to Section \ref{finalq} for the whole proof.
\subsection{Organization of the rest paper}
In the rest of the paper, Section \ref{sec:preliminaries} mainly introduce some basic key lemmas, including commutator and product estimates; modified Duhamel's principle; new transport equations and some elliptic estimates. Next, in Section \ref{sec:energyest}, we mainly prove two key energy theorems, i.e. Theorem \ref{VE} and Theorem \ref{VEt}. In Section \ref{Sec4}, we give a precise proof of the existence and uniqueness of solutions for Theorem \ref{dingli}. Then we give a precise proof for Proposition \ref{p4} in Section \ref{secp4}. In Section \ref{spp}, we give a proof for Proposition \ref{r5}. The next Section \ref{cd} give a proof of continuous dependence of solutions in Theorem \ref{dingli}. Additionally, Section \ref{Sec5} presents the proof of the well-posedness of solutions in Theorem \ref{dingli2}. Finally, we prove some auxiliary lemmas and give a guidance of notations in Section \ref{Ap}.
\section{Preliminaries}\label{sec:preliminaries}
\subsection{Commutator estimates and product estimates}
We start by recalling some commutator and product estimates, which have been proved in references \cite{KP,ST,WQRough,ML, AZ}.
\begin{Lemma}\label{jh}\cite{KP}
	Let $ a \geq 0$. Then for any scalar function $f_1, f_2$, we have
	\begin{equation*}
		\|\Lambda_x^a(f_1f_2)-(\Lambda_x^a f_1)f_2\|_{L_x^2} \lesssim \|\Lambda_x^{a-1}f_1\|_{L^{2}_x}\|\nabla f_2\|_{L_x^\infty}+ \|f_1\|_{L^p_x}\|\Lambda_x^a f_2\|_{L_x^{q}},
	\end{equation*}
	where $\frac{1}{p}+\frac{1}{q}=\frac{1}{2}$.
\end{Lemma}
\begin{Lemma}\label{cj}\cite{KP}
	Let $a \geq 0$. For any scalar function $f_1$ and $f_2$, we have
	\begin{equation*}
		\|f_1 f_2\|_{H_x^a} \lesssim \|f_1\|_{L_x^{\infty}}\| f_2\|_{H_x^a}+ \|f_2\|_{L_x^{\infty}}\| f_1 \|_{H_x^a}.
	\end{equation*}
\end{Lemma}
\begin{Lemma}\label{jh0}\cite{AZ}
	Let $F(u)$ be a smooth function of $u$, $F(0)=0$ and $u \in L^\infty_x$. For any $a \geq 0$, we have
	\begin{equation*}
		\|F(u)\|_{H_x^a} \lesssim  \|u\|_{H_x^{a}}(1+ \|u\|_{L^\infty_x}).
	\end{equation*}
\end{Lemma}

\begin{Lemma}\label{ps}\cite{ST}
	Suppose that $0 \leq r, r' < \frac{3}{2}$ and $r+r' > \frac{3}{2}$. Then for any scalar function $f_1$ and $f_2$, the following estimate holds:
	\begin{equation*}
		\|f_1f_2\|_{H^{r+r'-\frac{3}{2}}(\mathbb{R}^3)} \leq C_{r,r'} \|f_1\|_{H^{r}(\mathbb{R}^3)}\|f_2\|_{H^{r'}(\mathbb{R}^3)}.
	\end{equation*}
	If $-r \leq r' \leq r$ and $r>\frac{3}{2}$ then
	\begin{equation*}
		\|f_1 f_2\|_{H^{r'}(\mathbb{R}^3)} \leq C_{r,r'} \|f_1\|_{H^{r}(\mathbb{R}^3)}\|f_2\|_{H^{r'}(\mathbb{R}^3)}.
	\end{equation*}
\end{Lemma}

\begin{Lemma}\label{lpe}\cite{WQRough}
	Let $0 \leq a <1 $. Then
	\begin{equation*}
		\|\Lambda_x^a(f_1f_2)\|_{L^{2}_x(\mathbb{R}^3)} \lesssim \|f_1\|_{\dot{B}^{a}_{\infty,2}(\mathbb{R}^3)}\|f_2\|_{L^2_x(\mathbb{R}^3)}+ \|f_1\|_{L^\infty_x(\mathbb{R}^3)}\|f_2\|_{\dot{H}^a_x(\mathbb{R}^3)}.
	\end{equation*}
\end{Lemma}
\begin{Lemma}\label{wql}\cite{WQRough}
	Let $0 < a <1 $. Then
	\begin{equation*}
		\|\Lambda_x^a(f_1f_2f_3)\|_{L^{2}_x(\mathbb{R}^3)} \lesssim \|f_i\|_{H^{1+a}_x(\mathbb{R}^3)}\textstyle{\prod_{j\neq i}}\|f_j\|_{H^1_x(\mathbb{R}^3)}.
	\end{equation*}
\end{Lemma}
\begin{Lemma}\cite{AZ}\label{LPE}
	Let $0 < a <1 $. Let $b> a$ and sufficiently close to $a$. Then
	\begin{equation}\label{HF}
		\| f_1f_2\|_{\dot{B}^{a}_{\infty, 2}(\mathbb{R}^3)} \lesssim \|f_2\|_{L^{\infty}_x} \|f_1\|_{\dot{B}^{a}_{\infty, 2}(\mathbb{R}^3)}+\|f_2\|_{C^{b}(\mathbb{R}^3)} \|f_1\|_{L^\infty}.
	\end{equation}
\end{Lemma}
\begin{Lemma}\cite{AZ}\label{ceR}
	Let $0 \leq a <1$ and $\bu=(u^0,u^1,u^2,u^3)^{\mathrm{T}}$. Denote the Riesz operator $\mathbf{R} =\nabla^2(-\Delta)^{-1}$. Then
	\begin{equation*}
		\| [\mathbf{R}, (u^0)^{-1}u^i \partial_i ]f\|_{\dot{H}^{a}_x(\mathbb{R}^3)} \lesssim \| \bu\|_{\dot{B}^{1+a}_{\infty, \infty}} \|f\|_{L^2_{x}(\mathbb{R}^3)}+ \| \bu\|_{\dot{B}^1_{\infty, \infty}} \|f\|_{\dot{H}^a_x(\mathbb{R}^3)}.
	\end{equation*}
\end{Lemma}
\begin{Lemma}\cite{AZ}\label{YR}
	Let $2<s_1 \leq s_2$. Let $\mathbf{T}$ be defined in \eqref{opt}. Then
	\begin{equation*}
		\{ 2^{(s_1-1)j}\| [P_j, \mathbf{T}]f\|_{\dot{H}^{s_2-s_1}_x(\mathbb{R}^3)} \}_{l^2_j} \lesssim  \| \nabla \bu\|_{L^\infty_x }\|f\|_{\dot{H}^{s_2-1}_x(\mathbb{R}^3)}+\| \bu \|_{H^{s_2}}  \| f\|_{{H}^{1}}.
	\end{equation*}
\end{Lemma}
\begin{Lemma}\cite{AZ}\label{ce}
	Let $0 < a <1 $ and $\bu=(u^0,u^1,u^2,u^3)^{\mathrm{T}}$. Then
	\begin{equation*}
		\| [\Lambda^a_x, (u^0)^{-1}u^i \partial_i ]f\|_{L^2_x(\mathbb{R}^3)} \lesssim \|\nabla \bu\|_{L^\infty_x} \|f\|_{\dot{H}^{a}_{x}(\mathbb{R}^3)}.
	\end{equation*}
\end{Lemma}
\begin{Lemma}\cite{AZ}\label{yx}
	Let $2<s_1\leq s_2$. Let $\mathbf{g}$ be a Lorentz metric and $\mathbf{g}^{00}=-1$. We have
	\begin{equation}\label{YX}
		\begin{split}
			\left\{ 2^{(s_1-1)j}\|[\square_{\mathbf{g}}, P_j]f\|_{\dot{H}_x^{s_2-s_1}} \right\}_{l^2_j} \lesssim & \ \| d f\|_{L_x^\infty} \|d \mathbf{g}\|_{\dot{H}_x^{s_2-1}}+\| d \mathbf{g}\|_{L_x^\infty}\|d f\|_{\dot{H}_x^{s_2-1}}.
		\end{split}
	\end{equation}
\end{Lemma}
We next introduce a identity, which is very useful in our paper.
\begin{Lemma}[\cite{DS}, Lemma 4.1]\label{OE}
	Let $\bA=(A^0,A^1,A^2,A^3)^{\mathrm{T}}$ be a one-form. Let $\mathrm{vort}^\alpha(\bA)$ be defined in \eqref{VAd}. Then the following equality
	\begin{equation*}
		\partial_\alpha A_\beta - \partial_\beta A_\alpha = \epsilon_{\alpha\beta\gamma\delta}u^\gamma \mathrm{vort}^\delta(\bA)+ u_\alpha u^\kappa \partial_\beta A_\kappa - u_\beta u^\kappa \partial_\alpha A_\kappa+ u_\beta u^\kappa \partial_\kappa A_\alpha- u_\alpha u^\kappa \partial_\kappa A_\beta,
	\end{equation*}
	holds.
\end{Lemma}

\begin{remark}
	Let $(h, \bu)$ be a solution of \eqref{REE}. Let $\bw$ be defined in \eqref{VVd}. By Lemma \ref{OE}, we can get
	\begin{equation}\label{OE00}
		\partial_\alpha u_\beta - \partial_\beta u_\alpha = \epsilon_{\alpha\beta\gamma\delta}\mathrm{e}^{-h}u^\gamma w^\delta- u_\beta  \partial_\alpha h + u_\alpha  \partial_\beta h,
	\end{equation}
	and
	\begin{equation}\label{cr04}
		\partial_\alpha w_\beta - \partial_\beta w_\alpha = \epsilon_{\alpha \beta \gamma \delta}u^\gamma \mathrm{vort}^\delta(\bw)-u_\alpha u^\gamma \partial_\gamma w_\beta  + u_\alpha u^\gamma \partial_\beta w_\gamma + u_\beta u^\gamma \partial_\gamma w_\alpha - u_\beta u^\gamma \partial_\alpha w_\gamma .
	\end{equation}
	Additionally, we have
	\begin{equation}\label{cra0}
		w^\kappa \partial_\kappa u_\alpha= w^\kappa \partial_\alpha u_\kappa- u_\alpha w^\kappa \partial_\kappa h,
	\end{equation}
	and
	\begin{equation}\label{cra1}
		\epsilon^{\alpha \beta \gamma \delta} \partial_\gamma u_\delta =\mathrm{e}^{-h}w^\alpha u^\beta- \mathrm{e}^{-h}u^\alpha w^\beta- \epsilon^{\alpha \beta \gamma \delta}u_\delta \partial_\gamma h .
	\end{equation}
\end{remark}
In the following, we transfer to introduce a modified Duhamel's principle for acoustic metric.
\subsection{Modified Duhamel's principle}
\begin{Lemma}\label{LD}
	Let the metric $g$ is defined in \eqref{AMd}. If $f$ is the solution of
	\begin{equation*}
		\begin{cases}
			\square_{g}f=0, \qquad t> \tau,
			\\
			(f,\partial_t f)|_{t=\tau} =(F(\tau,x), 0),
		\end{cases}
	\end{equation*}
	then
	\begin{equation*}
		\phi(t,x)=\int^t_0 f(t,x;\tau)d\tau,
	\end{equation*}
	solves the linear wave equation
	\begin{equation}\label{ZH}
		\begin{cases}
			\square_{g}\phi=g^{0\alpha}\partial_\alpha F,
			\\
			(\phi, \partial_t \phi)|_{t=0}=(0,F(0,x)).
		\end{cases}
	\end{equation}
\end{Lemma}
\begin{proof}
	We first note $\square_g=g^{\alpha \beta}\partial^2_{\alpha \beta}$ and $g^{00}=-1$.  By direct calculations, we get
	\begin{equation}\label{d1}
		\begin{split}
			\partial_t \phi &= \int^t_0 \partial_t f(t,x;\tau)d\tau+f(t,x;t)
			\\
			&=\int^t_0 \partial_t f(t,x;\tau)d\tau+F(t,x).
		\end{split}
	\end{equation}
	Taking the operator $\partial_t$ on \eqref{d1}, we have
	\begin{equation}\label{d2}
		\begin{split}
			-\partial^2_t \phi &= -\int^t_0 \partial^2_t f(t,x;\tau)d\tau-\partial_t f(t,x;\tau)|_{\tau=t}-\partial_t F
			\\
			&= -\int^t_0 \partial^2_t f(t,x;\tau)d\tau-\partial_t F.
		\end{split}
	\end{equation}
	Taking the operator $\partial_i$ on \eqref{d2}, it follows
	\begin{equation}\label{d3}
		\begin{split}
			\partial_i \partial_t \phi &= \int^t_0 \partial_i \partial_t f(t,x;\tau)d\tau+\partial_{i} F.
		\end{split}
	\end{equation}
	On the other hand, we also have
	\begin{equation}\label{d4}
		\partial_{ij} \phi= \int^t_0  \partial_{ij} f(t,x;\tau)d\tau,
	\end{equation}
	Combining \eqref{d2}, \eqref{d3}, and \eqref{d4}, we can derive that
	\begin{equation*}
		\begin{split}
			\square_{g}\phi(t,x)&=\int^t_0 \square_{g}f(t,x;\tau)d\tau-g^{0i}\partial_{i}F(t,x)+\partial_t F(t,x)
			\\
			&= g^{0\alpha}\partial_{\alpha} F.
		\end{split}
	\end{equation*}
	Furthermore, we can verify the initial data as
	\begin{equation*}
		(\phi, \partial_t \phi)|_{t=0}=(0, F(0,x)).
	\end{equation*}
	At this stage, we complete the proof of Lemma \ref{LD}.
\end{proof}
\begin{remark}
Due to the acoustical metric, so we derive this modified Duhamel's principle.
\end{remark}
\subsection{Hyperbolic formulations and new transport system}
Let us introduce a symmetric hyperbolic structure of \eqref{mo3}-\eqref{OREE}, which will be used to derive energy estimates of the velocity and enthalpy.
\begin{Lemma}[\cite{Bru},Theorem 13.1]\label{QH}
	Let $(\varrho, \bu)$ be a solution of \eqref{mo3}-\eqref{OREE}. Denote $\bU=(p(h),u^1,u^2,u^3)^{\mathrm{T}}$. Then System \eqref{OREE} can be reduced to the following symmetric hyperbolic equation:
	\begin{equation}\label{QHl}
		A^\alpha(\bU) \partial_\alpha \bU=0,
	\end{equation}
	where
	\begin{equation}\label{A0}
		A^0=
		\left(
		\begin{array}{cccc}
			(\varrho+p)^{-2}\varrho'(p)u^0 & -(\varrho+p)^{-1}\frac{u^1}{u_0} & -(\varrho+p)^{-1}\frac{u^2}{u_0} & -(\varrho+p)^{-1}\frac{u^3}{u_0}   \\
			-(\varrho+p)^{-1}\frac{u^1}{u_0} & u^0(1+\frac{u^1u_1}{u^0u_0}) & u^0\frac{u^1 u_2}{u^0u_0} & u^0\frac{u^1u_3}{u^0u_0} \\
			-(\varrho+p)^{-1}\frac{u^2}{u_0} & u^0\frac{u^1 u_3}{u^0u_0} & u^0(1+\frac{u^2 u_2}{u^0u_0}) & u^0\frac{u^2u_3}{u^0u_0} \\
			-(\varrho+p)^{-1}\frac{u^3}{u_0} & u^0\frac{u^1u_3}{u^0u_0} & u^0\frac{u^2u_3}{u^0u_0} & u^0(1+\frac{u^3u_3}{u^0u_0})
		\end{array}
		\right ),
	\end{equation}
	\begin{equation}\label{A1}
		A^1=
		\left(
		\begin{array}{cccc}
			(\varrho+p)^{-2}\varrho'(p)u^1 & -(\varrho+p)^{-1} & 0 & 0   \\
			-(\varrho+p)^{-1} & u^1(1+\frac{u^1u_1}{u^0u_0}) & u^1\frac{u^1 u_2}{u^0u_0} & u^1\frac{u^1u_3}{u^0u_0} \\
			0 & u^1\frac{u^1 u_3}{u^0u_0} & u^1(1+\frac{u^2 u_2}{u^0u_0}) & u^1\frac{u^2u_3}{u^0u_0} \\
			0 & u^1\frac{u^1u_3}{u^0u_0} & u^1\frac{u^2u_3}{u^0u_0} & u^1(1+\frac{u^3u_3}{u^0u_0})
		\end{array}
		\right ),
	\end{equation}
	\begin{equation}\label{A2}
		A^2=
		\left(
		\begin{array}{cccc}
			(\varrho+p)^{-2}\varrho'(p)u^2 & 0 & -(\varrho+p)^{-1} & 0   \\
			0 & u^2(1+\frac{u^1u_1}{u^0u_0}) & u^2\frac{u^1 u_2}{u^0u_0} & u^2\frac{u^1u_3}{u^0u_0} \\
			-(\varrho+p)^{-1} & u^2\frac{u^1 u_3}{u^0u_0} & u^2(1+\frac{u^2 u_2}{u^0u_0}) & u^2\frac{u^2u_3}{u^0u_0} \\
			0 & u^2\frac{u^1u_3}{u^0u_0} & u^2\frac{u^2u_3}{u^0u_0} & u^2(1+\frac{u^3u_3}{u^0u_0})
		\end{array}
		\right ),
	\end{equation}
	\begin{equation}\label{A3}
		A^3=
		\left(
		\begin{array}{cccc}
			(\varrho+p)^{-2}\varrho'(p)u^1 & 0 & 0 &  -(\varrho+p)^{-1}   \\
			0 & u^3(1+\frac{u^1u_1}{u^0u_0}) & u^3\frac{u^1 u_2}{u^0u_0} & u^3\frac{u^1u_3}{u^0u_0} \\
			0 & u^3\frac{u^1 u_3}{u^0u_0} & u^3(1+\frac{u^2 u_2}{u^0u_0}) & u^3\frac{u^2u_3}{u^0u_0} \\
			-(\varrho+p)^{-1} & u^3\frac{u^1u_3}{u^0u_0} & u^3\frac{u^2u_3}{u^0u_0} & u^3(1+\frac{u^3u_3}{u^0u_0})
		\end{array}
		\right ).
	\end{equation}
\end{Lemma}
We next derive some good transport equations for $\bw$.
\begin{Lemma}\label{tr0}
Let $\bw$ and $\bW$ be defined in \eqref{VVd} and \eqref{MFd}. Then we have
	\begin{align}\label{CEQ}
			u^\kappa \partial_\kappa w^\alpha =& -u^\alpha w^\kappa \partial_\kappa h+ w^\kappa \partial_\kappa u^\alpha- w^\alpha \partial_\kappa u^\kappa,
			\\\label{CEQ0}
			\partial_\alpha w^\alpha=&-w^\kappa \partial_\kappa h,
			\\\label{CEQ1}
			u^\kappa \partial_\kappa {W}^\alpha
			=& {W}^\kappa   \partial_\kappa u^\alpha   -{2} {W}^\alpha \partial_\kappa u^\kappa+ u^\alpha {W}^\beta  u^\kappa \partial_\kappa u_\beta \nonumber
			-{2} \epsilon^{\alpha \beta \gamma \delta}u_\beta \partial_\delta u^\kappa \partial_\gamma w_\kappa
			\\
			&  -{2}\mathrm{e}^{-h}  w^\alpha w^\kappa \partial_\kappa h -  \epsilon^{\alpha \beta \gamma \delta} c^{-2}_s u_\beta w_\delta   \partial_\gamma u^\kappa \partial_\kappa h
			- \epsilon^{\kappa \beta \gamma \delta}c^{-2}_s u_\beta w_\delta \partial_\gamma h \partial_\kappa u^\alpha
			\\
			& + \epsilon^{\alpha \beta \gamma \delta} c^{-2}_s u_\beta w_\delta \partial_\gamma h \partial_\kappa u^\kappa
			+( c^{-2}_s +2 ) \epsilon^{\alpha \beta \gamma \delta} u_\beta w^\kappa \partial_\delta u_\kappa \partial_\gamma h. \nonumber
	\end{align}
\end{Lemma}
\begin{remark}
	We prove Lemma \ref{tr0} in Appendix \ref{appa}.
\end{remark}
	We next introduce a decomposition for $\partial_{\gamma} \partial^{\gamma} \bu$, whose role is very similar to Hodge's decomposition in non-relativistic fluids.
\begin{Lemma}\label{HD}
Let $(h,\bu)$ be a solution of \eqref{REE}. Then we have the following formula
\begin{equation}\label{HDe}
\begin{split}
  \partial_{\gamma}(\partial^{\gamma} u^\alpha) = &\mathrm{vort}^\alpha ( \mathrm{vort} \bu)+ \partial^\alpha (\partial_\gamma u^\gamma)
   -u^\kappa \partial_\kappa ( u^\gamma \partial_\gamma u^\alpha- u^\alpha \partial_\gamma u^\gamma )
  \\
  & +2u_\beta \partial_\gamma u^\alpha \partial^\beta u^\gamma - 2u_\beta \partial_\gamma u^\gamma \partial^\beta u^\alpha- u^\alpha \partial_\gamma u_\beta \partial^\gamma u^\beta
 + u^\gamma \partial_\gamma u_\beta \partial^\alpha u^\beta.
\end{split}
\end{equation}
\end{Lemma}
\begin{proof}
Firstly, we have $\mathrm{vort}^\alpha(\mathrm{vort} \bu)
=\epsilon^{\alpha\beta\gamma\delta}\epsilon_{\delta \eta \mu \nu} u_\beta \partial_\gamma\left( u^{\eta} \partial^{\mu} u^{\nu} \right)$. Thus, we have
\begin{align}\label{Or}
 & \mathrm{vort}^\alpha(\mathrm{vort} \bu) \nonumber
  \\ \nonumber
  =& (-\delta^\alpha_{\eta} \delta^\gamma_{\mu}  \delta^\beta_{\nu} + \delta^\alpha_{\eta} \delta^\beta_{\mu} \delta^\gamma_{\nu}
  - \delta^\beta_{\eta} \delta^\alpha_{\mu} \delta^\gamma_{\nu}
  +\delta^\beta_{\eta} \delta^\gamma_{\mu} \delta^\alpha_{\nu}
  - \delta^\gamma_{\eta} \delta^\beta_{\mu} \delta^\alpha_{\nu}+\delta^\gamma_{\eta} \delta^\alpha_{\mu} \delta^\beta_{\nu} ) u_\beta \partial_\gamma\left( u^{\eta} \partial^{\mu} u^{\nu} \right)
  \\ \nonumber
   = &  u_\beta \partial_\gamma u^\alpha \partial^\gamma  u^\beta-u_\beta \partial_\gamma u^\alpha \partial^\beta  u^\gamma+u_\beta \partial_\gamma u^\beta \partial^\alpha  u^\gamma
  - u_\beta \partial_\gamma u^\beta \partial^\gamma  u^\alpha
  \\ \nonumber
  & + u_\beta \partial_\gamma u^\gamma \partial^\beta  u^\alpha - u_\beta \partial_\gamma u^\gamma \partial^\alpha  u^\beta
   + u_\beta u^\alpha \partial^\gamma(\partial_\gamma u^\beta)-u_\beta u^\alpha \partial^\beta(\partial_\gamma u^\gamma )
  \\ \nonumber
  & + u_\beta u^\beta \partial^\alpha (\partial_\gamma u^\gamma)
   - u_\beta u^\beta \partial^\gamma(\partial_\gamma u^\alpha)
  +  u_\beta u^\gamma \partial^\beta (\partial_\gamma u^\alpha) -  u_\beta u^\gamma \partial^\alpha (\partial_\gamma u^\beta)
  \\ \nonumber
  = &  -u_\beta \partial_\gamma u^\alpha \partial^\beta  u^\gamma  + u_\beta \partial_\gamma u^\gamma \partial^\beta  u^\alpha
 + u_\beta u^\alpha \partial^\gamma(\partial_\gamma u^\beta)
 -u_\beta u^\alpha \partial^\beta(\partial_\gamma u^\gamma )
  \\
  &  + u_\beta u^\beta \partial^\alpha (\partial_\gamma u^\gamma)
   - u_\beta u^\beta \partial^\gamma(\partial_\gamma u^\alpha)
  +  u_\beta u^\gamma \partial^\beta (\partial_\gamma u^\alpha) -  u_\beta u^\gamma \partial^\alpha (\partial_\gamma u^\beta),
\end{align}
where we use the fact $u_\beta \partial_\gamma u^\beta=u^\beta \partial_\gamma u_\beta=\frac{1}{2}\partial_\gamma(u^\beta u_\beta)= 0$. By a direct calculation for
\begin{equation*}
  \begin{split}
  u_\beta u^\alpha \partial^\gamma(\partial_\gamma u^\beta)=& u^\alpha \left\{ \partial^\gamma\partial_\gamma(u_\beta u^\beta)-(\partial^\gamma\partial_\gamma u_\beta) u^\beta- 2\partial_\gamma u_\beta \partial^\gamma u^\beta  \right\}
  \\
  =&-u_\beta u^\alpha \partial^\gamma(\partial_\gamma u^\beta)- 2u^\alpha \partial_\gamma u_\beta \partial^\gamma u^\beta,
  \end{split}
\end{equation*}
we then have
\begin{equation}\label{Or1}
  u_\beta u^\alpha \partial^\gamma(\partial_\gamma u^\beta)=- u^\alpha \partial_\gamma u_\beta \partial^\gamma u^\beta.
\end{equation}
Thirdly, we can compute out
\begin{align}\label{Or2}
  -u_\beta u^\alpha \partial^\beta(\partial_\gamma u^\gamma )=&-u^\alpha u_\beta \partial^\beta(\partial_\gamma u^\gamma ) \nonumber
  \\ \nonumber
  =& u^\kappa \partial_\kappa \left(  -u^\alpha  \partial_\gamma u^\gamma  \right)+ u^\kappa \partial_\kappa u^\alpha \partial_\gamma u^\gamma \nonumber
  \\
  =& u^\kappa \partial_\kappa \left(  -u^\alpha  \partial_\gamma u^\gamma  \right)+ u_\kappa \partial^\kappa u^\alpha \partial_\gamma u^\gamma.
\end{align}
By using $u_\beta u^\beta=-1$, we get
\begin{align}\label{Or3}
  u_\beta u^\beta \partial^\alpha (\partial_\gamma u^\gamma)=& \partial^\alpha (-\partial_\gamma u^\gamma),
\\ \label{Or4}
  - u_\beta u^\beta \partial^\gamma(\partial_\gamma u^\alpha) = & \partial^\gamma(\partial_\gamma u^\alpha).
\end{align}
Also, we can deduce
\begin{equation}\label{Or5}
  u_\beta u^\gamma \partial^\beta (\partial_\gamma u^\alpha)=u^\kappa \partial_\kappa (u^\gamma \partial_\gamma u^\alpha)-u^\kappa\partial_\kappa u^\gamma \partial_\gamma u^\alpha,
\end{equation}
and
\begin{align}\label{Or6}
  -  u_\beta u^\gamma \partial^\alpha (\partial_\gamma u^\beta)=& u^\gamma \partial_\gamma ( -  u_\beta \partial^\alpha u^\beta )+ u^\gamma \partial_\gamma u_\beta \partial^\alpha u^\beta \nonumber
  \\
  = & u^\kappa \partial_\kappa ( -  u_\beta \partial^\alpha u^\beta )+ u^\gamma \partial_\gamma u_\beta \partial^\alpha u^\beta \nonumber
  \\
  = & u^\gamma \partial_\gamma u_\beta \partial^\alpha u^\beta.
\end{align}
Inserting \eqref{Or1}, \eqref{Or2}, \eqref{Or3}, \eqref{Or4}, \eqref{Or5}, \eqref{Or6} to \eqref{Or}, we can obtain \eqref{HDe}.
\end{proof}

\begin{Lemma}\label{VC}
Let $(h,\bu)$ be a solution of \eqref{REE}. Let $\bG=(G^0,G^1,G^2,G^3)^{\mathrm{T}}$ be defined in \eqref{MFd}. Then $\bG$ satisfies the following transport equation
\begin{equation}\label{SDe}
  u^\kappa \partial_\kappa (G^\alpha-F^\alpha)=\partial^\alpha \Gamma+ E^\alpha.
\end{equation}
Above, the scalar function $\Gamma$ is given by
\begin{equation}\label{YXg}
  \Gamma= -  2\partial^{\gamma} w^\kappa \partial_\gamma u_\kappa+2 w^\lambda \mathrm{vort}_\lambda \bw-\mathrm{e}^{-h}w_\kappa {W}^\kappa  ,
\end{equation}
and the vector-functions $\bF=(F^0,F^1,F^2,F^3)^\mathrm{T}$ and $\bE=(E^0,E^1,E^2,E^3)^\mathrm{T}$ defined by
\begin{equation}\label{YX0}
\begin{split}
  F^\alpha=&-2\epsilon^{\alpha \beta \gamma \delta} c_s^{-2} u_\beta \partial_\gamma h W_\delta-2u^\alpha \partial^\gamma w^\lambda \partial_\gamma u_\lambda+2 u_\lambda \partial_\gamma u^\gamma \partial^\alpha w^\lambda - 2c_s^{-2}\partial_\lambda h \partial^\alpha w^\lambda,
\end{split}
\end{equation}
and\footnote{Calculating $w^\kappa W^\kappa=- \epsilon^{\kappa \beta \gamma \delta}w_\kappa u_\beta \partial_\gamma w_\delta$, then $\Gamma$ only includes the type $d\bu\cdot d\bw$. By \eqref{YX0}, $\bF$ contains the term $(d\bu,dh)\cdot d\bw$, and $(d\bu,dh)\cdot (d\bu,dh) \cdot \bw$. Similarly, since \eqref{YX1}, $\bE$ includes the term
	$(d\bu,dh)\cdot d^2\bw, (d\bu,dh)\cdot (d\bu,dh) \cdot (d\bu,dh) \cdot \bw , (d\bu,dh)\cdot \bw \cdot (d^2\bu,d^2h)$, and $(d\bu,dh)\cdot (d\bu,dh) \cdot d\bw$. }
\begin{equation}\label{YX1}
\begin{split}
 E^\alpha=& \epsilon^{\alpha \beta \gamma \delta} u_\beta \partial_\gamma u^\kappa \partial_\kappa {W}_{\delta}
  - \epsilon^{\alpha \beta \gamma \delta} u^\kappa \partial_\kappa u_\beta \partial_\gamma {W}_\delta
   + 2\epsilon^{\alpha \beta \gamma \delta} c_s^{-2} u^\kappa \partial_\kappa u_\beta \partial_\gamma h W_\delta
  \\
  & + 2\epsilon^{\alpha \beta \gamma \delta} c_s^{-2} u_\beta u^\kappa  \partial_\gamma h \partial_\kappa W_\delta
  + 2\epsilon^{\alpha \beta \gamma \delta} u_\beta \partial_\kappa u^\kappa  \partial_\gamma W_\delta
   - \epsilon^{\alpha \beta \gamma \delta}u_\beta \partial_\kappa u_\delta \partial_\gamma {W}^\kappa
   \\
   & + \mathrm{e}^{-h} {W}^\kappa ( -u_\kappa u^\gamma \partial_\gamma w^\alpha
   + u_\kappa u^\gamma \partial^\alpha w_\gamma + u^\alpha u^\gamma \partial_\gamma w_\kappa - u^\alpha u^\gamma \partial_\kappa w_\gamma)
   \\
   &	+  \mathrm{e}^{-h} w_\kappa {W}^\kappa \partial^\alpha h  -  \mathrm{e}^{-h} w_\kappa \partial^\alpha{W}^\kappa
   -\epsilon_{\kappa \ \gamma \delta}^{ \ \alpha}  \epsilon^{\delta \beta \mu \nu} c_s^{-2}  \mathrm{e}^{-h}  u^\gamma  u_\beta w_\nu {W}^\kappa \partial_\mu h
   \\
   & +\epsilon^{\alpha \beta \gamma \delta} {W}^\kappa \partial_\kappa u_\beta \partial_\gamma u_\delta
    - 2\mathrm{e}^{-h}w^\alpha {W}^\kappa \partial_\kappa h
   - \epsilon^{\alpha \beta \gamma \delta}u_\beta \partial_\kappa u_\delta \partial_\gamma {W}^\kappa
   \\
   & +\epsilon^{\alpha \beta \gamma \delta} {W}^\kappa \partial_\kappa u_\beta \partial_\gamma u_\delta+4\partial^\gamma u_{\kappa}\partial^{\alpha} \partial_{\gamma} w^\kappa
   +2  u^\beta  \partial^{\gamma} w^\kappa  \partial_\beta u^\alpha \partial_{\gamma} u_\kappa
  \\
  & -2  u_\beta \partial_\gamma u^{\alpha}\partial^{\gamma} w^\kappa \partial^{\beta} u_\kappa
   -2  u_\beta  u^{\alpha} \partial_\gamma(\partial^{\gamma} w^\kappa) \partial^{\beta} u_\kappa
    -2u^\beta u^{\alpha} \partial_{\gamma} u_\kappa \partial_\beta (\partial^{\gamma} w^\kappa)
   \\
   & +2  u_\beta \partial_\gamma u^{\alpha}\partial^{\beta} w^\kappa \partial^{\gamma} u_\kappa
   + 2 u_\beta  u^{\alpha} \partial^{\gamma} u_\kappa \partial_\gamma(\partial^{\beta} w^\kappa)
   \\
   & - 2 u^{\alpha}u^\kappa w^\lambda \partial_\lambda h \partial_\gamma( \partial^{\gamma} u_\kappa )
   +2 u^{\alpha} w^\lambda \partial_\lambda u^\kappa  \partial_\gamma( \partial^{\gamma} u_\kappa )
   - 2 u^{\alpha}w^\kappa \partial_\lambda u^\lambda  \partial_\gamma( \partial^{\gamma} u_\kappa )
  \\
  & +  ( 4u_\beta \partial_\gamma u_\lambda \partial^\beta u^\gamma - 4u_\beta \partial_\gamma u^\gamma \partial^\beta u_\lambda- 2u_\lambda \partial_\gamma u_\beta \partial^\gamma u^\beta +2 u^\gamma \partial_\gamma u_\beta \partial_\lambda u^\beta ) \partial^\alpha w^\lambda
  \\
  & -2 c_s^{-2} \partial_\lambda u^\kappa \partial_\kappa h \partial^\alpha w^\lambda
  +6c_s^{-1}c'_s  \partial_\kappa u^\kappa  \partial_\lambda h \partial^\alpha w^\lambda+2c_s^{-2}  u^\kappa \partial_\lambda h   \partial_\kappa(\partial^\alpha w^\lambda)
  \\
  & -2 w^\lambda \partial^\alpha (\mathrm{vort}_\lambda \bw)  + 2 u^\kappa  ( u^\gamma \partial_\gamma u_\lambda- u_\lambda \partial_\gamma u^\gamma  ) \partial_\kappa ( \partial^\alpha w^\lambda)
  -  2\partial^{\alpha} u_\kappa \partial^\gamma \partial_{\gamma} w^\kappa
  \\
  &
  -2 u_\beta \partial_\gamma u^{\gamma}\partial^{\beta} w^\kappa \partial^{\alpha} u_\kappa
  -2  u_\beta  u^{\gamma} \partial^{\alpha} u_\kappa \partial^\beta(\partial_{\gamma} w^\kappa)
  +2   u^\kappa u^{\gamma} w^\lambda \partial_\lambda h \partial_\gamma( \partial^{\alpha} u_\kappa )
  \\
  & -2 u^{\gamma} w^\lambda \partial_\lambda u^\kappa \partial_\gamma( \partial^{\alpha} u_\kappa )
  +2  u^{\gamma} w^\kappa \partial_\lambda u^\lambda \partial_\gamma( \partial^{\alpha} u_\kappa )
  - 2u^{\kappa}\partial_\kappa u^\gamma \partial_{\gamma} u_\beta \partial^{\alpha} w^\beta
  \\
  & +  2u_\beta  u^{\gamma} \partial^\alpha(\partial_{\gamma} w^\kappa) \partial^{\beta} u_\kappa
  -  2u^{\kappa}u^\gamma \partial_{\gamma} u_\beta \partial^{\alpha}(\partial_\kappa w^\beta)
  +2u_\beta \partial_\gamma u^{\gamma}\partial^{\alpha} w^\kappa \partial^{\beta} u_\kappa
  \\
  &
   -2 \epsilon^{\alpha \beta \gamma \delta}  \epsilon_{\delta}^{ \ \eta \mu \nu} c_s^{-3} c'_s u_\beta  u_\eta w_\nu \partial_\gamma h    \partial_\mu  u^\kappa \partial_\kappa h + \epsilon^{\alpha \beta \gamma \delta}  \epsilon_{\delta}^{ \ \eta \mu \nu} c_s^{-2} u_\beta \partial_\gamma  u_\eta w_\nu   \partial_\mu u^\kappa \partial_\kappa h
  \\
  & + \epsilon^{\alpha \beta \gamma \delta}  \epsilon_{\delta}^{ \ \eta \mu \nu} c_s^{-2} u_\beta  u_\eta \partial_\gamma w_\nu   \partial_\mu u^\kappa \partial_\kappa h +  \epsilon^{\alpha \beta \gamma \delta}  \epsilon_{\delta}^{ \ \eta \mu \nu} c_s^{-2} u_\beta  u_\eta  w_\nu  \partial_\kappa h \partial_\gamma \partial_\mu u^\kappa
  \\
  & + \epsilon^{\alpha \beta \gamma \delta}  \epsilon_{\delta}^{ \ \eta \mu \nu} c_s^{-2} u_\beta  u_\eta w_\nu   \partial_\mu u^\kappa \partial_\gamma \partial_\kappa h - 2 \epsilon^{\alpha \beta \gamma \delta} \mathrm{e}^{-h} u_\beta    w_\delta w^\kappa \partial_\kappa h \partial_\gamma h
  \\
  & + 2 \epsilon^{\alpha \beta \gamma \delta} \mathrm{e}^{-h} u_\beta  w^\kappa \partial_\gamma w_\delta \partial_\kappa h
   + 2 \epsilon^{\alpha \beta \gamma \delta} \mathrm{e}^{-h} u_\beta  w_\delta \partial_\gamma w^\kappa \partial_\kappa h
  \\
  & + 2 \epsilon^{\alpha \beta \gamma \delta} \mathrm{e}^{-h}  u_\beta  w_\delta w^\kappa \partial_\gamma \partial_\kappa h
   -2 \epsilon^{\alpha \beta \gamma \delta} \epsilon^{\kappa \eta \mu \nu} c_s^{-3} c'_s u_\beta   u_\eta w_\nu \partial_\gamma h \partial_\mu h \partial_\kappa u_\delta
   \\
   &  + \epsilon^{\alpha \beta \gamma \delta} \epsilon^{\kappa \eta \mu \nu}  c_s^{-2} u_\beta w_\nu \partial_\gamma  u_\eta  \partial_\mu h \partial_\kappa u_\delta
   + \epsilon^{\alpha \beta \gamma \delta} \epsilon^{\kappa \eta \mu \nu} c_s^{-2} u_\eta u_\beta \partial_\gamma w_\nu \partial_\mu h \partial_\kappa u_\delta
  \\
  & + \epsilon^{\alpha \beta \gamma \delta} \epsilon^{\kappa \eta \mu \nu} c_s^{-2} u_\beta  u_\eta w_\nu \partial_\kappa u_\delta \partial_\gamma \partial_\mu h
   + \epsilon^{\alpha \beta \gamma \delta} \epsilon^{\kappa \eta \mu \nu} c_s^{-2} u_\beta   u_\eta w_\nu \partial_\mu h \partial_\gamma \partial_\kappa u_\delta
   \\
  & + 2 \epsilon^{\alpha \beta \gamma \delta} \epsilon_{\delta}^{\ \eta \mu \nu} c_s^{-3} c'_s u_\beta u_\eta w_\nu \partial_\gamma h   \partial_\mu h \partial_\kappa u^\kappa - \epsilon^{\alpha \beta \gamma \delta} \epsilon_{\delta}^{\ \eta \mu \nu} c_s^{-2} u_\beta \partial_\gamma  u_\eta w_\nu \partial_\mu h \partial_\kappa u^\kappa
   \\
   & - \epsilon^{\alpha \beta \gamma \delta} \epsilon_{\delta}^{\ \eta \mu \nu} c_s^{-2}  u_\beta u_\eta \partial_\gamma  w_\nu \partial_\mu h \partial_\kappa u^\kappa   - \epsilon^{\alpha \beta \gamma \delta} \epsilon_{\delta}^{\ \eta \mu \nu} c_s^{-2}  u_\beta u_\eta w_\nu \partial_\kappa u^\kappa \partial_\gamma  \partial_\mu h
   \\
   & - \epsilon^{\alpha \beta \gamma \delta} \epsilon_{\delta}^{\ \eta \mu \nu} c_s^{-2} u_\beta u_\eta w_\nu \partial_\mu h  \partial_\gamma \partial_\kappa u^\kappa
    + 2 \epsilon^{\alpha \beta \gamma \delta} \epsilon_{\delta}^{\ \eta \mu \nu} c_s^{-3} c'_s u_\beta  u_\eta w^\kappa \partial_\gamma h \partial_\nu u_\kappa \partial_\mu h
    \\
    & -\epsilon^{\alpha \beta \gamma \delta} \epsilon_{\delta}^{\ \eta \mu \nu}  ( c^{-2}_s +2 ) u_\beta w^\kappa \partial_\gamma   u_\eta  \partial_\nu u_\kappa \partial_\mu h
   -\epsilon^{\alpha \beta \gamma \delta} \epsilon_{\delta}^{\ \eta \mu \nu} ( c^{-2}_s +2 ) u_\eta u_\beta \partial_\gamma   w^\kappa \partial_\nu u_\kappa \partial_\mu h
   \\
   & -\epsilon^{\alpha \beta \gamma \delta} \epsilon_{\delta}^{\ \eta \mu \nu} ( c^{-2}_s +2 ) u_\beta  u_\eta w^\kappa \partial_\mu h \partial_\gamma \partial_\nu u_\kappa
   -\epsilon^{\alpha \beta \gamma \delta} \epsilon_{\delta}^{\ \eta \mu \nu} ( c^{-2}_s +2 )  u_\beta u_\eta w^\kappa \partial_\nu u_\kappa \partial_\gamma \partial_\mu h,
 \end{split}
\end{equation}
\end{Lemma}
\begin{proof}
Operating \eqref{CEQ1} with $-\epsilon^{\alpha \beta \gamma \delta}u_\beta \partial_\gamma$, we can derive
\begin{equation}\label{CEq3}
\begin{split}
 & -\epsilon^{\alpha \beta \gamma \delta}u_\beta \partial_\gamma( u^\kappa \partial_\kappa {W}_\delta)
 \\
 =&\underbrace{ -\epsilon^{\alpha \beta \gamma \delta}u_\beta \partial_\gamma( {W}^\kappa \partial_\kappa u_\delta) }_{\equiv \mathrm{R}_1}
 \underbrace{-\epsilon^{\alpha \beta \gamma \delta}u_\beta \partial_\gamma(u_\delta {W}^\kappa u^\lambda \partial_\lambda u_\kappa) }_{\equiv \mathrm{R}_2}
+ \underbrace{2\epsilon^{\alpha \beta \gamma \delta}u_\beta \partial_\gamma\left( \partial_\kappa u^\kappa {W}_\delta \right) }_{\equiv \mathrm{R}_0}
 \\
 &+\underbrace{ 2\epsilon^{\alpha \beta \gamma \delta}\epsilon_{\delta}^{\ \eta \mu \nu}u_\beta \partial_\gamma \big( u_{\eta}\partial_{\mu} w^\kappa \partial_{\nu} u_\kappa \big)  }_{\equiv \mathrm{R}_3}
 + \underbrace{ \epsilon^{\alpha \beta \gamma \delta}  \epsilon_{\delta}^{ \ \eta \mu \nu} u_\beta \partial_\gamma( c_s^{-2} u_\eta w_\nu   \partial_\mu u^\kappa \partial_\kappa h) }_{\equiv \mathrm{R}_4}
 \\
 & + \underbrace{ 2 \epsilon^{\alpha \beta \gamma \delta} u_\beta \partial_\gamma ( \mathrm{e}^{-h}  w_\delta w^\kappa \partial_\kappa h) }_{\equiv \mathrm{R}_5}
 + \underbrace{  \epsilon^{\alpha \beta \gamma \delta} \epsilon^{\kappa \eta \mu \nu} u_\beta \partial_\gamma ( c_s^{-2} u_\eta w_\nu \partial_\mu h \partial_\kappa u_\delta) }_{\equiv \mathrm{R}_6}
 \\
 &  \underbrace{  - \epsilon^{\alpha \beta \gamma \delta} \epsilon_{\delta}^{\ \eta \mu \nu} u_\beta \partial_\gamma ( c_s^{-2} u_\eta w_\nu \partial_\mu h \partial_\kappa u^\kappa ) }_{\equiv \mathrm{R}_7}
 \underbrace{  -\epsilon^{\alpha \beta \gamma \delta} \epsilon_{\delta}^{\ \eta \mu \nu} u_\beta \partial_\gamma \{   ( c^{-2}_s +2 ) u_\eta w^\kappa \partial_\nu u_\kappa \partial_\mu h \} }_{\equiv \mathrm{R}_8}.
 \end{split}
\end{equation}
By the chain rule, the left side of \eqref{CEq3} is
\begin{align}\label{LCEq}
  -\epsilon^{\alpha \beta \gamma \delta}u_\beta \partial_\gamma( u^\kappa \partial_\kappa {W}_\delta)&=
  -\epsilon^{\alpha \beta \gamma \delta} u_\beta \partial_\gamma u^\kappa \partial_\kappa {W}_{\delta}
  -\epsilon^{\alpha \beta \gamma \delta} u_\beta u^\kappa \partial_\kappa (\partial_\gamma {W}_\delta) \nonumber
  \\
  &=-\epsilon^{\alpha \beta \gamma \delta} u_\beta \partial_\gamma u^\kappa \partial_\kappa {W}_{\delta}
  - u^\kappa \partial_\kappa (\epsilon^{\alpha \beta \gamma \delta} u_\beta \partial_\gamma {W}_\delta)
  + \epsilon^{\alpha \beta \gamma \delta} u^\kappa \partial_\kappa u_\beta \partial_\gamma {W}_\delta \nonumber
  \\
  &=  u^\kappa \partial_\kappa G^\alpha- \epsilon^{\alpha \beta \gamma \delta} u_\beta \partial_\gamma u^\kappa \partial_\kappa {W}_{\delta}
  + \epsilon^{\alpha \beta \gamma \delta} u^\kappa \partial_\kappa u_\beta \partial_\gamma {W}_\delta.
\end{align}
From \eqref{CEq3} and \eqref{LCEq}, we can see
\begin{equation}\label{RCEq}
	\begin{split}
		u^\kappa \partial_\kappa G^\alpha
		&=  \sum^{8}_{a=0}\mathrm{R}_a+ \epsilon^{\alpha \beta \gamma \delta} u_\beta \partial_\gamma u^\kappa \partial_\kappa {W}_{\delta}
		- \epsilon^{\alpha \beta \gamma \delta} u^\kappa \partial_\kappa u_\beta \partial_\gamma {W}_\delta.
	\end{split}
\end{equation}
While, there are some difficult terms in $\mathrm{R}_0, \mathrm{R}_1, \cdots, \mathrm{R}_8$. So we need to discuss them in a proper way. For $\mathrm{R}_0$, by the chain rule and \eqref{REE}, we have
\begin{equation}\label{cr00}
\begin{split}
  \mathrm{R}_0 =& -2\epsilon^{\alpha \beta \gamma \delta}u_\beta \partial_\gamma\left(  c^{-2}_su^\kappa \partial_\kappa h {W}_\delta \right)
  \\
  =& 4 \epsilon^{\alpha \beta \gamma \delta}c_s^{-3}c'_s u_\beta u^\kappa \partial_\gamma h  \partial_\kappa h W_\delta
   - 2\epsilon^{\alpha \beta \gamma \delta} c_s^{-2} u_\beta u^\kappa \partial_\kappa h \partial_\gamma W_\delta
   -2\epsilon^{\alpha \beta \gamma \delta} c_s^{-2} u_\beta u^\kappa \partial_\kappa(\partial_\gamma h) W_\delta
   \\
   =&- 4 \epsilon^{\alpha \beta \gamma \delta}c_s^{-1}c'_s u_\beta  \partial_\gamma h  \partial_\kappa u^\kappa W_\delta
   - 2\epsilon^{\alpha \beta \gamma \delta} c_s^{-2} u_\beta u^\kappa \partial_\kappa h \partial_\gamma W_\delta
   -2\epsilon^{\alpha \beta \gamma \delta} c_s^{-2} u_\beta u^\kappa \partial_\kappa(\partial_\gamma h) W_\delta .
\end{split}
\end{equation}
Noting the last right term in \eqref{cr00}, it can be written as
\begin{equation}\label{cr01}
  \begin{split}
    -2\epsilon^{\alpha \beta \gamma \delta} c_s^{-2} u_\beta u^\kappa \partial_\kappa(\partial_\gamma h) W_\delta
   =&u^\kappa \partial_\kappa ( -2\epsilon^{\alpha \beta \gamma \delta} c_s^{-2} u_\beta \partial_\gamma h W_\delta )
   -4\epsilon^{\alpha \beta \gamma \delta} c_s^{-3}c'_s u^\kappa u_\beta \partial_\kappa h \partial_\gamma h W_\delta
   \\
   & + 2\epsilon^{\alpha \beta \gamma \delta} c_s^{-2} u^\kappa \partial_\kappa u_\beta \partial_\gamma h W_\delta
   + 2\epsilon^{\alpha \beta \gamma \delta} c_s^{-2} u_\beta u^\kappa  \partial_\gamma h \partial_\kappa W_\delta.
  \end{split}
\end{equation}
Gathering \eqref{cr00} and \eqref{cr01}, we then get
\begin{equation}\label{R0}
\begin{split}
  \mathrm{R}_0
  =& u^\kappa \partial_\kappa ( -2\epsilon^{\alpha \beta \gamma \delta} c_s^{-2} u_\beta \partial_\gamma h W_\delta )
    + 2\epsilon^{\alpha \beta \gamma \delta} c_s^{-2} u^\kappa \partial_\kappa u_\beta \partial_\gamma h W_\delta
   \\
   & + 2\epsilon^{\alpha \beta \gamma \delta} c_s^{-2} u_\beta u^\kappa  \partial_\gamma h \partial_\kappa W_\delta
  - 2\epsilon^{\alpha \beta \gamma \delta} c_s^{-2} u_\beta u^\kappa \partial_\kappa h \partial_\gamma W_\delta
  \\
  =& u^\kappa \partial_\kappa ( -2\epsilon^{\alpha \beta \gamma \delta} c_s^{-2} u_\beta \partial_\gamma h W_\delta )
  + 2\epsilon^{\alpha \beta \gamma \delta} c_s^{-2} u^\kappa \partial_\kappa u_\beta \partial_\gamma h W_\delta
  \\
  & + 2\epsilon^{\alpha \beta \gamma \delta} c_s^{-2} u_\beta u^\kappa  \partial_\gamma h \partial_\kappa W_\delta
  + 2\epsilon^{\alpha \beta \gamma \delta} u_\beta \partial_\kappa u^\kappa  \partial_\gamma W_\delta.
\end{split}
\end{equation}
For $\mathrm{R}_1$, we have
\begin{align}\label{R1a}
  \mathrm{R}_1&=-\epsilon^{\alpha \beta \gamma \delta}u_\beta \partial_\gamma( {W}^\kappa \partial_\kappa u_\delta) \nonumber
  \\
  &=-\epsilon^{\alpha \beta \gamma \delta}u_\beta \partial_\gamma {W}^\kappa \partial_\kappa u_\delta-\epsilon^{\alpha \beta \gamma \delta}u_\beta {W}^\kappa \partial_\gamma(\partial_\kappa u_\delta) \nonumber
  \\
  &=-\epsilon^{\alpha \beta \gamma \delta}u_\beta  \partial_\kappa u_\delta \partial_\gamma {W}^\kappa+{W}^\kappa \partial_\kappa(-\epsilon^{\alpha \beta \gamma \delta}u_\beta \partial_\gamma u_\delta)  +  \epsilon^{\alpha \beta \gamma \delta} {W}^\kappa \partial_\kappa u_\beta \partial_\gamma u_\delta \nonumber
  \\
  &= {W}^\kappa \partial_\kappa(\mathrm{vort}^\alpha \bu)-\epsilon^{\alpha \beta \gamma \delta}u_\beta  \partial_\kappa u_\delta \partial_\gamma {W}^\kappa+\epsilon^{\alpha \beta \gamma \delta} {W}^\kappa \partial_\kappa u_\beta \partial_\gamma u_\delta \nonumber
  \\
  &= \mathrm{e}^{-h} {W}^\kappa \partial_\kappa w^\alpha- \mathrm{e}^{-h}w^\alpha {W}^\kappa \partial_\kappa h
  - \epsilon^{\alpha \beta \gamma \delta}u_\beta \partial_\kappa u_\delta \partial_\gamma {W}^\kappa+\epsilon^{\alpha \beta \gamma \delta} {W}^\kappa \partial_\kappa u_\beta \partial_\gamma u_\delta .
\end{align}
On the other hand, we also have
\begin{equation}\label{R1b}
	\begin{split}
		\mathrm{e}^{-h} {W}^\kappa \partial_\kappa w^\alpha= &\mathrm{e}^{-h} {W}^\kappa ( \partial_\kappa w^\alpha- \partial^\alpha w_\kappa)+ \mathrm{e}^{-h} {W}^\kappa  \partial^\alpha w_\kappa
	\end{split}
\end{equation}
Inserting \eqref{cr04} to \eqref{R1b}, it yields
\begin{equation}\label{R1c}
	\begin{split}
	 \mathrm{e}^{-h} {W}^\kappa \partial_\kappa w^\alpha
	= &\mathrm{e}^{-h} {W}^\kappa ( \epsilon_{\kappa \ \gamma \delta}^{ \ \alpha} u^\gamma \mathrm{vort}^\delta \bw-u_\kappa u^\gamma \partial_\gamma w^\alpha
     + u_\kappa u^\gamma \partial^\alpha w_\gamma
	 \\
	 & \ + u^\alpha u^\gamma \partial_\gamma w_\kappa - u^\alpha u^\gamma \partial_\kappa w_\gamma)+ \mathrm{e}^{-h} {W}^\kappa  \partial^\alpha w_\kappa
	\\
	= & \mathrm{e}^{-h} {W}^\kappa  \epsilon_{\kappa \ \gamma \delta}^{ \ \alpha} u^\gamma \mathrm{vort}^\delta \bw + \mathrm{e}^{-h} {W}^\kappa ( -u_\kappa u^\gamma \partial_\gamma w^\alpha
	+ u_\kappa u^\gamma \partial^\alpha w_\gamma
	 + u^\alpha u^\gamma \partial_\gamma w_\kappa
	 \\
	 & \ - u^\alpha u^\gamma \partial_\kappa w_\gamma)
	  + \partial^\alpha( \mathrm{e}^{-h} {W}^\kappa  w_\kappa )
	 +  \mathrm{e}^{-h} w_\kappa {W}^\kappa \partial^\alpha h  -  \mathrm{e}^{-h} w_\kappa \partial^\alpha{W}^\kappa   .
	\end{split}
\end{equation}
Seeing from \eqref{MFd}, we can derive that
\begin{equation}\label{R1d}
	\begin{split}
		\mathrm{e}^{-h} {W}^\kappa  \epsilon_{\kappa \ \gamma \delta}^{ \ \alpha} u^\gamma \mathrm{vort}^\delta \bw=& \mathrm{e}^{-h} {W}^\kappa  \epsilon_{\kappa \ \gamma \delta}^{ \ \alpha} u^\gamma (W^\delta- c_s^{-2} \epsilon^{\delta \beta \mu \nu}u_\beta w_\nu \partial_\mu )
		\\
		=& -\epsilon_{\kappa \ \gamma \delta}^{ \ \alpha}  \epsilon^{\delta \beta \mu \nu} c_s^{-2}  \mathrm{e}^{-h}  u^\gamma  u_\beta w_\nu {W}^\kappa \partial_\mu h.
	\end{split}
\end{equation}
Combing \eqref{R1a}, \eqref{R1d}, and \eqref{R1c}, we get
\begin{equation}\label{R1}
	\begin{split}
		\mathrm{R}_1 = & -\epsilon_{\kappa \ \gamma \delta}^{ \ \alpha}  \epsilon^{\delta \beta \mu \nu} c_s^{-2}  \mathrm{e}^{-h}  u^\gamma  u_\beta w_\nu {W}^\kappa \partial_\mu h
		+ \mathrm{e}^{-h} {W}^\kappa ( -u_\kappa u^\gamma \partial_\gamma w^\alpha
		+ u_\kappa u^\gamma \partial^\alpha w_\gamma )
		\\
		&	+ \mathrm{e}^{-h} {W}^\kappa ( u^\alpha u^\gamma \partial_\gamma w_\kappa - u^\alpha u^\gamma \partial_\kappa w_\gamma)	+  \mathrm{e}^{-h} w_\kappa {W}^\kappa \partial^\alpha h  -  \mathrm{e}^{-h} w_\kappa \partial^\alpha{W}^\kappa
		\\
		& + \partial^\alpha( \mathrm{e}^{-h} {W}^\kappa  w_\kappa )
		- \epsilon^{\alpha \beta \gamma \delta}u_\beta \partial_\kappa u_\delta \partial_\gamma {W}^\kappa+\epsilon^{\alpha \beta \gamma \delta} {W}^\kappa \partial_\kappa u_\beta \partial_\gamma u_\delta .
	\end{split}
\end{equation}
As for $\mathrm{R}_2$, we have
\begin{align}\label{R2}
  \mathrm{R}_2=& - \epsilon^{\alpha \beta \gamma \delta}u_\beta \partial_\gamma(u_\delta {W}^\kappa u^\lambda \partial_\lambda u_\kappa)
  \nonumber
  \\
  = &- \epsilon^{\alpha \beta \gamma \delta}u_\beta \partial_\gamma u_\delta {W}^\kappa u^\lambda \partial_\lambda u_\kappa- \epsilon^{\alpha \beta \gamma \delta}u_\beta u_\delta \partial_\gamma{W}^\kappa u^\lambda \partial_\lambda u_\kappa
  \nonumber
  \\
  &- \epsilon^{\alpha \beta \gamma \delta}u_\beta u_\delta {W}^\kappa \partial_\gamma u^\lambda \partial_\lambda u_\kappa
  - \epsilon^{\alpha \beta \gamma \delta}u_\beta u_\delta {W}^\kappa \partial_\gamma u^\lambda \partial_\gamma(\partial_\lambda u_\kappa)
  \nonumber
\\
=& (- \epsilon^{\alpha \beta \gamma \delta}u_\beta  \partial_\gamma u_\delta)   {W}^\kappa ( u^\lambda \partial_\lambda u_\kappa )
\nonumber
\\
=& \mathrm{e}^{-h}w^\alpha   {W}^\kappa ( -\partial_\kappa h - u_\kappa u^\lambda \partial_\lambda h ) \nonumber
\\
=& -\mathrm{e}^{-h}w^\alpha   {W}^\kappa  \partial_\kappa h.
\end{align}
For $\mathrm{R}_3$, we deduce that
\begin{equation*}
  \begin{split}
  \mathrm{R}_3=&2\epsilon^{\alpha \beta \gamma \delta}\epsilon_{\delta \eta \mu \nu}u_\beta \partial_\gamma \big( u^{\eta} \partial^{\mu} w^\kappa \partial^{\nu} u_\kappa \big)
  \\
  =& -2u_\beta \partial_\gamma \big( u^{\eta}\partial^{\mu} w^\kappa \partial^{\nu} u_\kappa \big) \left( \delta^\alpha_{\eta} \delta^\gamma_{\mu} \delta^\beta_{\nu}- \delta^\alpha_{\eta} \delta^{\beta}_{\mu}\delta^\gamma_{\nu}
 + \delta^\beta_{\eta} \delta^\alpha_{\mu} \delta^\gamma_{\nu}-\delta^\beta_{\eta}\delta^\gamma_{\mu}\delta^\alpha_{\nu}
   + \delta^\gamma_{\eta} \delta^\beta_{\mu} \delta^\alpha_{\nu}- \delta^\gamma_{\eta} \delta^\alpha_{\mu} \delta^\beta_{\nu} \right)
  \\
  =& -2   \underbrace{ u_\beta \partial_\gamma ( u^{\alpha}\partial^{\gamma} w^\kappa \partial^{\beta} u_\kappa ) }_{\equiv \mathrm{R}_{31}}
   +2  \underbrace{  u_\beta \partial_\gamma ( u^{\alpha}\partial^{\beta} w^\kappa \partial^{\gamma} u_\kappa ) }_{\equiv \mathrm{R}_{32}}  -2 \underbrace{ u_\beta \partial_\gamma ( u^{\beta}\partial^{\alpha} w^\kappa \partial^{\gamma} u_\kappa ) }_{\equiv \mathrm{R}_{33}}
   \\
   & + 2  \underbrace{ u_\beta \partial_\gamma ( u^{\beta}\partial^{\gamma} w^\kappa \partial^{\alpha} u_\kappa ) }_{\equiv \mathrm{R}_{34}}
    -2 \underbrace{ u_\beta \partial_\gamma ( u^{\gamma}\partial^{\beta} w^\kappa \partial^{\alpha} u_\kappa ) }_{\equiv \mathrm{R}_{35}}
   + 2 \underbrace{ u_\beta \partial_\gamma ( u^{\gamma}\partial^{\alpha} w^\kappa \partial^{\beta} u_\kappa ) }_{\equiv \mathrm{R}_{36}}.
  \end{split}
\end{equation*}
So we have
\begin{equation}\label{r3e}
	\mathrm{R}_3=-2\mathrm{R}_{31}+2\mathrm{R}_{32}-2\mathrm{R}_{33}+2\mathrm{R}_{34}-2\mathrm{R}_{35}+2\mathrm{R}_{36}.
\end{equation}
Let us now calculate $\mathrm{R}_{31}, \mathrm{R}_{32}, \cdots$, and $\mathrm{R}_{36}$. For $\mathrm{R}_{31}$, we firstly have
\begin{equation}\label{cr02}
  \mathrm{R}_{31}=  u_\beta \partial_\gamma u^{\alpha}\partial^{\gamma} w^\kappa \partial^{\beta} u_\kappa
                    +  u_\beta  u^{\alpha} \partial_\gamma(\partial^{\gamma} w^\kappa) \partial^{\beta} u_\kappa
                   +  u_\beta  u^{\alpha}\partial^{\gamma} w^\kappa \partial_\gamma( \partial^{\beta} u_\kappa ).
\end{equation}
Secondly, we can also get
\begin{align}\label{cr03}
   u_\beta  u^{\alpha}\partial^{\gamma} w^\kappa \partial_\gamma( \partial^{\beta} u_\kappa )
  =& u^{\alpha}\partial^{\gamma} w^\kappa u^\beta \partial_\beta( \partial_{\gamma} u_\kappa ) \nonumber
  \\
  =& u^\beta \partial_\beta \big(   u^{\alpha}\partial^{\gamma} w^\kappa \partial_{\gamma} u_\kappa \big)
   -  u^\beta  \partial^{\gamma} w^\kappa  \partial_\beta u^\alpha \partial_{\gamma} u_\kappa
    - u^\beta u^{\alpha} \partial_{\gamma} u_\kappa \partial_\beta (\partial^{\gamma} w^\kappa) \nonumber
   \\
   =&u^\kappa \partial_\kappa \big(   u^{\alpha}\partial^{\gamma} w^\lambda \partial_{\gamma} u_\lambda \big)-  u^\beta  \partial^{\gamma} w^\kappa  \partial_\beta u^\alpha \partial_{\gamma} u_\kappa
    - u^\beta u^{\alpha} \partial_{\gamma} u_\kappa \partial_\beta (\partial^{\gamma} w^\kappa).
\end{align}
Substituting \eqref{cr03} to \eqref{cr02}, it follows
\begin{equation}\label{R31}
  \begin{split}
  \mathrm{R}_{31}=& u^\kappa \partial_\kappa \big(   u^{\alpha}\partial^{\gamma} w^\lambda \partial_{\gamma} u_\lambda \big)
  -  u^\beta  \partial^{\gamma} w^\kappa  \partial_\beta u^\alpha \partial_{\gamma} u_\kappa
  +  u_\beta \partial_\gamma u^{\alpha}\partial^{\gamma} w^\kappa \partial^{\beta} u_\kappa
                   \\
                   & +   u_\beta  u^{\alpha} \partial_\gamma(\partial^{\gamma} w^\kappa) \partial^{\beta} u_\kappa-u^\beta u^{\alpha} \partial_{\gamma} u_\kappa \partial_\beta (\partial^{\gamma} w^\kappa)  .
  \end{split}
\end{equation}
For $\mathrm{R}_{32}$, we can compute it by
\begin{equation}\label{R320}
  \begin{split}
  \mathrm{R}_{32}=& u_\beta \partial_\gamma u^{\alpha}\partial^{\beta} w^\kappa \partial^{\gamma} u_\kappa
                    +  u_\beta  u^{\alpha} \partial_\gamma(\partial^{\beta} w^\kappa) \partial^{\gamma} u_\kappa
                   +  u^{\alpha}u_\beta\partial^{\beta} w^\kappa \partial_\gamma( \partial^{\gamma} u_\kappa ).
  \end{split}
\end{equation}
Using \eqref{CEQ}, we can rewrite \eqref{R320} as
\begin{equation}\label{R32}
  \begin{split}
  \mathrm{R}_{32}=&  u_\beta \partial_\gamma u^{\alpha}\partial^{\beta} w^\kappa \partial^{\gamma} u_\kappa
                    +  u_\beta  u^{\alpha} \partial^{\gamma} u_\kappa \partial_\gamma(\partial^{\beta} w^\kappa)
                    -  u^{\alpha}u^\kappa w^\lambda \partial_\lambda h \partial_\gamma( \partial^{\gamma} u_\kappa )
                   \\
                   & + u^{\alpha} w^\lambda \partial_\lambda u^\kappa  \partial_\gamma( \partial^{\gamma} u_\kappa )
- u^{\alpha}w^\kappa \partial_\lambda u^\lambda  \partial_\gamma( \partial^{\gamma} u_\kappa ).
  \end{split}
\end{equation}
For $\mathrm{R}_{33}$ is the most difficult term, we postpone to consider it. We turn to calculate $\mathrm{R}_{34}$ now. Using the chain rule, $u_\beta u^{\beta}=-1$, and $u_\beta \partial_\gamma u^{\beta}=0$, we have
\begin{align}\label{R340}
  \mathrm{R}_{34}=&u_\beta \partial_\gamma ( u^{\beta}\partial^{\gamma} w^\kappa \partial^{\alpha} u_\kappa ) \nonumber
  \\
  =&  u_\beta \partial_\gamma u^{\beta}  \partial^{\gamma} w^\kappa \partial^{\alpha} u_\kappa
   +  u_\beta u^{\beta}  \partial^\gamma ( \partial_{\gamma} w^\kappa )\partial^{\alpha} u_\kappa
   +  u_\beta u^{\beta}  \partial^\gamma w^\kappa \partial_{\gamma} (\partial^{\alpha} u_\kappa) \nonumber
  \\
  =& -  \partial_\gamma (\partial^{\gamma} w^\kappa) \partial^{\alpha} u_\kappa -  \partial^{\gamma} w^\kappa \partial_{\gamma} ( \partial^\alpha u_\kappa).
\end{align}
On the other hand, we can rewrite
\begin{equation}\label{R341}
   -  \partial^{\gamma} w^\kappa \partial_{\gamma} ( \partial^\alpha u_\kappa)=  \partial^\alpha (-\partial_\gamma u_\kappa \partial^{\gamma} w^\kappa ) +  \partial_\gamma u_\kappa \partial^\alpha(\partial^{\gamma} w^\kappa).
\end{equation}
Inserting \eqref{R341} to \eqref{R340}, we obtain
\begin{equation}\label{R34}
  \mathrm{R}_{34}=
   \partial^\alpha (-  \partial^{\gamma} w^\kappa \partial_\gamma u_\kappa)-  \partial^{\alpha} u_\kappa \partial^\gamma \partial_{\gamma} w^\kappa
                     +  \partial_\gamma u_\kappa \partial^\alpha \partial^{\gamma} w^\kappa.
\end{equation}
For $\mathrm{R}_{35}$, by the chain rule and \eqref{CEQ}, we have
\begin{align}\label{R35}
  \mathrm{R}_{35}=& u_\beta \partial_\gamma ( u^{\gamma}\partial^{\beta} w^\kappa \partial^{\alpha} u_\kappa ) \nonumber
  \\
  =&  u_\beta \partial_\gamma u^{\gamma}\partial^{\beta} w^\kappa \partial^{\alpha} u_\kappa
                    + u_\beta  u^{\gamma} \partial^\beta(\partial_{\gamma} w^\kappa) \partial^{\alpha} u_\kappa
                   +  u_\beta\partial^{\beta} w^\kappa u^{\gamma}\partial_\gamma( \partial^{\alpha} u_\kappa ) \nonumber
\\
=&  u_\beta \partial_\gamma u^{\gamma}\partial^{\beta} w^\kappa \partial^{\alpha} u_\kappa
                    +  u_\beta  u^{\gamma} \partial^{\alpha} u_\kappa \partial^\beta(\partial_{\gamma} w^\kappa)
                    -  u^\kappa u^{\gamma} w^\lambda \partial_\lambda h \partial_\gamma( \partial^{\alpha} u_\kappa )   \nonumber
                   \\
                   & + u^{\gamma} w^\lambda \partial_\lambda u^\kappa \partial_\gamma( \partial^{\alpha} u_\kappa )
                    - u^{\gamma} w^\kappa \partial_\lambda u^\lambda \partial_\gamma( \partial^{\alpha} u_\kappa ).
\end{align}
For $\mathrm{R}_{36}$, a direct calculation tells us
\begin{align}\label{cr05}
  \mathrm{R}_{36}=& u_\beta \partial_\gamma ( u^{\gamma}\partial^{\alpha} w^\kappa \partial^{\beta} u_\kappa ) \nonumber
 \\
  =& u_\beta \partial_\gamma u^{\gamma}\partial^{\alpha} w^\kappa \partial^{\beta} u_\kappa
                   +  u_\beta  u^{\gamma} \partial^\alpha(\partial_{\gamma} w^\kappa) \partial^{\beta} u_\kappa
                   +  u^\gamma \partial^{\alpha} w^\kappa u_{\beta}\partial^\beta( \partial_{\gamma} u_\kappa ).
\end{align}
Note
\begin{align}\label{cr06}
   u^\gamma \partial^{\alpha} w^\kappa u_{\beta}\partial^\beta( \partial_{\gamma} u_\kappa )
  =&  u^\gamma \partial^{\alpha} w^\beta u^{\kappa}\partial_\kappa( \partial_{\gamma} u_\beta ) \nonumber
  \\
  =& u^{\kappa}\partial_\kappa ( u^\gamma \partial^{\alpha} w^\beta \partial_{\gamma} u_\beta)
   - u^{\kappa}\partial_\kappa u^\gamma \partial_{\gamma} u_\beta \partial^{\alpha} w^\beta
   -  u^{\kappa}u^\gamma \partial_{\gamma} u_\beta \partial^{\alpha}(\partial_\kappa w^\beta).
\end{align}
Combining \eqref{cr05} with \eqref{cr06}, we therefore have
\begin{equation}\label{R36}
  \begin{split}
  \mathrm{R}_{36}= & u^{\kappa}\partial_\kappa ( u^\gamma \partial^{\alpha} w^\beta \partial_{\gamma} u_\beta)- u^{\kappa}\partial_\kappa u^\gamma \partial_{\gamma} u_\beta \partial^{\alpha} w^\beta +  u_\beta  u^{\gamma} \partial^\alpha(\partial_{\gamma} w^\kappa) \partial^{\beta} u_\kappa
   \\
   &-  u^{\kappa}u^\gamma \partial_{\gamma} u_\beta \partial^{\alpha}(\partial_\kappa w^\beta)
  +u_\beta \partial_\gamma u^{\gamma}\partial^{\alpha} w^\kappa \partial^{\beta} u_\kappa .
  \end{split}
\end{equation}
It still remains for us to calculate $\mathrm{R}_{33}$. Using the chain rule, $u_\beta u^{\beta}=-1$, and $u_\beta \partial_\gamma u^{\beta}=0$, it follows that
\begin{align}\label{cr07}
  \mathrm{R}_{33} =&  u_\beta \partial_\gamma ( u^{\beta}\partial^{\alpha} w^\kappa \partial^{\gamma} u_\kappa ) \nonumber
  \\
  =& u_\beta \partial_\gamma  u^{\beta}  \partial^{\alpha} w^\kappa \partial^{\gamma} u_\kappa
  +  u_\beta  u^{\beta}   \partial^\gamma u_{\kappa}\partial^{\alpha} (\partial_{\gamma} w^\kappa )
                   + u_\beta  u^{\beta}  \partial^\alpha w^\kappa  \partial_{\gamma}(\partial^{\gamma} u_\kappa)
                   \nonumber
    \\
    =& - \partial^\gamma u_{\kappa}\partial^{\alpha} \partial_{\gamma} w^\kappa -   \partial^\alpha w^\lambda  \partial_{\gamma}\partial^{\gamma} u_\lambda.
\end{align}
By using \eqref{HDe} and $u_\beta \partial_\lambda u^\beta=0$, we can compute out
\begin{align}\label{cr08}
    &-   \partial^\alpha w^\lambda  \partial_{\gamma}(\partial^{\gamma} u_\lambda) \nonumber
    \\
  =& -   \partial^\alpha w^\lambda  \big\{ \mathrm{vort}_\lambda ( \mathrm{vort} \bu)+ \partial_\lambda (\partial_\gamma u^\gamma)
  -u^\kappa \partial_\kappa (u^\gamma \partial_\gamma u_\lambda- u_\lambda \partial_\gamma u^\gamma  )
  \nonumber
  \\
  &  +2u_\beta \partial_\gamma u_\lambda \partial^\beta u^\gamma - 2u_\beta \partial_\gamma u^\gamma \partial^\beta u_\lambda- u_\lambda \partial_\gamma u_\beta \partial^\gamma u^\beta
   + u^\gamma \partial_\gamma u_\beta \partial_\lambda u^\beta \big\}
  \nonumber
  \\
  =& -   \partial^\alpha w^\lambda  \big\{ \mathrm{vort}_\lambda ( \mathrm{vort} \bu)+ \partial_\lambda (\partial_\gamma u^\gamma) \}
  +  u^\kappa \partial_\kappa \big \{  \partial^\alpha w^\lambda ( u^\gamma \partial_\gamma u_\lambda- u_\lambda \partial_\gamma u^\gamma  ) \big \} \nonumber
  \\
  &
    -  ( 2u_\beta \partial_\gamma u_\lambda \partial^\beta u^\gamma - 2u_\beta \partial_\gamma u^\gamma \partial^\beta u_\lambda- u_\lambda \partial_\gamma u_\beta \partial^\gamma u^\beta + u^\gamma \partial_\gamma u_\beta \partial_\lambda u^\beta ) \partial^\alpha w^\lambda \nonumber
    \\
   &  -u^\kappa  ( u^\gamma \partial_\gamma u_\lambda+ u_\lambda \partial_\gamma u^\gamma  ) \partial_\kappa \partial^\alpha w^\lambda .
\end{align}
By using \eqref{VAd}, we have
\begin{equation}\label{cr09}
	 \partial^\alpha w^\lambda \{ \mathrm{vort}_\lambda ( \mathrm{vort} \bu) \}=\partial^\alpha ( w^\lambda \mathrm{vort}_\lambda \bw)-  w^\lambda \partial^\alpha (\mathrm{vort}_\lambda \bw).
\end{equation}
Using \eqref{REE}, we also derive
\begin{align}\label{cr10}
    \partial^\alpha w^\lambda  \{  \partial_\lambda (\partial_\gamma u^\gamma) \}
  =& -  \partial^\alpha w^\lambda \partial_\lambda (c^{-2}_s u^\kappa \partial_\kappa h) \nonumber
  \\
  =&   \partial^\alpha w^\lambda  \big \{ - c_s^{-2} \partial_\lambda u^\kappa \partial_\kappa h +c_s^{-3}c'_s \partial_\lambda h u^\kappa \partial_\kappa h- c_s^{-2}  u^\kappa \partial_\kappa (\partial_\lambda h) \big\} \nonumber
  \\
  =&
  - c_s^{-2} \partial_\lambda u^\kappa \partial_\kappa h \partial^\alpha w^\lambda
 -c_s^{-3}c'_s  u^\kappa \partial_\kappa h \partial_\lambda h \partial^\alpha w^\lambda \nonumber
  \\
  & - u^\kappa \partial_\kappa (c_s^{-2}\partial_\lambda h \partial^\alpha w^\lambda  )+c_s^{-2}  \partial_\lambda h  u^\kappa \partial_\kappa(\partial^\alpha w^\lambda)+u^\kappa \partial_\kappa (c_s^{-2}) \partial_\lambda h \partial^\alpha w^\lambda
  \nonumber
  \\
  =&
  - c_s^{-2} \partial_\lambda u^\kappa \partial_\kappa h \partial^\alpha w^\lambda
  +3c_s^{-1}c'_s  \partial_\kappa u^\kappa  \partial_\lambda h \partial^\alpha w^\lambda \nonumber
  \\
  & + u^\kappa \partial_\kappa (-c_s^{-2}\partial_\lambda h \partial^\alpha w^\lambda  )+c_s^{-2}  \partial_\lambda h  u^\kappa \partial_\kappa(\partial^\alpha w^\lambda) .
\end{align}
Combining \eqref{cr07}, \eqref{cr08}, \eqref{cr09}, and  \eqref{cr10}, we obtain
\begin{equation}\label{R33}
  \begin{split}
  \mathrm{R}_{33} =&  u^\kappa \partial_\kappa \left \{  \partial^\alpha w^\lambda ( u^\gamma \partial_\gamma u_\lambda- u_\lambda \partial_\gamma u^\gamma  + c_s^{-2}\partial_\lambda h ) \right \}
    - \partial^\gamma u_{\kappa}\partial^{\alpha} \partial_{\gamma} w^\kappa
    \\
    & -  ( 2u_\beta \partial_\gamma u_\lambda \partial^\beta u^\gamma - 2u_\beta \partial_\gamma u^\gamma \partial^\beta u_\lambda- u_\lambda \partial_\gamma u_\beta \partial^\gamma u^\beta + u^\gamma \partial_\gamma u_\beta \partial_\lambda u^\beta ) \partial^\alpha w^\lambda
    \\
    & + c_s^{-2} \partial_\lambda u^\kappa \partial_\kappa h \partial^\alpha w^\lambda
    -3c_s^{-1}c'_s  \partial_\kappa u^\kappa  \partial_\lambda h \partial^\alpha w^\lambda-c_s^{-2}  u^\kappa \partial_\lambda h   \partial_\kappa(\partial^\alpha w^\lambda)
    \\
    & -\partial^\alpha ( w^\lambda \mathrm{vort}_\lambda \bw)+  w^\lambda \partial^\alpha (\mathrm{vort}_\lambda \bw)  -u^\kappa  ( u^\gamma \partial_\gamma u_\lambda- u_\lambda \partial_\gamma u^\gamma  ) \partial_\kappa ( \partial^\alpha w^\lambda) .
  \end{split}
\end{equation}
Combining \eqref{r3e}, \eqref{R31}, \eqref{R32}, \eqref{R33}, \eqref{R34}, \eqref{R35} and \eqref{R36}, we have
\begin{equation}\label{R3}
  \begin{split}
  \mathrm{R}_{3}
  =&\partial^\alpha \Gamma+ u^\kappa \partial_\kappa\left(  -2u^\alpha \partial^\gamma w^\lambda \partial_\gamma u_\lambda+2 u_\lambda \partial_\gamma u^\gamma \partial^\alpha w^\lambda - 2c_s^{-2}\partial_\lambda h \partial^\alpha w^\lambda \right)
   +2  u^\beta  \partial^{\gamma} w^\kappa  \partial_\beta u^\alpha \partial_{\gamma} u_\kappa
  \\
  & -2  u_\beta \partial_\gamma u^{\alpha}\partial^{\gamma} w^\kappa \partial^{\beta} u_\kappa
   -2  u_\beta  u^{\alpha} \partial_\gamma(\partial^{\gamma} w^\kappa) \partial^{\beta} u_\kappa-2u^\beta u^{\alpha} \partial_{\gamma} u_\kappa \partial_\beta (\partial^{\gamma} w^\kappa)
  \\
  & +2  u_\beta \partial_\gamma u^{\alpha}\partial^{\beta} w^\kappa \partial^{\gamma} u_\kappa
  + 2 u_\beta  u^{\alpha} \partial^{\gamma} u_\kappa \partial_\gamma(\partial^{\beta} w^\kappa)
  - 2 u^{\alpha}u^\kappa w^\lambda \partial_\lambda h \partial_\gamma( \partial^{\gamma} u_\kappa )
  \\
  & +2 u^{\alpha} w^\lambda \partial_\lambda u^\kappa  \partial_\gamma( \partial^{\gamma} u_\kappa )
  - 2 u^{\alpha}w^\kappa \partial_\lambda u^\lambda  \partial_\gamma( \partial^{\gamma} u_\kappa )
  +4\partial^\gamma u_{\kappa}\partial^{\alpha} \partial_{\gamma} w^\kappa
  \\
  & +  ( 4u_\beta \partial_\gamma u_\lambda \partial^\beta u^\gamma - 4u_\beta \partial_\gamma u^\gamma \partial^\beta u_\lambda- 2u_\lambda \partial_\gamma u_\beta \partial^\gamma u^\beta +2 u^\gamma \partial_\gamma u_\beta \partial_\lambda u^\beta ) \partial^\alpha w^\lambda
  \\
  & -2 c_s^{-2} \partial_\lambda u^\kappa \partial_\kappa h \partial^\alpha w^\lambda
  +6c_s^{-1}c'_s  \partial_\kappa u^\kappa  \partial_\lambda h \partial^\alpha w^\lambda+2c_s^{-2}  u^\kappa \partial_\lambda h   \partial_\kappa(\partial^\alpha w^\lambda)
  \\
  & + 2 \partial^\alpha w^\lambda \mathrm{vort}_\lambda \bw  + 2 u^\kappa  ( u^\gamma \partial_\gamma u_\lambda- u_\lambda \partial_\gamma u^\gamma  ) \partial_\kappa ( \partial^\alpha w^\lambda)
   -  2\partial^{\alpha} u_\kappa \partial^\gamma \partial_{\gamma} w^\kappa
  \\
  &
  -2 u_\beta \partial_\gamma u^{\gamma}\partial^{\beta} w^\kappa \partial^{\alpha} u_\kappa
  -2  u_\beta  u^{\gamma} \partial^{\alpha} u_\kappa \partial^\beta(\partial_{\gamma} w^\kappa)
   +2   u^\kappa u^{\gamma} w^\lambda \partial_\lambda h \partial_\gamma( \partial^{\alpha} u_\kappa )
  \\
  & -2 u^{\gamma} w^\lambda \partial_\lambda u^\kappa \partial_\gamma( \partial^{\alpha} u_\kappa )
  +2  u^{\gamma} w^\kappa \partial_\lambda u^\lambda \partial_\gamma( \partial^{\alpha} u_\kappa )
   - 2u^{\kappa}\partial_\kappa u^\gamma \partial_{\gamma} u_\beta \partial^{\alpha} w^\beta
   \\
   & +  2u_\beta  u^{\gamma} \partial^\alpha(\partial_{\gamma} w^\kappa) \partial^{\beta} u_\kappa
  -  2u^{\kappa}u^\gamma \partial_{\gamma} u_\beta \partial^{\alpha}(\partial_\kappa w^\beta)
  +2u_\beta \partial_\gamma u^{\gamma}\partial^{\alpha} w^\kappa \partial^{\beta} u_\kappa.
 \end{split}
\end{equation}
It remains for us to consider $\mathrm{R}_4, \mathrm{R}_5, \mathrm{R}_6, \mathrm{R}_7$ and $\mathrm{R}_8$. They don't have derivatives loss given energy estimates, so it's enough for us to use chain rules. For $\mathrm{R}_4$, we have
\begin{equation}\label{R4}
	\begin{split}
	\mathrm{R}_4=
	&  -2 \epsilon^{\alpha \beta \gamma \delta}  \epsilon_{\delta}^{ \ \eta \mu \nu} c_s^{-3} c'_s u_\beta  u_\eta w_\nu \partial_\gamma h    \partial_\mu  u^\kappa \partial_\kappa h + \epsilon^{\alpha \beta \gamma \delta}  \epsilon_{\delta}^{ \ \eta \mu \nu} c_s^{-2} u_\beta \partial_\gamma  u_\eta w_\nu   \partial_\mu u^\kappa \partial_\kappa h
	\\
	& + \epsilon^{\alpha \beta \gamma \delta}  \epsilon_{\delta}^{ \ \eta \mu \nu} c_s^{-2} u_\beta  u_\eta \partial_\gamma w_\nu   \partial_\mu u^\kappa \partial_\kappa h +  \epsilon^{\alpha \beta \gamma \delta}  \epsilon_{\delta}^{ \ \eta \mu \nu} c_s^{-2} u_\beta  u_\eta  w_\nu  \partial_\kappa h \partial_\gamma \partial_\mu u^\kappa
	\\
	& + \epsilon^{\alpha \beta \gamma \delta}  \epsilon_{\delta}^{ \ \eta \mu \nu} c_s^{-2} u_\beta  u_\eta w_\nu   \partial_\mu u^\kappa \partial_\gamma \partial_\kappa h .
	\end{split}
\end{equation}
Similarly, we can compute out
\begin{equation}\label{R5}
	\begin{split}
		\mathrm{R}_5=&  - 2 \epsilon^{\alpha \beta \gamma \delta} \mathrm{e}^{-h} u_\beta    w_\delta w^\kappa \partial_\kappa h \partial_\gamma h
		+ 2 \epsilon^{\alpha \beta \gamma \delta} \mathrm{e}^{-h} u_\beta  w^\kappa \partial_\gamma w_\delta \partial_\kappa h
		\\
		& + 2 \epsilon^{\alpha \beta \gamma \delta} \mathrm{e}^{-h} u_\beta  w_\delta \partial_\gamma w^\kappa \partial_\kappa h
		+ 2 \epsilon^{\alpha \beta \gamma \delta} \mathrm{e}^{-h}  u_\beta  w_\delta w^\kappa \partial_\gamma \partial_\kappa h,
	\end{split}
\end{equation}
and
\begin{equation}\label{R6}
	\begin{split}
		\mathrm{R}_6=& \epsilon^{\alpha \beta \gamma \delta} \epsilon^{\kappa \eta \mu \nu} u_\beta \partial_\gamma ( c_s^{-2} u_\eta w_\nu \partial_\mu h \partial_\kappa u_\delta)
		\\
		=& -2 \epsilon^{\alpha \beta \gamma \delta} \epsilon^{\kappa \eta \mu \nu} c_s^{-3} c'_s u_\beta   u_\eta w_\nu \partial_\gamma h \partial_\mu h \partial_\kappa u_\delta  + \epsilon^{\alpha \beta \gamma \delta} \epsilon^{\kappa \eta \mu \nu}  c_s^{-2} u_\beta w_\nu \partial_\gamma  u_\eta  \partial_\mu h \partial_\kappa u_\delta
		\\
		& + \epsilon^{\alpha \beta \gamma \delta} \epsilon^{\kappa \eta \mu \nu} c_s^{-2} u_\eta u_\beta \partial_\gamma w_\nu \partial_\mu h \partial_\kappa u_\delta
		+ \epsilon^{\alpha \beta \gamma \delta} \epsilon^{\kappa \eta \mu \nu} c_s^{-2} u_\beta  u_\eta w_\nu \partial_\kappa u_\delta \partial_\gamma \partial_\mu h
		\\
		& + \epsilon^{\alpha \beta \gamma \delta} \epsilon^{\kappa \eta \mu \nu} c_s^{-2} u_\beta   u_\eta w_\nu \partial_\mu h \partial_\gamma \partial_\kappa u_\delta,
	\end{split}
\end{equation}
and
\begin{equation}\label{R7}
	\begin{split}
		\mathrm{R}_7=& - \epsilon^{\alpha \beta \gamma \delta} \epsilon_{\delta}^{\ \eta \mu \nu} u_\beta \partial_\gamma ( c_s^{-2} u_\eta w_\nu \partial_\mu h \partial_\kappa u^\kappa )
		\\
		=& 2 \epsilon^{\alpha \beta \gamma \delta} \epsilon_{\delta}^{\ \eta \mu \nu} c_s^{-3} c'_s u_\beta u_\eta w_\nu \partial_\gamma h   \partial_\mu h \partial_\kappa u^\kappa - \epsilon^{\alpha \beta \gamma \delta} \epsilon_{\delta}^{\ \eta \mu \nu} c_s^{-2} u_\beta \partial_\gamma  u_\eta w_\nu \partial_\mu h \partial_\kappa u^\kappa
		\\
		& - \epsilon^{\alpha \beta \gamma \delta} \epsilon_{\delta}^{\ \eta \mu \nu} c_s^{-2}  u_\beta u_\eta \partial_\gamma  w_\nu \partial_\mu h \partial_\kappa u^\kappa   - \epsilon^{\alpha \beta \gamma \delta} \epsilon_{\delta}^{\ \eta \mu \nu} c_s^{-2}  u_\beta u_\eta w_\nu \partial_\kappa u^\kappa \partial_\gamma  \partial_\mu h
		\\
		& - \epsilon^{\alpha \beta \gamma \delta} \epsilon_{\delta}^{\ \eta \mu \nu} c_s^{-2} u_\beta u_\eta w_\nu \partial_\mu h  \partial_\gamma \partial_\kappa u^\kappa,
	\end{split}
\end{equation}
and
\begin{equation}\label{R8}
	\begin{split}
		\mathrm{R}_8=& -\epsilon^{\alpha \beta \gamma \delta} \epsilon_{\delta}^{\ \eta \mu \nu} u_\beta \partial_\gamma \{   ( c^{-2}_s +2 ) u_\eta w^\kappa \partial_\nu u_\kappa \partial_\mu h \}
		\\
		=& 2 \epsilon^{\alpha \beta \gamma \delta} \epsilon_{\delta}^{\ \eta \mu \nu} c_s^{-3} c'_s u_\beta  u_\eta w^\kappa \partial_\gamma h \partial_\nu u_\kappa \partial_\mu h  -\epsilon^{\alpha \beta \gamma \delta} \epsilon_{\delta}^{\ \eta \mu \nu}  ( c^{-2}_s +2 ) u_\beta w^\kappa \partial_\gamma   u_\eta  \partial_\nu u_\kappa \partial_\mu h
		\\
		& -\epsilon^{\alpha \beta \gamma \delta} \epsilon_{\delta}^{\ \eta \mu \nu} ( c^{-2}_s +2 ) u_\eta u_\beta \partial_\gamma   w^\kappa \partial_\nu u_\kappa \partial_\mu h  -\epsilon^{\alpha \beta \gamma \delta} \epsilon_{\delta}^{\ \eta \mu \nu} ( c^{-2}_s +2 ) u_\beta  u_\eta w^\kappa \partial_\mu h \partial_\gamma \partial_\nu u_\kappa
		\\
		& -\epsilon^{\alpha \beta \gamma \delta} \epsilon_{\delta}^{\ \eta \mu \nu} ( c^{-2}_s +2 )  u_\beta u_\eta w^\kappa \partial_\nu u_\kappa \partial_\gamma \partial_\mu h .
	\end{split}
\end{equation}
Substituting \eqref{R0}, \eqref{R1}, \eqref{R2}, \eqref{R3}, \eqref{R4}, \eqref{R5}, \eqref{R6}, \eqref{R7}, and \eqref{R8} to \eqref{RCEq}, we therefore get \eqref{SDe}.
\end{proof}
We turn to introduce a good wave formulation for $\bu_{+}$ as follows.
\begin{Lemma}\label{um}
	Let $\bu_{+}$ be defined in \eqref{uf}. Then $\bu_{+}$ satisfies the following wave equation
	\begin{equation}\label{fdr}
	\begin{split}
		\square_{\mathbf{g}} \bu_{+}
		=& \mathbf{g}^{0\alpha} \partial_\alpha \mathbf{Z}+ \bQ + \bB,
	\end{split}
\end{equation}
where $\bQ=(Q^0,Q^1,Q^2,Q^3)^{\mathrm{T}}$ has defined in \eqref{err}, $\mathbf{Z}=(Z^0,Z^1,Z^2,Z^3)^{\mathrm{T}}$ and $\mathbf{B}=(B^0,B^1,B^2,B^3)^{\mathrm{T}}$ are given by
\begin{equation}\label{fdwZ}
	\begin{split}
		\mathbf{Z}= & -\Omega (u^0)^2  (c_s^2+1) \mathbf{T} \bu_{-},
	\end{split}
\end{equation}
and
\begin{equation}\label{fdwB}
	\begin{split}
		\mathbf{B}=& -((u^0)^{-1}u^i+\mathbf{g}^{0i})\partial_i \mathbf{Z}
		- \mathbf{T}  u^0 \{  \Omega(c_s^2+1) u^\gamma \partial_{\gamma}\bu_{-} \}
		\\
		& \ -  u^\beta \partial_{\beta}\{   \Omega(c_s^2+1) u^\gamma \} \partial_{\gamma}\bu_{-} + \Omega c^2_s u_{-}^\alpha.
	\end{split}
\end{equation}
\end{Lemma}
\begin{proof}
	For $\bu=\bu_{+} + \bu_{-}$, so we have
		\begin{equation}\label{B1}
		\square_{g} u^\alpha= \square_{g} u_{+}^\alpha + \square_{g} u_{-}^\alpha.
	\end{equation}
On the other hand, by \eqref{WTe} and \eqref{De0}, we know
	\begin{equation}\label{B2}
	\begin{split}
		& \square_g u^\alpha=-\Omega c_s^2\mathrm{e}^{-h}W^\alpha+Q^\alpha,
	\end{split}
\end{equation}
and
	\begin{equation}\label{B3}
	\begin{split}
		& ( \mathrm{I}+ \mathrm{P}^{\beta \gamma}\partial^2_{\beta \gamma} ) u_{-}^\alpha = -\mathrm{e}^{-h}W^\alpha.
	\end{split}
\end{equation}
Due to \eqref{B1} and \eqref{B2}, it yields
\begin{equation}\label{B4}
	\square_{g} u_{+}^\alpha + \square_{g} u_{-}^\alpha = \Omega c^2_s ( \mathrm{I}+ \mathrm{P}^{\beta \gamma}\partial^2_{\beta \gamma} ) u_{-}^\alpha + Q^\alpha.
\end{equation}
By using \eqref{De0} again, we have
\begin{equation}\label{B5}
	\mathrm{P}^{\beta \gamma}\partial_{\beta \gamma}^2u^\alpha_{-}= m^{ \beta \gamma }\partial^2_{ \beta \gamma} u^\alpha_{-}+2 u^\beta u^\gamma \partial^2_{\beta \gamma} u^\alpha_{-} .
\end{equation}
Inserting \eqref{B5} to \eqref{B4}, we infer
\begin{equation}\label{B7}
	\square_{g} u_{+}^\alpha + \square_{g} u_{-}^\alpha = \Omega c^2_s \left( m^{ \beta \gamma }\partial^2_{ \beta \gamma} u^\alpha_{-}+2 u^\beta u^\gamma \partial^2_{\beta \gamma} u^\alpha_{-} \right) + \Omega c^2_s u_{-}^\alpha  + Q^\alpha .
\end{equation}
Also, from the expression of $g$, we have
	\begin{equation}\label{B9}
		\square_{g} u_{-}^\alpha= \Omega c_s^2 m^{ \beta \gamma }\partial^2_{ \beta \gamma}  u_{-}^\alpha+\Omega (c_s^2-1)u^\beta u^\gamma  \partial^2_{ \beta \gamma} u_{-}^\alpha,
	\end{equation}
Substituting \eqref{B9} to \eqref{B7}, one can see
		\begin{equation}\label{UM}
		\square_g u^\alpha_{+}= \Omega (c_s^2+1) u^\beta u^\gamma \partial^2_{\beta \gamma}u^\alpha_{-}+ \Omega c^2_s u_{-}^\alpha  + {Q}^\alpha.
	\end{equation}
Referring Lemma \ref{LD}, \eqref{UM} is not so good for us to use this modified Duhamel's principle. 
So we need to rewrite \eqref{UM} in another way. Write \eqref{UM} as
	\begin{equation}\label{fda}
		\begin{split}
			\square_{\mathbf{g}} \bu_{+}=& \Omega(c_s^2+1) u^\beta u^\gamma \partial^2_{\beta \gamma}\bu_{-}+ \bQ + \Omega c^2_s u_{-}^\alpha
			\\
			=& \mathbf{g}^{00} (\mathbf{g}^{00})^{-1} \Omega(c_s^2+1)  u^\beta u^\gamma \partial^2_{\beta \gamma}\bu_{-}+ \bQ + \Omega c^2_s u_{-}^\alpha.
		\end{split}
	\end{equation}
A simple calculation tells us
	\begin{equation}\label{fdc}
		\begin{split}
			& (\mathbf{g}^{00})^{-1} \Omega(c_s^2+1)  u^\beta u^\gamma \partial^2_{\beta \gamma}\bu_{-}
			\\
			=& u^\beta \partial_{\beta}\{  (\mathbf{g}^{00})^{-1} \Omega(c_s^2+1) u^\gamma \partial_{\gamma}\bu_{-} \} - u^\beta \partial_{\beta}\{  (\mathbf{g}^{00})^{-1} \Omega(c_s^2+1) u^\gamma \} \partial_{\gamma}\bu_{-}
			\\
			=&  \mathbf{T}\{  u^0 (\mathbf{g}^{00})^{-1} \Omega(c_s^2+1) u^\gamma \partial_{\gamma}\bu_{-} \}
			-\mathbf{T}  u^0 \{ (\mathbf{g}^{00})^{-1} \Omega(c_s^2+1) u^\gamma \partial_{\gamma}\bu_{-} \}
			\\
			& - u^\beta \partial_{\beta}\{  (\mathbf{g}^{00})^{-1} \Omega(c_s^2+1) u^\gamma \} \partial_{\gamma}\bu_{-} .
		\end{split}
	\end{equation}
	Substituting \eqref{fdc} to \eqref{fda}, we have
	\begin{equation}\label{fdr1}
		\begin{split}
			\square_{\mathbf{g}} \bu_{+}=& \Omega(c_s^2+1) u^\beta u^\gamma \partial_{\beta \gamma}\bu_{-}+ \bQ + \Omega c^2_s u_{-}^\alpha
			\\
			=& \mathbf{g}^{00} \mathbf{T}\{  u^0 (\mathbf{g}^{00})^{-1} \Omega(c_s^2+1) u^\gamma \partial_{\gamma}\bu_{-} \}+ \bQ + \Omega c^2_s u_{-}^\alpha
			\\
			& - \mathbf{T}  u^0 \{  \Omega(c_s^2+1) u^\gamma \partial_{\gamma}\bu_{-} \}
		 - u^\beta \partial_{\beta}\{   \Omega(c_s^2+1) u^\gamma \} \partial_{\gamma}\bu_{-} .
		\end{split}
	\end{equation}
Consider $\mathbf{g}^{00}=-1$. For any function $F$, we can also derive that
	\begin{equation}\label{wform3}
		\mathbf{g}^{0\alpha} \partial_\alpha F= -\partial_t F+ \mathbf{g}^{0i} \partial_i F=\mathbf{g}^{00} \mathbf{T}F-\left( (u^0)^{-1}u^i+\mathbf{g}^{0i}\right) \partial_i F.
	\end{equation}
Seeing from \eqref{wform3} and \eqref{fdr1}, so we have
	\begin{equation*}
	\begin{split}
		\square_{\mathbf{g}} \bu_{+}
		=& \mathbf{g}^{0\alpha} \partial_\alpha \{  u^0 (\mathbf{g}^{00})^{-1} \Omega(c_s^2+1) u^\gamma \partial_{\gamma}\bu_{-} \}
		+ \bQ + \Omega c^2_s u_{-}^\alpha
		\\
		&-\left( (u^0)^{-1}u^i+\mathbf{g}^{0i}\right) \partial_i  \{  u^0 (\mathbf{g}^{00})^{-1} \Omega(c_s^2+1) u^\gamma \partial_{\gamma}\bu_{-} \}  
		\\
		& - \mathbf{T}  u^0 \{  \Omega(c_s^2+1) u^\gamma \partial_{\gamma}\bu_{-} \}
		-  u^\beta \partial_{\beta}\{  \Omega(c_s^2+1) u^\gamma \} \partial_{\gamma}\bu_{-}
		\\
		=& \mathbf{g}^{0\alpha} \partial_\alpha \{ - ( u^0 )^2 \Omega(c_s^2+1) \mathbf{T} \bu_{-} \}
		+ \bQ + \Omega c^2_s u_{-}^\alpha
		\\
		&+ \left( (u^0)^{-1}u^i+\mathbf{g}^{0i}\right) \partial_i  \{  u^0  \Omega(c_s^2+1) u^\gamma \partial_{\gamma}\bu_{-} \}  
		\\
		& - \mathbf{T}  u^0 \{  \Omega(c_s^2+1) u^\gamma \partial_{\gamma}\bu_{-} \}
		 -  u^\beta \partial_{\beta}\{  \Omega(c_s^2+1) u^\gamma \} \partial_{\gamma}\bu_{-}.
	\end{split}
\end{equation*}
Noting \eqref{fdwZ} and \eqref{fdwB}, so we have proved \eqref{fdr}.
\end{proof}
\begin{remark}
By \eqref{fdr}, thus $\bu_+$ satisfies a better wave equation than $\bu$. Moreover, \eqref{fdr} plays a key role in Strichartz estimates for $\bu_{+}$.
\end{remark}
\subsection{Elliptic estimates}
In this subsection, we mainly prove some useful elliptic estimates. We first derive introduce a space-time elliptic estimate.
\subsubsection{Space-time elliptic estimate}
\begin{Lemma}\label{tu}
Let $(h,\bu)$ be a solution of \eqref{WTe}. Let $\bu_{-}$, $\mathbf{P}$, and $\mathbf{T}$ be defined in \eqref{De0}, \eqref{uf}, \eqref{opt}. Then $\mathbf{T}\bu_{-}$ satisfies
	\begin{equation}\label{tue}
		\begin{split}
			\mathbf{P} ( \mathbf{T} u^\alpha_{-})=&(u^0)^{-1} \mathrm{e}^{-h}c^2_s W^\alpha \partial_\kappa u^\kappa+ (u^0)^{-1} \mathrm{e}^{-h}  \big\{  {W}^\kappa   \partial_\kappa u^\alpha
			-{2} {W}^\alpha \partial_\kappa u^\kappa+ u^\alpha {W}^\beta  u^\kappa \partial_\kappa u_\beta
			\\
			& -{2} \epsilon^{\alpha \beta \gamma \delta}u_\beta \partial_\delta u^\kappa \partial_\gamma w_\kappa
			- \epsilon^{\kappa \beta \gamma \delta}c^{-2}_s u_\beta w_\delta \partial_\gamma h \partial_\kappa u^\alpha -  \epsilon^{\alpha \beta \gamma \delta} c^{-2}_s u_\beta w_\delta   \partial_\gamma u^\kappa \partial_\kappa h
			\\
			& -{2}\mathrm{e}^{-h}  w^\alpha w^\kappa \partial_\kappa h
			+ \epsilon^{\alpha \beta \gamma \delta} c^{-2}_s u_\beta w_\delta \partial_\gamma h \partial_\kappa u^\kappa
			+( c^{-2}_s +2 ) \epsilon^{\alpha \beta \gamma \delta} u_\beta w^\kappa \partial_\delta u_\kappa \partial_\gamma h \big\}
			\\
			& -(u^0)^{-1}u^\kappa \partial_\kappa (m^{\beta \gamma}+2u^\beta u^\gamma)\partial^2_{\beta \gamma}u^\alpha_{-}
			+ (m^{\beta \gamma}+2u^\beta u^\gamma) \partial_\kappa u^\alpha_{-} \partial^2_{\beta \gamma} ( \frac{u^\kappa}{u^0} )
			\\
			& + (m^{\beta \gamma}+2u^\beta u^\gamma) \partial_\gamma u^\kappa \partial^2_{\beta \kappa} u^\alpha_{-}
			+ (m^{\beta \gamma}+2u^\beta u^\gamma) \partial_\beta u^\kappa \partial^2_{\kappa \gamma} u^\alpha_{-} .
		\end{split}
	\end{equation}
	Moreover, the following estimate holds:
	\begin{equation}\label{tue2}
		\| \mathbf{T} \bu_{-} \|_{L^2_t H_x^s} \lesssim  \| \bw \|_{L^2_t H^{\frac32+}_x } ( \|(u^0-1, \mathring{\bu}, h) \|_{L^\infty_t H_x^{s}}+\|(u^0-1, \mathring{\bu}, h) \|^2_{L^\infty_t H_x^{s}}).
	\end{equation}
\end{Lemma}
\begin{proof}
	Recall $\mathbf{T}=(u^0)^{-1} u^\kappa \partial_\kappa $ in \eqref{opt}. Operating $\mathbf{T}$ on \eqref{uf}, we get
	\begin{equation*}
		\begin{split}
			\mathbf{T} ( \mathbf{P} u^\alpha_{-} ) = & \mathbf{T} (\mathrm{e}^{-h} W^\alpha)= - \mathrm{e}^{-h}W^\alpha \mathbf{T}h+\mathrm{e}^{-h} \mathbf{T} W^\alpha.
		\end{split}
	\end{equation*}
	By using \eqref{REE} and \eqref{CEQ1}, we have
	\begin{equation}\label{c8}
		\begin{split}
			\mathbf{T} ( \mathbf{P} u^\alpha_{-} )
			=&(u^0)^{-1} \mathrm{e}^{-h}c^2_s W^\alpha \partial_\kappa u^\kappa+ (u^0)^{-1} \mathrm{e}^{-h}  \big\{  {W}^\kappa   \partial_\kappa u^\alpha
			-{2} {W}^\alpha \partial_\kappa u^\kappa+ u^\alpha {W}^\beta  u^\kappa \partial_\kappa u_\beta
			\\
			& -{2} \epsilon^{\alpha \beta \gamma \delta}u_\beta \partial_\delta u^\kappa \partial_\gamma w_\kappa
			- \epsilon^{\kappa \beta \gamma \delta}c^{-2}_s u_\beta w_\delta \partial_\gamma h \partial_\kappa u^\alpha -  \epsilon^{\alpha \beta \gamma \delta} c^{-2}_s u_\beta w_\delta   \partial_\gamma u^\kappa \partial_\kappa h
			\\
			& -{2}\mathrm{e}^{-h}  w^\alpha w^\kappa \partial_\kappa h
			+ \epsilon^{\alpha \beta \gamma \delta} c^{-2}_s u_\beta w_\delta \partial_\gamma h \partial_\kappa u^\kappa
			+( c^{-2}_s +2 ) \epsilon^{\alpha \beta \gamma \delta} u_\beta w^\kappa \partial_\delta u_\kappa \partial_\gamma h \big\}.
		\end{split}
	\end{equation}
	By the commutation rule, we get
	\begin{equation}\label{c9}
		\mathbf{T} ( \mathbf{P} u^\alpha_{-} ) =\mathbf{P}  (\mathbf{T} u^\alpha_{-} )+ [\mathbf{T}, \mathbf{P}]u^\alpha_{-}.
	\end{equation}
	By \eqref{De0}, we see $\mathbf{P}=\mathrm{I}-(m^{\beta \gamma}+2u^\beta u^\gamma)\partial_{\beta}\partial_{\gamma}$. So we can compute
	\begin{equation}\label{C10}
		\begin{split}
			[\mathbf{T}, \mathbf{P}]u^\alpha_{-}
			=& (u^0)^{-1}u^\kappa \partial_\kappa (m^{\beta \gamma}+2u^\beta u^\gamma)\partial^2_{\beta \gamma}u^\alpha_{-}
			-(m^{\beta \gamma}+2u^\beta u^\gamma) \partial_\kappa u^\alpha_{-} \partial^2_{\beta \gamma}\{ (u^0)^{-1} u^\kappa \}
			\\
			& -(m^{\beta \gamma}+2u^\beta u^\gamma) \partial_\gamma u^\kappa \partial^2_{\beta \kappa} u^\alpha_{-}
			-(m^{\beta \gamma}+2u^\beta u^\gamma) \partial_\beta u^\kappa \partial^2_{\kappa \gamma} u^\alpha_{-}.
		\end{split}
	\end{equation}
	Combining \eqref{c8}, \eqref{c9}, and \eqref{C10}, we can obtain \eqref{tue}. It remains for us to estimate \eqref{tue2}.

	By the space-time elliptic estimates \eqref{ellip}, and \eqref{tue}, \eqref{MFd}, and Lemma \ref{ps}, we have
	\begin{equation}\label{bsl}
		\begin{split}
			&\| \mathbf{T}  \bu_{-} \|_{L^2_t H_x^s}
			\\
			\lesssim & \| \bW \cdot d \bu \|_{L^2_t H_x^{s-2}}+ \| d\bw \cdot d\bu \|_{L^2_t H_x^{s-2}}+\| \bw\cdot d\bu \cdot dh \|_{L^2_t H_x^{s-2}}
			\\
			& + \|\nabla^2 \bu \cdot \nabla \bu_{-} \|_{L^2_t H_x^{s-2}}+ \|\nabla^2 \bu_{-} \cdot \nabla \bu \|_{L^2_t H_x^{s-2}}
			\\
			\lesssim & \| d\bw \cdot d\bu \|_{L^2_t H_x^{s-2}}+\| \bw\cdot d\bu \cdot dh \|_{L^2_t H_x^{s-2}}
			+ \|\nabla^2 \bu \cdot \nabla \bu_{-} \|_{L^2_t H_x^{s-2}}+ \|\nabla^2 \bu_{-} \cdot \nabla \bu \|_{L^2_t H_x^{s-2}}
			\\
			\lesssim & \| \bw \|_{L^2_t H_x^{\frac32+} } \|d\bu \|_{H_x^{s-1}}+\| \bw \|_{L^2_t H_x^{\frac32+}}\|d\bu \|_{H_x^{s-1}}\|dh \|_{H_x^{s-1}}
			+ \| \nabla \bu \|_{H_x^{s-1} } \|\nabla \bu_{-} \|_{L^2_t H_x^{s-1}}.
		\end{split}
	\end{equation}
	By using \eqref{De0}, Sobolev imbeddings, and elliptic estimates, we can derive that
	\begin{equation}\label{ETAs}
		\begin{split}
			\| \nabla \bu_{-} \|_{L^2_t H_x^{s-1}} \leq  \| \mathrm{e}^{-h} \bW \|_{L^2_t H_x^{s-2}} &\lesssim  (1+\|h\|_{H_x^{\frac{3}{2}+}} )\| \bw \|_{L^2_t H_x^{s-1}}.
		\end{split}
	\end{equation}
	Combining \eqref{bsl}, \eqref{ETAs}, we therefore conclude
	\begin{equation*}
		\begin{split}
			\| \mathbf{T}  \bu_{-} \|_{L^2_t H_x^s}
			\lesssim & \| \bw \|_{L^2_t H_x^{\frac32+} } ( \|(u^0-1, \mathring{\bu}, h) \|_{H_x^{s}}+\|(u^0-1, \mathring{\bu}, h) \|^2_{H_x^{s}}).
		\end{split}
	\end{equation*}
	At this stage, we have finished the proof of Lemma \ref{tu}.
\end{proof}

\subsubsection{Space elliptic estimates for vorticity}
Let us introduce a decomposition of $\Delta \bw$ and $\Delta \bW$. It plays an important role in the energy part of $\bw$. Using the following decomposition and \eqref{MFd}, the energy bound of $\bw$ can be transferred to estimate $\bW$ and $\bG$ through $\mathrm{vort}\bw$ and $\mathrm{vort}\bW$. 
\begin{Lemma}\label{OE1}
	Let the one-form $\bw$ be defined in \eqref{VVd}. Then we have
	\begin{equation}\label{OEe}
		\begin{split}
			\partial_j w_i - \partial_i w_j = & \epsilon_{ji\gamma\delta}u^\gamma \mathrm{vort}^\delta \bw- u_j \partial_i u^\kappa w_\kappa + u_i \partial_j u^\kappa w_\kappa  - u_i u_j w^\kappa \partial_\kappa h
			\\
			& + u_i w^\kappa \partial_\kappa u_j-u_i w_j \partial_\kappa u^\kappa
			+ u_j u_i w^\kappa \partial_\kappa h - u_j w^\kappa \partial_\kappa u_i+ u_j w_i \partial_\kappa u^\kappa.
		\end{split}
	\end{equation}
\end{Lemma}
\begin{proof}
	By Lemma \ref{OE}, we have
	\begin{equation}\label{OE0}
		\partial_j w_i - \partial_i w_j = \epsilon_{ji\gamma\delta}u^\gamma \mathrm{vort}^\delta \bw+ u_j u^\kappa \partial_i w_\kappa - u_i u^\kappa \partial_j w_\kappa +u_i u^\kappa \partial_\kappa w_j -u_j u^\kappa \partial_\kappa w_i .
	\end{equation}
	For $u^\kappa w_\kappa=0$, we can obtain
	\begin{equation}\label{OE2}
		u_j u^\kappa \partial_i w_\kappa - u_i u^\kappa \partial_j w_\kappa= -u_j \partial_i u^\kappa  w_\kappa + u_i \partial_j u^\kappa  w_\kappa.
	\end{equation}
	Substituting \eqref{OE2} and \eqref{CEQ} to \eqref{OE0}, it follows
	\begin{equation*}
		\begin{split}
			\partial_j w_i - \partial_i w_j = & \epsilon_{ji\gamma\delta}u^\gamma \mathrm{vort}^\delta\bw- u_j \partial_i u^\kappa w_\kappa + u_i \partial_j u^\kappa w_\kappa  - u_i u_j w^\kappa \partial_\kappa h
			\\
			&  + u_i w^\kappa \partial_\kappa u_j-u_i w_j \partial_\kappa u^\kappa
			+ u_j u_i w^\kappa \partial_\kappa h - u_j w^\kappa \partial_\kappa u_i+ u_j w_i \partial_\kappa u^\kappa.
		\end{split}
	\end{equation*}
	So we have proved \eqref{OEe}. 
\end{proof}

\begin{Lemma}\label{C}
	Let $W$ be defined in \eqref{MFd}. Then we have
	\begin{equation}\label{c2}
		\begin{split}
			\partial_i W_j - \partial_j W_i=&\epsilon_{ij\gamma\delta}u^\gamma \mathrm{vort}^\delta \bW- u_i W_\kappa \partial_ju^\kappa +u_j W_\kappa \partial_i u^\kappa
			\\
			&+ u_j \big\{ {W}^\kappa   \partial_\kappa u_i   -{2} {W}_i \partial_\kappa u^\kappa+ u_i {W}^\beta  u^\kappa \partial_\kappa u_\beta
			-{2} \epsilon_{i}^{ \ \beta \gamma \delta}u_\beta \partial_\delta u^\kappa \partial_\gamma w_\kappa
			\\
			&  -{2}\mathrm{e}^{-h}  w_i w^\kappa \partial_\kappa h -  \epsilon_i^{ \ \beta \gamma \delta} c^{-2}_s u_\beta w_\delta   \partial_\gamma u^\kappa \partial_\kappa h
			- \epsilon^{\kappa \beta \gamma \delta}c^{-2}_s u_\beta w_\delta \partial_\gamma h \partial_\kappa u_i
			\\
			& + \epsilon_i^{ \ \beta \gamma \delta} c^{-2}_s u_\beta w_\delta \partial_\gamma h \partial_\kappa u^\kappa
			+( c^{-2}_s +2 ) \epsilon_i^{\ \beta \gamma \delta} u_\beta w^\kappa \partial_\delta u_\kappa \partial_\gamma h \big\}
			\\
			&-u_i \big\{ {W}^\kappa   \partial_\kappa u_j   -{2} {W}_j \partial_\kappa u^\kappa+ u_j {W}^\beta  u^\kappa \partial_\kappa u_\beta
			-{2} \epsilon_{j}^{ \ \beta \gamma \delta}u_\beta \partial_\delta u^\kappa \partial_\gamma w_\kappa
			\\
			&  -{2}\mathrm{e}^{-h}  w_j w^\kappa \partial_\kappa h -  \epsilon_j^{ \ \beta \gamma \delta} c^{-2}_s u_\beta w_\delta   \partial_\gamma u^\kappa \partial_\kappa h
			- \epsilon^{\kappa \beta \gamma \delta}c^{-2}_s u_\beta w_\delta \partial_\gamma h \partial_\kappa u_j
			\\
			& + \epsilon_j^{ \ \beta \gamma \delta} c^{-2}_s u_\beta w_\delta \partial_\gamma h \partial_\kappa u^\kappa
			+( c^{-2}_s +2 ) \epsilon_j^{\ \beta \gamma \delta} u_\beta w^\kappa \partial_\delta u_\kappa \partial_\gamma h \big\} .
		\end{split}
	\end{equation}
\end{Lemma}
\begin{proof}
	Firstly, using Lemma \ref{OE}, we can derive
	\begin{align}\label{c3}
		\partial_i W_j - \partial_j W_i=\epsilon_{ij\gamma\delta}u^\gamma \mathrm{vort}^\delta \bW+ u_i u^\kappa \partial_jW_\kappa-u_j u^\kappa \partial_i W_\kappa
		+ u_j u^\kappa \partial_\kappa W_i-u_i u^\kappa \partial_\kappa W_j.
	\end{align}
	Secondly, from the definition of $\bW$, it yields
	\begin{equation}\label{uc}
		u_\kappa W^\kappa =\epsilon^{\kappa\alpha\beta\gamma}u_\alpha u_\kappa( - \partial_\beta w_\gamma+c^{-2}w_{\gamma}\partial_{\beta}h )= 0.
	\end{equation}
	By \eqref{uc}, we infer
	\begin{align}\label{c5}
		u_i u^\kappa \partial_jW_\kappa-u_j u^\kappa \partial_i W_\kappa=& -u_i \partial_j u^\kappa W_\kappa+ u_j \partial_i  u^\kappa W_\kappa.
	\end{align}
	Substituting \eqref{c5} and \eqref{CEQ1} to \eqref{c3}, we can obtain \eqref{c2}.
\end{proof}
We are ready to introduce a decomposition of $\Delta \bw$.
\begin{Lemma}\label{vor1}
	Let $(h,\bu)$ be a solution of \eqref{REE}. Let $\bw$ be defined in \eqref{VVd}. Then we have
	\begin{equation}\label{d5}
		\begin{split}
			\left(  \delta^{mj}+ (u^0)^{-2}u^m  u^j  \right)  \partial_j \partial_m w_i=&  (u^0)^{-2}u^m  u^j \partial_j ( \epsilon_{mi\gamma\delta}u^\gamma \mathrm{vort}^\delta\bw ) + \partial^j( \epsilon_{ji\gamma\delta}u^\gamma \mathrm{vort}^\delta \bw)
			\\
			& - \partial^j( u_j \partial_i u^\kappa w_\kappa)
			-\partial_i ( w^\kappa \partial_\kappa h )
			 + \partial^j( u_i \partial_j u^\kappa w_\kappa )
			\\
			& - \partial^j( u_i u_j w^\kappa \partial_\kappa h )
		 + \partial^j( u_i w^\kappa \partial_\kappa u_j )
			 - \partial^j( u_i w_j \partial_\kappa u^\kappa )
			 \\
			 & + \partial^j( u_j u_i w^\kappa \partial_\kappa h )
		 - \partial^j( u_j w^\kappa \partial_\kappa u_i )
		 + \partial_j ( u_j w_i \partial_\kappa u^\kappa )
		   \\
		   &   + \partial_i( w_m  \partial_t \left\{ (u^0)^{-1}u^m  \right\} )
			 -\partial_i( (u^0)^{-2}u^m  u^j) \partial_j w_m
			\\
			& - \partial_i (  (u^0)^{-2} u^m u_m w^\kappa \partial_\kappa h )
		 - \partial_i ( (u^0)^{-2}u^m w_m \partial_\kappa u^\alpha )
			\\
			& - \partial_i(  (u^0)^{-2}u^m w_m \partial_\kappa u^\kappa )
			  + (u^0)^{-2}u^m  u^j \partial_j \big\{  - u_m \partial_i u^\kappa w_\kappa
			\\
			& + u_i \partial_m u^\kappa w_\kappa
			 - u_i u_m w^\kappa \partial_\kappa h
			+ u_i w^\kappa \partial_\kappa u_m
			 -u_i w_m \partial_\kappa u^\kappa
			\\
			& + u_m u_i w^\kappa \partial_\kappa h - u_m w^\kappa \partial_\kappa u_i+ u_m w_i \partial_\kappa u^\kappa \big\}.
		\end{split}	
	\end{equation}
\end{Lemma}
	\begin{proof}
		To prove \eqref{d5}, we derive the formulation of $\Delta w_i$ by
		\begin{align}\label{d7}
			\Delta w_i=& \partial^j(\partial_j w_i- \partial_i w_j)+ \partial_i \partial^j w_j \nonumber
			\\
			=& \partial^j(\partial_j w_i- \partial_i w_j)+ \partial_i ( \partial^\kappa w_\kappa- \partial_t w^0 ) \nonumber
			\\
			=& \partial^j(\partial_j w_i- \partial_i w_j)- \partial_i ( w^\kappa \partial_\kappa h+ \partial_t w^0 ).
		\end{align}
		For $u_\kappa w^\kappa=0$, so we get
		\begin{equation}\label{w0d}
			w^0=-(u^0)^{-1}u^i w_i.
		\end{equation}
By \eqref{w0d}, we can calculate
		\begin{align}\label{d8}
			\partial_t w^0= \partial_t \{ -(u^0)^{-1}u^i w_i) \}
			= - w_i  \partial_t \left\{ (u^0)^{-1}u^i  \right\} - (u^0)^{-1}u^i \partial_t w_i.
		\end{align}
		By using \eqref{CEQ}, we infer
		\begin{equation}\label{d9}
			\partial_t w_i=-(u^0)^{-1}u^j \partial_j w_i	-(u^0)^{-1} u_i w^\kappa \partial_\kappa h+ (u^0)^{-1} w_i \partial_\kappa u^\alpha- (u^0)^{-1} w_i \partial_\kappa u^\kappa.
		\end{equation}
		Substituting \eqref{d9} to \eqref{d8}, it follows
		\begin{equation}\label{d10}
			\begin{split}
				\partial_t w^0=& - w_i  \partial_t \left\{ (u^0)^{-1}u^i  \right\}
				+ (u^0)^{-2}u^m  u^j \partial_j w_m	+ (u^0)^{-2} u^m u_m w^\kappa \partial_\kappa h
				\\
				& + (u^0)^{-2}u^m w_m \partial_\kappa u^\alpha + (u^0)^{-2}u^m w_m \partial_\kappa u^\kappa.
			\end{split}
		\end{equation}
		Inserting \eqref{OEe} and \eqref{d10} to \eqref{d7}, we obtain
		\begin{equation}\label{d11}
			\begin{split}
				\Delta w_i=&  \partial^j( \epsilon_{ji\gamma\delta}u^\gamma \mathrm{vort}^\delta \bw) - \partial^j( u_j \partial_i u^\kappa w_\kappa) + \partial^j( u_i \partial_j u^\kappa w_\kappa )
				\\
				&  - \partial^j( u_i u_j w^\kappa \partial_\kappa h )
				+ \partial^j( u_i w^\kappa \partial_\kappa u_j ) - \partial^j( u_i w_j \partial_\kappa u^\kappa )
				\\
				& + \partial^j( u_j u_i w^\kappa \partial_\kappa h ) - \partial^j( u_j w^\kappa \partial_\kappa u_i ) + \partial^j ( u_j w_i \partial_\kappa u^\kappa )
				\\
				& -\partial_i ( w^\kappa \partial_\kappa h ) + \partial_i( w_m  \partial_t \left\{ (u^0)^{-1}u^m  \right\} )
				-\partial_i( (u^0)^{-2}u^m  u^j \partial_j w_m )
				\\	
				&- \partial_i (  (u^0)^{-2} u^m u_m w^\kappa \partial_\kappa h )
				- \partial_i ( (u^0)^{-2}u^m w_m \partial_\kappa u^\alpha ) - \partial_i(  (u^0)^{-2}u^m w_m \partial_\kappa u^\kappa ).
			\end{split}	
		\end{equation}
		Calculate $-\partial_i( (u^0)^{-2}u^m  u^j \partial_j w_m )$ by
		\begin{align}\label{d12}
			&-\partial_i( (u^0)^{-2}u^m  u^j \partial_j w_m ) \nonumber
			\\
			= &- (u^0)^{-2}u^m  u^j \partial_i\partial_j w_m
			-\partial_i( (u^0)^{-2}u^m  u^j) \partial_j w_m, \nonumber
			\\
			= &- (u^0)^{-2}u^m  u^j \partial_j \partial_m w_i - (u^0)^{-2}u^m  u^j \partial_j ( \partial_i w_m - \partial_m w_i)
			  -\partial_i( (u^0)^{-2}u^m  u^j) \partial_j w_m \nonumber
			\\
			 = &- (u^0)^{-2}u^m  u^j \partial_j \partial_m w_i -\partial_i( (u^0)^{-2}u^m  u^j) \partial_j w_m
			 + (u^0)^{-2}u^m  u^j \partial_j ( \epsilon_{mi\gamma\delta}u^\gamma \mathrm{vort}^\delta\bw ) \nonumber
			 \\
			 &  + (u^0)^{-2}u^m  u^j \partial_j \{  - u_m \partial_i u^\kappa w_\kappa + u_i \partial_m u^\kappa w_\kappa  - u_i u_m w^\kappa \partial_\kappa h
			   + u_i w^\kappa \partial_\kappa u_m
			   \\  \nonumber
			   &   -u_i w_m \partial_\kappa u^\kappa
			 + u_m u_i w^\kappa \partial_\kappa h - u_m w^\kappa \partial_\kappa u_i+ u_m w_i \partial_\kappa u^\kappa \}.
		\end{align}
Inserting \eqref{d12} to \eqref{d11}, then we can obtain \eqref{d5}. Therefore, we have proved Lemma \ref{vor1}.
\end{proof}

\begin{Lemma}\label{vor2}
Let $(h,\bu)$ be a solution of \eqref{REE}. Let $\bw$, $\bW$, and $\mathring{\bW}$ be defined in \eqref{VVd}, \eqref{MFd}, \eqref{MFda}. Then we have the following estimate
	\begin{equation}\label{d17}
		\begin{split}
\| \nabla \mathring{ \bW } \|_{\dot{H}_x^{s_0-2}} \lesssim &	   \|\mathrm{vort} \bW \|_{\dot{H}_x^{s_0-2}}+  \| \bW \cdot (d\bu,dh) \|_{\dot{H}_x^{s_0-2}}
+   \| d\bw \cdot (d\bu,dh) \|_{\dot{H}_x^{s_0-2}}
\\
&	+  \| \bw \cdot d\bu \cdot dh \|_{\dot{H}_x^{s_0-2}}+  \| \bw \cdot \bw \cdot dh \|_{\dot{H}_x^{s_0-2}}.
		\end{split}
	\end{equation}
\end{Lemma}
\begin{proof}
	To derive the formulation of $\Delta W_i$, we calculate
	\begin{align}\label{d19}
		\Delta W_i=& \partial^j(\partial_j W_i- \partial_i W_j)+ \partial_i \partial^j W_j \nonumber
		\\\nonumber
		=& \partial^j(\partial_j W_i- \partial_i W_j)+ \partial_i ( \partial^\alpha W_\alpha) - \partial_i ( \partial^0 W_0 )
			\\
		=& \partial^j(\partial_j W_i- \partial_i W_j)+ \partial_i ( \partial^\alpha W_\alpha) - \partial_i ( \partial_t W^0 )  .
	\end{align}
By \eqref{d19}, we infer
	\begin{align}\label{d110}
	\| \nabla \mathring{\bW }\|_{\dot{H}_x^{s_0-2}}\leq & \sum_{i,j=1,2,3} \| \partial_j W_i- \partial_i W_j \|_{\dot{H}_x^{s_0-2}}+ \| \partial^\alpha W_\alpha \|_{\dot{H}_x^{s_0-2}} + \| \partial_t W^0 \|_{\dot{H}_x^{s_0-2}}  .
	\end{align}
	By \eqref{c2}, it yields
		\begin{equation}\label{d111}
			\begin{split}
					& \textstyle{\sum_{i,j=1,2,3}} \| \partial_j W_i- \partial_i W_j \|_{\dot{H}_x^{s_0-2}}
					\\
					\leq &  C \|\mathrm{vort} \bW \|_{\dot{H}_x^{s_0-2}} +   C\| \bW \cdot d\bu \|_{\dot{H}_x^{s_0-2}}+   C\| d\bw \cdot (d\bu,dh) \|_{\dot{H}_x^{s_0-2}}
				\\
				& + C \| \bw \cdot d\bu \cdot dh \|_{\dot{H}_x^{s_0-2}}+  C \| \bw \cdot \bw \cdot dh \|_{\dot{H}_x^{s_0-2}}.
			\end{split}
	\end{equation}
	A direct computation tells us 	
	\begin{equation*}
		u_\alpha W^\alpha = -\epsilon^{\alpha \beta \gamma \delta}u_\alpha u_{\beta}\partial_{\gamma}w_\delta+c_s^{-2}\epsilon^{\alpha\beta\gamma\delta}u_\alpha u_{\beta}w_{\delta}\partial_{\gamma}h=0,
	\end{equation*}
	so we get
	\begin{equation}\label{d20}
		W^0=-(u_0)^{-1}u^i W_i.
	\end{equation}
Due to \eqref{d20}, so we have
	\begin{align}\label{d21}
		\partial_t W^0= & \partial_t \{ -(u_0)^{-1}u^i W_i) \} \nonumber
		\\
		= & - W_i  \partial_t \left\{ (u_0)^{-1}u^i  \right\} - (u_0)^{-1}u^i \partial_t W_i.
	\end{align}
	By using \eqref{d21}, we can obtain
	\begin{equation}\label{d125}
		\begin{split}
			\|	\partial_t W^0 \|_{\dot{H}_x^{s_0-2}} \leq \| \bW \cdot d\bu \|_{\dot{H}_x^{s_0-2}}+ \frac{| \mathring{\bu}|}{|u_0|} \|\partial_t \mathring{\bW} \|_{\dot{H}_x^{s_0-2}}.
		\end{split}
	\end{equation}
By using \eqref{CEQ1} and \eqref{REE}, we get
\begin{equation}\label{d22}
	\begin{split}
	 \partial_t {W}_i
	=& -(u_0)^{-1}u^j \partial_j {W}_i+ {W}^\kappa   \partial_\kappa u_i   -{2}(u_0)^{-1} {W}_i \partial_\kappa u^\kappa
	 -  \epsilon_{i}^{\ \beta \gamma \delta} c^{-2}_s (u_0)^{-1} u_\beta w_\delta   \partial_\gamma u^\kappa \partial_\kappa h
	\\
	& + (u_0)^{-1} u_i {W}^\beta  u^\kappa \partial_\kappa u_\beta  -{2}  \epsilon_{i}^{\ \beta \gamma \delta}(u_0)^{-1} u_\beta \partial_\delta u^\kappa \partial_\gamma w_\kappa
	   -{2}\mathrm{e}^{-h} (u_0)^{-1} w_i w^\kappa \partial_\kappa h
	\\
	& - \epsilon^{\kappa \beta \gamma \delta}c^{-2}_s (u_0)^{-1} u_\beta w_\delta \partial_\gamma h \partial_\kappa u_i
 + \epsilon_i^{\ \beta \gamma \delta} c^{-2}_s (u_0)^{-1} u_\beta w_\delta \partial_\gamma h \partial_\kappa u^\kappa
	\\
	& +( c^{-2}_s +2 ) \epsilon_i^{\ \beta \gamma \delta} (u_0)^{-1} u_\beta w^\kappa \partial_\delta u_\kappa \partial_\gamma h
	\\
	=& -(u_0)^{-1}u^j \partial_j {W}_i+ {W}^\kappa   \partial_\kappa u_i   -{2}(u_0)^{-1} {W}_i \partial_\kappa u^\kappa
	-  \epsilon_{i}^{\ \beta \gamma \delta} c^{-2}_s (u_0)^{-1} u_\beta w_\delta   \partial_\gamma u^\kappa \partial_\kappa h
	\\
	& - (u_0)^{-1} u_i {W}^\beta  \partial_\beta h  -{2}  \epsilon_{i}^{\ \beta \gamma \delta}(u_0)^{-1} u_\beta \partial_\delta u^\kappa \partial_\gamma w_\kappa
	-{2}\mathrm{e}^{-h} (u_0)^{-1} w_i w^\kappa \partial_\kappa h
	\\
	& - \epsilon^{\kappa \beta \gamma \delta}c^{-2}_s (u_0)^{-1} u_\beta w_\delta \partial_\gamma h \partial_\kappa u_i
	+ \epsilon_i^{\ \beta \gamma \delta} c^{-2}_s (u_0)^{-1} u_\beta w_\delta \partial_\gamma h \partial_\kappa u^\kappa
	\\
	& +( c^{-2}_s +2 ) \epsilon_i^{\ \beta \gamma \delta} (u_0)^{-1} u_\beta w^\kappa \partial_\delta u_\kappa \partial_\gamma h.
	\end{split}	
\end{equation}
Due to \eqref{d22}, we have
\begin{equation}\label{d126}
	\begin{split}
 \|\partial_t \mathring{\bW} \|_{\dot{H}_x^{s_0-2}} \leq & \frac{| \mathring{\bu}|}{|u_0|}  \| \nabla \mathring{ \bW } \|_{\dot{H}_x^{s_0-2}} +   C\| \bW \cdot (d\bu,dh) \|_{\dot{H}_x^{s_0-2}}+   C\| d\bw \cdot (d\bu,dh) \|_{\dot{H}_x^{s_0-2}}
 \\
 & + C \| \bw \cdot d\bu \cdot dh \|_{\dot{H}_x^{s_0-2}}+  C \| \bw \cdot \bw \cdot dh \|_{\dot{H}_x^{s_0-2}}.
	\end{split}
\end{equation}
Combining \eqref{d125} and \eqref{d126}, it follows
\begin{equation}\label{d127}
	\begin{split}
		\|	\partial_t W^0 \|_{\dot{H}_x^{s_0-2}} \leq & \frac{| \mathring{\bu}|^2}{|u_0|^2}  \| \nabla \mathring{ \bW } \|_{\dot{H}_x^{s_0-2}} +   C\| \bW \cdot (d\bu,dh) \|_{\dot{H}_x^{s_0-2}}+   C\| d\bw \cdot (d\bu,dh) \|_{\dot{H}_x^{s_0-2}}
		\\
		& + C \| \bw \cdot d\bu \cdot dh \|_{\dot{H}_x^{s_0-2}}+  C \| \bw \cdot \bw \cdot dh \|_{\dot{H}_x^{s_0-2}}.
	\end{split}
\end{equation}
On the other hand, we can also derive
\begin{align}\label{d24}
	\partial_\alpha W^\alpha =
	& \partial_\alpha \left(
	-\epsilon^{\alpha \beta \gamma \delta} u_{\beta}\partial_{\gamma}w_\delta+c_s^{-2}\epsilon^{\alpha\beta\gamma\delta} u_{\beta}w_{\delta}\partial_{\gamma}h \right) \nonumber
	\\
	=& -\epsilon^{\alpha \beta \gamma \delta} \partial_\alpha   u_{\beta}\partial_{\gamma}w_\delta
	+c_s^{-2}\epsilon^{\alpha\beta\gamma\delta} w_{\delta} \partial_\alpha u_{\beta} \partial_{\gamma}h
	+c_s^{-2}\epsilon^{\alpha\beta\gamma\delta}  u_{\beta} \partial_\alpha w_{\delta} \partial_{\gamma}h.
\end{align}
Due to \eqref{d24}, we get
\begin{equation}\label{d124}
	\begin{split}
		\|	\partial_\alpha W^\alpha \|_{\dot{H}_x^{s_0-2}} \leq C \| d\bw \cdot (d\bu,dh) \|_{\dot{H}_x^{s_0-2}}
		 + C \| \bw \cdot d\bu \cdot dh \|_{\dot{H}_x^{s_0-2}}.
	\end{split}
\end{equation}
Inserting \eqref{d111}, \eqref{d127}, and \eqref{d124} to \eqref{d110}, so we get
\begin{equation}\label{d25}
	\begin{split}
\| \nabla \mathring{ \bW } \|_{\dot{H}_x^{s_0-2}} \leq &	\frac{| \mathring{\bu}|^2}{|u_0|^2}  \| \nabla \mathring{ \bW } \|_{\dot{H}_x^{s_0-2}} +  C\| \mathrm{vort}\bW \|_{\dot{H}_x^{s_0-2}}+  C\| \bW \cdot (d\bu,dh) \|_{\dot{H}_x^{s_0-2}}
\\
& +   C\| d\bw \cdot (d\bu,dh) \|_{\dot{H}_x^{s_0-2}}
	 + C \| \bw \cdot d\bu \cdot dh \|_{\dot{H}_x^{s_0-2}}+  C \| \bw \cdot \bw \cdot dh \|_{\dot{H}_x^{s_0-2}}.
	\end{split}
\end{equation}
By \eqref{muu}, we can see
\begin{equation}\label{d26}
	\begin{split}
| \mathring{\bu}|^2=|u_0|^2-1.
	\end{split}
\end{equation}
Using \eqref{d25} and \eqref{d26}, we have
\begin{equation*}
	\begin{split}
\frac{1}{|u_0|^2}	\| \nabla \mathring{ \bW } \|_{\dot{H}_x^{s_0-2}} \leq &	   C\| \mathrm{vort}\bW \|_{\dot{H}_x^{s_0-2}}+  C\| \bW \cdot (d\bu,dh) \|_{\dot{H}_x^{s_0-2}}
	 +   C\| d\bw \cdot (d\bu,dh) \|_{\dot{H}_x^{s_0-2}}
	\\
	&	+ C \| \bw \cdot d\bu \cdot dh \|_{\dot{H}_x^{s_0-2}}+  C \| \bw \cdot \bw \cdot dh \|_{\dot{H}_x^{s_0-2}}.
	\end{split}
\end{equation*}
For $ 1\leq |u_0| \leq 2+C_0$, we can obtain \eqref{d17}.
\end{proof}

\section{Energy estimates}\label{sec:energyest}
In this part, our goal is to prove the energy estimates for Theorem \ref{dingli} and \ref{dingli2}. Firstly, we introduce a classical energy estimate for density and velocity.
\begin{theorem}[Classical energy estimate,\cite{Bru}]\label{VHE}
	Let $(h,\bu)$ be a solution of \eqref{REE}. Then for any $a\geq 0$, we have
\begin{equation}\label{VHe}
 \| h\|^2_{H_x^a}+ \|\mathring{\bu}\|^2_{H_x^a} \lesssim  \left( \|h_0\|^2_{H^a}+ \|\mathring{\bu}_0\|^2_{H^a} \right) \exp( {\int^t_0} \|d\mathring{\bu}, dh\|_{L^\infty_x}d\tau), \quad t \in [0,T].
\end{equation}
\end{theorem}
\begin{remark}
(1)	Due to \eqref{muu}, we can get
	\begin{equation}\label{u00}
		u^0=\sqrt{(u^1)^2+(u^2)^2+(u^3)^2+1}, \quad u_0=-\sqrt{(u^1)^2+(u^2)^2+(u^3)^2+1}.
	\end{equation}
	By using classical energy estimates for the symmetric hyperbolic system \eqref{QHl}, and combining with \eqref{und}, Lemma \ref{jh0}, we can obtain \eqref{VHe}.

(2) For $a\geq 1$, using \eqref{QHl} and \eqref{VHe} again, we therefore have
\begin{equation}\label{VHea}
	\|(\partial_t  h, \partial_t \mathring{\bu})\|_{H_x^{a-1}} \lesssim \|(\nabla  h,\nabla \mathring{\bu})\|_{H_x^{a-1}} .
\end{equation}
By \eqref{u00} and \eqref{VHea}, it follows
\begin{equation}\label{VHeb}
	\|\partial_t  u^0\|_{H_x^{a-1}} \lesssim \|\partial_t \mathring{\bu}\|_{H_x^{a-1}} \lesssim \|(\nabla  h,\nabla \mathring{\bu})\|_{H_x^{a-1}} .
\end{equation}
\end{remark}
\subsection{Energy estimate for Theorem \ref{dingli}}\label{sec3.1}
\begin{theorem}[Energy estimate: type 1]\label{VE}
Let $(h,\bu)$ be a solution of \eqref{REE}. Let $\bw$ be defined in \eqref{VVd} and $\bw$ satisfy \eqref{CEQ}-\eqref{CEQ0}. For $2<s_0<s\leq \frac52$, the following energy estimate
\begin{equation}\label{WHa}
	 E_s(t) \leq E_0  \exp \left(    5 M(t) \cdot \mathrm{e}^{5 M(t)} \right),
\end{equation}
holds, where
\begin{equation}\label{WL}
\begin{split}
	& M(t)=   {\int^t_0} (\|d\bu, dh\|_{L^\infty_x}+\|d\bu, dh\|_{\dot{B}^{s_0-2}_{\infty,2}})  d\tau,
	\\
  & E_s(t) =  \|h\|^2_{H_x^s} + \|(\mathring{\bu}, u^0-1)\|^2_{H_x^s} +\|\bw\|^2_{H_x^{s_0}},
  \\
  & E_0= C \left( \|(h_0,\mathring{\bu}_0)\|^2_{H^s} +\|\bw_0\|^2_{H^{s_0}}+\|(h_0,\mathring{\bu}_0)\|^{10}_{H^s}+\|\bw_0\|^{10}_{H^{s_0}} \right).
  \end{split}
\end{equation}
Above, $C$ is also a universal constant.
\end{theorem}

\begin{proof}
To prove \eqref{WHa}, we still need to bound $\| u^0-1 \|_{{H}_x^{s}}$ and $\| \bw \|_{{H}_x^{s_0}}$. By using \eqref{u00} and \eqref{VHe}, we obtain
\begin{equation}\label{Wc0}
	\| u^0-1 \|^2_{{H}_x^{s}}(t) \lesssim  \|\mathring{\bu}\|^2_{H_x^{s}} \lesssim  \left( \|h_0\|^2_{H^{s}}+ \|\mathring{\bu}_0\|^2_{H^{s}} \right) \exp( {\int^t_0} \|d\mathring{\bu}, dh\|_{L^\infty_x}d\tau).
\end{equation}
For $\| \bw \|_{{H}_x^{s_0}}$, we consider
\begin{equation}\label{We0}
	\| \bw \|_{{H}_x^{s_0}}=\| \bw \|_{L_x^{2}}+ \| \bw \|_{\dot{H}_x^{s_0}}.
\end{equation}
So it's natural for us to utilize the transport structure of $\bw$ and the higher derivatives $\bW$ and $\bG$. Multiplying $(u^0)^{-1} w^\alpha$ on \eqref{CEQ} and integrating it on $[0,t] \times \mathbb{R}^3$, we have
\begin{equation}\label{We1}
\begin{split}
  \frac{d}{dt} \|\bw\|^2_{L_x^2} =& \int_{\mathbb{R}^3}  \partial_i ((u^0)^{-1} u^i) |\bw|^2 dx- \int_{\mathbb{R}^3} (u^0)^{-1} u^\alpha w^\alpha w^\kappa \partial_\kappa h  dx
  \\
  & + \int_{\mathbb{R}^3} (u^0)^{-1} w^\kappa \partial_\kappa u^\alpha w^\alpha dx- \int_{\mathbb{R}^3} (u^0)^{-1}  \partial_\kappa u^\kappa |\bw|^2 dx .
\end{split}
\end{equation}
By H\"older's inequality, we have
\begin{equation}\label{We2}
\begin{split}
  \frac{d}{dt} \|\bw\|^2_{L_x^2} \lesssim & ( \|dh \|_{L^\infty_x} + \| d\bu\|_{L^\infty_x})\| \bw\|^2_{L_x^2} .
\end{split}
\end{equation}
Integrating \eqref{We2} on the time interval $[0,t]$($t>0$), so we have
\begin{equation}\label{We3}
  \begin{split}
  \|\bw(t,\cdot)\|^2_{L_x^2} \lesssim & \|\bw_0\|^2_{L_x^2} + \int^t_0 \|dh , d\bu \|_{L^\infty_x} \|\bw\|^2_{L_x^2} d\tau.
\end{split}
\end{equation}
For $\| \bw \|_{\dot{H}^{s_0}}$, we notice
\begin{equation}\label{We4}
	\| \bw \|_{\dot{H}_x^{s_0}} = \| \Delta \bw \|_{\dot{H}_x^{s_0-2}}.
\end{equation}
Since $\bw=(w^0,\mathring{\bw})$, let us first consider $\|  \mathring{\bw} \|_{\dot{H}_x^{s_0}}$. Set the operator
\begin{equation}\label{dh}
	\Delta_{H} =(\delta^{mj}+(u^0)^{-2} u^mu^j )\partial_m \partial_j.
\end{equation}
By \eqref{muu}, $\Delta_{H}$ is also an elliptic operator on $\mathbb{R}^3$. Therefore, we deduce that
\begin{equation}\label{We5}
	\| \Delta_H \mathring{\bw} \|_{\dot{H}_x^{s_0-2}}  \lesssim	\| \Delta \mathring{\bw} \|_{\dot{H}_x^{s_0-2}}  \lesssim 	\| \Delta_H \mathring{\bw} \|_{\dot{H}_x^{s_0-2}} .
\end{equation}
Due t0 \eqref{d5}, we have
\begin{equation}\label{We6}
	\| \Delta_H \mathring{\bw} \|_{\dot{H}_x^{s_0-2}}
	\lesssim	
	\| \bu \cdot \mathrm{vort} \bw \|_{\dot{H}_x^{s_0-1}}
	+ \| \bw \cdot (d\bu,dh)\|_{\dot{H}_x^{s_0-1}}+  \| d\bu \cdot \nabla \mathring{\bw} \|_{\dot{H}_x^{s_0-2}}.
\end{equation}
Since \eqref{MFd}, \eqref{w0d}, and Lemma \ref{ps}, it follows
\begin{equation}\label{We7}
	\begin{split}
		\| \Delta_H \mathring{\bw} \|_{\dot{H}_x^{s_0-2}}  \lesssim	 & \|\bu \cdot \bW\|_{\dot{H}_x^{s_0-1}} + \| \bw \cdot (d\bu,dh)\|_{\dot{H}_x^{s_0-1}}+  \| d\bu \cdot \nabla \mathring{\bw} \|_{\dot{H}_x^{s_0-2}}
		\\
	\lesssim	& \|\bW\|_{\dot{H}_x^{s_0-1}}
	+ \| \mathring{\bw} \cdot (d\bu,dh)\|_{\dot{H}_x^{s_0-1}}+  \| d\bu \cdot \nabla \mathring{\bw} \|_{\dot{H}_x^{s_0-2}}
	\\
	\lesssim	 & \|\bW\|_{\dot{H}_x^{s_0-1}}
	+ \| \mathring{\bw} \|_{{H}_x^{\frac32+}}\| (d\bu,dh) \|_{{H}_x^{s_0-1}}  .
	\end{split}
\end{equation}
Combining \eqref{We5}, \eqref{We6} and \eqref{We7}, we infer
\begin{equation}\label{We8}
	\begin{split}
		\| \Delta \mathring{\bw} \|_{\dot{H}_x^{s_0-2}}
		\lesssim	 & \|\bW\|_{\dot{H}_x^{s_0-1}}
		+ \| \mathring{\bw} \|_{{H}_x^{\frac32+}}\| (d\bu,dh) \|_{{H}_x^{s_0-1}}.
	\end{split}
\end{equation}
Seeing from \eqref{w0d}, we obtain
\begin{equation}\label{We9}
	\begin{split}
		\| \Delta {w^0} \|_{\dot{H}_x^{s_0-2}}  \lesssim	 & \| \Delta \mathring{\bw} \|_{\dot{H}_x^{s_0-2}}+ \|\Delta \bu \cdot \mathring{\bw} \|_{\dot{H}_x^{s_0-2}} + \| \nabla \mathring{\bw} \cdot \nabla \bu \|_{\dot{H}_x^{s_0-2}}.
	\end{split}
\end{equation}
By H\"older's inequality, then \eqref{We9} becomes
\begin{equation}\label{We10}
	\begin{split}
		\| \Delta {w^0} \|_{\dot{H}_x^{s_0-2}}
	\lesssim & 	\| \Delta \mathring{\bw} \|_{\dot{H}_x^{s_0-2}} + \| \mathring{\bw} \|_{{H}_x^{\frac32+}}\| (d\bu,dh) \|_{{H}_x^{s_0-1}}  .
	\end{split}
\end{equation}
Inserting \eqref{We8} to \eqref{We10}, we have
\begin{equation}\label{We12}
	\begin{split}
		\| \Delta \bw \|_{\dot{H}_x^{s_0-2}}
		\lesssim	 & \|\bW\|_{\dot{H}_x^{s_0-1}}
		+  \| \mathring{\bw} \|_{{H}_x^{\frac32+}}\| (d\bu,dh) \|_{{H}_x^{s_0-1}} .
	\end{split}
\end{equation}
By using \eqref{MFd}, \eqref{CEQ}, and \eqref{w0d}, we can also prove
\begin{align}\label{We13}
	\|\bW\|_{\dot{H}_x^{s_0-1}} \lesssim	& \|\bu \cdot d \bw \|_{\dot{H}_x^{s_0-1}} + \|\bw \cdot d h\|_{\dot{H}_x^{s_0-1}}
	\nonumber
	\\
	\lesssim	& \|\nabla \bw \|_{\dot{H}_x^{s_0-1}} + \| \mathring{\bw} \cdot (d\bu, d h)\|_{\dot{H}_x^{s_0-1}}
	\nonumber
	\\
	\lesssim & \| \Delta \bw \|_{\dot{H}_x^{s_0-2}}
		+ \| \mathring{\bw} \|_{{H}_x^{\frac32+}} \| (d\bu,dh) \|_{{H}_x^{s_0-1}} .
\end{align}
Seeing from \eqref{We12} and \eqref{We13}, there is only a lower order term  in $\| \Delta \bw\|_{\dot{H}_x^{s_0-2}}$ and $\| \bW \|_{\dot{H}_x^{s_0-1}}$. So we transfer the goal to bound $\| \bW \|_{\dot{H}_x^{s_0-1}}$.

\textbf{Step 1: $\| \bW \|_{\dot{H}_x^{s_0-1}}$}. We also first consider $\| \mathring{\bW} \|_{\dot{H}_x^{s_0-1}}$. On one hand, we have
\begin{equation}\label{We14}
	\| \mathring{\bW} \|_{\dot{H}_x^{s_0-1}}=  \| \nabla \mathring{\bW} \|_{\dot{H}_x^{s_0-2}}.
\end{equation}
By \eqref{d17}, \eqref{We14} and \eqref{MFd}, we therefore get
\begin{equation}\label{Wc00}
	\begin{split}
\| \mathring{\bW} \|_{\dot{H}_x^{s_0-1}}\lesssim &	\|\bG \|_{\dot{H}_x^{s_0-2}}+  \| \bW \cdot (d\bu,dh) \|_{\dot{H}_x^{s_0-2}}
	+   \| d\bw \cdot (d\bu,dh) \|_{\dot{H}_x^{s_0-2}}
	\\
	&	+  \| \bw \cdot d\bu \cdot dh \|_{\dot{H}_x^{s_0-2}}+  \| \bw \cdot \bw \cdot dh \|_{\dot{H}_x^{s_0-2}}.
	\end{split}
\end{equation}
Using \eqref{MFd}, \eqref{d8}, \eqref{d9}, \eqref{VVd} and H\"older's inequality, then \eqref{Wc00} becomes
\begin{equation}\label{Wc01}
	\begin{split}
		\| \mathring{\bW} \|_{\dot{H}_x^{s_0-1}}\lesssim &	\|\bG \|_{\dot{H}_x^{s_0-2}}
		+   \| d\bw \cdot (d\bu,dh) \|_{\dot{H}_x^{s_0-2}}
			+  \| \bw \cdot d\bu \cdot dh \|_{\dot{H}_x^{s_0-2}}+  \| \bw \cdot \bw \cdot dh \|_{\dot{H}_x^{s_0-2}}
		\\
		\lesssim &	\|\bG \|_{\dot{H}_x^{s_0-2}}
		+   \| \nabla \bw \cdot (d\bu,dh) \|_{\dot{H}_x^{s_0-2}}
			+  \| \bw \cdot d\bu \cdot dh \|_{\dot{H}_x^{s_0-2}}
		\\
			\lesssim &	\|\bG \|_{\dot{H}_x^{s_0-2}}
		+   \|\bw \|_{{H}_x^{s_0-\frac12} } \| (d\bu,dh) \|_{{H}_x^{s_0-1}}(1+ \| (d\bu,dh) \|_{{H}_x^{s_0-1}} ).
	\end{split}
\end{equation}
Since \eqref{d20}, so it's easy for us to get
\begin{equation}\label{Wc02}
	\begin{split}
		\| {W}^0 \|_{\dot{H}_x^{s_0-1}}\lesssim &	\| \mathring{\bW} \|_{\dot{H}_x^{s_0-1}}.
	\end{split}
\end{equation}
Combining \eqref{Wc01} and \eqref{Wc02}, we have
\begin{equation}\label{Wc03}
	\begin{split}
		\| \bW \|_{\dot{H}_x^{s_0-1}}
		\lesssim &	\|\bG \|_{\dot{H}_x^{s_0-2}}
		+   \|\bw \|_{{H}_x^{s_0-\frac12} } \| (d\bu,dh) \|_{{H}_x^{s_0-1}}(1+ \| (d\bu,dh) \|_{{H}_x^{s_0-1}} ).
	\end{split}
\end{equation}
Due to \eqref{We4}, \eqref{We12} and \eqref{Wc03}, it yields
\begin{equation*}
	\begin{split}
		\| \bw \|_{\dot{H}_x^{s_0}}
		\leq &	C\|\bG \|_{\dot{H}_x^{s_0-2}}
		+   C\|\bw \|_{{H}_x^{s_0-\frac12} } \| (d\bu,dh) \|_{{H}_x^{s_0-1}}(1+ \| (d\bu,dh) \|_{{H}_x^{s_0-1}} ).
	\end{split}
\end{equation*}
By interpolation formula and Young's inequality, we get
\begin{equation}\label{Wc04}
	\begin{split}
		\| \bw \|_{\dot{H}_x^{s_0}}
		\leq &	C\|\bG \|_{\dot{H}_x^{s_0-2}}
		+   C\|\bw \|^{\frac12}_{{H}_x^{s_0-1} } \|\bw \|^{\frac12}_{{H}_x^{s_0} } \| (d\bu,dh) \|_{{H}_x^{s_0-1}}(1+ \| (d\bu,dh) \|_{{H}_x^{s_0-1}} )
		\\
		\leq &	C\|\bG \|_{\dot{H}_x^{s_0-2}}
		+   \frac{1}{10}\|\bw \|_{{H}_x^{s_0} }+ C\|\bw \|_{{H}_x^{s_0-1} } \| (d\bu,dh) \|^2_{{H}_x^{s_0-1}}(1+ \| (d\bu,dh) \|^2_{{H}_x^{s_0-1}} )
		\\
		\leq &	C\|\bG \|_{\dot{H}_x^{s_0-2}}
		+   \frac{1}{10}\|\bw \|_{{H}_x^{s_0} }+ C\|d\bu\|_{{H}_x^{s_0-1} } \| (d\bu,dh) \|^2_{{H}_x^{s_0-1}}(1+ \| (d\bu,dh) \|^2_{{H}_x^{s_0-1}} ).
	\end{split}
\end{equation}
Therefore, we turn to bound $\| \bG \|_{\dot{H}^{s_0-2}}$ as follows.

\textbf{Step 2: $\| \bG \|_{\dot{H}^{s_0-2}}$}. By \eqref{SDe}, we have
\begin{equation}\label{We15}
	\begin{split}
		(u^0)^{-1} u^\kappa \partial_\kappa \left(G^\alpha-F^\alpha \right)
		=&  (u^0)^{-1} E^\alpha
		+  \partial^\alpha \left( (u^0)^{-1} \Gamma \right)
		 -\Gamma  \partial^\alpha \left( (u^0)^{-1}\right) .
	\end{split}
\end{equation}
Operating $\Lambda_x^{s_0-2}$ on \eqref{We15}, we have
\begin{equation}\label{We16}
	\begin{split}
			& \partial_t \left(  \Lambda_x^{s_0-2} (G^\alpha-F^\alpha)  \right) +(u^0)^{-1} u^i \partial_i \left( \Lambda_x^{s_0-2}(G^\alpha-F^\alpha) \right)
		\\
		= & \partial^\alpha \left(  \Lambda_x^{s_0-2} ( (u^0)^{-1} \Gamma ) \right)  -[\Lambda_x^{s_0-2}, (u^0)^{-1} u^i \partial_i](G^\alpha-F^\alpha)		
		\\
		& + \Lambda_x^{s_0-2} \left( (u^0)^{-1} E^\alpha  \right)   -\Lambda_x^{s_0-2}\left\{ \Gamma  \partial^\alpha \left( (u^0)^{-1}\right)  \right\}.
	\end{split}
\end{equation}

Multiplying $\Lambda_x^{s_0-2}(G_\alpha-F_\alpha)$ on \eqref{We16} and integrating it on $\mathbb{R}^3$, we can derive that
\begin{equation}\label{We17}
\begin{split}
  \frac{d}{dt} \|\Lambda_x^{s_0-2}(\bG -\bF)\|^2_{L_x^2}=& \mathrm{H}_1+\mathrm{H}_2+\mathrm{H}_3+\mathrm{H}_4,
\end{split}
\end{equation}
where we record
\begin{align}\label{We18}
& \mathrm{H}_1= \int_{\mathbb{R}^3} \partial^\alpha \left\{ \Lambda_x^{s_0-2} \big( (u^0)^{-1} \Gamma\big) \right\} \cdot \Lambda_x^{s_0-2}(G_\alpha- F_\alpha) dx,
\\
\label{We19}
& \mathrm{H}_2={\int_{\mathbb{R}^3}} \Lambda_x^{s_0-2} \left\{ (u^0)^{-1} E^\alpha  \right\} \cdot   \Lambda_x^{s_0-2} (G_\alpha- F_\alpha) dx,
\\
\label{We20}
& \mathrm{H}_3=-{\int_{\mathbb{R}^3}} [\Lambda_x^{s_0-2}, (u^0)^{-1} u^i \partial_i](G^\alpha-F^\alpha) \cdot   \Lambda_x^{s_0-2} (G_\alpha- F_\alpha) dx,
\\
\label{We21}
& \mathrm{H}_4=-{\int_{\mathbb{R}^3}} \Lambda_x^{s_0-2}\left\{ \Gamma  \partial^\alpha \left( (u^0)^{-1}\right)  \right\}\cdot   \Lambda_x^{s_0-2} (G_\alpha- F_\alpha) dx.
\end{align}
For $\mathrm{H}_1$ is the most difficult term, so we will discuss it later.

\textit{The bound for $\mathrm{H}_2, \mathrm{H}_3, \mathrm{H}_4$.} To get the bound of $\mathrm{H}_2, \mathrm{H}_3, \mathrm{H}_4$, commutator estimates and H\"older's inequality are enough. By H\"older's inequality, we have
\begin{align}\label{We22}
	 | \mathrm{H}_2 | \leq & \|\bE\|_{\dot{H}_x^{s_0-2}} ( \|\bG\|_{\dot{H}_x^{s_0-2}}+ \|\bF\|_{\dot{H}_x^{s_0-2}}).
\end{align}
So we need to know the estimate of $\bF$, $\bE$ and $\bG$. Recalling \eqref{YX0}, and using H\"older's inequality, then it follows
\begin{align}\label{We23}
	\nonumber
	\|\bF\|_{\dot{H}_x^{s_0-2}} \lesssim &  \|(d\bu,dh)\cdot d\bw \|_{\dot{H}_x^{s_0-2}}+ \|(d\bu,dh)\cdot (d\bu,dh) \cdot \bw \|_{\dot{H}_x^{s_0-2}}
	\\
	\lesssim & \|\bw\|_{{H}^{s_0}_x} \| (d\bu, dh)\|_{{H}^{s_0-1}_x} (1+\| (d\bu, dh)\|_{{H}^{s_0-1}_x}).
\end{align}
Seeing from \eqref{YX1}, and using H\"older's inequality and Lemma \ref{lpe}, it yields
\begin{equation}\label{We24}
\begin{split}
	\|\bE\|_{\dot{H}_x^{s_0-2}} \lesssim &  \|(d\bu,dh)\cdot d^2\bw \|_{\dot{H}_x^{s_0-2}}+ \|(d\bu,dh)\cdot (d\bu,dh) \cdot (d\bu,dh) \cdot \bw \|_{\dot{H}_x^{s_0-2}}
	\\
	& + \|(d\bu,dh)\cdot \bw \cdot (d^2\bu,d^2h) \|_{\dot{H}_x^{s_0-2}}+ \|(d\bu,dh)\cdot (d\bu,dh) \cdot d\bw \|_{\dot{H}_x^{s_0-2}}
	\\
	\lesssim & ( \|(d\bu,dh)\|_{\dot{B}_{\infty,2}^{s_0-2}}+ \|(d\bu,dh)\|_{L^\infty_x} ) \| \bw\|_{H^{s_0}_x}+\|(d\bu,dh)\|_{H^{s_0-1}_x}  \| \bw\|_{H^{s_0}_x}
	\\
	& +\|(d\bu,dh)\|^2_{H^{s_0-1}_x}  \| \bw\|_{H^{s_0}_x}+\|(d\bu,dh)\|^3_{H^{s_0-1}_x}  \| \bw\|_{H^{s_0}_x}.
\end{split}
\end{equation}
By \eqref{MFd}, we compute out
\begin{align}\label{We25}
	& G^\alpha
	= \epsilon^{\alpha \beta \gamma \delta}\epsilon_{\delta}^{\ \eta \mu \nu} u_\beta \partial_\gamma \left( u_\eta \partial_\mu w_\nu
	- c^{-2}_s  u_\eta w_\mu \partial_\nu h  \right).
\end{align}
Due to \eqref{We25}, using H\"older's inequality, we infer
\begin{equation}\label{We29}
	\begin{split}
	\| \bG \|_{\dot{H}_x^{s_0-2}} \lesssim &\|d^2\bw \|_{\dot{H}_x^{s_0-2}}+ \|(dh,d\bu)\cdot d\bw \|_{\dot{H}_x^{s_0-2}}+ \|\bw \cdot (d^2\bu,d^2h) \|_{\dot{H}_x^{s_0-2}}
	\\
	& + \|(dh,d\bu)\cdot (dh,d\bu) \cdot \bw \|_{\dot{H}_x^{s_0-2}}
	\\
	\lesssim & 	\| \bw\|_{H^{s_0}_x} +\|(d\bu,dh)\|_{H^{s_0-1}_x}  \| \bw\|_{H^{s_0}_x}
+\|(d\bu,dh)\|^2_{H^{s_0-1}_x}  \| \bw\|_{H^{s_0}_x} .
	\end{split}
\end{equation}
Inserting \eqref{We23}, \eqref{We24}, \eqref{We29} to \eqref{We22}, we therefore obtain
\begin{equation}
	\begin{split}
| \mathrm{H}_2 |\lesssim	& ( \|(d\bu,dh)\|_{\dot{B}_{\infty,2}^{s_0-2}}+ \|(d\bu,dh)\|_{L^\infty_x} ) \| \bw\|^2_{H^{s_0}_x} (1+ \|(d\bu,dh)\|^2_{H^{s_0-1}_x} )
		\\
		& +(\|(d\bu,dh)\|_{H^{s_0-1}_x}+\|(d\bu,dh)\|^5_{H^{s_0-1}_x})  \| \bw\|^2_{H^{s_0}_x}.
	\end{split}
\end{equation}
By Lemma \ref{ce}, we can derive
\begin{align}\label{We26}
	\nonumber| \mathrm{H}_3 | \leq & \|d\bu\|_{L^\infty_x}  \|\bG-\bF\|^2_{\dot{H}_x^{s_0-2}}
	\\
	\lesssim & \|d\bu\|_{L^\infty_x} \| \bw\|^2_{H^{s_0}_x}(1 +\|(d\bu,dh)\|^2_{H^{s_0-1}_x}
	+\|(d\bu,dh)\|^4_{H^{s_0-1}_x} ).
\end{align}
For $\mathrm{H}_4$, seeing \eqref{YXg}, \eqref{We23}, \eqref{We29}, and using H\"older's inequality, we can show that
\begin{align}\label{We27}
\nonumber| \mathrm{H}_4 | \leq & \|\Gamma \cdot d\bu\|_{\dot{H}_x^{s_0-2}}  \|\bG-\bF\|_{\dot{H}_x^{s_0-2}}
	\\
		\lesssim & (\|d\bu\|_{\dot{B}_{\infty,2}^{s_0-2}}+ \|d\bu\|_{L^\infty_x}) \| \bw\|^2_{H^{s_0}_x}(1 +\|(d\bu,dh)\|^2_{H^{s_0-1}_x}
	+\|(d\bu,dh)\|^3_{H^{s_0-1}_x} ).
\end{align}
\textit{The bound for $\mathrm{H}_1$}. If we take the derivatives in $\Gamma$ directly, there will be a loss of derivative through H\"older's inequality. So we seek to capture the cancellation of highest derivatives from integrating by parts. Therefore, we divide $\mathrm{H}_1$ by two parts
\begin{equation}\label{We28}
  \begin{split}
  \mathrm{H}_1
  =& \underbrace{ \int_{\mathbb{R}^3} \partial_\alpha (\Lambda_x^{s_0-2} \big( (u^0)^{-1} \Gamma\big))\cdot \Lambda_x^{s_0-2}G^\alpha dx }_{\equiv \mathrm{H}_{11}} \underbrace{-\int_{\mathbb{R}^3} \partial_\alpha (\Lambda_x^{s_0-2} \big( (u^0)^{-1} \Gamma\big))\cdot \Lambda_x^{s_0-2}F^\alpha dx }_{\equiv \mathrm{H}_{12}}.
  \end{split}
\end{equation}
Consider
\begin{equation}\label{We30}
	\begin{split}
		\partial_\alpha G^\alpha= \partial_\alpha(-\epsilon^{\alpha\beta\gamma\delta}u_\beta \partial_\gamma W_\delta)
		= -\epsilon^{\alpha\beta\gamma\delta} \partial_\alpha u_\beta \partial_\gamma W_\delta.
	\end{split}
\end{equation}
For $\mathrm{H}_{11}$, integrating it parts, which yields
\begin{equation}\label{We60}
\begin{split}
  \mathrm{H}_{11}
  =&\int_{\mathbb{R}^3} \partial_\alpha (\Lambda_x^{s_0-2} \big( (u^0)^{-1} \Gamma\big) \cdot \Lambda_x^{s_0-2}G^\alpha ) dx-  \int_{\mathbb{R}^3}\Lambda_x^{s_0-2} \big( (u^0)^{-1} \Gamma\big) \cdot \Lambda_x^{s_0-2}(\partial_\alpha G^\alpha) dx
\\
  =&\frac{d}{dt}\int_{\mathbb{R}^3} \Lambda_x^{s_0-2} \big( (u^0)^{-1} \Gamma\big)\cdot \Lambda_x^{s_0-2}G^0 dx
  \underbrace{-    \int_{\mathbb{R}^3}\Lambda_x^{s_0-2} \big( (u^0)^{-1} \Gamma\big) \cdot \Lambda_x^{s_0-2}(\partial_\alpha G^\alpha) dx }_{\equiv \mathrm{H}_{111}}.
\end{split}
\end{equation}
By using \eqref{We30}, we infer
\begin{equation*}
	\mathrm{H}_{111}= \epsilon^{\alpha\beta\gamma\delta} \int_{\mathbb{R}^3} \Lambda_x^{s_0-2}\big( (u^0)^{-1} \Gamma\big) \cdot \Lambda_x^{s_0-2}(\partial_\alpha u_\beta \partial_\gamma W_\delta) dx .
\end{equation*}
Similarly, we can calculate $\mathrm{H}_{12}$ by
\begin{equation}\label{We61}
	\begin{split}
	\mathrm{H}_{12}=&	- \frac{d}{dt} \int_{\mathbb{R}^3} \Lambda_x^{s_0-2} \big( (u^0)^{-1} \Gamma\big) \cdot \Lambda_x^{s_0-2}F^0 dx +  \underbrace{\int_{\mathbb{R}^3}  \Lambda_x^{s_0-2} \big( (u^0)^{-1} \Gamma\big) \cdot \Lambda_x^{s_0-2}(\partial_{\alpha} F^{\alpha}) dx}_{\equiv \mathrm{H}_{121}}.
	\end{split}
\end{equation}
Inserting \eqref{We28}, \eqref{We60} and \eqref{We61} to \eqref{We17}, it follows
\begin{footnotesize}
	\begin{equation}\label{We62}
		\begin{split}
			\frac{d}{dt} \left( \|\Lambda_x^{s_0-2}(\bG -\bF)\|^2_{L_x^2}- \int_{\mathbb{R}^3} \Lambda_x^{s_0-2} \big( (u^0)^{-1} \Gamma\big) \cdot \Lambda_x^{s_0-2}(G^0-F^0) dx \right) =& \mathrm{H}_{111}+\mathrm{H}_{121}+\mathrm{H}_3+\mathrm{H}_4.
		\end{split}
	\end{equation}
\end{footnotesize}
On the right hand side of \eqref{We62}, we have estimated $\mathrm{H}_3, \mathrm{H}_4$ by \eqref{We26}-\eqref{We27}. It still remains for us to consider $\mathrm{H}_{111}$ and $\mathrm{H}_{121}$.

\textit{The bound for  $\mathrm{H}_{111}$.} Using Lemma \ref{lpe}, H\"older's inequality, \eqref{YXg} and \eqref{MFd}, we infer
\begin{equation}\label{We63}
	\begin{split}
		 | \mathrm{H}_{111} | \lesssim & \| \Gamma\|_{\dot{H}_x^{s_0-2}} \|d\bu \cdot d\bW\|_{\dot{H}_x^{s_0-2}}
		\\
		\lesssim & \| d\bu \cdot d\bw\|_{\dot{H}_x^{s_0-2}} ( \|d\bu \cdot ( d\bu,dh) \cdot d\bw \|_{\dot{H}_x^{s_0-2}}+\|d\bu \cdot d^2\bw \|_{\dot{H}_x^{s_0-2}}  +\|\bw \cdot d^2 h \|_{\dot{H}_x^{s_0-2}} )
		\\
		\lesssim & ( \|(d\bu,dh)\|_{\dot{B}_{\infty,2}^{s_0-2}}+ \|(d\bu,dh)\|_{L^\infty_x} )
		( \|(d\bu, dh)\|^2_{H^{s_0-1}_x} +  \| \bw\|^2_{H^{s_0}_x} ) \|(d\bu, dh)\|_{H^{s_0-1}_x}
		\\
		&   + ( \|(d\bu, dh)\|^2_{H^{s_0-1}_x} +  \| \bw\|^2_{H^{s_0}_x} ) \|(d\bu, dh)\|^3_{H^{s_0-1}_x}.
	\end{split}
\end{equation}
\textit{The bound for  $\mathrm{H}_{121}$.} Due to \eqref{YX0}, and $u_\lambda w^\lambda=0$, we can compute out
\begin{align*}
\nonumber	\partial_\alpha F^\alpha=& \partial_\alpha ( -2\epsilon^{\alpha \beta \gamma \delta} c_s^{-2} u_\beta \partial_\gamma h W_\delta-2u^\alpha \partial^\gamma w^\lambda \partial_\gamma u_\lambda+2 u_\lambda \partial_\gamma u^\gamma \partial^\alpha w^\lambda - 2c_s^{-2}\partial_\lambda h \partial^\alpha w^\lambda )
	\\
\nonumber	=& \partial_\alpha ( -2\epsilon^{\alpha \beta \gamma \delta} c_s^{-2} u_\beta \partial_\gamma h W_\delta-2u^\alpha \partial^\gamma w^\lambda \partial_\gamma u_\lambda-2 \partial^\alpha u_\lambda \partial_\gamma u^\gamma w^\lambda - 2c_s^{-2}\partial_\lambda h \partial^\alpha w^\lambda )
	\\
\nonumber	=&
 -2\epsilon^{\alpha \beta \gamma \delta} c_s^{-2} \partial_\alpha  u_\beta \partial_\gamma h W_\delta
  -2\epsilon^{\alpha \beta \gamma \delta} c_s^{-2} u_\beta \partial_\gamma h \partial_\alpha W_\delta
		-2 \partial_\alpha   w^\lambda \partial^\alpha u_\lambda \partial_\gamma u^\gamma
	\\
\nonumber
& -2 w^\lambda \partial_\gamma u^\gamma  \partial_\alpha  \partial^\alpha u_\lambda
 -2 w^\lambda \partial^\alpha u_\lambda \partial_\alpha  \partial_\gamma u^\gamma
-2 w^\lambda \partial^\alpha u_\lambda \partial_\gamma \partial_\alpha u^\gamma
	\\
\nonumber	&-2 \partial_\alpha  u^\alpha \partial^\gamma w^\lambda \partial_\gamma u_\lambda
-2u^\alpha \partial_\gamma u_\lambda \partial_\alpha  \partial^\gamma w^\lambda
-2u^\alpha \partial^\gamma w^\lambda \partial_\alpha   \partial_\gamma u_\lambda
		\\
		\nonumber
	&+4 c_s^{-3}c'_s \partial_\alpha h  \partial_\lambda h \partial^\alpha w^\lambda
	- 2c_s^{-2} \partial_\lambda h \partial_\alpha  \partial^\alpha w^\lambda
	- 2c_s^{-2} \partial_\alpha  \partial_\lambda h \partial^\alpha w^\lambda.
\end{align*}
To be simple, we record $\partial_\alpha F^\alpha$ as
\begin{align}\label{We64}
	\partial_\alpha F^\alpha=&  \Psi-2u^\alpha \partial^\gamma w^\lambda \partial_\alpha   \partial_\gamma u_\lambda - 2c_s^{-2} \partial_\alpha  \partial_\lambda h \partial^\alpha w^\lambda ,
\end{align}
where
\begin{equation}\label{We65}
	\begin{split}
		\Psi	=&
		-2\epsilon^{\alpha \beta \gamma \delta} c_s^{-2} \partial_\alpha  u_\beta \partial_\gamma h W_\delta
		-2\epsilon^{\alpha \beta \gamma \delta} c_s^{-2} u_\beta \partial_\gamma h \partial_\alpha W_\delta
		-2 \partial_\alpha   w^\lambda \partial^\alpha u_\lambda \partial_\gamma u^\gamma
		\\
		& -2 w^\lambda \partial_\gamma u^\gamma  \partial_\alpha  \partial^\alpha u_\lambda
		-2 w^\lambda \partial^\alpha u_\lambda \partial_\alpha  \partial_\gamma u^\gamma
		-2 w^\lambda \partial^\alpha u_\lambda \partial_\gamma \partial_\alpha u^\gamma
		-2 \partial_\alpha  u^\alpha \partial^\gamma w^\lambda \partial_\gamma u_\lambda
		\\
		&	-2u^\alpha \partial_\gamma u_\lambda \partial_\alpha  \partial^\gamma w^\lambda
		+4 c_s^{-3}c'_s \partial_\alpha h  \partial_\lambda h \partial^\alpha w^\lambda
		- 2c_s^{-2} \partial_\lambda h \partial_\alpha  \partial^\alpha w^\lambda .
	\end{split}
\end{equation}
Due to \eqref{We65} and \eqref{MFd}, using Lemma \ref{lpe} and H\"older's inequality, we can bound
\begin{equation}\label{We66}
	\begin{split}
	\|	\Psi \|_{\dot{H}^{s_0-2}_x} \lesssim & \| (dh,d\bu) \cdot d^2\bw \|_{\dot{H}^{s_0-2}_x}
	+ \| (dh,d\bu) \cdot (dh,d\bu) \cdot d\bw \|_{\dot{H}^{s_0-2}_x}
	\\
	& + \| (dh,d\bu) \cdot \bw \cdot (d^2h,d^2\bu)  \|_{\dot{H}^{s_0-2}_x}
	+ \| (dh,d\bu) \cdot \bw \cdot (dh,d\bu) \cdot (dh,d\bu)  \|_{\dot{H}^{s_0-2}_x}
	\\
	\lesssim & ( \|(d\bu,dh)\|_{\dot{B}_{\infty,2}^{s_0-2}}+ \|(d\bu,dh)\|_{L^\infty_x} )
	( \|(d\bu, dh)\|_{H^{s_0-1}_x} +  \| \bw\|_{H^{s_0}_x} )
	\\
	&   + ( \|(d\bu, dh)\|_{H^{s_0-1}_x} +  \| \bw\|_{H^{s_0}_x} ) ( \|(d\bu, dh)\|^2_{H^{s_0-1}_x} + \|(d\bu, dh)\|^3_{H^{s_0-1}_x} ),
	\end{split}
\end{equation}
For $\mathrm{H}_{121}$, using \eqref{We64} and \eqref{We65}, we have
\begin{footnotesize}
	\begin{equation}\label{We68}
		\begin{split}
			\mathrm{H}_{121}= & \underbrace{ \int_{\mathbb{R}^3}  \Lambda_x^{s_0-2} ( (u^0)^{-1}\Gamma ) \cdot \Lambda_x^{s_0-2} \Psi dx }_{\equiv \mathrm{H}_{12a}}
			- \underbrace{ 2\int_{\mathbb{R}^3}  \Lambda_x^{s_0-2} ( (u^0)^{-1}\Gamma )\cdot \Lambda_x^{s_0-2}( u^\alpha \partial^\gamma w^\lambda \partial_\alpha   \partial_\gamma u_\lambda) dx }_{\equiv \mathrm{H}_{12b}}
			\\
			& - \underbrace{ 2\int_{\mathbb{R}^3}  \Lambda_x^{s_0-2} ( (u^0)^{-1}\Gamma )\cdot \Lambda_x^{s_0-2}( c_s^{-2} \partial_\alpha  \partial_\lambda h \partial^\alpha w^\lambda) dx }_{\equiv \mathrm{H}_{12c}} .
		\end{split}
	\end{equation}
\end{footnotesize}
We now discuss $\mathrm{H}_{12a}, \mathrm{H}_{12b}$, and $\mathrm{H}_{12c}$ as follows. Note \eqref{YXg} and \eqref{MFd}. By H\"older's inequality, \eqref{We66}, and Lemma \ref{lpe}, it yields
\begin{equation}\label{We67}
	\begin{split}
	| \mathrm{H}_{12a} | \lesssim & \| \Gamma \|_{\dot{H}^{s_0-2}_x} \| \Psi \|_{\dot{H}^{s_0-2}_x}
	\\
	\lesssim &  ( \|(d\bu,dh)\|_{\dot{B}_{\infty,2}^{s_0-2}}+ \|(d\bu,dh)\|_{L^\infty_x} )
	( \|(d\bu, dh)\|^2_{H^{s_0-1}_x} +  \| \bw\|^2_{H^{s_0}_x} ) \|(d\bu, dh)\|_{H^{s_0-1}_x}
	\\
	&   + ( \|(d\bu, dh)\|^2_{H^{s_0-1}_x} +  \| \bw\|^2_{H^{s_0}_x} ) (\|(d\bu, dh)\|^3_{H^{s_0-1}_x} + \|(d\bu, dh)\|^4_{H^{s_0-1}_x}).
	\end{split}
\end{equation}
For $\mathrm{H}_{12b}$, we cannot bound it directly like $\mathrm{H}_{12a}$. We decompose $\mathrm{H}_{12b}$ by
\begin{footnotesize}
	\begin{equation}\label{We680}
		\begin{split}
			\mathrm{H}_{12b}
			=&	-2\int_{\mathbb{R}^3}  \Lambda_x^{s_0-2} ( (u^0)^{-1}\Gamma )\cdot \left\{  \Lambda_x^{s_0-2}( u^0 \partial^\gamma w^\lambda \partial_t   \partial_\gamma u_\lambda) + \Lambda_x^{s_0-2}( u^i \partial^\gamma w^\lambda \partial_i   \partial_\gamma u_\lambda) \right\} dx
			\\
			=& \underbrace{	-2\int_{\mathbb{R}^3}  \Lambda_x^{s_0-2} ( (u^0)^{-1}\Gamma )\cdot \Lambda_x^{s_0-2}( u^0 \partial^0 w^\lambda \partial^2_t   u_\lambda) dx }_{\equiv \mathrm{H}_{12b1}}
			\underbrace{ -2\int_{\mathbb{R}^3}  \Lambda_x^{s_0-2} ( (u^0)^{-1}\Gamma ) \cdot \Lambda_x^{s_0-2}( u^0 \partial^i w^\lambda \partial_t   \partial_i u_\lambda) dx }_{\equiv \mathrm{H}_{12b2}}
			\\
			&  \underbrace{ -2\int_{\mathbb{R}^3}  \Lambda_x^{s_0-2} ( (u^0)^{-1}\Gamma )\cdot \Lambda_x^{s_0-2}( u^i \partial^\gamma w^\lambda \partial_i   \partial_\gamma u_\lambda) dx }_{\equiv \mathrm{H}_{12b3}} .
		\end{split}
	\end{equation}
\end{footnotesize}
Due to \eqref{WTe}, we have
\begin{align}\label{We72}
	\partial^2_t u_\kappa=&g^{0i} \partial^2_{ti}u_\kappa+g^{ij}\partial^2_{ij}u_\kappa+ c_s^2\Omega \mathrm{e}^{-h}W_\kappa- Q_\kappa,
	\\
	\label{We73}
	\partial^2_t h=&g^{0i} \partial^2_{ti}h+g^{ij}\partial^2_{ij}h-D.
\end{align}
Inserting \eqref{We72} to $\mathrm{H}_{12b1}$, we derive that
\begin{footnotesize}
\begin{equation}\label{We74}
	\begin{split}
		& \mathrm{H}_{12b1}
		\\
		=& 	\underbrace{ -2\int_{\mathbb{R}^3}  \Lambda_x^{s_0-2} ((u^0)^{-1} \Gamma) \cdot \Lambda_x^{s_0-2}( u_0 \partial_t w^\lambda g^{0i} \partial^2_{ti}u_\lambda) dx }_{\equiv \mathrm{H}_{12ba}}
		  \underbrace{ -2\int_{\mathbb{R}^3}  \Lambda_x^{s_0-2} ((u^0)^{-1} \Gamma) \cdot \Lambda_x^{s_0-2}( u_0 \partial_t w^\lambda g^{ij}\partial^2_{ij}u_\lambda) dx }_{\equiv \mathrm{H}_{12bb}}
		 \\
		 &\underbrace{ + 2\int_{\mathbb{R}^3}  \Lambda_x^{s_0-2} ((u^0)^{-1} \Gamma)\cdot \Lambda_x^{s_0-2}( u_0 \partial_t w^\lambda Q_\lambda) dx }_{\equiv \mathrm{H}_{12bc}}
	  \underbrace{- 2\int_{\mathbb{R}^3}  \Lambda_x^{s_0-2} ((u^0)^{-1} \Gamma) \cdot \Lambda_x^{s_0-2}( u_0 \partial_t w^\lambda c_s^2\Omega \mathrm{e}^{-h}W_\lambda) dx }_{\equiv \mathrm{H}_{12bd}}.
	\end{split}
\end{equation}
\end{footnotesize}
For $\mathrm{H}_{12ba}$, we can compute out
\begin{footnotesize}
\begin{equation}\label{We74a}
	\begin{split}
	& \mathrm{H}_{12ba}
	\\
	=	& 	 -2\int_{\mathbb{R}^3}  \Lambda_x^{s_0-2} ((u^0)^{-1}\Gamma) \cdot \partial_{i} \Lambda_x^{s_0-2}( u_0 \partial_t w^\lambda g^{0i} \partial_{t}u_\lambda) dx
	 + 2\int_{\mathbb{R}^3}  \Lambda_x^{s_0-2} ((u^0)^{-1}\Gamma) \cdot \Lambda_x^{s_0-2}( \partial_{i} u_0 \partial_t w^\lambda g^{0i} \partial_{t}u_\lambda) dx
	\\
	& + 2\int_{\mathbb{R}^3}  \Lambda_x^{s_0-2} ((u^0)^{-1}\Gamma) \cdot \Lambda_x^{s_0-2}( u_0 \partial_{i} \partial_t w^\lambda g^{0i} \partial_{t}u_\lambda) dx + 2\int_{\mathbb{R}^3}  \Lambda_x^{s_0-2} ((u^0)^{-1}\Gamma)\cdot \Lambda_x^{s_0-2}( u_0  \partial_t w^\lambda \partial_{i} g^{0i} \partial_{t}u_\lambda) dx  .
	\end{split}
\end{equation}
\end{footnotesize}
For the first term in the right hand side of \eqref{We74a}, we can bound it by Plancherel's formula and H\"older's inequality. For the rest terms, we can estimate them by using H\"older's inequality and Lemma \ref{lpe}. Therefore, it follows
\begin{equation}\label{We75}
	\begin{split}
		| \mathrm{H}_{12ba} | \lesssim	& 	 \| \Gamma \|_{\dot{H}_x^{s_0-\frac32}}  \|d \bw  d\bu \|_{\dot{H}_x^{s_0-\frac32}}
		 + \| \Gamma \|_{\dot{H}_x^{s_0-2}}  \|d\bu \nabla d \bw  \|_{\dot{H}_x^{s_0-2}}
		 \\
		 & + \| \Gamma \|_{\dot{H}_x^{s_0-2}}  \|(d\bu,dh) \cdot d \bw  \cdot (d\bu,dh) \|_{\dot{H}_x^{s_0-2}}
		 \\
		 \lesssim & ( \|(d\bu,dh)\|_{\dot{B}_{\infty,2}^{s_0-2}}+ \|(d\bu,dh)\|_{L^\infty_x} )
		 ( \|(d\bu, dh)\|^2_{H^{s_0-1}_x} +  \| \bw\|^2_{H^{s_0}_x} ) \|(d\bu, dh)\|_{H^{s_0-1}_x}
		 \\
		 &   + ( \|(d\bu, dh)\|^2_{H^{s_0-1}_x} +  \| \bw\|^2_{H^{s_0}_x} ) (\|(d\bu, dh)\|^2_{H^{s_0-1}_x} + \|(d\bu, dh)\|^4_{H^{s_0-1}_x}).
	\end{split}
\end{equation}
Similarly, we can calculate $\mathrm{H}_{12bb}$ by
\begin{footnotesize}
\begin{equation*}
	\begin{split}
		\mathrm{H}_{12bb}=& -2\int_{\mathbb{R}^3}  \Lambda_x^{s_0-2} ( (u^0)^{-1}\Gamma ) \cdot \partial_{i}\Lambda_x^{s_0-2}( u_0 \partial_t w^\lambda g^{ij}\partial_{j}u_\lambda) dx
		 +2\int_{\mathbb{R}^3}  \Lambda_x^{s_0-2} ( (u^0)^{-1}\Gamma ) \cdot \Lambda_x^{s_0-2}( \partial_{i} u_0 \partial_t w^\lambda g^{ij}\partial_{j}u_\lambda) dx
		\\
		& +2\int_{\mathbb{R}^3}  \Lambda_x^{s_0-2} ( (u^0)^{-1}\Gamma ) \cdot \Lambda_x^{s_0-2}( u_0 \partial_{i} \partial_t w^\lambda g^{ij}\partial_{j}u_\lambda) dx
		 +2\int_{\mathbb{R}^3}  \Lambda_x^{s_0-2} ( (u^0)^{-1}\Gamma ) \cdot \Lambda_x^{s_0-2}( u_0 \partial_{i} \partial_t w^\lambda \partial_{i} g^{ij}\partial_{j}u_\lambda) dx .
	\end{split}
\end{equation*}
\end{footnotesize}
Due to Plancherel formula, H\"older's inequality, and Lemma \ref{lpe} again, we can bound $\mathrm{H}_{12bb}$ by
\begin{equation}\label{We76}
	\begin{split}
		| \mathrm{H}_{12bb} | \lesssim	& 	 \| \Gamma \|_{\dot{H}_x^{s_0-\frac32}}  \|d \bw  d\bu \|_{\dot{H}_x^{s_0-\frac32}}
		+ \| \Gamma \|_{\dot{H}_x^{s_0-2}}  ( \|d\bu d \bw  d\bu \|_{\dot{H}_x^{s_0-2}} + \|d\bu \nabla d \bw  \|_{\dot{H}_x^{s_0-2}} )
		\\
		\lesssim & ( \|(d\bu,dh)\|_{\dot{B}_{\infty,2}^{s_0-2}}+ \|(d\bu,dh)\|_{L^\infty_x} )
		( \|(d\bu, dh)\|^2_{H^{s_0-1}_x} +  \| \bw\|^2_{H^{s_0}_x} ) \|(d\bu, dh)\|_{H^{s_0-1}_x}
		\\
		&   + ( \|(d\bu, dh)\|^2_{H^{s_0-1}_x} +  \| \bw\|^2_{H^{s_0}_x} ) (\|(d\bu, dh)\|^2_{H^{s_0-1}_x} + \|(d\bu, dh)\|^4_{H^{s_0-1}_x}).
	\end{split}
\end{equation}
By H\"older's inequality and Lemma \ref{lpe}, we can estimate $\mathrm{H}_{12bc}$ and $\mathrm{H}_{12bd}$ by
\begin{equation}\label{We77}
	\begin{split}
		| \mathrm{H}_{12bc} | \lesssim &
	 \| \Gamma \|_{\dot{H}_x^{s_0-2}}  \|(d\bu,dh)\cdot (d\bu,dh)\cdot d \bw  \|_{\dot{H}_x^{s_0-2}}
	 \\
	 \lesssim & ( \|(d\bu, dh)\|^2_{H^{s_0-1}_x} +  \| \bw\|^2_{H^{s_0}_x} ) (\|(d\bu, dh)\|^3_{H^{s_0-1}_x}+\|(d\bu, dh)\|^4_{H^{s_0-1}_x}).
	\end{split}
\end{equation}
and
\begin{equation}\label{We78}
	\begin{split}
		| \mathrm{H}_{12bd} | \lesssim  &
		\| \Gamma \|_{\dot{H}_x^{s_0-2}}  ( \|d\bw d \bw  \|_{\dot{H}_x^{s_0-2}} + \|d \bu dh d \bw  \|_{\dot{H}_x^{s_0-2}} )
		\\
		\lesssim & ( \|d\bu\|_{\dot{B}_{\infty,2}^{s_0-2}}+ \|d\bu\|_{L^\infty_x} ) \|d\bw\|_{H^{s_0-2}_x} \|d\bw\|_{H^{\frac12}_x}\|d\bw\|_{H^{s_0-1}_x}
		\\
		&+( \|(d\bu, dh)\|^2_{H^{s_0-1}_x} +  \| \bw\|^2_{H^{s_0}_x} ) (\|(d\bu, dh)\|^3_{H^{s_0-1}_x}+\|(d\bu, dh)\|^4_{H^{s_0-1}_x})
		\\
		\lesssim & ( \|d\bu\|_{\dot{B}_{\infty,2}^{s_0-2}}+ \|d\bu\|_{L^\infty_x} ) (\|d\bu\|_{H^{s_0-1}_x}+\|d\bu\|_{H^{s_0-1}_x}\|dh\|_{H^{s_0-1}_x}) \|d\bw\|^2_{H^{s_0-1}_x}
		\\
		&+( \|(d\bu, dh)\|^2_{H^{s_0-1}_x} +  \| \bw\|^2_{H^{s_0}_x} ) (\|(d\bu, dh)\|^2_{H^{s_0-1}_x}+\|(d\bu, dh)\|^4_{H^{s_0-1}_x}).
	\end{split}
\end{equation}
To summarize \eqref{We74}, \eqref{We75}, \eqref{We76}, \eqref{We77}, and \eqref{We78}, we get
\begin{equation}\label{We79}
	\begin{split}
		| \mathrm{H}_{12b1} |
		\lesssim & ( \|d\bu\|_{\dot{B}_{\infty,2}^{s_0-2}}+ \|d\bu\|_{L^\infty_x} )(\|(d\bu,dh)\|^2_{H^{s_0-1}_x}+ \|\bw\|^2_{H^{s_0}_x}) \|(d\bu,dh)\|_{H^{s_0-1}_x}
		\\
		& +( \|d\bu\|_{\dot{B}_{\infty,2}^{s_0-2}}+ \|d\bu\|_{L^\infty_x} )(\|(d\bu,dh)\|^2_{H^{s_0-1}_x}+ \|\bw\|^2_{H^{s_0}_x}) \|(d\bu,dh)\|^2_{H^{s_0-1}_x}
		\\
		&+( \|(d\bu, dh)\|^2_{H^{s_0-1}_x} +  \| \bw\|^2_{H^{s_0}_x} ) (\|(d\bu, dh)\|^2_{H^{s_0-1}_x} + \|(d\bu, dh)\|^4_{H^{s_0-1}_x}).
	\end{split}
\end{equation}
By the chain rule, we are able to obtain
\begin{footnotesize}
	\begin{equation*}\label{We83}
		\begin{split}
			& \mathrm{H}_{12b2}
			\\
			=& 	- 2\int_{\mathbb{R}^3}  \Lambda_x^{s_0-2} ( (u^0)^{-1}\Gamma )\cdot \partial_i \Lambda_x^{s_0-2}( u^0 \partial^i w^\lambda \partial_t  u_\lambda) dx
			+ 2\int_{\mathbb{R}^3}  \Lambda_x^{s_0-2} ( (u^0)^{-1}\Gamma )\cdot \Lambda_x^{s_0-2}( \partial_i u^0 \partial^i w^\lambda \partial_t u_\lambda) dx
			\\
			& + 2\int_{\mathbb{R}^3}  \Lambda_x^{s_0-2} ( (u^0)^{-1}\Gamma ) \cdot \Lambda_x^{s_0-2}(  u^0 \partial_i \partial^i w^\lambda \partial_t  u_\lambda) dx
			,
		\end{split}
	\end{equation*}
\end{footnotesize}
and
\begin{footnotesize}
	\begin{equation*}\label{We84}
		\begin{split}
			& \mathrm{H}_{12b3}
			\\
			=& 	 - 2\int_{\mathbb{R}^3}  \Lambda_x^{s_0-2} ( (u^0)^{-1}\Gamma ) \cdot \partial_i \Lambda_x^{s_0-2}( u^i \partial^\gamma w^\lambda   \partial_\gamma u_\lambda) dx
			+ 2\int_{\mathbb{R}^3}  \Lambda_x^{s_0-2} ( (u^0)^{-1}\Gamma ) \cdot \Lambda_x^{s_0-2}( \partial_i  u^i \partial^\gamma w^\lambda   \partial_\gamma u_\lambda) dx
			\\
			& +  2\int_{\mathbb{R}^3}  \Lambda_x^{s_0-2} ( (u^0)^{-1}\Gamma ) \cdot \Lambda_x^{s_0-2}(   u^i \partial_i \partial^\gamma w^\lambda   \partial_\gamma u_\lambda) dx .
		\end{split}
	\end{equation*}
\end{footnotesize}
Therefore, using the Plancherel formula, H\"older's inequality and Lemma \ref{lpe} again, we obtain
	\begin{equation}\label{We85}
		\begin{split}
			& |\mathrm{H}_{12b2}|+	|\mathrm{H}_{12b3}|
			\\
			\lesssim & \| \Gamma \|_{\dot{H}_x^{s_0-\frac32}}  \|d \bw  d\bu \|_{\dot{H}_x^{s_0-\frac32}}
			+ \| \Gamma \|_{\dot{H}_x^{s_0-2}}  ( \|d\bu d \bw  d\bu \|_{\dot{H}_x^{s_0-2}} + \|d\bu \nabla d \bw  \|_{\dot{H}_x^{s_0-2}} )
			\\
			\lesssim & ( \|(d\bu,dh)\|_{\dot{B}_{\infty,2}^{s_0-2}}+ \|(d\bu,dh)\|_{L^\infty_x} )
			( \|(d\bu, dh)\|^2_{H^{s_0-1}_x} +  \| \bw\|^2_{H^{s_0}_x} ) \|(d\bu, dh)\|_{H^{s_0-1}_x}
			\\
			&   + ( \|(d\bu, dh)\|^2_{H^{s_0-1}_x} +  \| \bw\|^2_{H^{s_0}_x} ) (\|(d\bu, dh)\|^2_{H^{s_0-1}_x} + \|(d\bu, dh)\|^4_{H^{s_0-1}_x}) .
		\end{split}
	\end{equation}
Summarizing our outcome \eqref{We680}, \eqref{We79}, and \eqref{We85}, so it yields
\begin{equation}\label{We80}
	\begin{split}
		| \mathrm{H}_{12b} |
		\lesssim & ( \|d\bu\|_{\dot{B}_{\infty,2}^{s_0-2}}+ \|d\bu\|_{L^\infty_x} )(\|(d\bu,dh)\|^2_{H^{s_0-1}_x}+ \|\bw\|^2_{H^{s_0}_x}) \|(d\bu,dh)\|_{H^{s_0-1}_x}
		\\
		& +( \|d\bu\|_{\dot{B}_{\infty,2}^{s_0-2}}+ \|d\bu\|_{L^\infty_x} )(\|(d\bu,dh)\|^2_{H^{s_0-1}_x}+ \|\bw\|^2_{H^{s_0}_x}) \|(d\bu,dh)\|^2_{H^{s_0-1}_x}
		\\
		&+( \|(d\bu, dh)\|^2_{H^{s_0-1}_x} +  \| \bw\|^2_{H^{s_0}_x} ) (\|(d\bu, dh)\|^2_{H^{s_0-1}_x} + \|(d\bu, dh)\|^4_{H^{s_0-1}_x}).
	\end{split}
\end{equation}
We still need to bound $\mathrm{H}_{12c}$. A direct calculation tells us
\begin{footnotesize}
\begin{equation}\label{We82}
	\begin{split}
		\mathrm{H}_{12c}= & \underbrace{ -2\int_{\mathbb{R}^3}  \Lambda_x^{s_0-2} ( (u^0)^{-1}\Gamma ) \cdot \Lambda_x^{s_0-2}( c_s^{-2} \partial^2_t   h \partial_t w^0) dx }_{\equiv \mathrm{H}_{12c1}}
		\underbrace{- 2\int_{\mathbb{R}^3}  \Lambda_x^{s_0-2} ( (u^0)^{-1}\Gamma ) \cdot \Lambda_x^{s_0-2}( c_s^{-2} \partial_t \partial_i  h \partial_t w^i) dx }_{\equiv \mathrm{H}_{12c2}}
		\\
		&  \underbrace{ -2\int_{\mathbb{R}^3}  \Lambda_x^{s_0-2} ( (u^0)^{-1}\Gamma ) \cdot \Lambda_x^{s_0-2}( c_s^{-2} \partial_i  \partial_\lambda h \partial^i w^\lambda) dx  }_{\equiv \mathrm{H}_{12c3}}.
	\end{split}
\end{equation}
\end{footnotesize}
Substituting \eqref{We73} to $\mathrm{H}_{12c1}$, we have
\begin{footnotesize}
	\begin{equation*}
		\begin{split}
			\mathrm{H}_{12c1}= & -2\int_{\mathbb{R}^3}  \Lambda_x^{s_0-2} ( (u^0)^{-1}\Gamma) \cdot \Lambda_x^{s_0-2}( c_s^{-2} g^{0i} \partial^2_{ti}h \partial_t w^0) dx
			- 2\int_{\mathbb{R}^3}  \Lambda_x^{s_0-2} ( (u^0)^{-1}\Gamma) \cdot \Lambda_x^{s_0-2}( c_s^{-2} g^{ij}\partial^2_{ij}h \partial_t w^0) dx
			\\
			&+2\int_{\mathbb{R}^3}  \Lambda_x^{s_0-2} ( (u^0)^{-1}\Gamma) \cdot \Lambda_x^{s_0-2}( c_s^{-2} g^{ij} D \partial_t w^0 ) dx .
		\end{split}
	\end{equation*}
\end{footnotesize}
In a similar way on estimating $\mathrm{H}_{12ba}$, so we obtain
\begin{equation}\label{We86}
	\begin{split}
		| \mathrm{H}_{12c1} | \lesssim &  \| \Gamma \|_{\dot{H}_x^{s_0-\frac32}}  \|d \bw  dh \|_{\dot{H}_x^{s_0-\frac32}}
		+ \| \Gamma \|_{\dot{H}_x^{s_0-2}}  \|dh\nabla d \bw  \|_{\dot{H}_x^{s_0-2}}
		\\
		& + \| \Gamma \|_{\dot{H}_x^{s_0-2}}  \|(d\bu,dh)\cdot d \bw \cdot  (d\bu,dh) \|_{\dot{H}_x^{s_0-2}}
		\\
		\lesssim & ( \|(d\bu,dh)\|_{\dot{B}_{\infty,2}^{s_0-2}}+ \|(d\bu,dh)\|_{L^\infty_x} )
		( \|(d\bu, dh)\|^2_{H^{s_0-1}_x} +  \| \bw\|^2_{H^{s_0}_x} ) \|(d\bu, dh)\|_{H^{s_0-1}_x}
		\\
		&   + ( \|(d\bu, dh)\|^2_{H^{s_0-1}_x} +  \| \bw\|^2_{H^{s_0}_x} ) (\|(d\bu, dh)\|^2_{H^{s_0-1}_x} + \|(d\bu, dh)\|^4_{H^{s_0-1}_x}).
	\end{split}
\end{equation}
In a similar way on estimating $\mathrm{H}_{12bb}$, we infer
\begin{equation}\label{We87}
	\begin{split}
		& | \mathrm{H}_{12c2} |+ | \mathrm{H}_{12c3} |
		\\
		\lesssim &  \| \Gamma \|_{\dot{H}_x^{s_0-\frac32}}  \|d \bw  dh \|_{\dot{H}_x^{s_0-\frac32}}
		+ \| \Gamma \|_{\dot{H}_x^{s_0-2}}  ( \|dh\nabla d \bw  \|_{\dot{H}_x^{s_0-2}} +\|(d\bu,dh)\cdot d \bw \cdot  (d\bu,dh) \|_{\dot{H}_x^{s_0-2}} )
		\\
		\lesssim & ( \|(d\bu,dh)\|_{\dot{B}_{\infty,2}^{s_0-2}}+ \|(d\bu,dh)\|_{L^\infty_x} )
		( \|(d\bu, dh)\|^2_{H^{s_0-1}_x} +  \| \bw\|^2_{H^{s_0}_x} ) \|(d\bu, dh)\|_{H^{s_0-1}_x}
		\\
		&   + ( \|(d\bu, dh)\|^2_{H^{s_0-1}_x} +  \| \bw\|^2_{H^{s_0}_x} ) (\|(d\bu, dh)\|^2_{H^{s_0-1}_x} + \|(d\bu, dh)\|^4_{H^{s_0-1}_x}).
	\end{split}
\end{equation}
By summarizing \eqref{We82}, \eqref{We86}, and \eqref{We87}, we obtain
\begin{equation}\label{We89}
	\begin{split}
		| \mathrm{H}_{12c} |
		\lesssim & ( \|(d\bu,dh)\|_{\dot{B}_{\infty,2}^{s_0-2}}+ \|(d\bu,dh)\|_{L^\infty_x} )
		( \|(d\bu, dh)\|^2_{H^{s_0-1}_x} +  \| \bw\|^2_{H^{s_0}_x} ) \|(d\bu, dh)\|_{H^{s_0-1}_x}
		\\
		&   + ( \|(d\bu, dh)\|^2_{H^{s_0-1}_x} +  \| \bw\|^2_{H^{s_0}_x} ) (\|(d\bu, dh)\|^2_{H^{s_0-1}_x} + \|(d\bu, dh)\|^4_{H^{s_0-1}_x}).
	\end{split}
\end{equation}
Due to \eqref{We68}, \eqref{We67}, \eqref{We80}, and \eqref{We89}, it yields
\begin{equation}\label{We90}
	\begin{split}
		| \mathrm{H}_{121} |
		\lesssim & ( \|(d\bu,dh)\|_{\dot{B}_{\infty,2}^{s_0-2}}+ \|(d\bu,dh)\|_{L^\infty_x} )
		( \|(d\bu, dh)\|^2_{H^{s_0-1}_x} +  \| \bw\|^2_{H^{s_0}_x} ) \|(d\bu, dh)\|_{H^{s_0-1}_x}
		\\
		&   + ( \|(d\bu, dh)\|^2_{H^{s_0-1}_x} +  \| \bw\|^2_{H^{s_0}_x} ) (\|(d\bu, dh)\|^2_{H^{s_0-1}_x} + \|(d\bu, dh)\|^4_{H^{s_0-1}_x}).
	\end{split}
\end{equation}
To summarize \eqref{We62}, \eqref{We63}, \eqref{We90}, \eqref{We26}, and \eqref{We27}, we therefore have
\begin{equation}\label{We91}
	\begin{split}
		& \frac{d}{dt} \left( \|\Lambda_x^{s_0-2}(\bG -\bF)\|^2_{L_x^2}+ \int_{\mathbb{R}^3} \Lambda_x^{s_0-2} ( (u^0)^{-1}\Gamma ) \cdot \Lambda_x^{s_0-2}(G^0-F^0) dx \right)
		\\
		\lesssim & ( \|(d\bu,dh)\|_{\dot{B}_{\infty,2}^{s_0-2}}+ \|(d\bu,dh)\|_{L^\infty_x} )
		( \|(d\bu, dh)\|^2_{H^{s_0-1}_x} +  \| \bw\|^2_{H^{s_0}_x} ) ( 1+ \|(d\bu, dh)\|^2_{H^{s_0-1}_x})
		\\
		& + ( \|(d\bu,dh)\|_{\dot{B}_{\infty,2}^{s_0-2}}+ \|(d\bu,dh)\|_{L^\infty_x} )
		( \|(d\bu, dh)\|^2_{H^{s_0-1}_x} +  \| \bw\|^2_{H^{s_0}_x} ) \|(d\bu, dh)\|^3_{H^{s_0-1}_x}
			\\
		& + ( \|(d\bu,dh)\|_{\dot{B}_{\infty,2}^{s_0-2}}+ \|(d\bu,dh)\|_{L^\infty_x} )
		( \|(d\bu, dh)\|^2_{H^{s_0-1}_x} +  \| \bw\|^2_{H^{s_0}_x} ) \|(d\bu, dh)\|^4_{H^{s_0-1}_x}
		\\
		&   + ( \|(d\bu, dh)\|^2_{H^{s_0-1}_x} +  \| \bw\|^2_{H^{s_0}_x} ) (\|(d\bu, dh)\|^2_{H^{s_0-1}_x} + \|(d\bu, dh)\|^4_{H^{s_0-1}_x}).
	\end{split}
\end{equation}
Integrating on \eqref{We91} from $[0,t]$($t>0$), we obtain
\begin{equation}\label{We92}
	\begin{split}
		& \|\Lambda_x^{s_0-2}(\bG -\bF)(t,\cdot)\|^2_{L_x^2}+ {\int_{\mathbb{R}^3}} \Lambda_x^{s_0-2} ( (u^0)^{-1}\Gamma ) \cdot \Lambda_x^{s_0-2}(G^0-F^0) (t,\cdot)dx
		\\
		\lesssim & \|\Lambda_x^{s_0-2}(\bG -\bF)(0,\cdot)\|^2_{L_x^2}+ {\int_{\mathbb{R}^3}} \Lambda_x^{s_0-2} ( (u^0)^{-1}\Gamma ) \cdot \Lambda_x^{s_0-2}(G^0-F^0) (0,\cdot)dx
		\\
		 & + {\int^t_0} ( \|(d\bu,dh)\|_{\dot{B}_{\infty,2}^{s_0-2}}+ \|(d\bu,dh)\|_{L^\infty_x} )
		( \|(d\bu, dh)\|^2_{H^{s_0-1}_x} +  \| \bw\|^2_{H^{s_0}_x} ) d\tau
		\\
		& + {\int^t_0} ( \|(d\bu,dh)\|_{\dot{B}_{\infty,2}^{s_0-2}}+ \|(d\bu,dh)\|_{L^\infty_x} )
		( \|(d\bu, dh)\|^2_{H^{s_0-1}_x} +  \| \bw\|^2_{H^{s_0}_x} ) \|(d\bu, dh)\|^2_{H^{s_0-1}_x} d\tau
		\\
		& + {\int^t_0} ( \|(d\bu,dh)\|_{\dot{B}_{\infty,2}^{s_0-2}}+ \|(d\bu,dh)\|_{L^\infty_x} )
		( \|(d\bu, dh)\|^2_{H^{s_0-1}_x} +  \| \bw\|^2_{H^{s_0}_x} ) \|(d\bu, dh)\|^3_{H^{s_0-1}_x} d\tau
		\\
		& + {\int^t_0} ( \|(d\bu,dh)\|_{\dot{B}_{\infty,2}^{s_0-2}}+ \|(d\bu,dh)\|_{L^\infty_x} )
		( \|(d\bu, dh)\|^2_{H^{s_0-1}_x} +  \| \bw\|^2_{H^{s_0}_x} ) \|(d\bu, dh)\|^4_{H^{s_0-1}_x} d\tau
		\\
		&   + {\int^t_0} ( \|(d\bu, dh)\|^2_{H^{s_0-1}_x} +  \| \bw\|^2_{H^{s_0}_x} ) (\|(d\bu, dh)\|^2_{H^{s_0-1}_x} + \|(d\bu, dh)\|^4_{H^{s_0-1}_x}) d\tau.
	\end{split}
\end{equation}
For the left hand of \eqref{We92}, for $t> 0$, we observe that
\begin{equation}\label{We93}
	\begin{split}
		& \|\Lambda_x^{s_0-2}(\bG -\bF)(t,\cdot)\|^2_{L_x^2}+ {\int_{\mathbb{R}^3}} \Lambda_x^{s_0-2} ( (u^0)^{-1}\Gamma ) \cdot \Lambda_x^{s_0-2}(G^0-F^0)(t,\cdot) dx
		\\
		\geq & \|\bG(t,\cdot)\|^2_{\dot{H}^{s_0-2}_x} -\|\bF(t,\cdot)\|^2_{\dot{H}^{s_0-2}_x}- \| \Gamma(t,\cdot) \|_{\dot{H}^{s_0-2}_x} ( \|G^0(t,\cdot) \|_{\dot{H}^{s_0-2}_x}+\|F^0(t,\cdot) \|_{\dot{H}^{s_0-2}_x} )
		\\
		\geq & \|\bG(t,\cdot)\|^2_{\dot{H}^{s_0-2}_x} -C( \|(d\bu, dh)\|_{H^{s_0-1}_x}+ \|(d\bu, dh)\|^2_{H^{s_0-1}_x}) \|\bw  \|_{H^{s_0}_x} \|\bw \|_{H^{s_0-\frac12}_x}
		\\
		& -C( \|(d\bu, dh)(t,\cdot)\|^2_{H^{s_0-1}_x}+ \|(d\bu, dh)(t,\cdot)\|^4_{H^{s_0-1}_x}) \|\bw (t,\cdot) \|^2_{H^{s_0-\frac12}_x},
	\end{split}
\end{equation}
and
\begin{equation}\label{We94}
	\begin{split}
		& \|\Lambda_x^{s_0-2}(\bG -\bF)(0,\cdot)\|^2_{L_x^2}+ {\int_{\mathbb{R}^3}} \Lambda_x^{s_0-2} ( (u^0)^{-1}\Gamma ) \cdot \Lambda_x^{s_0-2}(G^0-F^0) (0,\cdot)dx
		\\
		\lesssim & \|(\mathring{\bu}_0, u_0^0-1)\|^2_{H_x^s} + \|h_0\|^2_{H_x^s} + \|\bw_0\|^2_{{H}^{s_0}_x}+\|(\mathring{\bu}_0, u_0^0-1)\|^6_{H_x^s} + \|h_0\|^6_{H_x^s} + \|\bw_0\|^6_{{H}^{s_0}_x} .
	\end{split}
\end{equation}
By the interpolation formula, it follows
\begin{equation}\label{We95}
	\|\bw (t,\cdot) \|_{H^{s_0-\frac12}_x} \leq \|\bw (t,\cdot) \|^{\frac12}_{H^{s_0-1}_x} \|\bw (t,\cdot) \|^{\frac12}_{H^{s_0}_x} .
\end{equation}
Substituting \eqref{We95} to \eqref{We93}, and using Young's inequality, we can show that\footnote{In \eqref{We96}, we give a coefficient $\frac{1}{10(1+C)}$ in $\frac{1}{10(1+C)}\|\bw (t,\cdot) \|^2_{H^{s_0}_x}$ by Young's inequality, where the constant $C$ is totally the same with in the right hand of \eqref{Wc04}.}
	\begin{equation}\label{We96}
		\begin{split}
			& \|\Lambda_x^{s_0-2}(\bG -\bF)\|^2_{L_x^2}+ {\int_{\mathbb{R}^3}} \Lambda_x^{s_0-2} ( (u^0)^{-1}\Gamma ) \cdot \Lambda_x^{s_0-2}(G^0-F^0) dx
			\\
			\geq & \|\bG\|^2_{\dot{H}^{s_0-2}_x}- \frac{1}{10(1+C)}\|\bw\|^2_{H^{s_0}_x}
			- C\left( \|(d\bu, dh)\|^4_{H^{s_0-1}_x}+ \|(d\bu, dh)\|^8_{H^{s_0-1}_x} \right) \|\bw\|^2_{H^{s_0-1}_x}
			\\
			\geq & \|\bG\|^2_{\dot{H}^{s_0-2}_x} - \frac{1}{10(1+C)}\|\bw\|^2_{H^{s_0}_x}
			- C\left( \|(d\bu, dh)\|^4_{H^{s_0-1}_x}+ \|(d\bu, dh)\|^{10}_{H^{s_0-1}_x} \right) .
		\end{split}
	\end{equation}
Above we use $\|\bw (t,\cdot) \|_{H^{s_0-1}_x} \lesssim \|d\bu (t,\cdot) \|_{H^{s_0-1}_x}+\|d\bu (t,\cdot) \|^2_{H^{s_0-1}_x}$. By \eqref{We92}, \eqref{We94} and \eqref{We96}, we have proved
\begin{equation}\label{We97}
	\begin{split}
		& \|\bG(t,\cdot)\|^2_{\dot{H}^{s_0-2}_x}- \frac{1}{10(1+C)}\|\bw (t,\cdot) \|^2_{H^{s_0}_x}
		\\
		\lesssim
		&  \|(\mathring{\bu}_0, u_0^0-1)\|^2_{H_x^s} + \|h_0\|^2_{H_x^s} + \|\bw_0\|^2_{{H}^{s_0}_x}+\|(\mathring{\bu}_0, u_0^0-1)\|^6_{H_x^s} + \|h_0\|^6_{H_x^s}
		\\
		& + \|\bw_0\|^6_{{H}^{s_0}_x}
	+  \|(d\bu, dh)(t,\cdot)\|^4_{H^{s_0-1}_x}+ \|(d\bu, dh)(t,\cdot)\|^{10}_{H^{s_0-1}_x}
		\\
		& + \textstyle{\int^t_0} ( \|(d\bu,dh)\|_{\dot{B}_{\infty,2}^{s_0-2}}+ \|(d\bu,dh)\|_{L^\infty_x} )
		( \|(d\bu, dh)\|^2_{H^{s_0-1}_x} +  \| \bw\|^2_{H^{s_0}_x} ) d\tau
		\\
		& + \textstyle{\int^t_0} ( \|(d\bu,dh)\|_{\dot{B}_{\infty,2}^{s_0-2}}+ \|(d\bu,dh)\|_{L^\infty_x} )
		( \|(d\bu, dh)\|^2_{H^{s_0-1}_x} +  \| \bw\|^2_{H^{s_0}_x} ) \|(d\bu, dh)\|^2_{H^{s_0-1}_x} d\tau
		\\
		& + \textstyle{\int^t_0} ( \|(d\bu,dh)\|_{\dot{B}_{\infty,2}^{s_0-2}}+ \|(d\bu,dh)\|_{L^\infty_x} )
		( \|(d\bu, dh)\|^2_{H^{s_0-1}_x} +  \| \bw\|^2_{H^{s_0}_x} ) \|(d\bu, dh)\|^3_{H^{s_0-1}_x} d\tau
		\\
		& + \textstyle{\int^t_0} ( \|(d\bu,dh)\|_{\dot{B}_{\infty,2}^{s_0-2}}+ \|(d\bu,dh)\|_{L^\infty_x} )
		( \|(d\bu, dh)\|^2_{H^{s_0-1}_x} +  \| \bw\|^2_{H^{s_0}_x} ) \|(d\bu, dh)\|^4_{H^{s_0-1}_x} d\tau
		\\
		&   + \textstyle{\int^t_0} ( \|(d\bu, dh)\|^2_{H^{s_0-1}_x} +  \| \bw\|^2_{H^{s_0}_x} ) (\|(d\bu, dh)\|^2_{H^{s_0-1}_x} + \|(d\bu, dh)\|^4_{H^{s_0-1}_x}) d\tau.
	\end{split}
\end{equation}
Due to \eqref{Wc04} and \eqref{We97}, we find that
\begin{equation}\label{We98}
	\begin{split}
		& \|\bw(t,\cdot)\|^2_{\dot{H}^{s_0}_x}- \frac{C}{10(1+C)}\|\bw (t,\cdot) \|^2_{H^{s_0}_x}-\frac{1}{10}\|\bw (t,\cdot) \|^2_{H^{s_0}_x}
		\\
		\lesssim
		& \|\mathring{\bu}_0\|^2_{H_x^s} + \|h_0\|^2_{H_x^s} + \|\bw_0\|^2_{{H}^{s_0}_x}+\|\mathring{\bu}_0\|^6_{H_x^s} + \|h_0\|^6_{H_x^s} + \|\bw_0\|^6_{{H}^{s_0}_x}
		\\
		& 	+  \|(d\bu, dh)(t,\cdot)\|^3_{H^{s_0-1}_x}
		+  \|(d\bu, dh)(t,\cdot)\|^4_{H^{s_0-1}_x}+ \|(d\bu, dh)(t,\cdot)\|^{10}_{H^{s_0-1}_x}
		\\
		& + {\int^t_0} ( \|(d\bu,dh)\|_{\dot{B}_{\infty,2}^{s_0-2}}+ \|(d\bu,dh)\|_{L^\infty_x} )
		( \|(d\bu, dh)\|^2_{H^{s_0-1}_x} +  \| \bw\|^2_{H^{s_0}_x} ) d\tau
		\\
		& + {\int^t_0} ( \|(d\bu,dh)\|_{\dot{B}_{\infty,2}^{s_0-2}}+ \|(d\bu,dh)\|_{L^\infty_x} )
		( \|(d\bu, dh)\|^2_{H^{s_0-1}_x} +  \| \bw\|^2_{H^{s_0}_x} ) \|(d\bu, dh)\|^2_{H^{s_0-1}_x} d\tau
		\\
		& + {\int^t_0} ( \|(d\bu,dh)\|_{\dot{B}_{\infty,2}^{s_0-2}}+ \|(d\bu,dh)\|_{L^\infty_x} )
		( \|(d\bu, dh)\|^2_{H^{s_0-1}_x} +  \| \bw\|^2_{H^{s_0}_x} ) \|(d\bu, dh)\|^3_{H^{s_0-1}_x} d\tau
		\\
		& + {\int^t_0} ( \|(d\bu,dh)\|_{\dot{B}_{\infty,2}^{s_0-2}}+ \|(d\bu,dh)\|_{L^\infty_x} )
		( \|(d\bu, dh)\|^2_{H^{s_0-1}_x} +  \| \bw\|^2_{H^{s_0}_x} ) \|(d\bu, dh)\|^4_{H^{s_0-1}_x} d\tau
		\\
		&   + {\int^t_0} ( \|(d\bu, dh)\|^2_{H^{s_0-1}_x} +  \| \bw\|^2_{H^{s_0}_x} ) (\|(d\bu, dh)\|^2_{H^{s_0-1}_x} + \|(d\bu, dh)\|^4_{H^{s_0-1}_x}) d\tau.
	\end{split}
\end{equation}
Above we also use $\|u_0^0-1\|_{H_x^s} \lesssim \|\mathring{\bu}_0\|_{H_x^s}$. Due to \eqref{QHl}, by a classical energy estimate, we infer
\begin{equation}\label{We100}
	\|(h, \mathring{\bu})\|^2_{{H}^{s}_x}(t) \lesssim \|(h_0, \mathring{\bu}_0)\|^2_{{H}^{s}_x}+ {\int^t_0}  \|(d\mathring{\bu},dh)\|_{L^\infty_x} \|(h, \mathring{\bu})\|^2_{{H}^{s}_x}(\tau) d\tau.
\end{equation}
By \eqref{We3}, \eqref{We98}, \eqref{We100}, and \eqref{Wc0}, we can conclude
\begin{equation}\label{We99}
	\begin{split}
		 & \|(h, \mathring{\bu})\|^2_{{H}^{s}_x}(t)+ \|u^0-1\|^2_{{H}^{s}_x}(t) + \|\bw\|^2_{{H}^{s_0}_x}(t)- \frac{C}{10(1+C)}\|\bw (t,\cdot) \|^2_{H^{s_0}_x}-\frac{1}{10}\|\bw (t,\cdot) \|^2_{H^{s_0}_x}
		\\
		\lesssim
		 & \|\mathring{\bu}_0\|^2_{H_x^s} + \|h_0\|^2_{H_x^s} + \|\bw_0\|^2_{{H}^{s_0}_x}+\|\mathring{\bu}_0\|^6_{H_x^s} + \|h_0\|^6_{H_x^s} + \|\bw_0\|^6_{{H}^{s_0}_x}
		\\
		& 	+  \textstyle{\sup}_{[0,t]} \|(d\bu, dh)\|^3_{H^{s_0-1}_x}
		+  \textstyle{\sup}_{[0,t]}\|(d\bu, dh)\|^4_{H^{s_0-1}_x}+ \textstyle{\sup}_{[0,t]}\|(d\bu, dh)\|^{10}_{H^{s_0-1}_x}
		\\
		& + {\int^t_0} ( \|(d\bu,dh)\|_{\dot{B}_{\infty,2}^{s_0-2}}+ \|(d\bu,dh)\|_{L^\infty_x} )
		( \|(d\bu, dh)\|^2_{H^{s_0-1}_x} +  \| \bw\|^2_{H^{s_0}_x} ) d\tau
		\\
		& + {\int^t_0} ( \|(d\bu,dh)\|_{\dot{B}_{\infty,2}^{s_0-2}}+ \|(d\bu,dh)\|_{L^\infty_x} )
		( \|(d\bu, dh)\|^2_{H^{s_0-1}_x} +  \| \bw\|^2_{H^{s_0}_x} ) \|(d\bu, dh)\|^2_{H^{s_0-1}_x} d\tau
		\\
		& + {\int^t_0} ( \|(d\bu,dh)\|_{\dot{B}_{\infty,2}^{s_0-2}}+ \|(d\bu,dh)\|_{L^\infty_x} )
		( \|(d\bu, dh)\|^2_{H^{s_0-1}_x} +  \| \bw\|^2_{H^{s_0}_x} ) \|(d\bu, dh)\|^3_{H^{s_0-1}_x} d\tau
		\\
		& + {\int^t_0} ( \|(d\bu,dh)\|_{\dot{B}_{\infty,2}^{s_0-2}}+ \|(d\bu,dh)\|_{L^\infty_x} )
		( \|(d\bu, dh)\|^2_{H^{s_0-1}_x} +  \| \bw\|^2_{H^{s_0}_x} ) \|(d\bu, dh)\|^4_{H^{s_0-1}_x} d\tau
		\\
		&   + {\int^t_0} ( \|(d\bu, dh)\|^2_{H^{s_0-1}_x} +  \| \bw\|^2_{H^{s_0}_x} ) (\|(d\bu, dh)\|^2_{H^{s_0-1}_x} + \|(d\bu, dh)\|^3_{H^{s_0-1}_x}) d\tau.
	\end{split}
\end{equation}
Since the left hand side of \eqref{We99} $\geq \frac12\|(h, \mathring{\bu})\|^2_{{H}^{s}_x}(t)+ \frac12\|u^0-1\|^2_{{H}^{s}_x}(t) + \frac12\|\bw\|^2_{{H}^{s_0}_x}(t)$, so we have
\begin{equation}\label{We103}
	\begin{split}
		& \|(h, \mathring{\bu})\|^2_{{H}^{s}_x}(t)+ \|u^0-1\|^2_{{H}^{s}_x}(t) + \|\bw\|^2_{{H}^{s_0}_x}(t)
		\\
		\lesssim
		& \|\mathring{\bu}_0\|^2_{H_x^s} + \|h_0\|^2_{H_x^s} + \|\bw_0\|^2_{{H}^{s_0}_x}+\|\mathring{\bu}_0\|^6_{H_x^s} + \|h_0\|^6_{H_x^s} + \|\bw_0\|^6_{{H}^{s_0}_x}
		\\
		& 	+  \textstyle{\sup}_{[0,t]} \|(d\bu, dh)(t,\cdot)\|^3_{H^{s_0-1}_x}
		+  \textstyle{\sup}_{[0,t]}\|(d\bu, dh)\|^4_{H^{s_0-1}_x}+ \textstyle{\sup}_{[0,t]}\|(d\bu, dh)\|^{10}_{H^{s_0-1}_x}
		\\
		& + {\int^t_0} ( \|(d\bu,dh)\|_{\dot{B}_{\infty,2}^{s_0-2}}+ \|(d\bu,dh)\|_{L^\infty_x} )
		( \|(d\bu, dh)\|^2_{H^{s_0-1}_x} +  \| \bw\|^2_{H^{s_0}_x} ) d\tau
		\\
		& + {\int^t_0} ( \|(d\bu,dh)\|_{\dot{B}_{\infty,2}^{s_0-2}}+ \|(d\bu,dh)\|_{L^\infty_x} )
		( \|(d\bu, dh)\|^2_{H^{s_0-1}_x} +  \| \bw\|^2_{H^{s_0}_x} ) \|(d\bu, dh)\|^2_{H^{s_0-1}_x} d\tau
		\\
		& + {\int^t_0} ( \|(d\bu,dh)\|_{\dot{B}_{\infty,2}^{s_0-2}}+ \|(d\bu,dh)\|_{L^\infty_x} )
		( \|(d\bu, dh)\|^2_{H^{s_0-1}_x} +  \| \bw\|^2_{H^{s_0}_x} ) \|(d\bu, dh)\|^3_{H^{s_0-1}_x} d\tau
		\\
		& + {\int^t_0} ( \|(d\bu,dh)\|_{\dot{B}_{\infty,2}^{s_0-2}}+ \|(d\bu,dh)\|_{L^\infty_x} )
		( \|(d\bu, dh)\|^2_{H^{s_0-1}_x} +  \| \bw\|^2_{H^{s_0}_x} ) \|(d\bu, dh)\|^4_{H^{s_0-1}_x} d\tau
		\\
		&   + {\int^t_0} ( \|(d\bu, dh)\|^2_{H^{s_0-1}_x} +  \| \bw\|^2_{H^{s_0}_x} ) (\|(d\bu, dh)\|^2_{H^{s_0-1}_x} + \|(d\bu, dh)\|^4_{H^{s_0-1}_x}) d\tau.
	\end{split}
\end{equation}
Taking $a=s$ in \eqref{VHe}, \eqref{VHea}, and \eqref{VHeb}, using \eqref{We103}, we have
\begin{equation*}
	\begin{split}
		 E_s(t)
		\lesssim
		& E_s(0)
		+ ( E_s(0)^{\frac32}+  E_s(0)^2+E_s(0)^{5} ) \exp (5\int^t_0 \|d\bu, dh\|_{L^\infty_x}d\tau)
		\\
		&+E_s(0)^3 + {\int^t_0} ( \|(d\bu,dh)\|_{\dot{B}_{\infty,2}^{s_0-2}}+ \|(d\bu,dh)\|_{L^\infty_x} )
		E_s(\tau) d\tau
		\\
		& + {\int^t_0} ( \|(d\bu,dh)\|_{\dot{B}_{\infty,2}^{s_0-2}}+ \|(d\bu,dh)\|_{L^\infty_x} ) E_s(0) \exp (\int^\tau_0 \|d\bu, dh\|_{L^\infty_x}dz)
		E_s(\tau) d\tau
		\\
		& + {\int^t_0} ( \|(d\bu,dh)\|_{\dot{B}_{\infty,2}^{s_0-2}}+ \|(d\bu,dh)\|_{L^\infty_x} ) E_s(0)^{\frac32} \exp (\frac32\int^\tau_0 \|d\bu, dh\|_{L^\infty_x}dz)
		E_s(\tau) d\tau
		\\
		& + {\int^t_0} ( \|(d\bu,dh)\|_{\dot{B}_{\infty,2}^{s_0-2}}+ \|(d\bu,dh)\|_{L^\infty_x} ) E_s(0)^2 \exp (2\int^\tau_0 \|d\bu, dh\|_{L^\infty_x}dz)
		E_s(\tau) d\tau
		\\
		&   + {\int^t_0} E_s(\tau) \left\{ E_s(0) \exp (\int^\tau_0 \|d\bu, dh\|_{L^\infty_x}dz)+ E_s(0)^{2} \exp (2\int^\tau_0 \|d\bu, dh\|_{L^\infty_x}dz) \right\} d\tau.
	\end{split}
\end{equation*}
Observing $E_s(0)^{\frac32}, E_s(0)^{2}, E_s(0)^{2} \leq \max\{E_s(0), E_s(0)^5\} $, so we have
\begin{equation*}
	\begin{split}
		E_s(t)
		\lesssim
		& E_s(0)+E_s(0)^5
		+ ( E_s(0)+E_s(0)^5 ) \exp (5\int^t_0 \|d\bu, dh\|_{L^\infty_x}d\tau)
		\\
		& + {\int^t_0} ( \|(d\bu,dh)\|_{\dot{B}_{\infty,2}^{s_0-2}}+ \|(d\bu,dh)\|_{L^\infty_x} )
		E_s(\tau) d\tau
		\\
		& + {\int^t_0} ( \|(d\bu,dh)\|_{\dot{B}_{\infty,2}^{s_0-2}}+ \|(d\bu,dh)\|_{L^\infty_x} ) E_s(0) \exp (\int^t_0 \|d\bu, dh\|_{L^\infty_x}dz)
		E_s(\tau) d\tau
		\\
		& + {\int^t_0} ( \|(d\bu,dh)\|_{\dot{B}_{\infty,2}^{s_0-2}}+ \|(d\bu,dh)\|_{L^\infty_x} ) E_s(0)^2 \exp (2\int^t_0 \|d\bu, dh\|_{L^\infty_x}dz)
		E_s(\tau) d\tau
		\\
		&   + {\int^t_0} E_s(\tau) \left\{ E_s(0) \exp (\int^t_0 \|d\bu, dh\|_{L^\infty_x}dz)+ E_s(0)^{2} \exp (2\int^t_0 \|d\bu, dh\|_{L^\infty_x}dz) \right\} d\tau.
	\end{split}
\end{equation*}
Therefore, \eqref{WHa}-\eqref{WL} can be immediately obtained after an application of Gronwall's inequality.
\end{proof}
With a relaxed regularity assumption on the vorticity, we can obtain another type of energy estimates.
\subsection{Energy estimate for Theorem \ref{dingli2}}\label{sec3.2}
\begin{theorem}[Energy estimate: type 2]\label{VEt}
	Let $(h,\bu)$ be a solution of \eqref{REE}. Let $\bw$ be defined in \eqref{VVd}. For $2<s\leq \frac52$, the following energy estimate hold:
	\begin{equation}\label{WH}
		\|\tilde{E}_s(t)\|_{H^{s}_x} \leq \tilde{E}_0  \exp \left\{    5{\int^t_0} \|(dh,d\bu)\|_{L^\infty_x}   d\tau \cdot \exp \big(  5{\int^t_0} \|(dh,d\bu)\|_{L^\infty_x}   d\tau \big)\right\},
	\end{equation}
	where
	\begin{equation}\label{WLa}
		\begin{split}
			\tilde{E}_s(t) =  & \|h\|^2_{H_x^s} + \|u^0-1\|^2_{H_x^s}+ \|\mathring{\bu}\|^2_{H_x^s} +\|\bw\|^2_{H_x^{2}},
			\\
			\tilde{E}_0= & C \left( \|(h_0,\mathring{\bu}_0)\|^2_{H^s} +\|\bw_0\|^2_{H^{2}}+\|(h_0,\mathring{\bu}_0)\|^{10}_{H^s}+\|\bw_0\|^{10}_{H^{2}} \right).
		\end{split}
	\end{equation}
Above, $C$ is an universal constant.
\end{theorem}
\begin{proof}
	For we have proved \eqref{VHe} and \eqref{Wc0}, so we only need to bound $\| \bw \|_{{H}_x^{2}}$. We consider
	\begin{equation}\label{Wd0}
		\| \bw \|_{{H}_x^{2}}=\| \bw \|_{L_x^{2}}+ \| \bw \|_{\dot{H}_x^{2}}.
	\end{equation}
	So it's natural for us to utilize the transport structure of $\bw$ and the higher derivatives $\bW$ and $\bG$. By \eqref{We3}, so we have
	\begin{equation}\label{Wd3}
		\begin{split}
			\|\bw(t,\cdot)\|^2_{L_x^2} \lesssim & \|\bw_0\|^2_{L_x^2} + \int^t_0 \|dh , d\bu \|_{L^\infty_x} \|\bw\|^2_{L_x^2} d\tau.
		\end{split}
	\end{equation}
	Note that
	\begin{equation}\label{Wd4}
		\| \bw \|_{\dot{H}_x^{2}} = \| \Delta \bw \|_{L_x^{2}}.
	\end{equation}
	Recall $\mathring{\bw}=(w^1,w^2,w^3)$. Let us first consider $	\|  \mathring{\bw} \|_{\dot{H}_x^{2}} $. Recall $\Delta_{H}$ as in \eqref{dh}. Then we get
	\begin{equation}\label{Wd5}
		\| \Delta_H \mathring{\bw} \|_{L_x^{2}}  \lesssim	\| \Delta \mathring{\bw} \|_{L_x^{2}}  \lesssim 	\| \Delta_H \mathring{\bw} \|_{L_x^{2}} .
	\end{equation}
	Due to \eqref{d5}, we have
	\begin{equation}\label{Wd6}
		\| \Delta_H \mathring{\bw} \|_{L_x^{2}}
		\lesssim	
		\| \bu \cdot \mathrm{vort} \bw \|_{\dot{H}_x^{1}}
		+ \| \bw \cdot (d\bu,dh)\|_{\dot{H}_x^{1}}+  \| d\bu \cdot \nabla \mathring{\bw} \|_{L_x^{2}}.
	\end{equation}
	Using \eqref{MFd}, \eqref{w0d}, and Lemma \ref{ps}, it follows
	\begin{equation}\label{Wd7}
		\begin{split}
			\| \Delta_H \mathring{\bw} \|_{L_x^{2}}  \lesssim	 & \|\bu \cdot \bW\|_{\dot{H}_x^{1}} + \| \bw \cdot (d\bu,dh)\|_{\dot{H}_x^{1}}+  \| d\bu \cdot \nabla \mathring{\bw} \|_{L_x^{2}}
			\\
			\lesssim	& \|\bW\|_{\dot{H}_x^{1}}
			+ \| \mathring{\bw} \cdot (d\bu,dh)\|_{\dot{H}_x^{1}}+  \| d\bu \cdot \nabla \mathring{\bw} \|_{L_x^{2}}
			\\
			\lesssim	 & \|\bW\|_{\dot{H}_x^{1}}
			+ \| \mathring{\bw} \|_{{H}_x^{\frac32}}\| (d\bu,dh) \|_{{H}_x^{s-1}}  .
		\end{split}
	\end{equation}
	Combining \eqref{Wd5}, \eqref{Wd6} and \eqref{Wd7}, we obtain
	\begin{equation}\label{Wd8}
		\begin{split}
			\| \Delta \mathring{\bw} \|_{L_x^{2}}
			\lesssim	 & \|\bW\|_{\dot{H}_x^{1}}
			+ \| \mathring{\bw} \|_{{H}_x^{\frac32}}\| (d\bu,dh) \|_{{H}_x^{s-1}}.
		\end{split}
	\end{equation}
	Seeing from \eqref{w0d}, we can obtain
	\begin{equation}\label{Wd9}
		\begin{split}
			\| \Delta {w^0} \|_{L_x^{2}}  \lesssim	 & \| \Delta \mathring{\bw} \|_{L_x^{2}}+ \|\Delta \bu \cdot \mathring{\bw} \|_{L_x^{2}} + \| \nabla \mathring{\bw} \cdot \nabla \bu \|_{L_x^{2}}.
		\end{split}
	\end{equation}
	By H\"older's inequality, then \eqref{Wd9} becomes
	\begin{equation}\label{Wd10}
		\begin{split}
			\| \Delta {w^0} \|_{L_x^{2}}
			\lesssim & 	\| \Delta \mathring{\bw} \|_{L_x^{2}} + \| \mathring{\bw} \|_{{H}_x^{\frac32}}\| (d\bu,dh) \|_{{H}_x^{s-1}}  .
		\end{split}
	\end{equation}
	Inserting \eqref{Wd8} to \eqref{Wd10}, we then derive that
	\begin{equation}\label{Wd12}
		\begin{split}
			\| \Delta \bw \|_{L_x^{2}}
			\lesssim	 & \|\bW\|_{\dot{H}_x^{1}}
			+  \| \mathring{\bw} \|_{{H}_x^{\frac32}}\| (d\bu,dh) \|_{{H}_x^{s-1}} .
		\end{split}
	\end{equation}
	By using \eqref{MFd}, \eqref{CEQ}, and \eqref{w0d}, we can also prove
	\begin{align}\label{Wd13}
			\|\bW\|_{\dot{H}_x^{1}} \lesssim	& \|\bu \cdot d \bw \|_{\dot{H}_x^{1}} + \|\bw \cdot d h\|_{\dot{H}_x^{1}}
			\nonumber
			\\
			\lesssim	& \|\nabla \bw \|_{\dot{H}_x^{1}} + \| \mathring{\bw} \cdot (d\bu, d h)\|_{\dot{H}_x^{1}}
			\nonumber
			\\
			\lesssim & \| \Delta \bw \|_{L_x^{2}}
			+ \| \mathring{\bw} \|_{{H}_x^{\frac32}} \| (d\bu,dh) \|_{{H}_x^{s-1}} .
	\end{align}
From \eqref{Wd12} and \eqref{Wd13}, there is only a lower order term between $\| \bW \|_{\dot{H}_x^{1}}$ and $\| \Delta \bw \|_{L_x^{2}}$. So we bound $\| \bW \|_{\dot{H}_x^{1}}$ as follows.

	\textbf{Step 1: $\| \bW \|_{\dot{H}_x^{1}}$}. We also first consider $\| \mathring{\bW} \|_{\dot{H}_x^{1}}$. On one hand, we have
	\begin{equation}\label{Wd14}
		\| \mathring{\bW} \|_{\dot{H}_x^{1}}=  \| \nabla \mathring{\bW} \|_{L^2_x}.
	\end{equation}
	By \eqref{d17}, \eqref{Wd14} and \eqref{MFd}, we therefore get
	\begin{equation}\label{Wf00}
		\begin{split}
			\| \mathring{\bW} \|_{\dot{H}_x^{1}}\lesssim &	\|\bG \|_{L^2_x}+  \| \bW \cdot (d\bu,dh) \|_{L^2_x}
			+   \| d\bw \cdot (d\bu,dh) \|_{L^2_x}
			\\
			&	+  \| \bw \cdot d\bu \cdot dh \|_{L^2_x}+  \| \bw \cdot \bw \cdot dh \|_{L^2_x}.
		\end{split}
	\end{equation}
	Using \eqref{MFd}, \eqref{d8}, \eqref{d9}, \eqref{VVd} and H\"older's inequality, then \eqref{Wf00} becomes
	\begin{equation}\label{Wf01}
		\begin{split}
			\| \mathring{\bW} \|_{\dot{H}_x^{1}}\lesssim &	\|\bG \|_{L^2_x}
			+   \| d\bw \cdot (d\bu,dh) \|_{L^2_x}
			+  \| \bw \cdot d\bu \cdot dh \|_{L^2_x}+  \| \bw \cdot \bw \cdot dh \|_{L^2_x}
			\\
			\lesssim &	\|\bG \|_{L^2_x}
			+   \| \nabla \bw \cdot (d\bu,dh) \|_{L^2_x}
			+  \| \bw \cdot d\bu \cdot dh \|_{L^2_x}
			\\
			\lesssim &	\|\bG \|_{L^2_x}
			+   \|\bw \|_{{H}_x^{\frac32} } \| (d\bu,dh) \|_{{H}_x^{s-1}}(1+ \| (d\bu,dh) \|_{{H}_x^{s-1}} ).
		\end{split}
	\end{equation}
	Since \eqref{d20}, we infer
	\begin{equation}\label{Wf02}
		\begin{split}
			\| {W}^0 \|_{\dot{H}_x^{1}}\lesssim &	\| \mathring{\bW} \|_{\dot{H}_x^{1}}.
		\end{split}
	\end{equation}
	Adding \eqref{Wf01} and \eqref{Wf02}, we have
	\begin{equation}\label{Wf03}
		\begin{split}
			\| \bW \|_{\dot{H}_x^{1}}
			\lesssim &	\|\bG \|_{L^2_x}
			+   \|\bw \|_{{H}_x^{\frac32} } \| (d\bu,dh) \|_{{H}_x^{s-1}}(1+ \| (d\bu,dh) \|_{{H}_x^{s-1}} ).
		\end{split}
	\end{equation}
	Combining \eqref{Wd12} with \eqref{Wf03}, it yields
	\begin{equation}\label{Wf05}
		\begin{split}
			\| \bw \|_{\dot{H}_x^{1}}
			\leq &	C\|\bG \|_{L^2_x}
			+   C\|\bw \|_{{H}_x^{\frac32} } \| (d\bu,dh) \|_{{H}_x^{s-1}}(1+ \| (d\bu,dh) \|_{{H}_x^{s-1}} ).
		\end{split}
	\end{equation}
	By interpolation formula and Young's inequality, we get
	\begin{equation}\label{Wf04}
		\begin{split}
			\| \bw \|_{\dot{H}_x^{s-1}}
			\leq &	C\|\bG \|_{L^2_x}
			+   C\|\bw \|^{\frac12}_{{H}_x^{1} } \|\bw \|^{\frac12}_{{H}_x^{2} } \| (d\bu,dh) \|_{{H}_x^{s-1}}(1+ \| (d\bu,dh) \|_{{H}_x^{s-1}} )
			\\
			\leq &	C\|\bG \|_{L^2_x}
			+   \frac{1}{10}\|\bw \|_{{H}_x^{2} }+ C\|\bw \|_{{H}_x^{s-1} } \| (d\bu,dh) \|^2_{{H}_x^{s-1}}(1+ \| (d\bu,dh) \|^2_{{H}_x^{s-1}} )
			\\
			\leq &	C\|\bG \|_{L^2_x}
			+   \frac{1}{10}\|\bw \|_{{H}_x^{2} }+ C\|d\bu\|_{{H}_x^{s-1} } \| (d\bu,dh) \|^2_{{H}_x^{s-1}}(1+ \| (d\bu,dh) \|^2_{{H}_x^{s-1}} ).
		\end{split}
	\end{equation}

	\textbf{Step 2: $\| \bG \|_{L^2_x}$}. By \eqref{We15},
	multiplying $ G_\alpha-F_\alpha $ on \eqref{We15} and integrating it on $\mathbb{R}^3$, we can derive that
	\begin{equation}\label{Wd17}
		\begin{split}
			\frac{d}{dt} \|\bG -\bF\|^2_{L_x^2}=& \mathrm{I}_0+\mathrm{I}_1+\mathrm{I}_2+\mathrm{I}_3,
		\end{split}
	\end{equation}
	where
	\begin{align}\label{Wd18}
			& \mathrm{I}_0= \int_{\mathbb{R}^3} \partial_i ( \frac{u^i}{u^0} )  (G^\alpha- F^\alpha)\cdot (G_\alpha- F_\alpha) dx,
			\\\label{Wd18a}
		& \mathrm{I}_1= \int_{\mathbb{R}^3} \partial^\alpha  \big( (u^0)^{-1} \Gamma\big) \cdot (G_\alpha- F_\alpha) dx,
		\\
		\label{Wd19}
		& \mathrm{I}_2={\int_{\mathbb{R}^3}}  (u^0)^{-1} E^\alpha \cdot    (G_\alpha- F_\alpha) dx,
		\\
		\label{Wd21}
		& \mathrm{I}_3=-{\int_{\mathbb{R}^3}} \Gamma  \partial^\alpha \left( (u^0)^{-1}\right) \cdot    (G_\alpha- F_\alpha) dx.
	\end{align}
	For $\mathrm{I}_1$ is the most difficult term, so we postpone to discuss it.

	\textit{The bound for $\mathrm{I}_0, \mathrm{I}_2, \mathrm{I}_3$.} To get the bound of $\mathrm{I}_0, \mathrm{I}_2, \mathrm{I}_3$, using the H\"older's inequality is enough. Due to H\"older's inequality, we get
	\begin{align}\label{Wd221}
	| \mathrm{I}_0 | \leq & \|d \bu\|_{L^\infty_x} ( \|\bG\|^2_{L^2_x}+ \|\bF\|^2_{L^2_x}).
\end{align}
and
	\begin{align}\label{Wd22}
		| \mathrm{I}_2 | \leq & \|\bE\|_{L^2_x} ( \|\bG\|_{L^2_x}+ \|\bF\|_{L^2_x}).
	\end{align}
	Recalling \eqref{YX0}, and using H\"older's inequality, then it follows
	\begin{align}\label{Wd23}
		\nonumber
		\|\bF\|_{L^2_x} \lesssim &  \|(d\bu,dh)\cdot d\bw \|_{L^2_x}+ \|(d\bu,dh)\cdot (d\bu,dh) \cdot \bw \|_{L^2_x}
		\\
		\lesssim & \|\bw\|_{{H}^{2}_x} \| (d\bu, dh)\|_{{H}^{s-1}_x} (1+\| (d\bu, dh)\|_{{H}^{s-1}_x}).
	\end{align}
	Seeing from \eqref{YX1}, and using H\"older's inequality, it yields
	\begin{equation}\label{Wd24}
		\begin{split}
			\|\bE\|_{L^2_x} \lesssim &  \|(d\bu,dh)\cdot d^2\bw \|_{L^2_x}+ \|(d\bu,dh)\cdot (d\bu,dh) \cdot (d\bu,dh) \cdot \bw \|_{L^2_x}
			\\
			& + \|(d\bu,dh)\cdot \bw \cdot (d^2\bu,d^2h) \|_{L^2_x}+ \|(d\bu,dh)\cdot (d\bu,dh) \cdot d\bw \|_{L^2_x}
			\\
			\lesssim & \|(d\bu,dh)\|_{L^\infty_x}  ( \| \bw\|_{H^{2}_x}+\|(d\bu,dh)\|_{H^{s-1}_x}  \| \bw\|_{H^{2}_x})
			\\
			& +\|(d\bu,dh)\|^2_{H^{s-1}_x}  \| \bw\|_{H^{2}_x}+\|(d\bu,dh)\|^3_{H^{s-1}_x}  \| \bw\|_{H^{2}_x}.
		\end{split}
	\end{equation}
	By \eqref{We25}, so we have
	\begin{equation}\label{Wd29}
		\begin{split}
			\| \bG \|_{L^2_x} \lesssim &\|d^2\bw \|_{L^2_x}+ \|(dh,d\bu)\cdot d\bw \|_{L^2_x}+ \|\bw \cdot (d^2\bu,d^2h) \|_{L^2_x}
			\\
			& + \|(dh,d\bu)\cdot (dh,d\bu) \cdot \bw \|_{L^2_x}
			\\
			\lesssim & 	\| \bw\|_{H^{2}_x} +\|(d\bu,dh)\|_{H^{s-1}_x}  \| \bw\|_{H^{2}_x}
			+\|(d\bu,dh)\|^2_{H^{s-1}_x}  \| \bw\|_{H^{2}_x} .
		\end{split}
	\end{equation}
	By using \eqref{Wd23},\eqref{Wd24}, \eqref{Wd29}, \eqref{Wd22}, and \eqref{Wd221}, we have
	\begin{equation}\label{Wd220}
		\begin{split}
			| \mathrm{I}_0 |+| \mathrm{I}_2 | \lesssim &  \|d\bu\|_{L^\infty_x} ( \|(d\bu,dh)\|^2_{H^{s-1}_x} + \| \bw\|^2_{H^{2}_x})(1 +\|(d\bu,dh)\|^2_{H^{s-1}_x}   )
			\\
			&+ \|d\bu\|_{L^\infty_x} ( \|(d\bu,dh)\|^2_{H^{s-1}_x} + \| \bw\|^2_{H^{2}_x})\|(d\bu,dh)\|^3_{H^{s-1}_x}
			\\
			& + ( \|(d\bu,dh)\|^2_{H^{s-1}_x} + \| \bw\|^2_{H^{2}_x})(\|(d\bu,dh)\|_{H^{s-1}_x} +\|(d\bu,dh)\|^2_{H^{s-1}_x}  )
			\\
			& + ( \|(d\bu,dh)\|^2_{H^{s-1}_x} + \| \bw\|^2_{H^{2}_x})(\|(d\bu,dh)\|^3_{H^{s-1}_x} +\|(d\bu,dh)\|^4_{H^{s-1}_x}  ) .
		\end{split}
	\end{equation}
	For $\mathrm{I}_3$, we can show that
	\begin{align}\label{Wd27}
		| \mathrm{I}_3 | \leq & \|\Gamma \cdot d\bu\|_{L^2_x}  \|\bG-\bF\|_{L^2_x}
		\lesssim  \|d\bu\|_{L^\infty_x} \| \bw\|^2_{H^{2}_x}(1 +\|(d\bu,dh)\|^2_{H^{s-1}_x}
		+\|(d\bu,dh)\|^3_{H^{s-1}_x} ).
	\end{align}
	\textit{The bound for $\mathrm{I}_1$}. Seeing \eqref{Wd18a}, and noting $\Gamma\approx d\bu \cdot d\bw$, there is a loss of derivative in $\partial^\alpha\Gamma$. So we seek to capture the cancellation of highest derivatives from integrating by parts. To do that, we divide $\mathrm{I}_1$ by two parts
	\begin{equation}\label{Wd28}
		\begin{split}
			\mathrm{I}_1
			=& \underbrace{ \int_{\mathbb{R}^3} \partial_\alpha ( (u^0)^{-1}\Gamma ) \cdot G^\alpha dx }_{\equiv \mathrm{I}_{11}} \underbrace{-\int_{\mathbb{R}^3} \partial_\alpha  ( (u^0)^{-1}\Gamma ) \cdot F^\alpha dx }_{\equiv \mathrm{I}_{12}}.
		\end{split}
	\end{equation}
	Integrating by parts on $\mathrm{I}_{11}$ and using \eqref{We30}, so we get
	\begin{equation}\label{Wd60}
		\begin{split}
			\mathrm{I}_{11}
			=&\int_{\mathbb{R}^3} \partial_\alpha (  (u^0)^{-1}\Gamma  \cdot G^\alpha ) dx-  \int_{\mathbb{R}^3}  (u^0)^{-1}\Gamma  \cdot \partial_\alpha G^\alpha dx
			\\
			=&\frac{d}{dt}\int_{\mathbb{R}^3}  (u^0)^{-1}\Gamma  \cdot G^0 dx
			 \underbrace{ -  \int_{\mathbb{R}^3}  (u^0)^{-1}\Gamma  \cdot \partial_\alpha G^\alpha dx }_{\equiv \mathrm{I}_{111}}.
		\end{split}
	\end{equation}
	Similarly, we calculate $\mathrm{I}_{12}$ as
	\begin{equation}\label{Wd61}
		\begin{split}
			\mathrm{I}_{12}=&	- \frac{d}{dt} \int_{\mathbb{R}^3}  (u^0)^{-1}\Gamma \cdot F^0 dx +  \underbrace{\int_{\mathbb{R}^3}   (u^0)^{-1}\Gamma \cdot \partial_{\alpha} F^{\alpha} dx}_{\equiv \mathrm{I}_{121}}.
		\end{split}
	\end{equation}
	Inserting \eqref{Wd28}, \eqref{Wd60} and \eqref{Wd61} to \eqref{Wd17}, it follows
	\begin{equation}\label{Wd62}
		\begin{split}
			\frac{d}{dt} \left( \|\bG -\bF\|^2_{L_x^2}+ \int_{\mathbb{R}^3}  (u^0)^{-1}\Gamma  \cdot (G^0-F^0) dx \right) =& \mathrm{I}_0+ \mathrm{I}_2+ \mathrm{I}_3+\mathrm{I}_{111}+\mathrm{I}_{121}.
		\end{split}
	\end{equation}
	Note \eqref{Wd220} and \eqref{Wd27}, so it only remains for us to bound $\mathrm{I}_{111}$ and $\mathrm{I}_{121}$.

	\textit{The bound for  $\mathrm{I}_{111}$.} By Lemma \ref{lpe}, H\"older's inequality, \eqref{YXg} and \eqref{MFd}, we infer
	\begin{equation}\label{Wd63}
		\begin{split}
			| \mathrm{I}_{111} | =& | \epsilon^{\alpha\beta\gamma\delta} \int_{\mathbb{R}^3}  (u^0)^{-1}\Gamma  \cdot (\partial_\alpha u_\beta \partial_\gamma W_\delta) dx |
			\\
			\lesssim & \| \Gamma\|_{L_x^2} \|d\bu \cdot d\bW\|_{L_x^2}
			\\
			\lesssim & \| d\bu \cdot d\bw\|_{L_x^2} ( \|d\bu \cdot ( d\bu,dh) \cdot d\bw \|_{L_x^2}+\|d\bu \cdot d^2\bw \|_{L_x^2}  +\|\bw \cdot d^2 h \|_{L_x^2} )
			\\
				\lesssim & \|(d\bu,dh)\|_{L^\infty_x}
			( \|(d\bu, dh)\|^2_{H^{s-1}_x} +  \| \bw\|^2_{H^{2}_x} ) \|(d\bu, dh)\|_{H^{s-1}_x}
			\\
			&   + ( \|(d\bu, dh)\|^2_{H^{s-1}_x} +  \| \bw\|^2_{H^{2}_x} ) \|(d\bu, dh)\|^3_{H^{s-1}_x}.
		\end{split}
	\end{equation}
	\textit{The bound for  $\mathrm{I}_{121}$.} Due to \eqref{We64}, using H\"older's inequality, we can bound $\Psi$ by
	\begin{equation}\label{Wd66}
		\begin{split}
			\|	\Psi \|_{L_x^2} \lesssim & \| (dh,d\bu) \cdot d^2\bw \|_{L_x^2}
			+ \| (dh,d\bu) \cdot (dh,d\bu) \cdot d\bw \|_{L_x^2}
			\\
			& + \| (dh,d\bu) \cdot \bw \cdot (d^2h,d^2\bu)  \|_{L_x^2}
			+ \| (dh,d\bu) \cdot \bw \cdot (dh,d\bu) \cdot (dh,d\bu)  \|_{L_x^2}
			\\
			\lesssim & \|(d\bu,dh)\|_{L^\infty_x}
			( \|(d\bu, dh)\|_{H^{s-1}_x} +  \| \bw\|_{H^{2}_x} )
			\\
			&   + ( \|(d\bu, dh)\|_{H^{s-1}_x} +  \| \bw\|_{H^{2}_x} ) ( \|(d\bu, dh)\|^2_{H^{s-1}_x} + \|(d\bu, dh)\|^3_{H^{s-1}_x} ),
		\end{split}
	\end{equation}
	Combining \eqref{Wd61} and \eqref{We64}, we have
\begin{footnotesize}
	\begin{equation}\label{Wd68}
		\begin{split}
			\mathrm{I}_{121}= & \underbrace{ \int_{\mathbb{R}^3}   (u^0)^{-1}\Gamma \cdot \Psi dx }_{\equiv \mathrm{I}_{12a}}
			- \underbrace{ 2\int_{\mathbb{R}^3}  (u^0)^{-1}\Gamma  \cdot ( u^\alpha \partial^\gamma w^\lambda \partial_\alpha   \partial_\gamma u_\lambda) dx }_{\equiv \mathrm{I}_{12b}}
			- \underbrace{ 2\int_{\mathbb{R}^3}  (u^0)^{-1}\Gamma \cdot ( c_s^{-2} \partial_\alpha  \partial_\lambda h \partial^\alpha w^\lambda) dx }_{\equiv \mathrm{I}_{12c}} .
		\end{split}
	\end{equation}
\end{footnotesize}	
	The above integrals will be discussed as follows. By \eqref{Wd66} and H\"older's inequality, we obtain
	\begin{equation}\label{Wd67}
		\begin{split}
			| \mathrm{I}_{12a} | \lesssim & \| \Gamma \|_{L^2_x} \| \Psi \|_{L^2_x}
			\\
			\lesssim &  \|(d\bu,dh)\|_{L^\infty_x}
			( \|(d\bu, dh)\|^2_{H^{s-1}_x} +  \| \bw\|^2_{H^{2}_x} ) \|(d\bu, dh)\|_{H^{s-1}_x}
			\\
			&   + ( \|(d\bu, dh)\|^2_{H^{s-1}_x} +  \| \bw\|^2_{H^{2}_x} ) (\|(d\bu, dh)\|^3_{H^{s-1}_x} + \|(d\bu, dh)\|^4_{H^{s-1}_x}).
		\end{split}
	\end{equation}
	For $\mathrm{I}_{12b}$, we may not bound it directly like $\mathrm{I}_{12a}$. We decompose
\begin{footnotesize}
	\begin{equation}\label{Wd680}
		\begin{split}
			\mathrm{I}_{12b}=&	-2\int_{\mathbb{R}^3} (u^0)^{-1} \Gamma\cdot ( u^0 \partial^\gamma w^\lambda \partial_t   \partial_\gamma u_\lambda) dx - 2\int_{\mathbb{R}^3}  (u^0)^{-1} \Gamma \cdot ( u^i \partial^\gamma w^\lambda \partial_i   \partial_\gamma u_\lambda) dx
			\\
			=& \underbrace{	-2\int_{\mathbb{R}^3}  (u^0)^{-1} \Gamma\cdot ( u^0 \partial^0 w^\lambda \partial^2_t   u_\lambda) dx }_{\equiv \mathrm{I}_{12b1}}
			\underbrace{ -2\int_{\mathbb{R}^3}  (u^0)^{-1} \Gamma \cdot ( u^0 \partial^i w^\lambda \partial_t   \partial_i u_\lambda) dx }_{\equiv \mathrm{I}_{12b2}}
			\underbrace{ -2\int_{\mathbb{R}^3}  (u^0)^{-1} \Gamma \cdot ( u^i \partial^\gamma w^\lambda \partial_i   \partial_\gamma u_\lambda) dx }_{\equiv \mathrm{I}_{12b3}} .
		\end{split}
	\end{equation}
\end{footnotesize}	
	Inserting \eqref{We72} to $\mathrm{I}_{12b1}$, we derive that
	\begin{equation}\label{Wd74}
		\begin{split}
			\mathrm{I}_{12b1}=& 	\underbrace{ -2\int_{\mathbb{R}^3}  (u^0)^{-1} \Gamma\cdot  u^0 \partial_t w^\lambda g^{0i} \partial^2_{ti}u_\lambda dx }_{\equiv \mathrm{I}_{12ba}}
			\underbrace{ -2\int_{\mathbb{R}^3}  (u^0)^{-1} \Gamma\cdot  u^0 \partial_t w^\lambda g^{ij}\partial^2_{ij}u_\lambda dx }_{\equiv \mathrm{I}_{12bb}}
			\\
			&\underbrace{ + 2\int_{\mathbb{R}^3}  (u^0)^{-1} \Gamma\cdot  u^0 \partial_t w^\lambda Q_\lambda dx }_{\equiv \mathrm{I}_{12bc}}
			\underbrace{- 2\int_{\mathbb{R}^3}  (u^0)^{-1} \Gamma \cdot  u^0 \partial_t w^\lambda c_s^2\Omega \mathrm{e}^{-h}W_\lambda dx }_{\equiv \mathrm{I}_{12bd}}.
		\end{split}
	\end{equation}
	For $\mathrm{I}_{12ba}$, we can compute out
	\begin{equation*}
		\begin{split}
			\mathrm{I}_{12ba}=	& 	 -2\int_{\mathbb{R}^3}   (u^0)^{-1} \Gamma \cdot \partial_{i} ( u^0 \partial_t w^\lambda g^{0i} \partial_{t}u_\lambda) dx
			+ 2\int_{\mathbb{R}^3}   (u^0)^{-1} \Gamma \cdot  \partial_{i} u^0 \partial_t w^\lambda g^{0i} \partial_{t}u_\lambda dx
			\\
			& + 2\int_{\mathbb{R}^3} (u^0)^{-1} \Gamma \cdot  u^0 \partial_{i} \partial_t w^\lambda g^{0i} \partial_{t}u_\lambda dx + 2\int_{\mathbb{R}^3} (u^0)^{-1} \Gamma \cdot  u^0  \partial_t w^\lambda \partial_{i} g^{0i} \partial_{t}u_\lambda dx  .
		\end{split}
	\end{equation*}
	Using the Plancherel formula and H\"older's inequality, we can bound $\mathrm{I}_{12ba}$ by
	\begin{equation}\label{Wd75}
		\begin{split}
			| \mathrm{I}_{12ba} | \lesssim	& 	 \| \Gamma \|_{\dot{H}_x^{\frac12}}  \|d \bw  d\bu \|_{\dot{H}_x^{\frac12}}
			+ \| \Gamma \|_{L^2_x}  \|d\bu \nabla d \bw  \|_{L^2_x}
			 + \| \Gamma \|_{L^2_x}  \|(d\bu,dh) \cdot d \bw  \cdot (d\bu,dh) \|_{L^2_x}
			\\
			\lesssim & \|(d\bu,dh)\|_{L^\infty_x}
			( \|(d\bu, dh)\|^2_{H^{s-1}_x} +  \| \bw\|^2_{H^{2}_x} ) \|(d\bu, dh)\|_{H^{s-1}_x}
			\\
			&   + ( \|(d\bu, dh)\|^2_{H^{s-1}_x} +  \| \bw\|^2_{H^{2}_x} ) (\|(d\bu, dh)\|^2_{H^{s-1}_x} + \|(d\bu, dh)\|^4_{H^{s-1}_x}).
		\end{split}
	\end{equation}
	Similarly, we can calculate $\mathrm{I}_{12bb}$ by
	\begin{equation*}
		\begin{split}
			\mathrm{I}_{12bb}=& -2\int_{\mathbb{R}^3}  (u^0)^{-1} \Gamma\cdot \partial_{i} ( u^0 \partial_t w^\lambda g^{ij}\partial_{j}u_\lambda) dx
			+2\int_{\mathbb{R}^3}  (u^0)^{-1} \Gamma\cdot  \partial_{i} u^0 \partial_t w^\lambda g^{ij}\partial_{j}u_\lambda dx
			\\
			& +2\int_{\mathbb{R}^3}  (u^0)^{-1} \Gamma\cdot  u^0 \partial_{i} \partial_t w^\lambda g^{ij}\partial_{j}u_\lambda dx
			+2\int_{\mathbb{R}^3}  (u^0)^{-1} \Gamma\cdot  u^0 \partial_{i} \partial_t w^\lambda \partial_{i} g^{ij}\partial_{j}u_\lambda dx .
		\end{split}
	\end{equation*}
	Using the Plancherel formula and H\"older's inequality again, we can bound $\mathrm{I}_{12bb}$ by
	\begin{equation}\label{Wd76}
		\begin{split}
			| \mathrm{I}_{12bb} | \lesssim	& 	 \| \Gamma \|_{\dot{H}_x^{\frac12}}  \|d \bw  d\bu \|_{\dot{H}_x^{\frac12}}
			+ \| \Gamma \|_{L^2_x}  ( \|d\bu d \bw  d\bu \|_{L^2_x} + \|d\bu \nabla d \bw  \|_{L^2_x} )
			\\
			\lesssim & \|(d\bu,dh)\|_{L^\infty_x}
			( \|(d\bu, dh)\|^2_{H^{s-1}_x} +  \| \bw\|^2_{H^{2}_x} ) \|(d\bu, dh)\|_{H^{s-1}_x}
			\\
			&   + ( \|(d\bu, dh)\|^2_{H^{s-1}_x} +  \| \bw\|^2_{H^{2}_x} ) (\|(d\bu, dh)\|^2_{H^{s-1}_x} + \|(d\bu, dh)\|^4_{H^{s-1}_x}).
		\end{split}
	\end{equation}
	By H\"older's inequality, we can estimate $\mathrm{I}_{12bc}$ and $\mathrm{I}_{12bd}$ by
	\begin{equation}\label{Wd77}
		\begin{split}
			| \mathrm{I}_{12bc} | \lesssim &
			\| \Gamma \|_{L^2_x}  \|(d\bu,dh)\cdot (d\bu,dh)\cdot d \bw  \|_{L^2_x}
			\\
			\lesssim & ( \|(d\bu, dh)\|^2_{H^{s-1}_x} +  \| \bw\|^2_{H^{2}_x} ) ( \|(d\bu, dh)\|^3_{H^{s-1}_x}+\|(d\bu, dh)\|^4_{H^{s-1}_x} ).
		\end{split}
	\end{equation}
	and
	\begin{equation}\label{Wd78}
		\begin{split}
			| \mathrm{I}_{12bd} | \lesssim  &
			\| \Gamma \|_{L^2_x}  ( \|d\bw d \bw  \|_{L^2_x} + \|d \bu dh d \bw  \|_{L^2_x} )
			\\
			\lesssim & \|d\bu\|_{L^\infty_x}  \|d\bw\|_{L^2_x} \|d\bw\|_{H^{\frac12}_x}\|d\bw\|_{H^{1}_x}
		+( \|(d\bu, dh)\|^2_{H^{s-1}_x} +  \| \bw\|^2_{H^{2}_x} ) \|(d\bu, dh)\|^3_{H^{s-1}_x}
			\\
			\lesssim & \|d\bu\|_{L^\infty_x}  (\|d\bu\|_{H^{s-1}_x}+\|(d\bu,dh)\|^2_{H^{s-1}_x}) \|d\bw\|^2_{H^{1}_x}
			\\
			& \quad +( \|(d\bu, dh)\|^2_{H^{s-1}_x} +  \| \bw\|^2_{H^{2}_x} ) \|(d\bu, dh)\|^4_{H^{s-1}_x}.
		\end{split}
	\end{equation}
	To summarize \eqref{Wd74}, \eqref{Wd75}, \eqref{Wd76}, \eqref{Wd77}, and \eqref{Wd78}, we get
	\begin{equation}\label{Wd79}
		\begin{split}
			| \mathrm{I}_{12b1} |
			\lesssim & \|d\bu\|_{L^\infty_x} (\|(d\bu,dh)\|^2_{H^{s-1}_x}+ \|\bw\|^2_{H^{2}_x}) ( \|(d\bu,dh)\|_{H^{s-1}_x}+ \|(d\bu,dh)\|^2_{H^{s-1}_x} )
			\\
			&+( \|(d\bu, dh)\|^2_{H^{s-1}_x} +  \| \bw\|^2_{H^{2}_x} ) (\|(d\bu, dh)\|^2_{H^{s-1}_x} + \|(d\bu, dh)\|^4_{H^{s-1}_x}).
		\end{split}
	\end{equation}
	By chain's rule, we are able to obtain
	\begin{equation}\label{Wd83}
		\begin{split}
			\mathrm{I}_{12b2}=& 	- 2\int_{\mathbb{R}^3}  (u^0)^{-1} \Gamma\cdot \partial_i ( u^0 \partial^i w^\lambda \partial_t  u_\lambda) dx
			\\
			&+ 2\int_{\mathbb{R}^3}  (u^0)^{-1} \Gamma\cdot  ( \partial_i u^0 \partial^i w^\lambda \partial_t u_\lambda +  u^0 \partial_i \partial^i w^\lambda \partial_t  u_\lambda) dx,
		\end{split}
	\end{equation}
	and
	\begin{equation}\label{Wd84}
		\begin{split}
			\mathrm{I}_{12b3}=& 	 - 2\int_{\mathbb{R}^3}  (u^0)^{-1} \Gamma\cdot \partial_i ( u^i \partial^\gamma w^\lambda   \partial_\gamma u_\lambda) dx
			\\
			&+ 2\int_{\mathbb{R}^3}  (u^0)^{-1} \Gamma \cdot  ( \partial_i  u^i \partial^\gamma w^\lambda   \partial_\gamma u_\lambda +  u^i \partial_i \partial^\gamma w^\lambda   \partial_\gamma u_\lambda ) dx .
		\end{split}
	\end{equation}
Seeing \eqref{Wd83}-\eqref{Wd84}, and using Plancherel formula and H\"older's inequality again, we infer
	\begin{equation}\label{Wd85}
		\begin{split}
			|\mathrm{I}_{12b2}|+	|\mathrm{H}_{12b3}|  \lesssim & \| \Gamma \|_{\dot{H}_x^{\frac12}}  \|d \bw  d\bu \|_{\dot{H}_x^{\frac12}}
			+ \| \Gamma \|_{L^2_x}  ( \|d\bu d \bw  d\bu \|_{L^2_x} + \|d\bu \nabla d \bw  \|_{L^2_x} )
			\\
			\lesssim & \|(d\bu,dh)\|_{L^\infty_x}
			( \|(d\bu, dh)\|^2_{H^{s-1}_x} +  \| \bw\|^2_{H^{2}_x} ) \|(d\bu, dh)\|_{H^{s-1}_x}
			\\
			&   + ( \|(d\bu, dh)\|^2_{H^{s-1}_x} +  \| \bw\|^2_{H^{2}_x} ) (\|(d\bu, dh)\|^2_{H^{s-1}_x} + \|(d\bu, dh)\|^4_{H^{s-1}_x}) .
		\end{split}
	\end{equation}
	Summarizing our outcome \eqref{Wd680}, \eqref{Wd79}, and \eqref{Wd85}, so it yields
	\begin{equation}\label{Wd80}
		\begin{split}
			| \mathrm{I}_{12b} |
			\lesssim & \|d\bu\|_{L^\infty_x} (\|(d\bu,dh)\|^2_{H^{s-1}_x}+ \|\bw\|^2_{H^{2}_x}) ( \|(d\bu,dh)\|_{H^{s-1}_x}
		+ \|(d\bu,dh)\|^2_{H^{s-1}_x})
			\\
			&+( \|(d\bu, dh)\|^2_{H^{s-1}_x} +  \| \bw\|^2_{H^{2}_x} ) (\|(d\bu, dh)\|^2_{H^{s-1}_x} + \|(d\bu, dh)\|^4_{H^{s-1}_x}).
		\end{split}
	\end{equation}
	We still need to bound $\mathrm{I}_{12c}$. A direct calculation tells us
	\begin{equation}\label{Wd82}
		\begin{split}
			\mathrm{I}_{12c}= & \underbrace{ -2\int_{\mathbb{R}^3}  (u^0)^{-1} \Gamma\cdot  c_s^{-2} \partial^2_t   h \partial_t w^0 dx }_{\equiv \mathrm{I}_{12c1}}
			\underbrace{- 2\int_{\mathbb{R}^3}  (u^0)^{-1} \Gamma\cdot c_s^{-2} \partial_t \partial_i  h \partial_t w^i dx }_{\equiv \mathrm{I}_{12c2}}
			 \\
			 & \ \underbrace{ -2\int_{\mathbb{R}^3}  (u^0)^{-1} \Gamma\cdot  c_s^{-2} \partial_i  \partial_\lambda h \partial^i w^\lambda dx  }_{\equiv \mathrm{I}_{12c3}}.
		\end{split}
	\end{equation}
	Substituting \eqref{We73} to $\mathrm{I}_{12c1}$, we have
	\begin{equation*}
		\begin{split}
			\mathrm{I}_{12c1}= & -2\int_{\mathbb{R}^3}  (u^0)^{-1} \Gamma\cdot  c_s^{-2} g^{0i} \partial^2_{ti}h \partial_t w^0 dx
			- 2\int_{\mathbb{R}^3}  (u^0)^{-1} \Gamma\cdot c_s^{-2} g^{ij}\partial^2_{ij}h \partial_t w^0 dx
			\\
			&\ +2\int_{\mathbb{R}^3}  (u^0)^{-1} \Gamma\cdot  c_s^{-2} g^{ij} D \partial_t w^0  dx .
		\end{split}
	\end{equation*}
	Using a similar way for controlling $\mathrm{I}_{12ba}$, so we obtain
	\begin{equation}\label{Wd86}
		\begin{split}
			| \mathrm{I}_{12c1} | \lesssim &  \| \Gamma \|_{\dot{H}_x^{\frac12}}  \|d \bw  dh \|_{\dot{H}_x^{\frac12}}
			+ \| \Gamma \|_{L^2_x}  \|dh\nabla d \bw  \|_{L^2_x}
			 + \| \Gamma \|_{L^2_x}  \|(d\bu,dh)\cdot d \bw \cdot  (d\bu,dh) \|_{L^2_x}
			\\
			\lesssim & \|(d\bu,dh)\|_{L^\infty_x}
			( \|(d\bu, dh)\|^2_{H^{s-1}_x} +  \| \bw\|^2_{H^{2}_x} ) \|(d\bu, dh)\|_{H^{s-1}_x}
			\\
			&   + ( \|(d\bu, dh)\|^2_{H^{s-1}_x} +  \| \bw\|^2_{H^{2}_x} ) (\|(d\bu, dh)\|^2_{H^{s-1}_x} + \|(d\bu, dh)\|^4_{H^{s-1}_x}).
		\end{split}
	\end{equation}
	Using a similar way for estimating $\mathrm{I}_{12bb}$, it yields
	\begin{equation}\label{Wd87}
		\begin{split}
			& | \mathrm{I}_{12c2} |+ | \mathrm{I}_{12c3} |
			\\
			\lesssim &  \| \Gamma \|_{\dot{H}_x^{\frac12}}  \|d \bw  dh \|_{\dot{H}_x^{\frac12}}
			+ \| \Gamma \|_{L^2_x}  ( \|dh\nabla d \bw  \|_{L^2_x} +\|(d\bu,dh)\cdot d \bw \cdot  (d\bu,dh) \|_{L^2_x} )
			\\
			\lesssim & \|(d\bu,dh)\|_{L^\infty_x}
			( \|(d\bu, dh)\|^2_{H^{s-1}_x} +  \| \bw\|^2_{H^{2}_x} ) \|(d\bu, dh)\|_{H^{s-1}_x}
			\\
			&   + ( \|(d\bu, dh)\|^2_{H^{s-1}_x} +  \| \bw\|^2_{H^{2}_x} ) (\|(d\bu, dh)\|^2_{H^{s-1}_x} + \|(d\bu, dh)\|^4_{H^{s-1}_x}).
		\end{split}
	\end{equation}
	By summarizing \eqref{Wd82}, \eqref{Wd86}, and \eqref{Wd87}, we obtain
	\begin{equation}\label{Wd89}
		\begin{split}
			| \mathrm{I}_{12c} |
			\lesssim & \|(d\bu,dh)\|_{L^\infty_x}
			( \|(d\bu, dh)\|^2_{H^{s-1}_x} +  \| \bw\|^2_{H^{2}_x} ) \|(d\bu, dh)\|_{H^{s-1}_x}
			\\
			&   + ( \|(d\bu, dh)\|^2_{H^{s-1}_x} +  \| \bw\|^2_{H^{2}_x} ) (\|(d\bu, dh)\|^2_{H^{s-1}_x} + \|(d\bu, dh)\|^4_{H^{s-1}_x}).
		\end{split}
	\end{equation}
Seeing \eqref{Wd68}, and combining \eqref{Wd67}, \eqref{Wd80} and \eqref{Wd89}, we infer
	\begin{equation}\label{Wd90}
		\begin{split}
			| \mathrm{I}_{121} |
			\lesssim & \|(d\bu,dh)\|_{L^\infty_x}
			( \|(d\bu, dh)\|^2_{H^{s-1}_x} +  \| \bw\|^2_{H^{2}_x} ) \|(d\bu, dh)\|_{H^{s-1}_x}
			\\
			&   + ( \|(d\bu, dh)\|^2_{H^{s-1}_x} +  \| \bw\|^2_{H^{2}_x} ) (\|(d\bu, dh)\|^2_{H^{s-1}_x} + \|(d\bu, dh)\|^4_{H^{s-1}_x}).
		\end{split}
	\end{equation}
	To summarize \eqref{Wd62}, \eqref{Wd220}, \eqref{Wd63}, \eqref{Wd90}, and \eqref{Wd27}, we therefore have
	\begin{equation}\label{Wd91}
		\begin{split}
			& \frac{d}{dt} \left( \|\bG -\bF\|^2_{L_x^2}+ \int_{\mathbb{R}^3}  (u^0)^{-1} \Gamma \cdot (G^0-F^0) dx \right)
			\\
			\lesssim & \|(d\bu,dh)\|_{L^\infty_x}
			( \|(d\bu, dh)\|^2_{H^{s-1}_x} +  \| \bw\|^2_{H^{2}_x} ) ( 1+ \|(d\bu, dh)\|^2_{H^{s-1}_x})
			\\
			& +  \|(d\bu,dh)\|_{L^\infty_x}
			( \|(d\bu, dh)\|^2_{H^{s-1}_x} +  \| \bw\|^2_{H^{2}_x} ) ( \|(d\bu, dh)\|^3_{H^{s-1}_x}
		+\|(d\bu, dh)\|^4_{H^{s-1}_x} )
			\\
			&   + ( \|(d\bu, dh)\|^2_{H^{s-1}_x} +  \| \bw\|^2_{H^{2}_x} ) (\|(d\bu, dh)\|^2_{H^{s-1}_x} + \|(d\bu, dh)\|^4_{H^{s_0-1}_x}).
		\end{split}
	\end{equation}
	Integrating \eqref{Wd91} from $0$ to $t$($t>0$), we obtain
	\begin{equation}\label{Wd92}
		\begin{split}
			& \|(\bG -\bF)(t,\cdot)\|^2_{L_x^2}+ {\int_{\mathbb{R}^3}}  (u^0)^{-1} \Gamma \cdot (G^0-F^0) (t,\cdot)dx
			\\
			\lesssim & \|(\bG -\bF)(0,\cdot)\|^2_{L_x^2}+  {\int_{\mathbb{R}^3}} (u^0)^{-1} \Gamma \cdot(G^0-F^0) (0,\cdot)dx
		 \\
		 & + {\int^t_0}  \|(d\bu,dh)\|_{L^\infty_x}
			( \|(d\bu, dh)\|^2_{H^{s-1}_x} +  \| \bw\|^2_{H^{2}_x} ) (1+\|(d\bu, dh)\|^2_{H^{s-1}_x}) d\tau
		\\
			& + {\int^t_0} \|(d\bu,dh)\|_{L^\infty_x}
			( \|(d\bu, dh)\|^2_{H^{s-1}_x} +  \| \bw\|^2_{H^{2}_x} ) (  \|(d\bu, dh)\|^3_{H^{s-1}_x} + \|(d\bu, dh)\|^4_{H^{s-1}_x}) d\tau
			\\
			&   + {\int^t_0} ( \|(d\bu, dh)\|^2_{H^{s-1}_x} +  \| \bw\|^2_{H^{2}_x} ) (\|(d\bu, dh)\|^2_{H^{s-1}_x} + \|(d\bu, dh)\|^4_{H^{s-1}_x}) d\tau.
		\end{split}
	\end{equation}
	For the left hand of \eqref{Wd92}, we can get a lower bound
	\begin{equation}\label{Wd93}
		\begin{split}
			& \|(\bG -\bF)(t,\cdot)\|^2_{L_x^2}+ {\int_{\mathbb{R}^3}}  (u^0)^{-1} \Gamma \cdot (G^0-F^0)(t,\cdot) dx
			\\
			\geq & \|\bG(t,\cdot)\|^2_{L^2_x} -\|\bF(t,\cdot)\|^2_{L^2_x}- \| \Gamma(t,\cdot) \|_{L^2_x} ( \|G^0(t,\cdot) \|_{L^2_x}+\|F^0(t,\cdot) \|_{L^2_x} )
			\\
			\geq & \|\bG(t,\cdot)\|^2_{L^2_x} - C( \|(d\bu, dh)(t,\cdot)\|_{H^{s-1}_x}+ \|(d\bu, dh)(t,\cdot)\|^2_{H^{s-1}_x}) \|\bw (t,\cdot) \|_{H^{2}_x} \|\bw (t,\cdot) \|_{H^{\frac32}_x}
			\\
			& \ -C( \|(d\bu, dh)(t,\cdot)\|^2_{H^{s-1}_x}+ \|(d\bu, dh)(t,\cdot)\|^4_{H^{s-1}_x}) \|\bw (t,\cdot) \|^2_{H^{\frac32}_x},
		\end{split}
	\end{equation}
	and an upper bound
	\begin{equation}\label{Wd94}
		\begin{split}
			& \|(\bG -\bF)(0,\cdot)\|^2_{L_x^2}+ {\int_{\mathbb{R}^3}} (u^0)^{-1} \Gamma \cdot (G^0-F^0) (0,\cdot)dx
			\\
			\lesssim & \|(\mathring{\bu}_0, u_0^0-1)\|^2_{H_x^s} + \|h_0\|^2_{H_x^s} + \|\bw_0\|^2_{{H}^{2}_x}+\|(\mathring{\bu}_0, u_0^0-1)\|^6_{H_x^s} + \|h_0\|^6_{H_x^s} + \|\bw_0\|^6_{{H}^{2}_x} .
		\end{split}
	\end{equation}
	By the interpolation formula, it follows
	\begin{equation}\label{Wd95}
		\|\bw (t,\cdot) \|_{H^{\frac32}_x} \leq \|\bw (t,\cdot) \|^{\frac12}_{H^{1}_x} \|\bw (t,\cdot) \|^{\frac12}_{H^{2}_x} .
	\end{equation}
	Substituting \eqref{Wd95} to \eqref{Wd93}, and using Young's inequality, we can show that\footnote{In \eqref{Wd96}, we give a coefficient $\frac{1}{10(1+C)}$ in $\frac{1}{10(1+C)}\|\bw (t,\cdot) \|^2_{H^{2}_x}$ by Young's inequality, where the constant $C$ is totally the same with in the right hand of \eqref{Wf04}.}
	\begin{equation}\label{Wd96}
		\begin{split}
			& \|\bG -\bF\|^2_{L_x^2}+ {\int_{\mathbb{R}^3}} (u^0)^{-1} \Gamma \cdot (G^0-F^0)dx
			\\
			\geq & \|\bG\|^2_{L^2_x} -\frac{1}{10(1+C)}\|\bw\|^2_{H^{2}_x}- C( \|(d\bu, dh)\|^4_{H^{s-1}_x}+ \|(d\bu, dh)\|^8_{H^{s-1}_x}) \|\bw\|^2_{H^{s-1}_x}
			\\
			\geq & \|\bG\|^2_{L^2_x} -\frac{1}{10(1+C)}\|\bw \|^2_{H^{2}_x}- C( \|(d\bu, dh)\|^4_{H^{s-1}_x}+ \|(d\bu, dh)\|^{10}_{H^{s-1}_x}) .
		\end{split}
	\end{equation}
	Above, we also use $\|\bw (t,\cdot) \|_{H^{1}_x} \approx \|d\bu (t,\cdot) \|_{H^{1}_x} \leq  \|d\bu (t,\cdot) \|_{H^{s-1}_x}$. Summing uo the outcome \eqref{Wd92}, \eqref{Wd94} and \eqref{Wd96}, we have proved
	\begin{equation}\label{Wd97}
		\begin{split}
			& \|\bG(t,\cdot)\|^2_{L^2_x}- \frac{1}{10(1+C)}\|\bw (t,\cdot) \|^2_{H^{2}_x}
			\\
			\lesssim
			&  \|(\mathring{\bu}_0, u_0^0-1)\|^2_{H_x^s} + \|h_0\|^2_{H_x^s} + \|\bw_0\|^2_{{H}^{2}_x}+\|(\mathring{\bu}_0, u_0^0-1)\|^6_{H_x^s} + \|h_0\|^6_{H_x^s}
			\\
			& + \|\bw_0\|^6_{{H}^{2}_x}
			+  \|(d\bu, dh)(t,\cdot)\|^4_{H^{s-1}_x}+ \|(d\bu, dh)(t,\cdot)\|^{10}_{H^{s-1}_x}
			\\
			& + {\int^t_0}  \|(d\bu,dh)\|_{L^\infty_x}
			( \|(d\bu, dh)\|^2_{H^{s-1}_x} +  \| \bw\|^2_{H^{2}_x} ) (1+ \|(d\bu, dh)\|^2_{H^{s-1}_x} ) d\tau
		\\
			& + {\int^t_0} \|(d\bu,dh)\|_{L^\infty_x}
			( \|(d\bu, dh)\|^2_{H^{s-1}_x} +  \| \bw\|^2_{H^{2}_x} ) (  \|(d\bu, dh)\|^3_{H^{s-1}_x}+\|(d\bu, dh)\|^4_{H^{s-1}_x}) d\tau
			\\
			&   + {\int^t_0} ( \|(d\bu, dh)\|^2_{H^{s-1}_x} +  \| \bw\|^2_{H^{2}_x} ) (\|(d\bu, dh)\|^2_{H^{s-1}_x} + \|(d\bu, dh)\|^4_{H^{s-1}_x}) d\tau.
		\end{split}
	\end{equation}
	Due to \eqref{Wf04}, \eqref{Wd97}, and $\|u_0^0-1\|_{H_x^s} \lesssim \|\mathring{\bu}_0\|_{H_x^s}$, we find that
	\begin{equation}\label{Wd98}
		\begin{split}
			& \|\bw(t,\cdot)\|^2_{\dot{H}^{2}_x}- \frac{C}{10(1+C)}\|\bw (t,\cdot) \|^2_{H^{2}_x}-\frac{1}{10}\|\bw (t,\cdot) \|^2_{H^{2}_x}
			\\
			\lesssim
			&    \|\mathring{\bu}_0\|^6_{H_x^s} + \|h_0\|^6_{H_x^s} + \|\bw_0\|^6_{{H}^{2}_x}
				+  \|(d\bu, dh)\|^3_{H^{s-1}_x} +  \|(d\bu, dh)\|^4_{H^{s-1}_x} + \|(d\bu, dh)\|^{10}_{H^{s-1}_x}
			\\
			& + \|h_0, \mathring{\bu}_0\|^2_{H_x^s} + {\int^t_0} \|(d\bu,dh)\|_{L^\infty_x}
			( \|(d\bu, dh)\|^2_{H^{s_0-1}_x} +  \| \bw\|^2_{H^{s_0}_x} ) (1+ \|(d\bu, dh)\|^2_{H^{s_0-1}_x} ) d\tau
			\\
			& + {\int^t_0}  \|(d\bu,dh)\|_{L^\infty_x}
			( \|(d\bu, dh)\|^2_{H^{s_0-1}_x} +  \| \bw\|^2_{H^{s_0}_x} ) ( \|(d\bu, dh)\|^3_{H^{s-1}_x} + \|(d\bu, dh)\|^4_{H^{s-1}_x} ) d\tau
			\\
			&   + \|\bw_0\|^2_{{H}^{2}_x}  + {\int^t_0} ( \|(d\bu, dh)\|^2_{H^{s-1}_x} +  \| \bw\|^2_{H^{2}_x} ) (\|(d\bu, dh)\|^2_{H^{s-1}_x} + \|(d\bu, dh)\|^4_{H^{s-1}_x}) d\tau,
		\end{split}
	\end{equation}
	By \eqref{Wd3}, \eqref{Wd98}, \eqref{We100}, and \eqref{Wc0}, it follows
	\begin{equation}\label{Wd99a}
		\begin{split}
			& \|(h, \mathring{\bu})\|^2_{{H}^{s}_x}(t)+ \|u^0-1\|^2_{{H}^{s}_x}(t) + \|\bw\|^2_{{H}^{2}_x}(t)- \frac{C}{10(1+C)}\|\bw \|^2_{H^{2}_x}(t)-\frac{1}{10}\|\bw \|^2_{H^{2}_x}(t)
			\\
			\lesssim
			& \|\mathring{\bu}_0\|^2_{H_x^s} + \|h_0\|^2_{H_x^s} + \|\bw_0\|^2_{{H}^{2}_x}+\|\mathring{\bu}_0\|^6_{H_x^s} + \|h_0\|^6_{H_x^s} + \|\bw_0\|^6_{{H}^{2}_x}
			\\
			& 	+  \textstyle{\sup}_{[0,t]} \left\{ \|(d\bu, dh)\|^3_{H^{s-1}_x}
			+  \|(d\bu, dh)\|^4_{H^{s-1}_x}+ \|(d\bu, dh)\|^{10}_{H^{s-1}_x} \right\}
			\\
			& + {\int^t_0}  \|(d\bu,dh)\|_{L^\infty_x}
			( \|(d\bu, dh)\|^2_{H^{s-1}_x} +  \| \bw\|^2_{H^{2}_x} ) (1+ \|(d\bu, dh)\|^2_{H^{s-1}_x} )d\tau
			\\
			& + {\int^t_0}  \|(d\bu,dh)\|_{L^\infty_x}
			( \|(d\bu, dh)\|^2_{H^{s-1}_x} +  \| \bw\|^2_{H^{2}_x} ) ( \|(d\bu, dh)\|^3_{H^{s-1}_x} + \|(d\bu, dh)\|^4_{H^{s-1}_x}) d\tau
			\\
			&   + {\int^t_0} ( \|(d\bu, dh)\|^2_{H^{s-1}_x} +  \| \bw\|^2_{H^{2}_x} ) (\|(d\bu, dh)\|^2_{H^{s-1}_x} + \|(d\bu, dh)\|^4_{H^{s-1}_x}) d\tau.
		\end{split}
	\end{equation}
	Since the left hand side of \eqref{Wd99a} $\geq \frac12\|(h, \mathring{\bu})\|^2_{{H}^{s}_x}(t)+ \frac12\|u^0-1\|^2_{{H}^{s}_x}(t) + \frac12\|\bw\|^2_{{H}^{2}_x}(t)$, so we have
	\begin{equation}\label{Wd103}
		\begin{split}
			 \tilde{E}_s(t)
			\lesssim
			& \|\mathring{\bu}_0\|^2_{H_x^s} + \|h_0\|^2_{H_x^s} + \|\bw_0\|^2_{{H}^{2}_x}+\|\mathring{\bu}_0\|^6_{H_x^s} + \|h_0\|^6_{H_x^s} + \|\bw_0\|^6_{{H}^{2}_x}
			\\
			& 	+ \textstyle{\sup}_{[0,t]} \left\{ \|(d\bu, dh)\|^3_{H^{s-1}_x}
			+  \|(d\bu, dh)\|^4_{H^{s-1}_x}+ \|(d\bu, dh)\|^{10}_{H^{s-1}_x} \right\}
			\\
			& + {\int^t_0}  \|(d\bu,dh)\|_{L^\infty_x}
			( \|(d\bu, dh)\|^2_{H^{s-1}_x} +  \| \bw\|^2_{H^{2}_x} ) (1+\|(d\bu, dh)\|^2_{H^{s-1}_x} )d\tau
			\\
			& + {\int^t_0} \|(d\bu,dh)\|_{L^\infty_x}
			( \|(d\bu, dh)\|^2_{H^{s-1}_x} +  \| \bw\|^2_{H^{2}_x} ) ( \|(d\bu, dh)\|^3_{H^{s-1}_x} + \|(d\bu, dh)\|^4_{H^{s-1}_x}) d\tau
			\\
			&   + {\int^t_0} ( \|(d\bu, dh)\|^2_{H^{s-1}_x} +  \| \bw\|^2_{H^{2}_x} ) (\|(d\bu, dh)\|^2_{H^{s-1}_x} + \|(d\bu, dh)\|^4_{H^{s-1}_x}) d\tau.
		\end{split}
	\end{equation}
	By using \eqref{VHe}, \eqref{VHea}, \eqref{VHeb}, and \eqref{Wd103}, we have
	\begin{equation*}
		\begin{split}
			\tilde{E}_s(t)
			\lesssim
			& \tilde{E}_s(0)+\tilde{E}_s(0)^3
			+ ( \tilde{E}_s(0)^{\frac32}+  \tilde{E}_s(0)^2+\tilde{E}_s(0)^5 ) \exp (5\int^t_0 \|d\bu, dh\|_{L^\infty_x}d\tau)
			\\
			& + {\int^t_0} \|(d\bu,dh)\|_{L^\infty_x}
			\tilde{E}_s(\tau) d\tau
			 + {\int^t_0} \|(d\bu,dh)\|_{L^\infty_x}  \tilde{E}_s(0) \exp (\int^\tau_0 \|d\bu, dh\|_{L^\infty_x}dz)
			\tilde{E}_s(\tau) d\tau
			\\
			&   + {\int^t_0} \tilde{E}_s(\tau) \left\{ \tilde{E}_s(0) \exp (\int^\tau_0 \|d\bu, dh\|_{L^\infty_x}dz)+ \tilde{E}_s(0)^{2} \exp (2 \int^\tau_0 \|d\bu, dh\|_{L^\infty_x}dz) \right\} d\tau
			\\
			& + {\int^t_0} \|(d\bu,dh)\|_{L^\infty_x}  (\tilde{E}_s(0)^{\frac32} + \tilde{E}_s(0)^2)\exp (2\int^\tau_0 \|d\bu, dh\|_{L^\infty_x}dz)
			\tilde{E}_s(\tau) d\tau.
		\end{split}
	\end{equation*}
	Observing $\tilde{E}_s(0)^{\frac32}, \tilde{E}_s(0)^{2}, \tilde{E}_s(0)^{3}  \leq \max\{\tilde{E}_s(0), \tilde{E}_s(0)^5\} $, so we have
		\begin{equation*}
		\begin{split}
			\tilde{E}_s(t)
			\lesssim
			& \tilde{E}_s(0)+\tilde{E}_s(0)^5
			+ ( \tilde{E}_s(0)+\tilde{E}_s(0)^5 ) \exp (5 \int^t_0 \|d\bu, dh\|_{L^\infty_x}d\tau)
			+ {\int^t_0} \|(d\bu,dh)\|_{L^\infty_x}
			\tilde{E}_s(\tau) d\tau
			\\
			& + {\int^t_0} \|(d\bu,dh)\|_{L^\infty_x}  ( \tilde{E}_s(0) + \tilde{E}_s(0)^5) \exp (2\int^t_0 \|d\bu, dh\|_{L^\infty_x}dz)
			E_s(\tau) d\tau
			\\
			&   + {\int^t_0} \tilde{E}_s(\tau) \left\{ \tilde{E}_s(0) \exp (\int^t_0 \|d\bu, dh\|_{L^\infty_x}dz)+ \tilde{E}_s(0)^{2} \exp (2\int^t_0 \|d\bu, dh\|_{L^\infty_x}dz) \right\} d\tau.
		\end{split}
	\end{equation*}
	Therefore, \eqref{WH}-\eqref{WLa} can be immediately obtained after an application of Gronwall's inequality.
\end{proof}
We next introduce a higher-order energy estimate as follows.
\begin{corollary}\label{hes}
	Let $(h,\bu)$ be a solution of \eqref{REE}. Let $2<\sstar<\frac52$. Let $\bw$ be defined in \eqref{VVd}. Then the following energy estimates hold:
	\begin{equation}\label{hes0}
		\begin{split}
			\mathbb{E}(t) \lesssim  \mathbb{E}_0 \exp \left\{ 18 \int^t_0 \|(d h, d \bu)\|_{L^\infty_x} d\tau \cdot \exp ( 3\int^t_0  \|(d h, d \bu) \|_{L^\infty_x}  d\tau ) \right\},
		\end{split}
	\end{equation}
	where $\mathbb{E}(t)$ and $\mathbb{E}_0$ are defined by
	\begin{equation*}
		\begin{split}
			\mathbb{E}(t)=& \|(h, u^0-1,\mathring{\bu})\|^2_{H_x^{\sstar+1}} + \|\bw\|^2_{H_x^{3}},
			\\
			\mathbb{E}_0 = &   \|(h_0, \mathring{\bu}_0)\|^2_{H_x^{\sstar+1}}+ \|\bw_0\|^2_{H_x^{3}}+\|(\mathring{\bu}_0, h_0)\|^3_{H^3}+\| \bw_0 \|^2_{{H}_x^{3}} \| h_0 \|^2_{{H}_x^{3}}+\|(\mathring{\bu}_0, h_0)\|^6_{H^3} .
		\end{split}
	\end{equation*}
\end{corollary}
\begin{proof}
	By using \eqref{VHe} and \eqref{u00}, we have
	\begin{equation}\label{hes1}
		\| (h,u^0-1,\mathring{\bu})\|^2_{H_x^{\sstar+1}}(t)\leq C\| (h_0,\mathring{\bu}_0)\|^2_{H_x^{\sstar+1}} \exp\{ \int^t_0 \|(dh, d\bu)\|_{L^\infty_x} d\tau \} .
	\end{equation}
	By \eqref{We3}, we get
	\begin{equation}\label{hes2}
		\| \bw \|^2_{L_x^2} \leq \| \bw_0 \|^2_{L_x^2} + C\int^t_0 \| d \bu\|_{L^\infty_x}\|\bw\|^2_{L^2_x}d\tau.
	\end{equation}
For $\| \bw \|_{\dot{H}^{3}}$, we notice
\begin{equation}\label{hes3}
	\| \bw \|_{\dot{H}_x^{3}} = \| \Delta \bw \|_{\dot{H}_x^{1}}.
\end{equation}
Recall $\mathring{\bw}=(w^1,w^2,w^3)$. Let us first consider $	\|  \mathring{\bw} \|_{\dot{H}_x^{s_0}} $. Let $\Delta_{H}$ be defined in \eqref{dh}. Using \eqref{MFd}, \eqref{w0d}, and Lemma \ref{ps}, it follows
\begin{equation}\label{hes4}
	\begin{split}
		\| \Delta_H \mathring{\bw} \|_{\dot{H}_x^{1}}  \lesssim	 & \|\bu \cdot \bW\|_{\dot{H}_x^{2}} + \| \bw \cdot (d\bu,dh)\|_{\dot{H}_x^{1}}+  \| d\bu \cdot \nabla \mathring{\bw} \|_{\dot{H}_x^{1}}
		\\
		\lesssim	& \|\bW\|_{\dot{H}_x^{2}}
		+ \| \mathring{\bw} \cdot (d\bu,dh)\|_{\dot{H}_x^{2}}+  \| d\bu \cdot \nabla \mathring{\bw} \|_{\dot{H}_x^{1}}
		\\
		\lesssim	 & \|\bW\|_{\dot{H}_x^{2}}
		+ \| \mathring{\bw} \|_{{H}_x^{2}}\| (d\bu,dh) \|_{{H}_x^{2}}  .
	\end{split}
\end{equation}
By \eqref{dh}, we can see
\begin{equation}\label{hes5}
	\| \Delta_H \mathring{\bw} \|_{\dot{H}_x^{1}}  \lesssim	\| \Delta \mathring{\bw} \|_{\dot{H}_x^{1}}  \lesssim 	\| \Delta_H \mathring{\bw} \|_{\dot{H}_x^{1}} .
\end{equation}
Combining \eqref{hes4} and \eqref{hes5}, we can obtain
\begin{equation}\label{hes6}
	\begin{split}
		\| \Delta \mathring{\bw} \|_{\dot{H}_x^{1}}  \lesssim	 	 & \|\bW\|_{\dot{H}_x^{2}}
		+ \| \mathring{\bw} \|_{{H}_x^{2}}\| (d\bu,dh) \|_{{H}_x^{2}}  .
	\end{split}
\end{equation}
Seeing from \eqref{w0d} and \eqref{hes6}, we can obtain
\begin{equation}\label{hes7}
	\begin{split}
		\| \Delta {w^0} \|_{\dot{H}_x^{1}}  \lesssim	 & \| \Delta \mathring{\bw} \|_{\dot{H}_x^{1}}+ \|\Delta \bu \cdot \mathring{\bw} \|_{\dot{H}_x^{1}} + \| \nabla \mathring{\bw} \cdot \nabla \bu \|_{\dot{H}_x^{1}}
		\\
		\lesssim	 	 & \|\bW\|_{\dot{H}_x^{2}}
		+ \| \mathring{\bw} \|_{{H}_x^{2}}\| (d\bu,dh) \|_{{H}_x^{2}}  .
	\end{split}
\end{equation}
Adding \eqref{hes6} and \eqref{hes7}, it follows
\begin{equation}\label{hes8}
	\begin{split}
		\| \bw \|_{\dot{H}_x^{3}}  \lesssim	   \|\bW\|_{\dot{H}_x^{2}}
		+ \| \mathring{\bw} \|_{{H}_x^{2}}\| (d\bu,dh) \|_{{H}_x^{2}}  .
	\end{split}
\end{equation}
Next, let us bound the term $\|\bW\|_{\dot{H}_x^{2}}$. Operating $\Lambda_x^2$ on \eqref{CEQ1} and multiplying $(u^0)^{-1} \Lambda_x^2 \bW$ on \eqref{CEQ1}, we therefore get the energy estimate
\begin{equation}\label{hes9}
	\begin{split}
		\| \bW \|_{\dot{H}_x^{2}}
		\lesssim	 & \| \bW_0 \|_{\dot{H}_x^{3}}
		+ \int^t_0 \|d\bu\|_{L^\infty_x} \|d\bw\|_{\dot{H}_x^{2}} \|\bW\|_{\dot{H}_x^{2}}d\tau
		\\
		& + \int^t_0 \|d\bu\|_{L^\infty_x} \|\bW\|^2_{\dot{H}_x^{2}}d\tau+  \int^t_0 \|\bw\cdot d\bu \cdot dh\|_{L^2_x} \|\bW\|_{\dot{H}_x^{2}}d\tau.
	\end{split}
\end{equation}
Due to \eqref{hes2}, \eqref{hes8}, \eqref{hes9}, and \eqref{MFd}, it yields
\begin{equation}\label{hes10}
	\begin{split}
		\| \bw \|^2_{{H}_x^{3}}  \leq	   &C\| \bW_0 \|^2_{\dot{H}_x^{3}} + C\| \bw_0 \|^2_{L_x^{2}}
		+ C\int^t_0 \|d\bu\|_{L^\infty_x} \|d\bw\|_{\dot{H}_x^{2}} \|\bW\|_{\dot{H}_x^{2}}d\tau
	 \\
	 & + C\int^t_0 \|d\bu\|_{L^\infty_x} \|\bW\|^2_{\dot{H}_x^{2}}d\tau
	 + C\int^t_0 \|d\bu\|_{L^\infty_x} \|\bw\|^2_{L_x^{2}}d\tau
	 \\
	 & +  C\int^t_0 \|\bw\cdot d\bu \cdot dh\|_{L^2_x} \|\bW\|_{\dot{H}_x^{2}}d\tau
		+ C\| \mathring{\bw} \|^2_{{H}_x^{2}}\| (d\bu,dh) \|^2_{{H}_x^{2}}
		\\
		  \leq	   &C\| \bw_0 \|^2_{{H}_x^{3}} + C\| \bw_0 \|^2_{{H}_x^{3}} \| h_0 \|^2_{{H}_x^{3}} + C\| {\bw} \|^2_{{H}_x^{2}}\| (d\bu,dh) \|^2_{{H}_x^{2}}
		 \\
		 & + C\int^t_0 \|d\bu\|_{L^\infty_x} \|\bw\|^2_{{H}_x^{3}} (1+ \|dh,d\bu\|^2_{{H}_x^{2}}+\|dh,d\bu\|^3_{{H}_x^{2}}) d\tau .
	\end{split}
\end{equation}
Because of Young's inequality, we have
\begin{equation}\label{hes11}
	\begin{split}
		C\| {\bw} \|^2_{{H}_x^{2}}\| (d\bu,dh) \|^2_{{H}_x^{2}} \leq & \frac{1}{2}\| {\bw} \|^2_{{H}_x^{3}} + C\| {\bw} \|^2_{{H}_x^{1}}\| (d\bu,dh) \|^4_{{H}_x^{2}}
		\\
		\leq & \frac{1}{2} \| {\bw} \|^2_{{H}_x^{3}} + C\| (d\bu,dh) \|^6_{{H}_x^{2}}.
	\end{split}
\end{equation}
Inserting \eqref{hes11} to \eqref{hes10}, we obtain
\begin{equation}\label{hes12}
	\begin{split}
		\| \bw \|^2_{{H}_x^{3}}  \lesssim	  	   & \| \bw_0 \|^2_{{H}_x^{3}} + \| \bw_0 \|^2_{{H}_x^{3}} \| h_0 \|^2_{{H}_x^{3}}
		+ \| (d\bu,dh) \|^6_{{H}_x^{2}}
		\\
		& + \int^t_0 \|d\bu\|_{L^\infty_x} \|\bw\|^2_{{H}_x^{3}} (1+ \|dh,d\bu\|^2_{{H}_x^{2}}+\|dh,d\bu\|^3_{{H}_x^{2}}) d\tau
		\\
		 \lesssim	  	   & \| \bw_0 \|^2_{{H}_x^{3}} + \| \bw_0 \|^2_{{H}_x^{3}} \| h_0 \|^2_{{H}_x^{3}}
		+ \|(\mathring{\bu}_0, h_0)\|^6_{H^3} \exp ( 6\|dh,d\bu\|_{L^\infty_x} )
		\\
		& + \int^t_0 \|d\bu\|_{L^\infty_x} \|\bw\|^2_{{H}_x^{3}} (1+ \|h,\mathring{\bu}\|^2_{{H}_x^{2}}+\|h,\mathring{\bu}\|^3_{{H}_x^{3}}) d\tau  .
	\end{split}
\end{equation}
Above, we also use Lemma \ref{QH}, \eqref{muu}, and \eqref{hes1}.

Finally, for \eqref{hes12}, by Gronwall's inequality, we have
\begin{equation}\label{hes13}
	\begin{split}
		\| \bw \|^2_{{H}_x^{3}}  \lesssim	  	   & (\| \bw_0 \|^2_{{H}_x^{3}} + \|(\mathring{\bu}_0, h_0)\|^2_{H^3}+\|(\mathring{\bu}_0, h_0)\|^3_{H^3}+\| \bw_0 \|^2_{{H}_x^{3}} \| h_0 \|^2_{{H}_x^{3}}+\|(\mathring{\bu}_0, h_0)\|^6_{H^3} )
		\\
		& \cdot \exp\{ ( 6\int^t_0 \|dh,d\bu\|_{L^\infty_x} d\tau ) (1+
		\mathrm{e}^{ 2\int^t_0 \|dh,d\bu\|_{L^\infty_x} d\tau }+  \mathrm{e}^{ 3\int^t_0 \|dh,d\bu\|_{L^\infty_x} d\tau } \}
		\\
		\lesssim
		& (\| \bw_0 \|^2_{{H}_x^{3}} + \|(\mathring{\bu}_0, h_0)\|^2_{H^3}+\|(\mathring{\bu}_0, h_0)\|^3_{H^3}+\| \bw_0 \|^2_{{H}_x^{3}} \| h_0 \|^2_{{H}_x^{3}}+\|(\mathring{\bu}_0, h_0)\|^6_{H^3} )
		\\
		& \cdot \exp \{ 18 \int^t_0 \|(d h, d \bu)\|_{L^\infty_x} d\tau \cdot \exp (3\int^t_0  \|(d h, d \bu) \|_{L^\infty_x}  d\tau )\} .
	\end{split}
\end{equation}
Adding \eqref{hes1} and \eqref{hes13}, we have proved \eqref{hes0}.
\end{proof}
Based on the above energy theorems, we can obtain the uniqueness of solutions.
\subsection{Uniqueness of the solution}
\begin{corollary}\label{cor}
	Assume $2<s_0<s\leq\frac52$ and \eqref{HEw}. Suppose $(h_1,\bu_1)$ and $(h_2,\bu_2)$ to be two solutions of \eqref{REE} with the same initial data $(h_0,\mathring{\bu}_0) \in H^{s} \times H^s$. We assume the initial modified vorticity $\bw_0=-\epsilon^{\alpha \beta \gamma \delta}\mathrm{e}^{h_0}u_{0\beta}\partial_{\gamma}u_{0\delta} \in H^{s_0}$. Then there exists a constant $T>0$ such that
	$(h_1, \mathring{\bu}_1, h_2, \mathring{\bu}_2) \in C([0,T],H_x^s) \cap C^1([0,T],H_x^{s-1})$, $\bw \in C([0,T],H_x^{s_0}) \cap C^1([0,T],H_x^{s_0-1})$ and $(dh_1, d\bu_1, dh_2, d\bu_2) \in {L^2_{[0,T]} L^\infty_x}$. Furthermore, we have
	\begin{equation*}
		\bu_1=\bu_2, \quad h_1=h_2.
	\end{equation*}
\end{corollary}
\begin{proof}
For $0<t\leq T$, by Theorem \ref{VE}, and using Strichartz estimates $(dh_1, d\bu_1, dh_2, d\bu_2) \in {L^2_{[0,t]} L^\infty_x}$, we can bound
	\begin{equation*}
		\|\bu_1-\bu_2, h_1-h_2\|_{L^2_x}(t) \lesssim \|\bu_1-\bu_2, h_1-h_2\|_{L^2_x}(0)=0.
	\end{equation*}
	So the solution is unique.
\end{proof}

Similarly, we can obtain the following corollary.
\begin{corollary}\label{cor2}
	Assume $2<s\leq\frac52$ and \eqref{HEw} hold. Suppose $(h_1,\bu_1)$ and $(h_2,\bu_2)$ to be two solutions of \eqref{REE} with the same initial data $(h_0,\mathring{\bu}_0) \in H^{s} \times H^s $. We also assume the initial modified vorticity $\bw_0=-\epsilon^{\alpha \beta \gamma \delta}\mathrm{e}^{h_0}u_{0\beta}\partial_{\gamma}u_{0\delta} \in H^{2}$. Then there exists a constant $T^*>0$ such that
	$(\mathring{\bu}_1,h_1, \mathring{\bu}_2, h_2) \in C([0,T^*],H_x^s) \cap C^1([0,T^*],H_x^{s-1})$, $\bw \in C([0,T^*],H_x^{2}) \cap C^1([0,T^*],H_x^{1})$ and $(dh_1, d\bu_1, dh_2, d\bu_2) \in {L^2_{[0,T^*]} L^\infty_x}$. Furthermore, we have
	\begin{equation*}
		\bu_1=\bu_2, \quad h_1=h_2.
	\end{equation*}
\end{corollary}
\section{Proof of Theorem \ref{dingli}}\label{Sec4}
In this section, firstly, we reduce Theorem \ref{dingli} to the case of smooth initial
data, i.e. Proposition \ref{p3}. Secondly, taking advantage of scaling and the finite speed of propagation to further simplify the problem to the case of smooth, small, supported initial
data, i.e. Proposition \ref{p1}. Finally, we reduce the result to a bootstrap argument, i.e. Proposition \ref{p4}. 
\subsection{Reduction to smooth initial data}\label{sec4.1}
We state a proposition as follows.
\begin{proposition}\label{p3}
Let $2<s_0<s\leq\frac52$. For each $M_0>0$, there exists $T, M>0$ (depending on $C_0,c_0,s,s_0$ and $M_0$) such that, for each smooth initial data\footnote{Here $\bw_0=\mathrm{vort}(\mathrm{e}^{h_0}\bu_0)=-\epsilon^{\alpha \beta \gamma \delta}\mathrm{e}^{h_0} u_{0\beta} \partial_\gamma u_{0\delta} $, ${\bu}_0=(u^0_0,\mathring{\bu}_0)^{\mathrm{T}}$.} $(h_0, {\bu}_0, \bw_0)$, which satisfies
\begin{equation}\label{ig1}
	\begin{split}
	&\|h_0\|_{H^s}+\|\bu_0\|_{H^s} + \| \bw_0\|_{H^{s_0}}  \leq M_0,
	\end{split}
	\end{equation}
there exists a smooth solution $(h,\bu,\bw)$ to \eqref{WTe} on $[0,T] \times \mathbb{R}^3$ satisfying
\begin{equation}\label{e9}
	\begin{split}
		 & \|h, u^0-1,\mathring{\bu}\|_{L^\infty_{[0,T]}H_x^s}+\|\bw\|_{L^\infty_{[0,T]}H_x^{s_0}} \leq M,
		 \\
		 & u^0 \geq 1, \quad \|u^0-1,\mathring{\bu}, h \|_{L^\infty_{[0,T]\times \mathbb{R}^3}} \leq 1+C_0.
	\end{split}
\end{equation}
Furthermore, the solution satisfies estimates

$\mathrm{(1)}$ dispersive estimate for $\bu$, $h$ and $\bu_+$
	\begin{equation}\label{p303}
	\|d \bu, d h\|_{L^2_{[0,T]} C^\delta_x}+\|\nabla \bu_+, d h, d \bu\|_{L^2_{[0,T]} \dot{B}^{s_0-2}_{\infty,2}} \leq M,
	\end{equation}

$\mathrm{(2)}$ Let $f$ satisfy the equation \eqref{linear}. For each $1 \leq r \leq s+1$, the Cauchy problem \eqref{linear} is well-posed in $H_x^r \times H_x^{r-1}$. Furthermore, the following energy estimate
	\begin{equation}\label{p304}
	\|f\|_{L^\infty_{[0,T]} H^r_x} + \|\partial_t f\|_{L^\infty_{[0,T]} H^{r-1}_x}  \leq C_{M_0} ( \| f_0\|_{H^r}+ \| f_1\|_{H^{r-1}}),
	\end{equation}
and Strichartz estimates
		\begin{equation}\label{305}
		\|\left< \nabla \right>^k f\|_{L^2_{[0,T]} L^\infty_x} \leq  C_{M_0}(\| f_0\|_{H^r}+ \| f_1\|_{H^{r-1}}),\ \ \ k<r-1,
	\end{equation}
hold. The same estimates hold with $\left< \nabla \right>^k$ replaced by $\left< \nabla \right>^{k-1}d$. Above, $C_{M_0}$ is also a constant depending on $C_0, c_0, s, s_0, M_0$.
\end{proposition}
Next, we show that Theorem \ref{dingli} is a consequence of Proposition \ref{p3}.
\begin{proof}[Proof of Theorem \ref{dingli} by using Proposition \ref{p3}]
Let $(h_0,{\bu}_0,\bw_0)$ be stated in Theorem \ref{dingli}. So we have
\begin{equation*}
	\|\mathring{\bu}_0\|_{H^s} + \|h_0 \|_{H^s} + \|\bw_0\|_{H^{s_0}} \leq M_0.
\end{equation*}
Let $\{ (h_{0k}, \mathring{\bu}_{0k}) \}_{k\in \mathbb{Z}^+}$ be a sequence of smooth data converging to $(h_0, \mathring{\bu}_0)$ in $H^s \times H^s$. Referring \eqref{muu}, we set
\begin{equation*}
	u^0_{0k}=\sqrt{1+|\mathring{\bu}_{0k}|^2 }, \quad {\bu}_{0k}=(u^0_{0k}, \mathring{\bu}_{0k}).
\end{equation*}
Referring \eqref{VVd}, we also define $\bw_{0k}= \textrm{vort}(\mathrm{e}^{h_{0k}} \bu_{0k})$. Therefore, $\{ (h_{0k}, {\bu}_{0k}, \bw_{0k}) \}_{k\in \mathbb{Z}^+}$ is a smooth sequence which converges to $(h_0, {\bu}_0, \bw_0)$ in $H^s \times H^s \times H^{s_0}$. By Proposition \ref{p3}, for every $k \in \mathbb{Z}^+$, there exists a solution $(h_k, \bu_k, \bw_k)$ to System \eqref{WTe} with
\begin{equation*}
	(h_k, \bu_k, \bw_k)|_{t=0}=(h_{0k}, \bu_{0k}, \bw_{0k}).
\end{equation*}
Note that the solution of \eqref{WTe} also satisfies \eqref{REE} and \eqref{QHl}.

Set $\bU_k=(p_k, u_{1k}, u_{2k}, u_{3k})^{\mathrm{T}}, k \in \mathbb{Z}^+$, where $p_k=p(h_k)$. For $j \in \mathbb{Z}^+$, we then have
	\begin{equation*}
		\begin{split}
			&A^0( \bU_k )\partial_t \bU_k+ \sum^3_{i=1}A^i( \bU_k )\partial_{i}\bU_k = 0,
			\\
			&A^0( \bU_l ) \partial_t \bU_l+ \sum^3_{i=1}A^i(\bU_l)\partial_{i}\bU_l=0.
		\end{split}
	\end{equation*}
By standard energy estimates, we can carry out
	\begin{equation*}
		\frac{d}{dt}\|\bU_k - \bU_l \|_{H_x^{s-1}} \leq C_{k, l} \left(\| d\bU_k, d\bU_l\|_{L^\infty_x}\|\bU_k-\bU_l\|_{H_x^{s-1}}+ \|\bU_k-\bU_l\|_{L^\infty_x}\| \nabla \bU_l\|_{H_x^{s-1}} \right),
	\end{equation*}
	where $C_{k, l}$ depends on the $L^\infty_x$ norm of $\bU_k$and $\bU_l$. Due to the Strichartz estimates of $d\bu_k$ and $dh_k$ $( k \in \mathbb{Z}^+ )$ in Proposition \ref{p3}, we can derive that
	\begin{equation*}
		\begin{split}
			\|(\bU_k-\bU_l)(t,\cdot)\|_{H_x^{s-1}} &\lesssim \|(\bU_k-\bU_l)(0,\cdot)\|_{H_x^{s-1}}
			\lesssim\|(\bu_{0k}-\bu_{0l}, h_{0k}-h_{0l}) \|_{H^{s-1}}.
		\end{split}
	\end{equation*}
	So $\{(p_k,u_{1k},u_{2k},u_{3k})\}_{k \in \mathbb{Z}^+}$ is a Cauchy sequence in $C([0,T];H_x^{s-1})$. Denote $(p,u_1,u_2,u_3 )$ being the limit. Then $(p,u_1,u_2,u_3 ) \in C([0,T];H_x^{s-1})$ and
	\begin{equation}\label{uk1}
		\lim_{k\rightarrow \infty}(p_k,u_{1k},u_{2k},u_{3k})=(p,u_{1},u_{2},u_{3}) \ \mathrm{in} \ C([0,T];H_x^{s-1}).
	\end{equation}
Note $u^2_{0k}-(u^2_{1k}+u^2_{2k}+u^2_{3k})=1$. By \eqref{uk1}, we infer
		\begin{equation*}
		\lim_{k\rightarrow \infty} u_{0k} =u_{0} \ \mathrm{in} \ C([0,T];H_x^{s-1}),
	\end{equation*}
	and
	\begin{equation*}
		u^2_{0}-(u^2_{1}+u^2_{2}+u^2_{3})=1.
	\end{equation*}
By \eqref{mo3}, \eqref{mo5}, and \eqref{mo9}, a direct calculation tells us
\begin{equation}\label{uk1a}
	h = \int \frac{\vartheta \varrho^{\vartheta-1}}{ \varrho^{\vartheta}+\varrho } d\varrho.
\end{equation}
Due to \eqref{uk1} and Lemma \ref{jh0}, we get
	\begin{equation}\label{uk1b}
	\lim_{k\rightarrow \infty}\varrho_k=\varrho \ \mathrm{in} \ C([0,T];H_x^{s-1}).
\end{equation}
By \eqref{uk1a}, then $h_k$ is also a function of $\varrho_k$. Combining with \eqref{uk1b} and Lemma \ref{jh0}, we infer
	\begin{equation*}
		\lim_{k\rightarrow \infty}h_k=h \ \mathrm{in} \ C([0,T];H_x^{s-1}).
	\end{equation*}
Note $\bw_k=\mathrm{vort} ( \mathrm{e}^{h_k} \bu_k )$.
Due to product estimates and Lemma \ref{jh0}, we have
	\begin{equation}\label{cx}
		\begin{split}
			\|\bw_k-\bw_l\|_{H_x^{s_0-2}} &\lesssim (\| \bu_k-\bu_l\|_{H_x^{s_0-1}}+\|h_k-h_l\|_{H_x^{s_0-1}})
			\\
			& \lesssim (\| \bu_{0k}-\bu_{0l}\|_{H^{s}}+\|h_{0k}-h_{0l} \|_{H^{s}}).
		\end{split}
	\end{equation}
From \eqref{cx}, $\{\bw_k\}_{k \in \mathbb{Z}^+}$ is a Cauchy sequence in $C([0,T];H_x^{s_0-2})$. Let $\bw$ be this limit. We have
	\begin{equation*}
		\lim_{k\rightarrow \infty} \bw_k = \bw \ \mathrm{in} \ C([0,T];H_x^{s_0-2}).
	\end{equation*}
	Since $(h_k,\bu_k)$ is uniformly bounded in $C([0,T];H_x^{s})$ and $\bw_k$ is bounded in $C([0,T];H_x^{s_0})$ respectively. Thus, $(h,u_1,u_2,u_3 ) \in C([0,T];H_x^{s})$ and $\bw \in C([0,T];H_x^{s_0})$.

	By using Proposition \ref{p3} again, the sequence $\{(d\bu_k,dh_k)\}_{k \in \mathbb{Z}^+}$ is uniformly bounded in $L^2([0,T];C_x^\delta)$. Consequently, the sequence $\{(d\bu_k, dh_k)\}_{k \in \mathbb{Z}^+}$ converges to $(d\bu, dh)$ in $L^2([0,T];L_x^\infty)$. That is,
	\begin{equation}\label{ccr}
		\lim_{k \rightarrow \infty} (d\bu_k, dh_k)=(d\bu, dh).
	\end{equation}
	On the other hand, we also deduce that the sequence $\{(d \bu_k,dh_k)\}_{k \in \mathbb{Z}^+}$ is bounded in $L^2([0,T];\dot{B}^{s_0-2}_{\infty,2})$.
	Combining \eqref{ccr} and  \eqref{p303}, we get
	\begin{equation}\label{cc1}
	\| d \bu, dh\|_{ L^2_{[0,T]}L^\infty_x }+	\| d \bu, dh\|_{ L^2_{[0,T]}\dot{B}^{s_0-2}_{\infty,2} } \leq M.
	\end{equation}
	By the fundamental formula of calculus, when $T\leq \frac{1}{(1+M)^2}$, we can obtain
	\begin{equation}\label{ccs}
		\|u^0-1,\mathring{\bu}, h\|_{L^\infty([0,T]\times \mathbb{R}^3)} \leq C_0+T^{\frac12}\| d \bu, dh\|_{ L^2_{[0,T]}L^\infty_x } \leq 1+C_0.
	\end{equation}
	For
	\begin{equation*}
		\bu=\bu_{+}+\bu_{-},
	\end{equation*}
we therefore have
	\begin{equation}\label{cc2}
		\begin{split}
			\| \nabla \bu_{+}\|_{\dot{B}^{s_0-2}_{\infty,2}} \leq \|  \nabla \bu\|_{\dot{B}^{s_0-2}_{\infty,2}}+\| \nabla \bu_{-}\|_{\dot{B}^{s_0-2}_{\infty,2}}.
		\end{split}
	\end{equation}
Note that $\mathbf{P}$ is a space-time elliptic operator on $[0,T]\times \mathbb{R}^3$. By using the elliptic estimate and Sobolev inequality, we have
	\begin{equation*}
		\begin{split}
			\| \nabla \bu_{-}\|_{L^2_{[0,T]}\dot{B}^{s_0-2}_{\infty, 2}} & \lesssim \| \mathbf{P}^{-1}(\mathrm{e}^{h} \bW) \|_{L^2_{[0,T]} \dot{B}^{s_0-\frac{1}{2}}_{2,2}}
			\lesssim \|\mathrm{e}^{h} \bW\|_{L^2_{[0,T]}\ \dot{H}^{s_0-\frac{3}{2}}}
			 \lesssim (1+\|h_0\|_{H^{s}})\|\bw_0 \|_{H^{s_0}}.
		\end{split}
	\end{equation*}
This implies $\nabla \bu_{-} \in L^2([0,T];\dot{B}^{s_0-2}_{\infty,2})$. By \eqref{cc2}, then $\nabla \bu_{+} \in L^2([0,T];\dot{B}^{s_0-2}_{\infty,2})$. Therefore, we complete the proof of Theorem \ref{dingli}.
\end{proof}

\subsection{Reduction to small, smooth, and compactly supported data}\label{sec4.2}
Denote by $c$ the largest speed of
propagation corresponding to the acoustic metric $g$. We can further reduce the problem to this following
result:
\begin{proposition}\label{p1}
	Let $2<s_0<s\leq \frac52$. Assume \eqref{a1}-\eqref{a0} hold. Suppose the initial data $(h_0, \mathring{\bu}_0, \bw_0)$ be smooth, supported in $B(0,c+2)$ and satisfying
	\begin{equation}\label{300}
	\begin{split}
	&\|\mathring{\bu}_0\|_{H^s} + \| h_0 \|_{H^s} + \|\bw_0\|_{H^{s_0}} \leq \epsilon_3.
	\end{split}
	\end{equation}
	Then the Cauchy problem of \eqref{WTe} admits a smooth solution $(h,\bu,\bw)$ on $[-2,2] \times \mathbb{R}^3$. Moreover, it has the following properties:
	
	$\mathrm{(1)}$ energy estimate
	\begin{equation}\label{402}
	\begin{split}
	&\|\mathring{\bu}\|_{L^\infty_{[-1,1]} H_x^{s}}+\|u^0-1\|_{L^\infty_{[-1,1]} H_x^{s}}+\| h \|_{L^\infty_{[-1,1]} H_x^{s}} + \| \bw\|_{L^\infty_{[-1,1]} H_x^{s_0}} \leq \epsilon_2,
	\\
	& 1\leq u^0 \leq 1+C_0, \quad \|\bu,h\|_{L^\infty_{[-1,1]\times \mathbb{R}^3}} \leq 1+C_0.
	\end{split}
	\end{equation}

	$\mathrm{(2)}$ dispersive estimate for $\bu$ and $h$
	\begin{equation}\label{s403}
	\|d \bu, d h\|_{L^2_{[-1,1]} C^\delta_x}+\| d h, \nabla \bu_{+}, d \bu\|_{L^2_{[-1,1]} \dot{B}^{s_0-2}_{\infty,2}} \leq \epsilon_2,
	\end{equation}

where $\bu_{+}$ is defined as in \eqref{De}-\eqref{De0}.

	$\mathrm{(3)}$ dispersive estimate for the linear equation

Let $f$ satisfy
	\begin{equation}\label{linearA}
	\begin{cases}
		& \square_g f=0, \qquad (t,x) \in [-2,2]\times \mathbb{R}^3,
		\\
		&(f, \partial_t f)|_{t=t_0}=(f_0, f_1) \in H^r(\mathbb{R}^3) \times H^{r-1}(\mathbb{R}^3),
	\end{cases}
\end{equation}
where $g$ is defined in \eqref{AMd2}-\eqref{AMd}. For each $1 \leq r \leq s+1$, then the Cauchy problem \eqref{linearA} has a unique solution in $C([-2,2]; H_x^r) \cap C^1([-2,2]; H_x^{r-1})$. Furthermore, for $k<r-1$, we also have
	\begin{equation}\label{304}
	\|\left< \nabla\right>^k f\|_{L^2_{[-2,2]} L^\infty_x} \lesssim  \| f_0\|_{H^r}+ \| f_1\|_{H^{r-1}},
	\end{equation}
	and the same estimates hold with $\left< \nabla \right>^k$ replaced by $\left< \nabla \right>^{k-1}d$.
\end{proposition}
Due to Proposition \ref{p1}, we are ready to prove Proposition \ref{p3}.
\begin{proof}[Proof of Proposition \ref{p3} by using Proposition \ref{p1}] 
We divide the proof into three steps.

$\textbf{Step 1: Scaling}$. Recall
\begin{equation}\label{a4}
\begin{split}
&\|\mathring{\bu}_0\|_{H^s}+ \| h_0\|_{H^s} + \|\bw_0\|_{H^{s_0}}  \leq M_0.
\end{split}
\end{equation}
Taking the scaling
\begin{equation*}
\begin{split}
&\underline{{\bu}}(t,x)=\bu(Tt,Tx),\quad \underline{h}(t,x)=h(Tt,Tx),
\end{split}
\end{equation*}
by \eqref{a4}, we can compute out
\begin{equation}\label{sca}
\begin{split}
&\|\underline{\bu}_0\|_{\dot{H}^s}=T^{s-\frac{3}{2}} \|{\bu}_0\|_{\dot{H}^s} \leq M_0 T^{s-\frac{3}{2}}.
\end{split}
\end{equation}
Similarly, we obtain
\begin{equation}\label{scab}
	\begin{split}
	& \|\underline{h}_0\|_{\dot{H}^s} \leq M_0 T^{s-\frac{3}{2}},
	\end{split}
\end{equation}
and
\begin{equation}\label{scac}
	\begin{split}
		 \|\epsilon^{\alpha \beta \gamma \delta} \partial_\gamma \underline{u}_{0\delta} \|_{\dot{H}^{s_0}} & =  T^{s_0-\frac{1}{2}} \|\epsilon^{\alpha \beta \gamma \delta} \partial_\gamma {u}_{0\delta} \|_{\dot{H}^{s_0}}
		\\
		& = T^{s_0-\frac{1}{2}} \|\mathrm{e}^{-{h}_0}\bw_0 \|_{\dot{H}^{s_0}},
		\\
		& \leq M_0 T^{s_0-\frac{1}{2}}.
	\end{split}
\end{equation}
Set $\underline{\bw}_0=(\underline{w}^0_0,\underline{w}^1_0,\underline{w}^2_0,\underline{w}^3_0)^{\mathrm{T}}$ and $\underline{w}^\alpha_0=-\epsilon^{\alpha\beta\gamma\delta}\underline{{u}}_{0\beta} \partial_\gamma ( \mathrm{e}^{\underline{h}_0}\underline{{u}}_{0\delta})$. By \eqref{scac}, we have
\begin{equation}\label{scad}
	\begin{split}
		\|\underline{\bw}_0 \|_{\dot{H}^{s_0}} \leq M_0 T^{s_0-\frac{1}{2}}.
	\end{split}
\end{equation}
Choose sufficiently small $T$ such that
\begin{equation}\label{sca0}
\max\{ M_0 T^{s-\frac{3}{2}}, M_0 T^{s_0-\frac{1}{2}} \} \ll \epsilon_3.
\end{equation}
Therefore, by \eqref{sca}, \eqref{scab}, \eqref{scad}, and \eqref{sca0}, we get
\begin{equation}\label{sca1}
\begin{split}
& \|\underline{\bu}_0\|_{\dot{H}^s}\leq \epsilon_3, \qquad \|\underline{h}_0\|_{\dot{H}^s} \leq \epsilon_3,
\\
&
\|\epsilon^{\alpha \beta \gamma \delta} \partial_\gamma \underline{u}_{0\delta} \|_{\dot{H}^{s_0}}\leq \epsilon_3 , \quad  \|\underline{\bw}_0\|_{\dot{H}^{s_0}} \leq \epsilon_3.
\end{split}
\end{equation}
Using \eqref{HEw} and \eqref{muu}, we also have
\begin{equation}\label{sca2}
		\|\underline{\bu}_0,\underline{h}_0\|_{L^\infty}\leq C_0, \quad \underline{u}^0_0 \geq 1.
	\end{equation}
Next, to reduce the initial data with support set, we use the physical localization technique.

$\textbf{Step 2: Localization}$. Note that the propagation speed of \eqref{WTe} is finite. Let $c$ be the largest speed of \eqref{WTe}. Set $\chi(x)$ be a smooth function supported in $B(0,c+2)$, and which equals $1$ in $B(0,c+1)$. To be simple, we set
\begin{equation*}
	\underline{\bu}_0=(\underline{u}^0_0,\mathring{\underline{\bu}}_0).
\end{equation*}
For given $y \in \mathbb{R}^3$, we define the localized initial data for the velocity and density near $y$:
\begin{equation}\label{sid0}
\begin{split}
\mathring{\widetilde{\bu}}_0(x)=&\chi(x-y) ( \mathring{\underline{\bu}}_0(x) - \mathring{\underline{\bu}}_0(y)) ,
\\
\widetilde{ h }_0(x)=&\chi(x-y)( \underline{h}_0(x)- \underline{h}_0(y) ).
\end{split}
\end{equation}
Since the velocity $\widetilde{\bu}$ satisfies \eqref{muu}, so we set
\begin{equation}\label{sid00}
	{\widetilde{u}^0}_0=\sqrt{1+ |\mathring{\widetilde{\bu}}_0|^2}, \quad \widetilde{\bu}_0=(	{\widetilde{u}^0}_0, \mathring{\widetilde{\bu}}_0).
\end{equation}
By using \eqref{sca2}, \eqref{sid0}, and \eqref{sid00}, we have
\begin{equation}\label{sca3}
	\|\mathring{\widetilde{\bu}}_0,\widetilde{ h }_0\|_{L^\infty}\leq C_0, \quad 1\leq{\widetilde{u}^0}_0 \leq 1+C_0.
\end{equation}
Since $s,s_0\in (2, \frac{5}{2})$, we can verify that
\begin{equation}\label{245}
	\| \widetilde{\bu}_0\|_{{H}^s} \lesssim  \|\underline{\bu}_0\|_{\dot{H}^s}, \quad  \| \widetilde{h}_0 \|_{H^s} \lesssim \|\underline{h}_0 \|_{\dot{H}^s}.
\end{equation}
We also need to define the specific vorticity:
\begin{equation}\label{wde}
	\widetilde{\bw}_0=-\epsilon^{\alpha \beta \gamma \delta}\mathrm{e}^{\widetilde{ h }_0+\widetilde{h}_0(y) }(\widetilde{u}_{0\beta}+ \underline{u}_{0\beta}(y)) \partial_{\gamma}\widetilde{u}_{0\delta}.
\end{equation}
For $\widetilde{\bu}_{0}$ only depending on spatial variables, so \eqref{wde} and \eqref{sca3} give us
\begin{equation}\label{sid1}
	\| \widetilde{\bw}_0 \|_{L^2} \lesssim \|\nabla \widetilde{\bu}_{0}\|_{L^2} .
\end{equation}
Substituting \eqref{sid0} to \eqref{wde}, so we have
\begin{equation}\label{sid3}
	\begin{split}
		\widetilde{\bw}_0 = &
		-\epsilon^{\alpha \beta \gamma \delta}\mathrm{e}^{\widetilde{ h }_0} \widetilde{u}_{0\beta}  \partial_{\gamma} \chi \cdot ( \underline{u}_{0\delta} - \underline{u}_{0\delta}(y))
	 - \chi^2  \mathrm{e}^{\widetilde{ h }_0} (\widetilde{u}_{0\beta}+ \underline{u}_{0\beta}(y)) \cdot \epsilon^{\alpha \beta \gamma \delta}
	\partial_{\gamma} \underline{u}_{0\delta}.
	\end{split}
\end{equation}
Due to \eqref{sid3}, \eqref{sca3}, and H\"older's inequality, we infer
\begin{equation}\label{sid4}
	\begin{split}
		\| \widetilde{\bw}_0 \|_{\dot{H}^{s_0}}\lesssim &
		\| \underline{u}_{0\delta}\|_{\dot{H}^{s_0}}+ \|  \underline{h}_{0}\|_{\dot{H}^{s_0}}
		+ \| \epsilon^{\alpha \beta \gamma \delta}
		\partial_{\gamma} \underline{u}_{0\delta} \|_{\dot{H}^{s_0}} .
	\end{split}
\end{equation}
Combining \eqref{sca1}, \eqref{245}, \eqref{sid1}, and \eqref{sid4}, we therefore have
\begin{equation}\label{sid}
\begin{split}
&	\|\widetilde{ \bu }_0 \|_{{H}^{s}}+\|\widetilde{h}_0 \|_{{H}^{s}}+ \|\widetilde{ \bw }_0 \|_{{H}^{s_0}} \lesssim \epsilon_3.
\end{split}
\end{equation}
$\textbf{Step 3: Using Proposition \ref{p1}}$.
Under the initial conditions \eqref{sca3} and \eqref{sid}, by Proposition \ref{p1}, there is a smooth solution $(\widetilde{h}, \widetilde{\bu}, \widetilde{\bw})$ on $[-2,2]\times \mathbb{R}^3$ such that
\begin{equation}\label{p}
\begin{cases}
\square_{\widetilde{g}} \widetilde{h}=\widetilde{D},
\\
\square_{\widetilde{g}} \widetilde{u}^\alpha=-\widetilde{c}^2_s \widetilde{\Omega} \mathrm{e}^{- (\widetilde{h} + \underline{h}_0(y))} \widetilde{W}^\alpha+\widetilde{Q}^\alpha,
\\
(\widetilde{u}^\kappa+\underline{u}_0^\kappa(y)) \partial_\kappa \widetilde{w}^\alpha= -(\widetilde{u}^\alpha+\underline{u}_0^\alpha(y)) \widetilde{w}^\kappa \partial_\kappa \widetilde{h}+ \widetilde{w}^\kappa \partial_\kappa \widetilde{u}^\alpha- \widetilde{w}^\alpha \partial_\kappa \widetilde{u}^\kappa,
\\
\partial_\alpha \widetilde{w}^\alpha =-\widetilde{w}^\kappa \partial_\kappa \widetilde{h},
\\
(\widetilde{h},\widetilde{u},\widetilde{w})|_{t=0}=(\widetilde{h}_0,\widetilde{u}_0,\widetilde{w}_0),
\end{cases}
\end{equation}
where the quantities $\widetilde{c}^2_s$, $\widetilde{{g}}$, $\widetilde{\Omega}$, and $\widetilde{\bW}=(\widetilde{W}^0,\widetilde{W}^1,\widetilde{W}^2,\widetilde{W}^3)^{\mathrm{T}}$ are defined by
\begin{equation}\label{DDE}
	\begin{split}
		\widetilde{c}^2_s&= \frac{dp}{d{h}}(\widetilde{h}+\underline{h}_0(y)),
		\\
		\widetilde{g}&=g(\widetilde{h}+\underline{h}_0(y),\widetilde{\bu}+\underline{\bu}_0(y) ),
		\\
\widetilde{\Omega}&=\Omega(\widetilde{h}+\underline{h}_0(y), \widetilde{\bu}+\underline{\bu}_0(y)),
	\end{split}
\end{equation}
and
\begin{equation}\label{DDE0}
	\begin{split}
		\widetilde{W}^\alpha& =-\epsilon^{\alpha \beta \gamma \delta} (\widetilde{u}_\beta+\underline{u}_{0\beta}(y)) \partial_{\gamma} \widetilde{w}_\delta+\widetilde{c}_s^{-2}\epsilon^{\alpha\beta\gamma\delta}(\widetilde{u}_\beta+\underline{u}_{0\beta}(y)) \widetilde{w}_{\delta}\partial_{\gamma}\widetilde{ h },
	\end{split}
\end{equation}
and
\begin{equation*}
	\begin{split}
		\widetilde{Q}^\alpha=& \widetilde{Q}_{1 \kappa }^{\alpha \beta \gamma } \partial_\beta \widetilde{h} \partial_\gamma \widetilde{u}^\kappa+\widetilde{Q}_{2\kappa  \delta}^{\alpha \beta \gamma } \partial_\beta \widetilde{u}^\kappa \partial_\gamma u^\delta +\widetilde{Q}_{3}^{\alpha \beta \gamma } \partial_\beta \widetilde{h} \partial_\gamma \widetilde{h},
		\\
		\widetilde{D}^\alpha=&\widetilde{D}_{1 \kappa }^{\beta \gamma } \partial_\beta \widetilde{h} \partial_\gamma \widetilde{u}^\kappa+\widetilde{D}_{2\kappa  \delta}^{ \beta \gamma } \partial_\beta \widetilde{u}^\kappa \partial_\gamma \widetilde{u}^\delta +\widetilde{D}_{3}^{ \beta \gamma } \partial_\beta \widetilde{h} \partial_\gamma \widetilde{h},
	\end{split}
\end{equation*}
and the functions $\widetilde{Q}_{1 \kappa }^{\alpha \beta \gamma }$, $\widetilde{Q}_{2\kappa  \delta}^{\alpha \beta \gamma }$, $\widetilde{Q}_{3}^{\alpha \beta \gamma }$, $\widetilde{D}_{1 \kappa }^{\beta \gamma }$, $\widetilde{D}_{2\kappa  \delta}^{ \beta \gamma }$, $\widetilde{D}_{3}^{ \beta \gamma }$ have the same formulations with ${Q}_{1 \kappa }^{\alpha \beta \gamma }$, ${Q}_{2\kappa  \delta}^{\alpha \beta \gamma }$, $\widetilde{Q}_{3}^{\alpha \beta \gamma }$, ${D}_{1 \kappa }^{\beta \gamma }$, ${D}_{2\kappa  \delta}^{ \beta \gamma }$, ${D}_{3}^{ \beta \gamma }$ by replacing $(\bu, {h})$ to $(\widetilde{\bu}+\underline{\bu}_0(y), \widetilde{h}+\underline{h}_0(y))$.
To see a better component of the velocity $\widetilde{\bu}$, we define
\begin{equation}\label{deu1}
\widetilde{\mathbf{P}}=\mathrm{I}-\left[ m^{\alpha \beta}+ 2 (\widetilde{u}^\alpha+ \underline{u}^\alpha_0(y) ) (\widetilde{u}^\beta+ \underline{u}^\beta_0(y) ) \right] \partial^2_{ \alpha \beta}.
\end{equation}
Then $\widetilde{\mathbf{P}}$ is a space-time elliptic operator (please see Lemma \ref{app1}). We next denote
\begin{equation}\label{deu2}
\widetilde{u}^\alpha_{-} =\mathbf{P}^{-1}(\mathrm{e}^{ \widetilde{h}+ \underline{h}_0(y) } \widetilde{ W }^\alpha), \quad  \widetilde{u}_{+}=\widetilde{u}-\widetilde{u}_{-}.
\end{equation}
and
\begin{equation}\label{deu3}
	\widetilde{\mathbf{T}}= \partial_t + \frac{\widetilde{u}^i+\underline{u}^i_0(y)}{\widetilde{u}^0+ \underline{u}^0_0(y) }\partial_i.
\end{equation}
By \eqref{UM}, so we infer
\begin{equation*}
\begin{split}
  \square_g  \widetilde{u}^\alpha_{+}&=\widetilde{\Omega} (\widetilde{c}^2_s+1) (\widetilde{u}^\gamma+\underline{u}^\gamma_0(y)) (\widetilde{u}^\beta+\underline{u}^\beta_0(y)) \partial^2_{\beta \gamma}\widetilde{u}^\alpha_{-}+\widetilde{\Omega} \widetilde{c_s}^2 \widetilde{u}_{-}^\alpha+ \widetilde{Q}^\alpha.
\end{split}
\end{equation*}
Using Proposition \ref{p1} again, the solution $(\widetilde{h}, \widetilde{{\bu}}, \widetilde{\bw})$ also satisfies
\begin{equation}\label{see0}
  \|\mathring{\widetilde{\bu}}\|_{L^\infty_{[-2,2]}H_x^s}+\|\widetilde{u}^0-1\|_{L^\infty_{[-2,2]}H_x^s}+ \|\widetilde{h}\|_{L^\infty_{[-2,2]}H_x^s}+ \| \widetilde{\bw} \|_{L^\infty_{[-2,2]}H_x^{s_0}} \leq \epsilon_2,
\end{equation}
and
\begin{equation}\label{see1}
  \|d\widetilde{\bu}, d\widetilde{h}\|_{L^2_{[-2,2]} C^\delta_x}+ \| d\widetilde{\bu}, \nabla \widetilde{\bu}_{+}, d\widetilde{h} \|_{L^2_{[-2,2]} \dot{B}^{s_0-2}_{\infty,2}} \leq \epsilon_2.
\end{equation}
Moreover, the linear wave equation endowed with the Lorentz metric $\widetilde{{g}}$
\begin{equation}\label{312}
\begin{cases}
&\square_{ \widetilde{{g}}} \tilde{f}=0,
\\
&\tilde{f}(t_0,\cdot)=\widetilde{f}_0, \ \partial_t\tilde{f}(t_0,)=\widetilde{f}_1,
\end{cases}
\end{equation}
also admits a solution $\widetilde{f} \in C([-2,2],H_x^r)\times C^1([-2,2],H_x^{r-1})$, and if $k<r-1$, the following estimate
	\begin{equation}\label{s31}
	\|\left< \nabla \right>^k \widetilde{f}\|_{L^2_{[-2,2]} L^\infty_x} \lesssim  \| \widetilde{f}_0\|_{H^r}+ \| \widetilde{f}_1\|_{H^{r-1}},
	\end{equation}
holds. On the other hand,  if we set 
\begin{equation}\label{vw}
	\begin{split}
		\bar{h}(t,x)  &=\widetilde{ h}+ \underline{h}_0(y),
		\\	
		\bar{{\bu}}(t,x)&=(\bar{u}^0, \mathring{\bar{\bu}}), \quad \bar{u}^0=\sqrt{1+|\mathring{\bar{\bu}}|^2},
		\\
		\mathring{\bar{\bu}}(t,x)  &= \mathring{\widetilde{\bu}}+\underline{\mathring{\bu}}_0(y),
		\\
		\bar{\bw}(t,x)  &=\mathrm{vort}(\mathrm{e}^{\bar{h}} \bar{\bu}),
	\end{split}
\end{equation}
then $(\bar{h},\bar{\bu},\bar{\bw})$ is also a solution of \eqref{p}, and its initial data coincides with $(h_0, \bu_0, \bw_0)$ in $B(y,c+1)$. Consider a restriction
\begin{equation}\label{RS}
   (\bar{h},\bar{\bu},\bar{\bw})|_{\mathrm{K}^y},
\end{equation}
where $y\in \mathbb{R}^3$ and $\mathrm{K}^y=\left\{ (t,x): ct+|x-y| \leq c+1, |t| <1 \right\}$. Then \eqref{RS} solves \eqref{p} on $\mathrm{K}^y$. By finite speed of propagation, a smooth solution $(\bar{h}, \bar{\bu}, \bar{\bw})$ solves \eqref{WTe} on $[-1,1] \times \mathbb{R}^3$. Without diffusion, we still use the notation $\bar{h}, \bar{\bu}$ and $\bar{\bw}$ being the restrictions on $\mathrm{K}^y$.

From \eqref{vw}, by time-space scaling $(t,x)$ to $(T^{-1}t,T^{-1}x)$, we also obtain
\begin{equation}\label{po0}
	\begin{split}
		&(h, \bu, \bw)=(\bar{h}, \bar{\bu}, \bar{\bw})(T^{-1}t,T^{-1}x),
		\\
		&( h, \bu, \bw)|_{t=0}=(\bar{h}, \bar{\bu}, \bar{\bw})(0,T^{-1}x)=(h_0, \bu_0, \bw_0).
	\end{split}
\end{equation}
Therefore, the function $(h, \bu, \bw)$ defined in \eqref{po0} is the solution of \eqref{WTe} according to the uniqueness of solutions, i.e. Corollary \ref{cor}. To obtain the energy estimates for $(h, \bu,\bw)$ on $[0,T]\times \mathbb{R}^3$, let us use the cartesian grid $3^{-\frac12} \mathbb{Z}^3$ in $\mathbb{R}^3$, and a corresponding smooth partition of unity
\begin{equation*}
	\textstyle{\sum}_{y \in 3^{-\frac12} \mathbb{Z}^3} \psi(x-y)=1,
\end{equation*}
such that the function $\psi$ is supported in the unit ball. Therefore, we have
\begin{equation}\label{vwa}
	\begin{split}
		& \mathring{\bar{\bu}}=\textstyle{\sum}_{y \in 3^{-\frac12} \mathbb{Z}^3} \psi(x-y)(\mathring{\widetilde{\bu}}+\underline{\mathring{\bu}}_0(y) ),
		\\
		& \bar{h}=\textstyle{\sum}_{y \in 3^{-\frac12} \mathbb{Z}^3} \psi(x-y)(\widetilde{h}+\underline{h}_0(y) ),
		\\
		& \bar{u}^0=\sqrt{1+|\mathring{\bar{\bu}}|^2}, \quad \bar{\bu}=(\bar{u}^0, \mathring{\bar{\bu}}) , \quad \bar{\bw} =\mathrm{vort}(\mathrm{e}^{\bar{h}} \bar{\bu}).
	\end{split}
\end{equation}
By \eqref{see1} and \eqref{vwa}, we can see
\begin{equation*}\label{po1}
	\begin{split}
		& \|d\bar{\bu}, d\bar{h}\|_{L^2_{[0,1]}C^{\delta}_x}+ \| d\bar{\bu}, \nabla \bar{\bu}_{+},d\bar{h} \|_{L^2_{[0,1]} \dot{B}^{s_0-2}_{\infty,2}}
		\\
		\leq & \sup_{y \in 3^{-\frac12} \mathbb{Z}^3} ( \|d\widetilde{\bu},d\widetilde{h}\|_{L^2_{[0,1]}{C^{\delta}_x}}+ \| d\widetilde{\bu}, \nabla \widetilde{\bu}_{+}, d\widetilde{h} \|_{L^2_{[0,1]} \dot{B}^{s_0-2}_{\infty,2}} )
		\\
		\leq & C(\|\mathring{\widetilde{\bu}}_0\|_{H_x^s}+ \|\widetilde{h}_0\|_{H_x^s}+ \| \widetilde{\bw}_0 \|_{H_x^{s_0}}).
	\end{split}
\end{equation*}
Due to \eqref{see0}, \eqref{sca2}, and \eqref{vwa}, we shall obtain
\begin{equation}\label{po03}
	\begin{split}
		& \|\bar{u}^0-1,\bar{\mathring{\bu}}, \bar{h}\|_{L^\infty_{[0,1]}L^\infty_x} \leq \|\widetilde{u}^0-1,\mathring{\widetilde{\bu}}, \widetilde{h} \|_{H_x^{s}} + \|\underline{u}_0^0-1,\mathring{\underline{\bu}}_0, \underline{h}_0\|_{L^\infty_x} \leq 1+C_0.
	\end{split}
\end{equation}
By changing of coordinates and using ${\delta} \in (0,s-2)$ and $2<s_0<s$, we get
\begin{equation}\label{po2}
	\begin{split}
		& \|d{\bu}, d{h}\|_{L^2_{[0,T]}C^{\delta}_x}+ \| d{\bu}, \nabla {\bu}_{+}, d{h} \|_{L^2_{[0,T]} \dot{B}^{s_0-2}_{\infty,2}}
		\\
		\leq & T^{-(\frac12+{\delta})}\|d\widetilde{\bu}, d\widetilde{h}\|_{L^2_{[0,1]}C^{\delta}_x}+ \| d\widetilde{\bu}, \nabla \widetilde{\bu}_{+}, d\widetilde{h} \|_{L^2_{[0,1]} \dot{B}^{s_0-2}_{\infty,2}}
		\\
		\leq
		&C (T^{-(\frac12+{\delta})}+T^{-(\frac12+s_0-2)}) (\|\mathring{\widetilde{\bu}}_0\|_{H_x^s}+ \|\widetilde{h}_0\|_{H_x^s}+ \| \widetilde{\bw}_0 \|_{H_x^{s_0}})
		\\
		\leq & C (T^{-(\frac12+{\delta})}+T^{-(\frac12+s_0-2)})(T^{s-\frac32}\|{\bu}_0\|_{\dot{H}_x^s}+ T^{s-\frac32}\|h_0\|_{\dot{H}_x^s}
		 + T^{s_0+1-\frac32}\| {\bw}_0 \|_{\dot{H}_x^{s_0}})
		\\
		\leq & C (\|\mathring{\bu}_0\|_{H_x^s}+ \|{h}_0\|_{H_x^s}+ \| {\bw}_0 \|_{H_x^{s_0}}).
	\end{split}
\end{equation}
By using \eqref{po03} and changing of coordinates from $(t,x) \rightarrow (T^{-1}t,T^{-1}x)$, we get
\begin{equation}\label{po05}
	\begin{split}
		& \|u^0-1, \mathring{\bu}, {h}\|_{L^\infty_{[0,T]\times \mathbb{R}^3}} = \|\bar{u}^0-1,\mathring{\bar{\bu}}, \bar{h}\|_{L^\infty_{[0,1]}L^{\infty}_x}  \leq 1+C_0,
	\end{split}
\end{equation}
and
\begin{equation}\label{po0a}
	\begin{split}
		& u^0=\sqrt{1+ |\mathring{\bu}|^2} \geq 1.
	\end{split}
\end{equation}
Using Theorem \ref{VE}, \eqref{po2}, \eqref{po05}, and \eqref{po0a}, we can obtain that $( h, \bu, \bw)$ satisfies \eqref{ig1}, \eqref{e9} and \eqref{p303}. It remains for us to prove \eqref{p304} and \eqref{305}.

For $1\leq r \leq s+1$, we consider the following homogeneous linear wave equation
\begin{equation}\label{po3}
	\begin{cases}
		\square_{{g}} f=0, \qquad [0,T]\times \mathbb{R}^3,
		\\
		(f,f_t)|_{t=0}=(f_0,f_1)\in H_x^r \times H^{r-1}_x.
	\end{cases}
\end{equation}
For \eqref{po3} is a scaling invariant system, so we can transfer \eqref{po3} to the localized linear equation
\begin{equation}\label{po4}
	\begin{cases}
		\square_{\widetilde{g}} {f}^y=0, \quad [0,1]\times \mathbb{R}^3,
		\\
		({f}^y,{f}^y_t)|_{t=0}=(f_0^y,f^y_1),
	\end{cases}
\end{equation}
where
\begin{equation*}\label{po5}
	\begin{split}
		& f_0^y=\chi(x-y)(\widetilde{f}_0-\widetilde{f}_0(y)), \quad f^y_1=\chi(x-y)\widetilde{f}_1,
		\\ & \widetilde{f}_0=f_0(Tx),\quad \widetilde{f}_1=f_1(Tx).
	\end{split}
\end{equation*}
Let
\begin{equation}\label{po6}
	\widetilde{f}=\textstyle{\sum}_{y\in 3^{-\frac12}\mathbb{Z}^3}\psi(x-y)({f}^y+\widetilde{f}_0(y)).
\end{equation}
Seeing \eqref{po4}, using Proposition \ref{p1} again, for $k<r-1$, we shall obtain that
\begin{equation}\label{po7}
	\begin{split}
		\|\left< \nabla \right>^{k-1} d {f}^y\|_{L^2_{[0,1]} L^\infty_x} \leq  & C(\| \chi(x-y)(\widetilde{f}_0-\widetilde{f}_0(y)) \|_{H_x^r}+ \| \chi(x-y)\widetilde{f}_1 \|_{H_x^{r-1}})
		\\
		\leq & C(\|\widetilde{f}_0\|_{\dot{H}_x^r}+ \| \widetilde{f}_1 \|_{\dot{H}_x^{r-1}}).
	\end{split}
\end{equation}
We again use the finite speed of propagation to conclude that $\widetilde{f}={f}^y+\widetilde{f}_0(y)$ in $\mathrm{K}^y$. Using \eqref{po6} and \eqref{po7}, so we get
\begin{equation}\label{po8}
	\begin{split}
		 \|\left< \nabla \right>^{k-1} d \widetilde{f}\|_{L^2_{[0,1]} L^\infty_x}
		= & \sup_{y\in 3^{-\frac12}\mathbb{Z}^3}\|\left< \nabla \right>^{k-1} d ({f}^y+\widetilde{f}_0(y))\|_{L^2_{[0,1]} L^\infty_x}
		\\
		= & \sup_{y\in 3^{-\frac12}\mathbb{Z}^3}\|\left< \nabla \right>^{k-1} d {f}^y\|_{L^2_{[0,1]} L^\infty_x}
		\\
		\leq & C(\|\widetilde{f}_0\|_{\dot{H}_x^r}+ \| \widetilde{f}_1 \|_{\dot{H}_x^{r-1}}).
	\end{split}
\end{equation}
By changing of coordinates from $(t,x)\rightarrow (T^{-1}t,T^{-1}x)$, we have
\begin{equation}\label{po9}
	\begin{split}
		\|\left< \nabla \right>^{k-1} d{f}\|_{L^2_{[0,T]} L^\infty_x} = & T^{k-\frac12}\|\left< \nabla \right>^{k-1} d\widetilde{f}\|_{L^2_{[0,1]} L^\infty_x}.
	\end{split}
\end{equation}
Taking $T\leq 1$, combing \eqref{po8} with \eqref{po9}, if $k<r-1$, then we have
\begin{equation}\label{po10}
	\begin{split}
		\|\left< \nabla \right>^{k-1} d{f}\|_{L^2_{[0,T]} L^\infty_x}
		\leq & CT^{r-1-k}(\|{f}_0\|_{\dot{H}_x^r}+ \| {f}_1 \|_{\dot{H}_x^{r-1}})
		\\
		\leq & C(\|{f}_0\|_{{H}_x^r}+ \| {f}_1 \|_{{H}_x^{r-1}}).
	\end{split}
\end{equation}
By energy estimates for \eqref{po3} and \eqref{po2}, we can carry out
\begin{equation}\label{po11}
	\begin{split}
		\|d {f}\|_{L^\infty_{[0,T]} H^{r-1}_x} \leq  & C(\| {f}_0\|_{H_x^r}+ \| {f}_1\|_{H_x^{r-1}}) \exp\{  \|d g \|_{L^1_{[0,T]L^\infty_x}}\}
		\\
		\leq  & C_{M_0}(\| {f}_0\|_{H_x^r}+ \| {f}_1\|_{H_x^{r-1}}) \exp\{  \|d \bu,dh \|_{L^1_{[0,T]L^\infty_x}}\}
		\\
		\leq & C_{M_0}(\| {f}_0\|_{H_x^r}+ \| {f}_1\|_{H_x^{r-1}}).
	\end{split}
\end{equation}
By using $f=\int^{t}_0 \partial_ t fd\tau -f_0$, then
\begin{equation}\label{po12}
	\begin{split}
		\|f\|_{L^\infty_{[0,T]} L^{2}_x}
		\leq & T \|d {f}\|_{L^\infty_{[0,T]} L^{2}_x}+ \| {f}_0\|_{L_x^2}
		\\
		\leq & C_{M_0}(\| {f}_0\|_{H_x^r}+ \| {f}_1\|_{H_x^{r-1}}).
	\end{split}
\end{equation}
Adding \eqref{po11} and \eqref{po12}, we therefore obtain
\begin{equation}\label{po13}
	\begin{split}
		\|{f}\|_{L^\infty_{[0,T]} H^{r}_x}+ \|\partial_t {f}\|_{L^\infty_{[0,T]} H^{r-1}_x} \leq  C_{M_0}(\| {f}_0\|_{H_x^r}+ \| {f}_1\|_{H_x^{r-1}}).
	\end{split}
\end{equation}
Because of \eqref{po10}-\eqref{po13}, \eqref{p304} and \eqref{305} hold. Thus, we give a complete proof for Proposition \ref{p3}.
\end{proof}
\subsection{Reduction to a bootstrap argument: Proposition \ref{p4}}\label{sec4.3}
We will establish Proposition \ref{p1} via a continuity argument. For technical reasons, it is convenient for us to replace the original acoustic metric $g$ by a truncated one. So we first introduce the process of truncation.

Set two vectors
\begin{equation}\label{0e}
	\mathbf{0}=(0,0,0,0)^{\mathrm{T}}, \quad \mathbf{e}=(1,0,0,0)^{\mathrm{T}}.
\end{equation}
Then $h=0, \bu=\mathbf{e}, \bw=\mathbf{0}$ is a equilibrium state for System \eqref{WTe}. So the data in \eqref{300} can be seen as a small perturbation around the equilibrium state. Taking $\bu=\mathbf{e}, h=0$ in $g$, we record it $g(0)$. The inverse matrix of the metric $g(0)$ is
\begin{equation*}
	g^{-1}(0)=
	\left(
	\begin{array}{cccc}
		-1 & 0 & 0 & 0\\
		0 & c^2_s(0) & 0& 0\\
		0 & 0 & c^2_s(0) & 0
		\\
		0 & 0 & 0 &   c^2_s(0)
	\end{array}
	\right ).
\end{equation*}
By a linear change of coordinates which preserves $dt$, we may assume that $g^{\alpha \beta}(0)=m^{\alpha \beta}$. Let $\chi$ be a smooth cut-off function supported in the region $B(0,3+2c) \times [-\frac{3}{2}, \frac{3}{2}]$, which equals to $1$ in the region $B(0,2+2c) \times [-1, 1]$. Let $\mathbf{g}= (\mathbf{g}^{\alpha \beta})_{4\times 4}$. Define
\begin{equation}\label{AMd3}
	\begin{split}
		\mathbf{g}^{\alpha \beta}=& \chi(t,x)(g^{\alpha \beta}-g^{\alpha \beta }(0))+g^{\alpha \beta }(0),
	\end{split}
\end{equation}
where $g^{\alpha\beta}=\Omega\{ c_s^{2}m^{\alpha\beta}+(c_s^2-1)u^\alpha u^\beta \}$ has been given by \eqref{AMd}. Consider a wave-transport system
\begin{equation}\label{XT}
	\begin{cases}
		\square_{\mathbf{g}} h=D,
		\\
		\square_{\mathbf{g}} u^\alpha=-c^2_s \Omega \mathrm{e}^{-h} W^\alpha+Q^\alpha,
		\\
		u^\kappa \partial_\kappa w^\alpha = -u^\alpha w^\kappa \partial_\kappa h+ w^\kappa \partial_\kappa u^\alpha- w^\alpha \partial_\kappa u^\kappa,
		\\
		\partial_\alpha w^\alpha =-w^\kappa \partial_\kappa h,
	\end{cases}
\end{equation}
where $D$, $\Omega$, $W^\alpha$, and $Q^\alpha$ have the formulations in \eqref{err1}, \eqref{AMd2}, \eqref{MFd} and \eqref{err}.

We denote by $\mathcal{H}$ the family of smooth solutions $(h, \bu, \bw)$ to Equation \eqref{XT} for $t \in [-2,2]$, with the initial data $(h_0, \mathring{\bu}_0, \bw_0)$ supported in $B(0,2+c)$, and for which
\begin{equation}\label{401}
	\|\mathring{\bu}_0\|_{H^s} + \|h_0\|_{H^s}+\| \bw_0\|_{H^{s_0}}  \leq \epsilon_3,
\end{equation}
\begin{equation}\label{402a}
	\|(\mathring{\bu},u^0-1)\|_{L^\infty_{[-2,2]} H_x^{s}}+\| h \|_{L^\infty_{[-2,2]} H_x^{s}}+\| \bw\|_{L^\infty_{[-2,2]} H_x^{s_0}} \leq 2 \epsilon_2,
\end{equation}
\begin{equation}\label{403}
	\| d \bu, d h \|_{L^2_{[-2,2]} C_x^\delta}+ \|d h, \nabla \bu_{+}, d\bu\|_{L^2_{[-2,2]} \dot{B}^{s_0-2}_{\infty,2}} \leq 2 \epsilon_2.
\end{equation}
Then our bootstrap argument can be
stated as follows:
\begin{proposition}\label{p4}
Assume that \eqref{a0} holds. Then there is a continuous functional\footnote{We set aside for the moment the definition of
	$\Re$.} $\Re: \mathcal{H} \rightarrow \mathbb{R}^{+}$, satisfying $\Re(0,\mathbf{e},\mathbf{0})=0$, so that for each $(h, \bu, \bw) \in \mathcal{H}$ satisfying $\Re(h,\bu,\bw) \leq 2 \epsilon_1$ the following hold:

$\mathrm{(1)}$ The function $h$, $\bu$, and $\bw$ satisfies $\Re(h,\bu) \leq \epsilon_1$.

$\mathrm{(2)}$ The following estimate holds,
\begin{equation}\label{404}
\|(\mathring{\bu},u^0-1)\|_{L^\infty_{[-2,2]} H_x^{s}}+ \|h\|_{L^\infty_{[-2,2]} H_x^{s}}+ \|\bw\|_{L^\infty_{[-2,2]} H_x^{s_0}} \leq \epsilon_2,
\end{equation}
\begin{equation}\label{405}
\|d \bu, d h\|_{L^2_{[-2,2]} C^\delta_x}+\|\nabla \bu_{+}, d h, d\bu\|_{L^2_{[-2,2]} \dot{B}^{s_0-2}_{\infty,2}} \leq \epsilon_2.
\end{equation}

$\mathrm{(3)}$ For any $1 \leq k \leq s+1$, and for each $t_0\in [-2,2)$, the linear wave equation
	\begin{equation}\label{linearB}
	\begin{cases}
		& \square_{\mathbf{g}} f=0, \qquad (t,x) \in (t_0,2]\times \mathbb{R}^3,
		\\
		&(f, \partial_t f)|_{t=t_0}=(f_0,f_1) \in H^r(\mathbb{R}^3)\times H^{r-1}(\mathbb{R}^3),
	\end{cases}
\end{equation}
admits a unique solution $f \in C([-2,2],H_x^r) \times C^1([-2,2],H_x^{r-1})$. Furthermore, the Strichartz estimates \eqref{304} hold.
\end{proposition}
Based on Proposition \ref{p4}, we can prove
Proposition \ref{p1} as follows.
\begin{proof}[Proof of Proposition \ref{p1} by using Proposition \ref{p4}]
Consider the initial data in Proposition \ref{p1} satisfying
\begin{equation*}
	\|\mathring{\bu}_0\|_{H^s} + \|h_0\|_{H^s}+ \| \bw_0\|_{H^{s_0}}  \leq \epsilon_3.
\end{equation*}
Denote by $\text{A}$ the subset of those $\gamma \in [0,1]$ such that the equation \eqref{XT} admits a smooth solution $(h_\gamma, \bu_\gamma, \bw_\gamma)$ having the initial data
\begin{equation*}
	\begin{split}
		h_\gamma(0)=&\gamma h_0,
		\\
		\mathring{\bu}_\gamma(0)=&\gamma \mathring{\bu}_0,
		\\
		u^0_\gamma(0)=& \sqrt{1+|\mathring{\bu}_\gamma(0)|^2},
		\\
		\bw_\gamma(0)=& \mathrm{vort}( \mathrm{e}^{h_\gamma(0)}\bu_\gamma(0)),
	\end{split}
\end{equation*}
and such that $\Re(h_\gamma,\bu_\gamma, \bw_\gamma) \leq \epsilon_1$ and \eqref{404}, \eqref{405} hold. If $\gamma=0$, then
\begin{equation*}
	(h_\gamma, {\bu}_\gamma, \bw_\gamma)(t,x)=(0,\mathbf{e},\mathbf{0}),
\end{equation*}
is a smooth solution of \eqref{XT} with initial data
\begin{equation*}
	(h_\gamma, {\bu}_\gamma, \bw_\gamma)(0,x)=(0,\mathbf{e},\mathbf{0}).
\end{equation*}
Thus, the set $\text{A}$ is not empty. If we can prove that $\text{A}=[0,1]$, then $1 \in \text{A}$. Also, \eqref{300}-\eqref{304} follow from Proposition \ref{p4}. As a result, Proposition \ref{p1} holds. It is sufficient to prove that $\text{A}$ is both open and closed in $[0,1]$.

(1) $\text{A}$ is open. Let $\gamma \in \text{A}$. Then $(h_\gamma, \bu_\gamma, \bw_\gamma)$ is a smooth solution to \eqref{XT}, where
\begin{equation*}
	\bw_\gamma= \mathrm{vort}(\mathrm{e}^{h_\gamma} \bu_\gamma).
\end{equation*}
Let $\beta$ be close to $\gamma$.  Since $(h^\gamma,\bu^\gamma,\bw^\gamma)$ is smooth, a perturbation argument
shows that the equation \eqref{XT} has a smooth solution $(h^\beta, \bu^\beta, \bw^\beta)$, which
depends continuously on $\beta$. By the continuity of $\Re$, for $\beta$ close to $\gamma$, we infer
\begin{equation*}
	\Re(h_\beta,\bu_\beta,\bw_\beta) \leq 2\epsilon_1,
\end{equation*}
and \eqref{401}, \eqref{402a}, \eqref{403} hold. Due to Proposition \ref{p4}, we get
\begin{equation*}
 \Re(h_\beta,\bu_\beta,\bw_\beta) \leq \epsilon_1,
\end{equation*}
and \eqref{404}-\eqref{405}. Thus, we have showed that $\beta \in \text{A}$.

(2) $\text{A}$ is closed. Let $\gamma_k \in \text{A}, k \in \mathbb{N}^+$. Let $\gamma$ be a limit satisfying $\lim_{k \rightarrow \infty} \gamma_k = \gamma$.
Then there exists a sequence $\{(h_{\gamma_k}, \bu_{\gamma_k},\bw_{\gamma_k}) \}_{k \in \mathbb{N}^+}$ being the smooth solutions to \eqref{XT} and
\begin{align*}
	& \|(u^0_{\gamma_k}-1, \mathring{\bu}_{\gamma_k}, h_{\gamma_k})\|_{L^\infty_{[-2,2]} H_x^{s}}+ \|(\partial_t \bu_{\gamma_k}, \partial_t h_{\gamma_k})\|_{L^\infty_{[-2,2]} H_x^{s-1}}+
	+ \| \bw_{\gamma_k} \|_{L^\infty_{[-2,2]} H_x^{s_0}}
	\\
	& +\| \partial_t \bw_{\gamma_k} \|_{L^\infty_{[-2,2]} H_x^{s_0-1}}+  \| d \bu_{\gamma_k}, d h_{\gamma_k} \|_{L^2_{[-2,2]} C_x^\delta}+ \| d \bu_{\gamma_k}, dh_{\gamma_k} \|_{L^2_{[-2,2]} \dot{B}_{\infty,2}^{s_0-2}} \leq \epsilon_2.
\end{align*}
Then there exists a subsequence $\{(h_{\gamma_k}, \bu_{\gamma_k}, \bw_{\gamma_k}) \}_{k \in \mathbb{N}^+}$  such that 
\begin{equation*}
	\lim_{k \rightarrow \infty} (h_{\gamma_k}, \bu_{\gamma_k}, \bw_{\gamma_k})=(h_{\gamma}, \bu_{\gamma}, \bw_{\gamma}) .
\end{equation*}
Moreover, $(h_{\gamma}, \bu_{\gamma}, \bw_{\gamma})$ satisfies
\begin{equation*}
	\begin{split}
		&\|(u^0_\gamma-1, \mathring{\bu}_{\gamma}, h_{\gamma})\|_{L^\infty_{[-2,2]} H_x^{s}}+ \|(\partial_t \bu_{\gamma}, \partial_t h_{\gamma})\|_{L^\infty_{[-2,2]} H_x^{s-1}}+ \| \bw_{\gamma} \|_{L^\infty_{[-2,2]} H_x^{s_0}}
		\\
		& \ +\| \partial_t \bw_{\gamma} \|_{L^\infty_{[-2,2]} H_x^{s_0-1}}+  \| d \bu_{\gamma}, dh_{\gamma} \|_{L^2_{[-2,2]} C_x^\delta}+  \| d \bu_{\gamma},dh_{\gamma} \|_{L^2_{[-2,2]} \dot{B}_{\infty,2}^{s_0-2}} \leq \epsilon_2,
	\end{split}
\end{equation*}
Thus, $\Re(h_\gamma,\bu_\gamma,\bw_\gamma) \leq \epsilon_1$. Then $\gamma \in \text{A}$. At this stage, we have proved Proposition \ref{p1}.
\end{proof}
\section{Proof of Proposition \ref{p4}}\label{secp4}
In this part, our goal is to prove Proposition \ref{p4}. Proposition \ref{p4} will follow as a result of Proposition \ref{r2}, Proposition \ref{r5}, and Proposition \ref{r4}. For a start, let us introduce a foliation of space-time and a definition of functional $\Re$.

\subsection{Geometric setup}

Let the metric $ \mathbf{g}$ which equals the Minkowski metric for $t \in [-2, -\frac{3}{2}]$. Following Smith-Tataru's paper \cite{ST}, for every $\theta \in \mathbb{S}^2$, we get a foliation of the slice $t=-2$ by taking level sets of the function $r_\theta(-2,x)=\theta \cdot x +2$. Noting the fact that
$\theta\cdot dx-dt$ would then be a covector field over $t=-2$ that is conormal to the level sets of $r_\theta(-2,x)$, so we set $\Gamma_{\theta}$ to be the flowout of this section under the
Hamilton flow of $\mathbf{g}$. A key idea in the proof of Strichartz estimates is to establish that, for each $\theta$, $\Gamma_\theta$ is the graph of a null covector field given by $dr_\theta$, where $r_\theta$ is a smooth extension of $\theta \cdot x -t$, and that the level sets of $r_\theta$ are small perturbations of the level sets of the function $\theta \cdot x -t$ in a certain norm captured by $\Re$. {In establishing Proposition \ref{p4}, we establish that $(h,\bu,\bw)\in \mathcal{H}$ implies $\Gamma_\theta$ is the graph of an appropriate null
covector field $dr_\theta$, so we only define $\Re(h,\bu,\bw)$ in the situation. For the definition of $\Re$, see \eqref{500} below.}

Assume that $\Gamma_\theta$ and $r_\theta$ are as above. Let $\Sigma_{\theta,r}$ for $t\in \mathbb{R}$ denote
the level sets of $r_\theta$. Thus, the characteristic hypersurface $\Sigma_{\theta,r}$ is the flowout of
the set $\theta \cdot x=r-2$ along the null geodesic flow in the direction $\theta$ at $t=-2$.

Let us introduce an orthonormal sets of coordinates on $\mathbb{R}^3$ by setting $x_{\theta}=\theta \cdot x$. Let $x'_{\theta}$ be given orthonormal coordinates on the hyperplane prependicular to $\theta$, which then define coordinate on $\mathbb{R}^3$ by projection along $\theta$. Then $(t,x'_{\theta})$ induce the coordinates on $\Sigma_{\theta,r}$, and $\Sigma_{\theta,r}$ is given by
\begin{equation}\label{sig}
  \Sigma_{\theta,r}=\left\{ (t,x): x_{\theta}-\phi_{\theta, r}=0  \right\},
\end{equation}
for a smooth function $\phi_{\theta, r}(t,x'_{\theta})$. We now introduce two norms for functions defined on $[-2,2] \times \mathbb{R}^3$,
\begin{equation}\label{d0}
  \begin{split}
  &\vert\kern-0.25ex\vert\kern-0.25ex\vert f \vert\kern-0.25ex\vert\kern-0.25ex\vert_{s_0, \infty} = \sup_{-2 \leq t \leq 2} \sup_{0 \leq j \leq 1} \| \partial_t^j f(t,\cdot)\|_{H_x^{s_0-j}(\mathbb{R}^3)},
  \\
  & \vert\kern-0.25ex\vert\kern-0.25ex\vert  f \vert\kern-0.25ex\vert\kern-0.25ex\vert_{s_0,2} = \big( \sup_{0 \leq j \leq 1} \int^{2}_{-2} \| \partial_t^j f(t,\cdot)\|^2_{H_x^{s_0-j}(\mathbb{R}^3)} dt \big)^{\frac{1}{2}}.
  \end{split}
\end{equation}
The same notation applies for functions in $[-2,2] \times \mathbb{R}^3$. We denote
\begin{equation*}
  \vert\kern-0.25ex\vert\kern-0.25ex\vert f\vert\kern-0.25ex\vert\kern-0.25ex\vert_{s_0,2,\Sigma_{\theta,r}}=\vert\kern-0.25ex\vert\kern-0.25ex\vert f|_{\Sigma_{\theta,r}} \vert\kern-0.25ex\vert\kern-0.25ex\vert_{s_0,2},
\end{equation*}
where the right hand side is the norm of the restriction of $f$ to ${\Sigma_{\theta,r}}$, taken over the $(t,x'_{\theta})$ variables used to parametrise ${\Sigma_{\theta,r}}$. Similarly, the notation $\|f\|_{H^{s_0}(\Sigma_{\theta,r})}$ denotes the $H^{s_0}(\mathbb{R}^2)$ norm of $f$ restricted to the time $t$ slice of ${\Sigma_{\theta,r}}$ using the $x'_{\theta}$ coordinates on ${\Sigma^t_{\theta,r}}$. To avoid confusion, let $\Delta_{x'}$ be the standard Laplacian operator on $\Sigma$. We also define
\begin{equation}\label{gam}
	\Lambda_{x'}=(-\Delta_{x'})^{\frac12}.
\end{equation}
We now define
\begin{equation}\label{500}
  \Re(h,\bu, \bw)= \sup_{\theta, r} \vert\kern-0.25ex\vert\kern-0.25ex\vert d \phi_{\theta,r}-dt\vert\kern-0.25ex\vert\kern-0.25ex\vert_{s_0,2,{\Sigma_{\theta,r}}}.
\end{equation}

Due to \eqref{500}, it suffices to consider $\theta=(0,0,1)$ and $r=0$. We fix the choice, and suppress $\theta$ and $r$ in our notation. We use $(x_3, x')$ instead of $(x_{\theta}, x'_{\theta})$. Then $\Sigma$ is defined by
\begin{equation*}
	\Sigma=\left\{ x_3- \phi(t,x')=0 \right\}.
\end{equation*}
The hypothesis $\Re(h,\bu,\bw) \leq 2 \epsilon_1$ (see Proposition \ref{p4}) implies that
\begin{equation}\label{503}
	\vert\kern-0.25ex\vert\kern-0.25ex\vert d \phi_{\theta,r}(t,\cdot)-dt \vert\kern-0.25ex\vert\kern-0.25ex\vert_{s_0,2, \Sigma} \leq 2 \epsilon_1.
\end{equation}
By Sobolev imbeddings, we have
\begin{equation}\label{504}
	\|d \phi(t,x')-dt \|_{L^2_{[-2,2]} C^{1,\delta_0}_{x'}} + \| \partial_t d \phi(t,x')\|_{L^2_{[-2,2]} C^{\delta_0}_{x'}} \lesssim \epsilon_1.
\end{equation}
In the following, we will discuss energy estimates along characteristic hypersurfaces.
\subsection{Energy estimates along characteristic hypersurfaces}
In this part, our goal is to establish energy estimates for $h,\bu$ and $\bw$ along the characteristic hypersurface $\Sigma$. We first introduce several inequalities on $\Sigma$ as follows.
\subsubsection{Useful lemmas}
\begin{Lemma}[\cite{ST}, Lemma 5.5]\label{te0}
Let $\tilde{f}(t,x)=f(t,x',x_3+\phi(t,x'))$. Then we have
\begin{equation*}
  \vert\kern-0.25ex\vert\kern-0.25ex\vert \tilde{f}\vert\kern-0.25ex\vert\kern-0.25ex\vert_{s_0,\infty}\lesssim \vert\kern-0.25ex\vert\kern-0.25ex\vert f\vert\kern-0.25ex\vert\kern-0.25ex\vert_{s_0,\infty}, \quad \|d\tilde{f}\|_{L^2_{[-2,2]}L_x^\infty}\lesssim \|d f\|_{{L^2_{[-2,2]}L_x^\infty}}, \quad  \|\tilde{f}\|_{H^{s_0}_{x}}\lesssim \|f\|_{H^{s_0}_{x}}.
\end{equation*}
\end{Lemma}
\begin{Lemma}[\cite{ST}, Lemma 5.4]\label{te2}
For $r\geq 1$, we have
\begin{equation*}
\begin{split}
  \sup_{t\in[-2,2]} \| f\|_{H^{r-\frac{1}{2}}(\mathbb{R}^n)} & \lesssim \vert\kern-0.25ex\vert\kern-0.25ex\vert f \vert\kern-0.25ex\vert\kern-0.25ex\vert_{r,2},
  \\
  \sup_{t\in[-2,2]} \| f\|_{H^{r-\frac{1}{2}}(\Sigma^t)} & \lesssim \vert\kern-0.25ex\vert\kern-0.25ex\vert f \vert\kern-0.25ex\vert\kern-0.25ex\vert_{r,2,\Sigma}.
\end{split}
\end{equation*}
If $r> \frac{n+1}{2}$, then
\begin{equation*}
  \vert\kern-0.25ex\vert\kern-0.25ex\vert f_1 f_2 \vert\kern-0.25ex\vert\kern-0.25ex\vert_{r,2}\lesssim  \vert\kern-0.25ex\vert\kern-0.25ex\vert f_1 \vert\kern-0.25ex\vert\kern-0.25ex\vert_{r,2} \vert\kern-0.25ex\vert\kern-0.25ex\vert f_2 \vert\kern-0.25ex\vert\kern-0.25ex\vert_{r,2}.
\end{equation*}
Similarly, if $r>\frac{n}{2}$, then
\begin{equation*}
  \vert\kern-0.25ex\vert\kern-0.25ex\vert f_1 f_2 \vert\kern-0.25ex\vert\kern-0.25ex\vert_{r,2,\Sigma}\lesssim  \vert\kern-0.25ex\vert\kern-0.25ex\vert f_1 \vert\kern-0.25ex\vert\kern-0.25ex\vert_{r,2,\Sigma} \vert\kern-0.25ex\vert\kern-0.25ex\vert f_2 \vert\kern-0.25ex\vert\kern-0.25ex\vert_{r,2,\Sigma}.
\end{equation*}
\end{Lemma}

\begin{Lemma}[\cite{AZ}, Lemma 6.5]\label{te1}
Assume $2<s_0<s<\frac52$. Let $\bU=(p(h), {u}^1, {u}^2, {u}^3)^{\mathrm{T}}$ satisfy the symmetric hyperbolic system
\begin{equation*}\label{505}
  A^\alpha(\bU) \partial_{ \alpha } \bU= 0,
\end{equation*}
where $A^\alpha$ ($\alpha=0,1,2,3$) is defined in \eqref{A0}-\eqref{A3}.
Then we have the characteristic energy estimates
\begin{equation*}\label{te10}
\begin{split}
 \vert\kern-0.25ex\vert\kern-0.25ex\vert  \bU \vert\kern-0.25ex\vert\kern-0.25ex\vert_{s_0,2,\Sigma} & \lesssim \|d \bU \|_{L^2_{[-2,2]} L^{\infty}_x}+ \| \bU \|_{L^{\infty}_{[-2,2]}H_x^{s_0}}.
   \end{split}
\end{equation*}
\end{Lemma}
\subsubsection{Characteristic energy estimates for enthalpy and velocity}
For $(h,\bu, \bw) \in \mathcal{H}$, using \eqref{402a} and \eqref{403}, so we have
\begin{equation}\label{5021}
	\vert\kern-0.25ex\vert\kern-0.25ex\vert h,\mathring{\bu},u^0-1 \vert\kern-0.25ex\vert\kern-0.25ex\vert_{s,\infty}+\vert\kern-0.25ex\vert\kern-0.25ex\vert \bw \vert\kern-0.25ex\vert\kern-0.25ex\vert_{s_0,\infty} +\|d\bu,dh, \nabla \bu_{+}\|_{L^2_{[-2,2]}C^\delta_x}+ \|dh, \nabla \bu_{+}, d \bu\|_{L^2_{[-2,2]} \dot{B}^{s_0-2}_{\infty,2}} \lesssim \epsilon_2.
\end{equation}
\begin{corollary}\label{vte}
Suppose that $(h, \bu, \bw) \in \mathcal{H}$. Then we have
\begin{equation}\label{335}
\vert\kern-0.25ex\vert\kern-0.25ex\vert (u^0-1,\mathring{\bu}) \vert\kern-0.25ex\vert\kern-0.25ex\vert_{s_0,2,\Sigma}+ \vert\kern-0.25ex\vert\kern-0.25ex\vert h \vert\kern-0.25ex\vert\kern-0.25ex\vert_{s_0,2,\Sigma} \lesssim \epsilon_2.
\end{equation}
\end{corollary}
\begin{proof}
	By \eqref{5021} and Lemma \ref{te1}, we obtain \eqref{335}.
\end{proof}

\begin{Lemma}\label{fre}
	Let $\bU=(p(h), u^1,u^2,u^3)^{\mathrm{T}}$ satisfy \eqref{QHl}. Let $P_{\leq j}=\sum_{k=0}^{j} P_k$. Then
	\begin{equation*}\label{508}
		\vert\kern-0.25ex\vert\kern-0.25ex\vert  2^j(\bU-P_{\leq j} \bU), d P_{\leq j} \bU, 2^{-j} d \nabla P_{\leq j} \bU \vert\kern-0.25ex\vert\kern-0.25ex\vert_{s_0-1,2,\Sigma} \lesssim \epsilon_2.
	\end{equation*}
\end{Lemma}
\begin{proof}
	Let $\Delta_0$ be a standard multiplier of order $0$ on $\mathbb{R}^3$, such that $\Delta_0$ is additionally bounded on $L^\infty(\mathbb{R}^3)$. Clearly,
	\begin{equation*}
		A^0(\bU)(\Delta_0U)_t+ \sum^3_{i=1}A^i(\bU)(\Delta_0 \bU)_{x_i}= -[\Delta_0, A^\alpha(\bU)]\partial_{x_\alpha}\bU.
	\end{equation*}
	Due to Lemma \ref{te1}, we have
	\begin{equation}\label{60}
		\vert\kern-0.25ex\vert\kern-0.25ex\vert \Delta_0U\vert\kern-0.25ex\vert\kern-0.25ex\vert^2_{s_0,2,\Sigma} \lesssim \|d \bU \|_{L^2_t L^{\infty}_x}\| \Delta_0 \bU\|^2_{L^{\infty}_tH_x^{s}}+\| [\Delta_0, A^\alpha(\bU)]\partial_{x_\alpha}\bU \|_{L^{2}_tH_x^{s_0}}\| \Delta_0 \bU\|_{L^{\infty}_tH_x^{s_0}}.
	\end{equation}
	By commutator estimates, we obtain
	\begin{equation*}
		\| [\Delta_0, A^\alpha(\bU)]\partial_{x_\alpha}\bU \|_{L^{2}_tH_x^{s_0}} \lesssim \|d \bU \|_{L^2_t L^{\infty}_x}\| \Delta_0 \bU\|_{L^{\infty}_tH_x^{s_0}}.
	\end{equation*}
	Hence, \eqref{60} yields
	\begin{equation}\label{60e}
		\vert\kern-0.25ex\vert\kern-0.25ex\vert \Delta_0 \bU \vert\kern-0.25ex\vert\kern-0.25ex\vert^2_{s_0,2,\Sigma} \lesssim \|d \bU \|_{L^2_t L^{\infty}_x}\| \Delta_0 \bU\|^2_{L^{\infty}_tH_x^{s_0}}.
	\end{equation}
	To control the norm of $2^j(\bU-P_{\leq j} \bU)$, we write
	\begin{equation*}
		2^j(\bU-P_{\leq j} \bU)= 2^j \sum_{m\geq j}\Delta_m \bU,
	\end{equation*}
	where $\Delta_m \bU$ satisfies the above conditions for $\Delta_0 \bU$. Applying \eqref{60e} and replacing $s_0$ to $s_0-1$, we get
	\begin{equation*}
		\begin{split}
			\vert\kern-0.25ex\vert\kern-0.25ex\vert 2^j(\bU-P_{\leq j} \bU) \vert\kern-0.25ex\vert\kern-0.25ex\vert^2_{s_0-1,2,\Sigma}
			\lesssim & \sum_{m\geq j} \vert\kern-0.25ex\vert\kern-0.25ex\vert 2^m \Delta_m \bU \vert\kern-0.25ex\vert\kern-0.25ex\vert^2_{s_0-1,2,\Sigma}
			\\
			\lesssim & \|d \bU \|_{L^2_t L^{\infty}_x} \sum_{m\geq j} (2^m\|\Delta_m \bU\|^2_{L^{\infty}_tH_x^{s_0}})
			\\
			\lesssim & \|d \bU \|_{L^2_t L^{\infty}_x} \| \bU \|^2_{L^{\infty}_tH_x^{s_0}} \lesssim \epsilon^2_2.
		\end{split}
	\end{equation*}
	Taking square of it, we can see
	\begin{equation*}
		\vert\kern-0.25ex\vert\kern-0.25ex\vert 2^j(\bU-P_{\leq j} \bU) \vert\kern-0.25ex\vert\kern-0.25ex\vert_{s_0-1,2,\Sigma}
		\lesssim  \epsilon_2.
	\end{equation*}
	Finally, applying \eqref{60} to $\Delta_0=P_{\leq j}$ and $\Delta_0=2^{-j}\nabla P_{\leq j}$ shows that
	\begin{equation*}
		\vert\kern-0.25ex\vert\kern-0.25ex\vert d P_{\leq j}\bU\vert\kern-0.25ex\vert\kern-0.25ex\vert_{s_0-1,2,\Sigma} +\vert\kern-0.25ex\vert\kern-0.25ex\vert 2^{-j} d \nabla P_{\leq j}\bU \vert\kern-0.25ex\vert\kern-0.25ex\vert_{s_0-1,2,\Sigma} \lesssim \epsilon_2.
	\end{equation*}
	Therefore, the proof of Lemma \ref{fre} is completed.
\end{proof}
Based on the above lemmas, we are ready to prove the following result.
\begin{proposition}\label{r1}
	Let $(h, \bu, \bw) \in \mathcal{H}$ so that $\Re(h,\bu, \bw) \leq 2 \epsilon_1$. Then
	\begin{equation}\label{501}
		\vert\kern-0.25ex\vert\kern-0.25ex\vert  {\mathbf{g}}^{\alpha \beta}-m^{\alpha \beta} \vert\kern-0.25ex\vert\kern-0.25ex\vert_{s_0,2,{\Sigma_{\theta,r}}} + \vert\kern-0.25ex\vert\kern-0.25ex\vert 2^{j}({\mathbf{g}}^{\alpha \beta}-S_j {\mathbf{g}}^{\alpha \beta}), d S_j {\mathbf{g}}^{\alpha \beta}, 2^{-j} \nabla S_j d {\mathbf{g}}^{\alpha \beta}  \vert\kern-0.25ex\vert\kern-0.25ex\vert_{s_0-1,2,{\Sigma_{\theta,r}}} \lesssim \epsilon_2.
	\end{equation}
\end{proposition}
\begin{proof}
	Note $(h, \bu, \bw) \in \mathcal{H}$. By Lemma \ref{fre}, it is sufficient for us to verify that
	\begin{equation*}
		\begin{split}
			\vert\kern-0.25ex\vert\kern-0.25ex\vert {\mathbf{g}}^{\alpha \beta}-\mathbf{m}^{\alpha \beta}\vert\kern-0.25ex\vert\kern-0.25ex\vert_{s_0,2,\Sigma_{\theta,r}} \lesssim \epsilon_2.
		\end{split}
	\end{equation*}
	By Lemma \ref{te1}, \eqref{504}, and \eqref{5021}, we have
	\begin{equation*}
		\sup_{\theta,r}\vert\kern-0.25ex\vert\kern-0.25ex\vert (u^0-1,\mathring{\bu}) \vert\kern-0.25ex\vert\kern-0.25ex\vert_{s_0,2,\Sigma_{\theta,r}}+ \sup_{\theta,r}\vert\kern-0.25ex\vert\kern-0.25ex\vert h \vert\kern-0.25ex\vert\kern-0.25ex\vert_{s_0,2,\Sigma_{\theta,r}} \lesssim \epsilon_2.
	\end{equation*}
	Noting \eqref{AMd3}, and $\mathbf{g}^{00}=-1$, by Lemma \ref{te2}, and Corollary \ref{vte} we derive that
	\begin{equation*}
		\begin{split}
			\vert\kern-0.25ex\vert\kern-0.25ex\vert  {\mathbf{g}}^{\alpha \beta}-\mathbf{m}^{\alpha \beta}\vert\kern-0.25ex\vert\kern-0.25ex\vert _{s_0,2,\Sigma_{\theta,r}} & \lesssim \vert\kern-0.25ex\vert\kern-0.25ex\vert \bu \cdot \mathring{\bu} \vert\kern-0.25ex\vert\kern-0.25ex\vert_{s_0,2,\Sigma_{\theta,r}}+\vert\kern-0.25ex\vert\kern-0.25ex\vert c_s^2(h)-c_s^2(0)\vert\kern-0.25ex\vert\kern-0.25ex\vert_{s_0,2,\Sigma_{\theta,r}}
			\\
			& \lesssim \vert\kern-0.25ex\vert\kern-0.25ex\vert \mathring{\bu} \vert\kern-0.25ex\vert\kern-0.25ex\vert_{s_0,2,\Sigma_{\theta,r}}+ \vert\kern-0.25ex\vert\kern-0.25ex\vert (u^0-1,  \mathring{\bu}) \cdot \mathring{\bu} \vert\kern-0.25ex\vert\kern-0.25ex\vert_{s_0,2,\Sigma_{\theta,r}}
			+\vert\kern-0.25ex\vert\kern-0.25ex\vert h \vert\kern-0.25ex\vert\kern-0.25ex\vert_{s_0,2,\Sigma_{\theta,r}}
			\\
			& \lesssim \epsilon_2.
		\end{split}
	\end{equation*}
	Therefore, we have proved Proposition \ref{r1}.
\end{proof}
\subsubsection{Characteristic energy estimates for vorticity} We first introduce a lemma for transport equations.
\begin{Lemma} [\cite{AZ}, Lemma 6.7]\label{te3}
	Suppose that $(h, \bu, \bw) \in \mathcal{H}$. Let $f$ satisfy the following transport equation
	\begin{equation}\label{333}
		\mathbf{T} f=F,
	\end{equation}
	where $\mathbf{T}$ is defined in \eqref{opt}. Set $\mathrm{L}= \nabla(-\Delta)^{-1}\mathrm{curl}$. Then the following estimate holds
	\begin{equation}\label{teE}
		\begin{split}
			\vert\kern-0.25ex\vert\kern-0.25ex\vert  \mathrm{L} f\vert\kern-0.25ex\vert\kern-0.25ex\vert^2_{s_0-2,2,\Sigma} \lesssim & \ \big( \| \nabla \bu\|_{L^2_{[-2,2]}\dot{B}^{s_0-2}_{\infty,2}}+\vert\kern-0.25ex\vert\kern-0.25ex\vert d\phi-dt \vert\kern-0.25ex\vert\kern-0.25ex\vert_{s_0,2,\Sigma} \big) \|f\|^2_{{H}^{s_0-2}_x}(1+\vert\kern-0.25ex\vert\kern-0.25ex\vert d\phi-dt \vert\kern-0.25ex\vert\kern-0.25ex\vert_{s_0,2,\Sigma})
			\\
			& + \sum_{\sigma\in\{0,s_0-2 \}}\left|\int^2_{-2} \int_{\mathbb{R}^3} \Lambda^{\sigma}_{x'}{\mathrm{L} F} \cdot \Lambda^{\sigma}_{x'}\mathrm{L} fdxd\tau \right|.
		\end{split}
	\end{equation}
	Here, $\Lambda_{x'}$ is defined in \eqref{gam}.
\end{Lemma}
We next give the characteristic energy estimates for vorticity as follows.

\begin{Lemma}\label{te20}
	Suppose that $(h, \bu, \bw) \in \mathcal{H}$. Let $\bW$ be defined in \eqref{MFd}. Then
	\begin{equation}\label{te201}
		\begin{split}
			\vert\kern-0.25ex\vert\kern-0.25ex\vert \bW  \vert\kern-0.25ex\vert\kern-0.25ex\vert_{s_0-1,2,\Sigma} \lesssim  \epsilon_2.
		\end{split}
	\end{equation}
\end{Lemma}
\begin{proof}
	Due to \eqref{d0}, we have
	\begin{equation}\label{Cd}
		\begin{split}
			\vert\kern-0.25ex\vert\kern-0.25ex\vert \mathbf{W} \vert\kern-0.25ex\vert\kern-0.25ex\vert_{s_0-1,2,\Sigma}= &\| \mathbf{W} \|_{L^2_t H^{s_0-1}_{x'}}+ \| \partial_t \mathbf{W} \|_{L^2_t H^{s_0-2}_{x'}}
			\\
			=	& \| \mathbf{W} \|_{L^2_t L^2_{x'}(\Sigma)}+\| \Lambda^{s_0-1}_{x'} \mathbf{W} \|_{L^2_t L^2_{x'}(\Sigma)}+ \| \partial_t \mathbf{W} \|_{L^2_t H^{s_0-2}_{x'}}.
		\end{split}
	\end{equation}

	\textit{The bound for $\| \mathbf{W} \|_{L^2_t L^2_{x'}(\Sigma)}$}. For \eqref{CEQ1}, by changing of coordinates $x_3 \rightarrow x_3-\phi(t,x')$, multiplying $\mathbf{W}$ and integrating \eqref{CEQ} on $[-2,2]\times \mathbb{R}^3$, we infer
	\begin{equation}\label{999}
		\begin{split}
			\| \mathbf{W} \|^2_{L^2_t L^2_{x'}(\Sigma)} \lesssim & \ \|\nabla \bu\|_{L^1_tL^\infty_x} ( \|\nabla \bw\|_{L^\infty_t L^2_x}+\| \mathbf{W}\|_{L^\infty_t L^2_x}) \| \mathbf{W} \|_{L^\infty_t L^2_x}.
		\end{split}
	\end{equation}
	Taking square of \eqref{999}, and using \eqref{402}-\eqref{403}, we then derive
	\begin{equation}\label{OM0}
		\| \mathbf{W} \|_{L^2_t L^2_{x'}(\Sigma)} \lesssim \epsilon_2.
	\end{equation}

	\textit{The bound for $\| \Lambda^{s_0-1}_{x'} \mathbf{W} \|_{L^2_t L^2_{x'}(\Sigma)}$}. To bound $\| \Lambda^{s_0-1}_{x'} \mathbf{W} \|_{L^2_t L^2_{x'}(\Sigma)}$, so consider $\|  \nabla \mathbf{\mathring{\bW}} \|_{L^2_t H^{s_0-2}_{x'}(\Sigma)}$ and $\|  \nabla W^0 \|_{L^2_t H^{s_0-2}_{x'}(\Sigma)}$. By Hodge decomposition, we have
	\begin{equation}\label{OM1}
		\begin{split}
			\|  \nabla \mathbf{\mathring{\bW}} \|_{L^2_t H^{s_0-2}_{x'}(\Sigma)} &= \| \mathrm{L} ( \mathrm{curl} \mathbf{\mathring{\bW}})+\mathrm{H} ( \mathrm{div} \mathbf{\mathring{\bW}})\|_{L^2_t H^{s_0-2}_{x'}(\Sigma)}
			\\
			& \leq \| \mathrm{L} ( \mathrm{curl} \mathbf{\mathring{\bW}})\|_{L^2_t H^{s_0-2}_{x'}(\Sigma)}+\| \mathrm{H} ( \mathrm{div} \mathbf{\mathring{\bW}})\|_{L^2_t H^{s_0-2}_{x'}(\Sigma)},
		\end{split}
	\end{equation}
	where the operators $\mathrm{L}$ and $\mathrm{H}$ are given by
	\begin{equation}\label{LH}
		\mathrm{L} = \nabla(-\Delta)^{-1}\mathrm{curl}, \quad    \mathrm{H} = (-\Delta)^{-1}\nabla^2.
	\end{equation}
	For $\mathrm{div} \mathbf{\mathring{\bW}}=\partial_i W^i=\partial_\alpha W^\alpha- \partial_t W^0$, we can carry out
	\begin{equation}\label{OM5}
		\begin{split}
			\| \mathrm{H} ( \mathrm{div} \mathbf{\mathring{\bW}})\|_{L^2_t H^{s_0-2}_{x'}(\Sigma)}\leq & \| \mathrm{H} ( \partial_\alpha W^\alpha) \|_{L^2_t H^{s_0-2}_{x'}(\Sigma)}+ \| \mathrm{H} ( \partial_t W^0)\|_{L^2_t H^{s_0-2}_{x'}(\Sigma)}.
		\end{split}
	\end{equation}
	By \eqref{d24}, trace theorem, and H\"older's inequality, it follows
	\begin{equation}\label{OM6}
		\begin{split}
			\| \mathrm{H} ( \partial_\alpha W^\alpha) \|_{L^2_t H^{s_0-2}_{x'}(\Sigma)}
			\lesssim & \|  ( d\bu,dh)\cdot d\bw \|_{L^2_t H^{s_0-2}_{x'}(\Sigma)}+\| \bw
			\cdot d\bu \cdot dh \|_{L^2_t H^{s_0-2}_{x'}(\Sigma)}
			\\
			\lesssim & \|  ( d\bu,dh)\cdot d\bw \|_{L^2_t H^{s_0-3/2}_{x}}+\| \bw
			\cdot d\bu \cdot dh \|_{L^2_t H^{s_0-3/2}_{x}}
			\\
			\lesssim & \|  ( d\bu,dh)\|_{L^\infty_t H^{s-1}_{x}} ( \|d\bw \|_{L^\infty_t H^{s_0-1}_{x}}+ \| \bw\|_{L^\infty_t H^{s_0}_{x}}
			\| d\bu \|_{L^\infty_t H^{s-1}_{x}} ).
		\end{split}
	\end{equation}
	Due to \eqref{d21} and \eqref{d22}, using trace theorem and H\"older's inequality, we have
	\begin{equation}\label{OM7}
		\begin{split}
			\| \mathrm{H} ( \partial_t W^0) \|_{L^2_t H^{s_0-2}_{x'}(\Sigma)} \leq & \| \mathring{\bW} \cdot d\bu\|_{L^2_t H^{s_0-2}_{x'}(\Sigma)}
			+  \| (u^0)^{-1} \mathring{\bu} \cdot \partial_t \mathring{\bW} \|_{L^2_t H^{s_0-2}_{x'}(\Sigma)}
			\\
			\leq & C\| \mathring{\bW} \cdot d\bu\|_{L^2_t H^{s_0-2}_{x'}(\Sigma)}
			+  |(u^0)^{-1} \mathring{\bu}|\cdot \| \partial_t \mathring{\bW} \|_{L^2_t H^{s_0-2}_{x'}(\Sigma)}
			\\
			\leq & |(u^0)^{-1} \mathring{\bu}|^2 \cdot \| \nabla \mathring{\bW} \|_{L^2_t H^{s_0-2}_{x'}(\Sigma)}+ C\| {\bW} \cdot (d\bu,dh) \|_{L^2_t H^{s_0-2}_{x'}(\Sigma)}
			\\
			& + C\|d{\bw} \cdot (d\bu,dh) \|_{L^2_t H^{s_0-2}_{x'}(\Sigma)} + C\| {\bw} \cdot d\bu \cdot dh\|_{L^2_t H^{s_0-2}_{x'}(\Sigma)}
			\\
			\leq  & (\frac{|\mathring{\bu}|}{|u^0|})^2 \cdot \| \nabla \mathring{\bW} \|_{L^2_t H^{s_0-2}_{x'}(\Sigma)}+ C\|  ( d\bu,dh)\|_{L^\infty_t H^{s-1}_{x}} \|d\bw \|_{L^\infty_t H^{s_0-1}_{x}}
			\\
			& +  C\|  ( d\bu,dh)\|_{L^\infty_t H^{s-1}_{x}} \| \bw\|_{L^\infty_t H^{s_0}_{x}}
			\| d\bu \|_{L^\infty_t H^{s-1}_{x}} .
		\end{split}
	\end{equation}
	Gathering \eqref{OM1}, \eqref{OM5}, \eqref{OM6}, and \eqref{OM7}, we can prove
	\begin{equation*}
		\begin{split}
			(1- \frac{|\mathring{\bu}|^2}{|u^0|^2} )	\| \nabla \mathbf{\mathring{\bW}}  \|_{L^2_t H^{s_0-2}_{x'}(\Sigma)} \leq &  \| \mathrm{L} ( \mathrm{curl} \mathbf{\mathring{\bW}})\|_{L^2_t H^{s_0-2}_{x'}(\Sigma)}+ C\|  ( d\bu,dh)\|_{L^\infty_t H^{s-1}_{x}} \|d\bw \|_{L^\infty_t H^{s_0-1}_{x}}
			\\
			& +  C\|  ( d\bu,dh)\|_{L^\infty_t H^{s-1}_{x}} \| \bw\|_{L^\infty_t H^{s_0}_{x}}
			\| d\bu \|_{L^\infty_t H^{s-1}_{x}} .
		\end{split}
	\end{equation*}
	Note $1\leq |u^0|\leq 1+C_0$. By using \eqref{muu}, \eqref{QHl}, \eqref{CEQ}, and \eqref{5021}, we therefore get
	\begin{equation}\label{OM8}
		\begin{split}
			\| \nabla \mathbf{\mathring{\bW}}  \|_{L^2_t H^{s_0-2}_{x'}(\Sigma)} \lesssim &  \| \mathrm{L} ( \mathrm{curl} \mathbf{\mathring{\bW}})\|_{L^2_t H^{s_0-2}_{x'}(\Sigma)}+ \|  ( d\bu,dh)\|_{L^\infty_t H^{s-1}_{x}} \|d\bw \|_{L^\infty_t H^{s_0-1}_{x}}
			\\
			& +  \|  ( d\bu,dh)\|_{L^\infty_t H^{s-1}_{x}} \| \bw\|_{L^\infty_t H^{s_0}_{x}}
			\| d\bu \|_{L^\infty_t H^{s-1}_{x}}
			\\
			\lesssim &  \| \mathrm{L} ( \mathrm{curl} \mathbf{\mathring{\bW}})\|_{L^2_t H^{s_0-2}_{x'}(\Sigma)}+ \epsilon_2^2.
		\end{split}
	\end{equation}
	By \eqref{d20}, we can obtain
	\begin{equation}\label{OM2}
		\begin{split}
			\|  \nabla W^0 \|_{L^2_t H^{s_0-2}_{x'}(\Sigma)} &\lesssim \|  \nabla \mathbf{\mathring{\bW}} \|_{L^2_t H^{s_0-2}_{x'}(\Sigma)}
			+ \|  \bW d\bu \|_{L^2_t H^{s_0-2}_{x'}(\Sigma)}.
		\end{split}
	\end{equation}
	By trace theorem, \eqref{MFd}, and \eqref{5021}, we have
	\begin{equation}\label{OM3}
		\begin{split}
			\|  \bW d\bu \|_{L^2_t H^{s_0-2}_{x'}(\Sigma)} \lesssim & \|  \bW d\bu \|_{L^2_t H^{s_0-1/2}_{x}}
			\\
			\lesssim & \| d\bw\|_{L^\infty_t H^{s_0-1}_x} \|d\bu \|_{L^{\infty}_t H^{s-1}_{x}}+ \| \bw\|_{L^\infty_t H^{s_0}_x}\|d\bu \|_{L^{\infty}_t H^{s-1}_{x}} \|dh \|_{L^{\infty}_t H^{s-1}_{x}}
			\\
			\lesssim & \epsilon^2_2.
		\end{split}
	\end{equation}
	Combining \eqref{OM8}, \eqref{OM2}, and \eqref{OM3}, we can prove that
	\begin{equation}\label{OM4}
		\begin{split}
			\|  \nabla \mathbf{{\bW}} \|_{L^2_t H^{s_0-2}_{x'}(\Sigma)} \lesssim \| \mathrm{L} ( \mathrm{curl} \mathbf{\mathring{\bW}})\|_{L^2_t H^{s_0-2}_{x'}(\Sigma)}+\epsilon^2_2 .
		\end{split}
	\end{equation}
	In \eqref{OM4}, we still need to bound $\|\Lambda^{s_0-2}_{x'}\mathrm{L} ( \mathrm{curl} \mathbf{\mathring{\bW}})\|_{L^2_t L^2_{x'}(\Sigma)}$. By \eqref{c2}, we have
	\begin{equation}\label{L0}
		\begin{split}
			& \|\Lambda^{s_0-2}_{x'}\mathrm{L} ( \mathrm{curl} \mathbf{\mathring{\bW}})\|_{L^2_t L^2_{x'}(\Sigma)}
			\\
			\lesssim & \|\Lambda^{s_0-2}_{x'}\mathrm{L} ( \bu\cdot {\bG})\|_{L^2_t L^2_{x'}(\Sigma)}
			+ \|\Lambda^{s_0-2}_{x'}\mathrm{L} ( \bW \cdot d\bu)\|_{L^2_t L^2_{x'}(\Sigma)}
			+ \|\Lambda^{s_0-2}_{x'}\mathrm{L} ( \bw \cdot d\bu \cdot dh)\|_{L^2_t L^2_{x'}(\Sigma)}
			\\
			\lesssim & \| \bu \cdot \Lambda^{s_0-2}_{x'}\mathrm{L} \bG\|_{L^2_t L^2_{x'}(\Sigma)}
			+ \| [\Lambda^{s_0-2}_{x'}\mathrm{L},\bu ]  \bG\|_{L^2_t L^2_{x'}(\Sigma)}
			+ \|\Lambda^{s_0-2}_{x'}\mathrm{L} ( \bW \cdot d\bu)\|_{L^2_t L^2_{x'}(\Sigma)}
			\\
			& + \|\Lambda^{s_0-2}_{x'}\mathrm{L} ( \bw \cdot d\bu \cdot dh)\|_{L^2_t L^2_{x'}(\Sigma)}.
		\end{split}
	\end{equation}
	By Sobolev imbeddings, trace theorem and \eqref{5021}, we can obtain
	\begin{equation}\label{L1}
		\begin{split}
			\| [\Lambda^{s_0-2}_{x'}\mathrm{L}, \bu ]  \bG\|_{L^2_t L^2_{x'}(\Sigma)} \lesssim & \|\bu \|_{L^\infty_t C_{x'}^{s_0-2}(\Sigma)} \|\bG\|_{L^2_t L^2_{x'}(\Sigma)}
			\\
			\lesssim & (\|u^0-1, \mathring{\bu}\|_{L^\infty_t H_{x'}^{s_0-1/2}(\Sigma)}+1)\|\bG\|_{L^2_t L^2_{x'}(\Sigma)}
			\\
			\lesssim & (\|u^0-1, \mathring{\bu}\|_{L^\infty_t H_{x}^{s}}+1)\|\bG\|_{L^2_t L^2_{x'}(\Sigma)}
			\\
			\lesssim & \|\bG\|_{L^2_t L^2_{x'}(\Sigma)}.
		\end{split}
	\end{equation}
	Due to H\"older's inequality, Sobolev imbeddings, trace theorem  and \eqref{5021}, we have
	\begin{equation}\label{L2}
		\begin{split}
			\| \bu \cdot \Lambda^{s_0-2}_{x'}\mathrm{L} \bG\|_{L^2_t L^2_{x'}(\Sigma)}
			\lesssim & \| \bu \|_{L^\infty_t L^\infty_{x'}(\Sigma)} \| \Lambda^{s_0-2}_{x'}\mathrm{L} \bG\|_{L^2_t L^2_{x'}(\Sigma)}
			\\
			\lesssim & (\|u^0-1, \mathring{\bu}\|_{L^\infty_t H_{x}^{s}}+1)\| \Lambda^{s_0-2}_{x'}\mathrm{L} \bG\|_{L^2_t L^2_{x'}(\Sigma)}
			\\
			\lesssim &\| \Lambda^{s_0-2}_{x'}\mathrm{L} \bG\|_{L^2_t L^2_{x'}(\Sigma)}.
		\end{split}
	\end{equation}
	Due to trace theorem and \eqref{5021}, we have
	\begin{equation}\label{L3}
		\begin{split}
			\| \Lambda^{s_0-2}_{x'}\mathrm{L}\left( \bW \cdot d\bu \right)\|_{L^2_t L^2_{x'}(\Sigma)} \lesssim & \| \mathrm{L}\left( \bW \cdot d\bu \right)\|_{L^2_t H^{s-3/2}_{x}}
			\\
			\lesssim & \| \bW \cdot d\bu \|_{L^\infty_t H^{s-3/2}_{x}}
			\\
			\lesssim & \|  d\bu \|_{L^\infty_tH^{s-1}_x} \| \bW \|_{L^\infty_tH^{1}_{x}}
			\\
			\lesssim & \epsilon^2_2.
		\end{split}
	\end{equation}
	Similarly, we can obtain
	\begin{equation}\label{L4}
		\begin{split}
			\| \Lambda^{s_0-2}_{x'}\mathrm{L}\left( ( \bw \cdot d\bu \cdot dh) \right)\|_{L^2_t L^2_{x'}(\Sigma)} \lesssim & \| \mathrm{L}\left(  \bw \cdot d\bu \cdot dh \right)\|_{L^2_t H^{s-3/2}_{x}}
			\\
			\lesssim & \| \bw \cdot d\bu \cdot dh \|_{L^\infty_t H^{s-3/2}_{x}}
			\\
			\lesssim & \|  d\bu \|_{L^\infty_tH^{s-1}_x} \|  d h \|_{L^\infty_tH^{s-1}_x} \| \bw \|_{L^\infty_tH^{s_0}_{x}}
			\\
			\lesssim & \epsilon^3_2.
		\end{split}
	\end{equation}
	To summarize from \eqref{OM4}, \eqref{L0}, \eqref{L1}, \eqref{L2}, \eqref{L3}, and \eqref{L4}, we get
	\begin{equation}\label{L6}
		\begin{split}
			\|\nabla \bW\|_{L^2_t H^{s_0-2}_{x'}(\Sigma)}
			\lesssim & \|\bG\|_{L^2_t L^2_{x'}(\Sigma)} +\| \Lambda^{s_0-2}_{x'}\mathrm{L} \bG\|_{L^2_t L^2_{x'}(\Sigma)}+ \epsilon_2^2.
		\end{split}
	\end{equation}

	\textit{The bound for $\|\bG\|_{L^2_t L^2_{x'}(\Sigma)}$.} By \eqref{SDe}, we can see
	\begin{equation}\label{L7}
		\begin{split}
			\mathbf{T} \left(G^\alpha-F^\alpha \right)
			=& K^\alpha,
		\end{split}
	\end{equation}
	where
	\begin{equation}\label{Kalpha}
		K^\alpha= (u^0)^{-1} E^\alpha
		+  \partial^\alpha \left( (u^0)^{-1} \Gamma \right)
		-\Gamma  \partial^\alpha \left( (u^0)^{-1}\right) .
	\end{equation}
	By changing of coordinates $x_3 \rightarrow x_3-\phi(t,x')$, then \eqref{L7} becomes to
	\begin{equation}\label{L8}
		\begin{split}
			(\partial_t+ \partial_t \phi \partial_{3}) (\widetilde{G}^\alpha-\widetilde{F}^\alpha )+ \frac{\widetilde{u}^i}{\widetilde{u}^0} \cdot (\partial_{i}+\partial_{i} \phi \partial_{3} )  (\widetilde{G}^\alpha-\widetilde{F}^\alpha )
			=& \widetilde{K}^\alpha,
		\end{split}
	\end{equation}
	where
	\begin{equation*}
		\begin{split}
			\widetilde{K}^\alpha=&  (\partial^\alpha+ \partial^\alpha \phi \partial_3) \left( (u^0)^{-1} \Gamma \right)+ (\widetilde{u}^0)^{-1} \widetilde{E}^\alpha
			-\widetilde{\Gamma}  (\partial^\alpha+ \partial^\alpha \phi \partial_3) \left( (\widetilde{u}^0)^{-1}\right) .
		\end{split}
	\end{equation*}
	Multiplying $(\widetilde{G}_\alpha-\widetilde{F}_\alpha )$ on \eqref{L8} and integrating it on $[-2,2]\times \mathbb{R}^3$, we therefore get
	\begin{equation}\label{L9}
		\begin{split}
			\vert\kern-0.25ex\vert\kern-0.25ex\vert \bG-\bF \vert\kern-0.25ex\vert\kern-0.25ex\vert^2_{0,2,\Sigma} =	
			\mathrm{J_1}+\mathrm{J_2}+\mathrm{J_3}+\mathrm{J_4}+\mathrm{J_5},
		\end{split}
	\end{equation}
	where $\mathrm{J_1}, \mathrm{J_2}, \cdots, \mathrm{J_5}$ are given by
	\begin{equation*}\label{L10}
		\begin{split}
			\mathrm{J_1}=&- \int^2_{-2} \int_{\mathbb{R}^3}\partial_t \phi \partial_{3} (\widetilde{G}^\alpha-\widetilde{F}^\alpha ) \cdot (\widetilde{G}_\alpha-\widetilde{F}_\alpha ) dx d\tau ,
			\\
			\mathrm{J_2}=& -\int^2_{-2} \int_{\mathbb{R}^3} \frac{\widetilde{u}^i}{\widetilde{u}^0} \cdot (\partial_{i}+\partial_{i} \phi \partial_{3} )  (\widetilde{G}^\alpha-\widetilde{F}^\alpha ) \cdot (\widetilde{G}_\alpha-\widetilde{F}_\alpha ) dx d\tau,
			\\
			\mathrm{J_3}=& \int^2_{-2} \int_{\mathbb{R}^3} (\widetilde{u}^0)^{-1} \widetilde{E}^\alpha  \cdot (\widetilde{G}_\alpha-\widetilde{F}_\alpha ) dx d\tau,
			\\
			\mathrm{J_4}=& \int^2_{-2} \int_{\mathbb{R}^3} (\partial^\alpha+ \partial^\alpha \phi \partial_3) \{   (\widetilde{u}^0)^{-1} \widetilde{\Gamma} \} \cdot (\widetilde{G}_\alpha-\widetilde{F}_\alpha ) dx d\tau,
			\\
			\mathrm{J_5}=& -\int^2_{-2} \int_{\mathbb{R}^3} \widetilde{\Gamma}  (\partial^\alpha+ \partial^\alpha \phi \partial_3) \left( (\widetilde{u}^0)^{-1}\right)  \cdot (\widetilde{G}_\alpha-\widetilde{F}_\alpha ) dx d\tau.
		\end{split}
	\end{equation*}
	For $\phi$ is independent of $x_3$, if we integrate $\mathrm{J_1}$ by parts, then we have
	\begin{equation}\label{J1}
		\mathrm{J_1}=0.
	\end{equation}
	Integrating $\mathrm{J_2}$ by parts, and using $\phi$ independent of $x_3$, we can obtain
	\begin{equation*}
		\begin{split}
			\mathrm{J_2}=& \int^2_{-2} \int_{\mathbb{R}^3} \partial_{i} (\frac{\widetilde{u}^i}{\widetilde{u}^0} ) (\widetilde{G}^\alpha-\widetilde{F}^\alpha ) \cdot (\widetilde{G}_\alpha-\widetilde{F}_\alpha ) dx d\tau
			\\
			&+
			\int^2_{-2} \int_{\mathbb{R}^3} \partial_{i} \phi \partial_{3}(\frac{\widetilde{u}^i}{\widetilde{u}^0} )    (\widetilde{G}^\alpha-\widetilde{F}^\alpha ) \cdot (\widetilde{G}_\alpha-\widetilde{F}_\alpha ) dx d\tau.
		\end{split}
	\end{equation*}
	Using H\"older's inequality, it follows
	\begin{equation}\label{J2e}
		\begin{split}
			| \mathrm{J_2} | \lesssim & \|\nabla \bu\|_{L^2_t L^\infty_x}(1+ \vert\kern-0.25ex\vert\kern-0.25ex\vert d\phi-dt \vert\kern-0.25ex\vert\kern-0.25ex\vert_{s_0,2,\Sigma} ) \|\bG-\bF\|^2_{L^\infty_t L^2_x}.
		\end{split}
	\end{equation}
	Inserting \eqref{Wd23}, \eqref{Wd29} to \eqref{J2e}, and using \eqref{5021}, it yields
	\begin{equation}\label{J2}
		\begin{split}
			| \mathrm{J_2} |
			\lesssim & \|d\bu\|_{L^2_t L^\infty_x} ( \|(d\bu,dh)\|^2_{L^\infty_t H^{s-1}_x} + \| \bw\|^2_{L^\infty_t H^{2}_x})(1 +\|(d\bu,dh)\|^2_{L^\infty_t H^{s-1}_x}   )
			\\
			&+ \|d\bu\|_{L^2_t L^\infty_x} ( \|(d\bu,dh)\|^2_{L^\infty_t H^{s-1}_x} + \| \bw\|^2_{L^\infty_t H^{2}_x})\|(d\bu,dh)\|^3_{L^\infty_t H^{s-1}_x}
			\\
			& + ( \|(d\bu,dh)\|^2_{L^\infty_t H^{s-1}_x} + \| \bw\|^2_{L^\infty_t H^{2}_x})(\|(d\bu,dh)\|_{L^\infty_t H^{s-1}_x} +\|(d\bu,dh)\|^2_{L^\infty_t H^{s-1}_x}  )
			\\
			& + ( \|(d\bu,dh)\|^2_{L^\infty_t H^{s-1}_x} + \| \bw\|^2_{L^\infty_t H^{2}_x})(\|(d\bu,dh)\|^3_{L^\infty_t H^{s-1}_x} +\|(d\bu,dh)\|^4_{L^\infty_t H^{s-1}_x}  )
			\\
			\lesssim & \epsilon^2_2.
		\end{split}
	\end{equation}
	By \eqref{Wd23}, \eqref{Wd24}, \eqref{Wd29}, \eqref{5021} and using H\"older's inequality again, we have
	\begin{equation}\label{J4}
		\begin{split}
			& | \mathrm{J_3} |+| \mathrm{J_5} |
			\\
			\lesssim & \int^2_{-2} \|\bE\|_{L^2_x}\|\bG-\bF\|_{L^2_x}d\tau+ \int^2_{-2} \|\Gamma \cdot d\bu\|_{L^2_x}\|\bG-\bF\|_{L^2_x}d\tau
			\\
			\lesssim & \|d\bu\|_{L^2_t L^\infty_x} ( \|(d\bu,dh)\|^2_{L^\infty_t H^{s-1}_x} + \| \bw\|^2_{L^\infty_t H^{2}_x})(1 +\|(d\bu,dh)\|^2_{L^\infty_t H^{s-1}_x}   )
			\\
			&+ \|d\bu\|_{L^2_t L^\infty_x} ( \|(d\bu,dh)\|^2_{L^\infty_t H^{s-1}_x} + \| \bw\|^2_{L^\infty_t H^{2}_x})\|(d\bu,dh)\|^3_{L^\infty_t H^{s-1}_x}
			\\
			& + ( \|(d\bu,dh)\|^2_{L^\infty_t H^{s-1}_x} + \| \bw\|^2_{L^\infty_t H^{2}_x})(\|(d\bu,dh)\|_{L^\infty_t H^{s-1}_x} +\|(d\bu,dh)\|^2_{L^\infty_t H^{s-1}_x}  )
			\\
			& + ( \|(d\bu,dh)\|^2_{L^\infty_t H^{s-1}_x} + \| \bw\|^2_{L^\infty_t H^{2}_x})(\|(d\bu,dh)\|^3_{L^\infty_t H^{s-1}_x} +\|(d\bu,dh)\|^4_{L^\infty_t H^{s-1}_x}  )
			\\
			\lesssim & \epsilon^2_2.
		\end{split}
	\end{equation}
	For $\mathrm{J_4}$, by changing of coordinates $x_3 \rightarrow x_3+\phi(t,x')$, we obtain
	\begin{equation*}\label{L12}
		\begin{split}
			\mathrm{J_4}=& \int^2_{-2} \int_{\mathbb{R}^3} \partial^\alpha \{   ({u}^0)^{-1} {\Gamma} \} \cdot ({G}_\alpha-{F}_\alpha ) dx d\tau.
		\end{split}
	\end{equation*}
	By chain rule, we further calculate $\mathrm{J_4}$ by
	\begin{equation}\label{L13}
		\begin{split}
			\mathrm{J_4}=& \int^2_{-2} \int_{\mathbb{R}^3}  \partial^\alpha\{  ({u}^0)^{-1} {\Gamma} \cdot {G}_\alpha \}  dx -\int^2_{-2} \int_{\mathbb{R}^3}   ({u}^0)^{-1} {\Gamma} \cdot \partial^\alpha {G}_\alpha dx
			d\tau
			\\
			&-\int^2_{-2} \int_{\mathbb{R}^3} \partial^\alpha \{   ({u}^0)^{-1} {\Gamma} \cdot {F}_\alpha \} dx - \int^2_{-2} \int_{\mathbb{R}^3}    ({u}^0)^{-1} {\Gamma}  \cdot \partial^\alpha {F}_\alpha  dx
			\\
			=& \left(\int_{\mathbb{R}^3} ({u}^0)^{-1} \Gamma \cdot G^0 dx\right)\big|^2_{-2} -\int^2_{-2} \int_{\mathbb{R}^3}   ({u}^0)^{-1} {\Gamma} \cdot \partial^\alpha {G}_\alpha dx
			d\tau
			\\
			&- \left(\int_{\mathbb{R}^3} ({u}^0)^{-1} \Gamma \cdot F^0 dx \right)\big|^2_{-2} + \int^2_{-2} \int_{\mathbb{R}^3}    ({u}^0)^{-1} {\Gamma}  \cdot \partial^\alpha {F}_\alpha  dx.
		\end{split}
	\end{equation}
	Inserting \eqref{L13} to \eqref{L9}, it follows
	\begin{equation}\label{L14}
		\begin{split}
			\vert\kern-0.25ex\vert\kern-0.25ex\vert \bG-\bF \vert\kern-0.25ex\vert\kern-0.25ex\vert^2_{0,2,\Sigma}
			=	& \mathrm{J_1}+\mathrm{J_2}+ \left(\int_{\mathbb{R}^3} ({u}^0)^{-1} \Gamma \cdot G^0 dx\right)\big|^2_{-2} - \left(\int_{\mathbb{R}^3} ({u}^0)^{-1} \Gamma \cdot F^0 dx \right)\big|^2_{-2}
			\\
			&+\mathrm{J_3}+\mathrm{J_5}-\int^2_{-2} \int_{\mathbb{R}^3}   ({u}^0)^{-1} {\Gamma} \cdot \partial^\alpha {G}_\alpha dx
			+\int^2_{-2} \int_{\mathbb{R}^3}    ({u}^0)^{-1} {\Gamma}  \cdot \partial^\alpha {F}_\alpha  dx.
		\end{split}
	\end{equation}
	Note $1\leq u^0 \leq 1+C_0$. Referring \eqref{Wd60}-\eqref{Wd63}, and using \eqref{5021}, we can derive that
	\begin{equation}\label{L15}
		\begin{split}
			& \left|\int^2_{-2} \int_{\mathbb{R}^3}   ({u}^0)^{-1} {\Gamma} \cdot \partial^\alpha {G}_\alpha dx \right|
			\\
			\lesssim &  \|(d\bu,dh)\|_{L^2_t L^\infty_x}
			( \|(d\bu, dh)\|^2_{L^\infty_t H^{s-1}_x} +  \| \bw\|^2_{L^\infty_t H^{2}_x} ) \|(d\bu, dh)\|_{L^\infty_t H^{s-1}_x}
			\\
			&   + ( \|(d\bu, dh)\|^2_{L^\infty_t H^{s-1}_x} +  \| \bw\|^2_{L^\infty_t H^{2}_x} ) \|(d\bu, dh)\|^3_{L^\infty_t H^{s-1}_x}
			\\
			\lesssim & \epsilon_2^2.
		\end{split}
	\end{equation}
	In a similar way, by referring \eqref{Wd61}-\eqref{Wd90} and using \eqref{5021}, we can obtain
	\begin{equation}\label{L16}
		\begin{split}
			& \left|\int^2_{-2} \int_{\mathbb{R}^3}    ({u}^0)^{-1} {\Gamma}  \cdot \partial^\alpha {F}_\alpha  dx \right|
			\\
			\lesssim & \|(d\bu,dh)\|_{L^2_t L^\infty_x}
			( \|(d\bu, dh)\|^2_{L^\infty_t H^{s-1}_x} +  \| \bw\|^2_{L^\infty_t H^{2}_x} ) \|(d\bu, dh)\|_{L^\infty_t H^{s-1}_x}
			\\
			&   + ( \|(d\bu, dh)\|^2_{L^\infty_t H^{s-1}_x} +  \| \bw\|^2_{L^\infty_t H^{2}_x} ) (\|(d\bu, dh)\|^2_{L^\infty_t H^{s-1}_x} + \|(d\bu, dh)\|^3_{L^\infty_t H^{s-1}_x})
			\\
			\lesssim & \epsilon_2^2.
		\end{split}
	\end{equation}
	By H\"older's inequality, \eqref{YXg}, \eqref{Wd29}, and \eqref{5021}, we also derive
	\begin{equation}\label{L00}
		\begin{split}
			& \| \int_{\mathbb{R}^3} ({u}^0)^{-1} \Gamma \cdot G^0 dx \|_{L^\infty_{[-2,2]}}
			\\
			\lesssim & ( \|(d\bu, dh)\|^2_{L^\infty_t H^{s-1}_x} +  \| \bw\|^2_{L^\infty_t H^{2}_x} ) (\|(d\bu, dh)\|^2_{L^\infty_t H^{s-1}_x} + \|(d\bu, dh)\|^3_{L^\infty_t H^{s-1}_x})
			\\
			\lesssim & \epsilon_2^2,
		\end{split}
	\end{equation}
	and
	\begin{equation}\label{L02}
		\begin{split}
			& \| \int_{\mathbb{R}^3} ({u}^0)^{-1} \Gamma \cdot F^0 dx \|_{L^\infty_{[-2,2]}}
			\\
			\lesssim & ( \|(d\bu, dh)\|^2_{L^\infty_t H^{s-1}_x} +  \| \bw\|^2_{L^\infty_t H^{2}_x} ) (\|(d\bu, dh)\|^2_{L^\infty_t H^{s-1}_x} + \|(d\bu, dh)\|^3_{L^\infty_t H^{s-1}_x})
			\\
			\lesssim & \epsilon_2^2.
		\end{split}
	\end{equation}
	To summarize \eqref{L14}, \eqref{J1}, \eqref{J2}, \eqref{J4}, \eqref{L15}, \eqref{L16}, \eqref{L00}, and \eqref{L02}, we can find that
	\begin{equation}\label{L17}
		\begin{split}
			\vert\kern-0.25ex\vert\kern-0.25ex\vert \bG-\bF \vert\kern-0.25ex\vert\kern-0.25ex\vert^2_{0,2,\Sigma}
			\lesssim & \epsilon_2^2.
		\end{split}
	\end{equation}

	\textit{The bound for $\vert\kern-0.25ex\vert\kern-0.25ex\vert \mathrm{L}\bG \vert\kern-0.25ex\vert\kern-0.25ex\vert_{s_0-2,2,\Sigma}$.} By Lemma \ref{te3} and \eqref{5021}, we have
	\begin{equation}\label{L20}
		\begin{split}
			\vert\kern-0.25ex\vert\kern-0.25ex\vert  \mathrm{L} \left(\bG-\bF \right)\vert\kern-0.25ex\vert\kern-0.25ex\vert^2_{s_0-2,2,\Sigma} \lesssim &  \| \nabla \bu\|_{L^2_{[-2,2]}\dot{B}^{s_0-2}_{\infty,2}} \|\bG-\bF\|^2_{L^\infty_t {H}^{s_0-2}_x}(1+\vert\kern-0.25ex\vert\kern-0.25ex\vert d\phi-dt \vert\kern-0.25ex\vert\kern-0.25ex\vert_{s_0,2,\Sigma})
			\\
			&+ \vert\kern-0.25ex\vert\kern-0.25ex\vert d\phi-dt \vert\kern-0.25ex\vert\kern-0.25ex\vert_{s_0,2,\Sigma}  \|\bG-\bF\|^2_{L^\infty_t {H}^{s_0-2}_x}(1+\vert\kern-0.25ex\vert\kern-0.25ex\vert d\phi-dt \vert\kern-0.25ex\vert\kern-0.25ex\vert_{s_0,2,\Sigma})
			\\
			& + \textstyle{\sum_{\sigma\in\{0,s_0-2 \}}}\big|{\int}^2_{-2} \int_{\mathbb{R}^3} \Lambda^{\sigma}_{x'}{\mathrm{L} K^\alpha } \cdot \Lambda^{\sigma}_{x'}\mathrm{L} (G_\alpha-F_\alpha)dxd\tau \big|
			\\
			\lesssim & \epsilon_2^2+ \underbrace{ \textstyle{\sum_{\sigma\in\{0,s_0-2 \}}}\big|{\int}^2_{-2} \int_{\mathbb{R}^3} \Lambda^{\sigma}_{x'}{\mathrm{L} K^\alpha } \cdot \Lambda^{\sigma}_{x'}\mathrm{L} (G_\alpha-F_\alpha)dxd\tau \big| }_{\equiv \mathrm{J}_6}.
		\end{split}
	\end{equation}
	From \eqref{Kalpha}, we divide $\mathrm{J}_6$ into
	\begin{equation}\label{L21}
		\begin{split}
			\mathrm{J}_6=&\underbrace{ \textstyle{\sum_{\sigma\in\{0,s_0-2 \}}}\big|{\int}^2_{-2} \int_{\mathbb{R}^3} \Lambda^{\sigma}_{x'}{\mathrm{L} \left\{  \partial^\alpha \left( (u^0)^{-1} \Gamma \right) \right\} } \cdot \Lambda^{\sigma}_{x'}\mathrm{L} (G_\alpha-F_\alpha)dxd\tau \big| }_{\equiv \mathrm{J}_{61}}
			\\
			& + \underbrace{ \textstyle{\sum_{\sigma\in\{0,s_0-2 \}}}\big|{\int}^2_{-2} \int_{\mathbb{R}^3} \Lambda^{\sigma}_{x'}{\mathrm{L} \left\{ ({u}^0)^{-1} {E}^\alpha \right\} } \cdot \Lambda^{\sigma}_{x'}\mathrm{L} (G_\alpha-F_\alpha)dxd\tau \big| }_{\equiv \mathrm{J}_{62}}
			\\
			&+ \underbrace{ \textstyle{\sum_{\sigma\in\{0,s_0-2 \}}}\big|{\int}^2_{-2} \int_{\mathbb{R}^3} \Lambda^{\sigma}_{x'}{\mathrm{L} \left\{{\Gamma}  \partial^\alpha \left( ({u}^0)^{-1}\right)  \right\} } \cdot \Lambda^{\sigma}_{x'}\mathrm{L} (G_\alpha-F_\alpha)dxd\tau \big| }_{\equiv \mathrm{J}_{63}} .
		\end{split}
	\end{equation}
	To bound $\mathrm{J}_{6}$, we need to discuss $\mathrm{J}_{61}$, $\mathrm{J}_{62}$ and $\mathrm{J}_{63}$ as follows. By chain rule, we get
	\begin{equation}\label{L22}
		\begin{split}
			\mathrm{J}_{61} \leq	& \underbrace{ {\textstyle{\sum}_{\sigma\in\{0,s_0-2 \}}}\big|{\int}^2_{-2} \int_{\mathbb{R}^3} \partial^\alpha \left\{ \Lambda^{\sigma}_{x'}{\mathrm{L}    \left( (u^0)^{-1} \Gamma \right)  } \cdot \Lambda^{\sigma}_{x'}\mathrm{L} (G_\alpha-F_\alpha) \right\} dxd\tau \big| }_{\equiv 	\mathrm{J}_{61a} }
			\\
			& + \underbrace{ {\textstyle{\sum}_{\sigma\in\{0,s_0-2 \}}}\big|{\int}^2_{-2} \int_{\mathbb{R}^3} \Lambda^{\sigma}_{x'}{\mathrm{L}   \left( (u^0)^{-1} \Gamma \right)  } \cdot \Lambda^{\sigma}_{x'}\mathrm{L} \partial^\alpha G_\alpha dxd\tau \big| }_{\equiv 	\mathrm{J}_{61b} }
			\\
			& + \underbrace{ {\textstyle{\sum}_{\sigma\in\{0,s_0-2 \}}}\big|{\int}^2_{-2} \int_{\mathbb{R}^3} \Lambda^{\sigma}_{x'}{\mathrm{L}   \left( (u^0)^{-1} \Gamma \right)  } \cdot \Lambda^{\sigma}_{x'}\mathrm{L} \partial^\alpha F_\alpha dxd\tau \big| }_{\equiv 	\mathrm{J}_{61c} }.
		\end{split}
	\end{equation}
	By a direct calculation, it follows
	\begin{equation}\label{L23}
		\begin{split}
			\mathrm{J}_{61a}=& {\textstyle{\sum}_{\sigma\in\{0,s_0-2 \}}}  \left|  \left( \int_{\mathbb{R}^3}  \Lambda^{\sigma}_{x'}{\mathrm{L}    \left( (u^0)^{-1} \Gamma \right)  } \cdot \Lambda^{\sigma}_{x'}\mathrm{L} (G^0-F^0) dx \right)\big|^2_{-2} \right|
			\\
			\lesssim & \|(u^0)^{-1} \Gamma \|_{L^\infty_{t}H^{s_0-2}_x} ( \|\bG \|_{L^\infty_{t}H^{s_0-2}_x} + \|\bF \|_{L^\infty_{t}H^{s_0-2}_x} )
			\\
			\lesssim & (\|(d\bu,dh)\|_{L^\infty_t H^{s_0-1}_x}+\|(d\bu,dh)\|^2_{L^\infty_t H^{s_0-1}_x}) \| \bw\|^{2}_{L^\infty_t H^{s_0-\frac12}_x}
			\\
			\lesssim & \epsilon^2_2.
		\end{split}
	\end{equation}
	Above, we also use \eqref{5021}. While, for $\mathrm{J}_{61b} $, recalling the term $\mathrm{H}_{111}$ in \eqref{We60}, so we find that $\mathrm{J}_{61b} $ can be controlled in a same way for $\mathrm{H}_{111}$. Therefore, from \eqref{We63}, we have
	\begin{equation*}
		\begin{split}
			\mathrm{J}_{61b}
			\lesssim & ( \|(d\bu,dh)\|_{L^2_t \dot{B}_{\infty,2}^{s_0-2}}+ \|(d\bu,dh)\|_{L^2_t L^\infty_x} )
			\|(d\bu, dh)\|^2_{L^\infty_t H^{s_0-1}_x}  \|(d\bu, dh)\|_{L^\infty_t H^{s_0-1}_x}
			\\
			& +  ( \|(d\bu,dh)\|_{L^2_t \dot{B}_{\infty,2}^{s_0-2}}+ \|(d\bu,dh)\|_{L^2_t L^\infty_x} )
			\| \bw\|^2_{L^\infty_t H^{s_0}_x}  \|(d\bu, dh)\|_{L^\infty_t H^{s_0-1}_x}
			\\
			&   + ( \|(d\bu, dh)\|^2_{L^\infty_t H^{s_0-1}_x} +  \| \bw\|^2_{L^\infty_t H^{s_0}_x} ) \|(d\bu, dh)\|^3_{L^\infty_t H^{s_0-1}_x}.
		\end{split}
	\end{equation*}
	Due to \eqref{5021}, it yields
	\begin{equation}\label{L24}
		\begin{split}
			\mathrm{J}_{61b}
			\lesssim & \epsilon^2_2.
		\end{split}
	\end{equation}
	For $\mathrm{J}_{61c} $, recalling the term $\mathrm{H}_{121}$ in \eqref{We61}, so we find that $\mathrm{J}_{61c} $ can be controlled in a same way for $\mathrm{H}_{121}$. Therefore, from \eqref{We90}, we have
	\begin{equation*}
		\begin{split}
			\mathrm{J}_{61c}
			\lesssim & ( \|(d\bu,dh)\|_{L^2_t \dot{B}_{\infty,2}^{s_0-2}}+ \|(d\bu,dh)\|_{L^2_t L^\infty_x} )
			( \|(d\bu, dh)\|^2_{L^\infty_t H^{s_0-1}_x} +  \| \bw\|^2_{L^\infty_t H^{s_0}_x} ) \|(d\bu, dh)\|_{L^\infty_t H^{s_0-1}_x}
			\\
			&   + ( \|(d\bu, dh)\|^2_{L^\infty_t H^{s_0-1}_x} +  \| \bw\|^2_{L^\infty_t H^{s_0}_x} ) (\|(d\bu, dh)\|^2_{L^\infty_t H^{s_0-1}_x} + \|(d\bu, dh)\|^3_{L^\infty_t H^{s_0-1}_x}).
		\end{split}
	\end{equation*}
	By \eqref{5021}, we have proved
	\begin{equation}\label{L25}
		\begin{split}
			\mathrm{J}_{61c}
			\lesssim  \epsilon_2^2.
		\end{split}
	\end{equation}
	Since \eqref{L22}, using \eqref{L23}, \eqref{L24}, together with \eqref{L25}, we have obtained
	\begin{equation}\label{J61}
		\begin{split}
			\mathrm{J}_{61}
			\lesssim  \epsilon_2^2.
		\end{split}
	\end{equation}
	For $\mathrm{J}_{62}$, using H\"older's inequality and \eqref{5021}, we can show that
	\begin{equation}\label{L26}
		\begin{split}
			\mathrm{J}_{62} 	\lesssim	& \| \bE \|_{L^1_t H^{s_0-2}_x} \|\bG-\bF\|_{L^\infty_t {H}_x^{s_0-2}}
			\\
			\lesssim & ( \|(d\bu,dh)\|_{L^2_t \dot{B}_{\infty,2}^{s_0-2}}+ \|(d\bu,dh)\|_{L^2_t L^\infty_x} ) \| \bw\|^2_{L^\infty_t H^{s_0}_x}
			\\
			& +(\|(d\bu,dh)\|_{L^\infty_t H^{s_0-1}_x}+\|(d\bu,dh)\|^5_{L^\infty_t H^{s_0-1}_x})  \| \bw\|^2_{L^\infty_t H^{s_0}_x}
			\\
			\lesssim & \epsilon_2^2.
		\end{split}
	\end{equation}
	For $\mathrm{J}_{63}$, using H\"older's inequality and \eqref{5021}, we can also bound it by
	\begin{equation}\label{L27}
		\begin{split}
			\mathrm{J}_{63}  \lesssim & \|\Gamma \cdot d\bu\|_{L^1_t {H}_x^{s_0-2}}  \|\bG-\bF\|_{L^\infty_t {H}_x^{s_0-2}}
			\\
			\lesssim & (\|d\bu\|_{L^2_t \dot{B}_{\infty,2}^{s_0-2}}+ \|d\bu\|_{L^2_t L^\infty_x}) \| \bw\|^2_{L^\infty_t H^{s_0}_x}(1 +\|(d\bu,dh)\|^2_{L^\infty_t H^{s_0-1}_x}  )
			\\
			& + (\|d\bu\|_{L^2_t \dot{B}_{\infty,2}^{s_0-2}}+ \|d\bu\|_{L^2_t L^\infty_x}) \| \bw\|^2_{L^\infty_t H^{s_0}_x}\|(d\bu,dh)\|^3_{L^\infty_t H^{s_0-1}_x}
			\\
			\lesssim & \epsilon_2^2.
		\end{split}
	\end{equation}
	To summarize our outcome \eqref{L20}, \eqref{L21}, \eqref{J61}, \eqref{L26}, and \eqref{L27}, we can conclude that
	\begin{equation}\label{L28}
		\begin{split}
			\vert\kern-0.25ex\vert\kern-0.25ex\vert  \mathrm{L} \left(\bG-\bF \right)\vert\kern-0.25ex\vert\kern-0.25ex\vert^2_{s_0-2,2,\Sigma} \lesssim \epsilon_2^2.
		\end{split}
	\end{equation}
	By trace theorem, \eqref{YX0}, H\"older's inequality, and \eqref{5021}, we have
	\begin{equation}\label{L29}
		\begin{split}
			\vert\kern-0.25ex\vert\kern-0.25ex\vert  \mathrm{L} \bF \vert\kern-0.25ex\vert\kern-0.25ex\vert_{s_0-2,2,\Sigma}
			\lesssim & \| \mathrm{L} \bF \|_{L^2_{[-2,2]}H_x^{s_0-\frac32}}
			\\
			\lesssim & \| d\bw \cdot (d\bu,dh) \|_{L^\infty_{[-2,2]}H_x^{s_0-\frac32}}+\| \bw \cdot dh \cdot dh \|_{L^\infty_{[-2,2]}H_x^{s_0-\frac32}}
			\\
			\lesssim &\| d\bw\|_{L^\infty_{[-2,2]}H_x^{s_0-1}} \|d\bu,dh) \|_{L^\infty_{[-2,2]}H_x^{1}}+\| \bw\|_{L^\infty_{[-2,2]}H_x^{s_0}} \|dh \|^2_{L^\infty_{[-2,2]}H_x^{1}}
			\\
			\lesssim & \epsilon_2^2.
		\end{split}
	\end{equation}
	Since \eqref{L28}, \eqref{L29}, and \eqref{L17}, so \eqref{L6} becomes
	\begin{equation}\label{L30}
		\begin{split}
			\|\nabla \bW\|_{L^2_t H^{s_0-2}_{x'}(\Sigma)}
			\lesssim \epsilon_2.
		\end{split}
	\end{equation}

	\textit{The bound for $\|\partial_{t} \bW\|_{L^2_t H^{s_0-2}_{x'}(\Sigma)}$.} For $\partial_{x'}\bW=\nabla \bW \cdot (0,\partial_{x'}\phi)^{\mathrm{T}}$, using \eqref{L30}, it yields
	\begin{equation}\label{L32}
		\begin{split}
			\|\partial_{x'} \bW\|_{L^2_t H^{s_0-2}_{x'}(\Sigma)}
			\lesssim \epsilon_2.
		\end{split}
	\end{equation}
	By using \eqref{d21}, \eqref{d22} and \eqref{L32}, we infer
	\begin{equation}\label{L33}
		\begin{split}
			\|\partial_{t} \bW\|_{L^2_t H^{s_0-2}_{x'}(\Sigma)}
			\lesssim \epsilon_2.
		\end{split}
	\end{equation}
	Adding \eqref{OM0}, \eqref{L32} and \eqref{L33}, so \eqref{te201} holds. Thus, we finish the proof of Lemma \ref{te20}.
\end{proof}
\subsection{Regularity of the characteristic hypersurface}\label{rch}
Our goal here is to establish the following:
\begin{proposition}\label{r2}
	Let $(h,\bu,\bw) \in \mathcal{H}$ so that $\Re(h,\bu, \bw) \leq 2 \epsilon_1$. Then

	\begin{equation}\label{G}
		\Re(h,\bu, \bw) \lesssim \epsilon_2.
	\end{equation}
	Furthermore, for each $t\in[-2,2]$ it holds that
	\begin{equation}\label{502}
		\|d \phi_{\theta,r}(t,\cdot)-dt \|_{C^{1,\delta_0}_{x'}} \lesssim \epsilon_2+  \| d {\mathbf{g}}(t,\cdot) \|_{C^{\delta_0}_x(\mathbb{R}^3)}.
	\end{equation}
\end{proposition}
To prove Proposition \ref{r2}, let us first introduce a null frame on $\Sigma$.
\subsubsection{Null frame}
Firstly, let
\begin{equation*}
	V=(dr)^*,
\end{equation*}
where $r$ is the defining function of the foliation $\Sigma$, and where $*$ denotes the identification of covectors and vectors induced by $\mathbf{g}$. Then $V$ is the null geodesic flow field tangent to $\Sigma$. Let
\begin{equation}\label{600}
	\sigma=dt(V), \qquad l=\sigma^{-1} V.
\end{equation}
Thus $l$ is the g-normal field to $\Sigma$ normalized so that $dt(l)=1$, hence
\begin{equation}\label{601}
	l=\left< dt,dx_3-d\phi\right>^{-1}_{\mathbf{g}} \left( dx_3-d \phi \right)^*.
\end{equation}
So the coefficients $l^j$ are smooth functions of $h, \bu$ and $d \phi$. Conversely,
\begin{equation}\label{602}
	dx_3-d \phi =\left< l,\partial_{x_3}\right>^{-1}_{\mathbf{g}} l^*.
\end{equation}
It also implies that $d \phi$ is a smooth function of $h, \bu$ and the coefficients of $l$.

Next we introduce the vector fields $e_a, a=1,2$ tangent to the fixed-time slice $\Sigma^t$ of $\Sigma$. We do this by applying Gram-Schmidt orthogonalization in the metric $\mathbf{g}$ to the $\Sigma^t$-tangent vector fields $\partial_{x_a}+ \partial_{x_a} \phi \partial_{x_3}, a=1, 2$.

Finally, we set
\begin{equation*}
	\underline{l}=l+2\partial_t.
\end{equation*}
Then $\{l, \underline{l}, e_1, e_2 \}$ is a null frame in the sense that
\begin{align*}
	& \left<l, \underline{l} \right>_{\mathbf{g}} =2, \qquad \qquad \ \ \ \ \ \ \left< e_a, e_b\right>_{\mathbf{g}}=\delta_{ab}, \ \ (a,b=1,2),
	\\
	& \left<l, l \right>_{\mathbf{g}} =\left<\underline{l}, \underline{l} \right>_{\mathbf{g}}=0, \qquad \left<l, e_a \right>_{\mathbf{g}}=\left<\underline{l}, e_a \right>_{\mathbf{g}}=0, \ \ (a=1,2).
\end{align*}
The coefficients of each fields are smooth functions of $h$, $\bu$ and $d \phi$. Also, by $\Re(h,\bu,\bw)\leq 2\epsilon_1$, we also have the pointwise bound
\begin{equation*}
	| e_1 - \partial_{x_1} | +| e_2 - \partial_{x_2} | + | l- (\partial_t+\partial_{x_3}) | + | \underline{l} - (-\partial_t+\partial_{x_3})|  \lesssim \epsilon_1.
\end{equation*}
After that, we can state the following lemma concerning to the decomposition of curvature tensor.
\begin{Lemma}[\cite{ST}, Lemma 5.8]\label{LLQ}
	Suppose $f$ satisfies
	$$\mathbf{g}^{\alpha \beta} \partial^2_{\alpha \beta}f=F.$$
	Let $(t,x',\phi(t,x'))$ denote the parametrization of $\Sigma$, and for $0 \leq \alpha, \beta \leq 2$, let $/\kern-0.55em \partial_\alpha$ denote differentiation along $\Sigma$ in the induced coordinates. Then, for $0 \leq \alpha, \beta \leq 2$, one can write
	\begin{equation*}
		/\kern-0.55em \partial_\alpha /\kern-0.55em \partial_\beta (f|_{\Sigma}) = l(f_2)+ f_1,
	\end{equation*}
	where
	\begin{equation*}
		\| f_2 \|_{L^2_t H^{s_0-1}_{x'}(\Sigma)}+\| f_1 \|_{L^1_t H^{s_0-1}_{x'}(\Sigma)} \lesssim \|df\|_{L^\infty_t H_x^{s_0-1}}+ \|df\|_{L^2_t L_x^\infty}+ \| F\|_{L^2_t H^{s_0-1}_x}+ \| F\|_{L^1_t H^{s_0-1}_{x'}(\Sigma)}.
	\end{equation*}
\end{Lemma}
\begin{corollary}\label{Rfenjie}
	Let $R$ be the Riemann curvature tensor of the metric ${\mathbf{g}}$. Let $e_0=l$. Then for any $0 \leq a, b, c,d \leq 2$, we can write
	\begin{equation}\label{603}
		\left< R(e_a, e_b)e_c, e_d \right>_{\mathbf{g}}|_{\Sigma}=l(f_2)+f_1,
	\end{equation}
	where $|f_1|\lesssim |\bW|+ |d \mathbf{g} |^2$ and $|f_2| \lesssim |d {\mathbf{g}}|$. Moreover, the characteristic energy estimates
	\begin{equation}\label{604}
		\|f_2\|_{L^2_t H^{s_0-1}_{x'}(\Sigma)}+\|f_1\|_{L^1_t H^{s_0-1}_{x'}(\Sigma)} \lesssim \epsilon_2,
	\end{equation}
	holds. Additionally, for any $t \in [-2,2]$,
	\begin{equation}\label{605}
		\|f_2(t,\cdot)\|_{C^{\delta_0}_{x'}(\Sigma^t)} \lesssim \|d \mathbf{g}\|_{C^{\delta_0}_x(\mathbb{R}^3)}.
	\end{equation}
\end{corollary}
\begin{proof}
	According to the expression of curvature tensor, we have
	\begin{equation*}
		\left< R(e_a, e_b)e_c, e_d \right>_{\mathbf{g}}= R_{\alpha \beta \mu \nu}e^\alpha_a e^\beta_b e_c^\mu e_d^\nu,
	\end{equation*}
	where
	\begin{equation*}
		R_{\alpha \beta \mu \nu}= \frac12 \left[ \partial^2_{\alpha \mu} \mathbf{g}_{\beta \nu}+\partial^2_{\beta \nu} \mathbf{g}_{\alpha \mu}-\partial^2_{\beta \mu} \mathbf{g}_{\alpha \nu}-\partial^2_{\alpha \nu} \mathbf{g}_{\beta \mu} \right]+  \Theta(\mathbf{g}^{\alpha \beta}, d \mathbf{g}_{\alpha \beta}),
	\end{equation*}
	where $\Theta$ is a sum of products of coefficients of $\mathbf{g}^{\alpha \beta} $ with quadratic forms in $d \mathbf{g}_{\alpha \beta}$. By using Proposition \ref{r1}, then $\Theta$ satisfies the bound required of $f_1$. It is sufficient to consider
	\begin{equation*}
		\frac12 e^\alpha_a e^\beta_b e_c^\mu e_d^\nu  \left[ \partial^2_{\alpha \mu} \mathbf{g}_{\beta \nu}+\partial^2_{\beta \nu} \mathbf{g}_{\alpha \mu}-\partial^2_{\beta \mu} \mathbf{g}_{\alpha \nu}-\partial^2_{\alpha \nu} \mathbf{g}_{\beta \mu} \right].
	\end{equation*}
	We therefore look at the term $ e^\alpha_a e_c^\mu \partial^2_{\alpha \mu} \mathbf{g}_{\beta \nu} $, which is typical. By \eqref{503}, Proposition \ref{r1}, and Lemma \ref{te3}, we get
	\begin{equation}\label{LLL}
		\vert\kern-0.25ex\vert\kern-0.25ex\vert  l^\alpha - \delta^{\alpha 0} \vert\kern-0.25ex\vert\kern-0.25ex\vert _{s_0,2,\Sigma} +\vert\kern-0.25ex\vert\kern-0.25ex\vert  \underline{l}^\alpha + \delta^{\alpha 0}-2\delta^{\alpha 3} \vert\kern-0.25ex\vert\kern-0.25ex\vert _{s_0,2,\Sigma} + \vert\kern-0.25ex\vert\kern-0.25ex\vert  e^\alpha_a- \delta^{\alpha a} \vert\kern-0.25ex\vert\kern-0.25ex\vert _{s_0,2,\Sigma} \lesssim \epsilon_1.
	\end{equation}
	By \eqref{LLL} and Proposition \ref{r1}, the term $ e_a (e_c^\mu) \partial_{ \mu} \mathbf{g}_{\beta \nu}$ satisfies the bound required of $f_1$, so we consider $e_a(e_c(\mathbf{g}_{\beta \nu}))$. Finally, since the coefficients of $e_c$ in the basis $/\kern-0.55em \partial_\alpha$ have tangential derivatives bounded in $L^2_tH^{s_0-1}_{x'}(\Sigma)$, we are reduced by Lemma \ref{LLQ} to verifying that
	\begin{equation*}
		\| \mathbf{g}^{\alpha \beta } \partial^2_{\alpha \beta} \mathbf{g}_{\mu \nu} \|_{L^1_t H^{s_0-1}_{x'}(\Sigma)} \lesssim \epsilon_2.
	\end{equation*}
	Note
	\begin{equation*}
		\mathbf{g}^{\alpha \beta } \partial^2_{\alpha \beta} \mathbf{g}_{\mu \nu}= \square_{\mathbf{g}} \mathbf{g}_{\mu \nu}.
	\end{equation*}
	By Lemma \ref{te20}, Corollary \ref{vte}, \eqref{XT}, and \eqref{403}, so it yields
	\begin{equation*}
		\begin{split}
			\| \square_{\mathbf{g}} \mathbf{g}_{\mu \nu}\|_{L^1_t H^{s_0-1}_{x'}(\Sigma)}
			\lesssim & \ \| \square_{\mathbf{g}} \bu\|_{L^2_t H^{s_0-1}_{x'}(\Sigma)}+\| \square_{\mathbf{g}} h \|_{L^2_t H^{s_0-1}_{x'}(\Sigma)}
			\\
			\lesssim & \ \| \bW\|_{L^1_t H^{s_0-1}_{x'}(\Sigma)}+ \| d{\mathbf{g}} \cdot d{\mathbf{g}}\|_{L^1_t H^{s_0-1}_{x'}(\Sigma)}
			\\
			\lesssim & \ \|  \bW\|_{L^2_t H^{s_0-1}_{x'}(\Sigma)}+ \| d{\mathbf{g}} \|_{L^2_t L^\infty_x }\| d{\mathbf{g}}\|_{L^2_t H^{s_0-1}_{x'}(\Sigma)}
			\\
			\lesssim & \ \epsilon_2.
		\end{split}
	\end{equation*}
	Above, $\bW$ and $(d{\mathbf{g}})^2$ are included in $f_1$. We hence complete the proof of Corollary \ref{Rfenjie}.
\end{proof}

\subsubsection{Estimate of the connection coefficients}
Define
\begin{equation}\label{Xab}
	\chi_{ab} = \left<D_{e_a}l,e_b \right>_{\mathbf{g}}, \qquad l(\ln \sigma)=\frac{1}{2}\left<D_{l}\underline{l},l \right>_{\mathbf{g}}, \qquad \mu_{0ab} = \left<D_{l}e_a,e_b \right>_{\mathbf{g}}.
\end{equation}
For $\sigma$, we set the initial data $\sigma=1$ at the time $-2$. Due to Proposition \ref{r1}, we have
\begin{equation}\label{606}
	\|\chi_{ab}\|_{L^2_t H^{s_0-1}_{x'}(\Sigma)} + \| l(\ln \sigma)\|_{L^2_t H^{s_0-1}_{x'}(\Sigma)} + \|\mu_{0ab}\|_{L^2_t H^{s_0-1}_{x'}(\Sigma)} \lesssim \epsilon_1.
\end{equation}
In a similar way, if we expand $l=l^\alpha /\kern-0.55em \partial_\alpha$ in the tangent frame $\partial_t, \partial_{x'}$ on $\Sigma$, we infer
\begin{equation}\label{607}
	l^0=1, \quad \|l^1\|_{s_0,2,\Sigma} \lesssim \epsilon_1.
\end{equation}
\begin{Lemma}\label{chi}
	Let $\chi_{ab}$ ($a,b=1,2$) be defined as \eqref{Xab}. Then
	\begin{equation}\label{608}
		\|\chi_{ab}\|_{L^2_t H^{s_0-1}_{x'}(\Sigma)} \lesssim \epsilon_2.
	\end{equation}
	Furthermore, for any $t \in [-2,2]$,
	\begin{equation*}\label{609}
		\| \chi_{ab} \|_{C^{\delta_0}_{x'}(\Sigma^t)} \lesssim \epsilon_2+ \|d \mathbf{g} \|_{C^{\delta_0}_{x}(\mathbb{R}^3)}.
	\end{equation*}
\end{Lemma}
\begin{proof}
	For $\chi_{ab}$, we know 
	\begin{equation*}
		l(\chi_{ab})=\left< R(l,e_a)l, e_b \right>_{\mathbf{g}}-\chi_{ac}\chi_{cb}-l(\ln \sigma)\chi_{ab}+\mu_{0ab} \chi_{cb}+ \mu_{0bc}\chi_{ac},
	\end{equation*}
	which is introduced by Klainerman-Rodnianski \cite{KR2}. Due to Corollary \ref{Rfenjie}, we can rewrite the above equation as
	\begin{equation}\label{610}
		l(\chi_{ab}-f_2)=f_1-\chi_{ac}\chi_{cb}-l(\ln \sigma)\chi_{ab}+\mu_{0ab} \chi_{cb}+ \mu_{0bc}\chi_{ac},
	\end{equation}
	where $f_1$ and $f_2$ is the same as in Corollary \ref{Rfenjie}. Let $\Lambda_{x'}^{s_0-1}$ be the fractional derivative operator in the $x'$ variables. Set
	\begin{equation*}
		\Psi=f_1-\chi_{ac}\chi_{cb}-l(\ln \sigma)\chi_{ab}+\mu_{0ab} \chi_{cb}+ \mu_{0bc}\chi_{ac}.
	\end{equation*}
	By \eqref{610}, then we have
	\begin{equation}\label{613}
		\begin{split}
			& \|\Lambda_{x'}^{s_0-1}(\chi_{ab}-f_2)(t,\cdot) \|_{ L^2_{x'}(\Sigma^t)}
			\lesssim \| [\Lambda_{x'}^{s_0-1},l](\chi_{ab}-f_2) \|_{L^1_tL^2_{x'}(\Sigma^t)}+ \| \Lambda_{x'}^{s_0-1}\Psi \|_{L^1_tL^2_{x'}(\Sigma^t)}.
		\end{split}
	\end{equation}
	Note that $H^{s_0-1}_{x'}(\Sigma^t)$ is an algebra for $s_0>2$. By direct calculations, we can compute out
	\begin{equation}\label{614}
		\begin{split}
			\| \Lambda_{x'}^{s-1}\Psi \|_{L^1_tL^2_{x'}(\Sigma^t)} &\lesssim \|f_1\|_{L^1_tH^{s_0-1}_{x'}(\Sigma^t)}+ \|\chi\|^2_{L^2_tH^{s_0-1}_{x'}(\Sigma^t)}
			\\
			& \quad + \|\chi\|_{L^2_tH^{s_0-1}_{x'}(\Sigma^t)}\cdot\|l(\ln \sigma)\|_{L^2_tH^{s_0-1}_{x'}(\Sigma^t)}
			\\
			& \quad + \|\mu\|_{L^2_tH^{s_0-1}_{x'}(\Sigma^t)}\cdot\|\chi\|_{L^2_tH^{s_0-1}_{x'}(\Sigma^t)}.
		\end{split}
	\end{equation}
We next bound
	\begin{align*}
		\| [\Lambda_{x'}^{s_0-1},l](\chi_{ab}-f_2) \|_{L^2_{x'}(\Sigma^t)} &\leq \| /\kern-0.55em \partial_{\alpha} l^{\alpha} (\chi_{ab}-f_2)(t,\cdot) \|_{H^{s_0-1}_{x'}(\Sigma^t)}
		\\
		& \quad \ + \|[\Lambda_{x'}^{s_0-1} /\kern-0.55em \partial_{\alpha}, l^{\alpha}](\chi-f_2)(t,\cdot) \|_{L^{2}_{x'}(\Sigma^t)}.
	\end{align*}
	By Kato-Ponce commutator estimate and Sobolev embeddings, the above is bounded by
	\begin{equation}\label{615}
		\|l^1(t,\cdot)\|_{H^{s_0-1}_{x'}(\Sigma^t)} \| \Lambda_{x'}^{s_0-1}(\chi_{ab}-f_2)(t,\cdot) \|_{L^{2}_{x'}(\Sigma^t)} .
	\end{equation}
	Gathering \eqref{606}, \eqref{607}, \eqref{604}, \eqref{613}, \eqref{614}, and \eqref{615} together, we thus prove that
	\begin{equation*}
		\sup_{t\in [-2,2]} \|(\chi_{ab}-f_2)(t,\cdot)\|_{H^{s_0-1}_{x'}(\Sigma^t)}  \lesssim \epsilon_2.
	\end{equation*}
	This bound combining with \eqref{604} give us \eqref{608}.

	From \eqref{610}, we see that
	\begin{equation}\label{616}
		\begin{split}
			\| \chi_{ab}-f_2\|_{C^{\delta}_{x'}} & \lesssim \| f_1 \|_{L^1_tC^{\delta}_{x'}}+ \|\chi_{ac}\chi_{cb}\|_{L^1_tC^{\delta}_{x'} }+\|l(\ln \sigma)\chi_{ab}\|_{L^1_tC^{\delta}_{x'}}
			\\
			& \quad + \|\mu_{0ab} \chi_{cb}\|_{L^1_tC^{\delta}_{x'} }+\|\mu_{0bc}\chi_{ac}\|_{L^1_tC^{\delta}_{x'} }.
		\end{split}
	\end{equation}
	Seeing \eqref{616}, using \eqref{604}, \eqref{605}, \eqref{606}, together with Sobolev imbedding\footnote{The number $\delta_0$ is defined in \eqref{a1} and $\delta_0\in (0,s_0-2)$.} $H^{s_0-1}(\mathbb{R}^2)\hookrightarrow C^{\delta_0}(\mathbb{R}^2)$ , we derive that
	\begin{equation*}
		\| \chi_{ab} \|_{C^{\delta_0}_{x'}(\Sigma^t)} \lesssim \epsilon_2+ \|d \mathbf{g}\|_{C^{\delta_0}_{x}(\mathbb{R}^3)}.
	\end{equation*}
	Therefore, we complete the proof of Lemma \ref{chi}.
\end{proof}
\subsubsection{Proof of Proposition \ref{r2}}
Our goal here is to prove \eqref{G} and \eqref{502}. We recall
\begin{equation*}
	\Re(h,\bu, \bw)= \vert\kern-0.25ex\vert\kern-0.25ex\vert d\phi(t,x')-dt\vert\kern-0.25ex\vert\kern-0.25ex\vert_{s_0,2, \Sigma}.
\end{equation*}
Combining \eqref{602}, with Proposition \ref{r1}, then \eqref{G} follows from the bound
\begin{equation*}
	\vert\kern-0.25ex\vert\kern-0.25ex\vert l-(\partial_t-\partial_{x_3})\vert\kern-0.25ex\vert\kern-0.25ex\vert_{s_0,2,\Sigma} \lesssim \epsilon_2,
\end{equation*}
where it is understood that one takes the norm of the coefficients of $l-(\partial_t-\partial_{x_3})$ in the standard frame on $\mathbb{R}^{3+1}$. Let $\Gamma^\alpha_{\beta \gamma}$ be the Christoffel symbols. Then we have
\begin{equation}\label{pow}
\|\Gamma^\alpha_{\beta \gamma}\|_{L^2_t L^\infty_x} \lesssim \|d {\mathbf{g}} \|_{L^2_t L^\infty_x}\lesssim \epsilon_2
\end{equation}
Due to \eqref{pow} and the geodesic equation, we infer
\begin{equation}\label{pow1}
	\|l-(\partial_t-\partial_{x_3})\|_{L^\infty_{t,x}} \lesssim \epsilon_2.
\end{equation}
So it's sufficient to bound the tangential derivatives of the coefficients of $l-(\partial_t-\partial_{x_3})$ in the norm $L^2_t H^{s_0-1}_{x'}(\Sigma)$. By using Proposition \ref{r1}, we can estimate the Christoffel symbols
\begin{equation}\label{pow2}
	\|\Gamma^\alpha_{\beta \gamma} \|_{L^2_t H^{s_0-1}_{x'}(\Sigma^t)} \lesssim \epsilon_2.
\end{equation}
For $s_0>2$, then $H^{s_0-1}_{x'}(\Sigma^t)$ is an algebra. By \eqref{pow}, \eqref{pow1}, \eqref{pow2}, and Proposition \ref{r1}, we get
\begin{equation*}
	\|\Gamma^\alpha_{\beta \gamma} e_1^\beta l^\gamma\|_{L^2_t H^{s_0-1}_{x'}(\Sigma^t)} \lesssim \epsilon_2.
\end{equation*}
We now establish the following bound,
\begin{equation}\label{qop}
	\| \left< D_{e_a}l, e_b \right>\|_{L^2_t H^{s_0-1}_{x'}(\Sigma^t)}+ \| \left< D_{e_a}l, \underline{l} \right>\|_{L^2_t H^{s_0-1}_{x'}(\Sigma^t)}+\|\left< D_{l}l, \underline{l} \right>\|_{L^2_t H^{s_0-1}_{x'}(\Sigma^t)} \lesssim \epsilon_2.
\end{equation}
The first term in \eqref{qop} is $\chi_{ab}$, which has been bounded by Lemma \ref{chi}. For the second term in \eqref{qop}, noting
\begin{equation*}
	\left< D_{e_a}l, \underline{l} \right>=\left< D_{e_a}l, 2\partial_t \right>=-2\left< D_{e_a}\partial_t,l \right>,
\end{equation*}
so it can be bounded via Proposition \ref{r1}. In a similar idea, the last term in \eqref{qop} can also be controlled by Proposition \ref{r1}.

It still remains for us to prove
\begin{equation*}
	\| d \phi(t,x')-dt \|_{C^{1,{\delta_0}}_{x'}(\mathbb{R}^2)}  \lesssim \epsilon_2+ \| d\mathbf{g}(t,\cdot)\|_{C^{\delta_0}_x(\mathbb{R}^3)}.
\end{equation*}
By \eqref{602}, it is sufficient to show that
\begin{equation*}
	\|l(t,\cdot)-(\partial_t-\partial_{x_3})\|_{C^{1,{\delta_0}}_{x'}(\mathbb{R}^2)} \lesssim \epsilon_2+ \| d \mathbf{g} (t,\cdot)\|_{C^{\delta_0}_x(\mathbb{R}^3)}.
\end{equation*}
The coefficients of $e_1$ are small in $C^{\delta_0}_{x'}(\Sigma^t)$ perturbations of their constant coefficient analogs, so it is sufficient to show that
\begin{equation*}
	\|\left< D_{e_1}l, e_1 \right>(t,\cdot)\|_{C^{\delta_0}_{x'}(\Sigma^t)}
	+\|\left< D_{e_1}l, \underline{l} \right>(t,\cdot)\|_{C^{\delta_0}_{x'}(\Sigma^t)}  \lesssim \epsilon_2+ \| d\mathbf{g}(t,\cdot)\|_{C^{\delta_0}_x(\mathbb{R}^3)}.
\end{equation*}
Above, the first term is bounded by Lemma \ref{chi}, and the second term can be estimated by using
\begin{equation*}
	\|\left< D_{e_1}\partial_t, l \right>(t,\cdot)\|_{C^{\delta_0}_{x'}(\Sigma^t)} \lesssim  \| d\mathbf{g}(t,\cdot)\|_{C^{\delta_0}_x(\mathbb{R}^3)}.
\end{equation*}
Therefore, we complete the proof of Proposition \ref{r2}.
\subsection{Strichartz estimates of solutions}\label{SEs}
Our goal is to prove the Strichartz estimates in Proposition \ref{p4}. We start to introduce a proposition as follows.
\begin{proposition} \label{r5}
	Suppose that $(h,\bu, \bw) \in \mathcal{{H}}$ and $\Re(h,\bu, \bw)\leq 2 \epsilon_1$. Let $\mathbf{g}$ be denoted in \eqref{AMd3}. For each $1 \leq r \leq s+1$, and for each $t_0 \in [-2,2]$, then the linear, homogeneous equation
	\begin{equation}\label{wform1}
		\begin{cases}
			& \square_{\mathbf{g}} f=0,\qquad (t,x)\in  [-2,2]\times \mathbb{R}^3,
			\\
			&(f,\partial_t f)|_{t=t_0}=(f_0, f_1),
		\end{cases}
	\end{equation}
	is well-posed on the time-interval $[-2,2]$ if the data $(f_0, f_1)\in H_x^r \times H_x^{r-1}$. Moreover, the solution satisfies, respectively, the energy estimates
	\begin{equation*}
		\| f\|_{L^\infty_{[-2,2]}H_x^r}+ \| \partial_t f\|_{L^\infty_{[-2,2]}H_x^{r-1}} \leq C \big( \|f_0\|_{H_x^r}+ \|f_1\|_{H_x^{r-1}} \big),
	\end{equation*}
	and the Strichartz estimates\footnote{here the constant $C$ is universal if the solution $(h, \bu,\bw) \in \mathcal{{H}}$.}
	\begin{equation}\label{SL}
		\| \left<\nabla \right>^k f\|_{L^2_{[-2,2]}L^\infty_x} \leq C \big( \|f_0\|_{H_x^r}+ \|f_1\|_{H_x^{r-1}} \big), \quad k<r-1,
	\end{equation}
	hold. The similar estimates also hold when we replace $\left<\nabla \right>^k$ by $\left<\nabla \right>^{k-1}d$.
\end{proposition}
\begin{remark}
	The proof for Proposition \ref{r5} can be found in section \ref{spp}.
\end{remark}
Due to Proposition \ref{r5}, we next introduce a result for Strichartz estimates of inhomogeneous waves.
\begin{corollary}\label{r8}
	Suppose that $(h,\bu, \bw) \in \mathcal{{H}}$ and $\Re(h,\bu, \bw)\leq 2 \epsilon_1$. Let $\mathbf{g}$ is denoted in \eqref{AMd3}. For each $1 \leq r \leq s+1$, and for each $t_0 \in [-2,2]$, then the linear wave equation
	\begin{equation}\label{wform5}
		\begin{cases}
			& \square_{\mathbf{g}} f=F,\qquad (t,x)\in [-2,2]\times \mathbb{R}^3,
			\\
			&(f,\partial_t f)|_{t=t_0}=(f_0, f_1),
		\end{cases}
	\end{equation}
	is well-posed on the time-interval $[-2,2]$ if the initial data $(f_0, f_1)\in H_x^r \times H_x^{r-1}$. Furthermore, the solution satisfies energy estimates
	\begin{equation}\label{QL0}
		\| f\|_{L^\infty_{[-2,2]}H_x^r}+ \| \partial_t f\|_{L^\infty_{[-2,2]}H_x^{r-1}} \leq C \big( \|f_0\|_{H_x^r}+ \|f_1\|_{H_x^{r-1}}+\|F\|_{L^1_{[-2,2]}H_x^{r-1}} \big),
	\end{equation}
	and the Strichartz estimates
	\begin{equation}\label{QL1}
		\| \left<\nabla \right>^k f\|_{L^2_{[-2,2]}L^\infty_x} \leq C \big( \|f_0\|_{H_x^r}+ \|f_1\|_{H_x^{r-1}} +\|F\|_{L^1_{[-2,2]}H_x^{r-1}} \big), \quad k<r-1,
	\end{equation}
	and the similar estimates also hold when we replace $\left<\nabla \right>^k$ by $\left<\nabla \right>^{k-1}d$. Additionally, if $2<r\leq s$, we have
	\begin{equation}\label{QL2}
		\| d f\|_{L^2_{[-2,2]}\dot{B}^0_{\infty,2}} \leq C \big( \|f_0\|_{H_x^r}+ \|f_1\|_{H_x^{r-1}} +\|F\|_{L^1_{[-2,2]}H_x^{r-1}} \big).
	\end{equation}
\end{corollary}
\begin{proof}
	By using Proposition \ref{r5} and Duhamel's principle, we can obtain \eqref{QL0} and \eqref{QL1} directly. To prove \eqref{QL2}, operating $P_j$ on \eqref{wform5}, it follows
		\begin{equation}\label{wform6}
		\begin{cases}
			& \square_{\mathbf{g}} P_jf=p_j F+ [\square_{\mathbf{g}}, P_j]f,
			\\
			&(P_jf, \partial_t P_jf)|_{t=t_0}=(P_jf_0, P_jf_1).
		\end{cases}
	\end{equation}
For $r>2$, by using \eqref{QL1} on \eqref{wform6}, we therefore get
\begin{equation}\label{ise1}
	\begin{split}
		 \| P_j df\|_{L^2_{[-2,2]} L^\infty_x}
		\lesssim  & \|P_j f_0 \|_{H_x^{r}}+\|P_j f_1 \|_{H_x^{r-1}}
		\\
		&+\|P_j F \|_{L^1_{[-2,2]} H_x^{r-1}}+ \|[\square_{\mathbf{g}}, P_j]f\|_{L^1_{[-2,2]} H_x^{r-1}}.
	\end{split}
\end{equation}	
Taking the $l^2$ sum on \eqref{ise1} for $j$, and using Lemma \ref{yx}, then we derive that
\begin{equation*}\label{ise2}
	\begin{split}
		\| d f\|_{L^2_{[-2,2]} \dot{B}^0_{\infty,2}}
		\lesssim  & \|f_0 \|_{H_x^{r}}+\|f_1 \|_{H_x^{r-1}}+\| F \|_{L^1_{[-2,2]} H_x^{r-1}}
		\\
		&+ \| d f\|_{L^1_{[-2,2]} L_x^\infty} \|d \mathbf{g}\|_{L^\infty_{[-2,2]} {H}_x^{r-1}}+\| d \mathbf{g}\|_{L^1_{[-2,2]} L_x^\infty}\|d f\|_{L^\infty_{[-2,2]}{H}_x^{r-1}}
		\\
		\lesssim  & \|f_0 \|_{H_x^{r}}+\|f_1 \|_{H_x^{r-1}}+\| F \|_{L^1_{[-2,2]} H_x^{r-1}}.
	\end{split}
\end{equation*}	
Above, we also use \eqref{5021}, \eqref{QL0}, and $2< r \leq s$.
\end{proof}
We also need the following Strichartz estimate:
\begin{proposition}\label{r3}
Suppose that $(h,\bu, \bw) \in \mathcal{{H}}$ and $\Re(h,\bu, \bw)\leq 2 \epsilon_1$. Let $\mathbf{g}$ is denoted in \eqref{AMd3}.
For any $1 \leq r \leq s+1$, and for each $t_0 \in [-2,2]$, the linear, non-homogeneous equation
\begin{equation*}\label{wform}
\begin{cases}
	& \square_{\mathbf{g}} f=\mathbf{g}^{0\alpha} \partial_\alpha F, \qquad (t,x) \in (t_0,2]\times \mathbb{R}^3,
	\\
	&(f,\partial_t f)|_{t=t_0}=( f_0,F(t_0,\cdot) ),
\end{cases}
\end{equation*}
admits a solution $f \in C([-2,2],H_x^r) \times C^1([-2,2],H_x^{r-1})$ and the following estimates holds:
\begin{equation}\label{lw0}
\begin{split}
	\| f\|_{L_{[-2,2]}^\infty H_x^r}+ \|\partial_t f\|_{L_{[-2,2]}^\infty H_x^{r-1}} \lesssim & \|f_0\|_{H_x^r}
	+\| F \|_{L^1_{[-2,2]}H_x^r},
\end{split}
\end{equation}
and provided $k<r-1$,
\begin{equation}\label{lw1}
\| \left<\nabla \right>^k f\|_{L^2_{[-2,2]}L^\infty_x} \lesssim \|f_0\|_{H_x^r}+\| F \|_{L^1_{[-2,2]}H_x^r},
\end{equation}
and the similar estimates also hold when we repace $\left<\nabla \right>^k$ by $\left<\nabla \right>^{k-1}d$. Additionally, if $2<r \leq s$, we have
	\begin{equation}\label{lw2}
	\| d f\|_{L^2_{[-2,2]}\dot{B}^0_{\infty,2}} \leq C \big( \|f_0\|_{H_x^r} +\|F\|_{L^1_{[-2,2]}H_x^{r}} \big).
\end{equation}
\end{proposition}
\begin{proof}
Let $V$ satisfy the the linear, homogenous equation
\begin{equation}\label{Vf}
\begin{cases}
	& \square_{\mathbf{g}} V=0,
	\\
	&V(t,x)|_{t=t_0}=f_0, \quad \partial_t V(t,x)|_{t=t_0}=0,
\end{cases}
\end{equation}
Let $K$ satisfy the the linear, nonhomogenous equation
\begin{equation}\label{Qf}
\begin{cases}
	& \square_{\mathbf{g}} K=\mathbf{g}^{0\alpha} \partial_\alpha F,
	\\
	&K(t,x)|_{t=t_0}=0, \quad \partial_t K(t,x)|_{t=t_0}=F(t_0,x),
\end{cases}
\end{equation}
By the superposition principle, and referring \eqref{Vf} and \eqref{Qf}, then $f= V+K$ satisfies
\begin{equation*}
\begin{cases}
	& \square_{\mathbf{g}} f=\mathbf{g}^{0\alpha}\partial_\alpha F,
	\\
	&(f, \partial_t f)|_{t=t_0}=(f_0, F(t_0,x)).
\end{cases}
\end{equation*}
Due to Proposition \ref{r5}, it follows
\begin{equation}\label{Vf1}
\| V\|_{L^\infty_{[-2,2]}H_x^r}+ \| \partial_t V\|_{L^\infty_{[-2,2]}H_x^{r-1}} \leq C  \|f_0\|_{H_x^r},
\end{equation}
and
\begin{equation}\label{Vf2}
\| \left<\nabla \right>^k V\|_{L^2_{[-2,2]}L^\infty_x} \leq  C\|f_0\|_{H_x^r}, \quad k<r-1.
\end{equation}
Using Corollary \ref{r8}, we have
\begin{equation}\label{Vf3}
	\| d V\|_{L^2_{[-2,2]}\dot{B}^{0}_{\infty,2}} \leq C \|f_0\|_{H_x^r}, \quad 2<r\leq s.
\end{equation}
By Lemma \ref{LD} and Proposition \ref{r5}, we can bound $K$ by
\begin{equation}\label{QE}
\| K\|_{L^\infty_{[-2,2]} H_x^r}+ \| \partial_t K\|_{L^\infty_{[-2,2]} H_x^{r-1}} \leq C  \|F\|_{L^1_{[-2,2]}H_x^{r}},
\end{equation}
and
\begin{equation}\label{SQ}
\| \left<\nabla \right>^k K\|_{L^2_{[-2,2]} L^\infty_x} \leq C \|F\|_{L^1_{[-2,2]}H_x^{r}}, \quad k<r-1.
\end{equation}
Using Corollary \ref{r8} again, we can obtain
\begin{equation}\label{Vf4}
	\| d V\|_{L^2_{[-2,2]}\dot{B}^{0}_{\infty,2}} \leq C \|F\|_{L^1_{[-2,2]}H_x^{r}}, \quad 2<r\leq s.
\end{equation}
Adding \eqref{Vf1} and \eqref{QE}, we get \eqref{lw0}. Adding \eqref{Vf2} and \eqref{SQ}, we get \eqref{lw1}.  Adding \eqref{Vf3} and \eqref{Vf4}, we get \eqref{lw2}. At this stage, we have finished the proof of Proposition \ref{r3}.
\end{proof}
Based on Proposition \ref{r5} and \ref{r3}, we can obtain the Strichartz estimates for the solutions of \eqref{REE} as follows.
\begin{proposition}\label{r6}
Suppose $(h, \bu, \bw) \in \mathcal{H}$ and $\Re(h, \bu, \bw)\leq 2 \epsilon_1$. Let $2<s_0<s\leq\frac52$. Let $\bu_{+}$ be defined in \eqref{De}. Then for $k<s-1$, we have
\begin{align}\label{fgh1}
	\|(\left< \nabla \right>^{k} h, \left< \nabla \right>^{k-1} \bu_{+})\|_{L^2_{[-2,2]} \dot{B}^{0}_{\infty, 2}}
	\lesssim  & \|(u^0-1, \mathring{\bu}, h) \|_{L^\infty_{[-2,2]} H_x^{s}}
	+\| \bw \|_{L^2_{[-2,2]} H_x^{\frac{3}{2}+}},
\end{align}
and the similar estimates hold if we replace $\left< \nabla \right>^{k} $ with $\left< \nabla \right>^{k-1} d$.
\end{proposition}
\begin{proof}
By \eqref{WTe} and \eqref{fdr}, we can obtain
\begin{equation}\label{fcr}
\begin{cases}
	& \square_{\mathbf{g}} h= D,
	\\
	&\square_{\mathbf{g}} \bu_{+}=\mathbf{g}^{0\alpha}\partial_\alpha \mathbf{Z}+ \bQ+\mathbf{B}.
\end{cases}
\end{equation}
By using \eqref{err1}, \eqref{fdwB}, \eqref{err}, Lemma \ref{cj}, and \eqref{5021}, we therefore get
\begin{equation}\label{mlo}
\begin{split}
	 \|D\|_{L^1_{[-2,2]}H_x^{s-1}} + \|\bQ\|_{L^1_{[-2,2]}H_x^{s-1}} \lesssim  &  \|(d \bu, dh)\|_{L^1_{[-2,2]} L_x^\infty} \|(d \bu, dh)\|_{L^\infty_{[-2,2]}H_x^{s-1}},
\\
\lesssim &  \|(d \bu, dh)\|_{L^1_{[-2,2]} L_x^\infty} \|(u^0-1, \mathring{\bu}, h) \|_{L^\infty_{[-2,2]} H_x^{s}}
\\
\lesssim &  \|(u^0-1, \mathring{\bu}, h) \|_{L^\infty_{[-2,2]} H_x^{s}}.
\end{split}
\end{equation}
By using \eqref{fdwB} and Lemma \ref{cj}, it yields
\begin{equation}\label{mlb}
	\begin{split}
		 \|\bB\|_{L^1_{[-2,2]}H_x^{s-1}}  \lesssim  & \| \mathbf{T}\bu_{-}\|_{L^2_{[-2,2]}H_x^{s}}+  \|(d \bu_{-}, d \bu, dh)\|_{L^1_{[-2,2]} L_x^\infty} \|(d \bu, d \bu_{-}, dh)\|_{L^\infty_{[-2,2]}H_x^{s-1}}.
	\end{split}
\end{equation}
Using \eqref{De0} and \eqref{5021}, it follows
\begin{equation}\label{mao}
	\begin{split}
		\|\bu_{-}\|_{H^{\frac52+}_x} \lesssim  \| \bw \|_{ H^{\frac32+}_x }(1+\| h \|_{ H^{s}_x })\lesssim  \| \bw \|_{ H^{\frac32+}_x }.
	\end{split}
\end{equation}
Due to \eqref{tue2} and \eqref{5021}, we can compute out
\begin{equation}\label{fcv}
	\begin{split}
		\| \mathbf{T}\bu_{-} \|_{ H_x^s} \lesssim  & \| \bw \|_{ H^{\frac32+}_x } ( \|(u^0-1, \mathring{\bu}, h) \|_{ H_x^{s}}+\|(u^0-1, \mathring{\bu}, h) \|^2_{H_x^{s}})
		\\
		\lesssim  & \| \bw \|_{ H^{\frac32+}_x }.
	\end{split}
\end{equation}
From \eqref{fdwZ}, so we get
\begin{equation}\label{fce}
	\begin{split}
		\| \mathbf{Z} \|_{ H_x^s} \lesssim  \| \mathbf{T}\bu_{-} \|_{ H_x^s} \lesssim \| \bw \|_{ H^{\frac32+}_x }.
	\end{split}
\end{equation}
Combing \eqref{fcv}, \eqref{mao}, \eqref{mlb}, and using \eqref{5021}, we have
\begin{equation}\label{mld}
	\begin{split}
		\|\bB\|_{L^1_{[-2,2]}H_x^{s-1}}  \lesssim  & \| \bw \|_{ L^2_{[-2,2]} H^{\frac32+}_x } ( 1+ \|(u^0-1, \mathring{\bu}, h) \|_{ L^\infty_{[-2,2]} H_x^{s}}+\|(u^0-1, \mathring{\bu}, h) \|^2_{L^\infty_{[-2,2]} H_x^{s}}  )
		\\
		&
		+  \|(d \bu, dh)\|_{L^2_{[-2,2]} L_x^\infty} \|(u^0-1, \mathring{\bu}, h) \|_{L^\infty_{[-2,2]} H_x^{s}}
		\\
		\lesssim & \| \bw \|_{ L^2_{[-2,2]} H^{\frac32+}_x }+\|(u^0-1, \mathring{\bu}, h) \|_{L^\infty_{[-2,2]} H_x^{s}}.
	\end{split}
\end{equation}
Using Propostion \ref{r8} and \ref{r3} for Equation \eqref{fcr}, combining the bounds \eqref{mlo}, \eqref{fce}, and \eqref{mld}, we can prove \eqref{fgh1}. At this stage, we have finished the proof of Proposition \ref{r6}.
\end{proof}

\begin{remark}
We expect that Proposition \ref{r6} yields a Strichartz estimate of solutions with  sharp regularity of the velocity, enthalpy, and vorticity.
\end{remark}
\begin{proposition}\label{r4}
Suppose $(h,\bu, \bw) \in \mathcal{H}$ and $\Re(h,\bu, \bw)\leq 2 \epsilon_1$. Let $2<s_0<s<\frac52$ and $\delta \in [0,s-2)$. Let $\bu_{+}$ be defined in \eqref{De}-\eqref{De0}. Then we have the following Strichartz estimates
\begin{equation}\label{strr}
\|d \bu, dh, d \bu_{+}\|_{L^2_{[-2,2]} C^{\delta}_x}+\|d \bu_{+}, d\bu, dh\|_{L^2_{[-2,2]} \dot{B}^{s_0-2}_{\infty,2}} \leq \epsilon_2,
\end{equation}
and the energy estimates
\begin{equation}\label{eef}
\| (h,u^0-1,\mathring{\bu})\|_{L^\infty_{[-2,2]} H_x^s} +\| \bw\|_{L^\infty_{[-2,2]} H_x^{s_0}}\leq \epsilon_2.
\end{equation}
\end{proposition}
\begin{proof}
Taking $k=s_0-1$ in Proposition \ref{r6}, we then obtain
\begin{equation}\label{epp0}
\begin{split}
	\|dh, d\bu_{+} \|_{L^2_{[-2,2]} \dot{B}^{s_0-2}_{\infty, 2}}
	\lesssim  & \|(u^0-1, \mathring{\bu}, h) \|_{L^\infty_{[-2,2]} H_x^{s}}
	+\| \bw \|_{L^\infty_{[-2,2]} H_x^{\frac{3}{2}+}}.
\end{split} 		
\end{equation}
Using Proposition \ref{r6} again, for $\delta \in [0,s-2)$, we shall get
\begin{equation}\label{epp1}
\begin{split}
	\|dh, d\bu_{+} \|_{L^2_{[-2,2]}C^{\delta}_x}
	\lesssim  & \|(u^0-1, \mathring{\bu}, h) \|_{L^\infty_{[-2,2]} H_x^{s}}
	+\| \bw \|_{L^\infty_{[-2,2]} H_x^{\frac{3}{2}+}}.
\end{split} 		
\end{equation}
Due to \eqref{De0} and \eqref{5021}, it follows
\begin{equation}\label{maox}
	\begin{split}
	\|\nabla \bu_{-}\|_{L^2_{[-2,2]} \dot{B}^{s_0-2}_{\infty, 2}}\lesssim &	\|\nabla \bu_{-}\|_{L^2_{[-2,2]}H^{s_0-\frac12}_x} 
	\\
	\lesssim &	\|\mathrm{e}^{-h}\bW\|_{L^2_{[-2,2]}H^{s_0-\frac32}_x} 
	\\
	\lesssim  & \| \bw \|_{ L^2_{[-2,2]} H^{s_0}_x }(1+\| h \|_{ L^\infty_{[-2,2]}H^{s}_x })
	\\
	\lesssim  & \| \bw \|_{ L^2_{[-2,2]} H^{s_0}_x }.
	\end{split}
\end{equation}
By using Theorem \ref{VE} and $(h, \bu, \bw) \in \mathcal{H}$, it follows that
\begin{equation}\label{epp8}
\begin{split}
	 \|(u^0-1, \mathring{\bu}, h) \|_{L^\infty_{[-2,2]} H_x^{s}}
	+\| \bw \|_{L^\infty_{[-2,2]} H_x^{s_0}}
	\leq  & C (\epsilon_3+ \epsilon^2_3) \exp \left\{ C (\epsilon_3+ \epsilon^2_3)\epsilon_2 \mathrm{e}^{C (\epsilon_3+ \epsilon^2_3)\epsilon_2 }  \right\}
	\\
	\leq & \epsilon_2.
\end{split} 		
\end{equation}
Therefore, by \eqref{epp8}, \eqref{epp1}, \eqref{fcv}, \eqref{maox}, and \eqref{a0}, it turns out
\begin{equation}\label{epp9}
\begin{split}
	\|d h, d\bu \|_{L^2_{[-2,2]} \dot{B}^{s_0-2}_{\infty, 2}} \leq & C(\|(u^0-1, \mathring{\bu}, h) \|_{L^\infty_{[-2,2]} H_x^{s}}
	+\| \bw \|_{L^\infty_{[-2,2]} H_x^{s_0}})
	\\
	\leq  & C (\epsilon_3+ \epsilon^2_3) \exp \left\{ C (\epsilon_3+ \epsilon^2_3)\epsilon_2 \mathrm{e}^{C (\epsilon_3+ \epsilon^2_3)\epsilon_2 }  \right\}
	\\
	\leq & \epsilon_2.
\end{split} 		
\end{equation}
Similarly, we can also prove that
\begin{equation}\label{epp10}
\begin{split}
	\|dh, d\bu \|_{L^2_{[-2,2]} C^\delta_x}
	\lesssim   \epsilon_2.
\end{split} 		
\end{equation}
To summarize \eqref{epp1}, \eqref{epp8}, \eqref{epp9} and \eqref{epp10}, we have finished the proof of Proposition \ref{r4}.
\end{proof}
\section{Strichartz estimate for linear waves}\label{spp}
In this part, our goal is to prove Proposition \ref{r5}. Following Smith-Tataru's paper \cite{ST}, we give a proof for Proposition \ref{r5}. The method relies on a wave-packet construction for the metric $\mathbf{g}$ in \eqref{AMd3}. We mention that the use of wave-packets is introduced by Smith \cite{Sm}, and developed by Wolff \cite{Wo} and Smith-Tataru \cite{ST}. Following this idea, so we can represent approximate solutions to the linear equation as a square summable superposition of wave packets, and the wave packets are localized in phase space. As a result, we can obtain Strichartz estimates of linear wave equations with sufficient regularity of the foliations. 

We also set $(h,\bu, \bw) \in \mathcal{{H}}$ and $ \Re\leq 2 \epsilon_1$, where $\mathcal{{H}}$ and $\Re$ are defined in \eqref{401}-\eqref{403} and \eqref{500} respectively. Then $\|d \mathbf{g}\|_{L^2_{[-2,2]} L^\infty_x} \leq \epsilon_0$ and $\|d \mathbf{g}\|_{L^2_{[-2,2]} C^\delta_x} \leq \epsilon_0$, where $\epsilon_0$ has been stated in \eqref{a0}.

Let $P_\lambda$ be the Littlewood-Paley operator with frequency $\lambda$. For $\lambda \geq 1$, we consider the smooth metric
\begin{equation*}
	\mathbf{g}_{\lambda}= \sum_{\lambda' < \lambda}P_{\lambda'} \mathbf{g}.
\end{equation*}
\subsection{Proof of Proposition \ref{r5}} We first state the following result in a phase space.
\begin{proposition}\label{AA1}
	Suppose that $(h, \bu, \bw) \in \mathcal{{H}}$ and $\Re(h,\bu,\bw)\leq 2 \epsilon_1$. Suppose $f$ satisfy
	\begin{equation}\label{linearC}
		\begin{cases}
			\square_{\mathbf{g}} f=0, \quad (t,x)\in [-2,2]\times \mathbb{R}^3,\\
			(f, \partial_t f)|_{t=t_0}=(f_0, f_1).
		\end{cases}
	\end{equation}
	Then for each $(f_0,f_1) \in H^1 \times L^2$ there exists a function $f_{\lambda} \in C^\infty([-2,2]\times \mathbb{R}^3)$, with
	\begin{equation*}
		\mathrm{supp} \widehat{f_\lambda(t,\cdot)} \subseteq \{ \xi\in\mathbb{R}^3: \frac{\lambda}{8} \leq |\xi| \leq 8\lambda \},
	\end{equation*}
	such that
	\begin{equation}\label{Yee}
		\begin{cases}
			& \| \square_{\mathbf{g}_\lambda} f_{\lambda} \|_{L^1_{[-2,2]} L^2_x} \lesssim \epsilon_0 (\| {P_\lambda } f_0\|_{H^1}+\| {P_\lambda } f_1 \|_{L^2} ),
			\\
			& f_\lambda(t_0)=P_\lambda f_0, \quad \partial_t f_{\lambda} (t_0)=P_{\lambda} f_1.
		\end{cases}
	\end{equation}
	Additionally, the following Strichartz estimates holds:
	\begin{equation}\label{Ase}
		\| f_{\lambda} \|_{L^2_{[-2,2]} L^\infty_x} \lesssim \epsilon_0^{-\frac{1}{2}} (\ln \lambda)^{\frac12}
		( \| {P_\lambda } f_0 \|_{H^1} + \| {P_\lambda } f_1 \|_{L^2} ).
	\end{equation}
\end{proposition}
\begin{remark}\label{rel}
	In fact, Proposition \ref{AA1} tells us we can find a good approximate solution $f_\lambda$ for the problem
	\begin{equation*}
		\begin{cases}
			\square_{\mathbf{g}_\lambda} f_\lambda=0, \quad (t,x)\in [-2,2]\times \mathbb{R}^3,\\
			(f_\lambda, \partial_t f_\lambda)|_{t=t_0}=(P_\lambda f_0, P_\lambda f_1).
		\end{cases}
	\end{equation*}
	For $\epsilon_0 \lambda\leq 1$, the result in Proposition \ref{AA1} is almost trivial. Indeed, we can set $f_\lambda=P_\lambda f$, where $f$ is the exact solution of
	\begin{equation}\label{linearD}
		\begin{cases}
			\square_{\mathbf{g}_\lambda} f=0, \quad (t,x)\in [-2,2]\times \mathbb{R}^3,\\
			(f, \partial_t f)|_{t=t_0}=(P_\lambda f_0, P_\lambda f_1).
		\end{cases}
	\end{equation}
	Using energy estimates for \eqref{linearD}, we can see
	\begin{equation}\label{eet}
		\|df\|_{L^\infty_{[-2,2]} L^2_x} \lesssim \| P_\lambda f_0\|_{H^1}+\|P_\lambda f_1\|_{L^2}.
	\end{equation}
	Moreover, for $\mathbf{g}^{00}_\lambda=-1$, so we have
	\begin{equation*}
		\begin{split}
			\| \square_{\mathbf{g}_\lambda} f_\lambda \|_{L^2_{[-2,2]} L^2_x} 
			\lesssim & 
			\| [ \mathbf{g}^{\alpha i}_\lambda, \partial_\alpha P_\lambda ] \partial_i f \|_{ L^2_{[-2,2]} L^2_x } 
			+ \| P_\lambda (\partial_\alpha \mathbf{g}^{\alpha i}_\lambda ) \partial_i f \|_{ L^2_{[-2,2]} L^2_x }
			\\
			\lesssim & \| d \mathbf{g}_\lambda \|_{L^2_{[-2,2]} L^\infty_x} \| d {f} \|_{L^\infty_{[-2,2]} L^2_x}
			\\
			\lesssim & \epsilon_0 ( \| P_\lambda f_0\|_{H^1}+\|P_\lambda f_1\|_{L^2} ).
		\end{split}
	\end{equation*}
	The Strichartz estimate \eqref{Ase} follows from Sobolev imbeddings and \eqref{eet}. Hence, we will restrict to establishing Proposition \ref{AA1} in the case that
	\begin{equation*}
		\epsilon_0 \lambda \gg 1.
	\end{equation*}
\end{remark}
Based on Proposition \ref{AA1}, we can prove Proposition \ref{r5} as follows. 
\begin{proof}[Proof of Proposition \ref{r5} by using Proposition \ref{AA1}]
	We first prove the case of $r=1$ by using Proposition \ref{AA1}. Based on the result for $r=1$, combining with Kato-Ponce commutator estimates and Gagliardo-Nirenberg inequality, we are able to give a proof for other case $1<r\leq 1+s$.

\textbf{Step 1:} The case $r=1$. Using basic energy estimates for \eqref{wform1}, we have
\begin{equation*}
	\begin{split}
		\| \partial_t f \|_{L^2_x} + \|\nabla f \|_{L^2_x}  \lesssim & \ (\| f_0\|_{H^1}+ \| f_1\|_{L^2}) \exp(\int^t_0 \| d \mathbf{g} \|_{L^\infty_x} d\tau)
		\\
		\lesssim & \ \| f_0\|_{H^1}+ \| f_1\|_{L^2}.
	\end{split}
\end{equation*}
Then the Cauchy problem \eqref{linear} holds a unique solution $f \in C([-2,2],H_x^1)$ and $\partial_t f \in C([-2,2],L_x^2)$. It remains to show that the solution $f$ also satisfies the Strichartz estimate \eqref{SL}.

Without loss of generality, we take $t_0=0$ in Proposition \ref{r5}. For any given initial data $(f_0,f_1) \in H^1 \times L^2$, and $t_0 \in [-2,2]$, we take a Littlewood-Paley decomposition
\begin{equation*}
	f_0=\textstyle{\sum}_{\lambda}P_{\lambda}f_0, \qquad  f_1=\textstyle{\sum}_{\lambda}P_{\lambda}f_1,
\end{equation*}
and for each $\lambda$ we take the corresponding $f_{\lambda}$ as in Proposition \ref{AA1}.  If we set
\begin{equation}\label{fstar}
	f^*=\textstyle{\sum}_{\lambda}f_{\lambda},
\end{equation}
then $f^*$ matches the initial data $(f_0,f_1)$ at the time $t=t_0$, and also satisfies the Strichartz estimates \eqref{Ase}. We claim that $f^*$ is also an approximate solution for $\square_{\mathbf{g}}$ in the sense that
\begin{equation}\label{QQ}
	\| \square_{\mathbf{g}} f^* \|_{L^1_{[-2,2]} L^2_x} \lesssim \epsilon_0(\| f_0 \|_{H^1}+\| f_1 \|_{L^2} ).
\end{equation}
We can derive \eqref{QQ} by using the decomposition
\begin{equation}\label{er1}
	\square_{\mathbf{g}} f^*=\textstyle{\sum_{\lambda}} \square_{\mathbf{g}_\lambda} f_\lambda+ \textstyle{\sum_{\lambda}} \square_{\mathbf{g}-\mathbf{g}_\lambda} f_\lambda.
\end{equation}
In \eqref{er1}, the first one $\textstyle{\sum_{\lambda}} \square_{\mathbf{g}_\lambda} f_\lambda$ can be controlled by Proposition \ref{AA1}. As for the second one $\textstyle{\sum_{\lambda}} \square_{\mathbf{g}-\mathbf{g}_\lambda} f_\lambda$, noting $\mathbf{g}^{00}=-1$, we then have
\begin{equation*}
	\begin{split}
		\textstyle{\sum_{\lambda}} \square_{\mathbf{g}-\mathbf{g}_\lambda} f_\lambda= \textstyle{\sum_{\lambda}} ({\mathbf{g}-\mathbf{g}_\lambda}) \nabla  df_\lambda.
	\end{split}
\end{equation*}
Due to H\"older inequality, we obtain
\begin{equation}\label{y1}
	\begin{split}
		\| \textstyle{\sum_{\lambda}} \square_{\mathbf{g}-\mathbf{g}_\lambda} f_\lambda \|_{L^2_x} \lesssim \mathrm{sup}_{\lambda} \left(\lambda \| {\mathbf{g}-\mathbf{g}_\lambda} \|_{L^\infty_x} \right) \left( \sum_{\lambda}  \| df_\lambda \|^2_{L^2_x} \right)^{\frac12}.
	\end{split}
\end{equation}
On the other hand, we have
\begin{equation}\label{y2}
	\begin{split}
		\mathrm{sup}_{\lambda} \left(\lambda \| {\mathbf{g}-\mathbf{g}_\lambda} \|_{L^\infty_x} \right) \lesssim & \ \mathrm{sup}_{\lambda} \big(\lambda \sum_{\mu > \lambda}\| \mathbf{g}_\mu \|_{L^\infty_x} \big)
		\\
		\lesssim & \ \mathrm{sup}_{\lambda} \big(\lambda \sum_{\mu > \lambda}\mu^{-(1+\delta)}\|d \mathbf{g}_\mu \|_{C^\delta_x} \big)
		\\
		\lesssim & \ \|d \mathbf{g} \|_{C^\delta_x}(\textstyle{\sum_{\mu}} \mu^{-\delta}) \lesssim \ \|d \mathbf{g} \|_{C^\delta_x}.
	\end{split}
\end{equation}
For $f_\lambda$ is an approximate solution for the problem
\begin{equation*}
	\begin{cases}
		\square_{\mathbf{g}_\lambda} v=0, \quad (t,x)\in [-2,2]\times \mathbb{R}^3,\\
		(v, \partial_t v)|_{t=t_0}=(P_\lambda f_0, P_\lambda f_1).
	\end{cases}
\end{equation*}
Due to energy estimate and \eqref{Yee}, it yields
\begin{equation}\label{eem}
	\begin{split}
		\|df_\lambda \|_{L^2_x} \lesssim & \|P_\lambda f_0\|_{H^1}+ \| P_\lambda f_1 \|_{L^2} + \| \square_{\mathbf{g}_\lambda} f_\lambda \|_{L^1_{[-2,2]} L^2_x}
		\\
		\lesssim  & \|P_\lambda f_0\|_{H^1}+ \| P_\lambda f_1 \|_{L^2} .
	\end{split}
\end{equation} 
Gathering \eqref{y1}, \eqref{y2}, and \eqref{eem}, we can prove that
\begin{equation*}
	\begin{split}
		\| \textstyle{\sum_{\lambda}} \square_{\mathbf{g}-\mathbf{g}_\lambda} f_\lambda \|_{L^1_{[-2,2]} L^2_x} \lesssim \epsilon_0(\| f_0 \|_{H^1}+\| f_1 \|_{L^2} ),
	\end{split}
\end{equation*}
where we also use $\|d \mathbf{g} \|_{L^2_{[-2,2]} C^\delta_x} \lesssim \epsilon_0$. Hence, we have proved \eqref{QQ}. But $f^*$ is not the exact solution to $\square_{\mathbf{g}} f=0$ with the Cauchy data $(f_0,f_1)$. So we seek another function $\mathbf{M}F$ satisfying
\begin{equation}\label{AQ1}
	\begin{cases}
		&\square_{\mathbf{g}} \mathbf{M}F= - \square_{\mathbf{g}} f^*,
		\\
		& (\mathbf{M}F, \partial_t \mathbf{M}F)=(0,0).
	\end{cases}
\end{equation}
For given $F=- \square_{\mathbf{g}} f^* \in L^1_{[-2,2]} L^2_x$, we now form the function
\begin{equation*}
	F^{(1)}=\int^t_0 f^{(1)}(\tau;t,x)d\tau,
\end{equation*}
where $f^{(1)}(\tau;t,x)$ is the approximate solution of $\square_{\mathbf{g}} \tilde{f}=0$ and $f^{(1)}$ is formed like $f^*$(cf.  \eqref{fstar}) with the Cauchy data
\begin{equation*}
	\begin{cases}
		&\square_{\mathbf{g}} f^{(1)}=0, \qquad (t, x) \in (\tau,2] \times \mathbb{R}^3,
		\\
		&		f^{(1)}(\tau,t,x)|_{t=\tau}=0, \quad \partial_t f^{(1)}(\tau,t,x)|_{t=\tau}=F(\tau,\cdot).
	\end{cases}
\end{equation*}
By calculating
\begin{equation*}
	\begin{split}
		\partial_t F^{(1)}=& \int^t_0 \partial_t f^{(1)}(\tau;t,x) d\tau+ f^{(1)}(t;t,x)=\int^t_0 \partial_t f^{(1)}(\tau;t,x) d\tau,
		\\
		\partial_{tt} F^{(1)}=& \int^t_0 \partial_{tt} f^{(1)}(\tau;t,x) d\tau+ \partial_t f^{(1)}(\tau;t,x)=\int^t_0 \partial_{tt} f^{(1)}(t;t,x) d\tau+F(t,x),
		\\
		\partial_{ti} F^{(1)}=& \int^t_0 \partial_{ti} f^{(1)}(t;t,x) d\tau, \quad i=1,2,3,
		\\
		\partial_{ij} F^{(1)}=& \int^t_0 \partial_{ij} f^{(1)}(t;t,x) d\tau, \quad i,j=1,2,3,
	\end{split}	
\end{equation*}
we then have
\begin{equation*}
	\square_{\mathbf{g}} F^{(1)}=\int^t_0 \square_{\mathbf{g}} f^{(1)}(\tau;t,x) d\tau+F.
\end{equation*}
By using \eqref{QQ}, we derive that
\begin{equation*}
	\| \square_{\mathbf{g}} F^{(1)}-F\|_{L^1_{[-2,2]} L^2_x} \lesssim \| \square_{\mathbf{g}}f^{(1)}(\tau;t,x)\|_{L^1_{[-2,2]} L^2_x} \lesssim \epsilon_0 \| F \|_{L^1_{[-2,2]} L^2_x}.
\end{equation*}
For $F^{(m)}, m\geq 2$, we perform an iteration as follows. Let
\begin{equation*}
	F^{(2)}=\int^t_0 f^{(2)}(\tau;t,x)d\tau,
\end{equation*}
where $f^{(2)}(\tau;t,x)$ is the approximate solution of $\square_{\mathbf{g}} \tilde{f}=0$ and $f^{(2)}$ is formed like $f^*$ (cf. \eqref{fstar}) with Cauchy data
\begin{equation*}
	\begin{cases}
		&\square_{\mathbf{g}} f^{(2)}=0,\qquad (t,\bx) \in (\tau,2] \times \mathbb{R}^3,
		\\
		&		f^{(2)}(\tau;t,x)|_{t=\tau}=0, \quad \partial_t f^{(2)}(\tau;t,x)|_{t=\tau}=F-\square_{\mathbf{g}} F^{(1)}.
	\end{cases}
\end{equation*}
Then we can prove that
\begin{equation*}
	\|\square_{\mathbf{g}}F^{(2)}- (F-\square_{\mathbf{g}} F^{(1)}) \|_{L^1_{[-2,2]} L^2_x} \lesssim \epsilon_0 \| F-\square_{\mathbf{g}} F^{(1)} \|_{L^1_{[-2,2]} L^2_x} \lesssim \epsilon^2_0 \| F \|_{L^1_{[-2,2]} L^2_x}.
\end{equation*}
We rewrite the above equality as
\begin{equation*}
	\|\square_{\mathbf{g}}(F^{(2)}+ F^{(1)})-F \|_{L^1_{[-2,2]} L^2_x} \lesssim \epsilon_0 \| F-\square_{\mathbf{g}} F^{(1)} \|_{L^1_{[-2,2]} L^2_x} \lesssim \epsilon^2_0 \| F \|_{L^1_{[-2,2]} L^2_x}.
\end{equation*}
For $m \geq 2$, we can set
\begin{equation*}
	F^{(m)}=\int^t_0 f^{(m)}(\tau;t,x)d\tau,
\end{equation*}
where $f^{(m)}(\tau;t,x)$ is the approximate solution of $\square_{\mathbf{g}} \tilde{f}=0$ and $f^{(m)}$ is formed like $f^*$ (cf. \eqref{fstar}) with Cauchy data
\begin{equation*}
	\begin{cases}
		&\square_{\mathbf{g}} f^{(m)}=0,\qquad (t, x) \in (\tau,2] \times \mathbb{R}^3,
		\\
		&		f^{(m)}(\tau;t,x)|_{t=\tau}=0, \quad \partial_t f^{(2)}(\tau;t,x)|_{t=\tau}=F-\square_{\mathbf{g}} (F^{(1)}+F^{(2)}+\cdots F^{(m-1)}).
	\end{cases}
\end{equation*}
we therefore obtain that
\begin{equation}\label{AQ4}
	\|\square_{\mathbf{g}}(F^{(m)}+F^{(m-1)}+\cdots F^{(1)})- F \|_{L^1_t L^2_x} \lesssim \epsilon^m_0 \| F \|_{L^1_t L^2_x}.
\end{equation}
Seeing \eqref{AQ4}, by using the contraction principle, there is a limit $$  \mathbf{M}F=\lim_{m\rightarrow \infty} (F^{(m)}+F^{(m-1)}+\cdots F^{(1)}). $$ Hence, we get
\begin{equation}\label{AQ0}
	\square_{\mathbf{g}} \mathbf{M}F = - \square_{\mathbf{g}} f^*,  \quad (\mathbf{M}F, \partial_t \mathbf{M}F)=(0,0).
\end{equation}
Gathering \eqref{AQ1} and \eqref{AQ0}, we can write the solution $f$ in the form
\begin{equation*}
	f= f^*+ \mathbf{M}F,
\end{equation*}
where $f^*$ is the approximation solution formed in \eqref{fstar} for initial data $(f_0,f_1)$ specified at time $t=0$, and it satisfies the Strichartz estimates \eqref{Ase}, and
\begin{equation*}
	\| F\|_{L^1_{[-2,2]} L^2_x} \lesssim \epsilon_0 ( \| f_0 \|_{H^1}+\| f_1\|_{L^2}).
\end{equation*}
The Strichartz estimates of $\mathbf{M}F$ now follow since they holds for each $f^{(m)}(\tau;t,x)$, $\tau \in [0,t]$. Therefore, the estimate \eqref{SL} holds.

\textbf{Step 2:} The case $r=2$. We also consider $t_0=0$. Given data $(f_0,f_1) \in H^2\times H^1$, we find a solution of the form $f=\left< \nabla \right>^{-1}\bar{f}$. Then we have
\begin{equation*}
	\square_{\mathbf{g}} \bar{f}=( \square_{\mathbf{g}}- \left< \nabla \right> \square_{\mathbf{g}} \left< \nabla \right>^{-1}) \bar{f}=[\mathbf{g}^{\alpha \beta}, \left< \nabla \right>]\left< \nabla \right>^{-1} \partial_{\alpha \beta} \bar{f}.
\end{equation*}
Hence, $\bar{f}$ solves
\begin{equation*}
	\begin{cases}
		\square_{\mathbf{g}} \bar{f}= [\mathbf{g}^{\alpha \beta}, \left< \nabla \right>]\left< \nabla \right>^{-1} \partial_{\alpha \beta} \bar{f}, \quad (t, x)\in [-2,2]\times \mathbb{R}^3,
		\\
		\bar{f}|_{t=0}=\left< \nabla \right>f_0, \quad \partial_t\bar{f}|_{t=0}=\left< \nabla \right>f_1.
	\end{cases}
\end{equation*}

For $\bar{F}=[\mathbf{g}^{\alpha \beta}, \left< \nabla \right>]\left< \nabla \right>^{-1} \partial_{\alpha \beta} \bar{f} \in L_{[-2,2]}^1 L^2_x$, we form $\mathbf{M}\bar{F}$ as above in Step 1. So we have $\square_{\mathbf{g}} \mathbf{M}\bar{F}=\bar{F}$, and $\mathbf{M}\bar{F}$ has vanishing Cauchy data at $t_0=0$. Let $\bar{f}^*$ be the solution for the homogeneous equation
\begin{equation*}
	\begin{cases}
		\square_{\mathbf{g}} \bar{f}^*=0, \quad (t, x)\in [-2,2]\times \mathbb{R}^3,
		\\
		\bar{f}^*|_{t=0}=\left< \nabla \right>f_0, \quad \partial_t\bar{f}^*|_{t=0}=\left< \nabla \right>f_1.
	\end{cases}
\end{equation*}
Then we can seek a solution $\bar{f}$ of the form $\bar{f}=\bar{f}^*+\mathbf{M}\bar{F}$, provided that
\begin{equation}\label{AWQ}
	\| [\mathbf{g}^{\alpha \beta}, \left< \nabla \right>]\left< \nabla \right>^{-1} \partial_{\alpha \beta} \mathbf{M}\bar{F}\|_{L^1_{[-2,2]} L^2_x} \lesssim \epsilon_0 \|\bar{F}\|_{L^1_{[-2,2]}L^2_x}.
\end{equation}
Let us prove \eqref{AWQ}. Considering $\mathbf{g}^{00}=-1$, by using Coifman-Maeyer estimate, we derive that
\begin{equation}\label{AQ2}
	\|  [\mathbf{g}^{\alpha \beta}, \left< \nabla \right>] \left< \nabla \right>^{-1} \partial_{\alpha \beta} \mathbf{M}\bar{F} \|_{L^2_x} \lesssim \|d\mathbf{g}\|_{L^\infty_x} \| d \mathbf{M}\bar{F} \|_{L^2_x}.
\end{equation}
By using \eqref{AQ0}, we also obtain
\begin{equation}\label{AQ3}
	\| d \mathbf{M}\bar{F}\|_{L^\infty_{[-2,2]}L^2_x} \lesssim \|\bar{F}\|_{L^1_{[-2,2]} L^2_x}.
\end{equation}
Combining \eqref{AQ2} with \eqref{AQ3}, and using $\|d\mathbf{g}\|_{L^2_{[-2,2]}L^\infty_x} \lesssim \epsilon_0$, we can get \eqref{AWQ}.

Therefore, by Duhamel's principle, in the case of $r=1$ or $r=2$, we can also obtain the Strichartz estimates for the linear, nonhomogeneous wave equation
\begin{equation*}
	\begin{cases}
		\square_{\mathbf{g}} f= G, \quad (t, x)\in [-2,2]\times \mathbb{R}^3,
		\\
		f|_{t=0}=f_0, \quad \partial_tf|_{t=0}=f_1,
	\end{cases}
\end{equation*}
as
\begin{equation}\label{eAQ4}
	\| \left< \nabla \right>^k f \|_{L^2_t L^\infty_x} \lesssim  \|f_0\|_{H^r}+ \| f_1 \|_{H^{r-1}}+ \|G\|_{L^1_{[-2,2]} H_x^{r-1}} , \quad   k<r-1.
\end{equation}
Now, we claim that the following bounds\footnote{Comparing with \eqref{eAQ4}, the bound \eqref{AQ5} is stronger, for it includes the time-derivatives of $f$. It plays an important role in the next step.} hold for $r=1$ and $r=2$,
\begin{equation}\label{AQ5}
	\| \left< \nabla \right>^k df \|_{L^2_{[-2,2]} L^\infty_x} \lesssim  \|f_0\|_{H^r}+ \| f_1 \|_{H^{r-1}}+ \|G\|_{L^1_{[-2,2]} H_x^{r-1}} , \quad   k<r-2.
\end{equation}
Let us now use \eqref{eAQ4} to prove \eqref{AQ5}.

To prove \eqref{AQ5} for $r=2$, we first consider $G=0$. That is,
$\square_{\mathbf{g}} f=0, (f,\partial_t f)|_{t=0}=(f_0,f_1)\in H^2 \times H^1$. For $\mathbf{g}^{00}=-1$, we then have
\begin{equation}\label{AQ6}
	\square_{\mathbf{g}} d f= d\mathbf{g}^{\alpha \beta} \partial_{\alpha \beta} f \in L^1_{[-2,2]} L^2_x.
\end{equation}
The right hand side of \eqref{AQ6} can be derived by energy estimates and Gronwall's inequality. That is
\begin{equation*}
	\begin{split}
			\|\partial_t d f\|_{L^2_x}+ \|\nabla d f\|_{L^2_x} & \lesssim ( \|f_0\|_{H^2}+\|f_1\|_{H^1} ) \exp \left(\int^{2}_{0} \|d \mathbf{g}\|_{L^\infty_x }dt \right) 
			\\
			& \lesssim  \|f_0\|_{H^2}+\|f_1\|_{H^1} . 
	\end{split}
\end{equation*}
Then $d\mathbf{g}^{\alpha \beta} \partial_{\alpha \beta} f \in L^1_{[-2,2]} L^2_x$ holds. Thus, \eqref{AQ5} holds for $r=2$ and $G=0$. If $G\neq 0$, we can use the Duhamel principle to handle it. Decompose $f=f^a+f^b$, where
\begin{equation}\label{aw0}
	\begin{cases}
		&\square_{\mathbf{g}} f^a=0,
		\\
		&(f^a,\partial_t f^a)|_{t=0}=(f_0,f_1) \in H^2 \times H^1,
	\end{cases}
\end{equation}
and
\begin{equation}\label{aw1}
	\begin{cases}
		&\square_{\mathbf{g}} f^b=G,
		\\
		&(f^a,\partial_t f^b)|_{t=0}=(0,0).
	\end{cases}
\end{equation}
So we can get the Strichartz estimates for $f^a$
\begin{equation*}
	\| \left< \nabla \right>^k df^a \|_{L^2_{[-2,2]} L^\infty_x} \lesssim  \|f_0\|_{H^r}+ \| f_1 \|_{H^{r-1}}, \quad   k<r-2.
\end{equation*}
For $f^b$, we use the iteration process like $\mathbf{M}F$ in Step 1, and the following Strichartz estimates for $f^b$
\begin{equation*}
	\| \left< \nabla \right>^k df^b \|_{L^2_{[-2,2]} L^\infty_x} \lesssim  \|G\|_{L^1_t H_x^{r-1}} , \quad   k<r-2.
\end{equation*}
So the conclusion \eqref{AQ5} holds for $f=f^a+f^b$.

To prove \eqref{AQ5} for $r=1$, we note that, $\left< \nabla \right>^{-1}f$ has the Cauchy data in $H^2 \times H^1$ if the Cauchy data of $f$ is in $H^1 \times L^2$, and
\begin{equation*}
	\begin{split}
		\| \square_{\mathbf{g}} \left< \nabla \right>^{-1}f \|_{L^1_{[-2,2]} H^1_x}= & \|\left< \nabla \right> \square_{\mathbf{g}} \left< \nabla \right>^{-1}f \|_{L^1_{[-2,2]} L^2_x}
		\\
		\lesssim & \|[ \left< \nabla \right>, \mathbf{g}^{\alpha \beta}] \left< \nabla \right>^{-1}\partial_{\alpha \beta}f \|_{L^1_{[-2,2]} L^2_x}+ \|G \|_{L^1_{[-2,2]} L^2_x}.
	\end{split}
\end{equation*}
Noting $\mathbf{g}^{00}=-1$, then we give the following bounds by Coifman-Meyer estimate and energy estimates for $f$
\begin{equation*}
	\begin{split}
		\|[ \left< \nabla \right>, \mathbf{g}^{\alpha \beta}] \left< \nabla \right>^{-1}\partial_{\alpha \beta}f \|_{L^1_{[-2,2]} L^2_x}  \lesssim & \|d\mathbf{g}\|_{L^1_{[-2,2]} L^2_x}\|df\|_{L^\infty_{[-2,2]} L^2_x}
		\\
		\lesssim & \epsilon_0 ( \|f_0\|_{H^1} + \|f_1\|_{L^2}).
	\end{split}
\end{equation*}
\quad \textbf{Step 3:} The case $1 < r <2$. We also set $t_0=0$. Based on the above result in Step 1 and Step 2, we also transform the initial data in $H^1 \times L^2$. Applying $\left< \nabla \right>^{r-1}$ to \eqref{linearC}, we have
\begin{equation*}
	\square_{\mathbf{g}} \left< \nabla \right>^{r-1}f=-[\square_{\mathbf{g}}, \left< \nabla \right>^{r-1}]f.
\end{equation*}
Let $\left< \nabla \right>^{r-1}f=\bar{f}$. Then $\bar{f}$ is a solution to
\begin{equation}\label{Qd}
	\begin{cases}
		\square_{\mathbf{g}}\bar{f}=-[\square_{\mathbf{g}}, \left< \nabla \right>^{r-1}]\left< \nabla \right>^{1-r}\bar{f},
		\\
		(\bar{f}, \partial_t \bar{f})|_{t=0} \in H^1 \times L^2.
	\end{cases}
\end{equation}
Considering $\mathbf{g}^{00}=-1$, let us calculate
\begin{equation}\label{Qd0}
	\begin{split}
		-[\square_{\mathbf{g}}, \left< \nabla \right>^{r-1}]\left< \nabla \right>^{1-r}\bar{f}=& [\mathbf{g}^{\alpha i}-\mathbf{m}^{\alpha i}, \left< \nabla \right>^{r-1} \partial_i ] \partial_{\alpha} \left< \nabla \right>^{1-r} \bar{f}
		\\
		& \ + \left< \nabla \right>^{r-1}( \partial_i \mathbf{g}^{\alpha i} \partial_{\alpha} \left< \nabla \right>^{1-r} \bar{f}).
	\end{split}
\end{equation}
For $r-1 \in (0,1)$, by product estimates and a refined Kato-Ponce inequalities (please cite Theorem 1.9 in \ref{LD}), we have
\begin{equation}\label{Qd1}
	\begin{split}
		& \| [\mathbf{g}^{\alpha i}-\mathbf{m}^{\alpha i}, \left< \nabla \right>^{r-1} \partial_i ] \partial_{\alpha} \left< \nabla \right>^{1-r} \bar{f}  \|_{L^2_x}+ \|\left< \nabla \right>^{r-1}( \partial_i \mathbf{g}^{\alpha i} \partial_{\alpha} \left< \nabla \right>^{1-r} \bar{f})\|_{L^2_x}
		\\
		\lesssim & \| d \mathbf{g} \|_{L^\infty_x} \| d \bar{f} \|_{L^2_x}+  \| \left< \nabla \right>^{r} (\mathbf{g}-\mathbf{m})\|_{L^{p_1}_x}   \|\left< \nabla \right>^{1-r} d \bar{f} \|_{L^{q_1}_x},
	\end{split}
\end{equation}
where $p_1=\frac{3}{\frac{3}{2}-s+r}$ and $q_1=\frac{3}{s-r}$. By Sobolev's inequality, we get
\begin{equation}\label{Qd2}
	\| \left< \nabla \right>^{r} (\mathbf{g}-\mathbf{m})\|_{L^{p_1}_x} \lesssim \| \mathbf{g}- \mathbf{m}\|_{H^{s}_x}.
\end{equation}
By Gagliardo-Nirenberg inequality, we can see that\footnote{The number $s$ and $s_0$ satisfying $2<s_0<s \leq \frac52$.}
\begin{equation}\label{Qd3}
	\|\left< \nabla \right>^{1-r} d \bar{f} \|_{L^{q_1}_x} \lesssim \| d \bar{f} \|^{\frac{s-s_0}{\frac52-s_0}}_{L^{2}_x} \|\left< \nabla \right>^{1-s_0} d \bar{f} \|^{1-\frac{s-s_0}{\frac52-s_0}}_{L^{\infty}_x}
\end{equation}
Considering\footnote{please see the definition of $\mathcal{H}$ and \eqref{401}-\eqref{403}}
\begin{equation*}
	\| d \mathbf{g} \|_{L^2_{[-2,2]} L^\infty_x} +  \| \mathbf{g}- \mathbf{m}\|_{L^\infty_{[-2,2]} H^{s}_x} \lesssim \epsilon_0,
\end{equation*}
and combining with \eqref{Qd0}-\eqref{Qd3}, we have
\begin{equation}\label{eh3}
	\| [\square_{\mathbf{g}}, \left< \nabla \right>^{r-1}]\left< \nabla \right>^{1-r}\bar{f}\|_{L^1_{[-2,2]} L^2_x} \lesssim \epsilon_0( \| d\bar{f}\|_{L^\infty_{[-2,2]} L^2_x}+  \| \left< \nabla \right>^{1-s_0} d \bar{f} \|_{L^2_{[-2,2]} L^\infty_x}).
\end{equation}
Using the discussion in case $r=1$, namely \eqref{AQ5}, we then get the Strichartz estimates for $\bar{f}$ for $\theta<-1$,
\begin{equation}\label{eh}
	\begin{split}
		\| \left< \nabla \right>^\theta d \bar{f} \|_{L^2_{[-2,2]} L^\infty_x} \lesssim  & \  \epsilon_0( \|f_0\|_{H^r}+ \| f_1 \|_{H^{r-1}}+ \| d \bar{f} \|_{L^2_x}+  \|\left< \nabla \right>^{1-s_0} d \bar{f} \|_{L^2_{[-2,2]} L^\infty_x} ).
	\end{split}
\end{equation}
Note $1-s_0<-1$. So we take $\theta=1-s_0$, and then we obtain
\begin{equation}\label{eh0}
	\begin{split}
		\| \left< \nabla \right>^{1-s_0} d \bar{f} \|_{L^2_{[-2,2]} L^\infty_x} \lesssim  & \  \epsilon_0( \|f_0\|_{H^r}+ \| f_1 \|_{H^{r-1}}+ \| d \bar{f} \|_{L^2_x} ).
	\end{split}
\end{equation}
Using \eqref{eh0},  \eqref{eh} yields
\begin{equation}\label{eh1}
	\begin{split}
		\| \left< \nabla \right>^\theta d \bar{f} \|_{L^2_{[-2,2]} L^\infty_x} \lesssim  & \  \epsilon_0( \|f_0\|_{H^r}+ \| f_1 \|_{H^{r-1}}+ \| d \bar{f} \|_{L^2_x} ), \quad \theta <-1.
	\end{split}
\end{equation}
By \eqref{eh0}, \eqref{eh3} and the energy estimates for $\bar{f}$, we have
\begin{equation}\label{AQ7}
	\| d \bar{f} \|_{L^2_x} \lesssim \| \bar{f}(0) \|_{H^1_x}+ \| \partial_t \bar{f}(0) \|_{L^2_x} \lesssim \|f_0\|_{H^r}+ \| f_1 \|_{H^{r-1}}.
\end{equation}
Substituting $\left< \nabla \right>^{r-1}f=\bar{f}$ in \eqref{eh1} and using \eqref{AQ7}, we can prove that
\begin{equation*}
	\| \left< \nabla \right>^\theta d {f} \|_{L^2_{[-2,2]} L^\infty_x} \lesssim   \ \epsilon_0( \|f_0\|_{H^r}+ \| f_1 \|_{H^{r-1}}) , \quad   \theta<r-2.
\end{equation*}
As a result, we also have
\begin{equation*}
	\| \left< \nabla \right>^k f \|_{L^2_{[-2,2]} L^\infty_x}
	\lesssim   \epsilon_0( \|f_0\|_{H^r}+ \| f_1 \|_{H^{r-1}} ) , \quad   k<r-1.
\end{equation*}

\textbf{Step 4:} The case $2 < r \leq s$. We also set $t_0=0$. Based on the above result in Step 1 and Step 2, we transform the initial data in $H^1 \times L^2$. Applying $\left< \nabla \right>^{r-1}$ to \eqref{linearC}, we have
\begin{equation*}
	\square_{\mathbf{g}} \left< \nabla \right>^{r-1}f=-[\square_{\mathbf{g}}, \left< \nabla \right>^{r-1}]f.
\end{equation*}
Let $\left< \nabla \right>^{r-1}f=\bar{f}$. Then $\bar{f}$ is a solution to
\begin{equation*}\label{qd}
	\begin{cases}
		\square_{\mathbf{g}}\bar{f}=-[\square_{\mathbf{g}}, \left< \nabla \right>^{r-1}]\left< \nabla \right>^{1-r}\bar{f},
		\\
		(\bar{f}, \partial_t \bar{f})|_{t=0} \in H^1 \times L^2.
	\end{cases}
\end{equation*}
Considering $\mathbf{g}^{00}=-1$, let us calculate
\begin{equation}\label{AQ9}
	\begin{split}
		-[\square_{\mathbf{g}}, \left< \nabla \right>^{r-1}]\left< \nabla \right>^{1-r}\bar{f}=& [\mathbf{g}^{\alpha i}-\mathbf{m}^{\alpha i}, \left< \nabla \right>^{r-1} \partial_i ] \partial_{\alpha} \left< \nabla \right>^{1-r} \bar{f}
		\\
		& \ + \left< \nabla \right>^{r-1}( \partial_i \mathbf{g}^{\alpha i} \partial_{\alpha} \left< \nabla \right>^{1-r} \bar{f}).
	\end{split}
\end{equation}
By Kato-Ponce commutator and product estimates, we have
\begin{equation}\label{qdy}
	\begin{split}
		& \| [\mathbf{g}^{\alpha i}-\mathbf{m}^{\alpha i}, \left< \nabla \right>^{r-1} \partial_i ] \partial_{\alpha} \left< \nabla \right>^{1-r} \bar{f}  \|_{L^2_x}+ \|\left< \nabla \right>^{r-1}( \partial_i \mathbf{g}^{\alpha i} \partial_{\alpha} \left< \nabla \right>^{1-r} \bar{f})\|_{L^2_x}
		\\
		\lesssim & \| d \mathbf{g} \|_{L^\infty_x} \| d \bar{f} \|_{L^2_x}+  \| \mathbf{g}^{\alpha \beta}- \mathbf{m}^{\alpha \beta}\|_{H^{r}_x}   \|\left< \nabla \right>^{1-r} d \bar{f} \|_{L^\infty_x},
	\end{split}
\end{equation}
Considering\footnote{please see the definition of $\mathcal{H}$ and \eqref{401}-\eqref{403}}
\begin{equation*}
	\| d \mathbf{g} \|_{L^2_{[-2,2]} L^\infty_x} +  \| \mathbf{g}^{\alpha \beta}- \mathbf{m}^{\alpha \beta}\|_{L^\infty_{[-2,2]} H^{r}_x} \lesssim \epsilon_0, \qquad 2 < r \leq s,
\end{equation*}
and combining with \eqref{AQ9}, \eqref{qdy}, we have
\begin{equation}\label{Eh3}
	\| [\square_{\mathbf{g}}, \left< \nabla \right>^{r-1}]\left< \nabla \right>^{1-r}\bar{f}\|_{L^1_{[-2,2]} L^2_x} \lesssim \epsilon_0( \| d\bar{f}\|_{L^\infty_{[-2,2]} L^2_x}+  \| \left< \nabla \right>^{1-r} d \bar{f} \|_{L^2_{[-2,2]} L^\infty_x}).
\end{equation}
Using the discussion in case $r=1$, namely \eqref{AQ5}, we then get the Strichartz estimates for $\bar{f}$ for $\theta<-1$,
\begin{equation}\label{Eh}
	\begin{split}
		\| \left< \nabla \right>^\theta d \bar{f} \|_{L^2_{[-2,2]} L^\infty_x} \lesssim  & \  \epsilon_0( \|f_0\|_{H^r}+ \| f_1 \|_{H^{r-1}}+ \| d \bar{f} \|_{L^2_x}+  \|\left< \nabla \right>^{1-r} d \bar{f} \|_{L^2_{[-2,2]} L^\infty_x} ).
	\end{split}
\end{equation}
Note $r>2$ and $1-r<-1$. We then take $\theta=1-r$. So we have
\begin{equation}\label{Eh0}
	\begin{split}
		\| \left< \nabla \right>^{1-r} d \bar{f} \|_{L^2_{[-2,2]} L^\infty_x} \lesssim  & \  \epsilon_0( \|f_0\|_{H^r}+ \| f_1 \|_{H^{r-1}}+ \| d \bar{f} \|_{L^2_x} ).
	\end{split}
\end{equation}
Using \eqref{eh0}, \eqref{Eh} yields
\begin{equation}\label{Eh1}
	\begin{split}
		\| \left< \nabla \right>^\theta d \bar{f} \|_{L^2_{[-2,2]} L^\infty_x} \lesssim  & \  \epsilon_0( \|f_0\|_{H^r}+ \| f_1 \|_{H^{r-1}}+ \| d \bar{f} \|_{L^2_x} ), \quad \theta <-1.
	\end{split}
\end{equation}
By \eqref{Eh0}, \eqref{Eh3} and the energy estimates for $\bar{f}$, we have
\begin{equation}\label{Eh4}
	\| d \bar{f} \|_{L^2_x} \lesssim \| \bar{f}(0) \|_{H^1_x}+ \| \partial_t \bar{f}(0) \|_{L^2_x} \lesssim \|f_0\|_{H^r}+ \| f_1 \|_{H^{r-1}}.
\end{equation}
Substituting $\left< \nabla \right>^{r-1}f=\bar{f}$ in \eqref{Eh1}, and using \eqref{Eh4}, we can prove that
\begin{equation*}
	\| \left< \nabla \right>^\theta d {f} \|_{L^2_{[-2,2]} L^\infty_x} \lesssim   \ \epsilon_0( \|f_0\|_{H^r}+ \| f_1 \|_{H^{r-1}}) , \quad   \theta<r-2.
\end{equation*}
As a result, we also have
\begin{equation*}
	\| \left< \nabla \right>^k f \|_{L^2_{[-2,2]} L^\infty_x}
	\lesssim   \epsilon_0( \|f_0\|_{H^r}+ \| f_1 \|_{H^{r-1}} ) , \quad   k<r-1.
\end{equation*}

\textbf{Step 5:} The case $s < r \leq 3$. We also set $t_0=0$. Based on the above result in Step 1 and Step 2, we transform the initial data in $H^1 \times L^2$. Applying $\left< \nabla \right>^{r-1}$ to \eqref{linearC}, we have
\begin{equation*}
	\square_{\mathbf{g}} \left< \nabla \right>^{r-1}f=-[\square_{\mathbf{g}}, \left< \nabla \right>^{r-1}]f.
\end{equation*}
Let $\left< \nabla \right>^{r-1}f=\bar{f}$. Then $\bar{f}$ is a solution to
\begin{equation*}\label{sy0}
	\begin{cases}
		\square_{\mathbf{g}}\bar{f}=-[\square_{\mathbf{g}}, \left< \nabla \right>^{r-1}]\left< \nabla \right>^{1-r}\bar{f},
		\\
		(\bar{f}, \partial_t \bar{f})|_{t=0} \in H^1 \times L^2.
	\end{cases}
\end{equation*}
Considering $\mathbf{g}^{00}=-1$, let us calculate
\begin{equation}\label{sy1}
	\begin{split}
		-[\square_{\mathbf{g}}, \left< \nabla \right>^{r-1}]\left< \nabla \right>^{1-r}\bar{f}=& [\mathbf{g}^{\alpha i}-\mathbf{m}^{\alpha i}, \left< \nabla \right>^{r-1}  ] \partial_{\alpha} \left< \nabla \right>^{1-r} \partial_i\bar{f} .
	\end{split}
\end{equation}
By Kato-Ponce commutator estimates, we have
\begin{equation}\label{sy2}
	\begin{split}
		& \| [\mathbf{g}^{\alpha i}-\mathbf{m}^{\alpha i}, \left< \nabla \right>^{r-1}  ] \partial_{\alpha} \left< \nabla \right>^{1-r} \partial_i \bar{f}  \|_{L^2_x}
		\\
		\lesssim & \| d \mathbf{g} \|_{L^\infty_x} \| d \bar{f} \|_{L^2_x}+  \| \left< \nabla \right>^{r-1}( \mathbf{g}- \mathbf{m})\|_{L^{p_2}_x}   \|\left< \nabla \right>^{2-r} d \bar{f} \|_{L^{q_2}_x},
	\end{split}
\end{equation}
where $p_2=\frac{3}{\frac{1}{2}-s+r}$ and $q_2=\frac{3}{1+s-r}$. By Sobolev's inequality, we get
\begin{equation}\label{sy3}
	\| \left< \nabla \right>^{r-1} (\mathbf{g}-\mathbf{m})\|_{L^{p_2}_x} \lesssim \| \mathbf{g}- \mathbf{m}\|_{H^{s}_x}.
\end{equation}
By Gagliardo-Nirenberg inequality, we can see that\footnote{The number $s$ and $s_0$ satisfying $2<s_0<s \leq \frac52$.}
\begin{equation}\label{sy4}
	\|\left< \nabla \right>^{2-r} d \bar{f} \|_{L^{q_2}_x} \lesssim \| d \bar{f} \|^{\frac{s-2}{3-s}}_{L^{2}_x} \|\left< \nabla \right>^{-\frac{s}{2}} d \bar{f} \|^{\frac{5-2s}{3-s}}_{L^{\infty}_x}
\end{equation}
By using \footnote{please see the definition of $\mathcal{H}$ and \eqref{401}-\eqref{403}}
\begin{equation*}
	\| d \mathbf{g} \|_{L^2_{[-2,2]} L^\infty_x} +  \| \mathbf{g}^{\alpha \beta}- \mathbf{m}^{\alpha \beta}\|_{L^\infty_{[-2,2]} H^{r-1}_x} \lesssim \epsilon_0, \qquad s < r \leq 3,
\end{equation*}
also \eqref{sy1}, \eqref{sy2}, and \eqref{sy3},  we have
\begin{equation*}
	\| [\square_{\mathbf{g}}, \left< \nabla \right>^{r-1}]\left< \nabla \right>^{1-r}\bar{f}\|_{L^1_{[-2,2]} L^2_x} \lesssim \epsilon_0( \| d\bar{f}\|_{L^\infty_{[-2,2]} L^2_x}+  \| \left< \nabla \right>^{-\frac{s}{2}} d \bar{f} \|_{L^2_{[-2,2]} L^\infty_x}).
\end{equation*}
Using the discussion in case $r=1$, namely \eqref{AQ5}, we then get the Strichartz estimates for $\bar{f}$ for $\theta<-1$,
\begin{equation}\label{syu4}
	\begin{split}
		\| \left< \nabla \right>^\theta d \bar{f} \|_{L^2_{[-2,2]} L^\infty_x} \lesssim  & \  \epsilon_0( \|f_0\|_{H^r}+ \| f_1 \|_{H^{r-1}}+ \| d \bar{f} \|_{L^2_x}+  \|\left< \nabla \right>^{2-r} d \bar{f} \|_{L^2_{[-2,2]} L^\infty_x} ).
	\end{split}
\end{equation}
Note $-\frac{s}{2}<-1$. We then take $\theta=-\frac{s}{2}$. So we have
\begin{equation}\label{sy5}
	\begin{split}
		\| \left< \nabla \right>^{2-r} d \bar{f} \|_{L^2_{[-2,2]} L^\infty_x} \lesssim  & \  \epsilon_0( \|f_0\|_{H^r}+ \| f_1 \|_{H^{r-1}}+ \| d \bar{f} \|_{L^2_x} ).
	\end{split}
\end{equation}
Using \eqref{sy5},  \eqref{syu4} yields
\begin{equation}\label{sy6}
	\begin{split}
		\| \left< \nabla \right>^\theta d \bar{f} \|_{L^2_{[-2,2]} L^\infty_x} \lesssim  & \  \epsilon_0( \|f_0\|_{H^r}+ \| f_1 \|_{H^{r-1}}+ \| d \bar{f} \|_{L^2_x} ), \quad \theta <-1.
	\end{split}
\end{equation}
By \eqref{sy4}, \eqref{sy5} and the energy estimates for $\bar{f}$, we have
\begin{equation}\label{sy7}
	\| d \bar{f} \|_{L^2_x} \lesssim \| \bar{f}(0) \|_{H^1_x}+ \| \partial_t \bar{f}(0) \|_{L^2_x} \lesssim \|f_0\|_{H^r}+ \| f_1 \|_{H^{r-1}}.
\end{equation}
Substituting $\left< \nabla \right>^{r-1}f=\bar{f}$ in \eqref{sy6}, and using \eqref{sy7}, we can prove that
\begin{equation*}
	\| \left< \nabla \right>^\theta d {f} \|_{L^2_{[-2,2]} L^\infty_x} \lesssim   \ \epsilon_0( \|f_0\|_{H^r}+ \| f_1 \|_{H^{r-1}}) , \quad   \theta<r-2.
\end{equation*}
As a result, we also have
\begin{equation*}
	\| \left< \nabla \right>^k f \|_{L^2_{[-2,2]} L^\infty_x}
	\lesssim   \epsilon_0( \|f_0\|_{H^r}+ \| f_1 \|_{H^{r-1}} ) , \quad   k<r-1.
\end{equation*}

\textbf{Step 6:} The case $3 < r \leq s+1$. We also set $t_0=0$. Based on the above result in Step 1 and Step 2, we transform the initial data in $H^1 \times L^2$.  Applying $\left< \nabla \right>^{r-1}$ to \eqref{linearC}, we have
\begin{equation*}
	\square_{\mathbf{g}} \left< \nabla \right>^{r-1}f=-[\square_{\mathbf{g}}, \left< \nabla \right>^{r-1}]f.
\end{equation*}
Let $\left< \nabla \right>^{r-1}f=\bar{f}$. Then $\bar{f}$ is a solution to
\begin{equation*}\label{sp0}
	\begin{cases}
		\square_{\mathbf{g}}\bar{f}=-[\square_{\mathbf{g}}, \left< \nabla \right>^{r-1}]\left< \nabla \right>^{1-r}\bar{f},
		\\
		(\bar{f}, \partial_t \bar{f})|_{t=0} \in H^1 \times L^2.
	\end{cases}
\end{equation*}
Due to $\mathbf{g}^{00}=-1$, so we calculate
\begin{equation}\label{sp1}
	\begin{split}
		-[\square_{\mathbf{g}}, \left< \nabla \right>^{r-1}]\left< \nabla \right>^{1-r}\bar{f}=& [\mathbf{g}^{\alpha i}-\mathbf{m}^{\alpha i}, \left< \nabla \right>^{r-1}  ] \partial_{\alpha} \left< \nabla \right>^{1-r} \partial_i\bar{f} .
	\end{split}
\end{equation}
By Kato-Ponce commutator estimates, we have
\begin{equation}\label{sp2}
	\begin{split}
		& \| [\mathbf{g}^{\alpha i}-\mathbf{m}^{\alpha i}, \left< \nabla \right>^{r-1}  ] \partial_{\alpha} \left< \nabla \right>^{1-r} \partial_i \bar{f}  \|_{L^2_x}
		\\
		\lesssim & \| d \mathbf{g} \|_{L^\infty_x} \| d \bar{f} \|_{L^2_x}+  \| \mathbf{g}^{\alpha \beta}- \mathbf{m}^{\alpha \beta}\|_{H^{r-1}_x}   \|\left< \nabla \right>^{2-r} d \bar{f} \|_{L^\infty_x},
	\end{split}
\end{equation}
Considering\footnote{please see the definition of $\mathcal{H}$ and \eqref{401}-\eqref{403}}
\begin{equation*}
	\| d \mathbf{g} \|_{L^2_{[-2,2]} L^\infty_x} +  \| \mathbf{g}^{\alpha \beta}- \mathbf{m}^{\alpha \beta}\|_{L^\infty_{[-2,2]} H^{r-1}_x} \lesssim \epsilon_0, \qquad 3 < r \leq s+1,
\end{equation*}
and combining with \eqref{sp1}, \eqref{sp2}, we have
\begin{equation}\label{sp3}
	\| [\square_{\mathbf{g}}, \left< \nabla \right>^{r-1}]\left< \nabla \right>^{1-r}\bar{f}\|_{L^1_{[-2,2]} L^2_x} \lesssim \epsilon_0( \| d\bar{f}\|_{L^\infty_{[-2,2]} L^2_x}+  \| \left< \nabla \right>^{2-r} d \bar{f} \|_{L^2_{[-2,2]} L^\infty_x}).
\end{equation}
Using the discussion in case $r=1$, namely \eqref{AQ5}, we then get the Strichartz estimates for $\bar{f}$ for $\theta<-1$,
\begin{equation}\label{sp4}
	\begin{split}
		\| \left< \nabla \right>^\theta d \bar{f} \|_{L^2_{[-2,2]} L^\infty_x} \lesssim  & \  \epsilon_0( \|f_0\|_{H^r}+ \| f_1 \|_{H^{r-1}}+ \| d \bar{f} \|_{L^\infty_{[-2,2]} L^2_x}+  \|\left< \nabla \right>^{2-r} d \bar{f} \|_{L^2_{[-2,2]} L^\infty_x} ).
	\end{split}
\end{equation}
Note $r>3$ and $2-r<-1$. We then take $\theta=2-r$. So we have
\begin{equation}\label{sp5}
	\begin{split}
		\| \left< \nabla \right>^{2-r} d \bar{f} \|_{L^2_{[-2,2]} L^\infty_x} \lesssim  & \  \epsilon_0( \|f_0\|_{H^r}+ \| f_1 \|_{H^{r-1}}+ \| d \bar{f} \|_{L^2_x} ).
	\end{split}
\end{equation}
Since \eqref{sp5} and \eqref{sp4}, it yields
\begin{equation}\label{sp6}
	\begin{split}
		\| \left< \nabla \right>^\theta d \bar{f} \|_{L^2_{[-2,2]} L^\infty_x} \lesssim  & \  \epsilon_0( \|f_0\|_{H^r}+ \| f_1 \|_{H^{r-1}}+ \| d \bar{f} \|_{L^2_x} ), \quad \theta <-1.
	\end{split}
\end{equation}
By \eqref{Eh0}, \eqref{Eh3} and the energy estimates for $\bar{f}$, we have
\begin{equation}\label{sp7}
	\| d \bar{f} \|_{L^2_x} \lesssim \| \bar{f}(0) \|_{H^1_x}+ \| \partial_t \bar{f}(0) \|_{L^2_x} \lesssim \|f_0\|_{H^r}+ \| f_1 \|_{H^{r-1}}.
\end{equation}
Substituting $\left< \nabla \right>^{r-1}f=\bar{f}$ in \eqref{sp6} and using \eqref{sp7}, we can prove that
\begin{equation*}
	\| \left< \nabla \right>^\theta d {f} \|_{L^2_{[-2,2]} L^\infty_x} \lesssim   \ \epsilon_0( \|f_0\|_{H^r}+ \| f_1 \|_{H^{r-1}}) , \quad   \theta<r-2.
\end{equation*}
As a result, we also have
\begin{equation*}
	\| \left< \nabla \right>^k f \|_{L^2_{[-2,2]} L^\infty_x}
	\lesssim   \epsilon_0( \|f_0\|_{H^r}+ \| f_1 \|_{H^{r-1}} ) , \quad   k<r-1.
\end{equation*}
Therefore, we have finished the proof of Proposition \ref{r5}.
\end{proof}
\subsection{Proof of Proposition \ref{AA1}}
By Remark \ref{rel}, we only need to consider the problem if the frequency $\lambda$ is large enough, namely
\begin{equation*}
	\lambda \geq \epsilon_0^{-1}.
\end{equation*}
In this case, we can construct an approximate solutions to \eqref{linearA} by using wave packets, and a wave packet has a finer spatial localization than a plane wave solutions. We divide the proof for \eqref{Yee}-\eqref{Ase} to two parts. The conclusion \eqref{Yee} is established in Proposition \ref{szy} and Proposition \ref{szi}, and the conclusion \eqref{Ase} is obtained by Proposition \ref{szt}.

We introduce a spatially localized mollifier $T_\lambda$ by
\begin{equation*}
	T_\lambda f = \chi_\lambda * f, \quad \chi_\lambda=\lambda^3 \chi(\lambda^{-1} y),
\end{equation*}
where $\chi \in C^\infty_0(\mathbb{R}^3)$ is supported in the ball $|x| \leq \frac{1}{32}$, and has integral $1$. By choosing $\chi$ appropriately, any function $v$ with frequency support contained in $\{\xi:|\xi|\leq 4\lambda\}$ can be factored $v=T_\lambda \widetilde{v}$, where $\| \widetilde{v} \|_{L^2_x} \approx \| v \|_{L^2_x}$.

Let us now introduce the wave packets which is related to $\mathbf{g}$ and the null hypersurfaces $\Sigma$.
\begin{definition}[\cite{ST}, Definition 8.1]
	Let $\gamma=\gamma(t)$ be a geodesic for $\mathbf{g}$, and let $\Sigma_{\omega,r}$ be the null surface introduced in Section \ref{secp4} that contains $\gamma$, defined by
	\begin{equation*}
		\Sigma_{\omega,r}=\{ (t,x): x_\omega- \phi_{\omega,r}(t,x'_\omega)=0 \},
	\end{equation*}
where $x_\omega=x \cdot \omega$, and $x'_\omega$ are projective coordinates along $\omega$. A normalized wave packet around $\gamma$ is a function $f$ of the form
	\begin{equation*}
		f=\epsilon_0^{\frac12} \lambda^{-1} T_\lambda(\varphi \psi),
	\end{equation*}
	where
	\begin{equation*}
		\varphi(t,x)=\delta(x_{\omega}-\phi_{\omega,r}(t,x'_\omega)), \quad \psi=\chi_0( (\epsilon_0 \lambda)^{\frac12}(x'_\omega-\gamma_\omega(t))  ).
	\end{equation*}
	Here, $\chi_0$ is a smooth function supported in the set $|x'| \leq 1$, with uniform bounds on its derivatives $|\partial^\alpha_{x'} \chi_0(x')| \leq c_\alpha$.
\end{definition}
We give two notations here. We denote $L(\varphi,\psi)$ to denote a translation invariant bilinear operator of the form
\begin{equation*}
	L(\varphi,\psi)(x)=\int K(y,z)\varphi(x+y)\psi(x+z)dydz,
\end{equation*}
where $K(y,z)$ is a finite measure. If $X$ is a Sobolev spaces, we then denote $X_a$ the same space but with the norm obtained by dimensionless rescaling by $a$,
\begin{equation*}
	\| \varphi \|_{X_a}=\| \varphi(a \cdot)\|_{X}.
\end{equation*}
Since $2(s_0-1)>2$, then for $a<1$ we have $\| \varphi \|_{H^{s_0-1}_a(\mathbb{R}^{2})} \lesssim \| \varphi \|_{H^{s_0-1}(\mathbb{R}^{2})}$. In the following, let us introduce what we can get when we take the operator $\square_{\mathbf{g}_\lambda}$ on wave packets.
\subsubsection{A normalized wave packet}
\begin{proposition}\label{np}
	Let $f$ be a normalized packet. Then there is another normalized wave packet $\tilde{f}$, and functions $\psi_m(t,x'_\omega), m=0,1,2$, so that
	\begin{equation}\label{np1}
		\square_{\mathbf{g}_\lambda} P_\lambda f= L(d \mathbf{g}, d \tilde{P}_\lambda \tilde{f})+ \epsilon_0^{\frac12}\lambda^{-1}P_{\lambda}T_{\lambda}\sum_{m=0,1,2}
		\psi_m\delta^{(m)}(x'_\omega-\phi_{\omega,r}),
	\end{equation}
	where the functions $\psi_m=\psi_m(t,x'_\omega)$ satisfy the scaled Sobolev estimates
	\begin{equation}\label{np2}
		\| \psi_m\|_{L^2_{[-2,2]} H^{s_0-1}_{a,x'_\omega}} \lesssim \epsilon_0 \lambda^{1-m}, \quad m=0,1,2, \quad a=(\epsilon_0 \lambda)^{-\frac12}.
	\end{equation}
\end{proposition}
\begin{proof}
		For brevity, we consider the case $\omega=(0,0,1)$. Then $x_\omega=x_3$, and $x'_\omega=x'$. We write
		\begin{equation}\label{js0}
			\square_{\mathbf{g}_\lambda} P_\lambda f
			= \lambda^{-1} ( [\square_{\mathbf{g}_\lambda}, P_\lambda T_\lambda]+P_\lambda T_\lambda \square_{\mathbf{g}_\lambda} )(\varphi \psi).
		\end{equation}
		For the first term in \eqref{js0}, noting $\mathbf{g}\lambda$ supported at frequency $\leq \frac \lambda8$, then we can write
		\begin{equation*}
			[\square_{\mathbf{g}_\lambda}, P_\lambda T_\lambda]=[\square_{\mathbf{g}_\lambda}, P_\lambda T_\lambda]\tilde{P}_\lambda \tilde{T}_\lambda
		\end{equation*}
		for some multipliers $\tilde{P}_\lambda, \tilde{T}_\lambda$ which have the same properties as $P_\lambda, T_\lambda$. Therefore, by using the kernal bounds for $P_\lambda T_\lambda$, we conclude that
		\begin{equation*}
			[\square_{\mathbf{g}_\lambda}, P_\lambda T_\lambda]f=L(d\mathbf{g}, df).
		\end{equation*}
		For the second term in \eqref{js0}, we use the Leibniz rule
		\begin{equation}\label{js1}
			\square_{\mathbf{g}_\lambda}(\varphi\psi)=\psi \square_{\mathbf{g}_\lambda} \varphi+(\mathbf{g}^{\alpha \beta}_\lambda+\mathbf{g}^{\beta \alpha}_\lambda)\partial_\alpha \varphi \partial_\beta \psi + \varphi \square_{\mathbf{g}_\lambda} \psi.
		\end{equation}
		We let $\nu$ denote the conormal vector field along $\Sigma$, $\nu=dx_3-d\phi(t, x')$. In the following, we take the greek indices $0 \leq \alpha, \beta \leq 2$.

		For the first term in \eqref{js1}, we can calculate it by
		\begin{equation*}
			\begin{split}
				\mathbf{g}_{\lambda}^{\alpha \beta} \partial_{\alpha \beta} \varphi =& \mathbf{g}_{\lambda}^{\alpha \beta}(t,x',\phi)\nu_\alpha \nu_\beta \delta^{(2)}_{x_3-\phi}
				-2 (\partial_3 \mathbf{g}_{\lambda}^{\alpha \beta})(t, x',\phi) \nu_\alpha \nu_\beta \delta^{(1)}_{x_3-\phi}
				\\
				&+(\partial^2_3 \mathbf{g}_{\lambda}^{\alpha \beta})(t, x',\phi) \nu_\alpha \nu_\beta \delta^{(0)}_{x_3-\phi}- \mathbf{g}_{\lambda}^{\alpha \beta})(t, x',\phi) \partial_{\alpha \beta}\phi \delta^{(1)}_{x_3-\phi}
				\\
				&+ \partial_3\mathbf{g}_{\lambda}^{\alpha \beta})(t, x',\phi) \partial_{\alpha \beta}\phi \delta^{(0)}_{x_3-\phi}.
			\end{split}
		\end{equation*}
		Here, $\delta^{(m)}_{x_3-\phi}=(\partial^m \delta)(x_3-\phi) $. By Leibniz rule, we can take
		\begin{equation*}
			\begin{split}
				\psi_0&=\psi \big\{ (\partial^2_3 \mathbf{g}_{\lambda}^{\alpha \beta})(t, x',\phi)\nu_\alpha \nu_\beta+ (\partial_3 \mathbf{g}_{\lambda}^{\alpha \beta})(t, x',\phi)\partial^2_{\alpha \beta}\phi  \big\},
				\\
				\psi_1&=\psi \big\{ 2(\partial_3 \mathbf{g}_{\lambda}^{\alpha \beta})(t, x',\phi)\nu_\alpha \nu_\beta-  \mathbf{g}_{\lambda}^{\alpha \beta}(t, x',\phi)\partial^2_{\alpha \beta}\phi \big\},
				\\
				\psi_2&=\psi  ( \mathbf{g}_{\lambda}^{\alpha \beta}-\mathbf{g}^{\alpha \beta})\nu_\alpha \nu_\beta,
			\end{split}
		\end{equation*}
		Due to \eqref{G}, Proposition \ref{r1}, and Corollary \ref{vte}, we can conclude that this settings of $\psi_0, \psi_1,$ and $\psi_2$ satisfy \eqref{np2}. For the second term in \eqref{js0}, we have
		\begin{equation*}
			\begin{split}
				(\mathbf{g}^{\alpha \beta}_\lambda+\mathbf{g}^{\beta \alpha}_\lambda)\partial_\alpha \varphi \partial_\beta \psi=& \frac12\nu_\alpha (\mathbf{g}^{\alpha \beta}_\lambda+\mathbf{g}^{\beta \alpha}_\lambda)(t, x',\phi)\partial_\beta \psi \delta^{(1)}_{x_3-\phi}
				\\
				&- \frac12\nu_\alpha \partial_3 (\mathbf{g}^{\alpha \beta}_\lambda+\mathbf{g}^{\beta \alpha}_\lambda)(t, x',\phi)\partial_\beta \psi \delta^{(0)}_{x_3-\phi}.
			\end{split}
		\end{equation*}
		We therefore take
		\begin{equation*}
			\psi_0= \frac12\nu_\alpha \partial_3 (\mathbf{g}^{\alpha \beta}_\lambda+\mathbf{g}^{\beta \alpha}_\lambda)(t, x',\phi)\partial_\beta \psi, \quad \psi_1= \frac12\nu_\alpha (\mathbf{g}^{\alpha \beta}_\lambda+\mathbf{g}^{\beta \alpha}_\lambda)(t, x',\phi)\partial_\beta \psi.
		\end{equation*}
		Since \eqref{G}, Proposition \ref{r1}, and Corollary \ref{vte}, we can conclude that this settings of $\psi_0$ and $\psi_1$ satisfy the estimates \eqref{np2}.
		For the third term in \eqref{js0}, we can take
		\begin{equation*}
			\psi_0=\mathbf{g}^{\alpha \beta}_\lambda(t,x',\phi)\partial^2_{\alpha \beta}\psi.
		\end{equation*}
		Due to \eqref{G}, Proposition \ref{r1}, and Corollary \ref{vte}, we can conclude that this settings of $\psi_0$ satisfies the estimates \eqref{np2}. So we have finished the proof of Proposition \ref{np}.
	\end{proof}
\begin{remark}
	Seeing \eqref{np1}-\eqref{np2}, a single wave packet is not sufficient for us to constructe approximate solutions to a linear wave equation. 
\end{remark}	
We next consider a superposition of these single wave packets.
	\subsubsection{Superpositions of wave packets}\label{aFs}
	Firstly, let us introduce some notations. The index $\omega$ stands for the initial orientation of the wave packet at $t=-2$, which varies over a maximal collection of approximately $\epsilon_0^{-1}\lambda$ unit vectors separated by at least $\epsilon_0^{\frac12}\lambda^{-\frac12}$. For each $\omega$, we have the orthonormal coordinate system $(x_\omega, x'_\omega)$ of $\mathbb{R}^2$, where $x_\omega=x \cdot \omega$, and $x'_\omega$ are projective along $\omega$.

	We denote $\mathbb{R}^3$ by a parallel tiling of rectangles, with length $(8\lambda)^{-1}$ in the $x_{\omega}$ direction, and $(4\epsilon_0 \lambda)^{-\frac12}$ in the other directions $\bx'_{\omega}$. The index $j$ corresponds to a counting of the rectangles in this decomposition. Let $R_{\omega,j}$ denote the collection of the doubles of these rectangles, and $\Sigma_{\omega,j}$ denote the element of the $\Sigma_{\omega}$ ($\Sigma_{\omega}=\cup_{r} \Sigma_{\omega,r}$) foliation upon which $R_{\omega,j}$ is centered. Let $\gamma_{\omega,j}$ denote the null geodesic contained in $\Sigma_{\omega,j}$ which passes through the center of $R_{\omega,j}$ at time $t=-2$.

	We let $T_{\omega,j}$ be the set
	\begin{equation}\label{twj}
		T_{\omega,j}=\Sigma_{\omega,j} \cap \{ | x'_{\omega}-\gamma_{\omega,j}| \leq (\epsilon_0 \lambda)^{-\frac12}\} .
	\end{equation}
	By \eqref{600} and \eqref{606}, then the estimate
	\begin{equation*}
		|dr_{\theta}-(\theta \cdot d x-dt)| \lesssim \epsilon_1,
	\end{equation*}
	holds pointwise uniformly on $[-2,2]\times \mathbb{R}^3$. This also implies that
	\begin{equation}\label{pr0}
		| \phi_{\theta,r}(t,x'_{\theta})-\phi_{\theta,r'}(t,x'_{\theta})-(r-r')| \lesssim \epsilon_1|r-r'|.
	\end{equation}
	For simplicity,  we set
	\begin{equation*}
		{\rho}(t)=\| d \mathbf{g} \|_{C^\delta_x}.
	\end{equation*}
	Then \eqref{502} tells us
	\begin{equation}\label{pr1}
		\| d^2_{x'_{\omega}}\phi_{\omega,r}(t,x'_{\omega})-d^2_{x'_{\omega}}\phi_{\omega,r'}(t,x'_{\omega})\|_{L^\infty_{x'_{\omega}}} \lesssim \epsilon_2+ {\rho}(t).
	\end{equation}
	By using \eqref{pr0} and \eqref{pr1}, we get
	\begin{equation}\label{pr2}
		\| d_{x'_{\omega}}\phi_{\omega,r}(t,x'_{\omega})-d_{x'_{\omega}}\phi_{\omega,r'}(t,x'_{\omega})\|_{L^\infty_{x'_{\omega}}} \lesssim (\epsilon_2+ {\rho}(t))^{\frac12}|r-r'|^{\frac12}.
	\end{equation}
	For $dx_{\omega}-d\phi_{\omega,r}$ is null and also $d \mathbf{g} \leq {\rho}(t)$, this also implies H\"older-$\frac12$ bounds on $d\phi_{\omega,r}$. So we suppose that $(t,x)\in \Sigma_{\omega,r}$ and $(t,y)\in \Sigma_{\omega,r'}$, that $| x'_{\omega}-y'_{\omega}| \leq 2(\epsilon_0 \lambda)^{-\frac12}$, and that $|r-r'|\leq 2 \lambda^{-1}$. Using \eqref{601}, we can obtain
	\begin{equation*}
		|l_{\omega}(t,x)-l_{\omega}(t,y)| \lesssim \epsilon_0^{\frac12}\lambda^{-\frac12}+\epsilon_0^{-\frac12} {\rho}(t)\lambda^{-\frac12}.
	\end{equation*}
	Due to $\dot{\gamma}_{\omega}=l_{\omega}$, and $\|{\rho}\|_{L^2_{[-2,2]}} \lesssim \epsilon_0$, so any geodesic in $\Sigma_{\omega}$ which intersects a slab $T_{\omega,j}$ should be contained in the similar slab of half the scale.

	We are ready to introduce a result about a superposition of wave packets from a certain fixed time.
	\begin{Lemma}\label{SWP}
		Let $0<\mu < \delta$. Let a scalar function $v$ be formulated by
		\begin{equation}\label{swp}
			v(t,x)=\epsilon_0^{\frac12}P_{\lambda} \sum_{\omega,j}T_\lambda(f^{\omega,j}\delta_{x_{\omega}-\phi_{\omega,j}(t,x'_{\omega})}).
		\end{equation}
		Set ${\rho}(t)=\| d \mathbf{g} \|_{C^\delta_x}$ and $a=(\epsilon_0 \lambda)^{-\frac12}$. Then we have
		\begin{equation}\label{swp0}
			\| {v}(t) \|^2_{L^2_x}\lesssim \sum_{\omega,j}\| f^{\omega,j}\|^2_{H^{1+\mu}_a}, \qquad \qquad \ \ \text{if} \ \ {\rho}(t) \leq \epsilon_0,
		\end{equation}
		and
		\begin{equation}\label{swp1}
			\| {v}(t) \|^2_{L^2_x}\lesssim \epsilon_0^{-1} {\rho}(t)  \sum_{\omega,j}\| f^{\omega,j}\|^2_{H^{1+\mu}_a},  \qquad \text{if} \ \ {\rho}(t) \geq \epsilon_0 .
		\end{equation}
		\begin{proof}
			This proof follows ideas of Smith-Tataru \cite{ST} (Lemma 8.5 and Lemma 8.6 on page 340-345). Therefore we omit the details here.
			%
			%
		\end{proof}
		\begin{remark}
			By \eqref{swp0} and \eqref{swp1}, we get
			\begin{equation}\label{swp3}
				\| v(t) \|^2_{L^2_x}\lesssim \left(1+\epsilon_0^{-1} {\rho}(t)\right)  \sum_{\omega,j}\| f^{\omega,j}\|^2_{H^{1+\mu}_a}.
			\end{equation}
		\end{remark}
	\end{Lemma}
	\begin{proposition}[\cite{ST}, Proposition 8.4]\label{szy}
		Let $f=\sum_{\omega,j}a_{\omega,j}f^{\omega,j}$, where $f^{\omega,j}$ are normalized wave packets supported in $T_{\omega,j}$. Then we have
		\begin{equation}\label{ese}
			\| d P_\lambda f\|_{L^\infty_t L^2_x} \lesssim (\sum_{\omega,j} a^2_{\omega,j})^{\frac12},
		\end{equation}
		and
		\begin{equation}\label{ese1}
			\| \square_{\mathbf{g}_{\lambda}} P_\lambda f \|_{L^1_t L^2_x} \lesssim \epsilon_0 (\sum_{\omega,j} a^2_{\omega,j})^{\frac12}.
		\end{equation}
	\end{proposition}
	\begin{proof}
		We first prove a weaker estimate comparing with \eqref{ese}
		\begin{equation}\label{ese2}
			\| d P_\lambda f\|_{L^2_t L^2_x} \lesssim (\sum_{\omega,j} a^2_{\omega,j})^{\frac12}.
		\end{equation}
		By using \eqref{swp3} and replacing $P_\lambda$ by $\lambda^{-1}\nabla P_\lambda$, and $f^{\omega,j}$ by $a_{\omega,j}\psi^{\omega,j}$, we have
		\begin{equation*}
			\| \nabla P_\lambda f(t)\|^2_{L^2_x} \lesssim (1+\epsilon_0^{-1} {\rho}(t)) \sum_{\omega,j} a^2_{\omega,j}.
		\end{equation*}
		Due to the fact $\|{\rho}\|_{L^2_t} \lesssim \epsilon_0$, we can see that
		\begin{equation}\label{ese3}
			\| \nabla P_\lambda f\|^2_{L^2_t L^2_x} \lesssim \sum_{\omega,j} a^2_{\omega,j} .
		\end{equation}
		We also need to get the similar estimate for the time derivatives. We can calculate
		\begin{equation*}
			\partial_t \psi = \dot{\gamma}(t) (\epsilon_0 \lambda)^{\frac12} \tilde{\psi}, \quad \partial_t \delta(x_\omega-\phi_{\omega,j})=\partial_t \phi_{\omega,j} \delta^{(1)}(x_\omega-\phi_{\omega,j}).
		\end{equation*}
		For $\dot{\gamma} \in L^\infty_t$ and $\partial_t \phi_{\omega,j} \in L^\infty_t$, then we have
		\begin{equation}\label{ese4}
			\| \partial_t P_\lambda f\|^2_{L^2_t L^2_x} \lesssim \sum_{\omega,j} a^2_{\omega,j} .
		\end{equation}
		Combining \eqref{ese3} and \eqref{ese4}, we have proved \eqref{ese2}. To prove \eqref{ese1}, we use the formula \eqref{np1}. Considering the right hand of \eqref{np1}, by using \eqref{ese2}, we can bound the first term by
		\begin{equation*}
			\| L(d\mathbf{g}, d \tilde{P}_{\lambda} \tilde{f}) \|_{L^1_tL^2_x} \lesssim \|d\mathbf{g}\|_{L^2_tL^\infty_x} \|d \tilde{P}_{\lambda} \tilde{f} \|_{L^2_tL^2_x} \lesssim \epsilon_0 (\sum_{\omega,j} a^2_{\omega,j})^{\frac12}.
		\end{equation*}
		It only remains for us to estimate the second right term on \eqref{np1}. Set
		\begin{equation*}
			\Phi=\epsilon_0^{\frac12} P_\lambda T_\lambda \left(\sum_{\omega,j}a_{\omega,j}\cdot \sum_{m=0,1,2}\psi^{\omega,j}_m \delta^{(m)}_{x_{\omega}-\phi_{\omega,j}} \right).
		\end{equation*}
		By \eqref{np2}, we then have
		\begin{equation*}
			\begin{split}
				\| \Phi \|^2_{L^2_x}=&(1+{\rho}(t)\epsilon_0^{-1}) \sum_{\omega,j}a^2_{\omega,j} \sum_{m=0,1,2} \lambda^{m-1}\|\psi^{\omega,j}_m (t) \|^2_{H^{1+\mu}_a}
				\\
				\lesssim  &(1+{\rho}(t)\epsilon_0^{-1}) \epsilon_0^2  \sum_{\omega,j}a^2_{\omega,j}.
			\end{split}
		\end{equation*}
		So we further get
		\begin{equation*}
			\begin{split}
				\| \Phi \|_{L^1_t L^2_x}
				\lesssim & \epsilon_0 (\sum_{\omega,j} a^2_{\omega,j})^{\frac12} \left(\int^2_{-2}[1+{\rho}(t)\epsilon_0^{-1}]^{\frac12}dt \right)
				\\
				\lesssim & \epsilon_0 (\sum_{\omega,j} a^2_{\omega,j})^{\frac12} \left(\int^2_{-2}[1+{\rho}(t)\epsilon_0^{-1}]dt \right)^{\frac12}
				\\
				\lesssim & \epsilon_0 (\sum_{\omega,j} a^2_{\omega,j})^{\frac12} .
			\end{split}
		\end{equation*}
		Hence, the second right term on \eqref{np1} can be estimated by $\epsilon_0 (\sum_{\omega,j} a^2_{\omega,j})^{\frac12}$. So we have proved \eqref{ese1}. Using \eqref{ese2} and \eqref{ese1}, and classical energy estimates for linear wave equation, we obtain \eqref{ese}.
	\end{proof}
	\subsubsection{Matching the initial data}
	Although we have constructed the approximate solutions using superpositions of normalized wave packets, we also need to complete this construction, i.e. matching the initial data for this type of solutions. Since the metric $\mathbf{g}$ equals to the Minkowski metric for $t \in [-2,-\frac32]$, so it's natural for us to work with wave packets near $t=-2$ for the Minkowski wave operator. We refer to the construction from \cite[Proposition 8.7, page 346]{ST}
	or Smith \cite[page 815, Section 4]{Sm}.
	\begin{proposition}[\cite{ST}, Proposition 8.7]\label{szi}
		Given any initial data $(f_0,f_1) \in H^1 \times L^2$, there exists a function of the form
		\begin{equation*}
			f=\sum_{\omega,j}a_{\omega,j}f^{\omega,j},
		\end{equation*}
		where the function $f^{\omega,j}$ are normalized wave packets, such that
		\begin{equation*}
			f(-2)=P_\lambda f_0, \quad \partial_t f(-2)=P_\lambda f_1.
		\end{equation*}
		Furthermore,
		\begin{equation*}
			\sum_{\omega,j}a_{\omega,j}^2 \lesssim \| P_\lambda f_0 \|^2_{H^1}+ \| P_\lambda f_1 \|^2_{L^2}.
		\end{equation*}
	\end{proposition}
	\begin{remark}
		The original version of Proposition 8.7 (\cite{ST}, page 346) takes the initial data $(u_0,u_1)$. By replacing $(u_0,u_1)$ in Proposition 8.7 (\cite{ST}, page 346) to $(P_\lambda f_0,P_\lambda f_1)$, our Proposition \ref{szi} holds.
	\end{remark}
	\begin{remark}
		Considering the definition of a wave packet, together with the regularity of the $\Sigma_{\omega,r}$, then we can use the quantities $\psi_{\omega,r}$ and $\gamma_{\omega,r}$ to keep the type of $f$ for $t\in[-2,2]$.

		If $t_0=-2$ in Proposition \eqref{AA1}, then the function $f$ is the approximate solution matching the initial data by using Proposition \ref{szi}.

		If $t_0\in(-2,2)$, then we can use an iteration process as follows. Basically, we hope to find the exact solution $\Phi$ to the Cauchy problem
		\begin{equation*}
			\begin{cases}
				& \square_{\mathbf{g}_{\lambda}} \Phi=0, \qquad (t, x) \in [-2,2] \times \mathbb{R}^3,
				\\
				& (\Phi, \partial_t \Phi)|_{t=t_0}=(\Phi_0,\Phi_1).
			\end{cases}
		\end{equation*}
		Let $\Phi_\lambda=P_{\lambda} \Phi$ be the approximate solution with the initial data $ (\Phi_\lambda, \partial_t \Phi_\lambda)|_{t=-2}=({P}_\lambda \Phi(-2), {P}_\lambda \partial_t \Phi(-2))$. Then we have
		\begin{equation*}
			\square_{\mathbf{g}_{\lambda}} \Phi_\lambda= 
			[\square_{\mathbf{g}_{\lambda}},P_{\lambda}]\Phi.
		\end{equation*}
		The energy estimates tell us that
		\begin{equation*}
			\| \square_{\mathbf{g}_{\lambda}} \Phi_\lambda \|_{L^1_{[-2,2]} L^2_x} \lesssim \| d\mathbf{g}\|_{L^2_{[-2,2]} L^\infty_x}\| d\Phi\|_{L^\infty_{[-2,2]} L^2_x} \lesssim \epsilon_0 (\|\Phi_0\|_{H^1}+\|\Phi_1\|_{L^2}).
		\end{equation*}
		However, it does not match the data at $t=t_0$. Fortunately, we can obtain
		\begin{equation*}
			\begin{split}
				& P_\lambda \Phi_0 -\Phi_\lambda (t_0)=P_\lambda \Phi_0^1, \qquad \Phi_0^1=P_\lambda \Phi_0-\Phi_\lambda (t_0),
				\\
				& P_\lambda \Phi_1 -\partial_t \Phi_\lambda (t_0)=P_\lambda \Phi_1^1, \ \quad \Phi_1^1=P_\lambda \Phi_1-\partial_t \Phi_\lambda (t_0).
			\end{split}
		\end{equation*}
		By energy estimates and commutator estimates, we have
		\begin{equation*}
			\begin{split}
				\|\Phi_0^1\|_{H^1} + \|\Phi_1^1\|_{L^2} \lesssim & \| \square_{\mathbf{g}_{\lambda}} ( P_\lambda \Phi - \Phi) \|_{L^1_{[-2,2]} L^2_x}
				\\
				\lesssim & \| \square_{\mathbf{g}_{\lambda}}  \Phi_\lambda \|_{L^1_t L^2_x}+ \| [ \square_{\mathbf{g}_{\lambda}} , \Phi_\lambda] \Phi \|_{L^1_{[-2,2]} L^2_x}
				\\
				\lesssim & \epsilon_0 (\|\Phi_0\|_{H^1}+\|\Phi_1\|_{L^2}).
			\end{split}
		\end{equation*}
		Since the norm of the error is much smaller than the initial size of the data,
		we can repeat this process with data $(\Phi_0^1,\Phi_1^1)$, and sum this series to
		obtain a smooth function $\Phi_\lambda$ with data $(P_\lambda \Phi_0,P_\lambda \Phi_1)$ at time $t=t_0$. As a result, it can also match the given data at time $t=t_0$.
	\end{remark}
	\subsubsection{Overlap estimates}
	Since the foliations $\Sigma_{\omega,r}$ varing with $\omega$ and $r$, so a fixed $\Sigma_{\omega,r}$ may intersect with other $\Sigma_{\omega',r'}$. As a result, we should be clear about the number of $\lambda$-slabs which contain two given points in the space-time $[-2,2]\times \mathbb{R}^3$.
	\begin{corollary}[\cite{ST}, Proposition 9.2]\label{corl}
		For all points $P_1=(t_1,x_1)$ and $P_2=(t_2,x_2)$ in space-time $\mathbb{R}^{+} \times \mathbb{R}^3$, and $\epsilon_0 \lambda \geq 1$, the number $N_{\lambda}(P_1,P_2)$ of slabs of scale $\lambda$ that contain both $P_1$ and $P_2$ satisfies the bound
		\begin{equation*}
			\begin{split}
				N_{\lambda}(P_1,P_2)\lesssim & \epsilon_0^{-1} |t_1-t_2|^{-1}.
			\end{split}
		\end{equation*}
	\end{corollary}
	\begin{remark}
		Taking $n=3$ for Proposition 9.2 in \cite{ST}, we then obtain Corollary \ref{corl}.
	\end{remark}
	\subsubsection{The proof of Estimate \eqref{Ase}}
	After the construction of approximate solutions, we still need to prove the key estimate \eqref{Ase}. This can be directly derived by the following result. Before that, let us define $\mathcal{T}=\cup_{\omega,j}T_{\omega,j}$, where $T_{\omega,j}$ is set in \eqref{twj}. We also denote $\chi_{{J}}$ be a smooth cut-off function, and $\chi_J=1$ on a set $J$.
	\begin{proposition}\label{szt}
		Let $t \in [-2,2]$ and
		\begin{equation*}
			f={\sum}_{J \in \mathcal{T}}a_{J}\chi_{J}f_{J},
		\end{equation*}
		where $\sum_{J \in \mathcal{T}}a_{J}^2 \leq 1$ and $f_J$ are normalized wave packets in $J$. Then\footnote{We need $r>1$ in \eqref{Ase} for there is a factor $(\ln \lambda)^{\frac12}$ in \eqref{Aswr}.}
		\begin{equation}\label{Aswr}
			\|f \|_{L^2_t L^\infty_x} \lesssim \epsilon_0^{-\frac12}(\ln \lambda)^{\frac12}.
		\end{equation}
	\end{proposition}
	\begin{proof}
		This proof also follows Smith-Tataru \cite{ST} (Proposition 10.1 on page 353). Here we use slightly different factors in \eqref{Aswr}.
		
		Let us first make a partition of the time-interval $[-2,2]$. By decomposition, there exists a partition $\left\{ I_j \right\}$ of the  interval $[-2,2]$ into disjoint subintervals $I_j$ such that with the size of each $I_j$, $|I_j| \simeq \lambda^{-1}$, and the number of subintervals $j \simeq\lambda$. By mean value theorem, there exists a number $t_j$ such that
		\begin{equation}\label{Aswrq}
			\|f \|^2_{L^2_t L^\infty_x} \leq \textstyle{\sum}_j \|f \|^2_{L^2_{I_j} L^\infty_x} \leq \textstyle{\sum}_j \|f(t_j,\cdot) \|^2_{ L^\infty_x} \lambda^{-1},
		\end{equation}
		where $t_j$ is located in $I_j$, and $|t_{j+1}-t_j| \simeq \lambda^{-1}$. Let us explain the \eqref{Aswrq} as follows. For simplicity, let  $I_{j_0}=[0,\lambda^{-1}]$ and $I_{j_0+1}=[\lambda^{-1},2\lambda^{-1}]$. By mean value theorem, on $I_{j_0}$, we have
		\begin{equation*}
			\|f \|^2_{L^2_{I_{j_0}} L^\infty_x} = \|f(t_{j_0},\cdot) \|^2_{ L^\infty_x}\lambda^{-1}.
		\end{equation*}
		We also have
		\begin{equation*}
			\|f \|^2_{L^2_{I_{j_0+1}} L^\infty_x} = \|f(t_{j_0+1},\cdot) \|^2_{ L^\infty_x}\lambda^{-1}.
		\end{equation*}
		If $t_{j_0} \leq \frac12 \lambda^{-1}$ or $t_{j_0+1} \geq \frac32 \lambda^{-1}$, then $|t_{j_0+1}-t_{j_0}| \geq \frac12 \lambda^{-1}$. Otherwise, $t_{j_0} \in [\frac12 \lambda^{-1}, \lambda^{-1}] $ and $t_{j_0+1} \in [\lambda^{-1}, \frac32 \lambda^{-1}]$. In this case,
		we combine $I_{j_0}, I_{j_0+1}$ together and set $I^*_{j_0}= I_{j_0}\cup I_{j_0+1}$. On the new interval  $I^*_{j_0}$, we can see that
		\begin{equation*}
			\|f \|^2_{L^2_{I^*_{j_0}} L^\infty_x} = 2\|f(t^*_{j_0},\cdot) \|^2_{ L^\infty_x}\lambda^{-1},
		\end{equation*}
		and
		\begin{equation*}
			2\|f(t^*_{j_0},\cdot) \|^2_{ L^\infty_x} = \|f(t_{j_0},\cdot) \|^2_{ L^\infty_x}+\|f(t_{j_0+1},\cdot) \|^2_{ L^\infty_x}. 
		\end{equation*}
		Since $f$ is a contituous function, we have
		\begin{equation*}
			\|f(t^*_{j_0},\cdot) \|_{ L^\infty_x} = \left\{  \frac12 \big( \|f(t_{j_0},\cdot) \|^2_{ L^\infty_x} +\|f(t_{j_0+1},\cdot) \|^2_{ L^\infty_x} \big) \right\}^{\frac12}, \qquad t^*_{j_0} \in [\frac12\lambda^{-1}, \frac32\lambda^{-1}].
		\end{equation*}
		In case $t^*_{j_0} \in [\frac12\lambda^{-1}, \frac32\lambda^{-1}]$, we set $t_{j_0}:=t^*_{j_0}$. On the next time-interval $I_{j_0+2}=[2\lambda^{-1}, 3\lambda^{-1}]$, we obtain
		\begin{equation*}
			\|f \|^2_{L^2_{I_{j_0+2}} L^\infty_x} = \|f(t_{j_0+1},\cdot) \|^2_{ L^\infty_x}\lambda^{-1}, \quad t_{j_0+1} \in [2\lambda^{-1}, 3\lambda^{-1}].
		\end{equation*}
		Hence, we have proved $|t_{j_0+1}-t_{j_0}| \geq \frac12 \lambda^{-1}$. So we can use this way to decompose $[-2,2]$. Therefore, to prove \eqref{Aswr}, and taking \eqref{Aswrq} into account, we only need to show that
		\begin{equation}\label{wee0}
			\textstyle{\sum}_j |f(t_j,x_j)|^2 \lesssim \epsilon_0^{-1}\lambda \ln \lambda,
		\end{equation}
		where $x_j$ is arbitrarily chosen. We then set the points $P_j=(t_j,x_j)$.

		Since each points lies in at most $\approx \epsilon_0^{-\frac12}\lambda^{\frac12}$ slabs, so we may assume that $|a_J| \geq \epsilon_0^{\frac12}\lambda^{-\frac12}$. Then we decompose the sum $f=\sum_{J \in \mathcal{T}}a_{J}\chi_{J}f_{J}$ dyadically with respect to the size of $a_J$. 
		We next decompose the sum over $j$ via a dyadic decomposition in the numbers of slabs containing $(t_j,x_j)$. We may assume that we are summing over $M$ points\footnote{We also remark that $M$ is $\simeq \lambda$, for $t_j \in [-2,2]$ and $|t_{j+1}-t_j|\simeq \lambda^{-1}$.} $(t_j,x_j)$, each of which is contained in approximately $L$ slabs\footnote{So $L \lesssim  \epsilon_0^{-1}\lambda$ by using Corollary \ref{corl}.}. Then $|f(t_j,x_j)|\lesssim N^{-\frac12}L$ and
		\begin{equation}\label{wee1}
			\textstyle{\sum}_j |f(t_j,\bx_j)|^2 \lesssim L^2 M N^{-1}.
		\end{equation}
		Combining \eqref{wee0} with \eqref{wee1}, we only need to prove
		\begin{equation}\label{wee2}
			L^2 M N^{-1} \lesssim \epsilon_0^{-1}\lambda \ln \lambda.
		\end{equation}
		This is a counting problem, and we shall prove \eqref{wee2} by calculating in two different ways. First consider the number $K$ of pairs $(i,j)$ for which $P_i$ and $P_j$ are contained in a common slab, counted with multiplicity. For $J \in \mathcal{T}$, we denote by $n_J$ the number of points $P_j$ contained\footnote{For $J$, if it contains $3$ points $P_1,P_2,P_3$, then we will count it $(P_1,P_2), (P_2,P_1), (P_1,P_3), (P_3,P_1),(P_2,P_3),(P_3,P_2)$ in $K$. Therefore it's $3^2$.} in $J$. Then
		\begin{equation*}
			K = \sum_{n_J \geq 2} n_J^2.
		\end{equation*}
		By Jensen's inequality, we have
		\begin{equation*}
			K = \sum_{n_J \geq 2} n_J^2 \gtrsim N^{-1}  (\sum_{n_J \geq 2} n_J)^2.
		\end{equation*}
		Note that $ \sum_{J \in \mathcal{T}} n_J \approx ML$. We consider it into two cases. If
		\begin{equation*}
			\sum_{n_J \geq 2} n_J \leq \sum_{n_J =1 } n_J,
		\end{equation*}
		then $N\approx ML$. In this case, combining with the fact that $L \lesssim \epsilon_0^{-1}\lambda$, then \eqref{wee2} holds. Otherwise, we have
		\begin{equation}\label{wee3}
			K \gtrsim N^{-1}M^2L^2.
		\end{equation}
		In this case, by using Corollary \ref{corl}, we obtain
		\begin{equation}\label{wee4}
			K \lesssim \epsilon_0^{-1} \sum_{1\leq i,j \leq M, i\neq j} |t_i-t_j|^{-1}. 
		\end{equation}
		The sum is maximized in the case that $t_j$ are close as possible, i.e. if the $t_j$ are consecutive multiples of $\lambda^{-1}$. Therefore,  \eqref{wee4} yields
		\begin{equation}\label{wee5}
			K \lesssim \epsilon_0^{-1} \lambda \sum_{1\leq i,j \leq M, i\neq j} |i-j|^{-1} \lesssim M \epsilon_0^{-1} \lambda \ln \lambda, 
		\end{equation}
		where we use the fact that $M \simeq \lambda$. Combining \eqref{wee5} and \eqref{wee3}, we get \eqref{wee2}. So we have finished the proof of Proposition \ref{szt}.
	\end{proof}
	

\section{Continuous dependence}\label{cd}
In this part, we discuss the continuous dependence of solutions by using Tao's frequency envelope approach \cite{Tao1}. The approach was used in wave maps \cite{Tao1,Tao2}, incompressible Euler equations \cite{Tao}, quasilinear hyperbolic system \cite{IT1}, and compressible Euler equations \cite{AZ,Z2}. We should also mention that there are more classical methods for continuous dependence goes back to \cite{Kato, BS, AK, AM}. 
\begin{corollary}
	[Continuous dependence]\label{cv} Assume $2<s_0<s\leq \frac{5}{2}$. Let the initial data $(\bu_{0}, h_{0}, \bw_{0})$ be stated in Theorem \ref{dingli}. Then
	\begin{equation}\label{Id}
		\| u_0^0-1,\mathring{\bu}_0\|_{H^s}+ \| h_0\|_{H^s} + \| \bw_0 \|_{H^{s_0}} \leq M_0.
	\end{equation}
	Let $\{(\bu_{0m}, h_{0m}, \bw_{0m})\}_{m \in \mathbb{N}^+}$ be a sequence of initial data converging to $(\bu_0, h_0, \bw_0)$ in space $H^s \times H^s \times H^{s_0}$, where $\bw_{0m}=\mathrm{vort}(\mathrm{e}^{h_{0m}}\bu_{0m})$. Then for every $m \in \mathbb{N}^+$, there exists $T>0$ such that the Cauchy problem \eqref{WTe} with the initial data $(\bu_{0m}, h_{0m}, \bw_{0m})$ has a unique solution $(\bu_{m}, h_m, \bw_{m})$ on $[0,T]$, and the Cauchy problem \eqref{WTe} with the initial data $( \bu_{0},h_{0}, \bw_{0})$ has a unique solution $(\bu, h, \bw)$ on $[0,T]$. Moreover, the sequence $\{(\bu_{m}, h_m, \bw_{m})\}_{m \in \mathbb{N}^+}$ converges to $(\bu,h,\bw)$ on $[0,T]$ in $ H^s_x \times H_x^s \times H_x^{s_0}$.
\end{corollary}
\begin{proof}
		For $\{(\bu_{0m}, h_{0m}, \bw_{0m})\}_{m \in \mathbb{N}^+}$ is a sequence which converges to $(\bu_0, h_0, \bw_0)$ in space $H^s \times H^s \times H^{s_0}$, by using \eqref{Id}, we  obtain
	\begin{equation*}
		\| u^0_{0m}-1,\mathring{\bu}_{0m}\|_{H^s}+  \| h_{0m}\|_{H^{s}} + \| \bw_{0m} \|_{H^{s_0}} \leq 100 M_0, \quad \text{for \ large} \ m,
	\end{equation*}
	By using the existence and uniqueness result in Theorem \ref{dingli} and Corollary \ref{cor}, there exists $T>0$ such that the Cauchy problem \eqref{WTe} with the initial data $(\bu_{0m}, h_{0m}, \bw_{0m})$ has a unique solution $(\bu_{m}, h_m, \bw_{m})$ on $[0,T]$, and the Cauchy problem \eqref{WTe} with the initial data $(\bu_{0}, h_0, \bw_{0})$ has a unique solution $(\bu, h, \bw)$ on $[0,T]$. Next, we divide the proof of convergence into three steps.

	\textit{Step 1: Convergence in weaker spaces}. By the $L^2_x$ energy estimates, and using the Strichartz estimate \eqref{SSr}, for $m,l \in \mathbb{N}^+$, we have
	\begin{equation*}
		\begin{split}
				\|(\bu_m-\bu_l, h_m-h_l,\bw_m-\bw_l)\|_{L_x^{2}}
			\lesssim \|(\bu_{0m}-\bu_{0l}, h_{0m}-h_{0l},\bw_{0m}-\bw_{0l})\|_{H^s \times H^s \times H^{s_0}}.
		\end{split}
	\end{equation*}
 As a result, for $t\in [0,T]$, we have
	\begin{equation*}
		\lim_{m \rightarrow \infty}(\bu_m, h_m, \bw_m)  =  (\bu, h, \bw) \quad \mathrm{in} \ \ L^{2}_x.
	\end{equation*}
Due to the interpolation formula, it follows
	\begin{equation*}
		\|\bu_m-\bu\|_{H_x^\sigma}  \lesssim \|\bu_m-\bu\|^{1-\frac{\sigma}{s}}_{L_x^{2}} \|\bu_m-\bu\|^{\frac{\sigma}{s}}_{H_x^s}, \quad 0 \leq \sigma <s.
	\end{equation*}
	This leads to
	\begin{equation*}
		\lim_{m\rightarrow \infty}\bu_m  =  \bu \ \ \mathrm{in} \ \ H^{\sigma}_x, \ \ 0 \leq \sigma <s,
	\end{equation*}
	In a similar way, we can obtain
	\begin{equation}\label{er0}
		\begin{split}
			& \lim_{m\rightarrow \infty}h_m =  h \ \ \mathrm{in} \ \ H^{\sigma}_x, \ \ 0 \leq \sigma <s,
			\\
			& \lim_{m\rightarrow \infty}\bw_m =  \bw \ \ \mathrm{in} \ \ H^{\theta}_x, \ \ 0 \leq \theta <s_0.
		\end{split}
	\end{equation}
	It remains for us to prove the convergence with the highest derivatives of $(\bu_m,h_m,\bw_m)$.

		\textit{Step 2: Constructing smooth solutions}. We set $\bU_0=(p(h_0),u^1_{0},u^2_{0},u^3_{0},)^{\mathrm{T}}$. By \cite{Tao,IT1}, there exists a sharp frequency envelope for $u^0_0-1, u^1_{0},u^2_{0}, u^3_{0}$ and $h_0$ respectively. Let $\{ c^{(a)}_{k} \}_{k \geq 0}$($a=0,1,2,3,$) be a sharp frequency envelope respectively for $u^0_0-1, u^1_{0},u^2_{0}$, and $u^3_{0}$ in the space $ H^s$. Let $\{ c^{(4)}_{k} \}_{k \geq 0}$ be a sharp frequency envelope respectively for $h_0$ in $ H^{s_0}$. Let $d^{(0)}_{k}, d^{(1)}_{k}, d^{(2)}_{k}, d^{(3)}_{k}$ be a sharp frequency envelope for $w^0_0, w^1_0, w^2_0$ and $w^3_0$ in the space $ H^{s_0}$. We choose a family of regularizations $\bU^{l}_0=(p(h^{l}_0),u^{1l}_0, u^{2l}_0,u^{3l}_0)^{\mathrm{T}}\in \cap^\infty_{a=0}H^a$ at frequencies $\lesssim 2^l$ where $l$ is a frequency parameter and $l \in \mathbb{N}^+$. Denote
	\begin{equation*}
		u^{0l}_0=\sqrt{1+ \mathring{\bu}_0^l}, \quad \bu^l_0=(u^{0l}_0, \mathring{\bu}_0^l)^\mathrm{T}, \quad \bw^l_0=\mathrm{vort}(\mathrm{e}^{h^{l}_0}\bw^l_0).
	\end{equation*}
	Then the functions $u^{0 l}_0, u^{1 l}_0, u^{2 l}_0, u^{3 l}_0, h^{l}_0$, and $\bw^l_0=(w^{0l}_0,w^{1l}_0,w^{2l}_0,w^{3l}_0)^{\mathrm{T}}$ have the following properties:

	(i)  uniform bounds
	\begin{equation}\label{pqr0}
		\begin{split}
			& \| P_k (u^{0 l}_0-1) \|_{H_x^s} \lesssim c^{(0)}_k, \qquad \| P_k u^{i l}_0 \|_{H_x^s} \lesssim c^{(i)}_k, \qquad \quad i=1,2,3,
			\\
			&  \| P_k h^{l}_0 \|_{H_x^s} \lesssim c^{(4)}_k, \qquad \qquad \| P_k  w^{\alpha l}_0 \|_{H_x^{s_0}} \lesssim d^{(\alpha)}_k, \qquad \alpha=0,1,2,3,
		\end{split}
	\end{equation}
	
	(ii)  high frequency bounds
	\begin{equation}\label{pqr1}
		\begin{split}
			& \| P_k (u^{0 l}_0-1) \|_{H_x^{s+1}} \lesssim 2^k c^{(0)}_k, \qquad \| P_k u^{i l}_0 \|_{H_x^{s+1}} \lesssim 2^kc^{(i)}_k, \qquad \quad i=1,2,3,
		\\
		&  \| P_k h^{l}_0 \|_{H_x^{s+1}} \lesssim 2^k c^{(4)}_k, \qquad \qquad \| P_k  w^{\alpha l}_0 \|_{H_x^{s_0+1}} \lesssim 2^k d^{(\alpha)}_k, \qquad \alpha=0,1,2,3,
		\end{split}
	\end{equation}
	
	(iii)  difference bounds
	\begin{equation}\label{pqr2}
		\begin{split}
			&\|  u^{\alpha(l+1)}_0- u^{\alpha l}_0\|_{L_x^{2}} \lesssim 2^{-sl}c^{(\alpha)}_l, \quad \ \ \ \alpha=0, 1,2,3,
			\\
			& \|  w^{\alpha(l+1)}_0- w^{\alpha l}_0\|_{L_x^{2}} \lesssim 2^{-s_0l}d^{(\alpha)}_l, \quad \alpha=0, 1,2,3,
			\\
			&\| h^{l+1}_0- h^{l}_0 \|_{L_x^{2}} \lesssim 2^{-{s}l}c^{(4)}_l,
		\end{split}
	\end{equation}
	
	(iv)  limit
	\begin{equation}\label{li}
		\begin{split}
			&\lim_{l \rightarrow \infty}(h_0^l, \bu^l_0,\bw^l_0) =(h_0, \bu_0, \bw_0) \ \ \ \mathrm{in} \ \ H^s\times H^s \times H^{s_0}.
		\end{split}
	\end{equation}
Consider the Cauchy problem \eqref{WTe} with the initial data
	\begin{equation*}
		\begin{split}
			& (h^l,\bu^l,\bw^l)|_{t=0}=(h^l_0,\bu^l_0,\bw^l_0).
		\end{split}
	\end{equation*}
By Proposition \ref{p3}, we can obtain a family of smooth solutions $(h^{l}, \bu^{l}, \bw^l)$ on a time interval $[0,T]$ ($T>0$). Furthermore, based on \eqref{pqr0}-\eqref{li}, using Proposition \ref{p3} again, we can claim that

$\bullet$  high frequency bounds
\begin{equation}\label{dbsh}
	\|  u^{0l}-1,\mathring{\bu}^{l}\|_{H_x^{s+1}}+\| h^{l} \|_{H_x^{s+1}}+ \| \bw^l \|_{H_x^{s_0+1}} \lesssim 2^{l} M_0 ,
\end{equation}

$\bullet$  difference bounds
\begin{equation}\label{dbs}
	\|  (\bu^{l+1}- \bu^{l}, h^{l+1}- h^{l}) \|_{L_x^{2}} \lesssim 2^{-sl}c_l,
\end{equation}
and
\begin{equation}\label{dbs2}
	\| \bw^{l+1}- \bw^{l} \|_{L_x^{2}} \lesssim 2^{-(s-1)l} c_l,
\end{equation}
and
\begin{equation}\label{dbs1}
	\| \bw^{l+1}- \bw^{l} \|_{\dot{H}_x^{s_0}} \lesssim c_l,
\end{equation}
where $c_l=\sum^{3}_{\alpha=0}c^{(a)}_l +c^{(4)}_l + \sum^{3}_{\alpha=0}d^{(\alpha)}_l$. By using \eqref{pqr1} and Proposition \ref{p3}, then the estimate \eqref{dbsh} holds. Let us postpone to prove \eqref{dbs}-\eqref{dbs1}. By using
\begin{equation*}
	h- h^{l}=\sum^\infty_{k=l}( h^{k+1}-h^k ),
\end{equation*}
and the estimate \eqref{dbs}, we can conclude
\begin{equation}\label{ei0}
	\| h- h^{l} \|_{H_x^{s}} \lesssim c_{\geq l}.
\end{equation}
Here $c_{\geq l}=\sum_{j \geq l}c_j$. Similarly, by using \eqref{dbs}, \eqref{dbs2} and \eqref{dbs1}, we also have
\begin{equation}\label{ei1}
	\|  \bu- \bu^{l}\|_{H_x^{s}} \lesssim c_{\geq l}, \quad \| \bw- \bw^{l} \|_{H_x^{s_0}} \lesssim c_{\geq l}.
\end{equation}

\textit{Step 3: Convergence of the highest derivatives}. For the initial data $(\bu_{0m}, h_{0m}, \bw_{0m})$ and $\bu_{0m}=(u^0_{0m},u^1_{0m},u^2_{0m},u^3_{0m})$, let $\{ c^{(i)m}_{k}\}_{k\geq 0}$ be frequency envelopes for the initial data $u^{i}_{0m}$ in $H^s (i=1,2,3)$. Let $\{ c^{(4)m}_{k}\}_{k\geq 0}$ be frequency envelopes for $u^0_{0m}-1$ in $H^{s}$. Let $\{ d^{(\alpha)m}_{k}\}_{k\geq 0}$ be frequency envelopes for the initial data $w^\alpha_{0m}$ in $H^{s_0}(\alpha=0,1,2,3)$. We choose a family of regularizations $(u^{1l}_{0m}, u^{2l}_{0m}, u^{3l}_{0m}, h^{l}_{0m})\in \cap^\infty_{a=0}H^a$ at frequencies $\lesssim 2^l$. Denote
\begin{equation*}
	u^{0l}_{0m}=\sqrt{1+|\mathring{\bu}^{l}_{0m}|^2}, \quad \bu^{l}_{0m}=(u^{0l}_{0m},\mathring{\bu}^{l}_{0m})^{\mathrm{T}}, \quad \bw^l_{0m}=\mathrm{vort}(\mathrm{e}^{h^{l}_{0m}}\bu^l_{0m}).
\end{equation*}
Then, $\bu^{l}_{0m}$, $h^{l}_{0m}$, and $\bw^l_{0m}$ satisfy

(1)  uniform bounds
\begin{equation*}
	\begin{split}
		& \| P_k (u^{0l}_{0m}-1) \|_{H_x^s} \lesssim c^{(0)m}_k, \quad \| P_k u^{il}_{0m} \|_{H_x^{s}} \lesssim c^{(i)m}_k, \ \quad i=1,2,3,
		\\
		&  \| P_k h^{l}_{0m} \|_{H_x^s} \lesssim c^{(4)m}_k, \qquad \| P_k w^{\alpha l}_{0m} \|_{H_x^{s_0}} \lesssim d^{(\alpha)m}_k,\qquad \alpha=0,1,2,3,
	\end{split}
\end{equation*}

(2)  high frequency bounds
\begin{equation*}
	\begin{split}
		& \|P_k   ( u^{0 l}_{0m} -1) \|_{H_x^{s+1}} \lesssim 2^{k}c^{(0)}_k, \quad \|P_k   u^{i l}_{0m} \|_{H_x^{s+1}} \lesssim 2^{k}c^{(i)}_k, \quad i=1,2, 3,
		\\
		& \|P_k   h^{l}_{0m} \|_{H_x^{s+1}} \lesssim 2^{k}c^{(4)}_k, \qquad  \|P_k w^{\alpha l}_{0m} \|_{H_x^{s_0}} \lesssim 2^{k}d^{(\alpha)}_k,\qquad \ \alpha=0,1,2,3,
	\end{split}
\end{equation*}

(3)  difference bounds
\begin{equation*}
	\begin{split}
		&\|  u^{\alpha(l+1)}_{0m}- u^{\alpha l}_{0m}\|_{L_x^{2}} \lesssim 2^{-sl}c^{(\alpha )m}_l, \quad \ \ \alpha=0, 1,2,3,
		\\
		&  \|  w^{\alpha (l+1)}_{0m}- w^{\alpha l}_{0m}\|_{L_x^{2}} \lesssim 2^{-s_0l}d^{(\alpha)m}_l, \quad \alpha=0, 1,2,3,
		\\
		&\| h^{l+1}_{0m}- h^{l}_{0m} \|_{L_x^{2}} \lesssim 2^{-sl}c^{(4)m}_l,
	\end{split}
\end{equation*}

(4)  limit
\begin{equation}\label{SSS}
	\begin{split}
		&\lim_{l\rightarrow \infty}(h^l_{0m},\bu_{0m}^l,\bw_{0m}^l) =(h_{0m},\bu_{0m},\bw_{0m}) \  \mathrm{in} \ H^s\times H^s \times H^{s_0}.
	\end{split}
\end{equation}
By Proposition \ref{p3}, we can obtain a family of smooth solutions $(h_j^{l}, \bu_j^{l}, \bw_j^l)$ on a time interval\footnote{For the initial sequence is uniformly bounded, then the time interval $[0,T]$ can be independent for $j$ and $l$.} $[0,T]$ with the initial data $ (h_{0m}^{l}, \bu_{0m}^{l}, \bw_{0m}^{l})$. Let us now prove
\begin{equation*}
	\lim_{m\rightarrow \infty} h_m = h,  \quad \mathrm{in} \ \ H^s_x.
\end{equation*}
Inserting $h^{l}_m$ into $h_{m}-h$, we have
\begin{equation}\label{cd1}
	\begin{split}
		\|h_{m}-h\|_{H_x^s}
		\leq & \|h^{l}_m-h^{l}\|_{H_x^s}+ \|h^{l}-h\|_{H_x^s}+ \|h^{l}_m-h_m\|_{H_x^s}.
	\end{split}
\end{equation}
Due to \eqref{SSS}, we have
\begin{equation}\label{cd4}
	\lim_{m\rightarrow \infty}h^l_{0m}= h^l_0 \quad  \mathrm{in} \ {H^\sigma_x}, \ 0\leq \sigma < \infty.
\end{equation}
Using similar idea of the proof of \eqref{er0}, we can derive that
\begin{equation}\label{cd5}
	\lim_{m\rightarrow \infty}h^l_{m}= h^l \quad  \mathrm{in} \ {H^\sigma_x}, \ 0 \leq \sigma < \infty.
\end{equation}
From \eqref{cd4}, we have
\begin{equation}\label{cd6}
	c^{(4)j}_{k} \rightarrow c^{(4)}_{k} , \ j\rightarrow \infty.
\end{equation}
By \eqref{ei0},  we then bound \eqref{cd1} by
\begin{equation}\label{cd3}
	\begin{split}
		\|h_{m}-h\|_{H_x^s}
		\lesssim \|h^{l}_m-h^{l}\|_{H_x^s}+ c_{\geq l}+ c^{(4)m}_{\geq l},
	\end{split}
\end{equation}
Taking the limit of \eqref{cd3} for $j\rightarrow \infty$, by using \eqref{cd5} and \eqref{cd6}, it leads to
\begin{equation}\label{Cd3}
	\begin{split}
		\lim_{m\rightarrow \infty }\|h_{m}-h\|_{H_x^s}
		\lesssim c_{\geq l}+ c^{(4)}_{\geq l} \lesssim c_{\geq l}.
	\end{split}
\end{equation}
Finally, taking $l\rightarrow \infty$ for \eqref{Cd3} and using $\lim_{l\rightarrow \infty} c_{\geq l}=0$, we have
\begin{equation*}
	\begin{split}
		\lim_{m\rightarrow \infty}\|h_{m}-h\|_{H_x^s}=0.
	\end{split}
\end{equation*}
Similarly, by using \eqref{dbs}, \eqref{dbs2}, and \eqref{dbs1}, we can also obtain
\begin{equation*}
	\lim_{m\rightarrow \infty}\|\bu_{m}-\bu\|_{H_x^s}=0, \quad
	\lim_{m\rightarrow \infty}\|\bw_{m}-\bw\|_{H_x^{s_0}}=0.
\end{equation*}
\quad \textit{Step 4: Proof of \eqref{dbsh}-\eqref{dbs1}}.
	Firstly, we can obtain \eqref{dbsh} by energy estimates and Strichartz estimate in Proposition \ref{r5}. To prove \eqref{dbs}-\eqref{dbs1}, we also need the hyperbolic system to discuss $\bu^{l+1}-\bu^l$ and $h^{l+1}-h^l$. But for $\bw^{l+1}-\bw^l$, we use it's transport equations of $\bw^{l+1}$ and $\bw^l$. However, if we consider the difference terms $(\bu^{l+1}-\bu^l, h^{l+1}-h^l, \bw^{l+1}-\bw^l)$, then the original system is destroyed or it's a disturbed system. To avoid the loss of derivatives of $\bu^{l+1}-\bu^l$ and $h^{l+1}-h^l$, the Strichartz estimates in Proposition \ref{r5} plays a key role. While, the difference $\bw^{l+1}-\bw^l$ is weaker than $\bv^{l+1}-\bv^l$ for there is no dispersion.

Let $\bU^{l}=(p(h^l), u^{1l}, u^{2l}, u^{3l})^\mathrm{T} $. Then $\bU^{l+1}-\bU^{l}$ satisfies
\begin{equation}\label{Fhe}
	\begin{cases}
		& A^0(\bU^{l+1}) \partial_t ( \bU^{l+1}- \bU^{l}) + A^i(\bU^{l+1}) \partial_i ( \bU^{l+1}- \bU^{l})=\Pi^l,
		\\
		& ( \bU^{l+1}-\bU^{l} )|_{t=0}= \bU_0^{l+1}-\bU_0^{l},
	\end{cases}
\end{equation}
where
\begin{equation}\label{Fh}
	\Pi^l=-[A^0(\bU^{l+1})-A^0(\bU^{l}) ]\partial_t  \bU^{l}- [A^i(\bU^{l+1})-A^i(\bU^{l}) ]\partial_i  \bU^{l}.
\end{equation}
Multiplying with $\bU^{l+1}-\bU^{l}$, integrating it on $[0,T] \times \mathbb{R}^3$, and using Strichartz estimates and Gronwall's inequality, we then get
\begin{equation}\label{fff1}
	\begin{split}
		\| \bU^{l+1}-\bU^{l} \|_{L^2_x} \lesssim \| \bU^{l+1}_0-\bU^{l}_0 \|_{L^2_x} \exp\left( \int^T_0  \| d\bU^{l+1}, d\bU^{l} \|_{L^\infty_x} d\tau \right) \lesssim 2^{-sl} c_l.
	\end{split}
\end{equation}
Hence, we have
\begin{equation}\label{fff2}
	\begin{split}
		\|h^{l+1}-h^l, \mathring{\bu}^{l+1}-\mathring{\bu}^{l} \|_{L^2_x} \lesssim 2^{-sl} c_l.
	\end{split}
\end{equation}
For $u^{0l}=\sqrt{1+ |\mathring{\bu}^l|^2}$, then we have
\begin{equation}\label{fff3}
	\begin{split}
		\| u^{0(l+1)}-u^{0l} \|_{L^2_x} \lesssim 2^{-sl} c_l.
	\end{split}
\end{equation}
So we have already proved \eqref{dbs}. For $0< t \leq T$, we assume
	\begin{equation}\label{lyr1}
		\begin{split}
			& \| \bw^{l+1}- \bw^{l} \|_{L^\infty_{[0,T]}L_x^{2}} \leq 10C 2^{-(s-1)l} c_l,\qquad \| \bw^{l+1}- \bw^{l} \|_{L^\infty_{[0,T]}\dot{H}_x^{s_0}} \leq 10C c_l,
		\end{split}
	\end{equation}
	and
	\begin{equation}\label{lyr2}
		\begin{split}
			& 2^l \|  \bu^{l+1}-  \bu^{l}, h^{l+1}-  h^{l} \|_{L^2_{[0,T]} L^\infty_x}
		 + 2^l\|  \bu^{l+1}-  \bu^{l}, h^{l+1}-  h^{l} \|_{L^2_{[0,T]} \dot{B}^{s_0-2}_{\infty,2}}  \leq Cc_l.
		\end{split}
	\end{equation}
By using Lemma \ref{tr0}, we get
\begin{equation}\label{lyr3}
	\begin{split}
		& \partial_t (\bw^{l+1}-\bw^l) + \left( (u^{0(l+1)})^{-1} \mathring{\bu}^{l+1} \cdot \nabla \right) (\bw^{l+1}-\bw^l)
		\\
		=&\left\{  \left( (u^{0(l+1)})^{-1} \mathring{\bu}^{l+1}-(u^{0l})^{-1} \mathring{\bu}^l \right) \cdot \nabla \right\}  \bw^l
		\\
		&-\left(  (u^{0(l+1)})^{-1} {\bu}^{l+1}-(u^{0l})^{-1} {\bu}^l \right) w^{\kappa(l+1)} \partial_\kappa h^{l+1}
		\\
		& -(u^{0l})^{-1} {\bu}^l ( w^{\kappa(l+1)}- w^(\kappa l)) \partial_\kappa h^{l+1}-(u^{0l})^{-1} {\bu}^l w^{\kappa l} \partial_\kappa ( h^{l+1} - h^l)
		\\
		& + \left( (u^{0(l+1)})^{-1}-(u^{0l})^{-1} \right) w^{\kappa(j+1)} \partial_\kappa  \bu^{l+1}+(u^{0l})^{-1} (w^{\kappa(l+1)}-w^{\kappa l}) \partial_\kappa  \bu^{l+1}
		\\
		&+(u^{0l})^{-1} w_{\kappa l} \partial_\kappa ( \bu^{l+1}-\bu^l)
		-\left( (u^{0(l+1)})^{-1} - (u^{0l})^{-1} \right) \bw^{l+1} \partial_\kappa u^{\kappa(j+1)}
		\\
		&-(u^{0l})^{-1} ( \bw^{l+1} - \bw^l) \partial_\kappa u^{\kappa(l+1)}
		-	(u^{0l})^{-1} \bw^l \partial_\kappa ( u^{\kappa(l+1)}- u^{\kappa l})
	\end{split}
\end{equation}
Multiplying \eqref{lyr3} with $\bw_{j+1}- \bw_{j}$ and integrating on $[0,T]\times \mathbb{R}^3$, we obtain
\begin{equation}\label{lyr4}
	\begin{split}
		& \| \bw^{l+1}- \bw^{l} \|^2_{L^\infty_{[0,T]} L^2_x}
		\\
		\leq  & \| \bw^{l+1}_0- \bw_{0}^l\|^2_{L^2_x}+ C\textstyle{\int}^{T}_0  \| \bw^{l+1}-\bw^l \|^2_{L^2_x}\| d h^{l+1},d\bu^{l+1}\|_{L^\infty_x} d\tau
		\\
		& + C\textstyle{\int}^{T}_0  \| \bu^{l+1}-\bu^l \|_{L^2_x}\| \nabla \bw^{l}\|_{L^\infty_x} \| \bw^{l+1}- \bw^{l} \|_{ L^2_x}d\tau
		\\
		& + C\textstyle{\int}^{T}_0 \| \bu^{l+1}-\bu^l\|_{L^2_x} \| \bw^{l+1}\|_{L^\infty_x}\|d h^{l+1},d\bu^{l+1} \|_{L^\infty} \| \bw^{l+1}- \bw^{l} \|_{ L^2_x}d\tau
		\\
		&+ C \textstyle{\int}^{T}_0 \| \bw^{l}\|_{L^\infty_x} \|d (h^{l+1}-h^l),d (\bu^{l+1}-\bu^l) \|_{L^2_x} \| \bw^{l+1}- \bw^{l} \|_{ L^2_x}d\tau.
	\end{split}
\end{equation}
By H\"older's inequality, we can bound \eqref{lyr4} by
\begin{equation}\label{lyr5}
	\begin{split}
		& \| \bw^{l+1}- \bw^{l} \|^2_{L^\infty_{[0,T]} L^2_x}
		\\
		\leq  & \| \bw^{l+1}_0- \bw_{0}^l\|^2_{L^2_x}+ C\textstyle{\int}^{T}_0  \| \bw^{l+1}-\bw^l \|^2_{L^2_x}\| d h^{l+1},d\bu^{l+1}\|_{L^\infty_x} d\tau
		\\
		& + C\textstyle{\int}^{T}_0  \| \bu^{l+1}-\bu^l \|^2_{L^2_x} ( \| \nabla \bw^{l}\|^2_{L^\infty_x} + \| \bw^{l+1}\|^2_{L^\infty_x}\|d h^{l+1},d\bu^{l+1} \|^2_{L^\infty}) d\tau
		\\
		&+ C \textstyle{\int}^{T}_0 \| \bw^{l}\|^2_{L^\infty_x}  \|\nabla (h^{l+1}-h^l),\nabla (\bu^{l+1}-\bu^l) \|^2_{L^2_x} d\tau.
	\end{split}
\end{equation}
Due to \eqref{lyr1} and \eqref{lyr2}, it yields
\begin{equation}\label{lyr6}
	\begin{split}
		\| \bw^{l+1}- \bw^{l} \|^2_{L^\infty_{[0,T]} L^2_x}
		\leq  & C^2 (2^{-sl}c_l)^2+ 100 C^2  (2^{-sl}c_l)^2 2^{-\frac{l}{2}}
		+ C^2 (2^{-(s-1)l}c_l)^2
		\\
		&  + C^2 \textstyle{\int}^{T}_0  \| \bw^{l+1}-\bw^l \|^2_{L^2_x}\| d h^{l+1},d\bu^{l+1}\|_{L^\infty_x} d\tau.
	\end{split}
\end{equation}
Using Gronwall's inequality for \eqref{lyr6}, we then get
\begin{equation}\label{lyr7}
	\| \bw^{l+1}- \bw^{l} \|_{L^\infty_{[0,T]} L_x^{2}} \leq 5C 2^{-(s-1)l} c_l.
\end{equation}
Similarly, we can also prove
\begin{equation}\label{lyr8}
	\| \bw^{l+1}- \bw^{l} \|_{L^\infty_{[0,T]}\dot{H}_x^{s_0}} \leq 5C c_l,
\end{equation}
and
\begin{equation}\label{lyr9}
\begin{split}
& 2^l \|  \bu^{l+1}-  \bu^{l}, h^{l+1}-  h^{l} \|_{L^2_{[0,T]} L^\infty_x}
+ 2^l\|  \bu^{l+1}-  \bu^{l}, h^{l+1}-  h^{l} \|_{L^2_{[0,T]} \dot{B}^{s_0-2}_{\infty,2}}  \leq \frac{C}{2}c_l.
\end{split}
\end{equation}
Therefore, \eqref{dbs2}-\eqref{dbs1} hold. At this stage, we have finished the whole proof of Corollary \ref{cv}.
\end{proof}
\section{Proof of Theorem \ref{dingli2}}\label{Sec5}
Our goal is to prove Theorem \ref{dingli2}. As a first step, we reduce the proof to Proposition \ref{DDL2}. Secondly, we prove Proposition \ref{DDL2} through Proposition \ref{DDL3}. Finally, we use Proposition \ref{p1} to obtain Proposition \ref{DDL3}.
\subsection{Proof of Theorem \ref{dingli2} by Proposition \ref{DDL2} }\label{keypra}
In Theorem \ref{dingli2}, it includes the existence, uniqueness, and continuous dependence of solutions. The key part for us is to prove the existence of solutions. To do that, we construct a sequence of smooth solutions by frequency truncation. After that, we prove the sequence converges in some Sobolev spaces. So the limit of this sequence is a solution. To find this sequence, for $j \in \mathbb{Z}^{+}$, we introduce the initial sequence $(\bu_{0j},h_{0j})$ by
\begin{equation}\label{Dss0}
	\begin{split}
			& \mathring{\bu}_{0j}=P_{\leq j}\mathring{\bu}_0, \quad u^0_{0j}=\sqrt{1+|\mathring{\bu}_{0j}|^2 },
			\\
			& h_{0j}=P_{\leq j}h_0, \quad \bu_{0j}=(u^0_{0j},\mathring{\bu}_{0j}),
	\end{split}
\end{equation}
where $h_0$ and $\mathring{\bu}_0$ is stated as \eqref{chuzhi3} in Theorem \ref{dingli2}, and $P_{\leq j}=\textstyle{\sum}_{k\leq j}P_k$. Following \eqref{VVd}, so we define
\begin{equation}\label{Qss0}
	\bw_{0j}=\mathrm{vort}( \mathrm{e}^{h_{0j}} \bu_{0j}).
\end{equation}
Calculate
\begin{equation*}
	\begin{split}
	w^\alpha_{0j}=&	\epsilon^{\alpha \beta \gamma \eta} \mathrm{e}^{h_{0j}}u_{0\beta j}\partial_\gamma u_{0 \eta j}
	\\
	=& \mathrm{e}^{h_{0j}}u_{0\beta j} ( \epsilon^{\alpha \beta \gamma \eta} \partial_\gamma u_{0 \eta j} )
	\\
	=& \mathrm{e}^{h_{0j}}u_{0\beta j} P_{\leq j }\left( \mathrm{e}^{h_0}w^\alpha_0 u^\beta_0- \mathrm{e}^{h_0}u^\alpha_0 w^\beta_0 \right).
	\end{split}
\end{equation*}
By H\"older's inequality, then we have
\begin{equation}\label{Ess0}
	\begin{split}
		\|\bw_{0j}\|_{H^2}
		\leq & C\|\bw_{0}\|_{H^2}( 1+ \|\rho_0\|^4_{H^2}+ \|\nabla {\bu}_{0}\|^4_{H^1} )
		\\
		\leq & C(M_*+M^5_*).
	\end{split}
\end{equation}
Adding \eqref{Dss0} and \eqref{Ess0}, we can see
\begin{equation}\label{pu0}
	\|\bu_{0j}\|_{H^s}+ \|h_{0j}\|_{H^s}+ \| \bw_{0j} \|_{H^2} \leq C(M_*+M^5_*)=E(0).
\end{equation}
Using \eqref{HEw} and \eqref{Dss0}, we get
\begin{equation}\label{pu00}
	| \mathring{\bu}_{0j}, h_{0j} | \leq C_0, \quad 1\leq u^0_{0j}\leq 1+C_0, \quad c_s|_{t=0}\geq c_0>0.
\end{equation}

Before we give a proof of Theorem \ref{dingli2}, let us now introduce Proposition \ref{DDL2}.
\begin{proposition}\label{DDL2}
	Let $s$ be stated in Theorem \ref{dingli2} and \eqref{HEw}-\eqref{chuzhi3} hold. For each $j\geq 1$, consider Cauchy problem \eqref{WTe} with the initial data $(h_{0j}, \bu_{0j}, \bw_{0j})$. Then for all $j \geq 1$, there exists two positive constants $T^{*}>0$ and ${M}_{2}>0$ ($T^*$ and ${M}_{2}$ only depends on $s,C_0, c_0$ and ${M}_{*}$) such that \eqref{WTe} has a unique solution $(h_{j},\bu_{j},\bw_j)$ satisfying $(h_{j},u^0_j-1,\mathring{\bu}_{j})\in C([0,T^*];H_x^s)\cap C^1([0,T^*];H_x^{s-1})$, $\bw_{j}\in C([0,T^*];H_x^2) \cap C^1([0,T^*];H_x^1)$. To be precise,

	$\mathrm{(1)}$ the solution $h_j, \bu_j$ and $\bw_j$ satisfy the energy estimates
	\begin{equation}\label{Duu0}
		\|h_j\|_{L^\infty_{[0,T^*]}H_x^{\sstar}}+\|(u^0_j-1,\mathring{\bu}_j)\|_{L^\infty_{[0,T^*]}H_x^{\sstar}}+ \|\bw_j\|_{L^\infty_{[0,T^*]}H_x^{2}} \leq {M}_{2},
	\end{equation}
	and
	\begin{equation}\label{Duu1}
		\|\partial_t \bu_j, \partial_th_j\|_{L^\infty_{[0,T^*]}H_x^{\sstar-1}}+ \|\partial_t \bw_j\|_{L^\infty_{[0,T^*]}H_x^{1}} \leq {M}_{2},
	\end{equation}
	and
	\begin{equation}\label{Duu00}
		\|h_j, u^0_j-1,\mathring{\bu}_j\|_{L^\infty_{[0,T^*]\times \mathbb{R}^3}} \leq 2+C_0, \quad u^0_j \geq 1.
	\end{equation}
	
	$\mathrm{(2)}$ the solution $h_j$ and $\bu_j$ satisfy the Strichartz estimate
	\begin{equation}\label{Duu2}
		\|dh_j, d\bu_j\|_{L^2_{[0,T^*]}L_x^\infty} \leq {M}_{2}.
	\end{equation}

	$\mathrm{(3)}$ for $\frac{s}{2} \leq r \leq 3$, consider the following linear wave equation
	\begin{equation}\label{Duu21}
		\begin{cases}
			\square_{{g}_j} f_j=0, \qquad [0,T^*]\times \mathbb{R}^3,
			\\
			(f_j,\partial_t f_j)|_{t=0}=(f_{0j},f_{1j}),
		\end{cases}
	\end{equation}
	where $(f_{0j},f_{1j})=(P_{\leq j}f_0,P_{\leq j}f_1)$ and $(f_0,f_1)\in H_x^r \times H^{r-1}_x$. Then there is a unique solution $f_j$ on $[0,T^*]\times \mathbb{R}^3$. Moreover, for $a\leq r-\frac{s}{2}$, we have
	\begin{equation}\label{Duu22}
		\begin{split}
			&\|\left< \nabla \right>^{a-1} d{f}_j\|_{L^2_{[0,T^*]} L^\infty_x}
			\leq  {M}_3(\|{f}_0\|_{{H}_x^r}+ \|{f}_1 \|_{{H}_x^{r-1}}),
			\\
			&\|{f}_j\|_{L^\infty_{[0,T^*]} H^{r}_x}+ \|\partial_t {f}_j\|_{L^\infty_{[0,T^*]} H^{r-1}_x} \leq  {M}_3 (\| {f}_0\|_{H_x^r}+ \| {f}_1\|_{H_x^{r-1}}),
		\end{split}
	\end{equation}
	where ${M}_3$ is a constant depends on $C_0, c_0, M_*, s$.
\end{proposition}
Based on Proposition \ref{DDL2}, we are now ready to prove Theorem \ref{dingli2}.
\medskip\begin{proof}[Proof of Theorem \ref{dingli2}] For there are four points in Theorem \ref{dingli2}, so we divide the proof into four steps, The following Step 1 and Step 2 is for Point 1 and Point 2 in Theorem \ref{dingli2}. Step 3 is for giving a proof to Point 3. At last, we prove Point 4 in Step 4.

	\textit{Step 1: Obtaining a solution of Theorem \ref{dingli2} in a weaker spaces by taking a limit.}
	If we set
	\begin{equation*}
		\bU_j=(p(h_j), u^1_j, u^2_j, u^3_j)^{\mathrm{T}},
	\end{equation*}
	by Lemma \ref{QH} and \eqref{CEQ}, for $j \in \mathbb{N}^{+}$ we have
	\begin{equation*}
		\begin{split}
			& A^0 (\bU_j) \partial_t \bU_j+ A^a (\bU_j) \partial_a \bU_j=0,
			\\
				& \partial_t w_j^\alpha + (u_j^0)^{-1} u_j^a \partial_a w_j^\alpha + (u_j^0)^{-1} u_j^\alpha w_j^\kappa \partial_\kappa h_j- (u_j^0)^{-1} w_j^\kappa \partial_\kappa u_j^\alpha+ (u_j^0)^{-1} w_j^\alpha \partial_\kappa u_j^\kappa=0.
		\end{split}
	\end{equation*}
	Then for any $j, l \in \mathbb{N}^{+}$, we have
	\begin{align}\label{ur}
			A^0 (\bU_j) \partial_t (\bU_j-\bU_l)+ A^a (\bU_j) \partial_a (\bU_j-\bU_l)=&-\{ A^\alpha (\bU_j)-A^\alpha (\bU_l)\} \partial_\alpha \bU_l,
			\\
			\nonumber
			\partial_t (w_j^\alpha-w_l^\alpha) + (u_j^0)^{-1} u_j^a \partial_a (w_j^\alpha-w_l^\alpha)=& \{ (u_j^0)^{-1}u_j^a-(u_l^0)^{-1} u_l^a\} \partial_a w_l^\alpha
			\\
			\label{ur0}
			& -(u_j^0)^{-1} u_j^\alpha w_j^\kappa \partial_\kappa h_j+(u_l^0)^{-1} u_l^\alpha w_l^\kappa \partial_\kappa h_l
			\\
			\nonumber
			& + (u_j^0)^{-1} w_j^\kappa \partial_\kappa u_j^\alpha-(u_l^0)^{-1} w_l^\kappa \partial_\kappa u_l^\alpha
			\\
			\nonumber
			& -(u_j^0)^{-1} w_j^\alpha \partial_\kappa u_j^\kappa+(u_l^0)^{-1} w_l^\alpha \partial_\kappa u_l^\kappa.
	\end{align}
	Using the estimates \eqref{Duu0}-\eqref{Duu2} in Proposition \ref{DDL2}, we can show that
	\begin{equation*}
	\small	\| \bU_{j}-\bU_{l}\|_{L^\infty_{[0, T^*]}H^1_x}+\| \bw_{j}-\bw_{l}\|_{L^\infty_{[0, T^*]}H^1_x} \leq C_{{M}_2} (\| \bu_{0j}-\bu_{0l}\|_{H^{\sstar}}+ \|h_{0j}-h_{0l}\|_{H^{\sstar}}+ \| \bw_{0j}-\bw_{0l}  \|_{H^{2}}).
	\end{equation*}
	Here $C_{{M}_2}$ is a constant depending on ${M}_2$. By Lemma \ref{jh0}, so we get
	\begin{equation*}
	\small	\| \mathring{\bu}_{j}-\mathring{\bu}_{l},h_{j}-h_{l}\|_{L^\infty_{[0, T^*]}H^1_x}+\| \bw_{j}-\bw_{l}\|_{L^\infty_{[0, T^*]}H^1_x} \leq C_{{M}_2} (\| \bu_{0j}-\bu_{0l}\|_{H^{\sstar}}+ \|h_{0j}-h_{0l}\|_{H^{\sstar}}+ \| \bw_{0j}-\bw_{0l}  \|_{H^{2}}).
	\end{equation*}
	For $u^0_j=\sqrt{1+ |\mathring{\bu}_{j}|^2 }$, so we can bound
	\begin{equation*}
	\small	\| {\bu}_{j}-{\bu}_{l},h_{j}-h_{l}\|_{L^\infty_{[0, T^*]}H^1_x}+\| \bw_{j}-\bw_{l}\|_{L^\infty_{[0, T^*]}H^1_x} \leq C_{{M}_2} (\| \bu_{0j}-\bu_{0l}\|_{H^{\sstar}}+ \|h_{0j}-h_{0l}\|_{H^{\sstar}}+ \| \bw_{0j}-\bw_{0l}  \|_{H^{2}}).
	\end{equation*}
	Using \eqref{ur} and \eqref{ur0} again, we have
	\begin{equation*}
	\small	\|\partial_t ( \bU_{j}-\bU_{l} )\|_{L^\infty_{[0, T^*]}L^2_x}+\|\partial_t (  \bw_{j}-\bw_{l} )\|_{L^\infty_{[0, T^*]}L^2_x} \leq C_{M_2} (\| \bv_{0j}-\bv_{0l}\|_{H^{\sstar}}+ \|\rho_{0j}-\rho_{0l}\|_{H^{\sstar}}+ \| \bw_{0j}-\bw_{0l}  \|_{H^{2}}).
	\end{equation*}
	By Lemma \ref{jh0} and $u^0_j=\sqrt{1+ |\mathring{\bu}_{j}|^2 }$, it follows
		\begin{equation*}
	\small	\|\partial_t ( \bu_{j}-\bu_{l} )\|_{L^\infty_{[0, T^*]}L^2_x}+\|\partial_t (  \bw_{j}-\bw_{l} )\|_{L^\infty_{[0, T^*]}L^2_x} \leq C_{M_2} (\| \bu_{0j}-\bu_{0l}\|_{H^{\sstar}}+ \|h_{0j}-h_{0l}\|_{H^{\sstar}}+ \| \bw_{0j}-\bw_{0l}  \|_{H^{2}}).
	\end{equation*}
	Therefore, the sequence $\{(h_{j}, \bu_{j}, \bw_{j})\}_{j\in \mathbb{N}^+}$ is a Cauchy sequence in $H^1_x$, and $\{(\partial_t h_{j}, \partial_t\bu_{j}, \partial_t\bw_{j})\}_{j\in \mathbb{N}^+}$ is a Cauchy sequence in $L^2_x$. So there is a limit $( h, \bu, \bw)$ such that
	\begin{equation}\label{rwe}
		\begin{split}
			(h_{j}, \bu_{j}, \bw_{j})\rightarrow & ( h, \bu, \bw) \quad \quad \ \ \quad \text{in} \ \ H_x^{1} \times H_x^{1} \times H_x^{1},
			\\
			(\partial_t h_{j}, \partial_t\bu_{j}, \partial_t\bw_{j})\rightarrow & ( \partial_t h, \partial_t\bu, \partial_t\bw) \quad \text{in} \ \ L_x^{2} \times L_x^{2} \times L_x^{2}.
		\end{split}
	\end{equation}
	By using \eqref{Duu0} and \eqref{Duu1}, then a subsequence of $\{(\bu_{j}, \rho_{j}, \bw_{j})\}_{j\in \mathbb{N}^+}$ and $\{(\partial_t \bu_{j}, \partial_t\rho_{j}, \partial_t\bw_{j})\}_{j\in \mathbb{N}^+}$ are weakly convergent. Therefore, if $m\rightarrow \infty$, then
	\begin{equation}\label{rwe0}
		\begin{split}
			(h_{j_m}, \bu_{j_m}, \bw_{j_m})\rightharpoonup & (h, \bu, \bw) \quad \qquad \ \ \text{in} \ \ H_x^{\sstar} \times H_x^{\sstar} \times H_x^{2},
			\\
			(\partial_t h_{j_m}, \partial_t\bu_{j_m}, \partial_t\bw_{j_m})\rightharpoonup & ( \partial_t h, \partial_t\bu, \partial_t\bw) \quad \text{in} \ \ H_x^{\sstar-1} \times H_x^{\sstar-1} \times H_x^{1}.
		\end{split}
	\end{equation}
	By \eqref{rwe}, \eqref{rwe0} and interpolation formula, we can obtain
	\begin{equation}\label{rwe1}
		(h_{j_m}, \bu_{j_m})\rightarrow( h, \bu) \quad \text{in} \ \ H_x^{s_1}(0 \leq s_1 <\sstar), \quad \bw_{j_m} \rightarrow \bw \quad \text{in} \ \ H_x^{s_2} (0\leq s_2<2).
	\end{equation}
	By \eqref{rwe1}, \eqref{ur}, and \eqref{ur0}, we also have
	\begin{equation}\label{rwe2}
		\begin{split}
			& (\partial_th_{j_m}, \partial_t\bu_{j_m})\rightarrow(\partial_t h, \partial_t \bu) \ \qquad \text{in} \ \ H_x^{s_3}(0 \leq s_3 <\sstar-1),
			\\
			& \partial_t\bw_{j_m} \rightarrow \partial_t\bw \qquad \qquad \qquad \qquad \text{in} \ \ H_x^{s_4} (0 \leq s_4<1).
		\end{split}
	\end{equation}
	The weak convergence \eqref{rwe0} and the strong convergence \eqref{rwe1}-\eqref{rwe2} makes us easy to check that $(h, \bu, \bw)$ is a strong solution of \eqref{WTe} with the initial data $(h, \bu,\bw)|_{t=0}=(h_0,\bu_0,\bw_0)$,
	and for $1\leq s_1<\sstar, 1 \leq s_2<2$,
	\begin{equation}\label{rwe4}
		\begin{split}
			& (h,u^0-1,\mathring{\bu})\in C([0,T^*];H_x^{s_1}),
			\qquad \qquad \quad \bw \in C([0,T^*];H_x^{s_2}),
			\\
			&(\partial_t h,\partial_t \bu)\in C([0,T^*];H_x^{s_1-1}),
			\qquad \qquad \ \ \partial_t\bw \in C([0,T^*];H_x^{s_2-1}),
			\\
			& (h,u^0-1,\mathring{\bu})\in L^\infty([0,T^*];H_x^{\sstar}), \qquad \qquad \ \ \bw \in L^\infty([0,T^*];H_x^2),
			\\
			& (\partial_t h,\partial_t \bu)\in L^\infty([0,T^*];H_x^{\sstar-1}), \ \quad \qquad \quad  \partial_t\bw \in L^\infty([0,T^*];H_x^1).
		\end{split}
	\end{equation}
	Furthermore, by using \eqref{Duu0} and \eqref{Duu1}, we obtain
	\begin{equation*}
		\begin{split}
			& {E}(t)= \|(h,u^0-1,\mathring{\bu})\|_{L^\infty_{[0,T^*]}H_x^{\sstar}}+ \|\bw\|_{L^\infty_{[0,T^*]}H_x^{2}} \leq {M}_2,
		\end{split}
	\end{equation*}
	and
	\begin{equation*}
		\begin{split}
			& \|d h\|_{L^2_{[0,T^*]}L_x^\infty}+\| d \bu\|_{L^2_{[0,T^*]}L_x^\infty}
			\leq {M}_2,
			\\
			& \|\partial_t h\|_{L^\infty_{[0,T^*]}H_x^{\sstar-1}}+ \|\partial_t \bu\|_{L^\infty_{[0,T^*]}H_x^{\sstar-1}}+ \|\partial_t \bw\|_{L^\infty_{[0,T^*]}H_x^{1}} \leq {M}_2.
		\end{split}
	\end{equation*}
	It also remains for us that the convergence \eqref{rwe1} and \eqref{rwe2} also hold in the highest derivatives.

	\textit{Step 2: Strong convergence of the limit in highest derivatives.} Denote
	\begin{equation}\label{Dss23}
		\begin{split}
			e^{(\alpha)}_{j} = & \| u^\alpha_{0(j+1)} - u^\alpha_{0j} \|_{H^s}, \quad \alpha=0,1,2,3,
			\\
			d^{(\alpha)}_{j} = & \| w^\alpha_{0(j+1)} - w^\alpha_{0j} \|_{H^2},\quad \alpha=0,1,2,3,
			\\
			e^{(4)}_{j} = & \| h_{0(j+1)} - h_{0j} \|_{H^s}.
		\end{split}
	\end{equation}
	Set
	\begin{equation}\label{Dss24}
		\begin{split}
			e_{j} = \textstyle{\sum}_{\alpha=0}^3 ( e^{(\alpha)}_{j}+d^{(\alpha)}_{j})+ e^{(4)}_{j}.
		\end{split}
	\end{equation}
	Therefore, by using \eqref{Dss0}, \eqref{Qss0}, \eqref{Dss23} and \eqref{Dss24}, it yields
	\begin{equation}\label{Dss25}
		\begin{split}
			\lim_{j \rightarrow \infty} e_{j} = 0.
		\end{split}
	\end{equation}
	Using \eqref{Dss0} and \eqref{Qss0} again, we can obtain:

	(i)  uniform bounds
	\begin{equation}\label{mp0}
		\begin{split}
			& \| ( h_{0j}, u^0_{0j}-1, \mathring{\bu}_{0j} )\|_{H^{\sstar}} \lesssim M_*+M^5_*, \quad \| \bw_{0j} \|_{H^{2}} \lesssim M_*+M^5_*,
		\end{split}
	\end{equation}
	
	(ii)  high frequency bounds
	\begin{equation}\label{mp1}
		\begin{split}
			& \|   ( h_{0j}, u^0_{0j}-1, \mathring{\bu}_{0j} ) \|_{H^{\sstar+1}}+\|  \bw_{0j} \|_{H^{3}} \lesssim 2^{j}(M_*+M^5_*),
		\end{split}	
	\end{equation}
	
	(iii)  difference bounds
	\begin{equation}\label{mp2}
		\begin{split}
			&\|  h_{0(j+1)}- h_{0j},\bu_{0(j+1)}- \bu_{0j}\|_{L_x^{2}} \leq 2^{-\sstar j}e_j,
			\\
			&\|  \bw_{0(j+1)}- \bw_{0j}\|_{L_x^{2}} \leq 2^{-2j}e_j,
		\end{split}
	\end{equation}
	
	(iv)  limit
	\begin{equation}\label{mp3}
		\lim_{l \rightarrow \infty}( h_{0j}, \bu_{0j}, \bw_{0j})=(h_0, \bu_0, \bw_0) \quad  \mathrm{in} \quad H^{\sstar} \times H^{\sstar}\times H^{2}.
	\end{equation}
	For the solutions $(h_j,\bu_j, \bw_j)$, we claim that

	$\bullet$  uniform bounds
	\begin{equation}\label{ebs0}
		\| h_{j} \|_{L^\infty_{[0,T^*]}H_x^{\sstar}}+\|(u^0_{j}-1, \mathring{\bu}_j) \|_{L^\infty_{[0,T^*]}H_x^{\sstar}}+\| \bw_{j}\|_{L^\infty_{[0,T^*]}H_x^{2}} \lesssim {M}_2,
	\end{equation}

	$\bullet$  higher-order norms
	\begin{equation}\label{ebs1}
		\|  h_{j} \|_{L^\infty_{[0,T^*]}H_x^{\sstar+1}}+\| (u^0_{j}-1, \mathring{\bu}_j) \|_{L^\infty_{[0,T^*]}H_x^{\sstar+1}}+\| \bw_{j}  \|_{L^\infty_{[0,T^*]}H_x^{3}} \lesssim 2^{j}{M}_2,
	\end{equation}
	
	$\bullet$  energy bounds of the difference
	\begin{equation}\label{ebs2}
		\|  (h_{j+1}- h_{j}, \bu_{j+1}- \bu_{j}) \|_{L^\infty_{[0,T^*]} L_x^{2}} \lesssim 2^{-\sstar j}e_j,
	\end{equation}
	and
	\begin{equation}\label{ebs00}
		\| \bw_{j+1}- \bw_{j} \|_{L^\infty_{[0,T^*]} L_x^{2}} \lesssim 2^{-(\sstar-1)j} e_j,
	\end{equation}
	and
	\begin{equation}\label{ebs3}
		\| \bw_{j+1}- \bw_{j} \|_{L^\infty_{[0,T^*]} \dot{H}_x^{2}} \lesssim e_j,
	\end{equation}

	$\bullet$  Strichartz estimates of the difference
	\begin{equation}\label{ebs4}
		2^j\| h_{j+1}- h_{j}, \bu_{j+1}- \bu_{j}\|_{L^2_{[0,T^*]} L_x^{\infty}} \lesssim e_j.
	\end{equation}
	By using \eqref{rwe}, \eqref{ebs2}, \eqref{ebs00} and \eqref{ebs3}, we get
	\begin{equation*}
		\|\bu_j-\bu\|_{H_x^{\sstar}}\lesssim \textstyle{\sum}_{l \geq j}e_{l}.
	\end{equation*}
	Hence, combining \eqref{Dss23} and \eqref{Dss24}, we have
	\begin{equation}\label{ebs5}
		\lim_{j\rightarrow \infty}\|\bu_j-\bu\|_{H_x^{\sstar}}\lesssim  \lim_{j\rightarrow \infty} \textstyle{\sum}_{l \geq j}e_{l}=0.
	\end{equation}
	Similarly, using \eqref{ebs2}, \eqref{ebs00} and \eqref{ebs3}, so we also conclude that
	\begin{equation}\label{ebs6}
		\begin{split}
			\lim_{j\rightarrow \infty}\|h_j-h\|_{H_x^{\sstar}}\lesssim  \lim_{j\rightarrow \infty} \textstyle{\sum}_{l \geq j}e_{l}=0,
			\\
			\lim_{j\rightarrow \infty}\|\bw_j-\bw\|_{H_x^{2}}\lesssim  \lim_{j\rightarrow \infty} \textstyle{\sum}_{l \geq j}e_{l}=0.
		\end{split}
	\end{equation}
	Combining \eqref{ebs5} and \eqref{ebs6}, we have proved the strong convergence $\lim_{j\rightarrow \infty}(h_j,\bu_j, \bw_j)=(h,\bu,\bw)$ in $H_x^{\sstar} \times H_x^{\sstar} \times H_x^{2}$. In a similar way, we can show that $\lim_{j\rightarrow \infty}(\partial_t h_j,\partial_t\bu_j, \partial_t\bw_j)=(\partial_t h,\partial_t\bu,\partial_t\bw)$ in $H_x^{\sstar-1} \times H_x^{\sstar-1} \times H_x^{1}$. Therefore, we conclude that $(h, \bu, \bw)$ is a strong solution of \eqref{WTe} with the initial data $(h, \bu,\bw)|_{t=0}=(h_0,\bu_0,\bw_0)$. Moreover, we have
	\begin{equation}\label{key1}
		\begin{split}
			& (h,u^0-1,\mathring{\bu})\in C([0,T^*];H^{\sstar}),
			\quad \quad \bw \in C([0,T^*];H^2),
			\\
			&(\partial_t h,\partial_t \bu)\in C([0,T^*];H_x^{\sstar-1}),
			\qquad \partial_t\bw \in C([0,T^*];H_x^{1}).
		\end{split}
	\end{equation}
	It now remains for us to prove \eqref{ebs0}-\eqref{ebs4}. By using \eqref{Duu0}, we get \eqref{ebs0}. By Corollary \ref{hes} , \eqref{Duu0}, \eqref{Duu2} and \eqref{mp1}, we can obtain \eqref{ebs1}.

	\textbf{The proof of \eqref{ebs2}.} By \eqref{ur}, we can see
\begin{equation*}
	\begin{split}
		A^0 (\bU_{j+1}) \partial_t (\bU_{j+1}-\bU_j)+ A^a (\bU_{j+1}) \partial_a (\bU_{j+1}-\bU_j)=&-\{ A^\alpha (\bU_{j+1})-A^\alpha (\bU_j)\} \partial_\alpha \bU_j.
	\end{split}
\end{equation*}
Multiplying with $\bU_{j+1}-\bU_j$, integrating it on $[0,T^*]\times \mathbb{R}^3$, and using \eqref{Duu2} and \eqref{mp2}, it follows
\begin{equation*}
	\begin{split}
		\|\bU_{j+1}-\bU_j\|^2_{L^\infty_{[0,T^*]}L^2_x}\lesssim & \|(\bU_{j+1}-\bU_j)(0,\cdot)\|^2_{L^2_x}\exp\left(   \int^{T^*}_0 \|d \bU_j, d\bU_{j+1}\|_{L^\infty_x}d\tau \right) \lesssim (2^{-sj}e_j)^2.
	\end{split}
\end{equation*}
Therefore, by Lemma \ref{jh0}, we can obtain
\begin{equation*}
	\begin{split}
		\|(h_{j+1}-h_j, \mathring{\bu}_{j+1}-\mathring{\bu}_j)\|_{L^\infty_{[0,T^*]}L^2_x}\lesssim  2^{-sj}e_j.
	\end{split}
\end{equation*}
For $u^0_j=\sqrt{1+ |\mathring{\bu}_{j}|^2 }$, we can also get $\|{u}^0_{j+1}-{u}^0_j)\|_{L^\infty_{[0,T^*]}L^2_x}\lesssim  2^{-sj}e_j$. To summarize the outcome, we have proved
\begin{equation}\label{esb12}
	\begin{split}
		\|(h_{j+1}-h_j, {\bu}_{j+1}-{\bu}_j)\|_{L^\infty_{[0,T^*]}L^2_x}\lesssim  2^{-sj}e_j.
	\end{split}
\end{equation}
It implies that we have proved \eqref{ebs2}.

\textbf{The proof of \eqref{ebs00}.} By using \eqref{ur0}, we get
\begin{equation}\label{Dss26}
	\begin{split}
		& \partial_t (\bw_{j+1}-\bw_j) + ( (u_{j+1}^0)^{-1} \mathring{\bu}_{j+1} \cdot \nabla) (\bw_{j+1}-\bw_j)
		\\
		=&\left\{  \left( (u_{j+1}^0)^{-1}\mathring{\bu}_{j+1}-(u_j^0)^{-1} \mathring{\bu}_j \right) \cdot \nabla \right\}  \bw_j
		\\
		&-\left(  (u_{j+1}^0)^{-1} {\bu}_{j+1}-(u_j^0)^{-1} {\bu}_j \right) w_{j+1}^\kappa \partial_\kappa h_{j+1}
		\\
		& -(u_j^0)^{-1} {\bu}_j ( w_{j+1}^\kappa- w_j^\kappa) \partial_\kappa h_{j+1}-(u_j^0)^{-1} {\bu}_j w_j^\kappa \partial_\kappa ( h_{j+1} - h_j)
		\\
		& + \left( (u_{j+1}^0)^{-1}-(u_j^0)^{-1} \right) w_{j+1}^\kappa \partial_\kappa  \bu_{j+1}+(u_j^0)^{-1} (w_{j+1}^\kappa-w_j^\kappa) \partial_\kappa  \bu_{j+1}
		\\
		&+(u_j^0)^{-1} w_j^\kappa \partial_\kappa ( \bu_{j+1}-\bu_j)
		-\left( (u_{j+1}^0)^{-1} - (u_j^0)^{-1} \right) \bw_{j+1} \partial_\kappa u_{j+1}^\kappa
		\\
		&-(u_j^0)^{-1} ( \bw_{j+1} - \bw_j) \partial_\kappa u_{j+1}^\kappa
		-	(u_j^0)^{-1} \bw_j \partial_\kappa ( u_{j+1}^\kappa- u_j^\kappa)
	\end{split}
\end{equation}
Multiplying \eqref{Dss26} with $\bw_{j+1}- \bw_{j}$ and integrating it on $[0,T^*]\times \mathbb{R}^3$, we obtain
\begin{equation}\label{esb15}
	\begin{split}
		& \| \bw_{j+1}- \bw_{j} \|^2_{L^\infty_{[0,T^*]} L^2_x}
		\\
		\leq  & \| \bw_{0(j+1)}- \bw_{0j}\|^2_{L^2_x}+ C\textstyle{\int}^{T^*}_0  \| \bw_{j+1}-\bw_j \|^2_{L^2_x}\| d h_{j+1},d\bu_{j+1}\|_{L^\infty_x} d\tau
		\\
		& + C\textstyle{\int}^{T^*}_0  \| \bu_{j+1}-\bu_j \|_{L^2_x}\| \nabla \bw_{j}\|_{L^\infty_x} \| \bw_{j+1}- \bw_{j} \|_{ L^2_x}d\tau
		\\
		& + C\textstyle{\int}^{T^*}_0 \| \bu_{j+1}-\bu_j\|_{L^2_x} \| \bw_{j+1}\|_{L^\infty_x}\|d h_{j+1},d\bu_{j+1} \|_{L^\infty} \| \bw_{j+1}- \bw_{j} \|_{ L^2_x}d\tau
		\\
		&+ C \textstyle{\int}^{T^*}_0 \| \bw_{j}\|_{L^\infty_x} \|d (h_{j+1}-h_j),d (\bu_{j+1}-\bu_j) \|_{L^2_x} \| \bw_{j+1}- \bw_{j} \|_{ L^2_x}d\tau.
	\end{split}
\end{equation}
By H\"older's inequality, we can bound \eqref{esb15} by
\begin{equation}\label{esb16}
	\begin{split}
		& \| \bw_{j+1}- \bw_{j} \|^2_{L^\infty_{[0,T^*]} L^2_x}
		\\
		\leq  & \| \bw_{0(j+1)}- \bw_{0j}\|^2_{L^2_x}+ C\textstyle{\int}^{T^*}_0  \| \bw_{j+1}-\bw_j \|^2_{L^2_x}\| d h_{j+1},d\bu_{j+1}\|_{L^\infty_x} d\tau
		\\
		& + C\textstyle{\int}^{T^*}_0  \| \bu_{j+1}-\bu_j \|^2_{L^2_x} ( \| \nabla \bw_{j}\|^2_{L^\infty_x} + \| \bw_{j+1}\|^2_{L^\infty_x}\|d h_{j+1},d\bu_{j+1} \|^2_{L^\infty}) d\tau
		\\
		&+ C \textstyle{\int}^{T^*}_0 \| \bw_{j}\|^2_{L^\infty_x}  \|\nabla (h_{j+1}-h_j),\nabla (\bu_{j+1}-\bu_j) \|^2_{L^2_x} d\tau.
	\end{split}
\end{equation}
Due to \eqref{Duu0}, \eqref{Duu1}, \eqref{Duu00}, \eqref{Duu2}, \eqref{mp2}, and \eqref{ebs2}, it yields
\begin{equation}\label{esb17}
	\begin{split}
		\| \bw_{j+1}- \bw_{j} \|^2_{L^\infty_{[0,T^*]} L^2_x}
		\leq  & C(2^{-sj}e_j)^2+ C M^2_2  (2^{-sj}e_j)^2 2^{-\frac{j}{2}}
		+ C M^2_2 (2^{-(s-1)j}e_j)^2
		\\
		&  + C\textstyle{\int}^{T^*}_0  \| \bw_{j+1}-\bw_j \|^2_{L^2_x}\| d h_{j+1},d\bu_{j+1}\|_{L^\infty_x} d\tau.
	\end{split}
\end{equation}
Using Gronwall's inequality for \eqref{esb17}, then \eqref{ebs00} holds.

It remains for us to prove \eqref{ebs3}-\eqref{ebs4}. By bootstrap arguments, if we assume
	\begin{equation}\label{ebs9}
	\begin{split}
	2^j\| (h_{j+1}- h_{j}, \bu_{j+1}- \bu_{j})\|_{L^2_{[0,T^*]} L_x^{\infty}} \lesssim &  2^{-\delta_1 j} e_j,
	\\
	 \| \bw_{j+1}- \bw_{j} \|_{L^\infty_{[0,T^*]} \dot{H}_x^{2}} \leq &  2C_0 e_j,
	\end{split}		
	\end{equation}
then we obtain
	\begin{equation}\label{ebs7}
		2^j\| (h_{j+1}- h_{j}, \bu_{j+1}- \bu_{j})\|_{L^2_{[0,T^*]} L_x^{\infty}} \lesssim 2^{-2\delta_1 j}e_j,
	\end{equation}
	and
		\begin{equation}\label{ebs10}
		\| \bw_{j+1}- \bw_{j} \|_{L^\infty_{[0,T^*]} \dot{H}_x^{2}} \leq C_0 e_j.
	\end{equation}
Hence, \eqref{ebs3} and \eqref{ebs7} holds. Above, $C_0=3C^2+C$, and $C$ is a universal constant which is defined in \eqref{ebst} below. So our goal now is to prove \eqref{ebs7} and \eqref{ebs10}.

	\textbf{The proof of \eqref{ebs10}.}  Set $\bW_{j}=\mathrm{vort}(\mathrm{e}^{h_j}\bw_j)$ and $\bG_{j}=\mathrm{vort}(\bW_j)$. From \eqref{Wd5}, \eqref{Wd6} and \eqref{QH}, we get
	\begin{equation*}
		\begin{split}
			\| \Delta (\bw_{j+1}-\bw_j) \|_{L_x^{2}} \leq & C	\| \mathrm{vort} \bw_{j+1}-\mathrm{vort} \bw_{j} \|_{\dot{H}_x^{1}}
			+ \| (\bw_{j+1}-\bw_j) \cdot (d\bu_j,dh_j)\|_{\dot{H}_x^{1}}
			\\
			&+\| \bw_j \cdot d(\bu_{j+1}-\bu_j)\|_{\dot{H}_x^{1}}+\| \bw_j \cdot d(h_{j+1}-h_j)\|_{\dot{H}_x^{1}}
			\\
			& +  \| d(\bu_{j+1}-\bu_j) \cdot \nabla \mathring{\bw}_j \|_{L_x^{2}}+   \| d\bu_j \cdot \nabla (\mathring{\bw}_{j+1}-\mathring{\bw}_{j+1}) \|_{L_x^{2}}
			\\
				\leq	 & C\| \bW_{j+1}-\bW_j \|_{\dot{H}_x^{1}}+ C \| \mathring{\bw}_{j+1} -\mathring{\bw}_{j}\|_{{H}_x^{\frac32}}\| (\nabla\bu_j,\nabla h_j) \|_{{H}_x^{s-1}}
			\\
			&+  C\| \mathring{\bw}_j \|_{{H}_x^{\frac32}}\| \nabla(\bu_{j+1}-\bu_j),\nabla(h_{j+1}-h_j)) \|_{{H}_x^{1}} .
		\end{split}
	\end{equation*}
		By using \eqref{ebs2}, \eqref{Duu0}, \eqref{ebs00}, and $\| f \|_{{H}_x^{\frac32}} \leq \| f \|^{\frac14}_{L_x^{2}} \| f \|^{\frac34}_{{H}_x^{2}}$, we can prove that
		\begin{equation}\label{ebs8}
		\begin{split}
			\| \Delta (\bw_{j+1}-\bw_j) \|_{L_x^{2}} 
			\leq	 & C\| \bW_{j+1}-\bW_j \|_{\dot{H}_x^{1}}+C M_2 C_0^{\frac34}2^{-\frac{(s-1)j}{4}}e_j+  C M_2 C_0^{\frac34} 2^{-(s-2)j}e_j.
		\end{split}
	\end{equation}
	Similarly, seeing from \eqref{Wf01}, \eqref{Wf03}, and \eqref{QH}, we have
	\begin{equation}\label{ebs15}
	\begin{split}
		& \| \bW_{j+1}-\bW_j \|_{\dot{H}_x^{1}}
		\\
		\leq	 & C\| \bG_{j+1}-\bG_j \|_{L_x^{2}}+ C \| \mathring{\bw}_{j+1} -\mathring{\bw}_{j}\|_{{H}_x^{\frac32}}\| (d\bu_j,d h_j) \|_{{H}_x^{s-1}}(1+\| (d\bu_j,d h_j) \|_{{H}_x^{s-1}} )
		\\
		&+  C \| \mathring{\bw}_{j} \|_{{H}_x^{\frac32}}\| d(\bu_{j+1}-\bu_j),d (h_{j+1}-h_j) \|_{{H}_x^{1}}(1+\| (d\bu_j,d h_j) \|_{{H}_x^{s-1}} )
		\\
	\leq	&  C\| \bG_{j+1}-\bG_j \|_{L_x^{2}} +C M_2 C_0^{\frac34}2^{-\frac{(s-1)j}{4}}e_j+  C M_2 C_0^{\frac34} 2^{-(s-2)j}e_j .
	\end{split}
\end{equation}
In view of \eqref{ebs8} and \eqref{ebs15}, for large $j$, we have
	\begin{equation}\label{ebsa}
	\begin{split}
		\| \Delta (\bw_{j+1}-\bw_j) \|_{L_x^{2}} \leq & C\| \bG_{j+1}-\bG_j \|_{L_x^{2}} +C M_2 C_0^{\frac34}2^{-\frac{(s-1)j}{4}}e_j+  C M_2 C_0^{\frac34} 2^{-(s-2)j}e_j
		\\
		\leq & C\| \bG_{j+1}-\bG_j \|_{L_x^{2}} + C e_j .
	\end{split}
\end{equation}
Now, we need to estimate $\| \bG_{j+1}-\bG_j \|_{L_x^{2}}$. By \eqref{We15}, we can directly have
\begin{equation*}
	\begin{split}
		(u_j^0)^{-1} u_j^\kappa \partial_\kappa \left(G_j^\alpha-F_j^\alpha \right)
		=&  (u_j^0)^{-1} E_j^\alpha
		+  \partial^\alpha \left( (u_j^0)^{-1} \Gamma_j \right)
		-\Gamma_j  \partial^\alpha \left( (u_j^0)^{-1}\right),
	\end{split}
\end{equation*}
where $F_j^\alpha$ and $E_j^\alpha$ is the $\alpha$-th component of $\bF_j$ and $\bE_j$, and $\bF_j$, $\bE_j$ is defined in the same way as above, but with $h,\bu,\bw,\bW$ replaced by $h_j,\bu_j,\bw_j,\bW_j$. As a result, the difference term $\bG_{j+1}-\bG_j$ satisfies
\begin{equation}\label{ebs17}
	\begin{split}
	& \partial_t \left\{  (G_{j+1}^\alpha-G_j^\alpha)-(F_{j+1}^\alpha-F_j^\alpha) \right\}+  (u_{j+1}^0)^{-1} u_{j+1}^a \partial_a \left\{  (G_{j+1}^\alpha-G_j^\alpha)-(F_{j+1}^\alpha-F_j^\alpha) \right\}
	\\
		=& \left\{  (u_{j+1}^0)^{-1} u_{j+1}^a-(u_{j}^0)^{-1} u_{j}^a \right\} \partial_a ( G_j^\alpha -F_j^\alpha ) +\left\{ (u_{j+1}^0)^{-1}-(u_{j}^0)^{-1}\right\} E_{j+1}^\alpha
		\\
		&+ (u_{j}^0)^{-1} ( E_{j+1}^\alpha-E_{j}^\alpha )
	 -(\Gamma_{j+1}-\Gamma_{j})  \partial^\alpha \left( (u_{j+1}^0)^{-1}\right)-\Gamma_j  \partial^\alpha \left( (u_{j+1}^0)^{-1}-(u_j^0)^{-1}\right)
	 \\
	 &	+  \partial^\alpha \left\{ \left( (u_{j+1}^0)^{-1} - (u_{j}^0)^{-1} \right) \Gamma_{j+1} \right\}
	 +  \partial^\alpha  \left\{  (u_{j}^0)^{-1}  \left( \Gamma_{j+1} - \Gamma_{j} \right) \right\}.
	\end{split}
\end{equation}
Multiplying with $(G_{j+1}^\alpha-G_j^\alpha)-(F_{j+1}^\alpha-F_j^\alpha)$ on \eqref{ebs17} and integrating it on $[0,T^*]\times \mathbb{R}^3$, so we get
\begin{equation}\label{ebs18}
	\begin{split}
		& \|  (\bG_{j+1}-\bG_j)-(\bF_{j+1}-\bF_j) \|^2_{L^2_x}
		-
		\|   (\bG_{0(j+1)}-\bG_{0j})-(\bF_{0(j+1)}-\bF_{0j}) \|^2_{L^2_x}
		\\
	=
		& \underbrace{-\textstyle{\int}^{T^*}_0\textstyle{\int}_{\mathbb{R}^3} (u_{j+1}^0)^{-1} u_{j+1}^a \partial_a \left\{  (G_{j+1}^\alpha-G_j^\alpha)-(F_{j+1}^\alpha-F_j^\alpha) \right\} (G_{j+1}^\alpha-G_j^\alpha)-(F_{j+1}^\alpha-F_j^\alpha)
			dx d\tau }_{\equiv J_1}
		\\
		& + \underbrace{ \textstyle{\int}^{T^*}_0 \textstyle{\int}_{\mathbb{R}^3} \left\{  (u_{j+1}^0)^{-1} u_{j+1}^a-(u_{j}^0)^{-1} u_{j}^a \right\} \partial_a ( G_j^\alpha -F_j^\alpha )
		(G_{j+1}^\alpha-G_j^\alpha)-(F_{j+1}^\alpha-F_j^\alpha) dx d\tau }_{\equiv J_2}
		\\
		&+  \underbrace{ \textstyle{\int}^{T^*}_0\textstyle{\int}_{\mathbb{R}^3} \left\{ (u_{j+1}^0)^{-1}-(u_{j}^0)^{-1}\right\} E_{j+1}^\alpha ( G_j^\alpha -F_j^\alpha )
		(G_{j+1}^\alpha-G_j^\alpha)-(F_{j+1}^\alpha-F_j^\alpha) dx d\tau  }_{\equiv J_3}
		\\
		&+ \underbrace{ \textstyle{\int}^{T^*}_0 \textstyle{\int}_{\mathbb{R}^3} (u_{j}^0)^{-1} ( E_{j+1}^\alpha-E_{j}^\alpha ) ( G_j^\alpha -F_j^\alpha )
		(G_{j+1}^\alpha-G_j^\alpha)-(F_{j+1}^\alpha-F_j^\alpha) dx d\tau   }_{\equiv J_4}
		\\
		&\underbrace{ - \textstyle{\int}^{T^*}_0 \textstyle{\int}_{\mathbb{R}^3} (\Gamma_{j+1}-\Gamma_{j})  \partial^\alpha \left( (u_{j+1}^0)^{-1}\right) (G_{j+1}^\alpha-G_j^\alpha)-(F_{j+1}^\alpha-F_j^\alpha) dx d\tau  }_{\equiv J_5}
		\\
		& \underbrace{ -  \textstyle{\int}^{T^*}_0 \textstyle{\int}_{\mathbb{R}^3} \Gamma_j  \partial^\alpha \left( (u_{j+1}^0)^{-1}-(u_j^0)^{-1}\right)( G_j^\alpha -F_j^\alpha )
		(G_{j+1}^\alpha-G_j^\alpha)-(F_{j+1}^\alpha-F_j^\alpha) dx d\tau }_{\equiv J_6}
		\\
		& +\underbrace{ \textstyle{\int}^{T^*}_0 \textstyle{\int}_{\mathbb{R}^3} \partial_\alpha \left\{ \left( (u_{j+1}^0)^{-1} - (u_{j}^0)^{-1} \right) \Gamma_{j+1} \right\} \left( (G_{j+1}^\alpha-G_j^\alpha)-(F_{j+1}^\alpha-F_j^\alpha) \right) dx d\tau }_{\equiv J_7}
		\\
		& +  \underbrace{ \textstyle{\int}^{T^*}_0 \textstyle{\int}_{\mathbb{R}^3}\partial_\alpha  \left\{  (u_{j}^0)^{-1}  \left( \Gamma_{j+1} - \Gamma_{j} \right) \right\} \left( (G_{j+1}^\alpha-G_j^\alpha)-(F_{j+1}^\alpha-F_j^\alpha) \right) dx d\tau }_{\equiv J_8}.
	\end{split}
\end{equation}

\textit{Bound of the right hand of \eqref{ebs18}}. We now turn to estimate $\bG_{j+1}-\bG_j$, $\bF_{j+1}-\bF_j$, $\Gamma_{j+1}-\Gamma_{j}$ and $\bE_{j+1}-\bE_{j}$. Recall that $\bG_j$ has the same formulations with $\bG$ only by replacing $(h,\bu,\bw,\bW)$ to $(h_j,\bu_j,\bw_j,\bW_j)$. Applying \eqref{MFd} and \eqref{YX0}, we can obtain
\begin{equation*}
	\begin{split}
		\|  \bG_{j+1}-\bG_j \|_{L^2_x} \leq  & C\| \bw_{j+1}-\bw_j\|_{H^{2}_x} (1+
		\|(d\bu_j,dh_j)\|^2_{H^{s_0-1}_x} )
		\\
		& + C\| \bw_j\|_{H^{2}_x} 
		\|d(\bu_{j+1}-\bu_j),d(h_{j+1}-h_j)\|^2_{H^{s_0-1}_x} ,
	\end{split}
\end{equation*}
and
\begin{equation*}
	\begin{split}
		\| \bF_{j+1}-\bF_j \|_{L^2_x} \lesssim  & \|(d\bw_{j+1}-d\bw_j) \cdot (du_j,dh_j)\|_{L^{2}_x} + \|d\bw_j \cdot (du_{j+1}-du_{j},dh_{j+1}-dh_{j})\|_{L^{2}_x} 
		\\
		& +  \|(\bw_{j+1}-\bw_{j}) \cdot dh_j \cdot dh_j \|_{L^{2}_x} +  \|\bw_{j} \cdot (dh_{j+1}-dh_j) \cdot dh_j \|_{L^{2}_x} .
	\end{split}
\end{equation*}
Due to H\"older's inequality, \eqref{Duu0}, \eqref{ebs2}, \eqref{ebs00}, and \eqref{ebs9}, we can show that
\begin{equation}\label{es0}
	\begin{split}
		\|  \bG_{j+1}-\bG_j \|_{L^2_x} \leq  & 2C(1+M_2^2) (1+C_0) e_j.
	\end{split}
\end{equation}
By H\"older's inequality, \eqref{Duu0}, \eqref{ebs2}, and \eqref{ebs00}, it yields
\begin{equation}\label{es1}
	\begin{split}
		& \| \bF_{j+1}-\bF_j \|_{L^2_x}
		\\
		 \leq  & C(M_2+M_2^2) \big( \|\bw_{j+1}-\bw_j\|_{H^{\frac32}_x} + \|(\bu_{j+1}-\bu_{j},h_{j+1}-h_{j})\|_{H^{\frac32}_x} \big)
		\\
		\leq & 
		C(M_2+M_2^2) \big( \|\bw_{j+1}-\bw_j\|^{\frac14}_{L^{2}_x}\|\bw_{j+1}-\bw_j\|^{\frac34}_{H^{2}_x} + \|(\bu_{j+1}-\bu_{j},h_{j+1}-h_{j})\|_{H^{\frac32}_x} \big)
		\\
		 \leq & 2C(M_2+M_2^2) ( C_0^{\frac34}2^{-\frac{(s-1)j}{4}}e_j+   2^{-(s-\frac32)j}e_j).
	\end{split}
\end{equation}
In a similar way, we can also bound $\Gamma_{j+1}-\Gamma_{j}$ by
\begin{equation}\label{es3}
	\begin{split}
		\| \Gamma_{j+1}-\Gamma_{j} \|_{L^2_x}  \leq & C\|(d\bu_{j+1}-d\bu_j) \cdot d\bw_j\|_{L^2_x}+\| d\bu_j \cdot (d\bw_{j+1}-d\bw_j)\|_{L^2_x}
		\\
		& + C\|(\bw_{j+1}-\bw_j) \cdot dh_j \cdot \bw_j\|_{L^2_x}+\|\bw_j \cdot (dh_{j+1}-dh_j) \cdot \bw_j\|_{L^2_x}
		\\
		\leq & 
		C(M_2+M_2^2) \big( \|\bw_{j+1}-\bw_j\|_{H^{\frac32}_x} + \|(\bu_{j+1}-\bu_{j},h_{j+1}-h_{j})\|_{H^{\frac32}_x} \big)
		\\
		\leq & 2C(M_2+M_2^2) ( C_0^{\frac34}2^{-\frac{(s-1)j}{4}}e_j+   2^{-(s-\frac32)j}e_j).
	\end{split}
\end{equation}
Using H\"older's inequality, we can estimate $\bE_{j+1}-\bE_{j}$ by
\begin{equation*}
	\begin{split}
		& \| \bE_{j+1}-\bE_{j} \|_{L^2_x}  
		\\
		\leq & C\|d(\bu_{j+1}-\bu_j),d(h_{j+1}-h_j)\|_{L^\infty_x}  ( \| \bw_j\|_{H^{2}_x}+\|(d\bu_j,dh_j)\|_{H^{s_0-1}_x}  \| \bw_j\|_{H^{2}_x})
		\\
		& + C\|(d\bu_j,dh_j)\|_{L^\infty_x}  ( \| \bw_{j+1}-\bw_{j}\|_{H^{2}_x}+\|(d\bu_j,dh_j)\|_{H^{s_0-1}_x}  \| \bw_{j+1}-\bw_{j} \|_{H^{2}_x})
		\\
		& + C\|(d\bu_j,dh_j)\|_{L^\infty_x}  \|d(\bu_{j+1}-\bu_j),d(h_{j+1}-h_j)\|_{H^{s_0-1}_x}  \| \bw_{j} \|_{H^{2}_x}
		\\
		& +C\| \bw_j\|_{H^{2}_x} ( \|d(\bu_{j+1}-\bu_j),d(h_{j+1}-h_j)\|^2_{H^{s_0-1}_x}  +\|d(\bu_{j+1}-\bu_j),d(h_{j+1}-h_j)\|^3_{H^{s_0-1}_x} )
		\\
		& +C\|d\bu_j,dh_j\|^2_{H^{s_0-1}_x}  \| \bw_{j+1}-\bw_{j} \|_{H^{2}_x}+\|d\bu_j,dh_j\|^3_{H^{s_0-1}_x}  \| \bw_{j+1}-\bw_{j} \|_{H^{2}_x}.
	\end{split}
\end{equation*}
Applying \eqref{Duu0}, \eqref{ebs2}, and \eqref{ebs00}, it follows
\begin{equation}\label{es2}
	\begin{split}
		 \| \bE_{j+1}-\bE_{j} \|_{L^2_x}  
		\leq & C(M_2+M_2^2) \|d(\bu_{j+1}-\bu_j),d(h_{j+1}-h_j)\|_{L^\infty_x} 
		\\
		& + C(1+M_2)\|(d\bu_j,dh_j)\|_{L^\infty_x}  \| \bw_{j+1}-\bw_{j}\|_{H^{2}_x}
		\\
		& + CM_2 \|(d\bu_j,dh_j)\|_{L^\infty_x}  \|\bu_{j+1}-\bu_j,h_{j+1}-h_j\|_{H^{s_0}_x}  
		\\
		& +CM_2 ( \|\bu_{j+1}-\bu_j,h_{j+1}-h_j\|^2_{H^{s_0}_x}  + \|\bu_{j+1}-\bu_j,h_{j+1}-h_j\|^3_{H^{s_0}_x} )
		\\
		& +C(M^2_2+M^3_2)  \| \bw_{j+1}-\bw_{j} \|_{H^{2}_x}.
	\end{split}
\end{equation}
Now, we are ready to bound the terms $J_1, J_2, \cdots$, and $J_6$. Integrating $J_1$ by parts, it's direct for us to obtain
\begin{equation*}
	\begin{split}
	|J_1| \leq	& C \int^{T^*}_0  \| \nabla \bu_{j+1}  \|_{L^\infty_x} ( \|  \bG_{j+1}-\bG_j \|^2_{L^2_x}+\| \bF_{j+1}-\bF_j \|^2_{L^2_x} ) d\tau .
	\end{split}
\end{equation*}
Let $T^*\leq 1$. When $j$ is large enough, applying \eqref{es1}, \eqref{es0} and \eqref{Duu2}, we get
\begin{equation}\label{ebsJ1}
	\begin{split}
		|J_1| \leq	 & CM_2 (T^*)^{\frac12}\{ 2C(1+M_2^2) (1+C_0) e_j \}^2
		\\
		& + \{ 2C(M_2+M_2^2) ( C_0^{\frac34}2^{-\frac{(s-1)j}{4}}e_j+   2^{-(s-\frac32)j}e_j) \}^2 CM_2 (T^*)^{\frac12}
		\\
		\leq & ( C e_j )^2.
	\end{split}
\end{equation}
For $j$ is large enough, by using H\"older's inequality, \eqref{es0}, and \eqref{ebs9} we can get
\begin{equation}\label{ebsJ2}
\begin{split}
		|J_2| \leq	
		&  C  \int^{T^*}_0 \|(\bu_{j+1}-\bu_{j}) \|_{L^\infty_x} \|\nabla ( \bG_j-\bF_j ) \|_{L^2_x} \|  (\bG_{j+1}-\bG_j)-(\bF_{j+1}-\bF_j) \|_{L^2_x} d\tau
\\
\leq &  C  \int^{T^*}_0 2^j \|(\bu_{j+1}-\bu_{j}) \|_{L^\infty_x} \| \bG_j-\bF_j  \|_{L^2_x} ( \|  \bG_{j+1}-\bG_j\|_{L^2_x} + \|\bF_{j+1}-\bF_j \|_{L^2_x} ) d\tau
\\
\leq &  C  (M_2+M^3_2)\int^{T^*}_0 2^j \|(\bu_{j+1}-\bu_{j}) \|_{L^\infty_x}  \|  \bG_{j+1}-\bG_j\|_{L^2_x} d\tau 
\\
& + C(M_2+M_2^2)2^{-\frac{(s-1)j}{4}}e_j \int^{T^*}_0 2^j \|(\bu_{j+1}-\bu_{j}) \|_{L^\infty_x}  d\tau
	\\
\leq & ( C e_j )^2.
\end{split}
\end{equation}
Similarly, we can also obtain
\begin{equation}\label{ebsJ3}
\begin{split}
	|J_3| \leq	
	&C  \int^{T^*}_0 \| \bu_{j+1}-\bu_{j}\|_{L^\infty_x} \|\bE_{j+1}\|_{L^2_x} \|  (\bG_{j+1}-\bG_j)-(\bF_{j+1}-\bF_j) \|_{L^2_x} d\tau
	\\
	\leq & ( C e_j )^2.
\end{split}
\end{equation}
\begin{equation}\label{ebsJ4}
	\begin{split}
	|J_4| \leq	
	& C  \int^{T^*}_0 \| \bE_{j+1}-\bE_{j} \|_{L^2_x} \|  (\bG_{j+1}-\bG_j)-(\bF_{j+1}-\bF_j) \|_{L^2_x} d\tau 	\\
	\leq & ( C e_j )^2.
\end{split}
\end{equation}
\begin{equation}\label{ebsJ5}
	\begin{split}
	|J_5| \leq	
	& C  \int^{T^*}_0 \|\Gamma_{j+1}-\Gamma_{j}\|_{L^2_x}  \|d  \bu_{j+1}\|_{L^\infty_x}\|  (\bG_{j+1}-\bG_j)-(\bF_{j+1}-\bF_j) \|_{L^2_x} d\tau 	\\
	\leq & ( C e_j )^2.
\end{split}
\end{equation}

\begin{equation}\label{ebsJ6}
	\begin{split}
	|J_6| \leq	
	&C  \int^{T^*}_0 \|\Gamma_j \|_{L^2_x} \|d \left( \bu_{j+1}-\bu_j \right)\|_{L^\infty_x} \|  (\bG_{j+1}-\bG_j)-(\bF_{j+1}-\bF_j) \|_{L^2_x} d\tau
		\\
	\leq & ( C e_j )^2.
\end{split}
\end{equation}
Let us consider $J_7$ and $J_8$, for it's different and difficult. Integrating $J_7$ by parts, we get
\begin{equation*}
	\begin{split}
		J_7= &  \int_{\mathbb{R}^3} \left( (u_{j+1}^0)^{-1} - (u_{j}^0)^{-1} \right) \Gamma_{j+1}  \left( (G_{j+1}^0-G_j^0)-(F_{j+1}^0-F_j^0) \right)dx|^{T^*}_0
			\\
		& -   \int^{T^*}_0  \int_{\mathbb{R}^3} \left\{ \left( (u_{j+1}^0)^{-1} - (u_{j}^0)^{-1} \right) \Gamma_{j+1} \right\} \partial_\alpha \left( (G_{j+1}^\alpha-G_j^\alpha)-(F_{j+1}^\alpha-F_j^\alpha) \right) dx d\tau .
	\end{split}
\end{equation*}
Applying H\"older's inequality, and Bernstein's inequality, we can obtain
\begin{equation*}
	\begin{split}
		J_7\leq &  C\| \bu_{j+1}- \bu_{j}\|_{L^\infty_{[0,T^*]}L^\infty_x} \|\Gamma_{j+1}\|_{L^\infty_{[0,T^*]}L^2_x}  (\| \bG_{j+1}-\bG_j\|_{L^\infty_{[0,T^*]}L^2_x} + \|\bF_{j+1}-\bF_j \|_{L^\infty_{[0,T^*]}L^2_x} )
		\\
		& +   \int^{T^*}_0  \| \bu_{j+1}- \bu_{j} \|_{L^\infty_x}  \|\Gamma_{j+1} \|_{L^2_x} \| \partial_\alpha  (G_{j+1}^\alpha-G_j^\alpha)\|_{L^2_x}+ \|\partial_\alpha (F_{j+1}^\alpha-F_j^\alpha) \|_{L^2_x} ) d\tau 
		\\
		\leq &  C\| \bu_{j+1}- \bu_{j}\|_{L^\infty_{[0,T^*]}L^\infty_x} \|\Gamma_{j+1}\|_{L^\infty_{[0,T^*]}L^2_x}  (\| \bG_{j+1}-\bG_j\|_{L^\infty_{[0,T^*]}L^2_x} + \|\bF_{j+1}-\bF_j \|_{L^\infty_{[0,T^*]}L^2_x} )
		\\
		& +   \int^{T^*}_0  2^j \| \bu_{j+1}- \bu_{j} \|_{L^\infty_x}  \|\Gamma_{j+1} \|_{L^2_x} \| \bG_{j+1}-\bG_j \|_{L^2_x}+ \|\bF_{j+1}-\bF_j \|_{L^2_x} ) d\tau .
	\end{split}
\end{equation*}
Due to \eqref{ebs2}, \eqref{ebs9}, \eqref{ebs00}, \eqref{Duu0}, \eqref{es0}, and \eqref{es1}, for large enough $j$, it follows
\begin{equation}\label{ebsJ7}
	\begin{split}
		|J_7|
		\leq &  C\| \bu_{j+1}- \bu_{j}\|_{L^\infty_{[0,T^*]}L^\infty_x} \|\Gamma_{j+1}\|_{L^\infty_{[0,T^*]}L^2_x}  (\| \bG_{j+1}-\bG_j\|_{L^\infty_{[0,T^*]}L^2_x} + \|\bF_{j+1}-\bF_j \|_{L^\infty_{[0,T^*]}L^2_x} )
		\\
		& +   \int^{T^*}_0  2^j \| \bu_{j+1}- \bu_{j} \|_{L^\infty_x}  \|\Gamma_{j+1} \|_{L^2_x} \| \bG_{j+1}-\bG_j \|_{L^2_x}+ \|\bF_{j+1}-\bF_j \|_{L^2_x} ) d\tau 
		\\
		\leq & C(M_2+M_2^2)2^{-\frac{j}{2}} e_j \left\{ C(M_2+M_2^2) ( C_0^{\frac34}2^{-\frac{(s-1)j}{4}}e_j+   2^{-(s-\frac32)j}e_j)+2C(1+M_2^2) (1+C_0) e_j \right\}
		\\
		& + C(M_2+M_2^2)2^{-\delta_1 j} e_j \left\{ C(M_2+M_2^2) ( C_0^{\frac34}2^{-\frac{(s-1)j}{4}}e_j+   2^{-(s-\frac32)j}e_j)+2C(1+M_2^2) (1+C_0) e_j \right\}
		\\
		\leq & (Ce_j)^2 .
	\end{split}
\end{equation}
For $J_8$, we also have
\begin{equation*}
	\begin{split}
	J_8=&  \int_{\mathbb{R}^3} (u_{j}^0)^{-1}  \left( \Gamma_{j+1} - \Gamma_{j} \right) \left( (G_{j+1}^0-G_j^0)-(F_{j+1}^0-F_j^0) \right) dx|^{T^*}_0
		\\
		& -   \int^{T^*}_0 \int_{\mathbb{R}^3}    (u_{j}^0)^{-1}  \left( \Gamma_{j+1} - \Gamma_{j} \right) \partial_\alpha \left( (G_{j+1}^\alpha-G_j^\alpha)-(F_{j+1}^\alpha-F_j^\alpha) \right) dx d\tau
			\\
		\leq &  C\| \Gamma_{j+1}- \Gamma_{j}\|_{L^\infty_{[0,T^*]}L^2_x} (\| \bG_{j+1}-\bG_j\|_{L^\infty_{[0,T^*]}L^2_x} + \|\bF_{j+1}-\bF_j \|_{L^\infty_{[0,T^*]}L^2_x} )
		\\
		& +   \int^{T^*}_0   \| \Gamma_{j+1}- \Gamma_{j} \|_{L^2_x}  ( \| \bG_{j+1}-\bG_j \|_{L^2_x}+ \|\bF_{j+1}-\bF_j \|_{L^2_x} ) d\tau  .
	\end{split}
\end{equation*}
Similarly, by using \eqref{ebs2}, \eqref{ebs9}, \eqref{ebs00}, and \eqref{Duu0}, we also obtain
\begin{equation}\label{ebsJ8}
	\begin{split}
		J_8 \leq (Ce_j)^2 .
	\end{split}
\end{equation}

\textit{Bound of the left hand of \eqref{ebs18}}. By \eqref{mp2} and \eqref{es1}, we deduce
\begin{equation}\label{ebsJ9}
	\begin{split}
		& \|  (\bG_{j+1}-\bG_j)-(\bF_{j+1}-\bF_j) \|^2_{L^2_x}
		-
		\|   (\bG_{0(j+1)}-\bG_{0j})-(\bF_{0(j+1)}-\bF_{0j}) \|^2_{L^2_x} 
		\\
		\geq & \|  \bG_{j+1}-\bG_j \|^2_{L^2_x}-\left(2C(M_2+M_2^2) ( C_0^{\frac34}2^{-\frac{(s-1)j}{4}}e_j+   2^{-(s-\frac32)j}e_j)\right)^2-  (Ce_j)^2
		\\
		\geq & \|  \bG_{j+1}-\bG_j \|^2_{L^2_x} -  (Ce_j)^2,
	\end{split}
\end{equation}
for $j$ is large enough. Combining \eqref{ebs18}, \eqref{ebsJ1}-\eqref{ebsJ9}, we get
\begin{equation}\label{ebsk}
	\begin{split}
		& \|  (\bG_{j+1}-\bG_j)-(\bF_{j+1}-\bF_j) \|_{L^2_x}
		 \leq  3Ce_j.
	\end{split}
\end{equation}
Since \eqref{ebsa} and \eqref{ebsk}, we have proved 
\begin{equation}\label{ebst}
	\begin{split}
		& \|  \bw_{j+1}-\bw_j\|_{\dot{H}^2_x} \leq (3C^2+C)e_j.
	\end{split}
\end{equation}
Therefore, we have proved \eqref{ebs10}.

	\textbf{The proof of \eqref{ebs7}.} By using \eqref{Duu21}-\eqref{Duu22} and taking $r=\frac{s}{2}=1+5\delta_{1}$, we can bound \eqref{ebs7} by
	\begin{equation*}
		\begin{split}
			&2^j\| h_{j+1}- h_{j}, \bu_{j+1}- \bu_{j}\|_{L^2_{[0,T^*]} L_x^{\infty}}
			\\
			\leq & C2^j ( \| h_{j+1}- h_{j}, \bu_{j+1}- \bu_{j}\|_{L^\infty_{[0,T^*]} H_x^{1+5\delta_{1}}} + \| \bw_{j+1}- \bw_{j} \|_{L^\infty_{[0,T^*]} H_x^{\frac12+5\delta_1}})
			\\
			\leq & C2^{-3\delta_1 j} e_j.
		\end{split}
	\end{equation*}
	Then \eqref{ebs7} holds. This implies that \eqref{ebs4} holds.

	\textit{Step 3: Obtaining a solution of linear wave by taking a limit.}
	Reapting the proof in Step 1 and Step 2, we shall also find
	\begin{equation}\label{key2}
		\begin{split}
			\lim_{j\rightarrow \infty} \|f_j-f\|_{H^r_x}=0, \quad \lim_{j\rightarrow \infty} \|\partial_t f_j- \partial_t f\|_{H^{r-1}_x}=0.
		\end{split}
	\end{equation}
	So taking $j\rightarrow \infty$, we have
	\begin{equation*}
		\begin{cases}
			& \square_g f=0, \qquad (t,x) \in (t_0,T^*]\times \mathbb{R}^3,
			\\
			&(f, \partial_t f)|_{t=0}=(f_0,f_1) \in H_x^r\times H_x^{r-1}.
		\end{cases}
	\end{equation*}
	Using \eqref{Duu22}, it follows
	\begin{equation}\label{key3}
		\begin{split}
			&\|\left< \nabla \right>^{a-1} d{f}\|_{L^2_{[0,T^*]} L^\infty_x}
			\leq  {M}_4 (\|{f}_0\|_{{H}_x^r}+ \|{f}_1 \|_{{H}_x^{r-1}}),
			\\
			&\|{f}\|_{L^\infty_{[0,T^*]} H^{r}_x}+ \|\partial_t {f}\|_{L^\infty_{[0,T^*]} H^{r-1}_x} \leq  {M}_4 (\| {f}_0\|_{H_x^r}+ \| {f}_1\|_{H_x^{r-1}}).
		\end{split}
	\end{equation}

	\textit{Step 4: Continuous dependence.}
	Based on \eqref{key3}, we can prove the continuous dependence of solutions on the initial data, which can follow Section \ref{cd}. Therefore, we finish the proof of Theorem \ref{dingli2}.
\end{proof}
\subsection{Proof of Proposition \ref{DDL3}}\label{pps}
\begin{proposition}\label{DDL3}
	Let $\epsilon_2$ and $\epsilon_3$ be stated in \eqref{a0}. Let\footnote{We take $\delta_1$ be a fixed number.} $\delta_1=\frac{\sstar-2}{40}>0$. Suppose the initial data $(h_0, \bu_0, \bw_0)$ be smooth, supported in $B(0,c+2)$ and satisfying\footnote{To obtain the smallness norm of $\| \bw_0\|_{H^{2+\delta_1}}$, the price we pay is that the time interval is smaller, please see \eqref{DTJ}.}
	\begin{equation}\label{DP30}
		\begin{split}
			&\| (h_0,u_0^0-1,\mathring{\bu}_0) \|_{H^{\sstar}} + \| \bw_0\|_{H^{2+\delta_1}}  \leq \epsilon_3,
		\end{split}
	\end{equation}
	there exists a smooth solution $(h, \bu, \bw)$ to \eqref{WTe} on $[-2,2] \times \mathbb{R}^3$ satisfying
	\begin{equation}\label{DP31}
		\|(h,u^0-1,\mathring{\bu})\|_{L^\infty_{[-2,2]}H_x^{\sstar}}+\|\bw\|_{L^\infty_{[-2,2]}H_x^{2+\delta_1}} \leq \epsilon_2.
	\end{equation}
	Furthermore, the solution has the following properties

	$\mathrm{(1)}$ dispersive estimate for $h$, $\bu$, and $\bu_+$
	\begin{equation}\label{DP32}
		\|d h,d \bu_+, d \bu\|_{L^2_{[-2,2]} C^{\delta}_x}+\|d h,d \bu_+, d \bu\|_{L^2_{[-2,2]} L^\infty_x} \leq \epsilon_2,
	\end{equation}

	$\mathrm{(2)}$ Let $f$ satisfy Equation \eqref{linearB}\footnote{Here the acoustic metric is better than in Proposition \ref{DDL2}, for we improve the regularity of vorticity to $2+\delta_1$, and $\delta_1$ is a fixed number.}. For each $1 \leq r \leq s+1$, the Cauchy problem \eqref{linear} is well-posed in $H_x^r \times H_x^{r-1}$, and the following estimate holds\footnote{This Strichartz estimates is stronger than \eqref{Duu2}, for the regularity of vorticity is different between Proposition \ref{DDL2} and Proposition \ref{DDL3}.}:
	\begin{equation}\label{DP33}
		\|\left< \nabla \right>^k f\|_{L^2_{[-2,2]} L^\infty_x} \lesssim  \| f_0\|_{H^r}+ \| f_1\|_{H^{r-1}},\quad \ k<r-1,
	\end{equation}
	and the same estimates hold with $\left< \nabla \right>^k$ replaced by $\left< \nabla \right>^{k-1}d$.
\end{proposition}
\begin{proof}
		By replacing the regularity exponent $s$ and $s_0$ to $s=\sstar, s_0=2+\delta_1$ in Proposition \ref{p1}, we can directly get Proposition \ref{DDL3}.
	\end{proof}

\subsection{Proof of Proposition \ref{DDL2} by Proposition \ref{DDL3}}\label{keypro}
In this part, we prove Proposition \ref{DDL2} through Proposition \ref{DDL3}. We observe that the corresponding statement In Proposition \ref{DDL3} considers the small, supported data. While, the initial data is large in Proposition \ref{DDL2}, and regularity of the vorticity in Proposition \ref{DDL2} is weaker than Proposition \ref{DDL3}. So we divide the proof into three subparts. Firstly, we can only get a solution of Proposition \ref{DDL2} on a short time interval, which is presented in subsection \ref{esess}. Then we are able to get some good Strichartz estimates for low-frequency and mid-frequency of solutions on these short time intervals by paying a loss of derivatives, which is proved in subsection \ref{esest}. After that, we can extend these solutions on a regular time-interval in subsection \ref{finalk}. Moreover, the solutions of linear wave can also be extended to the regular time-interval, which is obtained in \ref{finalq}.
\subsubsection{Energy estimates and Strichartz estimates on a short time-interval.} \label{esess}
Firstly, by using a scaling method, we can reduce the initial data in Proposition \ref{DDL2} to be small. Considering \eqref{Dss0}, for $T>0$, take space-time scaling
\begin{equation*}
	\begin{split}
		& \underline{\mathring{\bu}}_{0j}=\mathring{\bu}_{0j}(Tt,Tx), \quad \underline{h}_{0j}=h_{0j}(Tt,Tx).
	\end{split}
\end{equation*}
Referring to \eqref{muu} and \eqref{VVd}, we set
\begin{equation*}
	\begin{split}
		& \underline{u}^0_{0j}=\sqrt{ 1+|\underline{\mathring{\bu}}_{0j}|^2 }, \qquad \underline{{\bu}}_{0j}=(\underline{u}^0_{0j},\underline{\mathring{\bu}}_{0j})^{\mathrm{T}},
	\end{split}
\end{equation*}
and
\begin{equation*}
	\underline{\bw}_{0j}=\mathrm{vort}( \mathrm{e}^{\underline{h}_0} \underline{\bu}_{0j} ).
\end{equation*}
then we have the bounds
\begin{equation}\label{Dss2}
	\begin{split}
		\| (\underline{u}^0_{0j}-1, \underline{\mathring{\bu}}_{0j}) \|_{\dot{H}^{\sstar}} +  \| \underline{h}_{0j} \|_{\dot{H}^{\sstar}} \leq  & T^{\sstar-\frac32} (\| ({u}^0_{0j}-1, {\mathring{\bu}}_{0j})\|_{\dot{H}^{\sstar}}+\| {h}_{0j} \|_{\dot{H}^{\sstar}}),
		\\
		\|\underline{\bw}_{0j}\|_{\dot{H}^{2+\delta_1}}
		\leq & T^{\frac32+\delta_{1}} \|\mathrm{vort}( \mathrm{e}^{{h}_0} {\bu}_{0j} ) \|_{\dot{H}^{2+\delta_1}}
		\\
		= & T^{\frac32+\delta_{1}} \|{\bw}_{0j} \|_{\dot{H}^{2+\delta_1}}(1+ \|h_{0j} \|^4_{\dot{H}^{2+\delta_1}}+ \|{\bu}_{0j} \|^4_{\dot{H}^{2+\delta_1}})
		\\
		\leq & T^{\frac32+\delta_{1}} 2^{\delta_{1} j }  \|{\bw}_{0j} \|_{{H}^{2}}(1+ \|h_{0j} \|^4_{{H}^{s}}+
		\|( u^0_0-1,\mathring{\bu}_{0j} )\|^4_{{H}^{s}}).
	\end{split}
\end{equation}
Set
\begin{equation}\label{DTJ}
	T^*_j =2^{-\delta_1j}[E(0)]^{-1}.
\end{equation}
Above, $E(0)$ is stated in \eqref{pu0}. Taking $T$ in \eqref{Dss2} as $T^*_j$,  then \eqref{Dss2} becomes
\begin{equation}\label{pp7}
	\begin{split}
			\| (\underline{u}^0_{0j}-1, \underline{\mathring{\bu}}_{0j}) \|_{\dot{H}^{\sstar}} +  \| \underline{h}_{0j} \|_{\dot{H}^{\sstar}} \leq & 2^{-\delta_1 j({\frac12+10\delta_{1}})}[E(0)]^{-(\frac12+10\delta_{1})},
		\\
		 \| \underline{\bw}_{0j}\|_{\dot{H}^{2+\delta_1}}  \leq  & 2^{-(\frac{\delta_1}{2}+\delta_{1}^2)j} [E(0)]^{-(\frac12+\delta_{1})}.
	\end{split}
\end{equation}
Set
\begin{equation*}
	\underline{c}_{sj}=c_s(\underline{h}_j).
\end{equation*}
For $\| \underline{h}_{0j}\|_{L^\infty}  \leq  \|{h}_{0j}\|_{L^\infty} $, using \eqref{pu00}, we have
\begin{equation}\label{pp60}
	\begin{split}
		| \underline{u}^0_{0j}-1,\underline{\mathring{\bu}}_{0j}, \underline{h}_{0j} | \leq C_0, \quad \underline{u}^0_{0j}\geq 1, \quad \underline{c}_{sj}|_{t=0}\geq c_0>0.
	\end{split}
\end{equation}
For a small parameter $\epsilon_3$ stated in Proposition \ref{DDL3}, we can choose\footnote{Recall that $\delta_{1}=\frac{s-2}{40}$, $E(0)=C(M_*+M_*^5)$. So $N_0$ depends on $\delta_1, s, C_0, c_0$ and $M_*$.} $N_0=N_0(\delta_{1},E(0))$ such that
\begin{equation}\label{pp8}
	\begin{split}
		& 2^{-\delta_1 N_0}(1+E^6(0))\ll \epsilon_3,
		\\
		& 2^{-\delta_{1} N_0 } (1+C^3_*)\{ 1+ \frac{C_*}{E(0)}(1+ E^3(0))^{-1} \} \leq 1,
		\\
		& C2^{-\frac{\delta_1}{2}N_0} (1+E^3(0))E^{-\frac12}(0)  [\frac{1}{3}(1-2^{-\delta_{1}})]^{-2} \leq 2.
	\end{split}	
\end{equation}
Above, $C_*=C_*(\delta_{1}, E(0))$ is denoted by
\begin{equation}\label{Cstar}
	\begin{split}
		C_*=C(E(0)+E^5(0))\exp( 5 \textrm{e}^5 ) .
	\end{split}	
\end{equation}
As a result, for $j \geq N_0$, we have
\begin{equation}\label{Dss3}
	\begin{split}
		\| (\underline{u}^0_{0j}-1, \underline{\mathring{\bu}}_{0j}) \|_{\dot{H}^{\sstar}}+\| \underline{h}_{0j}\|_{\dot{H}^{\sstar}}+\| \underline{\bw}_{0j}\|_{\dot{H}^{2+\delta_1}}  \leq  \epsilon_3.
	\end{split}
\end{equation}
For \eqref{Dss3} is about the homogeneous norm, so we need to use the physical localization to get the same bound for the inhomogeneous norm. Utilize the standard physical localization as in Section 4. Since the speed of propagation of \eqref{WTe} is finite, we may set $c$ be the largest speed of propagation of \eqref{WTe}. Then the solution in a unit cylinder $[-1,1]\times B(y,1)$ is uniquely determined by the initial data in the ball $B(y,1+c)$. Hence
it is natural to truncate the initial data in a slightly larger region. Set $\chi$ be a smooth function supported in $B(0,c+2)$, and which equals $1$ in $B(0,c+1)$. For any given $y \in \mathbb{R}^3$, we define the localized initial data for the velocity and density near $\by$:
\begin{equation}\label{yyy0}
	\begin{split}
		\mathring{\widetilde{\bu}}_{0j}(x)=&\chi(x-y)\left( \underline{\mathring{\bu}}_{0j}(x)- \underline{\mathring{\bu}}_{0j}(y)\right),
		\\
		\widetilde{h}_{0j}(x)=&\chi(x-y)\left( \underline{h}_{0j}(x)-\underline{h}_{0j}(y)\right).
	\end{split}
\end{equation}
Referring to \eqref{muu} and \eqref{VVd}, we set
\begin{equation}\label{yyy1}
	\begin{split}
		\widetilde{u}^0_{0j}=&\sqrt{1+|\widetilde{\mathring{\bu}}_{0j}|^2}, \qquad \widetilde{\bu}_{0j}=(\widetilde{u}^0_{0j}, \mathring{\widetilde{\bu}}_{0j})^\mathrm{T},
		\\
		\widetilde{\bw}_{0j}= & -\epsilon^{\alpha}_{\ \beta \gamma \delta}\mathrm{e}^{\widetilde{ h }_{0j}+\widetilde{h}_{0j}(y) }(\widetilde{u}^\beta_{0j}+ \underline{u}^\beta_{0j}(y)) \partial^{\gamma}\widetilde{u}^\delta_{0j}.
	\end{split}
\end{equation}
By \eqref{Dss3}, \eqref{yyy0} and \eqref{yyy1}, so we have
\begin{equation}\label{sS4}
	\begin{split}
		& \|(\widetilde{u}^0_{0j}-1, \mathring{\widetilde{\bu}}_{0j} ) \|_{{H}^{\sstar}}+\| \widetilde{\rho}_{0j}\|_{{H}^{\sstar}}+\| \widetilde{\bw}_{0j}
		\|_{{H}^{2+\delta_1}}
		\\ \leq  & C(\| (\underline{u}^0_{0j}-1, \underline{\mathring{\bu}}_{0j} ) \|_{\dot{H}^{\sstar}}+\| \underline{\rho}_{0j}\|_{\dot{H}^{\sstar}}+\| \underline{\bw}_{0j} \|_{\dot{H}^{2+\delta_1}})
		\\
		\leq & \epsilon_3.
	\end{split}
\end{equation}
For every $j$, by Proposition \ref{DDL3}, there is a smooth solution $(\widetilde{\bu}_j, \widetilde{h}_j, \widetilde{\bw}_j)$ on $[-2,2]\times \mathbb{R}^3$ satisfying
\begin{equation}\label{yz0}
	\begin{cases}
\square_{\widetilde{g}_j} \widetilde{h}_j=\widetilde{D}_j,
\\
\square_{\widetilde{g}_j} \widetilde{u}_j^\alpha=-\widetilde{c}^2_{sj} \widetilde{\Omega}_j \mathrm{e}^{- (\widetilde{h}_j + \underline{h}_{0j}(y))} \widetilde{W}_j^\alpha+\widetilde{Q}_j^\alpha,
\\
(\widetilde{u}_j^\kappa+\underline{u}_{0j}^\kappa(y)) \partial_\kappa \widetilde{w}_j^\alpha= -(\widetilde{u}_j^\alpha+\underline{u}_{0j}^\alpha(y)) \widetilde{w}_j^\kappa \partial_\kappa \widetilde{h}_j+ \widetilde{w}_j^\kappa \partial_\kappa \widetilde{u}_j^\alpha- \widetilde{w}_j^\alpha \partial_\kappa \widetilde{u}_j^\kappa,
\\
\partial_\alpha \widetilde{w}_j^\alpha =-\widetilde{w}_j^\kappa \partial_\kappa \widetilde{h}_j,
\\
(\widetilde{h}_j,\widetilde{\bu}_j,\widetilde{\bw}_j)|_{t=0}=(\widetilde{h}_{0j},\widetilde{\bu}_{0j},\widetilde{\bw}_{0j}),
	\end{cases}
\end{equation}
Above, we define
\begin{equation}\label{QRY}
	\begin{split}
		\widetilde{c}^2_{sj} &= \frac{dp}{d{h}}(\widetilde{h}_j+\underline{h}_{0j}(y)),
		\\
		\widetilde{g}_j &=g({\widetilde\bu}_j+\underline{\bu}_{0j}(y), {\widetilde{h}}_j+\underline{h}_{0j}(y)),
		\\
\widetilde{\Omega}_j&=\Omega(\widetilde{h}_j+\underline{h}_{0j}(y), \widetilde{\bu}_j+\underline{\bu}_{0j}(y)),
\end{split}
\end{equation}
and
\begin{equation}\label{DDE0}
\begin{split}
\widetilde{W}_j^\alpha& =-\epsilon^{\alpha}_{ \ \beta \gamma \delta} (\widetilde{u}_{j}^\beta+\underline{u}^\beta_{0}(y)) \partial^{\gamma} \widetilde{w}_j^\delta+\widetilde{c}_{sj}^{-2}\epsilon^{\alpha}_{\ \beta\gamma\delta}(\widetilde{u}_j^\beta+\underline{u}^\beta_{0j}(y)) \widetilde{w}_j^{\delta}\partial^{\gamma}\widetilde{h}_j,
\end{split}
\end{equation}
and
\begin{equation*}
\begin{split}
\widetilde{Q}_j^\alpha=& \widetilde{Q}_{1 \kappa j}^{\alpha \beta \gamma } \partial_\beta \widetilde{h}_j \partial_\gamma \widetilde{u}_j^\kappa+\widetilde{Q}_{2\kappa  \delta j}^{\alpha \beta \gamma } \partial_\beta \widetilde{u}_j^\kappa \partial_\gamma u_j^\delta +\widetilde{Q}_{3 j}^{\alpha \beta \gamma } \partial_\beta \widetilde{h}_j \partial_\gamma \widetilde{h}_j,
\\
\widetilde{D}_j^\alpha=&\widetilde{D}_{1 \kappa j}^{\beta \gamma } \partial_\beta \widetilde{h}_j \partial_\gamma \widetilde{u}_j^\kappa+\widetilde{D}_{2\kappa  \delta j}^{ \beta \gamma } \partial_\beta \widetilde{u}_j^\kappa \partial_\gamma \widetilde{u}_j^\delta +\widetilde{D}_{3 j}^{ \beta \gamma } \partial_\beta \widetilde{h}_j \partial_\gamma \widetilde{h}_j,
\end{split}
\end{equation*}
and the functions $\widetilde{Q}_{1 \kappa j}^{\alpha \beta \gamma }$, $\widetilde{Q}_{2\kappa  \delta j}^{\alpha \beta \gamma }$, $\widetilde{Q}_{3 j}^{\alpha \beta \gamma }$, $\widetilde{D}_{1 \kappa j}^{\beta \gamma }$, $\widetilde{D}_{2\kappa  \delta j}^{ \beta \gamma }$, $\widetilde{D}_{3 j}^{ \beta \gamma }$ have the same formulations with ${Q}_{1 \kappa }^{\alpha \beta \gamma }$, ${Q}_{2\kappa  \delta}^{\alpha \beta \gamma }$, $\widetilde{Q}_{3}^{\alpha \beta \gamma }$, ${D}_{1 \kappa }^{\beta \gamma }$, ${D}_{2\kappa  \delta}^{ \beta \gamma }$, ${D}_{3}^{ \beta \gamma }$ by replacing $(\bu, {h})$ to $(\widetilde{\bu}_j+\underline{\bu}_{0j}(y), \widetilde{h}_j+\underline{h}_{0j}(y))$. Using Proposition \ref{DDL3} again, we find that
\begin{equation}\label{Dsee0}
	\|(\widetilde{u}^0_j-1, \mathring{\widetilde{\bu}}_j)\|_{L^\infty_{[-2,2]}H_x^{\sstar}}+ \|\widetilde{h}_j\|_{L^\infty_{[-2,2]}H_x^{\sstar}}+ \| \widetilde{\bw}_j \|_{L^\infty_{[-2,2]}H_x^{2+\delta_1}} \leq \epsilon_2,
\end{equation}
and
\begin{equation}\label{Dsee1}
	\|d\widetilde{\bu}_j, d\widetilde{h}_j \|_{L^2_{[-2,2]}C^{\delta}_x} \leq \epsilon_2.
\end{equation}
Above, $\delta \in (0,s-2)$ is stated in \eqref{a1}. Furthermore, using \eqref{DP33} in Proposition \ref{DDL3}, for $1\leq r_1 \leq s+1$ and $t_0 \in [-2,2]$, the linear equation
\begin{equation}\label{Dsee2}
	\begin{cases}
		&\square_{ \widetilde{g}_j } f=0,\quad [-2,2]\times \mathbb{R}^3,
		\\
		&(f, \partial_t f)|_{t=t_0}=(f_0, f_1),
	\end{cases}
\end{equation}
admits a solution $f \in C([-2,2],H_x^{r_1})\times C^1([-2,2],H_x^{r_1-1})$ and for $k<r_1-1$, the following estimate holds\footnote{For all $j$, the initial norm of the initial data \eqref{sS4} is uniformly controlled by the same small parameter $\epsilon_3$, and the regularity of the initial data \eqref{sS4} only depends on $\sstar$, so the constant in \eqref{3s1} is uniform for all $\widetilde{{g}}$ depending on $j$.}:
\begin{equation}\label{3s1}
	\begin{split}
		\|\left< \nabla \right>^{k-1}df\|_{L^2_{[-2,2]} L^\infty_x}+ \|\left< \nabla \right>^{k-1}df\|_{L^2_{[-2,2]} L^\infty_x} \leq  & C(\| f_0\|_{H_x^r}+ \| f_1\|_{H_x^{r-1}} ).
	\end{split}
\end{equation}
Here $\widetilde{g}_j$ is given by \eqref{QRY}.

Note \eqref{yz0}. Then the function $(\widetilde{\bu}_j+\underline{\bu}_{0j}(y), \widetilde{h}_j+\underline{h}_{0j}(y), \widetilde{\bw}_j)$ is also a solution of System \eqref{yz0} with the initial data $(h_{0j}, \bu_{0j}, \bw_{0j})$. Consider the restrictions, for $y\in \mathbb{R}^3$,
\begin{equation}\label{Dsee3}
	\left( \widetilde{h}_j+\underline{h}_{0j}(y) \right)|_{\mathrm{K}^y}, \quad \left( \widetilde{\bu}_j+\underline{\bu}_{0j}(y) \right)|_{\mathrm{K}^y},
	\quad \widetilde{\bw}_j|_{\mathrm{K}^y},
\end{equation}
where $\mathrm{K}^y=\left\{ (t,x): ct+|x-y| \leq c+1, |t| <1 \right\}$, then the restriction \eqref{Dsee3} solves \eqref{WTe} on $\mathrm{K}^y$. By finite speed of propagation, a smooth solution $(\bar{\bv}_j+\widetilde{\bv}_{0j}(y), \bar{\rho}_j+ \widetilde{\rho}_{0j}(y), \widetilde{\bw}_j)$ solves \eqref{WTe} on $\mathrm{K}^y$. Let us use the cartesian grid $3^{-\frac12} \mathbb{Z}^3$ in $\mathbb{R}^3$, and a corresponding smooth partition of unity
 \begin{equation*}
 	\textstyle{\sum}_{y \in 3^{-\frac12} \mathbb{Z}^3 } \psi(x-y)=1.
 \end{equation*}
Therefore, if we set
\begin{equation}\label{Dsee4}
	\begin{split}
		& \bar{h}_j(t,x)  =\textstyle{\sum}_{y \in 3^{-\frac12} \mathbb{Z}^3}\psi(x-y) ( \widetilde{h}_j+\underline{h}_{0j}(y) ),
		\\
		&	\mathring{\bar{\bu}}_j(t,x)  =\textstyle{\sum}_{y \in 3^{-\frac12} \mathbb{Z}^3 }\psi(x-y) (\mathring{\widetilde{\bu}}_j+\mathring{\underline{\bu}}_{0j}(y)),
			\\
		& \bar{u}^0_j =\sqrt{1+|\mathring{\bar{\bu}}_j|^2 },\quad \bar{\bu}_j=( \bar{u}^0_j,\mathring{\bar{\bu}}_j)^{\mathrm{T}}, \quad
		\bar{\bw}_j(t,x)  =\mathrm{vort}( \mathrm{e}^{\bar{h}_j} \bar{\bu}_j).
	\end{split}
\end{equation}
By \eqref{Dsee4}, and time-space scaling $((T^*_j)^{-1}t,(T^*_j)^{-1}x)$, we then obtain
\begin{align}\label{Dsek}
		&(h, {\bu}_{j}, {\bw}_{j})=(\bar{h}_{j},\bar{\bu}_{j}, \bar{\bw}_{j}) ((T^*_j)^{-1}t,(T^*_j)^{-1}x),
		\\\label{Dsej}
		&(h, {\bu}_{j}, {\bw}_{j})|_{t=0}=(\bar{h}_{j},\bar{\bu}_{j}, \bar{\bw}_{j}) (0,(T^*_j)^{-1}x)=(h_{0j},\bu_{0j},\bw_{0j}).
\end{align}
Thus, $(h_j,{\bu}_j,{\bw}_j)$ is a smooth solution of \eqref{WTe} on $[0,T^*_j]\times \mathbb{R}^3$ with the initial data \eqref{Dsej}. Referring \eqref{Dsee4} and using \eqref{Dsee2}-\eqref{3s1}, we can see
\begin{equation}\label{Dsee5}
	\begin{split}
		\|d\bar{h}_j, d\bar{\bu}_j\|_{L^2_{[0,1]}C^{\delta}_x}
		\leq & \sup_{y \in 3^{-\frac12} \mathbb{Z}^3}  \|d\widetilde{h}_j, d\widetilde{\bu}_j\|_{L^2_{[0,1]}{C^{\delta}_x}}
		\\
		\leq & C(\|(\widetilde{u}^0_{0j}-1, \mathring{\widetilde{\bu}}_{0j})\|_{H_x^s}+ \|\widetilde{h}_{0j}\|_{H_x^s}+ \| \widetilde{\bw}_{0j} \|_{H_x^{2}}).
	\end{split}
\end{equation}
By changing of coordinates $(t,x)\rightarrow ((T_j^*)^{-1}t,(T_j^*)^{-1}x)$, for each $j\geq 1$, we therefore get
\begin{equation}\label{Dsee6}
	\begin{split}
		\|d{h}_j, d{\bu}_j\|_{L^2_{[0,T^*_j]}C^{\delta}_x}
		\leq & (T^*_j)^{-(\frac12+{\delta})}\|d\bar{h}_j, d\bar{\bu}_j\|_{L^2_{[0,1]}C^{\delta}_x}.
	\end{split}
\end{equation}
Using \eqref{Dsee5}, \eqref{Dsee6}, and \eqref{DTJ}, and ${\delta} \in (0,s-2)$, it follows
\begin{equation}\label{Dsee7}
	\begin{split}
		\|d{h}_j, d{\bu}_j\|_{L^2_{[0,T^*_j]}C^{\delta}_x}
		\leq	&C (T^*_j)^{-(\frac12+{\delta})}(\|(\widetilde{u}^0_{0j}-1, \mathring{\widetilde{\bu}}_{0j})\|_{H_x^s}+ \|\widetilde{h}_{0j}\|_{H_x^s}+ \| \widetilde{\bw}_{0j} \|_{H_x^{2}})
		\\
		\leq & C (T^*_j)^{-(\frac12+{\delta})} \{  (T^*_j)^{s-\frac32}\|(h, {u}^0_{0j}-1, {\mathring{\bu}}_{0j})\|_{\dot{H}_x^s}
		+ (T^*_j)^{\frac32}\| {\bw}_{0j} \|_{\dot{H}_x^{2}} \}
		\\
		\leq & C (\|(h_0, {u}^0_{0j}-1, {\mathring{\bu}}_{0j})\|_{H_x^s}+ \| {\bw}_{0j} \|_{H_x^{2}}).
	\end{split}
\end{equation}
Comining \eqref{Dsee7} and \eqref{pu0}, we obtain
\begin{equation}\label{yz4}
	\begin{split}
		\|d{h}_j, d{\bu}_j\|_{L^2_{[0,T^*_j]}C^{\delta}_x}
		\leq & C (1+E(0)).
	\end{split}
\end{equation}
Set
\begin{equation}\label{Etr}
	E(T^*_{j})=\|h_{j}\|_{L^\infty_{[0,T^*_{j}]} H^{\sstar}_x}+\|(u^0_j-1,\mathring{\bu}_j)\|_{L^\infty_{[0,T^*_{j}]} H^{\sstar}_x}+
	\|\bw_{j}\|_{L^\infty_{[0,T^*_{j}]} H^{2}_x}.
\end{equation}
By using \eqref{yz4} and Theorem \ref{VEt}, we have\footnote{Note that $\| \bw_j \|_{L^\infty_{[0,T^*_{j}]}H_x^{2+}}$ is not uniform bounded. But $\| \widetilde{\bw}_j \|_{L^\infty_{[-2,2]}H_x^{2+}} \lesssim \epsilon_2$ for all $j \geq 1$.}
\begin{equation}\label{AMM3}
	\begin{split}
		E(T^*_j) \leq & C( E(0)+E^5(0) )\exp \left( 5(1+E(0))\mathrm{e}^{5(1+E(0))} \right).
	\end{split}
\end{equation}
We also expect some estimates for the homogeneous linear wave equation on a short time interval $[0,T^*_j]$. Following the proof of \eqref{po10} and \eqref{po13} for Equation \eqref{po3}, we deduce that for $1\leq r_1 \leq s+1$, then
\begin{equation}\label{ru03}
	\begin{cases}
		\square_{{g}_j} f=0, \quad [0,T^*_j]\times \mathbb{R}^3,
		\\
		(f,f_t)|_{t=0}=(f_0,f_1)\in H_x^{r_1} \times H^{r_1-1}_x,
	\end{cases}
\end{equation}
has a unique solution on $[0,T^*_j]$. Moreover, for $a_1<r_1-1$, we have
\begin{equation}\label{ru04}
	\begin{split}
		\|\left< \nabla \right>^{a_1-1} df\|_{L^2_{[0,T^*_j]} L^\infty_x}
		\leq & C(\|{f}_0\|_{{H}_x^{r_1}}+ \| {f}_1 \|_{{H}_x^{r_1-1}}),
	\end{split}
\end{equation}
and
\begin{equation}\label{ru05}
	\begin{split}
		\|f\|_{L^\infty_{[0,T^*_j]} H^{r_1}_x}+ \|\partial_t f\|_{L^\infty_{[0,T_j]} H^{r_1-1}_x} \leq  C(\| f_0\|_{H_x^{r_1}}+ \| f_1\|_{H_x^{r_1-1}}).
	\end{split}
\end{equation}
Seeing from \eqref{ru04} and \eqref{ru05}, there is a uniform bound for $h_{j}, \bu_{j}$ and $\bw_{j}$. Unfortunately, the time interval $[0,T^*_{j}]$ is not uniform for $j$. If the sequence $(h_{j}, \bu_{j}, \bw_{j})$ can be applied in the analysis, then we have to extend this solution on a regular time interval. In fact, \eqref{yz4} is a strong Strichartz estimate given the derivatives. If we only obtain a loss of derivative for Strichartz estimates, the time interval could be regular.

\subsubsection{A loss of Strichartz estimates on a short time-interval.} \label{esest}
To get a loss of Strichartz estimates on a short-time-interval, we will discuss it by using phase truncation\footnote{Here, we mainly inspired by Tataru \cite{T1}, Bahouri-Chemin \cite{BC2}, Ai-Ifrim-Tataru \cite{AIT}, and Andersson-Zhang \cite{AZ}. Similar to \cite{AZ}, we conclude the proof by induction method.}. Next, we divide it into two cases: the high and low frequency for $(h_{j}, {\bu}_{j})$.

\textit{Case 1: High frequency($k \geq j$).} Using \eqref{yz4}, we get
\begin{equation}\label{yz6}
	\begin{split}
		& \| d h_j, d \bu_j \|_{L^2_{[0,T^*_j]}C^a_x} \leq C(1+E(0)), \quad  a \in [0,s-2).
	\end{split}
\end{equation}
For $k \geq j$, by Bernstein inequality and \eqref{yz6}, we have
\begin{equation}\label{yz9}
	\begin{split}
		\| P_{k} d h_j, P_{k} d \bu_j \|_{L^2_{[0,T^*_j]}L^\infty_x}
		\leq & 2^{-ka}\| P_k d h_j, P_k d \bu_j \|_{L^2_{[0,T^*_j]}C^a_x}
		\\
		\leq & C2^{-ja} \| d h_j, d \bu_j \|_{L^2_{[0,T^*_j]}C^a_x}
		\\
		\leq &  C2^{-ja} (1+E(0)).
	\end{split}
\end{equation}
Taking $a=9\delta_1$ in \eqref{yz9}, so we obtain
\begin{equation}\label{yz10}
	\begin{split}
		\| P_{k} d h_j, P_{k} d \bu_j \|_{L^2_{[0,T^*_j]}L^\infty_x}
		\leq & C   2^{-\delta_{1}k} \cdot (1+E^3(0)) 2^{-7\delta_{1}j}, \quad k \geq j,
	\end{split}
\end{equation}
\textit{Case 2: Low frequency($k<j$).} In this case, it's a little different from $k \geq j$. Fortunately, there is some good estimates for difference terms $P_k(d{h}_{m+1}-d{h}_m)$ and $P_k(d{\bu}_{m+1}-d{\bu}_{m})$. Following \eqref{De} and \eqref{De0}, we set
\begin{equation}\label{etad0}
\begin{split}
	& u_j^\alpha= u^\alpha_{j+}+u^\alpha_{j-}, \quad
	{\mathbf{P}}_j= -{\mathrm{P}}_j^{\beta \gamma} \partial^2_{\beta \gamma}
	\\
	& \mathbf{P}_j u_{j-}^\alpha = \mathrm{e}^{-h_j}W_j^\alpha, \quad	\mathrm{P}_j^{\beta \gamma} =  m^{\beta \gamma}+2u_j^{\beta}u_j^{\gamma}.
\end{split}	
\end{equation}
We will obtain some good estimates of $P_k(d{h}_{m+1}-d{h}_m)$ and $P_k(d{\bu}_{m+1}-d{\bu}_m)$ by using a Strichartz estimates \eqref{ru04} of \eqref{ru03}. The good estimate is from a loss of derivatives of Strichartz estimates.

For $m\in \mathbb{Z}^+$, by Bernstein inequality, we get
\begin{equation}\label{yz12}
	\begin{split}
		& \|{\bu}_{0(m+1)}-{\bu}_{0m}\|_{L^2}
		\lesssim  2^{-sm} \|{\bu}_{0(m+1)}-{\bu}_{0m}\|_{\dot{H}^s},
\\
		& \|{h}_{0(m+1)}-{h}_{0m}\|_{L^2} \lesssim   2^{-sm}\|h_{0(m+1)}-h_{0m}\|_{\dot{H}^s},
		\\
		& \|{\bw}_{0(m+1)}-{\bw}_{0m}\|_{L^2} \lesssim   2^{-2m}\|{\bw}_{0(m+1)}-{\bw}_{0m}\|_{\dot{H}^2}.
	\end{split}	
\end{equation}
We claim that
\begin{align}\label{yu0}
	 \|\bu_{m+1}-\bu_{m}, h_{m+1}-h_{m}\|_{L^\infty_{ [0,T^*_{m+1}] } L^2_x} \leq & C2^{-\sstar m}E(0),
	\\\label{yu1}
	 \|\bw_{m+1}-\bw_{m}\|_{L^\infty_{ [0,T^*_{m+1}]  } L^2_x} \leq & C2^{- (\sstar-1)m}E(0),
	 \\\label{yub}
	 \|\bw_{m+1}-\bw_{m}\|_{L^\infty_{ [0,T^*_{m+1}]  } \dot{H}^2_x} \leq & C E(0).
\end{align}
We remark that the proof of \eqref{yu0}-\eqref{yub} can repeat the proof of \eqref{ebs2}-\eqref{ebs3}. So we omit the details.

On the other hand, using Strichartz estimates \eqref{ru04} and interpolation formula, then
\begin{equation}\label{see80}
	\begin{split}
		& \|\nabla({h}_{m+1}-{h}_{m}) \|_{L^2_{[0,T^*_{m+1}]} C^{\delta_{1}}_x}+\|\nabla ({\bu}_{+(m+1)}-{\bu}_{+m}) \|_{L^2_{[0,T^*_{m+1}]} C^{\delta_{1}}_x}
		\\
		\leq  &  C\| h_{m+1}-h_{m}, {\bu}_{m+1}-{\bu}_{m} \|_{L^\infty_{[0,T^*_{m+1}]} H^{2+2\delta_{1}}_x}+\| {\bw}_{m+1}-{\bw}_{m} \|_{L^2_{[0,T^*_{m+1}]} H^{\frac32+2\delta_{1}}_x}
		\\
		\leq  &  C\| h_{m+1}-h_{m}, {\bu}_{m+1}-{\bu}_{m} \|_{L^\infty_{[0,T^*_{m+1}]} H^{2+2\delta_{1}}_x}
		\\
		& +\| {\bw}_{m+1}-{\bw}_{m} \|^{\frac14-\delta_1}_{L^2_{[0,T^*_{m+1}]} L^{2}_x}\| {\bw}_{m+1}-{\bw}_{m} \|^{\frac34+\delta_1}_{L^\infty_{[0,T^*_{m+1}]} \dot{H}^{2}_x}.
	\end{split}
\end{equation}
Noting $s=2+40\delta_{1}$ and using \eqref{yu0}-\eqref{yub}, we can bound \eqref{see80} by
\begin{equation}\label{see99}
	\begin{split}
		& \|\nabla(h_{m+1}-h_{m}) \|_{L^2_{[0,T^*_{m+1}]} C^{\delta_{1}}_x}+\|\nabla ({\bu}_{+(m+1)}-{\bu}_{+m}) \|_{L^2_{[0,T^*_{m+1}]} C^{\delta_{1}}_x}
		\leq   2^{-7\delta_1m} E(0).
	\end{split}
\end{equation}
By \eqref{yu1} and Sobolev imbeddings, we shall get
\begin{equation}\label{siq}
	\begin{split}
		\|\nabla ({\bu}_{-(m+1)}-{\bu}_{-m}) \|_{L^2_{[0,T^*_{m+1}]} L^\infty_x}
		\leq & \| \nabla ({\bu}_{-(m+1)}-{\bu}_{-m}) \|_{L^2_{[0,T^*_{m+1}]} L^\infty_x}
		\\
		\leq  & C \|{\bw}_{m+1}- {\bw}_{m} \|_{L^\infty_{[0,T^*_{m+1}]} H^{\frac32+\delta_{1}}_x}
		\\
		\leq & C (1+E^2(0)) 2^{-8\delta_{1} m} .
	\end{split}
\end{equation}
If we add \eqref{see99} and \eqref{siq}, then we have
\begin{equation}\label{siw}
	\begin{split}
		& \|\nabla( {\rho}_{m+1}-{\rho}_{m} ) \|_{L^2_{[0,T^*_{m+1}]} C^{\delta_{1}}_x}+\| \nabla ({\bv}_{m+1}-{\bv}_{m}) \|_{L^2_{[0,T^*_{m+1}]} C^{\delta_{1}}_x}\leq  C (1+E^2(0)) 2^{-6\delta_{1} m}.
	\end{split}
\end{equation}
By using \eqref{QH} and \eqref{siw}, we shall obtain
\begin{equation}\label{sie}
	\begin{split}
		& \|\partial_t(h_{m+1}-h_{m})\|_{L^2_{[0,T^*_{m+1}]} C^{\delta_{1}}_x}+\| \partial_t({\bu}_{m+1}-{\bu}_{m}) \|_{L^2_{[0,T^*_{m+1}]} C^{\delta_{1}}_x}
		\\
		\leq  & \|\nabla(h_{m+1}-h_{m}) , \nabla ({\bu}_{m+1}-{\bu}_{m}) \|_{L^2_{[0,T^*_{m+1}]} C^{\delta_{1}}_x} \cdot (1+ \|(h_m, u^0_m-1,\mathring{\bu}_m)\|_{L^\infty_{[0,T^*_{m+1}]} H^s_x} )
		\\
		\leq & C (1+E^3(0)) 2^{-6\delta_{1} m}.
	\end{split}
\end{equation}
Due to \eqref{siw} and \eqref{sie}, for $k<j$, it yields
\begin{equation}\label{Sia}
	\|P_k d(h_{m+1}-h_{m}), P_k d ({\bu}_{m+1}-{\bu}_{m}) \|_{L^2_{[0,T^*_{m+1}]} L^\infty_x}
	\leq   2^{-\delta_1k}\cdot C (1+E^3(0)) 2^{-6\delta_1m},
\end{equation}
and
\begin{equation}\label{kz4}
	\|P_k d(h_{m+1}-h_{m}), P_k d ({\bu}_{m+1}-{\bu}_{m}) \|_{L^1_{[0,T^*_{m+1}]} L^\infty_x}
	\leq   2^{-\delta_1k} \cdot C (1+E^3(0)) 2^{-6\delta_1m}.
\end{equation}
Based on \eqref{yz10}, \eqref{Sia}, \eqref{kz4}, following section 10.3.3 in \cite{AZ}, we can extend the solution $(h_j, \bu_j, \bw_j)$ to a uniformly regular time-interval.
\subsubsection{Uniform energy and Strichartz estimates on a regular time-interval $[0,T^*_{N_0}]$.}\label{finalk}
We recall $T_{N_0}^*=[E(0)]^{-1}2^{-\delta_{1}N_0}$, which is stated in \eqref{DTJ}. Note \eqref{yz10} and \eqref{Sia}. Therefore, when $k\geq j$ and $j\geq N_0$, using \eqref{yz10} and H\"older's inequality, we have
\begin{equation}\label{kf01}
	\begin{split}
		\| P_{k} d h_j, P_{k} d \bu_j \|_{L^1_{[0,T^*_j]}L^\infty_x} \leq & (T^*_j)^{\frac12}\| P_{k} d h_j, P_{k} d \bu_j \|_{L^2_{[0,T^*_j]}L^\infty_x}
		\\
		\leq & (T^*_{N_0})^{\frac12} \cdot C  (1+E^3(0))2^{-\delta_{1}k} 2^{-7\delta_1j}.
	\end{split}
\end{equation}
When $k< j$ and $m\geq N_0$, using \eqref{Sia} and H\"older's inequality, we also see
\begin{equation}\label{kf02}
	\|P_k d(h_{m+1}-h_{m}), P_k d({\bu}_{m+1}-{\bu}_{m}) \|_{L^1_{[0,T^*_{m+1}]} L^\infty_x}
	\leq    (T^*_{N_0})^{\frac12} \cdot C  (1+E^3(0))2^{-\delta_{1}k} 2^{-6\delta_1m}.
\end{equation}
On the other side, when $k\geq j$ and $j< N_0$, using \eqref{yz10}, we have
\begin{equation}\label{kf03}
	\begin{split}
		\| P_{k} d h_j, P_{k} d \bu_j \|_{L^1_{[0,T^*_{N_0}]}L^\infty_x} \leq & (T^*_{N_0})^{\frac12} \| P_{k} d h_j, P_{k} d \bu_j \|_{L^2_{[0,T^*_{N_0}]}L^\infty_x}
		\\
		\leq & (T^*_{N_0})^{\frac12} \| P_{k} d h_j, P_{k} d \bu_j \|_{L^2_{[0,T^*_j]}L^\infty_x}
		\\
		\leq & (T^*_{N_0})^{\frac12} \cdot C  (1+E^3(0))2^{-\delta_{1}k} 2^{-7\delta_1j}.
	\end{split}
\end{equation}
When $k< j$ and $m< N_0$, using \eqref{Sia} and H\"older's inequality, we can see
\begin{equation}\label{kf04}
	\begin{split}
		& \|P_k d(h_{m+1}-h_{m}), P_k d({\bu}_{m+1}-{\bu}_{m}) \|_{L^1_{[0,T^*_{N_0}]} L^\infty_x}
		\\
		\leq & (T^*_{N_0})^{\frac12} \|P_k d(h_{m+1}-h_{m}), P_k d({\bu}_{m+1}-{\bu}_{m}) \|_{L^2_{[0,T^*_{N_0}]} L^\infty_x}
		\\
		\leq & (T^*_{N_0})^{\frac12} \|P_k d(h_{m+1}-h_{m}), P_k d({\bu}_{m+1}-{\bu}_{m}) \|_{L^2_{[0,T^*_{m+1}]} L^\infty_x}
		\\
		\leq    & (T^*_{N_0})^{\frac12} \cdot C  (1+E^3(0))2^{-\delta_{1}k} 2^{-6\delta_1m}.
	\end{split}
\end{equation}
Due to a different time-interval for the sequency  $(h_j,\bu_j,\bw_j)$, so we discuss the solutions $(h_j,\bu_j,\bw_j)$ if  $j \leq N_0$ or  $j \geq N_0+1$ as follows.

\textit{Case 1: $j \leq N_0$.} In this case, $(h_j, \bu_j, \bw_j)$ exists on $[0,T_j^*]$, and $[0,T^*_{N_0}]\subseteq [0,T^*_{j}]$. So we don't need to extend the solution $(h_j, \bu_j, \bw_j)$ if $j \leq N_0$. Using \eqref{yz4} and H\"older's inequality, we get
\begin{equation*}
	\| dh_j, d\bu_j\|_{L^1_{[0,T^*_{N_0}]}L^\infty_x} \leq (T^*_{N_0})^{\frac12} \| dh_j, d\bu_j\|_{L^2_{[0,T^*_{N_0}]}L^\infty_x} \leq C(T^*_{N_0})^{\frac12} (1+E(0)).
\end{equation*}
By \eqref{pp8}, this yields
\begin{equation}\label{ky0}
	\| dh_j, d\bu_j\|_{L^1_{[0,T^*_{N_0}]}L^\infty_x} \leq 2.
\end{equation}
By the Newton-Leibniz formula and \eqref{pu00}, it follows that
\begin{equation}\label{ky1}
	\|h_j, u^0_j-1,\mathring{\bu}_j\|_{L^\infty_{[0,T^*_{N_0}] \times \mathbb{R}^3}}\leq |h_{0j}, u^0_{0j}-1,\mathring{\bu}_{0j}|+ \| dh_j, d\bu_j\|_{L^1_{[0,T^*_{N_0}]}L^\infty_x} \leq 2+C_0.
\end{equation}
Using \eqref{ky0} and Theorem \ref{VEt}, we get
\begin{equation}\label{ky2}
	E(T^*_{N_0}) \leq C(E(0)+E^5(0))\exp(  5 \textrm{e}^5 )=C_* .
\end{equation}
Here $C_*$ is stated as in \eqref{Cstar}.

\textit{Case 2: $j \geq N_0+1$.} In this case, we have to extend the time interval $I_1=[0,T^*_j]$ to $[0,T^*_{N_0}]$. So we will use \eqref{yz10}, \eqref{Sia}, and Theorem \ref{VEt} to calculate the energy at time $T_j^*$. Then we start at $T_j^*$ and get a new time-interval.

To be simple, we set
\begin{equation}\label{ti1}
	I_1=[0,T^*_j]=[t_0,t_1], \quad \quad |I_1|=[E(0)]^{-1}2^{-\delta_{1} j}.
\end{equation}
By a frequency decomposition, we get
\begin{equation}\label{kz03}
	\begin{split}
		d \bu_j= & \textstyle{\sum}^{\infty}_{k=j} d \bu_j+ \textstyle{\sum}^{j-1}_{k=1}P_k d \bu_j
		\\
		=& P_{\geq j}d \bu_j+ \textstyle{\sum}^{j-1}_{k=1} \textstyle{\sum}_{m=k}^{j-1} P_k  (d\bu_{m+1}-d\bu_m)+ \textstyle{\sum}^{j-1}_{k=1}P_k d \bu_k .
	\end{split}
\end{equation}
Similarly, we also have
\begin{equation}\label{kz04}
	\begin{split}
		d h_j= & P_{\geq j}d h_j+ \textstyle{\sum}^{j-1}_{k=1} \textstyle{\sum}_{m=k}^{j-1} P_k  (d h_{m+1}-d h_m)+ \textstyle{\sum}^{j-1}_{k=1}P_k d h_k.
	\end{split}
\end{equation}
When $k\geq j$, using \eqref{kf01} and \eqref{kf03}, we have\footnote{In the case $j\geq N_0$ we use \eqref{kf01}. In the case $j < N_0$, we take $T^*_j = T^*_{N_0}$ and use \eqref{kf03} to give a bound on $[0,T^*_{N_0}]$.}
\begin{equation}\label{kz01}
	\begin{split}
		\| P_{k} d h_j, P_{k} d \bu_j \|_{L^1_{[0, T^*_j]}L^\infty_x} \leq   &   (T^*_{N_0})^{\frac12} \cdot C(1+E^3(0)) 2^{-\delta_{1}k} 2^{-7\delta_{1}j}.
	\end{split}
\end{equation}
When $k< j$, using \eqref{kf02} and \eqref{kf04}, we also see\footnote{In the case $m\geq N_0$ we use \eqref{kf02}. In the case $m < N_0$, we take $T^*_{m+1} = T^*_{N_0}$ and use \eqref{kf04} to give a bound on $[0,T^*_{N_0}]$.}
\begin{equation}\label{kz02}
	\|P_k d(h_{m+1}-h_{m}), P_k d({\bu}_{m+1}-{\bu}_{m}) \|_{L^1_{[0,T^*_{m+1}]} L^\infty_x}
	\leq    (T^*_{N_0})^{\frac12} \cdot C  (1+E^3(0))2^{-\delta_{1}k} 2^{-6\delta_1m}.
\end{equation}
We will use  and \eqref{kz01}-\eqref{kz02} to give a precise analysis of \eqref{kz03}-\eqref{kz04} and get some new time-intervals, and then we try to extend $h_j$ from $[0,T^*_j]$ to $[0,T^*_{N_0}]$. Our strategy is as follows. In step 1, we extend it from $[0,T^*_j]$ to $[0,T^*_{j-1}]$. In step 2, we use induction methods to conclude these estimates and also extend it to $[0,T_{N_0}^*]$.

\textbf{Step 1: Extending $[0,T^*_j]$ to $[0,T^*_{j-1}] \ (j \geq N_0+1)$.} To start, referring \eqref{DTJ}, so we need to calculate $E(T^*_j)$ for obtaining a length of a new time-interval. Then we shall calculate $\|d h_j, d \bu_j\|_{L^1_{[0,T^*_j]}L^\infty_x}$. Using \eqref{kz03} and \eqref{kz04}, we derive that
\begin{equation*}
	\begin{split}
		\|d h_j, d \bu_j\|_{L^1_{[0,T^*_j]} L^\infty_x }
		\leq & 	\|P_{\geq j}dh_j, P_{\geq j}d \bu_j\|_{L^1_{[0,T^*_j]} L^\infty_x }  + \textstyle{\sum}^{j-1}_{k=1} \|P_k dh_k, P_k d\bu_k\|_{L^1_{[0,T^*_j]}L^\infty_x}\\
		& + \textstyle{\sum}^{j-1}_{k=1} \textstyle{\sum}_{m=k}^{j-1}  \|P_k  (dh_{m+1}-dh_m), P_k  (d\bu_{m+1}-d\bu_m)\|_{L^1_{[0,T^*_j]}L^\infty_x}
		\\
		\leq & 	\|P_{\geq j}dh_j, P_{\geq j}d \bu_j\|_{L^1_{[0,T^*_j]}L^\infty_x} + \textstyle{\sum}^{j-1}_{k=1} \|P_k dh_k, P_k d\bu_k\|_{L^1_{[0,T^*_k]}L^\infty_x}\\
		& + \textstyle{\sum}^{j-1}_{k=1} \textstyle{\sum}_{m=k}^{j-1}\| P_k  (dh_{m+1}-dh_m), P_k  (d\bu_{m+1}-d\bu_m)\|_{L^1_{[0,T^*_{m+1}]}L^\infty_x}
		\\
		\leq & C(1+E^3(0))(T^*_{N_0})^{\frac12}  \textstyle{\sum}_{k=j}^{\infty} 2^{-\delta_{1} k} 2^{-7\delta_{1} j}
		\\
		& + C(1+E^3(0))(T^*_{N_0})^{\frac12} \textstyle{\sum}^{j-1}_{k=1} 2^{-\delta_{1} k} 2^{-7\delta_{1} k}
		\\
		& +  C(1+E^3(0))(T^*_{N_0})^{\frac12} \textstyle{\sum}^{j-1}_{k=1} \textstyle{\sum}_{m=k}^{j-1} 2^{-\delta_{1} k} 2^{-6\delta_{1} m}
		\\
		\leq & C(1+E^3(0))(T^*_{N_0})^{\frac12} [\frac{1}{3}(1-2^{-\delta_{1}})]^{-2}.
	\end{split}
\end{equation*}
So we get
\begin{equation}
	\begin{split}\label{kz05}
		\|d h_j, d \bu_j\|_{L^1_{[0,T^*_j]}L^\infty_x}
		\leq  C(1+E^3(0))(T^*_{N_0})^{\frac12} [\frac{1}{3}(1-2^{-\delta_{1}})]^{-2}.
	\end{split}
\end{equation}
By using \eqref{kz05}, \eqref{pp8} and \eqref{pu00}, we get
\begin{equation}
	\begin{split}\label{kp05}
		\| h_j,  u^0_j-1, \mathring{\bu}_j\|_{L^\infty_{[0,T^*_j]}L^\infty_x} \leq \| h_{0j},  u^0_{0j}-1, \mathring{\bu}_{0j}\|+	\|d h_j, d \bu_j\|_{L^1_{[0,T^*_j]}L^\infty_x}
		\leq  C_0+2.
	\end{split}
\end{equation}
By \eqref{kz05} and \eqref{pp8} and Theorem \ref{VEt}, we have
\begin{equation}\label{kz06}
	\begin{split}
		E(T^*_{j}) \leq C(E(0)+E^5(0))\exp(  5\textrm{e}^5  ) =C_*.
	\end{split}
\end{equation}
Above, $C_*$ has stated in \eqref{Cstar}. Starting at the time $T^*_j$, seeing \eqref{DTJ} and \eqref{kz06}, we shall get an extending time-interval with a length of $(C_*)^{-1}2^{-\delta_{1} j}$. But, if $T^*_j + (C_*)^{-1}2^{-\delta_{1} j} \geq T^*_{j-1}$, so we have finished this step. Else, we need to extend it again.

$\mathit{Case 1: T^*_j + (C_*)^{-1}2^{-\delta_{1} j} \geq T^*_{j-1} }$, then we get a new interval
\begin{equation}\label{deI2}
	I_2=[T^*_j, T^*_{j-1}], \quad |I_2|= (2^{\delta_{1}}-1) E(0)^{-1}2^{-\delta_{1} j}.
\end{equation}
Moreover, referring \eqref{kz01} and \eqref{kz02}, we can deduce
\begin{equation}\label{kz07}
	\| P_{k} d h_j, P_{k} d \bu_j \|_{L^1_{[T^*_j, T^*_{j-1}]}L^\infty_x}
	\leq  (T^*_{N_0})^{\frac12}  \cdot C  (1+C^3_*) 2^{-\delta_{1}k} 2^{-7\delta_{1}j}, \quad k \geq j,
\end{equation}
and $k<j$,
\begin{equation}\label{kz08}
	\|P_k (d{h}_{j}-d{h}_{j-1}), P_k (d{\bu}_{j}-d{\bu}_{j-1}) \|_{L^1_{[T^*_j, T^*_{j-1}]} L^\infty_x}
	\leq   (T^*_{N_0})^{\frac12}  \cdot C  (1+C^3_*) 2^{-\delta_{1}k} 2^{-6\delta_1(j-1)}.
\end{equation}
Using \eqref{pp8} and $j \geq N_0+1$, \eqref{kz07} and \eqref{kz08} yields
\begin{equation}\label{kz09}
	\| P_{k} d h_j, P_{k} d \bu_j \|_{L^1_{[T^*_j, T^*_{j-1}]}L^\infty_x}
	\leq  (T^*_{N_0})^{\frac12}  \cdot C  (1+E^3(0)) 2^{-\delta_{1}k} 2^{-6\delta_{1}j}, \quad k \geq j,
\end{equation}
and $k<j$,
\begin{equation}\label{kz10}
	\|P_k (d{h}_{j}-d{\bu}_{j-1}), P_k (d{\bu}_{j}-d{\bu}_{j-1}) \|_{L^1_{[T^*_j, T^*_{j-1}]} L^\infty_x}
	\leq    (T^*_{N_0})^{\frac12}  \cdot C   (1+E^3(0)) 2^{-\delta_{1}k} 2^{-5\delta_1(j-1)}.
\end{equation}
Therefore, we could get the following estimate
\begin{equation}\label{kz11}
	\begin{split}
		& \|d h_j, d\bu_j\|_{L^1_{[0,T^*_{j-1}]} L^\infty_x}
		\\
		\leq & 	\|P_{\geq j}dh_j, P_{\geq j}d \bu_j\|_{L^1_{[0,T^*_{j-1}]} L^\infty_x}  + \textstyle{\sum}^{j-1}_{k=1} \|P_k dh_k, P_k d\bu_k\|_{L^1_{[0,T^*_{j-1}]} L^\infty_x}
		\\
		& + \textstyle{\sum}^{j-1}_{k=1} \|P_k  (d h_{j}-d h_{j-1}), P_k  (d\bu_{j}-d\bu_{j-1})\|_{L^1_{[0,T^*_{j-1}]} L^\infty_x}
		\\
		& + \textstyle{\sum}^{j-2}_{k=1} \textstyle{\sum}_{m=k}^{j-2}  \|P_k  (d h_{m+1}-d h_m), P_k  (d\bu_{m+1}-d\bu_m)\|_{L^1_{[0,T^*_{j-1}]} L^\infty_x}.
	\end{split}
\end{equation}
Due to \eqref{kz01} and \eqref{kz09}, it yields
\begin{equation}\label{kz12}
	\begin{split}
		\|P_{\geq j}dh_j, P_{\geq j}d\bu_j\|_{L^1_{[0,T^*_{j-1}]} L^\infty_x}
		\leq & C   (1+E^3(0)) (T^*_{N_0})^{\frac12}  \textstyle{\sum}^{\infty}_{k=j}2^{-\delta_{1}k} (2^{-7\delta_1 j}+ 2^{-6\delta_{1} j})
		\\
		\leq & C   (1+E^3(0)) (T^*_{N_0})^{\frac12}  \textstyle{\sum}^{\infty}_{k=j} 2^{-\delta_{1}k} 2^{-6\delta_{1} j}\times 2.
	\end{split}
\end{equation}
Due to \eqref{kz02} and \eqref{kz10}, it yields
\begin{equation}\label{kz13}
	\begin{split}
		& \textstyle{\sum}^{j-1}_{k=1} \|P_k  (dh_{j}-dh_{j-1}), P_k  (d\bu_{j}-d\bu_{j-1})\|_{L^1_{[0,T^*_{j-1}]} L^\infty_x}
		\\
		\leq & C   (1+E^3(0)) (T^*_{N_0})^{\frac12}  \textstyle{\sum}^{j-1}_{k=1} 2^{-\delta_{1}k} (2^{-6\delta_1 (j-1)}+ 2^{-5\delta_{1} (j-1)})
		\\
		\leq & C   (1+E^3(0)) (T^*_{N_0})^{\frac12}  \textstyle{\sum}^{j-1}_{k=1} 2^{-\delta_{1}k} 2^{-5\delta_{1} (j-1)} \times 2.
	\end{split}
\end{equation}
Inserting \eqref{kz12}-\eqref{kz13} into \eqref{kz11}, we derive that
\begin{equation}\label{kz14}
	\begin{split}
		\|d h_j, d \bu_j\|_{L^1_{[0,T^*_{j-1}]} L^\infty_x}
		\leq & 	C(1+E^3(0)) (T^*_{N_0})^{\frac12}   \textstyle{\sum}_{k=j}^{\infty} 2^{-\delta_{1} k} 2^{-6\delta_{1} j}\times 2
		\\
		& + C   (1+E^3(0)) (T^*_{N_0})^{\frac12}  \textstyle{\sum}^{j-1}_{k=1} 2^{-\delta_{1}k} 2^{-5\delta_{1} (j-1)} \times 2
		\\
		& + C(1+E^3(0)) (T^*_{N_0})^{\frac12}  \textstyle{\sum}^{j-1}_{k=1} 2^{-\delta_{1} k} 2^{-7\delta_{1} k}
		\\
		& +  C(1+E^3(0)) (T^*_{N_0})^{\frac12}  \textstyle{\sum}^{j-2}_{k=1} \textstyle{\sum}_{m=k}^{j-2} 2^{-\delta_{1} k} 2^{-6\delta_{1} m}
		\\
		\leq & 	C(1+E^3(0))(T^*_{N_0})^{\frac12}  [\frac13(1-2^{-\delta_{1}})]^{-2}.
	\end{split}
\end{equation}
By \eqref{kz14}, \eqref{pp8} and Theorem \ref{VEt}, we also prove
\begin{equation}\label{kz15}
	\begin{split}
		&| h_j, u^0_j-1, \mathring{\bu}_j | \leq 2+C_0,
		\\
		& E(T^*_{j-1}) \leq C_*.
	\end{split}
\end{equation}
$\mathit{Case 2: T^*_j + (C_*)^{-1}2^{-\delta_{1} j} < T^*_{j-1} }$. In this situation, we will record
\begin{equation*}
	I_2=[T^*_j, t_2], \quad |I_2| = (C_*)^{-1}2^{-\delta_{1} j}.
\end{equation*}
Referring \eqref{kz01} and \eqref{kz02}, we also deduce that
\begin{equation}\label{kz16}
	\| P_{k} d h_j, P_{k} d \bu_j \|_{L^1_{I_2}L^\infty_x}
	\leq  (T^*_{N_0})^{\frac12} \cdot C  (1+C^3_*) 2^{-\delta_{1}k} 2^{-7\delta_{1}j}, \quad k \geq j,
\end{equation}
\begin{equation}\label{kz17}
	\|P_k (d{h}_{j}-d{h}_{j-1}), P_k  (d{\bu}_{j}-d{\bu}_{j-1}) \|_{L^1_{I_2} L^\infty_x}
	\leq    (T^*_{N_0})^{\frac12} \cdot C  (1+C^3_*) 2^{-\delta_{1}k} 2^{-6\delta_1(j-1)}, \quad k<j.
\end{equation}
Similarly, we can also get
\begin{equation*}
	\begin{split}
		\|dh_j, d\bu_j\|_{L^1_{ I_1 \cup I_2 } L^\infty_x }
		\leq & 	\|P_{\geq j}dh_j, P_{\geq j}d \bu_j\|_{L^1_{I_1 \cup I_2} L^\infty_x }   + \textstyle{\sum}^{j-1}_{k=1} \|P_k dh_k, P_k d\bu_k\|_{L^1_{[0,T^*_k]} L^\infty_x }
		\\
		& + \textstyle{\sum}^{j-1}_{k=1} \|P_k  (dh_{j}-dh_{j-1}), P_k  (d\bu_{j}-d\bu_{j-1})\|_{L^1_{I_1 \cup I_2 } L^\infty_x }
		\\
		& + \textstyle{\sum}^{j-2}_{k=1} \textstyle{\sum}_{m=k}^{j-2}  \|P_k  (dh_{m+1}-dh_m), P_k  (d\bu_{m+1}-d\bu_m)\|_{L^1_{I_1 \cup I_2} L^\infty_x } .
	\end{split}
\end{equation*}
Noting $ I_1 \cup I_2 \subseteq T^*_k$ if $k \leq j-1$, then it follows that
\begin{equation}\label{kz11A}
	\begin{split}
		& \|d h_j, d \bu_j\|_{L^1_{ I_1 \cup I_2 } L^\infty_x }
		\\
		\leq & 	\|P_{\geq j}dh_j, P_{\geq j}d \bu_j\|_{L^1_{I_1 \cup I_2} L^\infty_x }   + \textstyle{\sum}^{j-1}_{k=1} \|P_k dh_k, P_k d\bu_k\|_{L^1_{I_1 \cup I_2 } L^\infty_x }
		\\
		& + \textstyle{\sum}^{j-1}_{k=1} \|P_k  (d h_{j}-d h_{j-1}), P_k  (d\bu_{j}-d\bu_{j-1})\|_{L^1_{I_1 \cup I_2 } L^\infty_x }
		\\
		& + \textstyle{\sum}^{j-2}_{k=1} \textstyle{\sum}_{m=k}^{j-2}  \|P_k  (dh_{m+1}-dh_m), P_k  (d\bu_{m+1}-d\bu_m)\|_{L^1_{[0,T^*_{m+1}]} L^\infty_x }.
	\end{split}
\end{equation}
Inserting \eqref{kz01}, \eqref{kz02} and \eqref{kz16} and \eqref{kz17} to \eqref{kz11A}, we have
\begin{equation}\label{kz11a}
	\begin{split}		
		\|d h_j, d \bu_j\|_{L^1_{ I_1 \cup I_2 } L^\infty_x }
		\leq & 	C(1+E^3(0)) (T^*_{N_0})^{\frac12} \textstyle{\sum}_{k=j}^{\infty} 2^{-\delta_{1} k} 2^{-6\delta_{1} j}\times 2
		\\
		& + C   (1+E^3(0)) (T^*_{N_0})^{\frac12} \textstyle{\sum}^{j-1}_{k=1} 2^{-\delta_{1}k} 2^{-5\delta_{1} (j-1)} \times 2
		\\
		& + C(1+E^3(0)) (T^*_{N_0})^{\frac12} \textstyle{\sum}^{j-1}_{k=1} 2^{-\delta_{1} k} 2^{-7\delta_{1} k}
		\\
		& +  C(1+E^3(0)) (T^*_{N_0})^{\frac12} \textstyle{\sum}^{j-2}_{k=1} \textstyle{\sum}_{m=k}^{j-2} 2^{-\delta_{1} k} 2^{-6\delta_{1} m}
		\\
		\leq & 	C(1+E^3(0))(T^*_{N_0})^{\frac12} [\frac13(1-2^{-\delta_{1}})]^{-2}.
	\end{split}
\end{equation}
By \eqref{kz11a}, \eqref{pp8} and Theorem \ref{VEt}, we also prove
\begin{equation}\label{kz15f}
	\begin{split}
		& |h_j,u^0_j-1,\mathring{\bu}_j| \leq 2+C_0,
		\\
		& E(t_2) \leq C_*.
	\end{split}
\end{equation}
Therefore, we can repeat the process with a length with $(C_*)^{-1}2^{-\delta_{1} j}$ till extending it to $T^*_{j-1}$. Moreover, on every new time-interval with $(C_*)^{-1}2^{-\delta_{1} j}$ \eqref{kz16} and \eqref{kz17} hold. Set
\begin{equation}\label{times1}
	X_1= \frac{T^*_{j-1}-T^*_j}{C^{-1}_*2^{-\delta_{1} j}}= (2^{\delta_{1}}-1)C_* E^{-1}(0).
\end{equation}
So we need a maximum of $X_1$-times to reach the time $T^*_{j-1}$ both in case 2(it's also adapt to case 1 for calculating the times). As a result, we shall calculate
\begin{equation*}
	\begin{split}
		\|d h_j, d \bu_j\|_{L^1_{[0,T^*_{j-1}]} L^\infty_x}
		\leq & 	\|P_{\geq j}dh_j, P_{\geq j}d \bu_j\|_{L^1_{[0,T^*_{j-1}]} L^\infty_x}
		+ \textstyle{\sum}^{j-1}_{k=1} \|P_k dh_k, P_k d\bu_k\|_{L^1_{[0,T^*_{k}]} L^\infty_x}
		\\
		+ & \textstyle{\sum}^{j-1}_{k=1} \|P_k  (dh_{j}-dh_{j-1}), P_k  (d\bu_{j}-d\bu_{j-1})\|_{L^1_{[0,T^*_{j-1}]} L^\infty_x}.
		\\
		& + \textstyle{\sum}^{j-2}_{k=1} \textstyle{\sum}_{m=k}^{j-2}  \|P_k  (dh_{m+1}-dh_m), P_k  (d\bu_{m+1}-d\bu_m)\|_{L^1_{[0,T^*_{m}]} L^\infty_x}
	\end{split}
\end{equation*}	
Due to \eqref{kz01}, \eqref{kz02}, \eqref{kz16}, \eqref{kz17}, and \eqref{pp8}, we obtain
\begin{equation}\label{kzqt}
	\begin{split}
		& \|d h_j, d \bu_j\|_{L^1_{[0,T^*_{j-1}]} L^\infty_x}
		\\
		\leq & C(1+E^3(0)) (T^*_{N_0})^{\frac12} \textstyle{\sum}_{k=j}^{\infty} 2^{-\delta_{1} k} 2^{-6\delta_{1} j}\times (2^{\delta_{1}}-1)C_* E^{-1}(0)
		\\
		& + C   (1+E^3(0)) (T^*_{N_0})^{\frac12} \textstyle{\sum}^{j-1}_{k=1} 2^{-\delta_{1}k} 2^{-5\delta_{1} (j-1)} \times (2^{\delta_{1}}-1)C_* E^{-1}(0)
		\\
		& + C(1+E^3(0))(T^*_{N_0})^{\frac12} \textstyle{\sum}^{j-1}_{k=1} 2^{-\delta_{1} k} 2^{-7\delta_{1} k}
		\\
		& +  C(1+E^3(0))(T^*_{N_0})^{\frac12} \textstyle{\sum}^{j-2}_{k=1} \textstyle{\sum}_{m=k}^{j-2} 2^{-\delta_{1} k} 2^{-6\delta_{1} m}
		\\
		\leq & 	C(1+E^3(0))(T^*_{N_0})^{\frac12}[\frac13(1-2^{-\delta_{1}})]^{-2}.
	\end{split}
\end{equation}
Therefore, by \eqref{kzqt}, \eqref{pp8} and Theorem \eqref{VEt}, we can see that
\begin{equation*}
	\begin{split}
		& \| h_j,u^0_j-1,\mathring{\bu}_j \|_{L^\infty_{ [0,T^*_{j-1}]\times \mathbb{R}^3}} \leq 2+C_0, \quad E(T^*_{j-1}) \leq C_*.
	\end{split}
\end{equation*}
Since $\bu_j$ satisfies \eqref{muu}, so we have
\begin{equation}\label{kz270}
	\begin{split}
		& \| h_j,u^0_j-1,\mathring{\bu}_j \|_{L^\infty_{ [0,T^*_{j-1}]\times \mathbb{R}^3}} \leq 2+C_0, \quad, u^0_j \geq 1, \quad E(T^*_{j-1}) \leq C_*.
	\end{split}
\end{equation}
At this moment, both in case 1 or case 2, seeing from \eqref{kz270}, \eqref{kzqt}, \eqref{kz14}, \eqref{kz15}, through a maximum of $X_1=(2^{\delta_{1}}-1)C_* E^{-1}(0)$ times with each length $C_*^{-1}2^{-\delta_{1}j}$ or $(2^{\delta_{1}}-1)E(0)^{-1}2^{-\delta_{1} j}$, we shall extend the solutions $(\bv_j,\rho_j,\bw_j)$ from $[0,T^*_j]$ to $[0,T^*_{j-1}]$, and
\begin{equation}\label{kz27}
	\begin{split}
		& \| h_j,u^0_j-1,\mathring{\bu}_j \|_{L^\infty_{ [0,T^*_{j-1}]\times \mathbb{R}^3}} \leq 2+C_0, \quad u^0_j \geq 1,
	\quad E(T^*_{j-1}) \leq C_*,
		\\
		& \|d h_j, d \bu_j\|_{L^1_{[0,T^*_{j-1}]} L^\infty_x}
		\leq  	C(1+E^3(0))(T^*_{N_0})^{\frac12}[\frac13(1-2^{-\delta_{1}})]^{-2}.
	\end{split}		
\end{equation}
As a result, we also have extended the solutions $(\bv_m,\rho_m,\bw_m)$ ($m\in[N_0, j-1]$) from $[0,T^*_m]$ to $[0,T^*_{m-1}]$. Moreover, referring \eqref{kz01} and \eqref{kz02}, \eqref{kz27}, and \eqref{times1}, we get
\begin{equation}\label{kz29}
	\begin{split}
		& \| P_{k} d h_m, P_{k} d \bu_m \|_{L^1_{[0,T^*_{m-1}]}L^\infty_x}
		\\
		\leq  & \| P_{k} dh_m, P_{k} d \bu_m \|_{L^1_{[0,T^*_{m}]}L^\infty_x}
		 +\| P_{k} d h_m, P_{k} d \bu_m \|_{L^1_{[T^*_{m},T^*_{m-1}]}L^\infty_x}
		\\
		\leq  & C(1+E^3(0))(T^*_{N_0})^{\frac12} 2^{-\delta_{1}k} 2^{-7\delta_{1}m} + C(1+C^2_*) 2^{-\delta_{1}k} 2^{-7\delta_{1}m} \times (2^{\delta_{1}}-1)C_* E^{-1}(0)
		\\
		\leq  & C(1+E^3(0))(T^*_{N_0})^{\frac12} 2^{-\delta_{1}k} 2^{-7\delta_{1}m} + C(1+E^2(0)) 2^{-\delta_{1}k} 2^{-6\delta_{1}m} \times (2^{\delta_{1}}-1).
	\end{split}	
\end{equation}
For $k \geq m\geq N_0+1$, using \eqref{pp8},  \eqref{kz29} yields
\begin{equation}\label{kz31}
	\begin{split}
		\| P_{k} d h_m, P_{k} d \bu_m \|_{L^1_{[0,T^*_{m-1}]}L^\infty_x}
		\leq & C(1+E^3(0))(T^*_{N_0})^{\frac12} 2^{-\delta_{1}k} 2^{-5\delta_{1}m} \times 2^{\delta_{1}},
	\end{split}	
\end{equation}
Similarly, if $m\geq N_0+1$, using \eqref{pp8}, then the following estimate holds for $k<j$
\begin{equation}\label{kz34}
	\begin{split}
		& \|P_k (d{h}_{m}-d{h}_{m-1}), P_k (d{\bu}_{m}-d{\bu}_{m-1}) \|_{L^1_{[0,T^*_{m-1}]} L^\infty_x}
		\\
		\leq    &  \|P_k (d{h}_{m}-d{h}_{m-1}), P_k  (d{\bu}_{m}-d{\bu}_{m-1}) \|_{L^1_{[0,T^*_{m}]} L^\infty_x}
		\\
		& +\|P_k (d {h}_{m}-d {h}_{m-1}), P_k (d {\bu}_{m}-d {\bu}_{m-1}) \|_{L^1_{[T^*_m,T^*_{m-1}]} L^\infty_x}
		\\
		\leq & C  (1+E(0)^3)(T^*_{N_0})^{\frac12} 2^{-\delta_{1}k} 2^{-6\delta_1(m-1)}
		\\
		& + C  (1+C^3_*) 2^{-\delta_{1}k} 2^{-6\delta_1(m-1)} \times (2^{\delta_{1}}-1)C_* E^{-1}(0)
		\\
		\leq & C(1+E^3(0))(T^*_{N_0})^{\frac12} 2^{-\delta_{1}k} 2^{-5\delta_{1}(m-1)} \times 2^{\delta_{1}}.
	\end{split}	
\end{equation}
\quad \textbf{Step 2: Extending time interval $[0,T^*_j]$ to $[0,T^*_{N_0}]$}. Based the above analysis in Step 1, we can give a induction by achieving the goal. We assume the solutions $(h_j, \bu_j, \bw_j)$ can be extended from $[0,T^*_j]$ to $[0,T^*_{j-l}]$ through a maximam $X_l$ times and
\begin{equation}\label{kz41}
	\begin{split}
		X_l=& \frac{T_{j-l}^*- T_{j}^*}{C^{-1}_* 2^{-\delta_{1} j}}
		\\
		=& \frac{E(0)^{-1}( 2^{-\delta_{1}(j-l)} - 2^{-\delta_{1}j} )}{C^{-1}_* 2^{-\delta_{1} j}}
		\\
		=& \frac{C_*}{E(0)} (2^{\delta_{1}l}-1).
	\end{split}	
\end{equation}
Moreover, the following bounds
\begin{equation}\label{kz42}
	\begin{split}
		\| P_{k} d h_j, P_{k} d \bu_j \|_{L^1_{[0,T^*_{j-l}]}L^\infty_x}
		\leq & C(1+E^3(0))(T^*_{N_0})^{\frac12} 2^{-\delta_{1}k} 2^{-5\delta_{1}j} \times 2^{\delta_{1}l}, \qquad k \geq j,
	\end{split}	
\end{equation}
and
\begin{equation}\label{kz43}
	\begin{split}
		& \|P_k (d{h}_{m+1}-d{h}_{m}), P_k (d{\bu}_{m+1}-d{\bu}_{m}) \|_{L^1_{[0,T^*_{m-l}]} L^\infty_x}
		\\
		\leq  & C(1+E^3(0))(T^*_{N_0})^{\frac12} 2^{-\delta_{1}k} 2^{-5\delta_{1}m } \times 2^{\delta_{1}l}, \quad k<j,
	\end{split}	
\end{equation}
and
\begin{equation}\label{kz44}
	\|d h_j, d \bu_j\|_{L^1_{[0,T^*_{j-l}]} L^\infty_x}
	\leq  	C(1+E^3(0))(T^*_{N_0})^{\frac12}[\frac13(1-2^{-\delta_{1}})]^{-2}.
\end{equation}
and
\begin{equation}\label{kz45}
	|h_j,u^0_j-1,\mathring{\bu}_j| \leq 2+C_0, \quad u^0_j \geq 1, \quad  E(T^*_{j-l}) \leq C_*.
\end{equation}
In the following, we will check these estimates \eqref{kz41}-\eqref{kz45} hold when $l=1$, and then also holds when we replace $l$ by $l+1$.

Using \eqref{kz27}, \eqref{kz31}, \eqref{times1}, and \eqref{kz34}, then \eqref{kz42}-\eqref{kz45} hold by taking $l=1$. Let us now check it for $l+1$. In this case, it implies that  $T^*_{j-l} \leq T^*_{N_0}$. Therefore, $j-l\geq N_0+1$ should hold. Starting at the time $T^*_{j-l}$, seeing \eqref{DTJ} and \eqref{kz45}, we shall get an extending time-interval of $(h_j,\bu_j,\bw_j)$ with a length of $(C_*)^{-1}2^{-\delta_{1} j}$. So we can go to the case 2 in step 1, and the length every new time-interval is $C_*^{-1} 2^{-\delta_{1} j}$. So the times is
\begin{equation}\label{kz46}
	X= \frac{T^*_{j-(l+1)}-T^*_{j-l}}{C^{-1}_*2^{-\delta_{1} j}}= 2^{\delta_{1} l}(2^{\delta_{1}}-1)C_* E^{-1}(0).
\end{equation}
Thus, we can deduce that
\begin{equation}\label{kz40}
	X_{l+1}=X_l+X= (2^{\delta_{1}(l+1)}-1)C_* E^{-1}(0).
\end{equation}
Moreover, for $k \geq j$, we have
\begin{equation}\label{kz48}
	\begin{split}
		\| P_{k} d h_j, P_{k} d\bu_j \|_{L^1_{[0, T^*_{j-(l+1)}]}L^\infty_x}\leq & \| P_{k} d h_j, P_{k} d\bu_j \|_{L^1_{[0,T^*_{j-l}]}L^\infty_x}
		\\
		& +	\| P_{k} d h_j, P_{k} d\bu_j \|_{L^1_{[T^*_{j-l}, T^*_{j-(l+1)}]}L^\infty_x}.
	\end{split}
\end{equation}
Using \eqref{kz01} and \eqref{kz45}
\begin{equation}\label{kz49}
	\begin{split}
		& \| P_{k} d h_j, P_{k} d\bu_j \|_{L^1_{[T^*_{j-l}, T^*_{j-(l+1)}]}L^\infty_x}
		\\
		\leq   & C  (1+C^3_*)(T^*_{N_0})^{\frac12} 2^{-\delta_{1}k} 2^{-7\delta_{1}j}\times 2^{\delta_{1} l}(2^{\delta_{1}}-1)C_* E^{-1}(0), \quad k\geq j.
	\end{split}
\end{equation}
Due to \eqref{pp8}, we can see
\begin{equation*}
	\begin{split}
		(1+C^3_*) C_* E^{-1}(0)(1+E^3(0))^{-1}2^{-\delta_{1} N_0} \leq 1.
	\end{split}
\end{equation*}
Hence, from \eqref{kz49} we have
\begin{equation}\label{kz50}
	\begin{split}
		& \| P_{k} d h_j, P_{k} d\bu_j \|_{L^1_{[T^*_{j-l}, T^*_{j-(l+1)}]}L^\infty_x}
		\\
		\leq   & C  (1+E^3(0)) (T^*_{N_0})^{\frac12} 2^{-\delta_{1}k} 2^{-5\delta_{1}j}\times 2^{\delta_{1} l}(2^{\delta_{1}}-1), \quad k\geq j.
	\end{split}
\end{equation}
Using \eqref{kz42} and \eqref{kz50}, so we get
\begin{equation}\label{kz51}
	\begin{split}
		\| P_{k} d h_j, P_{k} d \bu_j \|_{L^1_{[0,T^*_{j-(l+1)}]}L^\infty_x}
		\leq & C(1+E^3(0)) (T^*_{N_0})^{\frac12} 2^{-\delta_{1}k} 2^{-5\delta_{1}j} \times 2^{\delta_{1}(l+1)}, \qquad k \geq j.
	\end{split}	
\end{equation}
If $k<j$, then we have
\begin{equation}\label{kz52}
	\begin{split}
		& \|P_k (d{h}_{m+1}-d{h}_{m}), P_k (d{\bu}_{m+1}-d{\bu}_{m}) \|_{L^1_{[0,T^*_{m-(l+1)}]} L^\infty_x}
		\\
		\leq  & \|P_k (d{h}_{m+1}-d{h}_{m}), P_k (d{\bu}_{m+1}-d{\bu}_{m}) \|_{L^1_{[0,T^*_{m-l}]} L^\infty_x}
		\\
		& + \|P_k (d{h}_{m+1}-d{h}_{m}), P_k (d{\bu}_{m+1}-d{\bu}_{m}) \|_{L^1_{[T^*_{m-l},T^*_{m-(l+1)}]} L^\infty_x}.
	\end{split}	
\end{equation}
When we extend the solutions $(h_j, \bu_j, \bw_j)$ from $[0,T^*_{j-l}]$ to $[0,T^*_{j-(l+1)}]$, then the solutions $(h_m, \bu_m, \bw_m)$ is also extended from $[0,T^*_{m-l}]$ to $[0,T^*_{m-(l+1)}]$. Seeing \eqref{kz02} and \eqref{kz45}, we can obtain
\begin{equation}\label{kz53}
	\begin{split}
		& \|P_k (d{h}_{m+1}-d{h}_{m}), P_k (d{\bu}_{m+1}-d{\bu}_{m}) \|_{L^1_{[T^*_{m-l},T^*_{m-(l+1)}]} L^\infty_x}
		\\
		\leq  & C(1+C_*^3)(T^*_{N_0})^{\frac12} 2^{-\delta_{1}k} 2^{-6\delta_{1}m } \times 2^{\delta_{1} l}(2^{\delta_{1}}-1)C_* E^{-1}(0).
	\end{split}	
\end{equation}
Hence, for $m-l \geq N_0+1$, using \eqref{pp8}, it yields
\begin{equation*}
	(1+C_*^3)  C_* E^{-1}(0) (1+E^3(0))^{-1} 2^{-\delta_{1}N_0 } \leq 1.
\end{equation*}
So \eqref{kz53} becomes
\begin{equation}\label{kz54}
	\begin{split}
		& \|P_k (d{h}_{m+1}-d{h}_{m}), P_k (d{\bu}_{m+1}-d{\bu}_{m}) \|_{L^1_{[T^*_{m-l},T^*_{m-(l+1)}]} L^\infty_x}
		\\
		\leq  & C(1+E^3(0))(T^*_{N_0})^{\frac12} 2^{-\delta_{1}k} 2^{-5\delta_{1}m } \times 2^{\delta_{1} l}(2^{\delta_{1}}-1).
	\end{split}	
\end{equation}
Inserting \eqref{kz43} and \eqref{kz54} to \eqref{kz52}, we have
\begin{equation}\label{kz55}
	\begin{split}
		& \|P_k (d{h}_{m+1}-d{h}_{m}), P_k (d{\bu}_{m+1}-d{\bu}_{m}) \|_{L^1_{[0,T^*_{m-(l+1)}]} L^\infty_x}
		\\
		\leq  & C(1+E^3(0))(T^*_{N_0})^{\frac12} 2^{-\delta_{1}k} 2^{-5\delta_{1}m } \times 2^{\delta_{1}(l+1)}, \quad k<j,
	\end{split}	
\end{equation}
Let us now bound the following Strichartz estimate
\begin{equation}\label{kz56}
	\begin{split}
		& \|d h_j, d \bu_j\|_{L^1_{[0,T^*_{j-(l+1)}]} L^\infty_x}
		\\
		\leq & \|P_{\geq j}dh_j, P_{\geq j}d \bu_j\|_{L^1_{[0,T^*_{j-(l+1)}]} L^\infty_x}
		\\
		& + \textstyle{\sum}^{j-1}_{k=j-l}\textstyle{\sum}^{j-1}_{m=k} \|P_k  (dh_{m+1}-dh_{m}), P_k  (d\bu_{m+1}-d\bu_{m})\|_{L^1_{[0,T^*_{j-(l+1)}]} L^\infty_x}
		\\
		& + \textstyle{\sum}^{j-1}_{k=1} \|P_k dh_k, P_k d\bu_k\|_{L^1_{[0,T^*_{j-(l+1)}]} L^\infty_x}
		\\
		=& 	\|P_{\geq j}dh_j, P_{\geq j}d \bu_j\|_{L^1_{[0,T^*_{j-(l+1)}]} L^\infty_x}
		\\
		& + \textstyle{\sum}^{j-1}_{k=j-l} \|P_k dh_k, P_k d\bu_k\|_{L^1_{[0,T^*_{j-(l+1)}]} L^\infty_x}
		+ \textstyle{\sum}^{j-(l+1)}_{k=1} \|P_k dh_k, P_k d\bu_k\|_{L^1_{[0,T^*_{j-(l+1)}]} L^\infty_x}
		\\
		& + \textstyle{\sum}^{j-1}_{k=1} \textstyle{\sum}^{j-1}_{m=j-(l+1)} \|P_k  (dh_{m+1}-dh_{m}), P_k  (d\bu_{m+1}-d\bu_{m})\|_{L^1_{[0,T^*_{j-(l+1)}]} L^\infty_x}
		\\
		& + \textstyle{\sum}^{j-(l+2)}_{k=1} \textstyle{\sum}_{m=k}^{j-(l+2)}  \|P_k  (dh_{m+1}-dh_m), P_k  (d\bu_{m+1}-d\bu_m)\|_{L^1_{[0,T^*_{j-(l+1)}]} L^\infty_x}
		\\
		=& \Theta_1+  \Theta_2+ \Theta_3+ \Theta_4+ \Theta_5,
	\end{split}
\end{equation}
where
\begin{equation}\label{Theta}
	\begin{split}
		\Theta_1= &  \|P_{\geq j}dh_j, P_{\geq j}d \bu_j\|_{L^1_{[0,T^*_{j-(l+1)}]} L^\infty_x} ,
		\\
		\Theta_2= & \textstyle{\sum}^{j-(l+2)}_{k=1}\textstyle{\sum}^{j-(l+2)}_{m=k} \|P_k  (dh_{m+1}-dh_{m}), P_k  (d\bu_{m+1}-d\bu_{m})\|_{L^1_{[0,T^*_{j-(l+1)}]} L^\infty_x},
		\\
		\Theta_3= & \textstyle{\sum}^{j-1}_{k=1}\textstyle{\sum}^{j-1}_{m=j-(l+1)} \|P_k  (dh_{m+1}-dh_{m}), P_k  (d\bu_{m+1}-d\bu_{m})\|_{L^1_{[0,T^*_{j-(l+1)}]} L^\infty_x},
		\\
		\Theta_4 =& \textstyle{\sum}^{j-(l+1)}_{k=1}	\|P_{k}dh_k, P_{k}d\bu_k\|_{L^1_{[0,T^*_{j-(l+1)}]} L^\infty_x},
		\\
		\Theta_5 =& \textstyle{\sum}^{j-1}_{k=j-l}	\|P_{k}dh_k, P_{k}d\bu_k\|_{L^1_{[0,T^*_{j-(l+1)}]} L^\infty_x}.
	\end{split}
\end{equation}
To get the estimate on time-interval $[0,T^*_{j-(l+1)}]$ for $\Theta_1, \Theta_2, \Theta_3, \Theta_4$ and $\Theta_5$, we should note that
there is no growth for $\Theta_2$ and $\Theta_4$ in this extending process. For example, considering $\Theta_2$, the existing time-interval of $P_k  (d\bu_{m+1}-d\bu_{m})$ is actually $[0,T^*_{m+1}]$, and $[0,T^*_{j-(l+1)}] \subseteq [0,T^*_{m+1}]$ if $m \geq j-(l+2)$. So we can use the initial bounds \eqref{kz01} and \eqref{kz02} to handle $\Theta_2$ and $\Theta_4$. While, considering $\Theta_1, \Theta_3$, and $\Theta_5$, we need to calculate the growth in the Strichartz estimate. Based on this idea, let us give a precise analysis on \eqref{Theta}.

According to \eqref{kz51}, we can estimate $\Theta_1$ as
\begin{equation}\label{kz57}
	\begin{split}
		\Theta_1 \leq & C(1+E^3(0))(T^*_{N_0})^{\frac12} \textstyle{\sum}^{\infty}_{k=j} 2^{-\delta_{1}k} 2^{-5\delta_{1}j} \times 2^{\delta_{1}(l+1)}.
	\end{split}
\end{equation}
Due to \eqref{kz01}, we have
\begin{equation}\label{kz58}
	\begin{split}
		\Theta_2\leq &	\textstyle{\sum}^{j-(l+2)}_{k=1}\textstyle{\sum}^{j-(l+2)}_{m=k} \|P_k  (dh_{m+1}-dh_{m}), P_k  (d\bu_{m+1}-d\bu_{m})\|_{L^1_{[0,T^*_{m+1}]} L^\infty_x},
		\\
		\leq &  C(1+E^3(0))(T^*_{N_0})^{\frac12} \textstyle{\sum}^{j-(l+2)}_{k=1}\textstyle{\sum}^{j-(l+2)}_{m=k} 2^{-\delta_{1}k} 2^{-6\delta_{1}m}.
	\end{split}	
\end{equation}
For $1 \leq k \leq j-1$ and $m \leq j-1$, using \eqref{kz55}, it follows
\begin{equation}\label{kz59}
	\begin{split}
		\Theta_3\leq & \textstyle{\sum}^{j-1}_{k=1}\textstyle{\sum}^{j-1}_{m=j-(l+1)} \|P_k  (dh_{m+1}-dh_{m}), P_k  (d\bu_{m+1}-d\bu_{m})\|_{L^1_{[0,T^*_{m-(l+1)}]} L^\infty_x}
		\\
		\leq &  C  (1+E^3(0))(T^*_{N_0})^{\frac12} \textstyle{\sum}^{j-1}_{k=1}\textstyle{\sum}^{j-1}_{m=j-(l+1)}  2^{-\delta_{1}k} 2^{-5\delta_{1}m} \times 2^{\delta_{1}(l+1)}.
	\end{split}
\end{equation}
Due to \eqref{kz01}, we can see
\begin{equation}\label{kz60}
	\begin{split}
		\Theta_4 \leq & \textstyle{\sum}^{j-(l+1)}_{k=1}	\|P_{k}dh_k, P_{k}d\bu_k\|_{L^1_{[0,T^*_{k}]} L^\infty_x}
		\\
		\leq & C  (1+E^3(0))(T^*_{N_0})^{\frac12} \textstyle{\sum}^{j-(l+1)}_{k=1} 2^{-\delta_{1}k} 2^{-7\delta_{1}k}.
	\end{split}
\end{equation}
If $j-l \leq k\leq j-1$, then $k+l+1-j \leq l$. By using \eqref{kz42}, we can estimate
\begin{equation}\label{kz61}
	\begin{split}
		\Theta_5 =& \textstyle{\sum}^{j-1}_{k=j-l}	\|P_{k}dh_k, P_{k}d\bu_k\|_{L^1_{[0,T^*_{j-(l+1)}]} L^\infty_x}
		\\
		= & \textstyle{\sum}^{j-1}_{k=j-l}	\|P_{k}dh_k, P_{k}d\bu_k\|_{L^1_{[0,T^*_{k-(k+l+1-j)}]} L^\infty_x}
		\\
		\leq	& C  (1+E^3(0))(T^*_{N_0})^{\frac12} \textstyle{\sum}^{j-1}_{k=j-l} 2^{-\delta_{1}k} 2^{-5\delta_{1}j}2^{\delta_{1}(k+l+1-j)}.
	\end{split}
\end{equation}
Inserting \eqref{kz57}-\eqref{kz61} to \eqref{kz56}, it follows
\begin{equation}\label{kz62}
	\begin{split}
		& \|d h_j, d \bu_j\|_{L^1_{[0,T^*_{j-(l+1)}]} L^\infty_x}
		\\
		\leq & C(1+E^3(0)) (T^*_{N_0})^{\frac12} \textstyle{\sum}^{\infty}_{k=j} 2^{-\delta_{1}k} 2^{-5\delta_{1}j} \times 2^{\delta_{1}(l+1)}
		\\
		& + C(1+E^3(0))(T^*_{N_0})^{\frac12} \textstyle{\sum}^{j-(l+2)}_{k=1}\textstyle{\sum}^{j-(l+2)}_{m=k} 2^{-\delta_{1}k} 2^{-6\delta_{1}m}
		\\
		&+C  (1+E^3(0))(T^*_{N_0})^{\frac12} \textstyle{\sum}^{j-1}_{k=1}\textstyle{\sum}^{j-1}_{m=j-(l+1)}  2^{-\delta_{1}k} 2^{-5\delta_{1}m} \times 2^{\delta_{1}(l+1)}
		\\
		& + C  (1+E^3(0))(T^*_{N_0})^{\frac12} \textstyle{\sum}^{j-(l+1)}_{k=1} 2^{-\delta_{1}k} 2^{-7\delta_{1}k}
		\\
		& + C  (1+E^3(0))(T^*_{N_0})^{\frac12} \textstyle{\sum}^{j-1}_{k=j-l} 2^{-\delta_{1}k} 2^{-5\delta_{1}k}2^{\delta_{1}(k+l+1-j)}
	\end{split}
\end{equation}
In the case of $j-l \geq N_0+1$ and $j\geq N_0+1$, \eqref{kz62} yields
\begin{equation}\label{kz63}
	\begin{split}
		& \|d h_j, d \bu_j\|_{L^1_{[0,T^*_{j-(l+1)}]} L^\infty_x}
		\\
		\leq & C(1+E^3(0))(T^*_{N_0})^{\frac12}(1-2^{-\delta_{1}})^{-2} \big\{
		2^{-\delta_{1}j} 2^{\delta_{1}(l+1)}+ 2^{-6\delta_{1}}
		\\
		& \quad +2^{-5\delta_{1}[j-(l+1)]}  2^{\delta_{1}(l+1)}+ 2^{-6\delta_{1}}+ 2^{-5\delta_{1}(j-l)}2^{\delta_{1}(l+1-j)}
		\big\}
		\\
		\leq & C(1+E^3(0))(T^*_{N_0})^{\frac12}(1-2^{-\delta_{1}})^{-2} \left\{
		2^{-6\delta_{1}N_0} + 2^{-6\delta_{1}} +	2^{-5\delta_{1}N_0} + 2^{-6\delta_{1}}+ 	2^{-6\delta_{1}N_0} 	\right\}
		\\
		\leq & C(1+E^3(0))(T^*_{N_0})^{\frac12}[\frac13(1-2^{-\delta_{1}})]^{-2}.
	\end{split}
\end{equation}
By using \eqref{kz63}, \eqref{pu00}, \eqref{pp8}, and Theorem \ref{VEt}, we have proved
\begin{equation*}
	\begin{split}
		|h_j,u^0_j-1,\mathring{\bu}_j|\leq 2+C_0, \quad  E(T^*_{j-(l+1)}) \leq C_*.
	\end{split}
\end{equation*}
For $\bu_j$ always satisfies \eqref{muu}, so we have
\begin{equation*}\label{kz64}
	\begin{split}
		|h_j,u^0_j-1,\mathring{\bu}_j|\leq 2+C_0, \quad  u^0_j \geq 1, \quad E(T^*_{j-(l+1)}) \leq C_*.
	\end{split}
\end{equation*}
Gathering \eqref{kz40}, \eqref{kz51}, \eqref{kz55}, \eqref{kz63} and \eqref{kz64}, we know that \eqref{kz41}-\eqref{kz45} holding for $l+1$. So our induction hold \eqref{kz41}-\eqref{kz45} for $l=1$ to $l=j-N_0$. Therefore, we can extend the solutions $(h_j,\bu_j,\bw_j)$ from $[0,T^*_j]$ to $[0,T^*_{N_0}]$ when $j \geq N_0$. Let us denote
\begin{equation}\label{Tstar}
	T^*=T^*_{N_0}=[E(0)]^{-1}2^{-\delta_{1} N_0}.
\end{equation}
Setting $l=j-N_0$ in \eqref{kz44}-\eqref{kz45}, we therefore get
\begin{equation}\label{kz65}
	\begin{split}
		& E(T^*) \leq  C_*, \quad \|h_j,u^0_j-1,\mathring{\bu}_j\|_{L^\infty_{[0,T^*]} L^\infty_x} \leq 2+C_0, \quad u^0_j \geq 1,
		\\
		& \|d h_j, d \bu_j\|_{L^1_{[0,T^*]} L^\infty_x}
		\leq  C(1+E^3(0))(T^*_{N_0})^{\frac12}[\frac13(1-2^{-\delta_{1}})]^{-2},
	\end{split}
\end{equation}
where $N_0$ and $C_*$ depends on $C_0, c_0, s, M_*$. It is stated in \eqref{pp8} and \eqref{Cstar}. In this process, we can also conclude
\begin{equation}\label{kz66}
	\begin{split}
		\|d h_j, d \bu_j\|_{L^2_{[0,T^*]} L^\infty_x}
		\leq & C(1+E^3(0))[\frac13(1-2^{-\delta_{1}})]^{-2}.
	\end{split}
\end{equation}
\subsubsection{Strichartz estimates of linear wave equation on time-interval $[0,T^*_{N_0}]$.}\label{finalq}
We still expect the behaviour of a linear wave equation endowed with $g_j=g(h_j,\bu_j)$. So we claim a theorem as follows
\begin{proposition}\label{rut}
	For $\frac{s}{2} \leq r \leq 3$, there is a solution $f_j$ on $[0,T^*_{N_0}]\times \mathbb{R}^3$ satisfying the following linear wave equation
	\begin{equation}\label{ru01}
		\begin{cases}
			\square_{{g}_j} f_j=0,
			\\
			(f_j,\partial_t f_j)|_{t=0}=(f_{0j},f_{1j}),
		\end{cases}
	\end{equation}
	where $(f_{0j},f_{1j})=(P_{\leq j}f_0,P_{\leq j}f_1)$ and $(f_0,f_1)\in H_x^r \times H^{r-1}_x$. Moreover, for $a\leq r-\frac{s}{2}$, we have
	\begin{equation}\label{ru02}
		\begin{split}
			&\|\left< \nabla \right>^{a-1} d{f}_j\|_{L^2_{[0,T^*_{N_0}]} L^\infty_x}
			\leq  C_{M_*}(\|{f}_0\|_{{H}_x^r}+ \|{f}_1 \|_{{H}_x^{r-1}}),
			\\
			&\|{f}_j\|_{L^\infty_{[0,T^*_{N_0}]} H^{r}_x}+ \|\partial_t {f}_j\|_{L^\infty_{[0,T^*_{N_0}]} H^{r-1}_x} \leq  C_{M_*}(\| {f}_0\|_{H_x^r}+ \| {f}_1\|_{H_x^{r-1}}).
		\end{split}
	\end{equation}
\end{proposition}
\begin{proof}[Proof of Proposition \ref{rut}.] Our proof also relies a Strichartz estimates on a short time-interval. Then a loss of derivatives are then obtained by summing up the short time estimates with respect to these time intervals.

	Using \eqref{ru01}, \eqref{ru03} and \eqref{ru04}, we have
	\begin{equation}\label{ru06}
		\begin{split}
			\| \left< \nabla \right>^{a-1+\frac92\delta_{1}} df_j\|_{L^2_{[0,T^*_j]} L^\infty_x} \leq & C( \| f_{0j} \|_{H^r_x} + \| f_{1j} \|_{H^{r-1}_x}  )
			\\
			\leq & C ( \| f_{0} \|_{H^r_x} + \| f_{1} \|_{H^{r-1}_x}  ),
		\end{split}
	\end{equation}
	where we use
	\begin{equation}\label{ru060}
		a-1+\frac92\delta_{1} \leq r-\frac{s}{2}+\frac92\delta_{1} <r-1.
	\end{equation}
	By Bernstein inequality, for $k\geq j$, we shall obtain
	\begin{equation}\label{ru07}
		\begin{split}
			\| \left< \nabla \right>^{a-1} P_k df_j\|_{L^2_{[0,T^*_j]} L^\infty_x}
			=& C2^{-\frac92\delta_{1} k}\| \left< \nabla \right>^{a-1+\frac92\delta_{1}} P_k df_j\|_{L^2_{[0,T^*_j]} L^\infty_x}
			\\
			\leq & C2^{-\frac12\delta_{1} k} 2^{-4\delta_{1} j} \| \left< \nabla \right>^{a-1+\frac92\delta_{1}} df_j\|_{L^2_{[0,T^*_j]} L^\infty_x}.
		\end{split}
	\end{equation}
	Combining \eqref{ru06} and \eqref{ru07}, for $a \leq r-\frac{s}{2}$, we get
	\begin{equation}\label{ru070}
		\begin{split}
			\| \left< \nabla \right>^{a-1} P_k df_j\|_{L^2_{[0,T^*_j]} L^\infty_x}
			\leq & C2^{-\frac12\delta_{1} k} 2^{-4\delta_{1} j} ( \| f_{0} \|_{H^r_x} + \| f_{1} \|_{H^{r-1}_x}  ), \quad k\geq j.
		\end{split}
	\end{equation}
	On the other hand, for any integer $m\geq 1$, we also have
	\begin{equation*}
		\begin{cases}
			\square_{{g}_m} f_m=0, \quad [0,T^*_{m}]\times \mathbb{R}^3,
			\\
			(f_m,\partial_t f_m)|_{t=0}=(f_{0m},f_{1m}),
		\end{cases}
	\end{equation*}
	and
	\begin{equation*}
		\begin{cases}
			\square_{{g}_{m+1}} f_{m+1}=0, \quad [0,T^*_{m+1}]\times \mathbb{R}^3,
			\\
			(f_{m+1},\partial_t f_{m+1})|_{t=0}=(f_{0(m+1)},f_{1(m+1)}).
		\end{cases}
	\end{equation*}
	So the difference term $f_{m+1}-f_m$ satisfies
	\begin{equation}\label{ru08}
		\begin{cases}
			\square_{{g}_{m+1}} (f_{m+1}-f_m)=({g}^{\alpha i}_{m+1}-{g}^{\alpha i}_{m} )\partial_{\alpha i} f_m, \quad [0,T^*_{m+1}]\times \mathbb{R}^3,
			\\
			(f_{m+1}-f_m,\partial_t (f_{m+1}-f_m))|_{t=0}=(f_{0(m+1)}-f_{0m},f_{1(m+1)}-f_{1m}).
		\end{cases}
	\end{equation}
	By Duhamel's principle, \eqref{ru04} and \eqref{ru060}, it yields
	\begin{equation}\label{ru09}
		\begin{split}
			& \| \left< \nabla \right>^{a-1} P_k (d f_{m+1}-df_m)\|_{L^2_{[0,T^*_{m+1}]} L^\infty_x}
			\\
			\leq & C\| f_{0(m+1)}-f_{0m}\|_{H_x^{r-\frac92\delta_{1}}} + C\|f_{1(m+1)}-f_{1m}\|_{H_x^{r-1-\frac92\delta_{1}}}
			\\
			& \ + C\|({g}_{m+1}-{g}_{m}) \cdot\nabla d f_m\|_{L^1_{[0,T^*_{m+1}]} H^{r}_x}
			\\
			\leq & C2^{-\frac92\delta_{1} m}( \| f_{0} \|_{H^r_x} + \| f_{1} \|_{H^{r-1}_x}  ) + C \| {g}_{m+1}-{g}_{m}\|_{L^1_{[0,T^*_{m+1}]} L^\infty_x} \| \nabla d f_m\|_{L^\infty_{[0,T^*_{m+1}]} H^{r-1}_x} .
		\end{split}
	\end{equation}
	By using the energy estimates
	\begin{equation}\label{ru10}
		\begin{split}
			\| \nabla d f_m\|_{L^\infty_{[0,T^*_{m+1}]} H^{r-1}_x} \leq & C\| d f_m\|_{L^\infty_{[0,T_{m+1}]} H^{r}_x}
			\\
			\leq & C(\|f_{0m}\|_{H^{r+1}_x}+\|f_{1m}\|_{H^{r}_x})
			\\
			\leq & C2^m (\|f_{0m}\|_{H^r_x}+\|f_{1m}\|_{H^{r-1}_x}) .
		\end{split}
	\end{equation}
	By using Strichartz estimates \eqref{ru04} and Lemma \ref{LD} again, we shall prove that
	\begin{equation}\label{ru11}
		\begin{split}
			& 2^m \| {g}_{m+1}-{g}_{m}\|_{L^1_{[0,T^*_{m+1}]} L^\infty_x}
			\\
			\leq & 2^m (T^*_{m+1})^{\frac12}\| {\bv}_{m+1}-{\bv}_{m},  {\rho}_{m+1}-{\rho}_{m}\|_{L^2_{[0,T^*_{m+1}]} L^\infty_x}
			\\
			\leq & C2^m ( \| {\bv}_{m+1}-{\bv}_{m},  {\rho}_{m+1}-{\rho}_{m}\|_{L^\infty_{[0,T^*_{m+1}]} H^{1+\delta_{1}}_x} + \| {\bw}_{m+1}-{\bw}_{m}\|_{L^2_{[0,T^*_{m+1}]} H^{\frac12+\delta_{1}}_x})
			\\
			\leq & C2^{-7\delta_{1}m } (1+E^3(0)).
		\end{split}
	\end{equation}
	Above, we use \eqref{yu0}, \eqref{yu1}, and \eqref{DTJ}.

	Due to \eqref{ru10} and \eqref{ru11}, for $k<j$, and $k\leq m$,  \eqref{ru09} yields
	\begin{equation}\label{ru12}
		\begin{split}
			\| \left< \nabla \right>^{a-1} P_k (df_{m+1}-df_m)\|_{L^2_{[0,T^*_{m+1}]} L^\infty_x}
			\leq  C2^{-\frac12\delta_{1} k}2^{-4\delta_{1} m} \| (f_{0},f_1) \|_{H^r_x\times H^{r-1}_x}  (1+E^3(0)).
		\end{split}
	\end{equation}
	
	In the following part, let us consider the energy estimates. Taking the operator $\left< \nabla \right>^{r-1}$ on \eqref{ru01}, then we get
	\begin{equation}\label{ru15}
		\begin{cases}
			\square_{{g}_j} \left< \nabla \right>^{r-1}f_j=[ \square_{{g}_j} ,\left< \nabla \right>^{r-1}] f_j,
			\\
			(\left< \nabla \right>^{r-1}f_j,\partial_t \left< \nabla \right>^{r-1}f_j)|_{t=0}=(\left< \nabla \right>^{r-1}f_{0j},\left< \nabla \right>^{r-1}f_{1j}).
		\end{cases}
	\end{equation}
	To estimate $[ \square_{{g}_j} ,\left< \nabla \right>^{r-1}] f_j$, we will divide it into two cases $\frac{s}{2}<r\leq s$ and $s<r\leq 3$.

	\textit{Case 1: $\frac{s}{2}<r\leq s$.} For $\frac{s}{2}<r\leq s$, note that
	\begin{equation}\label{ru14}
		\begin{split}
			[ \square_{{g}_j} ,\left< \nabla \right>^{r-1}] f_j
			=& [{g}^{\alpha i}_j-\mathbf{m}^{\alpha i}, \left< \nabla \right>^{r-1} \partial_i ] \nabla_{\alpha}f_j + \left< \nabla \right>^{r-1}( \partial_i g_j \partial_{\alpha}f_j )
			\\
			= &[{g}_j-\mathbf{m}, \left< \nabla \right>^{r-1} \nabla] df_j + \left< \nabla \right>^{r-1}( \nabla g_j df).
		\end{split}
	\end{equation}
	By \eqref{ru14} and Kato-Ponce estimates, we have\footnote{If $r=s$, then $L^{\frac{3}{s-r}}_x=L^\infty_x$.}
	\begin{equation}\label{ru16}
		\| [ \square_{{g}_j} ,\left< \nabla \right>^{r-1}] f_j \|_{L^2_x} \leq C ( \|dg_j\|_{L^\infty_x} \|d f_j \|_{H^{r-1}_x} + \|\left< \nabla \right>^{r} (g_j-\mathbf{m}) \|_{L^{\frac{3}{\frac32-s+r}}_x} \| df_j \|_{L^{\frac{3}{s-r}}_x} )
	\end{equation}
	By Sobolev's inequality, it follows
	\begin{equation}\label{ru17}
		\|\left< \nabla \right>^{r} (g_j-\mathbf{m}) \|_{L^{\frac{3}{\frac32-s+r}}_x} \leq C\|g_j-\mathbf{m}\|_{H^s_x}.
	\end{equation}
	By Gagliardo-Nirenberg inequality and Young's inequality, we can get
	\begin{equation}\label{ru18}
		\begin{split}
			\| df_j \|_{L^{\frac{3}{s-r}}_x} \leq & C\|\left< \nabla \right>^{r-1} df_j \|^{\frac{s-2}{3-s}}_{L^{2}_x} \|\left< \nabla \right>^{r-\frac{s}{2}-1} df_j \|^{\frac{5-2s}{3-s}}_{L^{\infty}_x}
			\\
			\leq & C( \|df_j \|_{H^{r-1}_x}+ \|\left< \nabla \right>^{r-\frac{s}{2}-1} df_j \|_{L^{\infty}_x})
		\end{split}
	\end{equation}
	\textit{Case 2: $s<r\leq 3$.} For $s<r\leq 3$, using Kato-Ponce estimates, we have
	\begin{equation}\label{ru150}
		\begin{split}
			&  \| [ \square_{{g}_j} ,\left< \nabla \right>^{r-1}] f_j \|_{L^2_x}
			\\
			= & \| {g}^{\alpha i}_j-\mathbf{m}^{\alpha i},\left< \nabla \right>^{r-1}] \partial_{\alpha i}f_j \|_{L^2_x}
			\\
			\leq & C ( \|dg_j\|_{L^\infty_x} \|d f_j \|_{H^{r-1}_x} + \|\left< \nabla \right>^{r-1} (g_j-\mathbf{m}) \|_{L^{\frac{3}{\frac12-s+r}}_x} \|\nabla df_j \|_{L^{\frac{3}{1+s-r}}_x} ) .
		\end{split}
	\end{equation}
	By Sobolev's inequality, it follows
	\begin{equation}\label{ru151}
		\|\left< \nabla \right>^{r-1} (g_j-\mathbf{m}) \|_{L^{\frac{3}{\frac12-s+r}}_x} \leq C\|g_j-\mathbf{m}\|_{H^s_x}.
	\end{equation}
	By Gagliardo-Nirenberg inequality and Young's inequality, we can get
	\begin{equation}\label{ru152}
		\begin{split}
			\| \nabla df_j \|_{L^{\frac{3}{1+s-r}}_x} \leq & C\|\left< \nabla \right>^{r-1} df_j \|^{\frac{s-2}{3-s}}_{L^{2}_x} \|\left< \nabla \right>^{r-1-\frac{s}{2}} df_j \|^{\frac{5-2s}{3-s}}_{L^{\infty}_x}
			\\
			\leq & C( \|df_j \|_{H^{r-1}_x}+ \|\left< \nabla \right>^{r-\frac{s}{2}-1} df_j \|_{L^{\infty}_x})
		\end{split}
	\end{equation}
	For $\frac{s}{2}<r \leq 3$, gathering \eqref{ru16}-\eqref{ru18} with \eqref{ru150}- \eqref{ru152}, so we have
	\begin{equation}\label{ru19}
		\begin{split}
			\| [ \square_{{g}_j} ,\left< \nabla \right>^{r-1}] f_j \|_{L^2_x}\|d f_j \|_{H^{r-1}_x} \leq & C  \|dg_j\|_{L^\infty_x} \|d f_j \|^2_{H^{r-1}_x}
			+ C  \| g_j-\mathbf{m} \|_{H^s_x}\|d f_j \|^2_{H^{r-1}_x}
			\\
			& + C  \|g_j-\mathbf{m}\|_{H^s_x} \|\left< \nabla \right>^{r-\frac{s}{2}-1} df_j \|_{L^{\infty}_x} \|d f_j \|_{H^{r-1}_x}
			\\
			\leq & C  ( \|dg_j\|_{L^\infty_x}+ \|g_j-\mathbf{m}\|_{H^s_x}+ \|\left< \nabla \right>^{r-\frac{s}{2}-1} df_j \|_{L^{\infty}_x} ) \|d f_j \|^2_{H^{r-1}_x}
			\\
			&+ C  \|g_j-\mathbf{m}\|^2_{H^s_x}  \|\left< \nabla \right>^{r-\frac{s}{2}-1} df_j \|_{L^{\infty}_x} .
		\end{split}	
	\end{equation}
	By \eqref{ru15} and \eqref{ru19}, we get
	\begin{equation}\label{ru20}
		\begin{split}
			\frac{d}{dt}\|d f_j \|^2_{H^{r-1}_x}
			\leq & C  (\|dg_j\|_{L^\infty_x}+ \|g_j-\mathbf{m}\|_{H^s_x}+ \|\left< \nabla \right>^{r-\frac{s}{2}-1} df_j \|_{L^{\infty}_x} ) \|d f_j \|^2_{H^{r-1}_x}
			\\
			&+ C  \|g_j-\mathbf{m}\|^2_{H^s_x}  \|\left< \nabla \right>^{r-\frac{s}{2}-1} df_j \|_{L^{\infty}_x} .
		\end{split}	
	\end{equation}
	Using Gronwall's inequality, we get
	\begin{equation}\label{ru21}
		\begin{split}
			\|d f_j(t) \|^2_{H^{r-1}_x}
			\leq & C  (\|d f_j(0) \|^2_{H^{r-1}_x} + C \int^t_0 \|g_j-\mathbf{m}\|^2_{H^s_x}  \|\left< \nabla \right>^{r-\frac{s}{2}-1} df_j \|_{L^{\infty}_x}d\tau)
			\\
			& \qquad \cdot\exp\{\int^t_0  \|dg_j\|_{L^\infty_x}+ \|g_j-\mathbf{m}\|_{H^s_x}+ \|\left< \nabla \right>^{r-\frac{s}{2}-1} df_j \|_{L^{\infty}_x} d\tau \} .
		\end{split}	
	\end{equation}
	Note
	\begin{equation}\label{ru22}
		\begin{split}
			\|f_j(t) \|_{L^{2}_x}
			\leq & \| f_j(0) \|_{L^2_x} + \int^t_0 \| \partial_t f_j \|_{L^2_x} d\tau.
		\end{split}	
	\end{equation}
	By \eqref{ru21}, \eqref{ru22} and \eqref{kz65}, if $t\in[0,T^*_{N_0}]$, then it follows
	\begin{equation}\label{ru23}
		\begin{split}
			\| f_j(t) \|^2_{H^{r}_x} + \| d f_j(t) \|^2_{H^{r-1}_x}
			\leq & C  \textrm{e}^{C_*}(\|f_0 \|^2_{H^{r}_x}+\|f_1 \|^2_{H^{r-1}_x} + C_*)
			\\
			& \quad \cdot\exp\{\int^t_0   \|\left< \nabla \right>^{r-\frac{s}{2}-1} df_j \|_{L^{\infty}_x} d\tau \cdot \textrm{e}^{\int^t_0   \|\left< \nabla \right>^{r-\frac{s}{2}-1} df_j \|_{L^{\infty}_x} d\tau}   \} .
		\end{split}	
	\end{equation}
	Based on \eqref{ru070}, \eqref{ru12},  and \eqref{ru23}, following the extending method in subsecion \ref{esest}, for $a\leq r-\frac{s}{2}$ we shall obtain
	\begin{equation}\label{ru234}
		\begin{split}
			\| \left< \nabla \right>^{a-1} P_k df_j\|_{L^2_{[0,T^*_{N_0}]} L^\infty_x}
			\leq & C2^{-\frac12\delta_{1} k} 2^{-4\delta_{1} j} ( \| f_{0} \|_{H^r_x} + \| f_{1} \|_{H^{r-1}_x}  ) \times \{ 2^{\delta_{1} (j-N_0)} \}^{\frac12}
			\\
			\leq & C2^{-\frac12\delta_{1}N_0} 2^{-\frac12\delta_{1} k} 2^{-2\delta_{1} j} ( \| f_{0} \|_{H^r_x} + \| f_{1} \|_{H^{r-1}_x}  ), \quad k \geq j,
		\end{split}
	\end{equation}
	and
	\begin{equation}\label{ru25}
		\begin{split}
			& \| \left< \nabla \right>^{a-1} P_k (df_{m+1}-df_m)\|_{L^2_{[0,T^*_{N_0}]} L^\infty_x}
			\\
			\leq & C2^{-\frac12\delta_{1} k}2^{-3\delta_{1} m}( \| f_{0} \|_{H^r_x} + \| f_{1} \|_{H^{r-1}_x}  )(1+E^3(0)) \times \{ 2^{\delta_{1} (m+1-N_0)} \}^{\frac12}
			\\
			\leq & C2^{-\frac12\delta_{1} k}2^{-2\delta_{1} m}2^{-\frac12\delta_{1}N_0}( \| f_{0} \|_{H^r_x} + \| f_{1} \|_{H^{r-1}_x}  ) (1+E^3(0)), \quad k<j.
		\end{split}
	\end{equation}
	By phase decomposition, we have
	\begin{equation*}
		df_j= P_{\geq j} df_j+ \textstyle{\sum}^{j-1}_{k=1} \textstyle{\sum}^{j-1}_{m=k} P_k (df_{m+1}-df_m)+ \textstyle{\sum}^{j-1}_{k=1} P_k df_k.
	\end{equation*}
	By using \eqref{ru234} and \eqref{ru25}, we get
	\begin{equation}\label{ru230}
		\begin{split}
			\| \left< \nabla \right>^{a-1} df_{j} \|_{L^2_{[0,T^*_{N_0}]} L^\infty_x}
			\leq & C( \| f_{0} \|_{H^r_x} + \| f_{1} \|_{H^{r-1}_x}  )(1+E^3(0))[\frac13(1-2^{-\delta_{1}})]^{-2}.
		\end{split}
	\end{equation}
	Using \eqref{ru23} and \eqref{ru230}, we get
	\begin{equation}\label{ru26}
		\begin{split}
			\|f_j\|_{L^\infty_{[0,T^*_{N_0}]} H^{r}_x}+ \|d f_j\|_{L^\infty_{[0,T^*_{N_0}]} H^{r-1}_x }
			\leq & A_* \exp ( B_* \mathrm{e}^{B_*} ),
		\end{split}	
	\end{equation}
	where $A_{*}=C  \textrm{e}^{C_*}(\|f_0 \|_{H^{r}_x}+\|f_1 \|_{H^{r-1}_x} + C_*)$, $B_*=C(\| f_{0} \|_{H^r_x} + \| f_{1} \|_{H^{r-1}_x})(1+E^3(0))[\frac13(1-2^{-\delta_{1}})]^{-2}$ and $C_*$ is stated in \eqref{Cstar}. Therefore, the conclusion \eqref{ru02} hold.
\end{proof}

\section{Appendix}\label{Ap}
\subsection{Proofs for some lemmas}\label{appa}
We will prove Lemma \ref{WT} and Lemma \ref{tr0} as follows.
\begin{proof}[proof of Lemma \ref{WT}]
	Operate $u^\beta \partial_{\beta}$ on \eqref{REE}, we have
	\begin{equation}\label{a00}
			u^\beta \partial_{\beta}(u^\kappa \partial_{\kappa} h) + u^\beta \partial_{\beta}( c^2_s \partial_\kappa u^\kappa)=0,
	\end{equation}
	and
	\begin{equation}\label{a01}
		u^\beta \partial_{\beta}(u^\kappa \partial_{\kappa} u^\alpha)  + u^\beta \partial_{\beta}\{ (m^{\alpha \kappa}+u^{\alpha}u^{\kappa}) \partial_\kappa h  \}=0 .
	\end{equation}
We begin with the first equation \eqref{a00}. A direct calculation on \eqref{a00} tells us
\begin{equation}\label{a03}
	u^\beta u^\kappa \partial_{\beta \kappa} h  + \underbrace{c^2_s u^\beta \partial_{\beta \kappa} u^\kappa}_{\equiv R_1} + u^\beta \partial_{\beta}u^\kappa  \partial_{\kappa} h +  u^\beta \partial_{\kappa} u^\kappa \partial_{\beta}(c^2_s)=0.
\end{equation}
For $R_1$, we can compute
\begin{align}\label{a04}
	 R_1= & c^2_s \partial_\kappa  ( u^\beta \partial_{\beta}  u^\kappa ) -c^2_s \partial_\kappa  u^\beta \partial_{\beta}  u^\kappa  \nonumber
	 \\
	 =& -c^2_s \partial_\kappa \{ (m^{\beta \kappa}+u^\beta u^\kappa)\partial_{\beta}h  \}-c^2_s \partial_\kappa  u^\beta \partial_{\beta}  u^\kappa  \nonumber
	 \\
	 =& -c^2_s  m^{\beta \kappa} \partial_{\beta \kappa}h  - c^2_s u^\beta u^\kappa \partial_{\beta \kappa}h
	 -c^2_s  u^\kappa \partial_{\kappa} u^\beta \partial_\beta h \nonumber
	 \\
	 &- c^2_s  u^\beta \partial_{\kappa} u^\kappa \partial_\beta h
	 -c^2_s \partial_\kappa  u^\beta \partial_{\beta}  u^\kappa
\end{align}
Inserting \eqref{a04} to \eqref{a03}, it yields
\begin{equation}\label{a05}
\begin{split}
	 0=& -c^2_s  m^{\beta \kappa} \partial_{\beta \kappa}h + (1- c^2_s) u^\beta u^\kappa \partial_{\beta \kappa}h  + u^\beta \partial_{\beta}u^\kappa  \partial_{\kappa} h +  u^\beta \partial_{\kappa} u^\kappa \partial_{\beta}(c^2_s)
	\\
	& \ -c^2_s  u^\kappa \partial_{\kappa} u^\beta \partial_\beta h
	- c^2_s  u^\beta \partial_{\kappa} u^\kappa \partial_\beta h
	-c^2_s \partial_\kappa  u^\beta \partial_{\beta}  u^\kappa.
\end{split}
\end{equation}
Multiplying \eqref{a05} with $\Omega$ (defined in \eqref{AMd2}) and using $\square_g=g^{\alpha \beta} \partial^2_{ \alpha \beta}$, we can rewritten \eqref{a05} as
\begin{equation}\label{a06}
		\square_g h=\Omega \left\{ (1-c^2_s)u^\beta \partial_{\beta}u^\kappa  \partial_{\kappa} h +  u^\beta \partial_{\kappa} u^\kappa \partial_{\beta}(c^2_s)
		- c^2_s  u^\beta \partial_{\kappa} u^\kappa \partial_\beta h
		-c^2_s \partial_\kappa  u^\beta \partial_{\beta}  u^\kappa \right\}.
\end{equation}
So we have proved \eqref{err1}. We still need to prove \eqref{err}. So we next consider \eqref{a01}. By direct calculation, we obtain
\begin{equation}\label{a07}
	\begin{split}
	& u^\beta u^\kappa \partial_{\beta \kappa} u^\alpha + u^\beta \partial_{\beta} u^\kappa \partial_{\kappa} u^\alpha
 + u^\beta \partial_{\beta} (m^{\alpha \kappa}+u^{\alpha}u^{\kappa}) \partial_\kappa h
	  + \underbrace{u^\beta  (m^{\alpha \kappa}+ u^{\alpha}u^{\kappa}) \partial_{\beta \kappa} h }_{\equiv R_2} =0.
	\end{split}
\end{equation}
We continue to consider $R_2$. Inserting \eqref{REE}, it yields
\begin{equation}\label{a08}
	\begin{split}
		R_2= & (m^{\alpha \kappa}+ u^{\alpha}u^{\kappa}) \partial_{\kappa} ( u^\beta \partial_{\beta} h)-(m^{\alpha \kappa}+ u^{\alpha}u^{\kappa}) \partial_{\kappa} u^\beta  \partial_{\beta} h
		\\
		=& (m^{\alpha \kappa}+ u^{\alpha}u^{\kappa}) \partial_{\kappa} ( -c^2_s \partial_\beta u^\beta )-(m^{\alpha \kappa}+ u^{\alpha}u^{\kappa}) \partial_{\kappa} u^\beta  \partial_{\beta} h
		\\
		=& \underbrace{ -c^2_s m^{\alpha \kappa}\partial_{\kappa \beta} u^\beta }_{\equiv R_3}
		 \underbrace{ -c^2_s u^{\alpha}u^{\kappa}\partial_{\kappa \beta} u^\beta }_{\equiv R_4}
		-(m^{\alpha \kappa}+ u^{\alpha}u^{\kappa}) \partial_{ \beta} u^\beta \partial_{\kappa} (c^2_s)   -(m^{\alpha \kappa}+ u^{\alpha}u^{\kappa}) \partial_{\kappa} u^\beta  \partial_{\beta} h.
	\end{split}
\end{equation}
For $R_3$, using \eqref{OE00} and \eqref{MFd}, we have
\begin{equation}\label{a09}
	\begin{split}
		R_3=& -c^2_s m^{\alpha \kappa}\partial_{ \beta} ( \partial_\kappa u^\beta-\partial^\beta u_\kappa)-c^2_s m^{\alpha \kappa}\partial^{\beta}\partial_\beta u_\kappa
		\\
		=& -c^2_s m^{\alpha \kappa}\partial^{ \beta} ( \epsilon_{\kappa \beta \gamma \delta}\mathrm{e}^{-h}u^\gamma w^\delta-u_\beta \partial_\kappa h+ u_\kappa \partial_{ \beta} h )-c^2_s m^{\alpha \kappa}\partial^{\beta}\partial_\beta u_\kappa
		\\
		=& -c^2_s \epsilon^{\alpha \beta \gamma \delta}\mathrm{e}^{-h}u_\gamma \partial_{ \beta} w_\delta +  c^2_s u^\beta \partial_\beta\partial^\alpha h  -c^2_su^\alpha \partial^\beta\partial_\beta h -c^2_s \partial^{\beta}\partial_\beta u^\alpha
		\\
		& -c^2_s \epsilon^{\alpha \beta \gamma \delta}\mathrm{e}^{-h}u_\gamma  \partial_{ \beta} h w_\delta+c^2_s \epsilon^{\alpha \beta \gamma \delta}\mathrm{e}^{-h} \partial_{ \beta} u_\gamma  w_\delta+  c^2_s \partial_\beta u^\beta \partial^\alpha h -c^2_s \partial^\beta u^\alpha \partial_\beta h
		\\
		=
		& -c^2_s \epsilon^{\alpha \beta \gamma \delta}\mathrm{e}^{-h}\partial_\beta u_\gamma w_\delta    +  c^2_s \partial_\beta u^\beta \partial^\alpha h -c^2_s \partial^\beta u^\alpha \partial_\beta h
		\\
		& -c^2_s \mathrm{e}^{-h} W^\alpha + \underbrace{ c^2_s u^\beta \partial_\beta\partial^\alpha h }_{\equiv R_5} \underbrace{-c^2_su^\alpha \partial^\beta\partial_\beta h}_{\equiv R_7} -c^2_s \partial^{\beta}\partial_\beta u^\alpha.
	\end{split}
\end{equation}
It remains for us to calculate $R_4$ and $R_5$. For $R_4$, we compute
\begin{equation}\label{a10}
	\begin{split}
		R_4=& -c^2_s u^{\alpha}u^{\kappa}\partial_{\kappa \beta} u^\beta
		\\
		=& -c^2_s u^{\alpha} \partial_{\beta} ( u^{\kappa}\partial_{\kappa} u^\beta )+c^2_s u^{\alpha} \partial_{\beta}  u^{\kappa} \partial_{\kappa} u^\beta
		\\
		=& c^2_s u^{\alpha} \partial_{\beta} \{ (m^{\beta \kappa}+u^\beta u^\kappa )\partial_{\kappa} h  \} +c^2_s u^{\alpha} \partial_{\beta}  u^{\kappa} \partial_{\kappa} u^\beta
		\\
		=& \underbrace{ c^2_s u^{\alpha}  \partial^{\beta } \partial_{\beta} h }_{\equiv R_8}  + \underbrace{ c^2_s u^{\alpha} u^\beta u^\kappa \partial_{\beta \kappa} h  }_{\equiv R_6}
		+c^2_s u^{\alpha} u^\kappa \partial_{\beta} u^\beta  \partial_{\kappa} h
		 + c^2_s u^{\alpha}  u^\beta  \partial_{\beta} u^\kappa \partial_{\kappa} h
		 +c^2_s u^{\alpha} \partial_{\beta}  u^{\kappa} \partial_{\kappa} u^\beta.
	\end{split}
\end{equation}
Note that
\begin{equation}\label{a11a}
	R_7+R_8=0.
\end{equation}
We now turn to calculate $R_5$. Inserting \eqref{REE} to $R_5$, we get
\begin{equation}\label{a11}
	\begin{split}
		R_5=& c^2_s u^\beta \partial_\beta(-u^\kappa \partial_\kappa u^\alpha-u^\alpha u^\kappa \partial_\kappa h)
		\\
		=& -c^2_s u^\beta u^\kappa \partial_{\beta \kappa} u^\alpha - c^2_s u^\alpha u^\beta u^\kappa \partial_{\beta \kappa} h
		 -c^2_s u^\beta \partial_{\beta} u^\kappa \partial_{\kappa} u^\alpha
		 \\
		 & - c^2_s  u^\beta u^\kappa \partial_{\beta} u^\alpha \partial_{\kappa} h- c^2_s  u^\beta u^\alpha  \partial_{\beta} u^\kappa \partial_{\kappa} h.
	\end{split}
\end{equation}
From \eqref{a11}, we can see
\begin{equation}\label{a12}
	\begin{split}
		R_5+R_6
		=& -c^2_s u^\beta u^\kappa \partial_{\beta \kappa} u^\alpha
		-c^2_s u^\beta \partial_{\beta} u^\kappa \partial_{\kappa} u^\alpha
		\\
		& - c^2_s  u^\beta u^\kappa \partial_{\beta} u^\alpha \partial_{\kappa} h- c^2_s  u^\beta u^\alpha  \partial_{\beta} u^\kappa \partial_{\kappa} h.
	\end{split}
\end{equation}
Gathering \eqref{a07}-\eqref{a12}, we get
\begin{equation}\label{a13}
	\begin{split}
		c^2_s m^{\beta \gamma}\partial_{\beta \gamma} u^\alpha+(c^2_s-1) u^{\beta}u^{\gamma}\partial_{\beta \gamma} u^\alpha
		=& -c^2_s \mathrm{e}^{-h} W^\alpha+Q_*^\alpha,
	\end{split}
\end{equation}
where
\begin{equation}\label{a14}
	\begin{split}
			Q_*^\alpha=&u^\beta \partial_{\beta} u^\kappa \partial_{\kappa} u^\alpha
			+ u^\beta \partial_{\beta} (m^{\alpha \kappa}+u^{\alpha}u^{\kappa}) \partial_\kappa h-(m^{\alpha \kappa}+ u^{\alpha}u^{\kappa}) \partial_{ \beta} u^\beta \partial_{\kappa} (c^2_s)
			\\
			& -(m^{\alpha \kappa}+ u^{\alpha}u^{\kappa}) \partial_{\kappa} u^\beta  \partial_{\beta} h+c^2_s \epsilon^{\alpha \beta \gamma \delta}\mathrm{e}^{-h} w_\delta \partial_{ \beta} u_\gamma
			+
		  c^2_s \partial_\beta u^\beta \partial^\alpha h
		   -c^2_s \partial^\beta u^\alpha \partial_\beta h
		   \\
		   & +c^2_s u^{\alpha} u^\kappa \partial_{\beta} u^\beta  \partial_{\kappa} h
			 +c^2_s u^{\alpha} \partial_{\beta}  u^{\kappa} \partial_{\kappa} u^\beta
			-c^2_s u^\beta \partial_{\beta} u^\kappa \partial_{\kappa} u^\alpha
			 - c^2_s  u^\beta u^\kappa \partial_{\beta} u^\alpha \partial_{\kappa} h.
	\end{split}
\end{equation}
Multiplying $\eqref{a13}$ with $\Omega$ (defined in \eqref{AMd2})
, then we can write \eqref{a14} as \eqref{err}, where $Q^\alpha=\Omega \cdot Q_*^\alpha$.
\end{proof}
\begin{remark}
	(1) The right terms in \eqref{a06} and \eqref{a13} are a little different from \cite{DS}(Theorem 3.1), for the definition of $\square_g$ is different. (2) From the second-order derivative terms in \eqref{a05}, so it's clear for us to define the acoustical metric $g$ such that $g^{00}=-1$.
\end{remark}
\begin{proof}[proof of Lemma \ref{tr0}]
Operating $\partial_\gamma$ on \eqref{REE}, we have
\begin{equation}\label{a20}
	\begin{split}
		u^\kappa \partial_{\kappa}(\partial_\gamma u_{\delta})+\partial_{\gamma}u^\kappa \partial_\kappa u_\delta+\partial_\gamma\partial_\delta h+u_{\delta} \partial_\gamma u_\kappa \partial^\kappa h + u_\kappa \partial_\gamma u_\delta \partial^\kappa h+u_\kappa u_\delta \partial_\gamma \partial^\kappa h=0
	\end{split}
\end{equation}
Multiplying $\epsilon^{\alpha \beta \gamma \delta}\mathrm{e}^h u_\beta$ on \eqref{a20}, it follows
\begin{equation}\label{a21}
	\begin{split}
		0=& u^\kappa \partial_{\kappa}( \epsilon^{\alpha \beta \gamma \delta}\mathrm{e}^h u_\beta \partial_\gamma u_{\delta})
		  -\epsilon^{\alpha \beta \gamma \delta}\mathrm{e}^h u_\beta  u^\kappa \partial_{\kappa}h \partial_\gamma u_{\delta}
		\underbrace{ -\epsilon^{\alpha \beta \gamma \delta}\mathrm{e}^h   u^\kappa \partial_{\kappa} u_\beta \partial_\gamma u_{\delta} }_{\equiv R_1}
		\\
		& +\underbrace{ \epsilon^{\alpha \beta \gamma \delta}\mathrm{e}^h u_\beta \partial_{\gamma}u^\kappa \partial_\kappa u_\delta }_{\equiv R_2}+\underbrace{\epsilon^{\alpha \beta \gamma \delta}\mathrm{e}^h u_\beta \partial_\gamma\partial_\delta h }_{\equiv 0}
		 +\underbrace{\epsilon^{\alpha \beta \gamma \delta}\mathrm{e}^h u_\beta u_{\delta} \partial_\gamma u_\kappa \partial^\kappa h }_{\equiv 0}
		 \\
		 & +  \epsilon^{\alpha \beta \gamma \delta}\mathrm{e}^h u_\beta u_\kappa \partial_\gamma u_\delta \partial^\kappa h +\underbrace{ \epsilon^{\alpha \beta \gamma \delta}\mathrm{e}^h u_\beta u_\kappa u_\delta \partial_\gamma \partial^\kappa h }_{\equiv 0}.
	\end{split}
\end{equation}
Then \eqref{a21} tells us
\begin{equation}\label{a22}
		 u^\kappa \partial_{\kappa}w^\alpha=R_1+R_2.
\end{equation}
Let us calculate $R_1$ and $R_2$ as follows. For $R_1$, we have
\begin{equation}\label{a23}
\begin{split}
	R_1=& -\epsilon^{\alpha \beta \gamma \delta}\mathrm{e}^h   u^\kappa  \partial_\gamma u_{\delta} (\partial_{\kappa} u_\beta- \partial_{\beta} u_\kappa)-\epsilon^{\alpha \beta \gamma \delta}\mathrm{e}^h   u^\kappa  \partial_\gamma u_{\delta} \partial_{\beta} u_\kappa
	\\
	=& -\epsilon^{\alpha \beta \gamma \delta}\mathrm{e}^h   u^\kappa  \partial_\gamma u_{\delta} (\partial_{\kappa} u_\beta- \partial_{\beta} u_\kappa)
	\\
	=& -\epsilon^{\alpha \beta \gamma \delta}\mathrm{e}^h   u^\kappa  \partial_\gamma u_{\delta} \left( \epsilon_{\kappa \beta \eta \mu} \mathrm{e}^{-h}u^\eta w^\mu-u_\beta \partial_\kappa h +u_\kappa \partial_\beta h \right)
	\\
	=& \epsilon^{\alpha \beta \gamma \delta}\mathrm{e}^h   u^\kappa    u_\beta \partial_\kappa h \partial_\gamma u_{\delta} -\epsilon^{\alpha \beta \gamma \delta}\mathrm{e}^h   u^\kappa u_\kappa \partial_\gamma u_{\delta} \partial_\beta h
	\\
	=& -w^\alpha u^\kappa \partial_\kappa h + \epsilon^{\alpha \beta \gamma \delta}\mathrm{e}^h \partial_\gamma u_{\delta} \partial_\beta h .
\end{split}
\end{equation}
By using \eqref{cra1}, we can get
\begin{equation}\label{a24}
	\begin{split}
	\epsilon^{\alpha \beta \gamma \delta}\mathrm{e}^h \partial_\gamma u_{\delta} \partial_\beta h=& w^\alpha u^\beta \partial_\beta h-u^\alpha w^\beta \partial_\beta h   -\epsilon^{\alpha \beta \gamma \delta} \mathrm{e}^h u_\delta \partial_\beta h  \partial_\gamma h
	\\
	=& w^\alpha u^\beta-u^\alpha w^\beta \partial_\beta h .
	\end{split}
\end{equation}
Inserting \eqref{a24} to \eqref{a23}, we know
\begin{equation}\label{a25}
	\begin{split}
		R_1=-u^\alpha w^\beta \partial_\beta h.
	\end{split}
\end{equation}
For $R_2$, we calculate it by
\begin{equation}\label{a26}
	\begin{split}
		R_2=& \epsilon^{\alpha \beta \gamma \delta}\mathrm{e}^h u_\beta ( \partial_{\gamma}u^\kappa - \partial^\kappa u_\gamma) \partial_\kappa u_\delta+  \epsilon^{\alpha \beta \gamma \delta}\mathrm{e}^h u_\beta \partial^\kappa u_\gamma \partial_\kappa u_\delta
		\\
		=&  \epsilon^{\alpha \beta \gamma \delta}\mathrm{e}^h u_\beta ( \partial_{\gamma}u^\kappa - \partial^\kappa u_\gamma) \partial_\kappa u_\delta
		\\
		=& \epsilon^{\alpha \beta \gamma \delta}\mathrm{e}^h u_\beta ( \epsilon_{\gamma}^{\ \kappa \eta \mu} \mathrm{e}^{-h} u_\eta w_\mu- u^\kappa \partial_\gamma h+ u_\gamma \partial^\kappa h ) \partial_\kappa u_\delta
		\\
		=& \underbrace{\epsilon^{\alpha \beta \gamma \delta} \epsilon_{\gamma}^{\ \kappa \eta \mu} u_\beta u_\eta w_\mu \partial_\kappa u_\delta}_{\equiv R_3} \underbrace{- \epsilon^{\alpha \beta \gamma \delta} \mathrm{e}^h u_\beta u^\kappa \partial_\gamma h \partial_\kappa u_\delta }_{\equiv R_4}.
	\end{split}
\end{equation}
Let us first consider $R_4$. Using \eqref{REE}, we get
\begin{equation}\label{a27}
	\begin{split}
		R_4= \epsilon^{\alpha \beta \gamma \delta} \mathrm{e}^h u_\beta \partial_\gamma h \partial_\delta h+\epsilon^{\alpha \beta \gamma \delta} \mathrm{e}^h u_\beta u_\delta u_\kappa \partial_\gamma h  \partial^\kappa h=0.
	\end{split}
\end{equation}
For $R_3$, it yields
\begin{equation}\label{a28}
	\begin{split}
		R_3=& \left( \delta_\kappa^\alpha \delta^\delta_\eta \delta^\beta_\mu - \delta_\kappa^\alpha \delta^\beta_\eta \delta^\delta_\mu
		+\delta^\beta_\kappa \delta^\alpha_\eta \delta^\delta_\mu - \delta_\kappa^\beta \delta^\delta_\eta \delta^\alpha_\mu
		+ \delta_\kappa^\delta \delta^\beta_\eta \delta^\alpha_\mu - \delta_\kappa^\delta \delta^\alpha_\eta \delta^\beta_\mu  \right) u_\beta u_\eta w_\mu \partial_\kappa u_\delta
		\\
		=&  w^\mu \partial^\alpha u_\mu - w^\alpha \partial_\kappa u^\kappa+  u^\alpha w^\mu u^\kappa \partial_\kappa u_\mu.
	\end{split}
\end{equation}
Let us discuss the term $w^\mu \partial^\alpha u_\mu$ and $u^\alpha w^\mu u^\kappa \partial_\kappa u_\mu$ of \eqref{a28} in a further way. Using \eqref{REE} again, we get
\begin{equation}\label{a29}
	\begin{split}
	u^\alpha w^\mu u^\kappa \partial_\kappa u_\mu= u^\alpha w^\mu(-\partial_\mu h - u_\mu u_\kappa \partial^\kappa h)=-u^\alpha w^\mu \partial_\mu
	 h.
	\end{split}
\end{equation}
While, for $w^\mu \partial^\alpha u_\mu$, we calculate
\begin{equation}\label{a30}
	\begin{split}
		w^\mu \partial^\alpha u_\mu= & w^\mu ( \partial^\alpha u_\mu - \partial_\mu u^\alpha)+ w^\mu \partial_\mu u^\alpha
		\\
		=&  w^\mu ( \mathrm{e}^{-h} \epsilon^\alpha_{\ \mu \beta \gamma} u^\beta w^\gamma-u_\mu \partial^\alpha h +u^\alpha \partial_\mu h  )+ w^\mu \partial_\mu u^\alpha
		\\
		=& u^\alpha w^\mu \partial_\mu h+ w^\mu \partial_\mu u^\alpha.
	\end{split}
\end{equation}
Substituting \eqref{a29} and \eqref{a30} to \eqref{a28}, we have
\begin{equation}\label{a31}
	\begin{split}
R_3= w^\mu \partial_\mu u^\alpha - w^\alpha \partial_\kappa u^\kappa.
	\end{split}
\end{equation}
Combining \eqref{a22}, \eqref{a25}, \eqref{a26}, \eqref{a27} and \eqref{a31}, it follows
\begin{equation}\label{a32}
	u^\kappa \partial_\kappa w^\alpha = -u^\alpha w^\kappa \partial_\kappa h+ w^\kappa \partial_\kappa u^\alpha- w^\alpha \partial_\kappa u^\kappa.
\end{equation}
By the definition ${w}^\alpha=-\epsilon^{\alpha \beta \gamma \delta}\mathrm{e}^{h}u_{\beta}\partial_{\gamma}u_\delta$, we can directly compute
\begin{equation}\label{a33}
	\begin{split}
		\partial_\alpha {w}^\alpha = & -\epsilon^{\alpha \beta \gamma \delta}\mathrm{e}^{h} u_{\beta} \partial_\alpha h \partial_{\gamma}u_\delta
		-\epsilon^{\alpha \beta \gamma \delta}\mathrm{e}^{h} \partial_\alpha u_{\beta} \partial_{\gamma}u_\delta
		-\epsilon^{\alpha \beta \gamma \delta}\mathrm{e}^{h}  u_{\beta} \partial_\alpha \partial_{\gamma}u_\delta
		\\
		= & -\epsilon^{\alpha \beta \gamma \delta}\mathrm{e}^{h} u_{\beta} \partial_\alpha h \partial_{\gamma}u_\delta
		\\
		= & w^{\alpha} \partial_\alpha h.
	\end{split}
\end{equation}
In the following, we need to derive the transport equation of $\bW$. To achieve this goal, we operate $\partial_\gamma$ on \eqref{a32}. Then we have
\begin{equation}\label{a34}
	\begin{split}
		u^\kappa \partial_\kappa ( \partial_\gamma  w_\delta ) = & -\partial_\gamma u^\kappa \partial_\kappa w_\delta-\partial_\gamma u_\delta w^\kappa \partial_\kappa h - u_\delta \partial_\gamma w^\kappa \partial_\kappa h
		\\
		&-u_\delta w^\kappa \partial_\gamma \partial_\kappa h
		+ \partial_\gamma w^\kappa \partial_\kappa u_\delta+  w^\kappa \partial_\gamma \partial_\kappa u_\delta
		\\
		&- \partial_\gamma w_\delta \partial_\kappa u^\kappa- w_\delta \partial_\gamma  \partial_\kappa u^\kappa.
	\end{split}
\end{equation}
Multiplying $ \epsilon^{\alpha \beta \gamma \delta}u_\beta$ on \eqref{a34} and setting $ \overline{W}^\alpha=-\epsilon^{\alpha \beta \gamma \delta}u_\beta \partial_\gamma  w_\delta$, we obtain
\begin{equation}\label{a35}
	\begin{split}
		-u^\kappa \partial_\kappa \overline{W}^\alpha = & \underbrace{ \epsilon^{\alpha \beta \gamma \delta} u^\kappa \partial_\kappa u_\beta \partial_\gamma w_\delta }_{\equiv I_1}
	     \underbrace{ -\epsilon^{\alpha \beta \gamma \delta}u_\beta \partial_\gamma u^\kappa \partial_\kappa w_\delta }_{\equiv I_2}
		 \underbrace{-\epsilon^{\alpha \beta \gamma \delta}u_\beta w^\kappa  \partial_\gamma u_\delta \partial_\kappa h }_{\equiv I_3}
		 \\
		 & \underbrace{ - \epsilon^{\alpha \beta \gamma \delta}u_\beta u_\delta \partial_\gamma w^\kappa \partial_\kappa h }_{\equiv I_4}
		\underbrace{-\epsilon^{\alpha \beta \gamma \delta}u_\beta u_\delta w^\kappa \partial_\gamma \partial_\kappa h }_{\equiv I_5}
		+ \underbrace{\epsilon^{\alpha \beta \gamma \delta}u_\beta \partial_\gamma w^\kappa \partial_\kappa u_\delta }_{\equiv I_6}
		\\
		& +  \underbrace{\epsilon^{\alpha \beta \gamma \delta}u_\beta w^\kappa \partial_\gamma \partial_\kappa u_\delta }_{\equiv I_7}
		\underbrace{-\epsilon^{\alpha \beta \gamma \delta}u_\beta\partial_\gamma w_\delta \partial_\kappa u^\kappa }_{\equiv I_8}
			\underbrace{- \epsilon^{\alpha \beta \gamma \delta}u_\beta w_\delta \partial_\gamma  \partial_\kappa u^\kappa }_{\equiv I_9}.
	\end{split}
\end{equation}
Next, let us calculate the above terms $I_1, I_2, \cdots, I_9$. Firstly, it's easy for us to see
\begin{equation}\label{a36}
	I_4=0, \quad I_5=0.
\end{equation}
According to the definition of $\bw$ and $\overline{W}$, we can get
\begin{equation}\label{a37}
	I_3=\mathrm{e}^{-h}w^\alpha w^\kappa \partial_\kappa h,
\end{equation}
and also
\begin{equation}\label{a38}
	I_8=\overline{W}^\alpha \partial_\kappa u^\kappa.
\end{equation}
Then we compute $I_2$ by
\begin{equation}\label{a39}
	\begin{split}
		I_2=&\underbrace{ -\epsilon^{\alpha \beta \gamma \delta}u_\beta \partial_\gamma u^\kappa ( \partial_\kappa w_\delta - \partial_\delta w_\kappa) }_{\equiv I_{21}}
		-\epsilon^{\alpha \beta \gamma \delta}u_\beta \partial_\gamma u^\kappa \partial_\delta w_\kappa.
	\end{split}
\end{equation}
Using \eqref{cr04}, we get
\begin{equation}\label{a40}
	\begin{split}
		I_{21}=& -\epsilon^{\alpha \beta \gamma \delta}u_\beta \partial_\gamma u^\kappa ( \epsilon_{\kappa \delta \mu \nu} u^\mu \overline{W}^\nu- u^\mu u_\kappa \partial_\mu w_\delta  + u^\mu u_\kappa \partial_\delta w_\mu + u^\mu u_\delta \partial_\mu w_\kappa  - u^\mu u_\delta \partial_\kappa w_\mu )
		\\
		=& -\epsilon^{\alpha \beta \gamma \delta}u_\beta \partial_\gamma u^\kappa \cdot \epsilon_{\kappa \delta \mu \nu} u^\mu \overline{W}^\nu
		\\
		=& (-\delta_\kappa^\alpha \delta^\gamma_\mu \delta^\beta_\nu + \delta_\kappa^\alpha \delta^\beta_\mu \delta^\gamma_\nu
		-\delta^\beta_\kappa \delta^\alpha_\mu \delta^\gamma_\nu + \delta_\kappa^\beta \delta^\gamma_\mu \delta^\alpha_\nu
		- \delta^\gamma_\kappa \delta^\beta_\mu \delta^\alpha_\nu + \delta_\kappa^\gamma \delta^\alpha_\mu \delta^\beta_\nu  )u_\beta u^\mu  \partial_\gamma u^\kappa \overline{W}^\nu
		\\
		=& u_\mu u^\mu \partial_\nu u^\alpha \overline{W}^\nu-u_\mu u^\mu \partial_\kappa u^\kappa \overline{W}^\alpha
		\\
		=&-\overline{W}^\kappa \partial_\kappa u^\alpha+ \overline{W}^\alpha \partial_\kappa u^\kappa.
	\end{split}
\end{equation}
Inserting \eqref{a40} to \eqref{a39}, it follows
\begin{equation}\label{a41}
		I_2=-\overline{W}^\kappa \partial_\kappa u^\alpha+ \overline{W}^\alpha \partial_\kappa u^\kappa
		+\epsilon^{\alpha \beta \gamma \delta}u_\beta \partial_\delta u^\kappa \partial_\gamma w_\kappa.
\end{equation}
For $I_6$, we have
\begin{equation}\label{a42}
	\begin{split}
		I_6=& \epsilon^{\alpha \beta \gamma \delta}u_\beta  \partial_\gamma  w^\kappa ( \partial_\kappa u_\delta - \partial_\delta u_\kappa)
		+ \epsilon^{\alpha \beta \gamma \delta}u_\beta  \partial_\gamma  w^\kappa \partial_\delta u_\kappa
		\\
		=& \epsilon^{\alpha \beta \gamma \delta}u_\beta  \partial_\gamma  w^\kappa ( \epsilon_{\kappa \delta \mu \nu} \mathrm{e}^{-h}u_\mu w_\nu- u_\delta \partial_\kappa h + u_\kappa \partial_\delta h )
		+ \epsilon^{\alpha \beta \gamma \delta}u_\beta  \partial_\gamma  w^\kappa \partial_\delta u_\kappa
		\\
		=& \epsilon^{\alpha \beta \gamma \delta}\epsilon_{\kappa \delta \mu \nu}\mathrm{e}^{-h} u_\beta u_\mu w_\nu \partial_\gamma  w^\kappa+\epsilon^{\alpha \beta \gamma \delta} u_\beta u_\kappa \partial_\delta h \partial_\gamma  w^\kappa
		+ \epsilon^{\alpha \beta \gamma \delta}u_\beta  \partial_\gamma  w^\kappa \partial_\delta u_\kappa.
	\end{split}
\end{equation}
In a similar way, we can compute
\begin{equation}\label{a43}
	\begin{split}
		I_7=& \epsilon^{\alpha \beta \gamma \delta}u_\beta  w^\kappa \partial_\gamma ( \partial_\kappa u_\delta - \partial_\delta u_\kappa)
		+ \epsilon^{\alpha \beta \gamma \delta}u_\beta  w^\kappa \partial_\gamma \partial_\delta u_\kappa
		\\
		=& \epsilon^{\alpha \beta \gamma \delta}u_\beta  w^\kappa \partial_\gamma ( \partial_\kappa u_\delta - \partial_\delta u_\kappa)
		\\
		=& \epsilon^{\alpha \beta \gamma \delta}u_\beta  w^\kappa \partial_\gamma ( \epsilon_{\kappa \delta \mu \nu} \mathrm{e}^{-h}u_\mu w_\nu- u_\delta \partial_\kappa h + u_\kappa \partial_\delta h )
		\\
		=& \underbrace{ \epsilon^{\alpha \beta \gamma \delta}\epsilon_{\kappa \delta \mu \nu} u_\beta  w^\kappa \partial_\gamma (  \mathrm{e}^{-h}u_\mu w_\nu) }_{\equiv I_{71}} - \epsilon^{\alpha \beta \gamma \delta}u_\beta  w^\kappa \partial_\gamma u_\delta \partial_\kappa h + \epsilon^{\alpha \beta \gamma \delta}u_\beta  w^\kappa \partial_\gamma u_\kappa \partial_\delta h .
	\end{split}
\end{equation}
Above, we have used \eqref{OE00}. On the other hand, we can calculate
\begin{equation}\label{a44}
	\begin{split}
		I_{71}=& \epsilon^{\alpha \beta \gamma \delta}\epsilon_{\kappa \delta \mu \nu} u_\beta \left\{   \partial_\gamma (  \mathrm{e}^{-h} w^\kappa u_\mu w_\nu) - \mathrm{e}^{-h}  u_\mu w_\nu \partial_\gamma w^\kappa \right\}
		\\
		=& -\epsilon^{\alpha \beta \gamma \delta}\epsilon_{\kappa \delta \mu \nu}  \mathrm{e}^{-h} u_\beta  u_\mu w_\nu \partial_\gamma w^\kappa.
	\end{split}
\end{equation}
Inserting \eqref{a44} to \eqref{a43}, we can see
\begin{equation}\label{a45}
	\begin{split}
		I_7
		=& -\epsilon^{\alpha \beta \gamma \delta}\epsilon_{\kappa \delta \mu \nu}  \mathrm{e}^{-h} u_\beta  u_\mu w_\nu \partial_\gamma w^\kappa - \epsilon^{\alpha \beta \gamma \delta}u_\beta  w^\kappa \partial_\gamma u_\delta \partial_\kappa h + \epsilon^{\alpha \beta \gamma \delta}u_\beta  w^\kappa \partial_\gamma u_\kappa \partial_\delta h .
	\end{split}
\end{equation}
Adding \eqref{a42} and \eqref{a45}, then using $u^\kappa w_\kappa=0$, it yields
\begin{equation}\label{a46}
	\begin{split}
		I_6+I_7
		=& \epsilon^{\alpha \beta \gamma \delta} u_\beta u_\kappa \partial_\delta h \partial_\gamma  w^\kappa
		+ \epsilon^{\alpha \beta \gamma \delta}u_\beta  \partial_\gamma  w^\kappa \partial_\delta u_\kappa
		\\
		& - \epsilon^{\alpha \beta \gamma \delta}u_\beta  w^\kappa \partial_\gamma u_\delta \partial_\kappa h + \epsilon^{\alpha \beta \gamma \delta}u_\beta  w^\kappa \partial_\gamma u_\kappa \partial_\delta h
		\\
		=& \epsilon^{\alpha \beta \gamma \delta} u_\beta  \partial_\delta h \partial_\gamma  ( u_\kappa w^\kappa)
		- \epsilon^{\alpha \beta \gamma \delta} u_\beta w^\kappa  \partial_\delta h \partial_\gamma   u_\kappa
		+ \epsilon^{\alpha \beta \gamma \delta}u_\beta  \partial_\gamma  w^\kappa \partial_\delta u_\kappa
		\\
		& - \epsilon^{\alpha \beta \gamma \delta}u_\beta  w^\kappa \partial_\gamma u_\delta \partial_\kappa h + \epsilon^{\alpha \beta \gamma \delta}u_\beta  w^\kappa \partial_\gamma u_\kappa \partial_\delta h
		\\
		=&  \epsilon^{\alpha \beta \gamma \delta}u_\beta  \partial_\gamma  w^\kappa \partial_\delta u_\kappa
		 - \epsilon^{\alpha \beta \gamma \delta}u_\beta  w^\kappa \partial_\gamma u_\delta \partial_\kappa h
		 \\
		 =&  \epsilon^{\alpha \beta \gamma \delta}u_\beta  \partial_\gamma  w^\kappa \partial_\delta u_\kappa
		 + \mathrm{e}^{-h}  w^\alpha w^\kappa \partial_\kappa h.
	\end{split}
\end{equation}
It remains for us to compute $I_1$ and $I_9$. Using \eqref{cr04}, we have
\begin{equation}\label{a47}
	\begin{split}
		I_1=& u^\kappa \partial_\kappa u_\beta \left\{ \overline{W}^\alpha u^\beta -u^\alpha \overline{W}^\beta
		+ \epsilon^{\alpha \beta \gamma \delta} u^\mu \partial_\mu w_\gamma u_\delta -\epsilon^{\alpha \beta \gamma \delta} u^\mu \partial_\gamma w_\mu u_\delta \right\}
		\\
		=& -u^\alpha \overline{W}^\beta u^\kappa \partial_\kappa u_\beta + \underbrace{  \epsilon^{\alpha \beta \gamma \delta} u^\kappa u_\delta \partial_\kappa u_\beta \cdot u^\mu \partial_\mu w_\gamma }_{\equiv I_{11} }  \underbrace{ -\epsilon^{\alpha \beta \gamma \delta}u_\delta  u^\kappa \partial_\kappa u_\beta \cdot u^\mu \partial_\gamma w_\mu }_{\equiv I_{12} }  .
	\end{split}
\end{equation}
Using \eqref{a32}, we get
\begin{equation}\label{a48}
	\begin{split}
		I_{11}=&  \epsilon^{\alpha \beta \gamma \delta} u^\kappa u_\delta \partial_\kappa u_\beta (-u_\gamma w^\mu \partial_\mu h+ w^\mu \partial_\mu u_\gamma- w_\gamma \partial_\mu u^\mu)
		\\
		=& \epsilon^{\alpha \beta \gamma \delta} u^\kappa u_\delta w^\mu \partial_\kappa u_\beta   \partial_\mu u_\gamma
		-\epsilon^{\alpha \beta \gamma \delta} u^\kappa u_\delta w_\gamma \partial_\kappa u_\beta   \partial_\mu u^\mu.
	\end{split}
\end{equation}
Using $u^\mu \partial_\gamma w_\mu=\partial_\gamma(u^\mu w_\mu)- \partial_\gamma u^\mu w_\mu=- \partial_\gamma u^\mu w_\mu$, we derive
\begin{equation}\label{a49}
	\begin{split}
		I_{12}=&  \epsilon^{\alpha \beta \gamma \delta}u_\delta  w_\mu u^\kappa   \partial_\kappa u_\beta \partial_\gamma u^\mu.
\end{split}
\end{equation}
Inserting \eqref{a48} and \eqref{a49} to \eqref{a47}, it follows
\begin{equation}\label{a50}
	\begin{split}
		I_1
		=& -u^\alpha \overline{W}^\beta u^\kappa \partial_\kappa u_\beta  + \epsilon^{\alpha \beta \gamma \delta} u^\kappa u_\delta w^\mu \partial_\kappa u_\beta   \partial_\mu u_\gamma
		\\
		& -\epsilon^{\alpha \beta \gamma \delta} u^\kappa u_\delta w_\gamma \partial_\kappa u_\beta   \partial_\mu u^\mu +\epsilon^{\alpha \beta \gamma \delta}u_\delta  w_\mu u^\kappa   \partial_\kappa u_\beta \partial_\gamma u^\mu.
	\end{split}
\end{equation}
We still need to consider $I_9$. Using \eqref{REE}, we have
\begin{equation}\label{a51}
	\begin{split}
		I_9
		=
		& \epsilon^{\alpha \beta \gamma \delta} u_\beta w_\delta \partial_\gamma( c^{-2}_s u^\kappa \partial_\kappa h )
		\\
		=& -2\epsilon^{\alpha \beta \gamma \delta} c'_s c^{-3}_s u_\beta w_\delta u^\kappa \partial_\kappa h \partial_\gamma h
		+\epsilon^{\alpha \beta \gamma \delta} c^{-2}_s u_\beta w_\delta   \partial_\gamma u^\kappa \partial_\kappa h
		+ \underbrace{ \epsilon^{\alpha \beta \gamma \delta} c^{-2}_s u_\beta  w_\delta u^\kappa  \partial_\gamma  \partial_\kappa h }_{\equiv I_{91}}.
	\end{split}
\end{equation}
For $I_{91}$, we obtain
\begin{equation}\label{a52}
	\begin{split}
		I_{91}= & \epsilon^{\alpha \beta \gamma \delta} c^{-2}_s u_\beta  w_\delta u^\kappa \partial_\kappa (  \partial_\gamma   h)
		\\
		=& u^\kappa \partial_\kappa\{ \epsilon^{\alpha \beta \gamma \delta} c^{-2}_s u_\beta  w_\delta   \partial_\gamma  h \}
		+2 \epsilon^{\alpha \beta \gamma \delta} c'_s c^{-3}_s u^\kappa u_\beta  w_\delta   \partial_\gamma  h \partial_\kappa h
		\\
		& - \epsilon^{\alpha \beta \gamma \delta} c^{-2}_s u^\kappa \partial_\kappa u_\beta  w_\delta   \partial_\gamma  h
		- \epsilon^{\alpha \beta \gamma \delta} c^{-2}_s u^\kappa  u_\beta \partial_\kappa  w_\delta   \partial_\gamma  h .
	\end{split}
\end{equation}
Substituting \eqref{a52} to \eqref{a51}, we get
\begin{equation}\label{a53}
	\begin{split}
		I_9
		=&
		u^\kappa \partial_\kappa \left\{ \epsilon^{\alpha \beta \gamma \delta} c^{-2}_s u_\beta  w_\delta   \partial_\gamma  h \right\}
		 - \epsilon^{\alpha \beta \gamma \delta} c^{-2}_s u^\kappa \partial_\kappa u_\beta  w_\delta   \partial_\gamma  h
		\\
		& +\mathrm{e}^{-h}w^\alpha w^\kappa \partial_\kappa h- \epsilon^{\alpha \beta \gamma \delta} c^{-2}_s u^\kappa  u_\beta \partial_\kappa  w_\delta   \partial_\gamma  h
		 + \epsilon^{\alpha \beta \gamma \delta} c^{-2}_s u_\beta w_\delta   \partial_\gamma u^\kappa \partial_\kappa h  .
	\end{split}
\end{equation}
In a result, if we insert \eqref{a36}, \eqref{a37}, \eqref{a38}, \eqref{a41}, \eqref{a46}, \eqref{a50}, \eqref{a53} to \eqref{a35}, then we obtain
\begin{equation*}\label{a54}
	\begin{split}
		-u^\kappa \partial_\kappa \overline{W}^\alpha
		=& -u^\alpha \overline{W}^\beta u^\kappa \partial_\kappa u_\beta  + \epsilon^{\alpha \beta \gamma \delta} u^\kappa u_\delta w^\mu \partial_\kappa u_\beta   \partial_\mu u_\gamma
		 -\epsilon^{\alpha \beta \gamma \delta} u^\kappa u_\delta w_\gamma \partial_\kappa u_\beta   \partial_\mu u^\mu
		 \\
		 &  +\epsilon^{\alpha \beta \gamma \delta}u_\delta  w_\mu u^\kappa   \partial_\kappa u_\beta \partial_\gamma u^\mu
		 -\overline{W}^\kappa \partial_\kappa u^\alpha+ \overline{W}^\alpha \partial_\kappa u^\kappa
		+\epsilon^{\alpha \beta \gamma \delta}u_\beta \partial_\delta u^\kappa \partial_\gamma w_\kappa
		\\
		&  +\mathrm{e}^{-h}  w^\alpha w^\kappa \partial_\kappa h
		 + \epsilon^{\alpha \beta \gamma \delta}u_\beta  \partial_\gamma  w^\kappa \partial_\delta u_\kappa
		+ \mathrm{e}^{-h}  w^\alpha w^\kappa \partial_\kappa h
		+ \overline{W}^\alpha \partial_\kappa u^\kappa
		\\
		 & + u^\kappa \partial_\kappa \left\{ \epsilon^{\alpha \beta \gamma \delta} c^{-2}_s u_\beta  w_\delta   \partial_\gamma  h \right\}
		- \epsilon^{\alpha \beta \gamma \delta} c^{-2}_s u^\kappa \partial_\kappa u_\beta  w_\delta   \partial_\gamma  h
		\\
		& - \epsilon^{\alpha \beta \gamma \delta} c^{-2}_s u^\kappa  u_\beta \partial_\kappa  w_\delta   \partial_\gamma  h
		+ \epsilon^{\alpha \beta \gamma \delta} c^{-2}_s u_\beta w_\delta   \partial_\gamma u^\kappa \partial_\kappa h  .
	\end{split}
\end{equation*}
Noting $W^\alpha= \overline{W}^\alpha +\epsilon^{\alpha \beta \gamma \delta} c^{-2}_s u_\beta  w_\delta   \partial_\gamma  h $, so we get
\begin{equation}\label{a55}
	\begin{split}
		-u^\kappa \partial_\kappa {W}^\alpha
		=& -u^\alpha ( {W}^\beta -\epsilon^{\beta \nu \gamma \delta}c^{-2}_s u_\nu w_\delta \partial_\gamma h ) u^\kappa \partial_\kappa u_\beta -( {W}^\kappa  - \epsilon^{\kappa \beta \gamma \delta}c^{-2}_s u_\beta w_\delta \partial_\gamma h) \partial_\kappa u^\alpha
		\\
		& + {2}( {W}^\alpha - \epsilon^{\alpha \beta \gamma \delta}c^{-2}_s u_\beta w_\delta \partial_\gamma h ) \partial_\kappa u^\kappa
		 + \epsilon^{\alpha \beta \gamma \delta} u^\kappa u_\delta w^\mu \partial_\kappa u_\beta   \partial_\mu u_\gamma
		\\
		&-\epsilon^{\alpha \beta \gamma \delta} u^\kappa u_\delta w_\gamma \partial_\kappa u_\beta   \partial_\mu u^\mu
		  +\epsilon^{\alpha \beta \gamma \delta}u_\delta  w_\mu u^\kappa   \partial_\kappa u_\beta \partial_\gamma u^\mu		
		+{2} \epsilon^{\alpha \beta \gamma \delta}u_\beta \partial_\delta u^\kappa \partial_\gamma w_\kappa
		\\
		&  +{2}\mathrm{e}^{-h}  w^\alpha w^\kappa \partial_\kappa h
		- \epsilon^{\alpha \beta \gamma \delta} c^{-2}_s u^\kappa \partial_\kappa u_\beta  w_\delta   \partial_\gamma  h
		 - \epsilon^{\alpha \beta \gamma \delta} c^{-2}_s u^\kappa  u_\beta \partial_\kappa  w_\delta   \partial_\gamma  h
		\\
		& + \epsilon^{\alpha \beta \gamma \delta} c^{-2}_s u_\beta w_\delta   \partial_\gamma u^\kappa \partial_\kappa h .
	\end{split}
\end{equation}
After a simple calculation on \eqref{a55}, we have
\begin{equation}\label{a56}
	\begin{split}
		-u^\kappa \partial_\kappa {W}^\alpha
		=& \sum\nolimits_{a=0}^6 J_a-u^\alpha  {W}^\beta  u^\kappa \partial_\kappa u_\beta - {W}^\kappa   \partial_\kappa u^\alpha
		+  \epsilon^{\alpha \beta \gamma \delta} c^{-2}_s u_\beta w_\delta   \partial_\gamma u^\kappa \partial_\kappa h
		+{2} {W}^\alpha \partial_\kappa u^\kappa
		\\
		&  +{2}\mathrm{e}^{-h}  w^\alpha w^\kappa \partial_\kappa h +{2} \epsilon^{\alpha \beta \gamma \delta}u_\beta \partial_\delta u^\kappa \partial_\gamma w_\kappa
		+ \epsilon^{\kappa \beta \gamma \delta}c^{-2}_s u_\beta w_\delta \partial_\gamma h \partial_\kappa u^\alpha,
	\end{split}
\end{equation}
where
\begin{equation*}\label{a57}
	\begin{split}
	J_0= &  \epsilon^{\beta \nu \gamma \delta}c^{-2}_s u^\alpha u_\nu w_\delta \partial_\gamma h (u^\kappa \partial_\kappa u_\beta) ,
	\\
	J_1= &   \epsilon^{\alpha \beta \gamma \delta}  u_\delta w^\mu     \partial_\mu u_\gamma (u^\kappa \partial_\kappa u_\beta) ,
		\\
	J_2=	&  -\epsilon^{\alpha \beta \gamma \delta}  u_\delta w_\gamma  \partial_\mu u^\mu (u^\kappa \partial_\kappa u_\beta ) ,
	 \\
	J_3=& \epsilon^{\alpha \beta \gamma \delta}u_\delta  w_\mu  \partial_\gamma u^\mu (u^\kappa   \partial_\kappa u_\beta)	,
		\\
	J_4=	&  - {2}\epsilon^{\alpha \beta \gamma \delta}c^{-2}_s u_\beta w_\delta \partial_\gamma h  \partial_\kappa u^\kappa ,
	\\
	J_5= &	 - \epsilon^{\alpha \beta \gamma \delta} c^{-2}_s   w_\delta   \partial_\gamma  h (u^\kappa \partial_\kappa u_\beta) ,
	\\
	J_6=&	 - \epsilon^{\alpha \beta \gamma \delta} c^{-2}_s   u_\beta  \partial_\gamma  h (u^\kappa \partial_\kappa  w_\delta) .
	\end{split}
\end{equation*}
Inserting \eqref{REE} to $J_0$, we obtain
\begin{equation}\label{a58}
	\begin{split}
		J_0= &  \epsilon^{\beta \nu \gamma \delta}c^{-2}_s u^\alpha u_\nu w_\delta \partial_\gamma h (-\partial_{ \beta} h- u_\beta u^\kappa \partial_\kappa h) =0.
	\end{split}
\end{equation}
Similarly, we can also derive
\begin{equation}\label{a59}
	\begin{split}
		J_1= &  -\epsilon^{\alpha \beta \gamma \delta}  u_\delta w^\mu     \partial_\mu u_\gamma \partial_\beta h
		\\
		=& -\epsilon^{\alpha \beta \gamma \delta}  u_\delta (w^\mu     \partial_\gamma u_\mu- w^\kappa u_\gamma \partial_\kappa h ) \partial_\beta h
		\\
		=& -\epsilon^{\alpha \beta \gamma \delta}  u_\delta  w^\mu     \partial_\gamma u_\mu  \partial_\beta h.
	\end{split}
\end{equation}
where we use \eqref{cra1}. In a similar way, we also obtain
\begin{equation}\label{a60}
	\begin{split}
		J_2=	& -\epsilon^{\alpha \beta \gamma \delta}  u_\delta w_\gamma  \partial_\mu u^\mu (-\partial_\beta h- u_\beta u^\kappa \partial_\kappa h)
		\\
		=& \epsilon^{\alpha \beta \gamma \delta}  u_\delta w_\gamma  \partial_\mu u^\mu \partial_\beta h
		\\
		=& \epsilon^{\alpha \beta \gamma \delta}  u_\beta w_\delta  \partial_\kappa u^\kappa \partial_\gamma h,
	\end{split}
\end{equation}
and
\begin{equation}\label{a61}
	\begin{split}
		J_3=& \epsilon^{\alpha \beta \gamma \delta}u_\delta  w_\mu  \partial_\gamma u^\mu (-\partial_\beta h- u_\beta u^\kappa \partial_\kappa h)
		\\
		=& -\epsilon^{\alpha \beta \gamma \delta}u_\delta  w_\mu  \partial_\gamma u^\mu \partial_\beta h,
		\\
		J_5= &	 - \epsilon^{\alpha \beta \gamma \delta} c^{-2}_s   w_\delta   \partial_\gamma  h (-\partial_\beta h- u_\beta u^\kappa \partial_\kappa h)
		\\
		= & \epsilon^{\alpha \beta \gamma \delta} c^{-2}_s  u_\beta  w_\delta   \partial_\gamma  h  ( u^\kappa \partial_\kappa h )
		\\
		=& -\epsilon^{\alpha \beta \gamma \delta} u_\beta  w_\delta   \partial_\gamma  h   \partial_\kappa u^\kappa ,
		\\
		J_6=&	 - \epsilon^{\alpha \beta \gamma \delta} c^{-2}_s   u_\beta  \partial_\gamma  h (-u_\delta u^\kappa \partial_\kappa h+ w^\kappa \partial_\kappa u_\delta- w_\delta \partial_\kappa u^\kappa)
		\\
		=&  - \epsilon^{\alpha \beta \gamma \delta} c^{-2}_s   u_\beta w^\kappa \partial_\gamma  h   \partial_\kappa u_\delta
		 +  \epsilon^{\alpha \beta \gamma \delta} c^{-2}_s   u_\beta  w_\delta  \partial_\gamma  h\partial_\kappa u^\kappa.
	\end{split}
\end{equation}
Above, we also use \eqref{a32}. From \eqref{a59}, \eqref{a60}, and \eqref{a61}, it's easy for us to see
\begin{equation}\label{a62}
	\begin{split}
		J_2+J_5=	& 0,
		\\
	J_1+J_3+J_4+J_6=& - \epsilon^{\alpha \beta \gamma \delta} c^{-2}_s u_\beta w_\delta \partial_\gamma h \partial_\kappa u^\kappa
	-( c^{-2}_s +2 ) \epsilon^{\alpha \beta \gamma \delta} u_\beta w^\kappa \partial_\delta u_\kappa \partial_\gamma h.
	\end{split}
\end{equation}
Inserting \eqref{a58} and \eqref{a62} to \eqref{a56}, we prove that
\begin{equation}\label{a63}
	\begin{split}
		u^\kappa \partial_\kappa {W}^\alpha
		=& u^\alpha  {W}^\beta  u^\kappa \partial_\kappa u_\beta + {W}^\kappa   \partial_\kappa u^\alpha-{2} {W}^\alpha \partial_\kappa u^\kappa
		-  \epsilon^{\alpha \beta \gamma \delta} c^{-2}_s u_\beta w_\delta   \partial_\gamma u^\kappa \partial_\kappa h
		\\
		&  -{2}\mathrm{e}^{-h}  w^\alpha w^\kappa \partial_\kappa h -{2} \epsilon^{\alpha \beta \gamma \delta}u_\beta \partial_\delta u^\kappa \partial_\gamma w_\kappa
		- \epsilon^{\kappa \beta \gamma \delta}c^{-2}_s u_\beta w_\delta \partial_\gamma h \partial_\kappa u^\alpha
		\\
		& + \epsilon^{\alpha \beta \gamma \delta} c^{-2}_s u_\beta w_\delta \partial_\gamma h \partial_\kappa u^\kappa
		+( c^{-2}_s +2 ) \epsilon^{\alpha \beta \gamma \delta} u_\beta w^\kappa \partial_\delta u_\kappa \partial_\gamma h.
	\end{split}
\end{equation}
Combining \eqref{a32}, \eqref{a33}, and \eqref{a63}, we have proved Lemma \ref{tr0}.
\end{proof}
In the next, we prove some lemmas as a complement.
\begin{Lemma}\label{equ0}
	Under the settings \eqref{mo5}-\eqref{mo9}, the system \eqref{OREE} can be written as \eqref{REE}.
\end{Lemma}
\begin{proof}
By using \eqref{mo6} and \eqref{mo9}, we get
\begin{equation}\label{mo10}
	\begin{split}
		p+\varrho=q\text{e}^h.
	\end{split}
\end{equation}
Operating derivatives to $\varrho$ for \eqref{mo10}, we can see
\begin{equation}\label{mo11}
	\begin{split}
		p'(\varrho)+1=\text{e}^h \frac{dq}{d\varrho}+q\text{e}^h \frac{dh}{d\varrho}.
	\end{split}
\end{equation}
Inserting \eqref{mo4}, \eqref{mo5} to \eqref{mo11}, it yields
\begin{equation}\label{mo12}
	\begin{split}
		c^2_s+1=\text{e}^h \frac{q}{p+\varrho}+q\text{e}^h \frac{dh}{d\varrho}.
	\end{split}
\end{equation}
Using \eqref{mo10}, then \eqref{mo12} tells us that
\begin{equation}\label{mo13}
	\begin{split}
		c^2_s=q\text{e}^h \frac{dh}{d\varrho}.
	\end{split}
\end{equation}
From \eqref{mo13}, we find
\begin{equation}\label{mo14}
	\begin{split}
		\partial_\kappa \varrho=c^{-2}_sq\text{e}^h\partial_\kappa h.
	\end{split}
\end{equation}
Using \eqref{mo10}, we can derive that
\begin{equation}\label{mo15}
	\begin{split}
		 (p+\varrho) \partial_\kappa u^\kappa  = & q\text{e}^h \partial_\kappa u^\kappa,
		\\
		 (p+\varrho)  u^\kappa \partial_\kappa u^\alpha=& q\text{e}^h  u^\kappa \partial_\kappa u^\alpha.
	\end{split}
\end{equation}
From \eqref{mo14}, we obtain
\begin{equation}\label{mo16}
	\begin{split}
		 u^\kappa \partial_\kappa \varrho=& c^{-2}_sq\text{e}^h u^\kappa \partial_\kappa h,
		\\
	 \partial_\kappa p=& c^2_s \partial_\kappa \varrho=q\text{e}^h\partial_\kappa h.
	\end{split}
\end{equation}
Combining \eqref{mo15}-\eqref{mo16} and using \eqref{OREE} yields the system \eqref{REE}.
\end{proof}
\begin{Lemma}\label{app1}
	Let the second-order operator $\mathbf{P}=\mathrm{I}-\mathrm{P}^{\beta \gamma} \partial^2_{\beta \gamma}$, where $\mathrm{P}^{\beta \gamma} = m^{\beta \gamma}+2u^{\beta}u^{\gamma}$. Then $\mathbf{P}$ is an elliptic operator on $[0,T]\times \mathbb{R}^{3}$, where $[0,T]$ is the liftime interval of the solution of \eqref{REE}. Furthermore, for any function $f$ defined on $[0,T]\times \mathbb{R}^{3}$, then the following estimate
		\begin{equation}\label{appl}
		\| f \|_{L_{[0,T]}^2 H_x^a} \lesssim \|\mathbf{P} f \|_{L_{[0,T]}^2 H_x^{a-2}}, \quad a \in \mathbb{R},
	\end{equation}
	holds.
\end{Lemma}
\begin{proof}
	We only need to prove that the matrix $\mathrm{P}=(\mathrm{P}^{\beta \gamma})_{4\times 4}$ is a positive definite matrix. So let us pick four leading principal minors
	\begin{equation*}
		\begin{split}
			p_1=\left | \begin{matrix}
				-1+2(u^0)^2
			\end{matrix} \right | , \qquad p_2=\left | \begin{matrix}
				-1+2(u^0)^2 &2u^0u^1    \\
				2u^0u^1 &1+2(u^1)^2   \\
			\end{matrix} \right |,
		\end{split}
	\end{equation*}
	and
		\begin{equation*}
		\begin{split}
		p_3= \left | \begin{matrix}
			-1+2(u^0)^2 &2u^0u^1   & 2u^0u^2 \\
			2u^0u^1 &1+2(u^1)^2 & 2u^1u^2  \\
			2u^0u^2 & 2u^1u^2 &1+2(u^2)^2 \\
		\end{matrix} \right |  ,
		\end{split}
	\end{equation*}
	and
		\begin{equation*}
		\begin{split}
			p_4= \left | \begin{matrix}
				-1+2(u^0)^2 &2u^0u^1   & 2u^0u^2 & 2u^0u^3 \\
				2u^0u^1 &1+2(u^1)^2 & 2u^1u^2 & 2u^1u^3  \\
				2u^0u^2 & 2u^1u^2 &1+2(u^2)^2  & 2u^2u^3  \\
				2u^0u^3 & 2u^1u^3 &2u^2u^3  & 1+2(u^3)^2 \\
			\end{matrix} \right | .
		\end{split}
	\end{equation*}
For $p_1$, using $m_{\alpha \beta}u^\alpha u^\beta=-1$, so we get
\begin{equation}\label{app2}
	\begin{split}
		p_1=-1+2(u^0)^2 \geq 1.
	\end{split}
\end{equation}
By direct calculations and using $m_{\alpha \beta}u^\alpha u^\beta=-1$, it also yields
		\begin{equation}\label{app3}
	\begin{split}
		p_2=& -1+2\{(u^0)^2 - (u^1)^2\} \geq 1,
		\\
		p_3=& -1+2\{(u^0)^2 - (u^1)^2- (u^2)^2 \} \geq 1,
		\\
		p_4=& -1+2\{ (u^0)^2 - (u^1)^2- (u^2)^2- (u^3)^2  \} \geq 1.
	\end{split}
\end{equation}
Combing \eqref{app2} and \eqref{app3}, so $\mathrm{P}$ is a positive definite matrix. Therefore, $\mathbf{P}$ is an elliptic operator.

To prove \eqref{appl}, we give a extension for the relativistic velocity $\bu$ and $f$:
\begin{equation*}
	\bar{f}=\begin{cases}
		 & f, \quad   t\in [0,T],
		\\
		 & 0,  \quad t\in (-\infty,0)\cup (T,\infty),
	\end{cases}
\end{equation*}
and
\begin{equation*}
	\bar{\bu}=\begin{cases}
		 & \bu, \quad \quad \quad \quad \quad t\in [0,T],
		\\
		 & (1,0,0,0)^{\mathrm{T}}, \quad t \in (-\infty,0)\cup (T,\infty).
	\end{cases}
\end{equation*}
Set
\begin{equation*}
	\mathbf{\bar{P}} f=f+ ( m^{\alpha \beta} + 2 \bar{u}^\alpha  \bar{u}^\beta) \partial^2_{\alpha \beta} f. 
\end{equation*}
Then we have
\begin{equation*}
	\mathbf{\bar{P}}= \begin{cases}
		& \mathbf{P}, \quad \quad \quad \quad \quad \quad \quad \quad t\in [0,T],
		\\
		& \mathrm{I}+\partial^2_{t}+\Delta, \quad t \in (-\infty,0)\cup (T,\infty),
	\end{cases} 
\end{equation*}
where $\Delta=\partial^2_1+\partial^2_2+\partial^2_3$. As a result, $\mathbf{\bar{P}}$ is an elliptic operator on $\mathbb{R}^{1+3}$. By space-time Fourier transform, it yields
\begin{equation*}
	\begin{split}
		\| f \|_{L^2_{[0,T]}H^a_x} = & \| \bar{f}\|_{L^2_{\mathbb{R}} H^a_x}
		\\
		=& \| (I+\Delta)^{\frac{a}{2}}\mathbf{P}^{-1} (\mathbf{\bar{P}} \bar{f})\|_{L^2_{\mathbb{R}} L^2_x}
		\\
		\lesssim & \| (1+|\xi|^2)^{\frac{a}{2}} (1+|\tau|^2+|\xi|^2)^{-1}   \widetilde{\mathbf{\bar{P}} \bar{f}}\|_{L^2_{\tau,\xi}}
			\\
		\lesssim & \| (1+|\xi|^2)^{\frac{a-2}{2}}  \widetilde{\mathbf{\bar{P}} \bar{f}}\|_{L^2_{\tau,\xi}}
		\\
		=& \| \mathbf{\bar{P}} \bar{f} \|_{L^2_{\mathbb{R}}H^{a-2}_x}
		\\
		=& \| \mathbf{{P}} {f} \|_{L^2_{[0,T]}H^{a-2}_x},
	\end{split}
\end{equation*}
where $\widetilde{\cdot}$ represents the Fourier transform of space-time. Therefore, the estimate \eqref{appl} holds.
\end{proof}

\subsection{Guidance of notations} There are many notations in our paper. So we give a guidance on some fixed notations.
\begin{equation*}
	\begin{split}
		& \bu=(u^0,u^1,u^2,u^3)^{\mathrm{T}}: \mathrm{velocity} \qquad \qquad \qquad \mathring{\bu}=(u^1,u^2,u^3)^{\mathrm{T}}, 
		\\
		& \varrho: \mathrm{energy \ density} \qquad \qquad \qquad \qquad \qquad  \  h: \mathrm{logarithmic \ enthalpy},
		\\
		& c_s: \mathrm{sound \ speed}, \mathrm{c.f.} \eqref{mo4}, \qquad \qquad \qquad  \bw=(w^0,w^1,w^2,w^3)^{\mathrm{T}}: \mathrm{modified \ vorticity}, \mathrm{c.f.} \eqref{VVd},
		\\
		& g: \mathrm{acoustic \ metric}, \mathrm{c.f.} \eqref{AMd}, \qquad \qquad \qquad \Omega: \mathrm{a \ function \ relying \ on} \ h, \bu, \mathrm{c.f.} \eqref{AMd2}, 
		\\
		& \bW=(W^0,W^1,W^2,W^3)^{\mathrm{T}}, \mathrm{c.f.} \eqref{MFd}\qquad \qquad 
		\bG=(G^0,G^1,G^2,G^3), \mathrm{c.f.} \eqref{MFd},
		\\
		& \mathring{\bw}=(w^1,w^2,w^3)^{\mathrm{T}}, \mathrm{c.f.} \eqref{MFda}\qquad \qquad \qquad \qquad \mathring{\bW}=(W^1,W^2,W^3)^{\mathrm{T}}, \mathrm{c.f.} \eqref{MFda}\qquad \qquad
		\\
		& \bQ=(Q^0,Q^1,Q^2,Q^3)^{\mathrm{T}}: \mathrm{c.f.} \eqref{err}, \qquad \qquad \qquad \mathring{\bG}=(G^1,G^2,G^3)^{\mathrm{T}}, \mathrm{c.f.} \eqref{MFda}, 
		\\
		& D:\ \mathrm{quadrtic \ terms \ of} \ du, dh, \mathrm{c.f.} \eqref{err1},\qquad \quad \Gamma: \mathrm{a \ scalar \ function \ for \ } d\bu\cdot d\bw, \mathrm{c.f.} \eqref{YXg},
		\\
		& \bF=(F^0,F^1,F^2,F^3)^{\mathrm{T}}: \mathrm{c.f.} \eqref{YX0}, \qquad \qquad \qquad \bE=(F^0,F^1,F^2,F^3)^{\mathrm{T}}: \mathrm{c.f.} \eqref{YX1},
		\\
		& \mathbf{Z}=(Z^0,Z^1,Z^2,Z^3)^{\mathrm{T}}, \mathrm{c.f.} \eqref{fdwZ},\qquad \qquad \qquad \mathbf{B}=(B^0,B^1,B^2,B^3)^{\mathrm{T}}, \mathrm{c.f.} \eqref{fdwB},
	\end{split}
\end{equation*}
\begin{equation*}
	\begin{split}
		& \mathrm{I}: \mathrm{identity \ operator}, \qquad \qquad \qquad \qquad \mathbf{P}=\mathrm{I}-(m^{ \beta \gamma}+ 2u^\beta u^\gamma)\partial^2_{\beta \gamma},
		\\
		& \square_g= g^{\alpha \beta} \partial^2_{\alpha \beta}, \qquad \qquad \qquad \qquad \qquad \nabla=(\partial_1,\partial_2,\partial_3)^{\mathrm{T}},
		\\
		& \mathbf{T}=\partial_t + \frac{u^i}{u^0}\partial_i,\qquad \qquad \qquad \qquad  \mathbf{{P}}^{-1}: \mathrm{the \ inverse \ operator \ for} \ \mathbf{{P}},
		\\
		& \bu_{-}=\mathbf{{P}}^{-1} (\mathrm{e}^{-h}\bW), \qquad \qquad \qquad \qquad \qquad \bu_{+}=\bu-\bu_{-},
		\\
		& \Delta=\partial^1_1+\partial^1_2+\partial^1_2, \qquad \qquad \qquad \qquad \qquad \Lambda_x=(-\Delta)^{\frac12},
		\\
		& \mathcal{{H}}: \mathrm{c.f.} \eqref{401}-\eqref{403},\qquad \qquad \qquad \qquad \qquad \Re: \mathrm{c.f.} \eqref{500},
		\\
		& \Sigma_{\theta,r}: \mathrm{c.f.} \eqref{sig}, \qquad \qquad \qquad \qquad \delta, \delta_0,\delta_1: \mathrm{c.f.} \eqref{a1},
		\\
		& \epsilon_0, \epsilon_1, \epsilon_2, \epsilon_3: \mathrm{c.f.} \eqref{a0}.
			\end{split}
\end{equation*}
\section*{Acknowledgments} The author would like to express great gratitude to the reviewers, who helped us improve the paper. We also express thanks to Lars Andersson for his encouragement and support on studying low regularity problems. We also thank Philippe G. LeFloch and Siyuan Ma for some suggestions. The author is also supported by National Natural Science Foundation of China (Grant No. 12101079), and the Fundamental Research Funds for the Central Universities.

\end{document}